\documentclass[12pt]{book}
\usepackage{index}
\makeindex

\newindex{notation}{adx}{and}{Index of symbols}
\newindex{notion}{bdx}{bnd}{Index of notions}
\input{package.tex}
\newcommand\notion[1]{\textit{#1}\index[notion]{#1}}
\newcommand\wcnotion[2]{\textit{#1}\index[notion]{#2}}
\newcommand\wcnotionsym[3]{\textit{#1}\index[notation]{#2}\index[notion]{#3}}
\newcommand\wcsnotion[3]{\textit{#1}\index[notion]{#2!\textit{#3}}}
\newcommand\snotion[2]{\textit{#1}\index[notion]{#1!\textit{#2}}}
\newcommand\snotionsym[3]{\textit{#1}\index[notion]{#1!\textit{#3}}\index[notation]{#2!\textit{#3}}}
\newcommand\wcsnotionsym[4]{\textit{#1}\index[notation]{#2!\textit{#4}}\index[notion]{#3!\textit{#4}}}

\newcommand\wcnotation[2]{\textit{#1}\index[notation]{#2}}
\newcommand\wcsnotation[3]{\textit{#1}\index[notation]{#2!\textit{#3}}}

\newcommand\sym[1]{\index[notation]{#1}}
\newcommand\ssym[2]{\index[notation]{#1!\textit{#2}}}

\newcommand{\exclam}{!}

\newcommand{\Zb}{\mathbb{Z}} 
 
\newcommand{\Nb}{\mathbb{N}}
\newcommand{\Tb}{\mathbf{T}} 
 
\newcommand{\Ib}{\mathbb{I}}

\newcommand{\Sb}{\mathbb{S}}

\def\-{\raisebox{.75pt}{-}}

\newcommand{\uvar}{\_}

\newcommand{\Db}{\mathbf{D}} 
\DeclareMathOperator*{\dom}{dom}
\DeclareMathOperator*{\codom}{codom}
\DeclareMathOperator{\tw}{tw}

\newcommand{\Rb}{\mathbf{R}} 
\newcommand{\Lb}{\mathbf{L}} 
\newcommand{\Fb}{\mathbf{F}} 
\DeclareMathOperator{\Gb}{G} 
  
\DeclareMathOperator{\N}{N}
\DeclareMathOperator{\T}{T}
\DeclareMathOperator{\J}{J}

\DeclareMathOperator*{\W}{W}
\DeclareMathOperator*{\Wm}{tW}
\DeclareMathOperator*{\Wseg}{W_{Seg}}
\DeclareMathOperator*{\Wsat}{W_{Sat}}

\DeclareMathOperator*{\M}{M}
\DeclareMathOperator*{\Mm}{tM}
\DeclareMathOperator*{\Mseg}{M_{Seg}}
\DeclareMathOperator*{\Msat}{M_{Sat}}

\DeclareMathOperator*{\I}{I}
\DeclareMathOperator*{\F}{F}

\DeclareMathOperator*{\CDA}{ADC}
\DeclareMathOperator*{\CDAB}{ADC_B}

\newcommand\omegacat{\omega\mbox{-$\cat$}}
\DeclareMathOperator\Set{Set}
\DeclareMathOperator\Sp{Sp}

\DeclareMathOperator*{\Sq}{Sq}

\DeclareMathOperator{\Hom}{Hom}

\DeclareMathOperator*{\Lfib}{LFib}

\DeclareMathOperator*{\LCartoperator}{LCart}

\newcommand{\LCart}{\mbox{$\LCartoperator$}}

\newcommand{\LCartc}{\mbox{$\LCartoperator$}^c}
\DeclareMathOperator*{\RCart}{RCart}

\newcommand{\uLCart}{\underline{\LCartoperator}}
\newcommand{\uLCartc}{\underline{\LCartoperator}^c}

\DeclareMathOperator{\uHom}{\underline{Hom}}
\DeclareMathOperator{\gHom}{\underline{Hom}_{\ominus}}
\DeclareMathOperator{\Map}{Map}
\DeclareMathOperator{\im}{Im}

\newcommand{\uni}{\underline{\omega}}
\newcommand\w[1]{\widehat{#1}}

\DeclareMathOperator*{\ev}{ev}
\DeclareMathOperator*{\Arr}{Arr}
\newcommand{\Noiun}{\N_{\tiny{(\omega,1)}}}

\newcommand{\colim}{\operatornamewithlimits{colim}}
\newcommand{\laxcolim}{\operatornamewithlimits{laxcolim}}
\newcommand{\laxlim}{\operatornamewithlimits{laxlim}}

\DeclareMathOperator{\Lan}{Lan}

\newcommand\iun{(\infty,1)}
\newcommand\io{(\infty,\omega)}
\newcommand\ioun{(\infty,\omega,1)}

\newcommand\zo{(0,\omega)}

\DeclareMathOperator{\hstar}{\hat{\star}}
\DeclareMathOperator{\htimes}{\hat{\times}}

\newcommand{\costar}{\mathbin{\overset{co}{\star}}}
\newcommand{\fwedge}{\mathbin{\rotatebox[origin=c]{270}{$\gtrdot$}}}

\newcommand{\invamalg}{\mathbin{\rotatebox[origin=c]{180}{$\amalg$}}}
\DeclareMathOperator{\botimes}{\bar{\otimes}}
\DeclareMathOperator\cst{cst}
\DeclareMathOperator\Operatormark{mk}
\newcommand{\mk}{\Operatormark}

\DeclareMathOperator\Fun{Fun}

\DeclareMathOperator\End{End}

\DeclareMathOperator\mcat{cat_m}
\DeclareMathOperator\cat{cat}
\DeclareMathOperator\grd{grd}
\DeclareMathOperator\R{R}

\newcommand\ocat{(\infty,\omega)\mbox{-$\cat$}}
\newcommand\ouncat{(\infty,\omega,1)\mbox{-$\cat$}}
\newcommand\ocatm{{(\infty,\omega)\mbox{-$\mcat$}}}
\newcommand\zocatm{(0,\omega)\mbox{-$\mcat$}}
\newcommand\zocat{(0,\omega)\mbox{-$\cat$}}
\DeclareMathOperator\zocatB{\zocat_B}
\newcommand\icat{(\infty,1)\mbox{-$\cat$}}
\newcommand\qcat{\mbox{Q$\cat$}}
\newcommand\ncat[1]{(\infty, #1)\mbox{-$\cat$}}
\newcommand\zncat[1]{(0, #1)\mbox{-$\cat$}}
\newcommand\igrd{\infty\mbox{-$\grd$}}

\DeclareMathOperator{\OperatorinfiniPsh}{Psh^\infty}
\DeclareMathOperator{\OperatorinfinitPsh}{tPsh^\infty}
\DeclareMathOperator{\OperatorPsh}{Psh}
\DeclareMathOperator{\OperatormPsh}{mPsh}
\DeclareMathOperator{\OperatortPsh}{tPsh}
\newcommand\iPsh[1]{\OperatorinfiniPsh({#1})}
\newcommand\tiPsh[1]{\OperatorinfinitPsh({#1})}
\newcommand\Psh[1]{\OperatorPsh({#1})}
\newcommand\ssetPsh[1]{\OperatorPsh_\Delta({#1})}
\newcommand\tPsh[1]{\OperatortPsh({#1})}
\newcommand\tPshM[1]{{\OperatortPsh}_M({#1})}
\newcommand\mPsh[1]{\OperatormPsh({#1})}
\newcommand\mPshM[1]{{\OperatormPsh}_M({#1})}

\DeclareMathOperator{\OperatorSeg}{Seg}
\DeclareMathOperator{\OperatortSeg}{tSeg}
\DeclareMathOperator{\OperatormSeg}{mSeg}
\newcommand\Seg{\OperatorSeg}
\newcommand\mSeg{\OperatormSeg}
\newcommand\stratSeg{\OperatortSeg}

\DeclareMathOperator{\Sset}{\Psh{\Delta}}
\newcommand{\mSset}{\mPsh{\Delta}}
\newcommand{\stratSset}{\tPsh{\Delta}}

\DeclareMathOperator{\U}{\mathbf{U}}
\DeclareMathOperator{\V}{\mathbf{V}}
\DeclareMathOperator{\Wcard}{\mathbf{W}}
\DeclareMathOperator{\Z}{\mathbf{Z}}

\newcommand{\ringpartial}{\mathring{\partial}}

\usepackage{fancyhdr}
\usepackage{titlesec}
\usepackage{textcase}

\title{\Huge{Theory and models of  $(\infty,\omega)$-categories}}
\author{Félix Loubaton}
\date{}
\begin{document}
\newgeometry{top=0cm, bottom=0cm, left=0cm, right=0cm}

\noindent\makebox[\textwidth]{\includegraphics[width=\paperwidth]{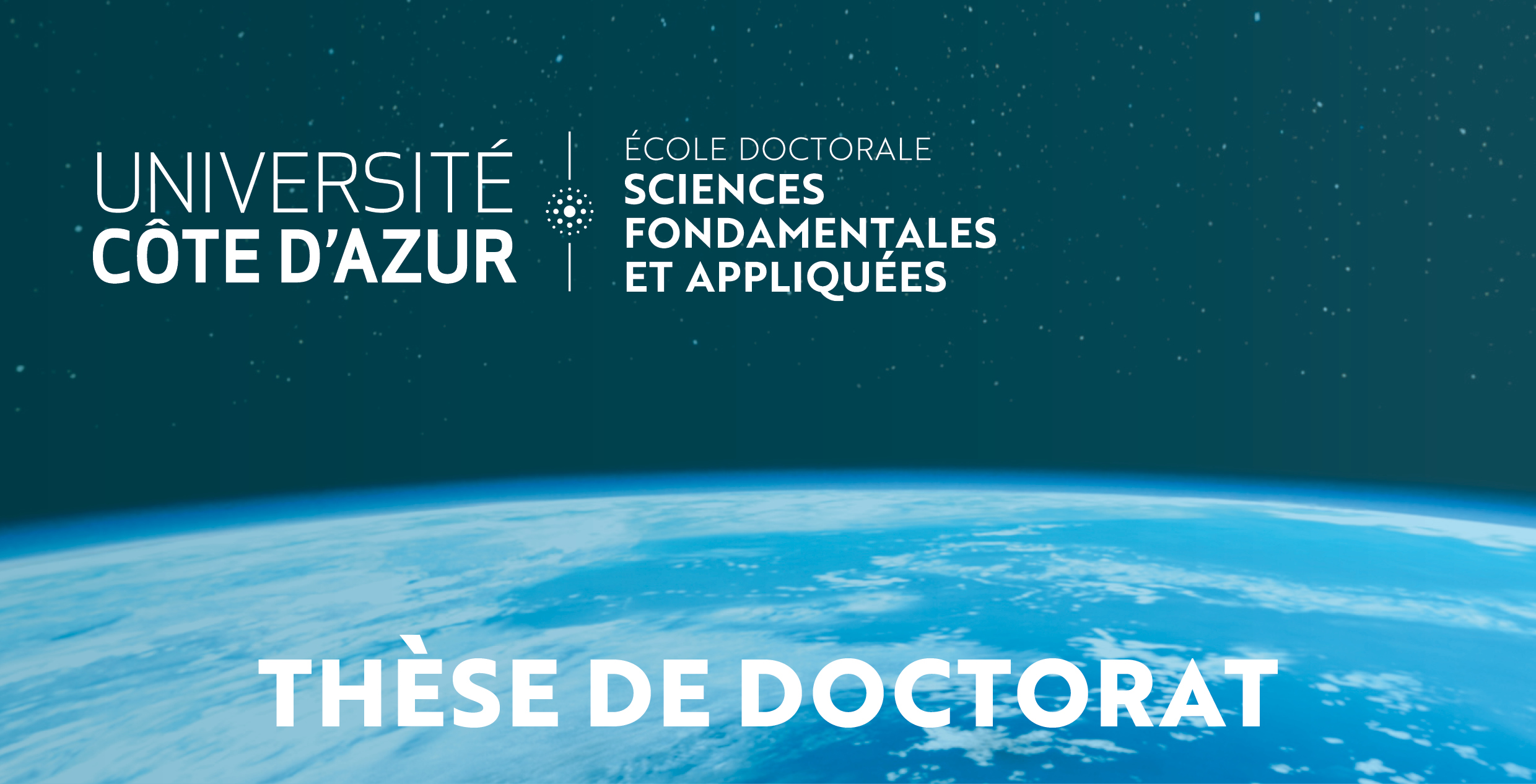}}

\vspace{1cm} 
\begin{center}
  {\Huge\bfseries Théorie et modèles des $(\infty,\omega)$-catégories} \\
  \vspace{1cm}
  {\Large\bfseries Félix Loubaton} \\
    \vspace{0.5cm}
  {\Large Laboratoire J.A. Dieudonné}
\end{center}

\hspace{2cm}

\begin{center}
\begin{minipage}[t]{0.30\linewidth}	
\raggedright
\textbf{Présentée en vue de l’obtention 
du grade de docteur en mathématiques
de l'Université Côte d’Azur.}\\
\textbf{Dirigée par :} Carlos Simpson\\
\textbf{Co-dirigé par :} Denis-Charles Cisinski\\
\textbf{Soutenue le} : 10 Octobre 2023
\end{minipage}
\hspace{0.1\linewidth}
\begin{minipage}[t]{0.40\linewidth}
\raggedright
\textbf{Devant le jury, composé de :}\\	
Dimitri Ara, Maitre de Conférences, Université d'Aix-Marseille\\
Clemens Berger, Professeur, Université Côte d'Azur \\
Yonatan Harpaz, Chargé de Recherche, Université Sorbonne Paris Nord\\
Georges Maltsiniotis, Directeur de Recherche Emerite, Université Paris Cité\\
Emily Riehl, Professeure, Université Johns Hopkins\\
Dominic Verity, Professeur Emerite, Université Macquarie\\
\end{minipage}
\end{center}
\restoregeometry

\clearpage
\thispagestyle{empty}

\cleardoublepage

\thispagestyle{empty}

\vspace{3cm}
	
\begin{center}
$$~$$  $$~$$
  {\Huge\bfseries Theory and models of $(\infty,\omega)$-categories} \\
  \vspace{1cm}
\end{center}
\vspace{1cm}
\vspace{1cm}

\paragraph{Jury :}

\paragraph{Rapporteurs}$~$\\
Yonatan Harpaz, Chargé de Recherche, Université Sorbonne Paris Nord\\
Dominic Verity, Professeur Emerite, Université Macquarie\\

\paragraph{Examinateur·rice·s}$~$\\
Dimitri Ara, Maitre de Conférences, Université d'Aix-Marseille\\
Clemens Berger, Professeur, Université Côte d'Azur \\
Georges Maltsiniotis, Directeur de Recherche Emerite, Université Paris Cité\\
Emily Riehl, Professeure, Université Johns Hopkins\\

\paragraph{Directeurs de Thèse}$~$\\
Carlos Simpson, Directeur de Recherche, Université d'Aix-Marseille\\
Denis-Charles Cisinski (co-directeur), Professeur, Université de Ratisbonne\\

\clearpage
\thispagestyle{empty}

\cleardoublepage

\pagestyle{empty}

%
%
%
%
%
%
%
%
%
%

\subsection*{Résumé}

La théorie des $(\infty,1)$-catégories est aujourd'hui un domaine de recherche prolifique avec des applications dans divers domaines. Ces dernières années ont également vu l'essor des $(\infty,n)$-catégories. Par exemple, les travaux de Gaitsgory et Rozenblyum (\cite{Gaitsgory_A_study_on_DAG}) en géométrie algébrique dérivée utilisent les $(\infty,2)$-catégories pour encoder le formalisme des six foncteurs. On peut aussi citer la théorie topologique des champs quantiques, qui utilise la notion de $(\infty,n)$-catégories dans la formalisation et la preuve de l'hypothèse du cobordisme (\cite{Baez_Higher-dimensional_algebra_and_topological_quantum_field_theory}, \cite{lurie_on_the_classification_of_topological_field_theories}, \cite{Grady_the_geometric_cobordism_hypothesis}, \cite{Calaque_a_note_on_the_category_of_cobordism}).

Il convient donc de développer une théorie des $(\infty,n)$-catégories. Cependant, pour réaliser une telle tâche, il est utile de manipuler les catégories $(\infty,k)$ pour $k \geq n$. Par exemple, la construction de Grothendieck, qui est toujours essentielle lorsque l'on travaille avec n'importe quel type de catégories, est une colimite lax dans la $(\infty,n+1)$-catégorie ambiante des $(\infty,n)$-catégories. Un deuxième exemple provient du produit tensoriel de Gray, qui est nécessaire pour encoder la notion de transformation lax, et donc pour définir la notion de colimite et limite lax. Le produit tensoriel de Gray ajoute la dimension des entrées, et les $(\infty,n)$-catégories ne sont donc pas stables sous ce bifoncteur. Une façon d'éviter tous ces problèmes liés à l'augmentation de la dimension est de se concentrer directement sur les $\io$-catégories, ce qui sera le parti pris de ce travail.

Dans la première partie de cette thèse, nous étudions les modèles des $\io$-catégories. Le résultat principal consiste à établir une équivalence de Quillen entre les $\Theta$-espaces complets de Segal et les ensembles compliciaux de Verity. Une des conséquences majeures de ce résultat est que lorsque l'on travaillera dans la $\iun$-catégorie correspondant à ces structures modèles, son lien avec les $\Theta$-espaces complets de Segal de Rezk nous permettra d'utiliser le langage globulaire, tandis que son lien avec les ensembles compliciaux nous donnera accès au produit tensoriel de Gray.

Dans la seconde partie de cette thèse, nous adapterons les constructions de la théorie classique des catégories au cas $\io$. Le chapitre \ref{chapter:the infini 1 categorory of io categories} est consacré à la théorie de base des $\io$-catégories. Le chapitre \ref{chapter:chapter the 1 category of marked categories} introduit la notion de {$\io$-catégories marquées} et étudie les {fibrations cartésiennes}. Le chapitre \ref{chapter:The io-category of small io-categories} est consacré à {la construction de Grothendieck}, à l'{univalence}, au {lemme de Yoneda}, et à d'autres constructions catégoriques standard.

\paragraph{Mots clés.} Ensembles compliciaux, $\io$-catégories, $(\infty,n)$-catégories, fibrations cartésiennes, construction de Grothendieck, univalence, lemme de Yoneda, (co)limites lax.
\newpage
\subsection*{Abstract}
The theory of $(\infty,1)$-categories is today a prolific area of research with applications in a variety of fields. Recent years have also seen the rise of $(\infty,n)$-categories. For example, the work of Gaitsgory and Rozenblyum (\cite{Gaitsgory_A_study_on_DAG}) in derived algebraic geometry uses $(\infty,2)$-categories to encode the formalism of six functors. Another example is topological quantum field theory, which uses the notion of $(\infty,n)$-categories in the formalization and proof of the cobordism hypothesis (\cite{Baez_Higher-dimensional_algebra_and_topological_quantum_field_theory}, \cite{lurie_on_the_classification_of_topological_field_theories}, \cite{Grady_the_geometric_cobordism_hypothesis}, \cite{Calaque_a_note_on_the_category_of_cobordism}).

A theory of $(\infty,n)$-categories therefore needs to be developed. However, to accomplish such a task, it is useful to manipulate $(\infty,k)$-categories for $k \geq n$. For example, the Grothendieck construction, which is always essential when working with any type of category, is a lax colimit in the ambient $(\infty,n+1)$-category of $(\infty,n)$-categories. A second example comes from the Gray tensor product, which is needed to encode the notion of lax transformation, and thus to define the notion of lax colimit and limit. The Gray tensor product adds the dimension of the inputs, and the $(\infty,n)$-categories are not stable under this bifonctor. One way of avoiding all these problems associated with the increasing of the dimension is to focus directly on $\io$-categories.

In the first part of this thesis, we study models of $\io$-categories. The main result is to establish a Quillen equivalence between Segal $\Theta$-complete spaces and Verity complicial sets. A major consequence of this result is that when working in the $\iun$-category corresponding to these model structures, its link with Rezk's $\Theta$-complete Segal spaces will allow us to use the globular language, while its link with complicial sets will give us access to Gray's tensor product.

In the second part of this thesis, we will adapt the constructions of classical category theory to the $\io$ case. Chapter \ref{chapter:the infini 1 categorory of io categories} is devoted to the basic theory of $\io$-categories. The chapter \ref{chapter:chapter the 1 category of marked categories} introduces the notion of {$\io$-marked categories} and studies {cartesian fibrations}. The chapter \ref{chapter:The io-category of small io-categories} is devoted to the {Grothendieck construction}, the {univalence}, the {Yoneda lemma}, and other standard categorical constructions.

\paragraph{Keywords.} Complicial sets, $\io$-categories, $(\infty,n)$-categories, cartesian fibrations,  Grothendieck construction, univalence, lemme de Yoneda, lax (co)limits.


\clearpage

%
%
%
%
%
%
%
%
%
%

\section*{Remerciements}
Tout d'abord, je tiens à remercier mes directeurs de thèse. Merci à vous, Carlos, d'avoir immédiatement répondu oui à ma demande d'inscrire ma thèse à Nice, sous votre direction, alors que nous ne nous connaissions pas. Merci de votre invitation enthousiasmante à participer à un séminaire à Miami avec vous et de nombreux mathématiciens dès la première année : le Covid ne l'a pas permis, cela reste un grand regret. Merci de m'avoir permis de travailler dans de si bonnes conditions à Nice, au sein du LJAD. Merci à toi, Denis-Charles. Travailler sous ta direction fut un honneur et un plaisir. Merci pour les visions, les intuitions et, je ne sais pas comment le dire autrement, la sagesse que tu as partagée avec moi pendant les dernières années. Aussi bien tes recherches que ta pratique de la recherche m'ont profondément inspiré.

I would like to thank Dominic Verity and Yonatan Harpaz for the honor of being my referees. Thank you, Yonatan Harpaz, for helping me to improve my text through your careful and rigorous reading. Thank you, Dominic Verity, for your extraordinarily detailed report, and I have no doubt that all your comments and advice will help me for the continuation of this work and beyond. As my work is built upon your work, having you as a referee has a real special meaning for me.

I would also like to thank Dimitri Ara, Clemens Berger, George Maltsiniotis, and Emily Riehl for agreeing to be part of my jury. 
I think you'll have understood that each of you has produced works that have deeply inspired me. It is therefore a great honor to defend my thesis in front of you.

\vspace{1cm}
Merci à Marnie Valentini pour la relecture attentive de ce manuscrit. Cela fait maintenant plusieurs années qu'elle est l'une de mes plus fidèles (re)lectrices.

\vspace{1cm}
Merci à Clara Salaun du LJAD et Birgit Tiefenbach de l'université de Ratisbonne pour avoir rendu l'organisation de tous les déplacements que j'ai effectués lors des dernières années si simple. Merci à Roland Ruelle et Jean-Marc Lacroix pour les multiples aides informatiques.
 Merci plus généralement à tous ceux qui travaillent dans l'administration et la gestion des laboratoires que j'ai fréquentés.

\vspace{1cm}
Cette thèse n'est pas seulement le résultat de trois années de travail, mais aussi l'aboutissement d'un long parcours scolaire qui a commencé il y a un peu plus d'une vingtaine d'années. Merci à tous les professeurs, du primaire au secondaire, que j'ai eus. Merci à Louis Ritter, Thomas Rouvier, Etienne Chardonnet, Sophie Durrieu et mes autres amis qui ont choisi cette profession. Nos métiers sont cousins, et de toutes les vocations liées à la transmission, la vôtre est sans doute l'une des plus importantes.

Ces trois dernières années ont été rythmées par les emissions des radios du service public. En particulier, merci à Nicolas Stoufflet (Le jeu des 1000 euros) qui sonnait le moment où je devais aller à la fac, merci à Fabienne Sintes (Le téléphone sonne) pour tous les dîners que nous avons partagés, et merci à Jacques Monin (Secrets d'info) pour tous les dimanches midi que nous avons passés ensemble. 

Enfin, dans une envolée que je m'autorise, je tiens à remercier plus généralement tous les agents du service public pour leur engagement.

\vspace{1cm}
Merci à tous les chercheurs qui m'ont accueilli parmi eux. Merci à George Maltsiniotis pour m'avoir initié à la recherche avec tant de patience, d'acharnement, et de générosité. C'est aussi grâce à toi que Marie se passionne maintenant pour ceux qui pratiquent les mathématiques, et je te remercie pour cela. Merci à Dimitri Ara pour sa prévenance et le climat de confiance qu'il a su instaurer entre nous, si précieux pour moi. Merci à Paul-André Melliès pour sa constante chaleur et curiosité. Merci à Simon Henry grâce à qui j'ai pu me remettre sur pied après presque un an de travail infructueux. Thanks to Viktoriya Ozornova for welcoming me to her team. I'm looking forward to my two years in Bonn!

\vspace{1cm}
Merci à tous ceux qui ont été doctorants en même temps que moi, et qui rendent moins solitaire ce travail qui l'est parfois tant. Merci à Hugo Pourcelot, qui fut la première personne que j'ai rencontrée par les mathématiques, et qui me prouva que ce monde était peuplé de belles personnes. Merci à Corentin Le Bars pour cette fin de thèse, ce moment si spécial, que nous avons partagé et durant lequel nous nous sommes aidés. Merci à Hugo Moeneclaey pour avoir été un si bon guide, puis compagnon, dans la découverte de l'homotopie et la logique. Merci à Nicolas Longuet Marx, pour sa proche et constante présence malgré 6419 kilomètres. Thanks to Niklas Kipp, Sebastian Wolf, Linda Hu, and all the people I met during my stays in Regensburg. The warmth of your welcome was more important than you might imagine.

Merci encore à Arnaud Vanhaecke, Léonard Pille-Schneider, Lucie Leszez, Nicolas Le Borgne, Dimitri Navarro, Pauline Rocca, Jonas Pentzien, Leo Hubert, et à tous les doctorants avec lesquels nous avons partagé nos peines et nos joies.

\vspace{1cm}
Merci à tous ceux qui rendent le départ de Nice plus triste. Merci à Victor Iwaniack pour sa gentillesse, sa franchise et pour avoir toujours eu un moment pour rêver math, bavarder ou prendre un verre. Je pense que nous avons été de fidèles alliés. Merci à Yash Chopra pour son exigence du doute et son amour du pathétique, ce sont pour moi de très belles qualités. Merci à Victor Pecanha Brittes pour ce bureau que nous avons partagé trop peu de temps. Merci encore à Christian Tayou Fotso, Alex Moriani, Antoine Commaret, Jérémie Marquès, Gustave Billon, Marc Monticelli et toutes les autres personnes gravitant autour du LJAD. A ceux qui arriveront après mon départ, parlez de moi comme celui qui apporta le café.

Merci à Simon Girel et Maëlle Bertier pour avoir été le plus proches de ce qui ressemblait à une famille, et merci à Violette pour l'accueil si chaleureux qu'elle a réservé à notre ami commun.

Merci à Sophie et Laurent d'avoir tant fait pour que je me sente chez moi à Nice.

\vspace{1cm}
Merci à tous mes amis proches et ma sœur qui participent pour beaucoup à mon bonheur, et qui sont une composante essentielle de ma vie. En cela, ils ont aidé à la réalisation de ce travail.

Merci à mon père qui sait mieux que quiconque s'occuper du concret, et qui en même temps, rêve peut-être encore plus que moi aux objets que je manipule. Merci à ma mère. Ce choix de vie doit beaucoup à l'admiration que j'ai pour toi, j'espère que tu t'en rends compte. Merci à vous deux pour votre soutien constant.

Enfin, merci Marie. Il n'est pas facile d'exprimer dans un texte public, à la mesure de ce que je pense, ma reconnaissance. Je me contenterais donc pudiquement de te remercier pour transformer tout ce qui aurait pu nous éloigner en des choses qui nous rapprochent, et bien sûr, encore plus, pour tout le reste.

%
%
%
%
%
%

\cleardoublepage
\pagestyle{plain}

\setcounter{page}{1}

\dominitoc
\tableofcontents

\clearpage
\pagestyle{fancy}
\fancyhf{}
\fancyhfoffset[RO,LE]{0.5cm}
\fancyhfoffset[LE,RO]{0.5cm}
\fancyhead[RO]{ }
\fancyhead[LE]{Introduction }
\fancyfoot[C]{\thepage}
\minitoc
\vspace{1cm}

\cleardoublepage
\phantomsection
\addcontentsline{toc}{part}{Introduction} 

\chapter*{Introduction}

%
%
%
%
%
%

The theory of $(\infty,1)$-categories is now a prolific field of research with applications in various domains. The past years have also witnessed the rise of $(\infty,2)$-categories. We will provide two reasons motivating the study of $(\infty,2)$-categories.

A first motivation comes from their applications in other domains. We think in particular of the work of Gaitsgory and Rozenblyum (\cite{Gaitsgory_A_study_on_DAG}) in derived algebraic geometry, where $(\infty,2)$-categories are an essential tool for encoding the six functor formalism.

A second motivation for considering $(\infty,2)$-categories arises from the theory of $\iun$-categories itself. Just as $1$-categories organize into a $2$-category, $\iun$-categories organize into an $(\infty,2)$-category. Working with this richer structure provides a powerful framework for developing formal category theory, as performed in \cite{Gray_Formal_category_theory} for the strict case and \cite{Riehl_element_of_infini_categories} for $\iun$-categories.

However, there is no reason to stop at dimension $2$. Let us once again mention two reasons for exploring $(\infty,n)$-categories for $n\in \Nb\cup\{\omega\}$.

Firstly, $(\infty,n)$-categories are already being used in other research fields, such as topological quantum field theory, where this notion is essential to the formalization and proof of the cobordism hypothesis (\cite{Baez_Higher-dimensional_algebra_and_topological_quantum_field_theory}, \cite{lurie_on_the_classification_of_topological_field_theories}, \cite{Grady_the_geometric_cobordism_hypothesis}, \cite{Calaque_a_note_on_the_category_of_cobordism}).

Secondly, even to understand the theory of $(\infty,n)$-categories, it is useful to manipulate $(\infty,k)$-categories for $k \geq n$. A first example is given by the fact that $(\infty,n)$-categories organize into an $(\infty,n+1)$-category, and this richer structure plays an important role in the theory of $(\infty,n)$-categories. For instance, the Grothendieck construction, which is always essential when working with any flavor of categories, is a lax colimit in the ambient $(\infty,n+1)$-category of $(\infty,n)$-categories. A second example arises from the Gray tensor product, which is a fundamental operation that arises when $n>1$. This operation is necessary to encode the notion of lax transformation, which leads to the concepts of lax colimits and limits. It is also worth noticing that it plays a crucial role in \cite{Gaitsgory_A_study_on_DAG}.
\begin{example*}[examples of some Gray tensor products]
We denote by $\Db_1$ the $1$-category generated by the $1$-graph
\[\begin{tikzcd}
	0 & 1
	\arrow[from=1-1, to=1-2]
\end{tikzcd}\]
and  by $\Db_2$ the $2$-category generated by the $2$-graph
\[\begin{tikzcd}
	0 & 1
	\arrow[""{name=0, anchor=center, inner sep=0}, curve={height=-12pt}, from=1-1, to=1-2]
	\arrow[""{name=1, anchor=center, inner sep=0}, curve={height=12pt}, from=1-1, to=1-2]
	\arrow[shorten <=3pt, shorten >=3pt, Rightarrow, from=0, to=1]
\end{tikzcd}\]
The Gray tensor product of $\Db_1$ with itself, denoted by $\Db_1\otimes\Db_1$, is the $2$-category generated by the diagram
\[\begin{tikzcd}
	00 & 01 \\
	10 & 11
	\arrow[from=1-1, to=2-1]
	\arrow[from=2-1, to=2-2]
	\arrow[from=1-1, to=1-2]
	\arrow[from=1-2, to=2-2]
	\arrow[shorten <=4pt, shorten >=4pt, Rightarrow, from=1-2, to=2-1]
\end{tikzcd}\]
The Gray tensor product of $\Db_2$ with $\Db_1$, denoted by $\Db_2\otimes\Db_1$, is the $3$-category generated by the diagram
\[\begin{tikzcd}
	00 & 01 & 00 & 01 \\
	10 & 11 & 10 & 11
	\arrow[from=1-1, to=1-2]
	\arrow[""{name=0, anchor=center, inner sep=0}, from=1-1, to=2-1]
	\arrow[from=2-1, to=2-2]
	\arrow[""{name=1, anchor=center, inner sep=0}, from=1-2, to=2-2]
	\arrow[shorten <=4pt, shorten >=4pt, Rightarrow, from=1-2, to=2-1]
	\arrow[""{name=2, anchor=center, inner sep=0}, from=1-3, to=2-3]
	\arrow[from=1-3, to=1-4]
	\arrow[""{name=3, anchor=center, inner sep=0}, from=1-4, to=2-4]
	\arrow[shorten <=4pt, shorten >=4pt, Rightarrow, from=1-4, to=2-3]
	\arrow[""{name=4, anchor=center, inner sep=0}, curve={height=30pt}, from=1-1, to=2-1]
	\arrow[from=2-3, to=2-4]
	\arrow[""{name=5, anchor=center, inner sep=0}, curve={height=-30pt}, from=1-4, to=2-4]
	\arrow["{ }"', shorten <=6pt, shorten >=6pt, Rightarrow, from=0, to=4]
	\arrow["{ }"', shorten <=6pt, shorten >=6pt, Rightarrow, from=5, to=3]
	\arrow[shift left=0.7, shorten <=6pt, shorten >=8pt, no head, from=1, to=2]
	\arrow[shift right=0.7, shorten <=6pt, shorten >=8pt, no head, from=1, to=2]
	\arrow[shorten <=6pt, shorten >=6pt, from=1, to=2]
\end{tikzcd}\]
\end{example*}
As we can see from these examples, the Gray tensor product adds the dimension of the inputs (in contrast to the cartesian product, which takes the maximum). Thus, $(\infty,n)$-categories are not stable under this operation. One can handle this by considering a truncated version of the Gray tensor product, but we believe that avoiding such violent operation will lead to a more natural understanding of the complex combinatorics it encodes.

One way to avoid all these issues related to the increasing of dimension is to directly focus on $(\infty,\omega)$-categories, which will be the standpoint of this thesis.

\phantomsection
\addcontentsline{toc}{section}{A brief definition of $(\gamma,n)$-categories for $n\in \Nb\cup\{\omega\}$}
\section*{A brief definition of $(\gamma,n)$-categories for $n\in \Nb\cup\{\omega\}$}

A \textit{globular set} is the data of a diagram of sets
\[\begin{tikzcd}
	{X_0} & {X_1} & {X_2} & {...}
	\arrow["{\pi_0^+}"', shift right=2, from=1-2, to=1-1]
	\arrow["{\pi_1^+}"', shift right=2, from=1-3, to=1-2]
	\arrow["{\pi_3^+}"', shift right=2, from=1-4, to=1-3]
	\arrow["{\pi_0^-}", shift left=2, from=1-2, to=1-1]
	\arrow["{\pi_1^-}", shift left=2, from=1-3, to=1-2]
	\arrow["{\pi_3^-}", shift left=2, from=1-4, to=1-3]
\end{tikzcd}\]
with the relations $\pi_{n-1}^{\epsilon}\pi_{n}^+ = \pi_n^{\epsilon} \pi_{n}^-$ for any $n>0$ and $\epsilon \in \{+,-\}$. We also denote by $\pi^{\epsilon}_k$ the map $X_n \to X_k$ for $k< n$ obtained by composing any string of arrows starting with $\pi^\epsilon_{k}$. An \textit{$\omega$-category} is a globular set $X$ together with
\begin{enumerate}
\item operations of \textit{compositions}
\[ X_n\times_{X_k} X_n\to X_n ~~~(0\leq k<n) \]
which associate to two $n$-cells $(x,y)$ verifying $\pi_k^+(x) = \pi_k^-(y)$, an $n$-cell $x\circ_ky$,
\item as well as \textit{units}
\[X_n\to X_{n+1}\]
which associate to an $n$-cell $x$, an $(n+1)$-cell $\Ib_x$, 
\end{enumerate}
and satisfying some associativity and unitaly axioms which will be expected by any reader familiar with $2$-categories (see \ref{para:def of omega cat} for the precise formulation of these axioms).
A \textit{morphism of $\omega$-categories} is a map of globular sets commuting with both operations. The category of $\omega$-categories is denoted by \textit{$\omegacat$}.

The category $\Theta$ of Joyal is the full subcategory of $\omegacat$ spanned by the \textit{globular sums}. These objects are precisely defined in paragraph \ref{para:les sommes glob}. Roughly speaking, globular sums are the $\omega$-categories obtained by "directed" gluing of \textit{globes}. In particular, globes are the easiest example of globular sums. Here are a few examples of globes and globular sums, where we identify the pasting diagrams with the $\omega$-categories they generate.

\begin{example*}[some examples of globes]
\label{exe:exemple 0}
\[\begin{tikzcd}
	{} & \bullet & {} & \bullet & \bullet & {} & \bullet & \bullet & {} & \bullet & \bullet \\
	\\
	{}
	\arrow[from=1-4, to=1-5]
	\arrow[""{name=0, anchor=center, inner sep=0}, curve={height=-24pt}, from=1-7, to=1-8]
	\arrow[""{name=1, anchor=center, inner sep=0}, curve={height=24pt}, from=1-7, to=1-8]
	\arrow["{\Db_0:=}"{description}, draw=none, from=1-1, to=1-2]
	\arrow["{\Db_1:=}"{description}, draw=none, from=1-3, to=1-4]
	\arrow["{\Db_2:=}"{description}, draw=none, from=1-6, to=1-7]
	\arrow[""{name=2, anchor=center, inner sep=0}, curve={height=-24pt}, from=1-10, to=1-11]
	\arrow[""{name=3, anchor=center, inner sep=0}, curve={height=24pt}, from=1-10, to=1-11]
	\arrow["{\Db_3:=}"{description}, draw=none, from=1-9, to=1-10]
	\arrow[shorten <=6pt, shorten >=6pt, Rightarrow, from=0, to=1]
	\arrow[""{name=4, anchor=center, inner sep=0}, shift left=3, shorten <=6pt, shorten >=6pt, Rightarrow, from=2, to=3]
	\arrow[""{name=5, anchor=center, inner sep=0}, shift right=3, shorten <=6pt, shorten >=6pt, Rightarrow, from=2, to=3]
	\arrow["\Rrightarrow"{description}, draw=none, from=5, to=4]
\end{tikzcd}\]
\end{example*}

\begin{example*}[some examples of globular sums]
\label{exe:exemple 1}
\[\begin{tikzcd}
	{} & \bullet & \bullet & \bullet & {} & \bullet & \bullet & {} & \bullet & \bullet & \bullet
	\arrow[""{name=0, anchor=center, inner sep=0}, curve={height=-24pt}, from=1-6, to=1-7]
	\arrow[""{name=1, anchor=center, inner sep=0}, curve={height=24pt}, from=1-6, to=1-7]
	\arrow[""{name=2, anchor=center, inner sep=0}, from=1-6, to=1-7]
	\arrow[from=1-2, to=1-3]
	\arrow[from=1-3, to=1-4]
	\arrow[""{name=3, anchor=center, inner sep=0}, from=1-9, to=1-10]
	\arrow[""{name=4, anchor=center, inner sep=0}, curve={height=24pt}, from=1-9, to=1-10]
	\arrow[""{name=5, anchor=center, inner sep=0}, curve={height=-24pt}, from=1-9, to=1-10]
	\arrow[""{name=6, anchor=center, inner sep=0}, curve={height=-24pt}, from=1-10, to=1-11]
	\arrow[""{name=7, anchor=center, inner sep=0}, curve={height=24pt}, from=1-10, to=1-11]
	\arrow["{a_0:=}"{description}, draw=none, from=1-1, to=1-2]
	\arrow["{a_1:=}"{description}, draw=none, from=1-5, to=1-6]
	\arrow["{a_2:=}"{description}, draw=none, from=1-8, to=1-9]
	\arrow[shorten <=3pt, shorten >=3pt, Rightarrow, from=0, to=2]
	\arrow[shorten <=3pt, shorten >=3pt, Rightarrow, from=2, to=1]
	\arrow[""{name=8, anchor=center, inner sep=0}, shift left=3, shorten <=3pt, shorten >=5pt, Rightarrow, from=3, to=4]
	\arrow[""{name=9, anchor=center, inner sep=0}, shift right=3, shorten <=3pt, shorten >=5pt, Rightarrow, from=3, to=4]
	\arrow[shorten <=3pt, shorten >=3pt, Rightarrow, from=5, to=3]
	\arrow[shorten <=6pt, shorten >=6pt, Rightarrow, from=6, to=7]
	\arrow["\Rrightarrow"{description}, shift left=1, draw=none, from=9, to=8]
\end{tikzcd}\]
\end{example*}
\begin{example*}[some examples of morphisms between globular sums]
\[\begin{tikzcd}[column sep=0.367in]
	\bullet && \bullet && \bullet & \bullet & \bullet & \bullet & \bullet && \bullet & \bullet \\
	\\
	\\
	\bullet & \bullet & \bullet && \bullet & \bullet & \bullet & \bullet & \bullet && \bullet & \bullet
	\arrow[from=4-11, to=4-12]
	\arrow[""{name=0, anchor=center, inner sep=0}, curve={height=-24pt}, from=4-6, to=4-7]
	\arrow[""{name=1, anchor=center, inner sep=0}, curve={height=24pt}, from=4-6, to=4-7]
	\arrow[""{name=2, anchor=center, inner sep=0}, curve={height=24pt}, from=4-8, to=4-9]
	\arrow[""{name=3, anchor=center, inner sep=0}, curve={height=-24pt}, from=4-8, to=4-9]
	\arrow["{f_3}", shorten <=19pt, shorten >=19pt, maps to, from=4-9, to=4-11]
	\arrow[""{name=4, anchor=center, inner sep=0}, from=4-5, to=4-6]
	\arrow["{f_2}", shorten <=19pt, shorten >=19pt, maps to, from=4-3, to=4-5]
	\arrow[""{name=5, anchor=center, inner sep=0}, curve={height=-24pt}, from=1-8, to=1-9]
	\arrow[""{name=6, anchor=center, inner sep=0}, curve={height=24pt}, from=1-8, to=1-9]
	\arrow[""{name=7, anchor=center, inner sep=0}, curve={height=-24pt}, from=1-11, to=1-12]
	\arrow[""{name=8, anchor=center, inner sep=0}, curve={height=24pt}, from=1-11, to=1-12]
	\arrow[""{name=9, anchor=center, inner sep=0}, from=1-11, to=1-12]
	\arrow["{f_1}", shorten <=19pt, shorten >=19pt, maps to, from=1-9, to=1-11]
	\arrow[curve={height=-24pt}, from=4-2, to=4-3]
	\arrow[curve={height=24pt}, from=4-1, to=4-2]
	\arrow[from=1-1, to=1-3]
	\arrow["{f_0}", shorten <=19pt, shorten >=19pt, maps to, from=1-3, to=1-5]
	\arrow[from=1-5, to=1-6]
	\arrow[from=1-6, to=1-7]
	\arrow[""{name=10, anchor=center, inner sep=0}, curve={height=-24pt}, from=4-5, to=4-6]
	\arrow[""{name=11, anchor=center, inner sep=0}, curve={height=24pt}, from=4-5, to=4-6]
	\arrow[shorten <=6pt, shorten >=6pt, Rightarrow, from=0, to=1]
	\arrow[shorten <=6pt, shorten >=6pt, Rightarrow, from=3, to=2]
	\arrow[shorten <=3pt, shorten >=3pt, Rightarrow, from=7, to=9]
	\arrow[shorten <=3pt, shorten >=3pt, Rightarrow, from=9, to=8]
	\arrow[shorten <=6pt, shorten >=6pt, Rightarrow, from=5, to=6]
	\arrow[shorten <=3pt, shorten >=3pt, Rightarrow, from=10, to=4]
	\arrow[""{name=12, anchor=center, inner sep=0}, shift left=3, shorten <=3pt, shorten >=5pt, Rightarrow, from=4, to=11]
	\arrow[""{name=13, anchor=center, inner sep=0}, shift right=3, shorten <=3pt, shorten >=5pt, Rightarrow, from=4, to=11]
	\arrow["\Rrightarrow"{description}, shift left=1, shorten <=2pt, shorten >=2pt, from=13, to=12]
\end{tikzcd}\]
\end{example*}

 For $n\in \Nb\cup \{\omega\}$, we define $\Theta_n$ as the full subcategory of $\Theta$ whose objects correspond to $n$-categories. In particular, $\Theta_0$ is the terminal category, $\Theta_1$ is $\Delta$, and $\Theta_\omega$ is $\Theta$.

Let $\gamma$ be a complete $\iun$-category and $n\in \Nb\cup \{\omega\}$. A \textit{$(\gamma,n)$-category} is a functor $\Theta_n^{op}\to \gamma$ that satisfies the \textit{Segal conditions} and \textit{completeness conditions}. We denote by $(\gamma,n)\mbox{-$\cat$}$ the $\iun$-category of $(\gamma,n)$-categories. Since we have not given a precise definition of $\Theta$, we cannot explicitly state these conditions, but we will try to explain their essence.

\textbf{Segal conditions.} As the diagrams given in the examples suggest, every globular sum is a colimit of globes. For instance, $a_2$ is the colimit of the following diagram
\[\begin{tikzcd}
	{\Db_2} \\
	{\Db_1} & {\Db_0} & {\Db_2} \\
	{\Db_3}
	\arrow["{i_1^+}", from=2-1, to=1-1]
	\arrow["{i_1^-}"', from=2-1, to=3-1]
	\arrow["{i_0^+}"', from=2-2, to=2-1]
	\arrow["{i_0^-}", from=2-2, to=2-3]
\end{tikzcd}\]
A functor $X:\Theta_n^{op}\to \gamma$ satisfies the \textit{Segal conditions} if it sends these colimits to limits. For instance, the presheaf $X$ must send $a_2$ to the limit of the diagram 
\[\begin{tikzcd}
	{X(\Db_2)} \\
	{X(\Db_1)} & {X(\Db_0)} & {X(\Db_2)} \\
	{X(\Db_3)}
	\arrow["{\pi_1^+}"', from=1-1, to=2-1]
	\arrow["{\pi_1^-}", from=3-1, to=2-1]
	\arrow["{\pi_0^+}", from=2-1, to=2-2]
	\arrow["{\pi_0^-}"', from=2-3, to=2-2]
\end{tikzcd}\]
The morphisms $X(f_0)$ and $X(f_1)$ can then be interpreted as compositions and the morphism $X(f_3)$ as a unit.

\textbf{Completeness conditions.} Let $X:\Theta_n^{op}\to \gamma$ be a functor satisfying the Segal conditions. Given an integer $k\leq n$, we have two notions of equivalence on the $k$-cells of $X$, i.e. the morphisms $1\to X(\Db_k)$. The first comes from the canonical equivalence provided by the $\infty$-groupoid $\Hom(1, X(\Db_k))$, and the second is more categorical and identifies \textit{isomorphic} elements, i.e. $k$-cells $a,b$ such that there exists $(k+1)$-cells $f:a\to b$, $g:b\to a$ and equivalences
$$g\circ_k f\sim id_a~~~~~~~~~\mbox{and}~~~~~~~~~f\circ_k g\sim id_b.$$
 The presheaf $X$ satisfies the completeness condition if these two notions of equivalence coincide. Thus, \textit{groupoids}, i.e., $(\gamma,n)$-categories in which all $k$-cells are equivalent to the identity of their source (or target), correspond to constant functors $\Theta^{op}\to \gamma$. The datum of the $(\infty,1)$-category $\gamma$ can be understood as a \textit{choice of a notion of groupoid}.

\paragraph{}
When $\gamma$ is the category of sets, the $(\gamma,n)$-categories will simply be denoted as $(0,n)$-categories, and when $\gamma$ is the $\iun$-category of spaces, they will be denoted as $(\infty,n)$-categories.

 For instance, $(0,\omega)$-categories correspond to $\Theta$-sets satisfying the Segal and completeness conditions. The first one induce an inclusion of $(0,\omega)$-categories into $\omega$-categories and the latter forces isomorphisms to be identities. The $(0,\omega)$-categories then correspond to \textit{Gaunt $\omega$-categories}.

Although this concept is not studied in the present thesis, it is worth noticing that one could define $(k,n)$-categories for any $k\in \Nb$. In this case, we would consider the $(\gamma,n)$-categories with $\gamma$ being the $\iun$-category of $k$-truncated $\infty$-groupoids. This notation is compatible with the one given in \cite{Rezk_a_cartesian_of_weak_n_categories} when $k\geq n$ but it also allows to  give meaning to $(k,n)$-categories for $k<n$.

\paragraph{}
As stated earlier, this work is devoted to the concept of $\io$-categories, which corresponds to the case where $\gamma$ is the category of spaces. This notion is sometimes considered ambiguous. Indeed, Schommer-Pries and Rezk have independently argued (\cite{134099}) that there should be more than one notion of $(\infty,\omega)$-categories. The one we use here is commonly referred to as \textit{the inductive one}, in the sense that $\ocat$ is identified with the limit of the sequence:
$$\ncat{0}\xleftarrow{\tau_0} \ncat{1}\xleftarrow{}... \leftarrow\ncat{n} \xleftarrow{\tau_{n}}\ncat{n+1}\xleftarrow{}...$$
where the functors $\tau_n$ "forget" the cells of dimension $n$. For a more detailed discussion in the (semi-)strict case, we refer to \cite{Henry_an_inductive_model_structure_for_infini_categories}.

\vspace{1cm}
\phantomsection
\addcontentsline{toc}{section}{Overview of the thesis}
\section*{Overview of the thesis}

This thesis is divided into two parts which can be read independently. However, each of them uses results from the preliminary section.

\phantomsection
\addcontentsline{toc}{subsection}{Preliminaries}
\subsection*{Preliminaries}

\paragraph{Chapter \ref{chapter:The category of zocategories}.}
The first section is devoted to the definition of $\zo$-categories and of the category $\Theta$ of Joyal. We also show that the category $\Theta$ presents the category of $\zo$-categories, and we also exhibit an other presentation of this category (corollary \ref{cor:changing theta}).

The second section begins with a review of Steiner theory, which is an extremely useful tool for providing concise and computational descriptions of $\zo$-categories. Following Ara and Maltsiniotis, we employ this theory to define the Gray tensor product, denoted by $\otimes$, in $\zo$-categories. We then introduce the Gray operations, starting with the Gray cylinder $\uvar\otimes[1]$ which is the Gray tensor product with the directed interval $[1]:=0\to 1$. Then, we have the Gray cone and Gray $\circ$-cone, denoted by $\uvar\star 1$ and $1\costar \uvar$, that send an $\zo$-category $C$ onto the following pushouts:
\[\begin{tikzcd}
	{C\otimes\{1\}} & {C\otimes[1]} && {C\otimes\{0\}} & {C\otimes[1]} \\
	1 & {C\star 1} && 1 & {1\costar C}
	\arrow[from=1-5, to=2-5]
	\arrow[from=1-4, to=2-4]
	\arrow[from=2-4, to=2-5]
	\arrow[from=1-4, to=1-5]
	\arrow[from=1-2, to=2-2]
	\arrow[from=2-1, to=2-2]
	\arrow[from=1-1, to=2-1]
	\arrow[from=1-1, to=1-2]
	\arrow["\lrcorner"{anchor=center, pos=0.125, rotate=180}, draw=none, from=2-2, to=1-1]
	\arrow["\lrcorner"{anchor=center, pos=0.125, rotate=180}, draw=none, from=2-5, to=1-4]
\end{tikzcd}\]

We also present a formula that illustrates the interaction between the suspension and the Gray cylinder. As this formula plays a crucial role in both Part I and Part II, we provide its intuition at this stage.

 If $A$ is any $\zo$-category, the suspension of $A$, denoted by $[A,1]$, is the $\zo$-category having two objects - denoted by $0$ and $1$- and such that 
$$\Hom_{[A,1]}(0,1) := A,~~~\Hom_{[A,1]}(1,0) := \emptyset,~~~\Hom_{[A,1]}(0,0)=\Hom_{[A,1]}(1,1):=\{id\}.$$
We also define $[1]\vee[A,1]$ as the gluing of $[1]$ and $[A,1]$ along the $0$-target of $[1]$ and the $0$-source of $[A,1]$. We define similarly $[A,1]\vee[1]$.
These two objects come along with \textit{whiskerings}:
$$\triangledown:[A,1]\to [1]\vee [A,1] ~~~~\mbox{and}~~~~ \triangledown:[A,1] \to [A,1]\vee [1]$$ 
that preserve the extremal points.

The $\zo$-category $[1]\otimes [1]$ is induced by the diagram:
\[\begin{tikzcd}
	00 & 01 \\
	10 & 11
	\arrow[from=1-1, to=2-1]
	\arrow[from=2-1, to=2-2]
	\arrow[from=1-1, to=1-2]
	\arrow[from=1-2, to=2-2]
	\arrow[shorten <=4pt, shorten >=4pt, Rightarrow, from=1-2, to=2-1]
\end{tikzcd}\]
and is then equal to the colimit of the following diagram: 
$$[1]\vee [1]\xleftarrow{\triangledown} [1]\hookrightarrow [[1],1]\hookleftarrow[1]\xrightarrow{\triangledown } [1]\vee [1].$$
The $\zo$-category $ [[1],1]\otimes [1]$ is induced by the diagram:
\[\begin{tikzcd}
	00 & 01 & 00 & 01 \\
	10 & 11 & 10 & 11
	\arrow[from=1-1, to=1-2]
	\arrow[""{name=0, anchor=center, inner sep=0}, from=1-1, to=2-1]
	\arrow[from=2-1, to=2-2]
	\arrow[""{name=1, anchor=center, inner sep=0}, from=1-2, to=2-2]
	\arrow[shorten <=4pt, shorten >=4pt, Rightarrow, from=1-2, to=2-1]
	\arrow[""{name=2, anchor=center, inner sep=0}, from=1-3, to=2-3]
	\arrow[from=1-3, to=1-4]
	\arrow[""{name=3, anchor=center, inner sep=0}, from=1-4, to=2-4]
	\arrow[shorten <=4pt, shorten >=4pt, Rightarrow, from=1-4, to=2-3]
	\arrow[""{name=4, anchor=center, inner sep=0}, curve={height=30pt}, from=1-1, to=2-1]
	\arrow[from=2-3, to=2-4]
	\arrow[""{name=5, anchor=center, inner sep=0}, curve={height=-30pt}, from=1-4, to=2-4]
	\arrow["{ }"', shorten <=6pt, shorten >=6pt, Rightarrow, from=0, to=4]
	\arrow["{ }"', shorten <=6pt, shorten >=6pt, Rightarrow, from=5, to=3]
	\arrow[shift left=0.7, shorten <=6pt, shorten >=8pt, no head, from=1, to=2]
	\arrow[shift right=0.7, shorten <=6pt, shorten >=8pt, no head, from=1, to=2]
	\arrow[shorten <=6pt, shorten >=6pt, from=1, to=2]
\end{tikzcd}\]
and is then equal to the colimit of the following diagram: 
 $$[1]\vee[[1],1]\xleftarrow{\triangledown} [[1]\otimes\{0\},1]\hookrightarrow[[1]\otimes[1],1]\hookleftarrow [[1]\otimes\{1\},1]\xrightarrow{\triangledown}[[1],1]\vee[1]$$
We prove a formula that combines these two examples:

\begin{itheorem}[\ref{theo:appendice formula for otimes}]
In the category of $\zo$-categories, there exists an isomorphism, natural in $A$, between $[A,1]\otimes[1]$ and the colimit of the following diagram
\[\begin{tikzcd}
	{[1]\vee[A,1]} & {[A\otimes\{0\},1]} & { [A\otimes[1],1]} & {[A\otimes\{1\},1]} & {[A,1]\vee[1]}
	\arrow["\triangledown"', from=1-2, to=1-1]
	\arrow[from=1-4, to=1-3]
	\arrow["\triangledown", from=1-4, to=1-5]
	\arrow[from=1-2, to=1-3]
\end{tikzcd}\]
\end{itheorem} 

We also provide similar formulas for the \textit{Gray cone} and the \textit{Gray $\circ$-cone}.
\begin{itheorem}[\ref{theo:appendice formula for star}]
There is a natural identification between $1\costar [A,1]$ and the colimit of the following diagram
\[\begin{tikzcd}
	{[1]\vee[A,1]} & {[A,1]} & { [A\star 1,1]}
	\arrow["\triangledown"', from=1-2, to=1-1]
	\arrow[from=1-2, to=1-3]
\end{tikzcd}\]
There is a natural identification between $[A,1]\star 1$ and the colimit of the following diagram
\[\begin{tikzcd}
	{ [1\costar A,1]} & {[A,1]} & {[A,1]\vee[1]}
	\arrow[from=1-2, to=1-1]
	\arrow["\triangledown", from=1-2, to=1-3]
\end{tikzcd}\]
\end{itheorem}

\phantomsection
\addcontentsline{toc}{subsection}{On the side of models}
\subsection*{On the side of models}

Following the terminology of Barwick and Schommer-Pries (\cite{Barwick_on_the_unicity_of_the_theory_of_higher_categories}), we call \textit{model of $(\infty,n)$-categories} any model category whose corresponding $(\infty, 1)$-category is $\ncat{n}$.

With the definition of $(\infty,n)$-categories given above, we have a natural model for the $\iun$-category $\ncat{n}$, given by Rezk's complete Segal $\Theta_n$-spaces, i.e. space valued presheaves on $\Theta_n$ satisfying the (homotopical) Segal conditions and (homotopical) completeness conditions. However, there are many other models, see for instance \cite{Ara_Higher_quasi_cat}, \cite{Bergner_Comparison_of_model_of_infini_n_categories}, \cite{Bergner_Comparison_of_model_for_infini_n_categories_II}, \cite{Bergner_reedy_category_and_the_theta_construction} (we refer to \cite{Barwick_on_the_unicity_of_the_theory_of_higher_categories}
for a comprehensive presentation of these models and their equivalences). For example, one can mention $n$-fold Segal spaces and Simpson's and Tamsamani's Segal $n$-categories among others.

It was conjectured (\cite{Street_algebra_of_orianted_simplexes}, \cite{Verity_a_complicial_compendium}, \cite{Barwick_on_the_unicity_of_the_theory_of_higher_categories}) that Verity's $n$-complicial sets were also a model of $(\infty,n)$-categories. This would imply that Campion-Kapulkin-Maehara's $n$-comical sets also are, as they are shown to be Quillen equivalent to $n$-complicial sets in \cite{Doherty_Equivalence_of_cubical_and_simplicial_approaches}. In the second chapter, we will give a positive answer to this conjecture (theorem \ref{theo:letheo}).

One of the major consequences of this result is to endow $\ocat$ with a monoidal product called the \textit{Gray tensor product}. This operation will play a crucial role in the second part of this thesis, which is dedicated to the theory of $\io$-categories.

\vspace{1cm}
The two main models we work with are Verity's complical sets (definition \ref{defi:complicial set}) and (a slight modification of) Segal $A$-precategories (defined in paragraph \ref{para:def sega a cat}) as developed by Simpson (\cite{Simpson_Homotopy_theory_of_higher_categories}). In the complical model, we will make crucial use of the strictification results of Ozornova and Rovelli (\cite{Ozornova_Fundamental_pushouts_of_n_complical_set}, \cite{Ozornova_a_quillen_adjunction_between_globular_and_complicial}). 

\paragraph{Chapter \ref{chapter:Studies of the complicial model}.}
One of the benefits of complicial sets is that they admit a simple definition of the Gray tensor product. Being strongly linked to $\zo$-categories by the Street nerve, they are also a privileged framework for stating and proving strictification results, as done in \cite{Ozornova_Fundamental_pushouts_of_n_complical_set}, \cite{Gagna_Nerves_and_cones_of_free_loop_free_omega_categories}, \cite{Ozornova_a_quillen_adjunction_between_globular_and_complicial} and \cite{Maehara_oriental_as_free_weak_omega_categories}. 
However, they do not interact \textit{a priori} well with the globular language. The goal of this chapter is to show that, with some computation, it is possible to have a globular point of view in this model. 

The first section is a recollection of usual results and definitions about complicial sets. 
In the second section, we aim to prove an analogue of the formula given in \ref{theo:appendice formula for otimes} to the complicial setting.
We also have a suspension in this category, which is denoted by $X\mapsto \Sigma X$. Objects $[1]\fwedge \Sigma X$ and $\Sigma X\fwedge [1]$ are defined in \ref{subsection:wedge}, but for now, we can suppose that they are fibrant replacements of respectively $[1]\coprod_{[0]}\Sigma X$ and $\Sigma X\coprod_{[0]}[1]$.
They come along with morphisms that are analogue to whiskerings, and that we also note by $\triangledown$: 
$$\triangledown:\Sigma X\to [1]\fwedge\Sigma X ~~~~\mbox{and}~~~~ \triangledown:\Sigma X\to\Sigma X\fwedge [1].$$ 
We then show the following theorem:
\begin{itheorem}[\ref{theo:interval_first_formula}]
There exists a zigzag of acyclic cofibrations, natural in $X$, between $(\Sigma X)\otimes [1]$ and the colimit of the following diagram:
 $$\Sigma X\fwedge [1]\xleftarrow{\triangledown} \Sigma (X\otimes\{0\}) \hookrightarrow \Sigma (X\otimes[1])\hookleftarrow \Sigma (X\otimes\{1\})\xrightarrow{\triangledown} [1]\fwedge \Sigma X.$$
\end{itheorem}
We also provide similar formulas for the \textit{Gray cone} and Gray \textit{$\circ$-cone}:
\begin{itheorem}[\ref{theo:cyl_formula}]
There exists a zigzag of acyclic cofibrations, natural in $X$, between $\Sigma X \star[0]$ and the colimit of the following diagram: 
$$ \Sigma X\fwedge [1]\leftarrow \Sigma X\to \Sigma([0]\costar X).$$
There exists a zigzag of acyclic cofibrations, natural in $X$, between  $[0]\costar \Sigma X$ and the colimit of the following diagram: 
$$\Sigma(X\star[0]) \leftarrow \Sigma X\to [1]\fwedge\Sigma X.$$
\end{itheorem}

The third section uses this formula and the strictification result of Gagna, Ozornova and	 Rovelli (\cite{Gagna_Nerves_and_cones_of_free_loop_free_omega_categories}) to demonstrate a criterion for detecting autoequivalences of complicial sets by their behavior on globes.
Indeed, in section \ref{section:Globular equivalences}, by iterating the suspension, we construct a globular object: 
\[\begin{tikzcd}
	{\Db_0} & {\Db_1} & {\Db_2} & {...}
	\arrow["{i_0^+}", shift left=2, from=1-1, to=1-2]
	\arrow["{i_1^+}", shift left=2, from=1-2, to=1-3]
	\arrow["{i_3^+}", shift left=2, from=1-3, to=1-4]
	\arrow["{i_0^-}"', shift right=2, from=1-1, to=1-2]
	\arrow["{i_1^-}"', shift right=2, from=1-2, to=1-3]
	\arrow["{i_3^-}"', shift right=2, from=1-3, to=1-4]
\end{tikzcd}\]
\begin{itheorem}[\ref{theo:criterion_to_be_linked_to_identity}]
Let $i$ be a left Quillen endofunctor for the model category for complicial sets. Suppose that there exists a zigzag of weakly invertible natural transformations:
$$i(\Db_{\uvar}) \leftrightsquigarrow \Db_{\uvar}.$$
Then, there exists a zigzag of weakly invertible natural transformations between $i$ and $id$.
\end{itheorem} 
Proposition 15.10 of \cite{Barwick_on_the_unicity_of_the_theory_of_higher_categories} provides a similar result for models of $(\infty,n)$-categories.

\paragraph{Chapter \ref{chapter:complicial set as a model of io categories}.}
Results of Bergner, Gagna, Harpaz, Lanari, Lurie and Rezk (\cite{Bergner_Comparison_of_model_of_infini_n_categories},\cite{Bergner_Comparison_of_model_for_infini_n_categories_II}, \cite{Rezk_a_cartesian_of_weak_n_categories}, \cite{Lurie_Htt},\cite{Lurie_goodwillie_calculus}, \cite{Gagna_on_the_equivallence_of_all_model_for_infini2_cat}) imply that $2$-complicial sets are a model of $(\infty,2)$-categories (see \cite{Gagna_on_the_equivallence_of_all_model_for_infini2_cat} to understand how to use all this source to obtained the desired result and \cite{Bergner_explicit_comparaison_bt_theta_2_space_and_2_complicial_set} for a direct comparison between complete Segal $\Theta_2$-spaces and $2$-complicial sets).
The purpose of this chapter is to generalize this result to any $n\in \Nb\cup\{\omega\}$.

To this extend, we first address the more general problem of finding sufficient conditions on a model category $A$ to build a \textit{Gray cylinder} $C\mapsto I\otimes C$ and a \textit{Gray cone} $C\mapsto e\star C$ on Segal precategories enriched in $A$. These two operations should be linked by the following homotopy cocartesian square
\[\begin{tikzcd}
	{\{0\}\otimes C} & {I\otimes C} \\
	e & {e\star C}
	\arrow[from=1-2, to=2-2]
	\arrow[from=1-1, to=2-1]
	\arrow[from=2-1, to=2-2]
	\arrow[from=1-1, to=1-2]
\end{tikzcd}\]
where $e$ is the terminal object. The conditions that $A$ has to	 fulfill are encapsulated in the notion of \textit{Gray module} (paragraph \ref{para:Gray module}). Thanks to the Gray cylinder and cone, we can show the following theorem:

\begin{itheorem}[\ref{theo:Quillen adjunction}]
If $A$ is a Gray module, there is a Quillen adjunction between the Ozornova-Rovelli model structure for $\omega$-complicial sets on stratified simplicial sets and stratified Segal precategories enriched in $A$ where the left adjoint sends $[n]$ to $e\star e\star ... \star e\star \emptyset$
\end{itheorem} 

We will apply this theorem to the case where $A$ is the category of stratified simplicial sets endowed with the model structure for $\omega$-complicial sets, and after tedious work, we get
\begin{itheorem}[\ref{theo:letheo}]
Let $n\in \Nb$.
The model structure for $n$-complicial sets is a model of $(\infty,n)$-categories.
\end{itheorem}
As a corollary we have
\begin{itheorem}[\ref{theo:lecorozo}]
The adjunction between the model structure for complete Segal $\Theta$-spaces and $\omega$-complicial set constructed in \cite{Ozornova_a_quillen_adjunction_between_globular_and_complicial} is a Quillen equivalence.
\end{itheorem}

\phantomsection
\addcontentsline{toc}{subsection}{On the side of theory} 
\subsection*{On the side of theory}

In the second part of this thesis, we will adapt the constructions of classical category theory to the case $\io$. 
In this part, we will freely use the language of $\iun$-categories\footnote{ As there are currently several directions for the formalization of the language of $\iun$-categories (\cite{Riehl_element_of_infini_categories}, \cite{Riehl_A_type_theory_for_synthetic_-categories}, \cite{North_Towards_a_directed_homotopy_type_theory}, \cite{Cisinski-Univalent-Directed-Type-theory}), talking about "the" language of $\iun$-categories may be confusing.

In such case, the reader may consider that we are working within the quasi-category $\qcat$ of $\Tb$-small quasi-categories for $\Tb$ a Grothendieck universe. This quasi-category may be obtained either using the coherent nerve as described in \cite[chapter 3]{Lurie_Htt}, or by considering it as the codomain of the universal cocartesian fibration with $\Tb$-small fibers as done in \cite{Cisinski_The_universal_coCartesian_fibration}. In both cases, the straightening/unstraightening correspondence provides a morphism
$$\N(\Sset_{\Tb})\to \qcat$$
that exhibits $\qcat$ as the quasi-categorical localization of $\N(\Sset_{\Tb})$ with respect to the weak equivalences of the Joyal's model structure (\cite[theorem 8.13]{Cisinski_The_universal_coCartesian_fibration}). 

The constructions we use to build new objects - (co)limits of functor between quasi-categories, quasi-categories of functor, localization of quasi-categories, sub maximal Kan complex, full sub quasi-category, adjunction, left and right Kan extension, Yoneda lemma - are well documented in the Joyal model structure (see \cite{Lurie_Htt} or \cite{Cisinski_Higher_categories_and_homotopical_algebra})
, and therefore have direct incarnation in the quasi-category $\qcat$. }.

Chapter \ref{chapter:the infini 1 categorory of io categories} is devoted to the basic theory of $\io$-categories. Chapter \ref{chapter:chapter the 1 category of marked categories} introduces the notion of \textit{marked $\io$-categories} and studies \textit{left Cartesian fibrations}. Chapter \ref{chapter:The io-category of small io-categories} is dedicated to the \textit{Grothendieck construction}, \textit{univalence}, the \textit{Yoneda lemma}, and other standard categorical constructions.

Several of these results, or their analogues in the $(\infty,n)$ setting for some integer $n$, are already present in the literature. The case $n=1$, i.e. that of $(\infty,1)$-category theory, is now a prolific research field, and it would be impossible to list all the authors who have contributed to it. Nonetheless, we would like to mention Joyal for his pioneering work (\cite{Joyal_Quasi-categories_and_Kan_complexes}), Lurie for his major contribution (\cite{Lurie_Htt}), and Cisinski (\cite{Cisinski_Higher_categories_and_homotopical_algebra}) because his approach has deeply inspired the present work.

For the case $n=2$, the Grothendieck construction as well as lax limits and colimits have been extensively studied by Gagna, Lanari and Harpaz in \cite{Gagna_fibrations_and_lax_limit_infini_2_categories} and \cite{Gagna_Cartesian_Fibrations_of_infini_2_categories}, as well as by García and Stern in \cite{Garcia_2_cartesian_fibration_I} and \cite{Garcia_2_cartesian_fibration_II}.

For arbitrary $n$, Grothendieck construction has been described in \cite{Nuiten_on_straightening_for_segal_spaces} and \cite{Rasekh_yoneda_lemma_for_simplicial_spaces}. A partial version of the Yoneda lemma is also present in \cite{Rasekh_yoneda_lemma_for_simplicial_spaces}, \cite{Hinich_colimit_in_enriched_infini_categories}, and \cite{Heine_an_equivalence_between_enricherd_infini_categorories_and_categories_with_weak_action}.

\vspace{1cm}

\paragraph{Chapter \ref{chapter:the infini 1 categorory of io categories}.} 
This chapter is dedicated to the basic definition of $\io$-categories. In the first section, we recall some results on factorization systems in presentable $\iun$-categories. In the second section, we define $\io$-categories and give some basic properties. 
We also define and study \textit{discrete Conduché functor}, which are morphisms having the unique right lifting property against 
units $\Ib_{n+1}:\Db_{n+1}\to \Db_n$ for any integer $n$, and against compositions $\triangledown_{k,n}:\Db_n\to \Db_n\coprod_{\Db_k}\Db_n$ for any pair of integers $k\leq n$. This notion was originally defined and studied in the context of strict $\omega$-category by Guetta in \cite{Guetta_conduche}.
\begin{itheorem}[\ref{theo:pullback along conduche preserves colimits}]
Let $f:C\to D$ be a discrete Conduché functor. The pullback functor $f^*:\ocat_{/D}\to \ocat_{/C}$ preserves colimits.
\end{itheorem}

 In the third section, we study Gray operations for $\io$-categories. We conclude this chapter by proving results of strictification. In particular, we demonstrate the following theorem:
\begin{itheorem}[\ref{prop:strict stuff are pushout}]
Let $C$ be an $\io$-category, $b$ a globular sum, and $f:b\to C$ any morphism. The $\io$-categories $$1\costar b\coprod_b C,~C\coprod_b b\otimes[1]~\mbox{and}~C\coprod_b b\star 1$$
are strict whenever $C$ is.
\end{itheorem}
We will also prove the following theorem:
\begin{itheorem}[\ref{theo:strictness}]
If $C$ is strict, so are $C\star 1$, $1\costar C$ and $C\otimes [1]$.
\end{itheorem}
In the process, we will demonstrate another fundamental equation combining $C\otimes[1]$, $1\costar C$, $C\star 1$, and $[C,1]$.
\begin{itheorem}[\ref{theo:formula between pullback of slice and tensor}]
Let $C$ be an $\io$-category. The five squares appearing in the following canonical diagram are both cartesian and cocartesian:
\[\begin{tikzcd}
	& {C\otimes\{0\}} & 1 \\
	{C\otimes\{1\}} & {C\otimes[1]} & {C\star 1} \\
	1 & {1\costar C} & {[C,1]}
	\arrow[from=2-3, to=3-3]
	\arrow[from=3-2, to=3-3]
	\arrow[from=2-2, to=3-2]
	\arrow[from=2-2, to=2-3]
	\arrow[from=1-2, to=1-3]
	\arrow[from=1-3, to=2-3]
	\arrow[from=1-2, to=2-2]
	\arrow[from=2-1, to=2-2]
	\arrow[from=3-1, to=3-2]
	\arrow[from=2-1, to=3-1]
\end{tikzcd}\]
where $[C,1]$ is the \textit{suspension of $C$}.
\end{itheorem}

\paragraph{Chapter \ref{chapter:chapter the 1 category of marked categories}.}
This chapter is dedicated to the study of \textit{marked $\io$-categories}, which are pairs $(C,tC)$, where $C$ is an $\io$-category and $tC:=(tC_n)_{n>0}$ is a sequence of full sub $\infty$-groupoids of $C_n$ that include identities and are stable under composition and whiskering with (possibly unmarked) cells of lower dimensions. There are two canonical ways to mark an $\io$-category $C$. In the first, denoted by $C^\flat$, we mark as little as possible. In the second, denoted by $C^\sharp$, we mark everything.

The first section of the chapter defines these objects and establishes analogs of many results from section \ref{chapter:Basica construciton} to this new framework. In particular, the \textit{marked Gray cylinder} $\uvar\otimes [1]^\sharp$ is defined. If $A$ is an $\io$-category, the underlying $\io$-category of $A^\sharp\otimes[1]^\sharp$ is $A\times [1]$, and the underlying $\io$-category of $A^\flat\otimes[1]^\sharp$ is $A\otimes[1]$. By varying the marking, and at the level of underlying $\io$-categories, we "continuously" move from the cartesian product with the directed interval to the Gray tensor product with the directed interval.

The motivation for introducing markings comes from the notion of \textit{left (and right) cartesian fibrations}. A left cartesian fibration is a morphism between marked $\io$-categories such that only the marked cells of the codomain have cartesian lifting, and the marked cells of the domain correspond exactly to such cartesian lifting. For example, a left cartesian fibration $X\to A^\sharp$ is just a "usual" left cartesian fibration where we have marked the cartesian lifts of the domain, and every morphism $C^\flat \to D^\flat$ is a left cartesian fibration. This shows that marking plays a very different role here than in the case of marked simplicial sets, where it was there to represent (weak) invertibility. For example, if we had wanted to carry out this work in a complicial-like model category, we would have had to consider bimarked simplicial sets.

After defining and enumerating the stability properties enjoyed by this class of left (and right) cartesian fibration, we give several characterizations of this notion in theorem \ref{theo:other characterisation of left caresian fibration}. 

The more general subclass of left cartesian fibrations that still behaves well is the class of \textit{classified left cartesian fibrations}. 
This corresponds to left cartesian fibrations $X\to A$ such that there exists a cartesian square:
\[\begin{tikzcd}
	X & Y \\
	A & {A^\sharp}
	\arrow[from=1-1, to=2-1]
	\arrow[from=2-1, to=2-2]
	\arrow[from=1-1, to=1-2]
	\arrow[from=1-2, to=2-2]
	\arrow["\lrcorner"{anchor=center, pos=0.125}, draw=none, from=1-1, to=2-2]
\end{tikzcd}\]
 where the right vertical morphism is a left cartesian fibration and $A^\sharp$ is obtained from $A$ by marking all cells. In the second section, we prove the following fundamental result:

\begin{itheorem}[\ref{theo:pullback along un marked cartesian fibration}]
Let $p:X\to A$ be a classified left cartesian fibration. Then the functor $p^*:\ocatm_{/A}\to \ocatm_{/X}$ preserves colimits.
\end{itheorem}

The third subsection is devoted to the proof of the following theorem
\begin{itheorem}[\ref{theo:left cart stable by colimit}]
Let $A$ be an $\io$-category and $F:I\to \ocatm_{/A^\sharp}$ be a diagram that is pointwise a left cartesian fibration. The induced morphism 
$\colim_IF$ is a left cartesian fibration over $A^\sharp$.
\end{itheorem}

In the fourth subsection we study \textit{smooth} and \textit{proper} morphisms and we obtain the following expected result:
\begin{iprop}[\ref{prop:quillent theorem A}]
For a morphism $X\to A^\sharp$, and an object $a$ of $A$, we denote by $X_{/a}$ the marked $\io$-category fitting in the following cartesian squares. 
\[\begin{tikzcd}
	{X_{a/}} & X \\
	{A^\sharp_{a/}} & {A^\sharp}
	\arrow[from=2-1, to=2-2]
	\arrow[from=1-2, to=2-2]
	\arrow[from=1-1, to=1-2]
	\arrow[from=1-1, to=2-1]
	\arrow["\lrcorner"{anchor=center, pos=0.125}, draw=none, from=1-1, to=2-2]
\end{tikzcd}\]
We denote by $\bot:\ocatm\to \ocat$ the functor sending a marked $\io$-category to its localization by marked cells.
\begin{enumerate}
\item Let $E$, $F$ be two elements of $\ocatm_{/A^\sharp}$ corresponding to morphisms $X\to A^\sharp$, $Y\to A^\sharp$, and
 $\phi:E\to F$ a morphism between them. We denote by $\Fb E$ and $\Fb F$ the left cartesian fiborant replacement of $E$ and $F$. 
 
The induced morphism $\Fb\phi:\Fb E\to \Fb F$ is an equivalence if and only if for any object $a$ of $A$, the induced morphism 
$$\bot X_{/a}\to \bot Y_{/a}$$ 
is an equivalence of $\io$-categories.
\item A morphism $X\to A^\sharp$ is initial if and only if for any object $a$ of $A$, $\bot X_{/a}$ is the terminal $\io$-category.
\end{enumerate}
\end{iprop}

Finally, in the last subsection, for a marked $\io$-category $I$, we define and study a (huge) $\io$-category $\uLCartc(I)$ that has classified left cartesian fibrations as objects and morphisms between classified left cartesian fibrations as arrows.

\paragraph{Chapter \ref{chapter:The io-category of small io-categories}.}
This chapter aims to establish analogs of the fundamental categorical constructions to the $\io$ case. In the first section, we define the $\io$-category of small $\io$-categories $\uni$ (paragraph \ref{para:defi of uni}), and we prove a first incarnation of the Grothendieck construction:
\begin{icor}[\ref{cor: Grt equivalence}]
Let $\uni$ be the $\io$-category of small $\io$-categories, and $A$ an $\io$-category. There is an equivalence
$$\int_A:\Hom(A,\uni)\to \tau_0 \LCart(A^\sharp).$$
where $\tau_0 \LCart(A^\sharp)$ is the $\infty$-groupoid of left cartesian fibrations over $A^\sharp$ with small fibers.
\end{icor}
Given a functor $f:A\to \uni$, the left cartesian fibration $\int_Af$ is a colimit (computed in $\ocatm_{/A^\sharp}$) of
a simplicial object whose value on $n$ is of shape
$$\coprod_{x_0,...,x_n:A_0}X(x_0)^\flat\times\hom_A(x_0,...,x_n)^\flat\times A^\sharp_{x_n/}\to A^\sharp$$
This formula is similar to the one given in \cite{Gepner_Lax_colimits_and_free_fibration}
 for $\iun$-categories, and to the one given in \cite{Warren_the_strict_omega_groupoid_interpretation_of_type_theory} for strict $\omega$-categories.

We also prove a univalence result:

\begin{icor}[\ref{cor:univalence}]
Let $I$ be a marked $\io$-category. We denote by $I^\sharp$ the marked $\io$-category obtained from $I$	 by marking all cells and $\iota:I\to I^\sharp$ the induced morphism. There is a natural correspondence between \begin{enumerate}
\item functors
$f:I\otimes [1]^\sharp\to \uni^\sharp,$

\item pairs of small left cartesian fibration $X\to I^\sharp$, $Y\to I^\sharp$ together with diagrams 
\[\begin{tikzcd}
	& {\iota^*X} && X \\
	{\iota^*Y} && Y \\
	& I && {I^\sharp}
	\arrow[""{name=0, anchor=center, inner sep=0}, "\iota"', from=3-2, to=3-4]
	\arrow[from=2-1, to=3-2]
	\arrow[from=2-3, to=3-4]
	\arrow[from=2-1, to=2-3]
	\arrow[from=1-2, to=3-2]
	\arrow[from=1-4, to=3-4]
	\arrow[from=1-2, to=1-4]
	\arrow["\phi"{description}, from=1-2, to=2-1]
	\arrow["\lrcorner"{anchor=center, pos=0.125}, draw=none, from=1-2, to=0]
	\arrow["\lrcorner"{anchor=center, pos=0.125}, draw=none, from=2-1, to=0]
\end{tikzcd}\]
\end{enumerate}
\end{icor}

Recall that if $I$ is of shape $B^\sharp$, then the underlying $\io$-category of $B^\sharp\otimes[1]^\sharp$ is $B\times [1]$, and if $I$ is of shape $B^\flat$, the underlying $\io$-category of $B^\flat\otimes[1]^\sharp$ is $B\otimes[1]$. On the other hand, if $I$ is $B^\sharp$, $\iota$ is the identity, and $\phi$ then preserves all cartesian liftings, and if $I$ is $B^\flat$, $\phi$ doesn't need to preserve cartesian liftings.

By varying the marking, we can continuously move from the cartesian product with the interval to the Gray product with the interval on one side, and on the other side, we can continuously move from morphisms between left cartesian fibrations that preserve the marking to the ones that do not preserve it \textit{a priori}.

Eventually, we also get an $\io$-functorial Grothendieck construction, expressed by the following corollary:

\begin{icor}[\ref{cor:lcar et hom}]
Let $A$ be a $\U$-small $\io$-category.
Let $\uLCart(A^\sharp)$ be the $\io$-category of small left cartesian fibrations over $A^\sharp$. 
There is an equivalence
$$\uHom(A,\uni)\sim \uLCart(A^\sharp)$$
natural in $A$.
\end{icor}

In the second section of this chapter, for a locally small $\io$-category $C$, we construct the Yoneda embedding, which is a functor $y:C\to \widehat{C}$ where $\widehat{C}:=\uHom(C^t,\uni)$. We prove the Yoneda lemma:
\begin{itheorem}[\ref{theo:Yoneda ff}]
The Yoneda embedding is fully faithful.
\end{itheorem}
\begin{itheorem}[\ref{theo:Yoneda lemma}]
Let $C$ be an $\io$-category. There is an equivalence between the functor
$$\hom_{\w{C}}(y_{\uvar},\uvar):C^t\times \w{C}\to \uni$$ and
the functor 
$$ev:C^t\times \w{C}\to \uni .$$
\end{itheorem}
In the last three sections, we use these results to study and demonstrate the basic properties of adjunctions, lax (co)limits, and left Kan extensions.

We begin by studying adjunctions, and we establish the following expected result.
\begin{itheorem}[\ref{theo:two adjunction definition}]
Let $u:C\to D$ and $v:D\to C$ be two functors between locally $\U$-small $\io$-categories. 
The two following are equivalent. 
\begin{enumerate}
\item The pair $(u,v)$ admits an adjoint structure.
\item Their exists a pair of natural transformations $\mu: id_C \to vu$ and $\epsilon:uv\to id_D$ together with equivalences $(\epsilon\circ_0 u)\circ_1(u\circ_0 \mu) \sim id_{u}$ and $(v\circ_0 \epsilon)\circ_1 (\mu \circ_0 v )\sim id_{v}$.
\end{enumerate}
\end{itheorem}

In the next subsection, given a morphism $f:I\to C^\sharp$ between marked $\io$-categories, we define the notion of lax colimit and lax limit for the functor $f$. If $f$ admits such a lax colimit, for any $1$-cell $i:a\to b$ in $I$, we have a triangle
\[\begin{tikzcd}
	{} & {F(b)} \\
	{F(a)} & {\laxcolim_IF}
	\arrow["{F(i)}", curve={height=-30pt}, from=2-1, to=1-2]
	\arrow[from=2-1, to=2-2]
	\arrow[shorten <=8pt, shorten >=8pt, Rightarrow, from=1-2, to=2-1]
	\arrow[draw=none, from=1-1, to=2-1]
	\arrow[from=1-2, to=2-2]
\end{tikzcd}\]
If $i$ is marked, the preceding $2$-cell is an equivalence. 
For any $2$-cell $u:i\to j$, we have a diagram
\[\begin{tikzcd}
	& {F(b)} & {} & {F(b)} \\
	{F(a)} & {\laxcolim_IF} & {F(a)} & {\laxcolim_IF}
	\arrow[""{name=0, anchor=center, inner sep=0}, "{F(i)}"{description}, from=2-1, to=1-2]
	\arrow[""{name=1, anchor=center, inner sep=0}, from=2-1, to=2-2]
	\arrow[from=1-2, to=2-2]
	\arrow[""{name=2, anchor=center, inner sep=0}, from=1-2, to=2-2]
	\arrow[""{name=3, anchor=center, inner sep=0}, "{F(j)}", curve={height=-30pt}, from=2-1, to=1-2]
	\arrow["{F(j)}", curve={height=-30pt}, from=2-3, to=1-4]
	\arrow[from=2-3, to=2-4]
	\arrow[from=1-4, to=2-4]
	\arrow[shorten <=8pt, shorten >=8pt, Rightarrow, from=1-4, to=2-3]
	\arrow[""{name=4, anchor=center, inner sep=0}, draw=none, from=1-3, to=2-3]
	\arrow[shift right=2, shorten <=12pt, shorten >=12pt, Rightarrow, from=2, to=1]
	\arrow[shorten <=4pt, shorten >=4pt, Rightarrow, from=3, to=0]
	\arrow[shift left=0.7, shorten <=14pt, shorten >=16pt, no head, from=2, to=4]
	\arrow[shorten <=14pt, shorten >=14pt, from=2, to=4]
	\arrow[shift right=0.7, shorten <=14pt, shorten >=16pt, no head, from=2, to=4]
\end{tikzcd}\]
If $u$ is marked, the $3$-cell is an equivalence. We can continue these diagrams in higher dimensions and we have
similar assertions for lax limits.
The marking therefore allows us to play on the "lax character" of the universal property that the lax colimit must verify.

After providing several characterizations of lax colimits and limits, we prove the following result:
\begin{itheorem}[\ref{theo:presheaevs colimi of representable}]
Let $C$ be a $\U$-small $\io$-category. Let $f$ be an object of $\w{C}$. We define $C^\sharp_{/f}$ as the following pullback
\[\begin{tikzcd}
	{C^\sharp_{/f}} & {\w{C}^\sharp_{/f}} \\
	{C^\sharp} & {\w{C}^\sharp}
	\arrow[from=1-1, to=2-1]
	\arrow[from=1-1, to=1-2]
	\arrow[from=1-2, to=2-2]
	\arrow["{y^\sharp}"', from=2-1, to=2-2]
\end{tikzcd}\]
The colimit of the functor 
$\pi:C^\sharp_{/f}\to C^\sharp\xrightarrow{y^\sharp} \w{C}^\sharp$ is $f$.
\end{itheorem}

We conclude this chapter by studying Kan extensions.

\phantomsection
\addcontentsline{toc}{section}{Notice of authority}
\section*{Notice of authority}

The chapter \ref{chapter:Studies of the complicial model} is a shorter version of the preprint \cite{Loubaton_dualities_in_the_complicial_model}. Chapter \ref{chapter:complicial set as a model of io categories} is almost identical to the preprint \cite{Loubaton_complicial_sets_as_a_model_of_infini_n_categories}. During this thesis, two other papers were written: \cite{Loubaton_condition_de_kan} (in progress of publication at the SMF) and \cite{Henry_an_inductive_model_structure_for_infini_categories} (in collaboration with Simon Henry). Although the topics are similar, the questions addressed are quite different, and these papers are thus not included in the present text.


\cleardoublepage

\pagestyle{fancy}
\fancyhf{}
\fancyhfoffset[RO,LE]{0.5cm}
\fancyhfoffset[LE,RO]{0.5cm}

\fancyhead[RO]{\rmfamily\nouppercase{\rightmark}}
\fancyhead[LE]{\rmfamily\nouppercase{\leftmark}}
\fancyfoot[C]{\thepage}

\phantomsection
\addcontentsline{toc}{part}{Preliminaries} 
\part*{Preliminaries}
%
%
%
%
%
%
%
%
%
%

\chapter{The category of $\zo$-categories}
\label{chapter:The category of zocategories}

\minitoc
\vspace{1cm}
The first section is devoted to the definition of $\zo$-categories and of the category $\Theta$ of Joyal. We also show that the category $\Theta$ presents the category of $\zo$-categories, and we also exhibit an other presentation of this category (corollary \ref{cor:changing theta}).

The second section begins with a review of Steiner theory, which is an extremely useful tool for providing concise and computational descriptions of $\zo$-categories. Following Ara and Maltsiniotis, we employ this theory to define the Gray tensor product, denoted by $\otimes$, in $\zo$-categories. We then introduce the Gray operations, starting with the Gray cylinder $\uvar\otimes[1]$ which is the Gray tensor product with the directed interval $[1]:=0\to 1$. Then, we have the Gray cone and Gray $\circ$-cone, denoted by $\uvar\star 1$ and $1\costar \uvar$, that send an $\zo$-category $C$ onto the following pushouts:
\[\begin{tikzcd}
	{C\otimes\{1\}} & {C\otimes[1]} && {C\otimes\{0\}} & {C\otimes[1]} \\
	1 & {C\star 1} && 1 & {1\costar C}
	\arrow[from=1-5, to=2-5]
	\arrow[from=1-4, to=2-4]
	\arrow[from=2-4, to=2-5]
	\arrow[from=1-4, to=1-5]
	\arrow[from=1-2, to=2-2]
	\arrow[from=2-1, to=2-2]
	\arrow[from=1-1, to=2-1]
	\arrow[from=1-1, to=1-2]
	\arrow["\lrcorner"{anchor=center, pos=0.125, rotate=180}, draw=none, from=2-2, to=1-1]
	\arrow["\lrcorner"{anchor=center, pos=0.125, rotate=180}, draw=none, from=2-5, to=1-4]
\end{tikzcd}\]

We also present a formula that illustrates the interaction between the suspension and the Gray cylinder. As this formula plays a crucial role in both Part I and Part II, we provide its intuition at this stage.

 If $A$ is any $\zo$-category, the suspension of $A$, denoted by $[A,1]$, is the $\zo$-category having two objects - denoted by $0$ and $1$- and such that 
$$\Hom_{[A,1]}(0,1) := A,~~~\Hom_{[A,1]}(1,0) := \emptyset,~~~\Hom_{[A,1]}(0,0)=\Hom_{[A,1]}(1,1):=\{id\}.$$
We also define $[1]\vee[A,1]$ as the gluing of $[1]$ and $[A,1]$ along the $0$-target of $[1]$ and the $0$-source of $[A,1]$. We define similarly $[A,1]\vee[1]$.
These two objects come along with \textit{whiskerings}:
$$\triangledown:[A,1]\to [1]\vee [A,1] ~~~~\mbox{and}~~~~ \triangledown:[A,1] \to [A,1]\vee [1]$$ 
that preserve the extremal objects.

The $\zo$-category $[1]\otimes [1]$ is induced by the diagram:
\[\begin{tikzcd}
	00 & 01 \\
	10 & 11
	\arrow[from=1-1, to=2-1]
	\arrow[from=2-1, to=2-2]
	\arrow[from=1-1, to=1-2]
	\arrow[from=1-2, to=2-2]
	\arrow[shorten <=4pt, shorten >=4pt, Rightarrow, from=1-2, to=2-1]
\end{tikzcd}\]
and is then equal to the colimit of the following diagram: 
$$[1]\vee [1]\xleftarrow{\triangledown} [1]\hookrightarrow [[1],1]\hookleftarrow[1]\xrightarrow{\triangledown } [1]\vee [1].$$
The $\zo$-category $ [[1],1]\otimes [1]$ is induced by the diagram:
\[\begin{tikzcd}
	00 & 01 & 00 & 01 \\
	10 & 11 & 10 & 11
	\arrow[from=1-1, to=1-2]
	\arrow[""{name=0, anchor=center, inner sep=0}, from=1-1, to=2-1]
	\arrow[from=2-1, to=2-2]
	\arrow[""{name=1, anchor=center, inner sep=0}, from=1-2, to=2-2]
	\arrow[shorten <=4pt, shorten >=4pt, Rightarrow, from=1-2, to=2-1]
	\arrow[""{name=2, anchor=center, inner sep=0}, from=1-3, to=2-3]
	\arrow[from=1-3, to=1-4]
	\arrow[""{name=3, anchor=center, inner sep=0}, from=1-4, to=2-4]
	\arrow[shorten <=4pt, shorten >=4pt, Rightarrow, from=1-4, to=2-3]
	\arrow[""{name=4, anchor=center, inner sep=0}, curve={height=30pt}, from=1-1, to=2-1]
	\arrow[from=2-3, to=2-4]
	\arrow[""{name=5, anchor=center, inner sep=0}, curve={height=-30pt}, from=1-4, to=2-4]
	\arrow["{ }"', shorten <=6pt, shorten >=6pt, Rightarrow, from=0, to=4]
	\arrow["{ }"', shorten <=6pt, shorten >=6pt, Rightarrow, from=5, to=3]
	\arrow[shift left=0.7, shorten <=6pt, shorten >=8pt, no head, from=1, to=2]
	\arrow[shift right=0.7, shorten <=6pt, shorten >=8pt, no head, from=1, to=2]
	\arrow[shorten <=6pt, shorten >=6pt, from=1, to=2]
\end{tikzcd}\]
and is then equal to the colimit of the following diagram: 
 $$[1]\vee[[1],1]\xleftarrow{\triangledown} [[1]\otimes\{0\},1]\hookrightarrow[[1]\otimes[1],1]\hookleftarrow [[1]\otimes\{1\},1]\xrightarrow{\triangledown}[[1],1]\vee[1]$$
We prove a formula that combines these two examples:

\begin{itheorem}[\ref{theo:appendice formula for otimes}]
In the category of $\zo$-categories, there exists an isomorphism, natural in $A$, between $[A,1]\otimes[1]$ and the colimit of the following diagram
\[\begin{tikzcd}
	{[1]\vee[A,1]} & {[A\otimes\{0\},1]} & { [A\otimes[1],1]} & {[A\otimes\{1\},1]} & {[A,1]\vee[1]}
	\arrow["\triangledown"', from=1-2, to=1-1]
	\arrow[from=1-4, to=1-3]
	\arrow["\triangledown", from=1-4, to=1-5]
	\arrow[from=1-2, to=1-3]
\end{tikzcd}\]
\end{itheorem} 

We also provide similar formulas for the \textit{Gray cone} and the \textit{Gray $\circ$-cone}.
\begin{itheorem}[\ref{theo:appendice formula for star}]
There is a natural identification between $1\costar [A,1]$ and the colimit of the following diagram
\[\begin{tikzcd}
	{[1]\vee[A,1]} & {[A,1]} & { [A\star 1,1]}
	\arrow["\triangledown"', from=1-2, to=1-1]
	\arrow[from=1-2, to=1-3]
\end{tikzcd}\]
There is a natural identification between $[A,1]\star 1$ and the colimit of the following diagram
\[\begin{tikzcd}
	{ [1\costar A,1]} & {[A,1]} & {[A,1]\vee[1]}
	\arrow[from=1-2, to=1-1]
	\arrow["\triangledown", from=1-2, to=1-3]
\end{tikzcd}\]
\end{itheorem}

\section{Basic constructions}
\label{chapter:Basica construciton preliminaire}
\subsection{$\zo$-Categories}
\label{section:zocategories}
\p A \notion{globular set} is a presheaf on the \textit{category of globes} $\Gb$, which is the category induces by the diagram
\[\begin{tikzcd}
	{\Db_0} & {\Db_1} & {\Db_2} & {...}
	\arrow["{i_0^+}", shift left=2, from=1-1, to=1-2]
	\arrow["{i_1^+}", shift left=2, from=1-2, to=1-3]
	\arrow["{i_3^+}", shift left=2, from=1-3, to=1-4]
	\arrow["{i_0^-}"', shift right=2, from=1-1, to=1-2]
	\arrow["{i_1^-}"', shift right=2, from=1-2, to=1-3]
	\arrow["{i_3^-}"', shift right=2, from=1-3, to=1-4]
\end{tikzcd}\]
with the relations $i_n^{+} i_{n-1}^\epsilon = i_n^{-} i_{n-1}^\epsilon $ for any $n>0$ and $\epsilon \in \{+,-\}$. We also denote by $i^{\epsilon}_k$ the map $\Db_{k} \to \Db_n$ for $k< n$ obtained by composing any string of arrows ending with $i^\epsilon_{k}$. These and the identity arrows are the only maps in the category $\Gb$.

If $X$ is a globular set, one denotes by $X_n$ the set $X(\Db_n)$. Its elements are called \wcsnotion{$n$-cells}{cell@$n$-cell}{for $\zo$-categories}. The $0$-cells are sometimes called \textit{objects}. The maps $X_n \to X_k$ induced by $i^\epsilon_k : \Db_k \to \Db_n$ is denoted by $\pi^\epsilon_k$.

\p
\label{para:def of omega cat}
An \wcnotion{$\omega$-category}{category@$\omega$-category} is a globular set $X$ together with
\begin{enumerate}
\item operations of \textit{compositions}
\[ X_n\times_{X_k} X_n\to X_n ~~~(0\leq k<n) \]
which associate to two $n$-cells $(x,y)$ verifying $\pi_k^-(x) = \pi_k^+(y)$, a $n$-cells $x\circ_ky$,
\item as well as \textit{units}
\[X_n\to X_{n+1}\]
which associate to an $n$-cell $x$, a $(n+1)$-cell $\Ib_x$, 
\end{enumerate}
and satisfying the following axioms:
\begin{enumerate}

\item $\forall x \in X_n, \pi^\epsilon_n(\Ib_x) = x $.

\item $\pi^+_k (x \circ_n y) = \pi_k^{+}(x)$ and $\pi^-_k(x \circ_n y) = \pi_k^-(y)$ whenever the composition is defined and $k \leqslant n$.

\item $\pi^\epsilon_k (x \circ_n y) = \pi_k^{\epsilon}(x) \circ_n \pi^\epsilon_k(y)$ whenever the composition is defined and $k > n$.

\item $ x \circ_n \Ib_{\pi^-_n x} = x$ and $ \Ib_{\pi^+_n x} \circ_n x = x$.

\item $(x \circ_n y) \circ_n z = x \circ_n (y \circ_n z) $ as soon as one of these is defined.

\item If $k <n$

\[ (x \circ_n y) \circ_k ( z \circ_n w) = (x \circ_k z) \circ_n (y \circ_k w) \]
when the left-hand side is defined.

\end{enumerate}
A $n$-cell $a$ is \textit{non trivial} if is not in the image of the application $\Ib:X_{n-1}\to X_n$.

A \textit{morphism of $\omega$-categories} is a map of globular sets commuting with both operations. The category of $\omega$-categories is denoted by \textit{$\omegacat$}.

\p
\index[notion]{globe@$n$-globe!for $\zo$-categories}
By abuse of notation, we also denote by \wcsnotation{$\Db_n$}{(da@$\Db_n$}{for $\zo$-categories} the $\omega$-category that admits for any $k<n$ only two $k$-non-trivial cells, denoted by $e_k^-$ and $e_k^+$, and a single $n$-non-trivial cell, denoted by $e_n$ verifying :
\[
\begin{array}{rcl}
\pi_l^{-}(e_k^\epsilon)= e_l^{-}&\pi_l^{+}(e_k^\epsilon)= e_l^{+}& \mbox{ for $l\leq k<n$}\\
\pi_l^{-}(e_n)= e_l^{-}&\pi_l^{+}(e_n)= e_l^{+}& \mbox{ for $l\leq n$}\\
\end{array}
\]

Remark furthermore that the $\omega$-category $\Db_n$ represents $n$-cells, in the sense that $\Hom(\Db_n,C)\cong C_n$. We will not make the difference between $n$-cells and the corresponding morphism of $\Db_n\to C$. 

The $\omega$-category $\partial\Db_n$ is obtained from $\Db_n$ by removing the $n$-cell $e_n$. We thus have a morphism
\[i_n: \partial\Db_n\to \Db_n.\]
Note that $\partial \Db_0 = \emptyset$.

\p 
We say that an $\zo$-category $X$ is a \notion{polygraph} if it can be constructed from the empty $\zo$-category by freely adding arrows with specified source and target. That is if $X$ can be obtained as a transfinite composition $\emptyset = X_0 \to X_1 \to \dots \to X_i \to \colim X_i = X$ where for each $i$, the map $X_i \to X_{i+1}$ is a pushout of $\coprod_S \partial \Db_n \to \coprod_S \Db_{n+1}$.

 An arrow of a polygraph is said to be a \emph{generator} if it is one of the arrows that has been freely added at some stage.

Each cell in a polygraph can be written as an iterated composite of generators or iterated unit of generators (not necessarily in a unique way). For a $n$-cell $f$, the set of generators of dimension $n$ that appear in such an expression (and even the number of times they appear) is the same for all such expressions. As a consequence, a iterated composition of non trivial cells is always non trivial.

\p \label{para:dualities strict case}
 For any subset $S$ of $\Nb^*$, we define the functor $(\uvar)^S:\omegacat\to \omegacat$ \ssym{((b49@$(\uvar)^S$}{for $\zo$-categories} sending a $\omega$-category $C$ to the category $C^S$ such that for any $n$, there is an isomorphism $C_n\to C_{n}^S$ that sends every $n$-cell $f$ to a cell $\overline{f}$ fulfilling
$$\pi_{n-1}^-(\overline{f})=\overline{\pi^+_{n-1}(f)}~~~~\pi_{n-1}^+(\overline{f})=\overline{\pi^-_{n-1}(f)}$$
if $i\in S$ and 
$$\pi_{n-1}^-(\overline{f})=\overline{\pi^-_{n-1}(f)}~~~~\pi_{n-1}^+(\overline{f})=\overline{\pi^+_{n-1}(f)}$$
if $i\notin S$.
These functors are called \snotion{dualities}{for $\zo$-categories} as they are inverse of themselves. Even if there are plenty of them, we will be interested in only a few of them. In particular, we have the \snotionsym{odd duality}{((b60@$(\uvar)^{op}$}{for $\zo$-categories} $(\uvar)^{op}$, corresponding to the set of odd integer, the \snotionsym{even duality}{((b50@$(\uvar)^{co}$}{for $\zo$-categories} $(\uvar)^{co}$, corresponding to the subset of non negative even integer, the \snotionsym{full duality}{((b80@$(\uvar)^{\circ}$}{for $\zo$-categories} $(\uvar)^{\circ}$, corresponding to $\Nb^*$ and the \snotionsym{transposition}{((b70@$(\uvar)^t$}{for $\zo$-categories} $(\uvar)^t$, corresponding to the singleton $\{1\}$. Eventually, we have equivalences
$$((\uvar)^{co})^{op}\sim (\uvar)^{\circ} \sim ((\uvar)^{op})^{co}.$$

\p Let $\Psh{\Gb}_{\bullet,\bullet}$ be the category of globular set with two distinguished points, i.e. of triples $(X,a,b)$ where $a$ and $b$ are elements of $X_0$.
Let $[\uvar,1]:\Gb\to \Psh{\Gb}_{\bullet,\bullet}$ be the functor sending $\Db_n$ on $(\Db_{n+1},\{0\},\{1\})$ and $i_n^{\epsilon}$ on $i_{n+1}^{\epsilon}$. This induces a functor $[\uvar,1]:\Psh{\Gb}\to \Psh{\Gb}$ that we call the \textit{suspension}. We leave it to the reader to check that whenever $C$ has a structure of $\omega$-category, $[C,1]$ inherits one from it. This functor then induces a functor 
$$[\uvar,1]:\omegacat\to \omegacat$$
that we calls again the \snotionsym{suspension}{((d60@$[\uvar,1]$}{for $\zo$-categories}. Eventually, we denote by $i_0^-:\{0\}\to [C,1]$ (resp. $i_0^+:\{1\}\to [C,1]$) the morphism corresponding to the left point (resp. to the right point). For an integer $n$, we define by induction the functor $\Sigma^n:\Psh{\Gb}\to \Psh{\Gb}$\ssym{(sigma@$\Sigma^n$}{for $\zo$-categories} with the formula:
$$\Sigma^0:= id ~~~~~\Sigma^{n+1}:=\Sigma^n[\uvar,1].$$

\p Let $n$ be a non null integer.
A $n$-cells $f:s\to t$ is an \notion{equivalence} if there exists $n$-cells $g:t\to s$ and $g':t\to s$ such that 
$$f\circ_{n-1} g=\Ib_t~~~~~~g\circ_{n-1} f=\Ib_s$$
A \wcnotion{$\zo$-category}{category1@$\zo$-category} is an $\omega$-category whose only equivalences are the identities.
These objects are called \textit{Gaunt $\omega$-categories} in \cite{Barwick_on_the_unicity_of_the_theory_of_higher_categories} and \textit{rigid $\omega$-categories} in \cite{Rezk_a_cartesian_of_weak_n_categories}. Remark that $\zo$-categories are stable under suspensions and dualities.
We then define \wcnotation{$\zocat$}{((a20@$\zocat$} as the full subcategory of $\omegacat$ whose objects are the $\zo$-categories.

\p
Let $n$ be an integer. An \wcnotion{$(0,n)$-category}{category2@$(0,n)$-category} is an $\zo$-category whose cell of dimension strictly higher than $n$ are units. The category of $n$-categories is denoted by \wcnotation{$\zncat{n}$}{((a10@$\zncat{n}$} and is the full subcategory of $\zocat$ whose objects are $(0,n)$-categories.

 Remark that the category $\zncat{n}$ is the localization of $\zocat$ along morphisms $\Db_{k}\to \Db_{n}$ for $k\geq n$. We then have for any $n$ an adjunction 
\[\begin{tikzcd}
	{i_n:\zncat{n}} & {\zocat:\tau_n}
	\arrow[""{name=0, anchor=center, inner sep=0}, shift left=2, from=1-2, to=1-1]
	\arrow[""{name=1, anchor=center, inner sep=0}, shift left=2, from=1-1, to=1-2]
	\arrow["\dashv"{anchor=center, rotate=-90}, draw=none, from=1, to=0]
\end{tikzcd}\]
The right adjoint is called the \wcsnotionsym{$n$-truncation}{(tau@$\tau_n$}{truncation@$n$-truncation}{for $\zo$-categories}.
For any $n$, we define the colimit preserving functor $\tau^i_n:\zocat\to \zncat{n}$, called the \snotionsym{intelligent $n$-truncation}{(taui@$\tau^i_n$}{for $\zo$-categories}, sending $\Db_k$ on $\Db_{\min(n,k)}$. The functor $\tau^i_n$ fits in an adjunction
\[\begin{tikzcd}
	{\tau^i_n:\zocat} & {\zncat{n}:i_n}
	\arrow[""{name=0, anchor=center, inner sep=0}, shift left=2, from=1-1, to=1-2]
	\arrow[""{name=1, anchor=center, inner sep=0}, shift left=2, from=1-2, to=1-1]
	\arrow["\dashv"{anchor=center, rotate=-90}, draw=none, from=0, to=1]
\end{tikzcd}\]
We will identify objects of $\zncat{n}$ with their image in $\zocat$ and we will then also note by $\tau_n$ and $\tau^i_n$ the composites $i_n\tau^i_n$ and $i_n\tau^i_n$.

\p The family of truncation functor induces a sequence 
$$...\to \zncat{n+1}\xrightarrow{\tau_{n}} \zncat{n}\to...\to \zncat{1}\xrightarrow{\tau_{0}}\zncat{0}.$$
The canonical morphism
$$\zocat\to \lim_{n:\Nb}\zncat{n},$$
that sends an $\zo$-category $C$ to the sequence $(\tau_n C,\tau_n\tau_{n+1}C\cong \tau_n C)$, has an inverse given by the functor
$$\colim_{\Nb}:\lim_{n:\Nb}\zncat{n}\to \zocat$$
that sends a sequence $(C_n, \tau_{n}C_{n+1}\cong C_n) $ to the colimit of the induced sequence:
$$i_0C_0\to i_1C_1\to...\to i_nC_n\to...$$	
 We then have an equivalence 
$$\zocat\cong \lim_{n:\Nb}\zncat{n}.$$

\subsection{The category $\Theta$}
\label{subsection:the categoru theta}

\p Let $n$ be a non negative integer and $\textbf{a}:=\{a_0,a_1,...,a_{n-1}\}$ a sequence of $\zo$-categories. We denote \wcnotation{$[\textbf{a},n]$}{((g00@$[\textbf{a},n]$} the colimit of the following diagram:
\[\begin{tikzcd}
	& 1 && 1 && 1 \\
	{[a_0,1]} && {[a_1,1]} && {...} && {[a_{n-1},1]}
	\arrow["{i_0^+}"', from=1-2, to=2-1]
	\arrow["{i_0^-}", from=1-2, to=2-3]
	\arrow["{i_0^+}"', from=1-4, to=2-3]
	\arrow["{i_0^-}", from=1-4, to=2-5]
	\arrow["{i_0^+}"', from=1-6, to=2-5]
	\arrow["{i_0^-}", from=1-6, to=2-7]
\end{tikzcd}\]

\p \label{para:les sommes glob}
We define \wcnotation{$\Theta$}{(theta@$\Theta$} as the smallest full subcategory of $\zocat$ that includes the terminal $\zo$-category $[0]$, and such that
for any non negative integer $n$, and any finite sequence $\textbf{a}:=\{a_0,a_1,...,a_{n-1}\}$ of objects of $\Theta$, it includes the $\zo$-category $[\textbf{a},n]$.
Objects of $\Theta$ are called \notion{globular sum}.

Remark that a morphism $g:[\textbf{a},n]\to [\textbf{b},m]$ is exactly the data of a morphism $f:[n]\to [m]$, and for any integer $i$, a morphism
$$a_i\to \prod_{f(i)\leq k< f(i+1)}b_k.$$

\begin{example}
\label{exemple:of globular sum}
For any $n$, $\Db_n$ is a globular sum. The $\zo$-category induced by the $\omega$-graph 
\[\begin{tikzcd}
	\bullet & \bullet & \bullet
	\arrow[from=1-1, to=1-2]
	\arrow[""{name=0, anchor=center, inner sep=0}, curve={height=-24pt}, from=1-2, to=1-3]
	\arrow[""{name=1, anchor=center, inner sep=0}, curve={height=24pt}, from=1-2, to=1-3]
	\arrow[""{name=2, anchor=center, inner sep=0}, from=1-2, to=1-3]
	\arrow[shorten <=3pt, shorten >=3pt, Rightarrow, from=0, to=2]
	\arrow[shorten <=3pt, shorten >=3pt, Rightarrow, from=2, to=1]
\end{tikzcd}\]
is a globular sum.
\end{example}

\p For a globular sum $a$ and an integer $n$, we define $[a,n]:=[\{a,a,...,a\},n]$.\ssym{((g10@$[a,n]$}{for $\zo$-categories}
For a sequence of integer $\{n_0,..,n_k\}$ and a sequence of globular sum $\{a_0,..,a_k\}$, we define \wcsnotation{$[a_0,n_0]\vee[a_1,n_1]\vee...\vee [a_k,n_k]$}{((g20@$[a_0,n_0]\vee[a_1,n_1]\vee...\vee [a_k,n_k]$}{for $\Theta$} as the globular sum $[\{a_0,..,a_1,...,a_k,...\},n_0+n_1+...+n_k]$.

We denote by $[0]$ the terminal $\io$-category, and $[n]$ the globular sum $[[0],n]$.
We have a fully faithful functor $\Delta\to \Theta$ sending $[n]$ onto $[n]$..

\p \label{para:reedy}
 A \notion{Reedy category} is a small category $A$ equipped with two subcategories $A_+$, $A_-$ and a \textit{degree} function $d:ob(A)\to \Nb$ such that: 
\begin{enumerate}
\item for every non identity morphism $f:a\to b$, if $f$ belongs to $A_-$, $d(a)>d(b)$, and if $f$ belongs to $A_+$, $d(a)<d(b)$.
\item every morphism of $A$ uniquely factors as a morphism of $A_-$ followed by a morphism of $A_+$.
\end{enumerate}

A Reedy category $A$ is \wcnotion{elegant}{elegant Reedy category} if for any presheaf $X$ on $A$, for any $a\in A$ and any $c\in X(a)$, there exists a unique morphism $f:a\to a'\in A_{-}$ and a unique non degenerate object $c'\in X(a')$ such that $c=X(f)(c')$. 

\begin{prop}
\label{prop:elelangat stable by slice}
Let $X$ be a presheaf on an elegant Reedy category $A$. The category $A_{/X}$ is an elegant Reedy category.
\end{prop}
\begin{proof}
We have a canonical projection $\pi:A_{/X}\to A$. A morphism is positive (resp. negative) if it's image by $\pi$ is. The degree of an element $c$ of $A_{/X}$ is the degree of $\pi(c)$. We leave it to the reader to check that this endows $A_{/X}$ with a structure of Reedy category. 

The fact that $A_{/X}$ is elegant is a direct consequence of the isomorphism $\Psh{A_{/X}}\cong \Psh{A}_{/X}$.
\end{proof}

\p We define by induction the \wcnotion{dimension}{dimension of a globular sum} of a globular sum $a$, denoted by $|a|$. The dimension of $[0]$ is $0$, and the dimension of $[\textbf{a},n]$ is the maximum of the set $\{|a_k|+1\}_{k< n}$. We denote by \wcnotation{$\Theta_n$}{(thetan@$\Theta_n$} the full subcategory of $\Theta$ whose objects are the globular sum of dimension inferior or equal to $n$.

\begin{prop}[Berger, Bergner-Rezk]
\label{prop:theta is elegan reedy}
The category $\Theta$ and, for any $n\in \Nb$, the category $\Theta_n$ are elegant Reedy category.

A morphism $g:[\textbf{a},n]\to [\textbf{b},m]$ is \wcnotion{degenerate}{degenerate morphism of $\Theta$} (i.e a morphism of $\Theta_{-}$) if the corresponding morphism $f:[n]\to [m]$ is a degenerate morphism of $\Delta$, and for any $i<n$ and any $f(i)\leq k<f(k+1)$, the corresponding morphism $a_i\to b_k$ is degenerate. Furthermore, a morphism is degenerate  if and only if it is a epimorphism in $\Psh{\Theta}$.

A morphism is in $\Theta^+$ if and only if it is a monomorphism in $\Psh{\Theta}$. 
\end{prop}
\begin{proof}
The Reedy structure is a consequence of lemma 2.4 of \cite{Berger_a_cellular_nerve}. The fact that for any $n<\omega$, $\Theta_n$ is elegant is 
\cite[corollary 4.5.]{Bergner_reedy_category_and_the_theta_construction}. As for any $n<\omega$, the inclusion $\Theta_n\to \Theta$ preserves strong pushout, the characterization of elegant Reedy category given by \cite[proposition 3.8.]{Bergner_reedy_category_and_the_theta_construction} implies that $\Theta$ is also elegant.
\end{proof}

\p 
\label{para:algebraic and globular}
We recall that a morphism $g:[\textbf{a},n]\to [\textbf{b},m]$ is exactly the data of a morphism $f:[n]\to [m]$, and for any integer $i$, a morphism
$$a_i\to \prod_{f(i)\leq k< f(i+1)}b_k.$$
The morphism $g$ is \wcsnotion{globular}{globular morphism}{for $\zo$-categories} if for any $k<n$, $f(k+1)=f(k)+1$ and the morphism $a_k\to b_k$ is globular. The morphism $g$ is \wcnotion{algebraic}{algebraic morphism of $\Theta$} if it cannot be written as a composite $ig'$ where $i$ is a globular morphism.

\begin{example}
The morphism 
\[\begin{tikzcd}
	\bullet & \bullet & \bullet && \bullet & \bullet & \bullet
	\arrow[from=1-5, to=1-6]
	\arrow[""{name=0, anchor=center, inner sep=0}, curve={height=-24pt}, from=1-6, to=1-7]
	\arrow[""{name=1, anchor=center, inner sep=0}, curve={height=24pt}, from=1-6, to=1-7]
	\arrow[""{name=2, anchor=center, inner sep=0}, from=1-6, to=1-7]
	\arrow[""{name=3, anchor=center, inner sep=0}, from=1-2, to=1-3]
	\arrow[""{name=4, anchor=center, inner sep=0}, curve={height=-24pt}, from=1-2, to=1-3]
	\arrow[from=1-1, to=1-2]
	\arrow[shorten <=14pt, shorten >=14pt, maps to, from=1-3, to=1-5]
	\arrow[shorten <=3pt, shorten >=3pt, Rightarrow, from=0, to=2]
	\arrow[shorten <=3pt, shorten >=3pt, Rightarrow, from=2, to=1]
	\arrow[shorten <=3pt, shorten >=3pt, Rightarrow, from=4, to=3]
\end{tikzcd}\]
is globular. This is not the case for the morphism
\[\begin{tikzcd}
	\bullet & \bullet & \bullet && \bullet & \bullet & \bullet
	\arrow[from=1-5, to=1-6]
	\arrow[""{name=0, anchor=center, inner sep=0}, curve={height=-24pt}, from=1-6, to=1-7]
	\arrow[""{name=1, anchor=center, inner sep=0}, curve={height=24pt}, from=1-6, to=1-7]
	\arrow[""{name=2, anchor=center, inner sep=0}, from=1-6, to=1-7]
	\arrow[""{name=3, anchor=center, inner sep=0}, curve={height=24pt}, from=1-2, to=1-3]
	\arrow[""{name=4, anchor=center, inner sep=0}, curve={height=-24pt}, from=1-2, to=1-3]
	\arrow[from=1-1, to=1-2]
	\arrow[shorten <=14pt, shorten >=14pt, maps to, from=1-3, to=1-5]
	\arrow[shorten <=3pt, shorten >=3pt, Rightarrow, from=0, to=2]
	\arrow[shorten <=3pt, shorten >=3pt, Rightarrow, from=2, to=1]
	\arrow[shorten <=6pt, shorten >=6pt, Rightarrow, from=4, to=3]
\end{tikzcd}\]
that sends the $2$-cell of the left globular sum on the $1$-composite of the two $2$-cells of the right globular sum.
\end{example}

\begin{prop}[{\cite[Proposition 3.3.10]{Ara_thesis}}]
\label{prop:algebraic ortho to globular}
Every morphism in $\Theta$ can be factored uniquely in an algebraic morphism followed by a globular morphism.
\end{prop}

\p 
\label{para:definition of source et but}
We define for any globular sum $a$ and any integer $n$ a globular sum $s_n(a):=:t_n(a)$ and two morphisms
$$s_n(a)\to a\leftarrow t_n(a).$$
We first set $s_0(a):=:t_0(a) :=[0]$. The inclusion $s_0(a)\to a$ corresponds to the initial point and $t_0(a)\to a$ to the terminal point.
 For $n>0$, we define $s_n([\textbf{a},n]):=:t_n([\textbf{a},n]) :=[s_{n-1}(\textbf{a}),n]$ where $s_{n-1}(\textbf{a})$ is the sequence 
 $\{s_{n-1}(a_i)\}_{i<n}$.

\begin{example}
If $a$ is the globular sum of example \ref{exemple:of globular sum}, we have:
\[\begin{tikzcd}
	{s_1(a):=} & \bullet & \bullet & \bullet \\
	\\
	{a:=} & \bullet & \bullet & \bullet \\
	\\
	{t_1(a):=} & \bullet & \bullet & \bullet
	\arrow[from=3-2, to=3-3]
	\arrow[""{name=0, anchor=center, inner sep=0}, curve={height=-24pt}, from=3-3, to=3-4]
	\arrow[""{name=1, anchor=center, inner sep=0}, curve={height=24pt}, from=3-3, to=3-4]
	\arrow[""{name=2, anchor=center, inner sep=0}, from=3-3, to=3-4]
	\arrow[curve={height=-24pt}, from=1-3, to=1-4]
	\arrow[from=1-2, to=1-3]
	\arrow[from=5-2, to=5-3]
	\arrow[curve={height=24pt}, from=5-3, to=5-4]
	\arrow[shorten <=13pt, shorten >=13pt, maps to, from=1-3, to=3-3]
	\arrow[shorten <=13pt, shorten >=13pt, maps to, from=5-3, to=3-3]
	\arrow[shorten <=3pt, shorten >=3pt, Rightarrow, from=0, to=2]
	\arrow[shorten <=3pt, shorten >=3pt, Rightarrow, from=2, to=1]
\end{tikzcd}\]
\end{example}

\p
\label{para:definition of W}
The morphism $[\uvar,1]:\Theta\to \Theta$ induces by extension by colimit a functor 
$$[\uvar,1]:\Psh{\Theta}\to \Psh{\Theta}.$$
We define by induction on $a$ a $\Theta$-presheaf \wcnotation{$\Sp_a$}{(sp@$\Sp_{a}$} and a morphism $\Sp_a\to a$. 
If $a$ is $[0]$, we set $\Sp_{[0]}:=[0]$. For $n>0$, we define 
$\Sp_{[\textbf{a},n]}$ as the set valued presheaf on $\Theta$ obtained as the colimit of the diagram
\[\begin{tikzcd}
	& 1 && 1 && 1 \\
	{[ \Sp_{a_0},1]} && {[\Sp_{a_1},1]} && \cdots && {[\Sp_{a_{n-1}},1]}
	\arrow["{i_0^-}", from=1-2, to=2-3]
	\arrow["{i_0^+}"', from=1-4, to=2-3]
	\arrow["{i_0^+}"', from=1-2, to=2-1]
	\arrow["{i_0^-}", from=1-6, to=2-7]
	\arrow["{i_0^-}", from=1-4, to=2-5]
	\arrow["{i_0^+}"', from=1-6, to=2-5]
\end{tikzcd}\]
We define \wcnotation{$E^{eq}$}{(eeq@$E^{eq}$} as the set valued preheaves on $\Delta$ obtained as the colimit of the diagram
\[\begin{tikzcd}
	& {[1]} && {[1]} && {[1]} \\
	{[0]} && {[2]} && {[2]} && {[0]}
	\arrow[from=1-2, to=2-1]
	\arrow["{d^1}", from=1-2, to=2-3]
	\arrow["{d^1}"', from=1-6, to=2-5]
	\arrow["{d^0}"', from=1-4, to=2-3]
	\arrow["{d^2}", from=1-4, to=2-5]
	\arrow[from=1-6, to=2-7]
\end{tikzcd}\]
For any integer $n$, the morphism $\Sigma^n:\Theta\to \Theta$, which is the $n$-iteration of $[\uvar,1]$, induces by colimit a functor \ssym{(sigma@$\Sigma^n$}{for $\io$-categories}
$$\Sigma^n:\Psh{\Theta}\to \Psh{\Theta}.$$
 We define two sets of morphisms of $\Psh{\Theta}$:\sym{(w@$\W$}\sym{(wseg@$\Wseg$}\sym{(wsat@$\Wsat$}
$$\Wseg := \{\Sp_a\to a,~a\in\Theta\}~~~~\Wsat:= \{\Sigma^n E^{eq}\to \Db_n\}$$
and we set $$\W:=\Wseg\cup \Wsat.$$
For any $n$, we also define $$\mbox{$\W_n$}:= \W\cap \Theta_n.$$

\p
\label{para:defi of delta theta}
We recall that for an integer $n$ and a globular sum $a$, we defined $[a,n]:=[\{a,a,...,a\},n]$.
 This defines a functor $i:\Delta[\Theta] \to \Theta$
sending $(n,a)$ on $[a,n]$ where 
 \wcnotation{$\Delta[\Theta]$}{(deltaTheta@$\Delta[\Theta]$} is the following pushout of category: 
\[\begin{tikzcd}
	{\{[0]\}\times \Theta} & \Delta\times\Theta \\
	1 & {\Delta[\Theta]}
	\arrow[from=1-1, to=1-2]
	\arrow[from=1-1, to=2-1]
	\arrow[from=1-2, to=2-2]
	\arrow[from=2-1, to=2-2]
	\arrow["\lrcorner"{anchor=center, pos=0.125, rotate=180}, draw=none, from=2-2, to=1-1]
\end{tikzcd}\]
For the sake of simplicity, we will also denote by $[a,n]$ (resp. $[n]$) the object of $\Delta[\Theta]$ corresponding to $(n,a)$ (resp. to $(n,[0])$).
 We define two sets of morphisms:\sym{(m@$\M$} \sym{(mseg@$\Mseg$}\sym{(msat@$\Msat$}
$$
\begin{array}{c}
\Mseg := \{[a,\Sp_n]\to [a,n],~a:\Theta\}\cup\{[f,1],~f\in \Wseg\}\\
\Msat:= \{E^{eq}\to [0]\} \cup\{[f,1],~f\in \Wsat\}
\end{array}$$
and we set $$\M := \Mseg \cup \Msat.$$

For an integer $n$, we define \wcnotation{$\Delta[\Theta_n]$}{(deltaThetan@$\Delta[\Theta_n]$} as the following pushout of category: 
\[\begin{tikzcd}
	{\{[0]\}\times\Theta_n} & {\Delta\times\Theta_n} \\
	1 & {\Delta[\Theta_n]}
	\arrow[from=1-1, to=1-2]
	\arrow[from=1-1, to=2-1]
	\arrow[from=1-2, to=2-2]
	\arrow[from=2-1, to=2-2]
	\arrow["\lrcorner"{anchor=center, pos=0.125, rotate=180}, draw=none, from=2-2, to=1-1]
\end{tikzcd}\]
and the functor $i$ induces a functor $\Delta[\Theta_n]\to \Theta_{n+1}$.
For any $n$, we define $$\mbox{$\M_n$}:= \M\cap \Delta[\Theta_n].$$

\p Let $C$ be a presentable category and $S$ a set of monomorphisms with small codomains. An object $x$ is \textit{$S$-local} if for any $i:a\to b\in S$, the induced functor $\Hom(i,x):\Hom(b,x)\to \Hom(a,x)$ is an isomorphism. 
We define \textit{$C_{S}$} as the full subcategory of $C$ composed of $S$-local objects.
According to theorem \ref{theo:adjunction between presheaves and local presheaves}, the inclusion $\iota:C_S\to C$ is part of an adjunction
\[\begin{tikzcd}
	{\Fb_S:C} & {C_S:\iota}
	\arrow[""{name=0, anchor=center, inner sep=0}, shift left=2, from=1-1, to=1-2]
	\arrow[""{name=1, anchor=center, inner sep=0}, shift left=2, from=1-2, to=1-1]
	\arrow["\dashv"{anchor=center, rotate=-90}, draw=none, from=0, to=1]
\end{tikzcd}\]
Moreover, the theorem \textit{op cit} also states that $\Fb_S:C\to C_S$ is the localization of $C$ by the smallest class of morphisms containing $S$ and stable under composition and colimit.

Suppose given an other category $D$ fitting in an adjunction
\[\begin{tikzcd}
	{F:C} & {D:G}
	\arrow[""{name=0, anchor=center, inner sep=0}, shift left=2, from=1-1, to=1-2]
	\arrow[""{name=1, anchor=center, inner sep=0}, shift left=2, from=1-2, to=1-1]
	\arrow["\dashv"{anchor=center, rotate=-90}, draw=none, from=0, to=1]
\end{tikzcd}\]
with unit $\nu$ and counit $\epsilon$,
as well as a set of morphisms $T$ of $D$ such that $F(S)\subset T$. 
By adjunction property, it implies that for any $T$-local object $d\in D$, $G(d)$ is $S$-local.
The previous adjunction induces a derived adjunction
\[\begin{tikzcd}
	{\Lb F:C_S} & {D_T:\Rb G}
	\arrow[""{name=0, anchor=center, inner sep=0}, shift left=2, from=1-1, to=1-2]
	\arrow[""{name=1, anchor=center, inner sep=0}, shift left=2, from=1-2, to=1-1]
	\arrow["\dashv"{anchor=center, rotate=-90}, draw=none, from=0, to=1]
\end{tikzcd}\]
where $\Lb F$ is defined by the formula $c\mapsto \Fb_T F(c)$ and $\Rb G$ is the restriction of $G$ to $D_T$. The unit is given by $\nu\circ \Fb_S$ and the counit by the restriction of $\epsilon$ to $D_T$.

\p The functor $i:\Delta[\Theta]\to \Theta$ defined in paragraph \ref{para:defi of delta theta} induces an adjunction:
$$
\begin{tikzcd}
	{ i_!:\Psh{\Delta[\Theta]}} & {\Psh{\Theta}:i^*}
	\arrow[shift left=2, from=1-1, to=1-2]
	\arrow[shift left=2, from=1-2, to=1-1]
\end{tikzcd}
$$
where the left adjoint is the left Kan extension of the functor $\Delta[\Theta]\to \Theta\to \Psh{\Theta}$.
Remark that there is an obvious inclusion $i_!(\M)\subset \W$. In virtue of the last paragraph, this induces an adjunction between derived categories:
\begin{equation}
\label{eq:derived adjunction strict}
\begin{tikzcd}
	{\Lb i_!:\Psh{\Delta[\Theta]}_{\M}} & {\Psh{\Theta}_{\W}:\Rb i^*}
	\arrow[shift left=2, from=1-1, to=1-2]
	\arrow[shift left=2, from=1-2, to=1-1]
\end{tikzcd}
\end{equation}
The corollary 12.3 of \cite{Barwick_on_the_unicity_of_the_theory_of_higher_categories} and  the corollary \ref{cor:changing theta} (which is proved in the next section) induce equivalences
$$\zocat\cong \Psh{\Theta}_{\W}\cong \Psh{\Delta[\Theta]}_{\M}.$$

Similarly, for any integer $n$, the inclusion $i:\Delta[\Theta_n]\to \Theta_{n+1}$ induces an adjunction between derived categories:
\begin{equation}
\label{eq:derived adjunction case n strict}
\begin{tikzcd}
	{\Lb i_!:\Psh{\Delta[\Theta]_n}_{\M_n}} & {\Psh{\Theta_{n+1}}_{\W_n}:\Rb i^*}
	\arrow[shift left=2, from=1-1, to=1-2]
	\arrow[shift left=2, from=1-2, to=1-1]
\end{tikzcd}
\end{equation}
and corollary 12.3 of \cite{Barwick_on_the_unicity_of_the_theory_of_higher_categories} and corollary \ref{cor:changing theta} induce equivalences
$$\zncat{n+1}\cong \Psh{\Theta_{n+1}}_{\W_{n+1}}\cong \Psh{\Delta[\Theta_n]}_{\M_{n+1}}.$$

\subsection{The link between presheaves on $\Theta$ and on $\Delta[\Theta]$}
\p 
\label{para:precomplet}
A class of monomorphism $T$ is \wcnotionsym{precocomplete}{(ss@$\overline{S}$}{precocomplete set of arrows} if
\begin{enumerate}
\item[$-$] It is closed by transfinite compositions and pushouts.
\item[$-$] It is closed by left \textit{cancellation}, i.e for any pair of composable morphisms $f$ and $g$, if $gf$ and $f$ are in $S$ , so is $g$.
\item[$-$] For any elegant Reedy category $A$, and any functor $F:A\to \Arr(C)$ such that the induced morphism $\colim_{\partial a}F\to F(a)$ is a monomorphism for any object $a$, and such that  $F$ is pointwise in $S$, then $\colim_AF$ is in $S$.
\end{enumerate} 
For a set of morphisms $S$, we denote $\overline{S}$ the smallest precocomplete class of morphisms containing $S$.

\p
The aim of this subsection is to demonstrate the following proposition:

\begin{theorem}
\label{theo:unit and counit are in W}
For any $a\in \Theta$ and $b\in \Delta[\Theta]$, morphisms $i_!i^*a\to a$ and $b\to i^*i_! b$ are respectively in $\overline{\W}$ and $\overline{\M}$.
\end{theorem}

As a corollary, we directly have:
\begin{cor}
\label{cor:changing theta}
The adjunction 
$$\begin{tikzcd}
	{\Lb i_!:\Psh{\Delta[\Theta]}_{\M}} & {\Psh{\Theta}_{\W}:\Rb i^*}
	\arrow[shift left=2, from=1-1, to=1-2]
	\arrow[shift left=2, from=1-2, to=1-1]
\end{tikzcd}$$
given in \eqref{eq:derived adjunction strict} is an adjoint equivalence. For any integer $n$, the adjunction 
$$\begin{tikzcd}
	{\Lb i_!:\Psh{\Delta[\Theta]_n}_{\M_n}} & {\Psh{\Theta_{n+1}}_{\W_n}:\Rb i^*}
	\arrow[shift left=2, from=1-1, to=1-2]
	\arrow[shift left=2, from=1-2, to=1-1]
\end{tikzcd}$$
given in \eqref{eq:derived adjunction case n strict} is an adjoint equivalence. 
\end{cor}
\begin{proof}
The first assertion is a consequence of theorem \ref{theo:unit and counit are in W} and of the fact that $\overline{\W}$ (resp. $\overline{\M}$) is a included in the smallest class containing $\W$ (resp. $\M$) and stable by two out of three and colimits.
We prove the second assertion similarly.
\end{proof}

\p 
We denote by 
$$[\uvar,\uvar]: \Psh{\Theta}\times \Psh{\Delta}\to \Psh{\Delta[\Theta]}$$
the extension by colimit of the functor $\Theta\times \Delta\to \Psh{\Delta[\Theta]}$ sending $(a,n)$ onto $[a,n]$.
For an integer $n$, we denote
$$[\uvar,n]:\Psh{\Theta}^n\to \Psh{\Theta}$$ 
the extension by colimit of the functor 
$\Theta^n\to\Psh{\Theta}$ sending $\textbf{a}:=\{a_1,...,a_n\}$ onto $[\textbf{a},n]$. Eventually, we define 
$$[\uvar,d^0\cup d^n]:\Psh{\Theta}^n\to \Psh{\Theta}$$ 
the extension by colimit of the functor 
$\Theta^n\to\Psh{\Theta}$ sending $\textbf{a}:=\{a_1,...,a_n\}$ onto the colimit of the span.
$$[\{a_0,...,a_{n-2}\},{n-1}]\leftarrow [\{a_1,...,a_{n-2}\},{n-2}]\to [\{a_1,...,a_{n-1}\},{n-1}]$$

\begin{lemma}
\label{lemma:the functor [] preserves classes}
The image of $\overline{\W}\times \overline{\W_1}$ by the functor $[\uvar,\uvar]:\Psh{\Theta}\times \Psh{\Delta}\to \Psh{\Delta[\Theta]}$ is included in $\overline{\W}$.
\end{lemma}
\begin{proof}
As $[\uvar,\uvar]$ preserves colimits and monomorphisms, it is enough to show that for any pair $f,g\in \W\times \W_1$, $[f,g]$ is in $\W$ which is obvious.
\end{proof}

\begin{lemma}
\label{lemma:i etoile of W is in M 0}
For any globular sum $v$, and any integer $n$,
the morphism $[v,d^0\cup d^n]\cup[\partial v,n]\to [v,n]$ appearing in the diagram
\[\begin{tikzcd}
	{[\partial v,d^0\cup d^n]} & {[v,d^0\cup d^n]} \\
	{[\partial v,n]} & {[\partial v,n]\cup[v,d^0\cup d^n]} \\
	&& {[ v,n]}
	\arrow[from=1-1, to=2-1]
	\arrow[from=1-2, to=2-2]
	\arrow[from=2-2, to=3-3]
	\arrow[curve={height=18pt}, from=2-1, to=3-3]
	\arrow[curve={height=-18pt}, from=1-2, to=3-3]
	\arrow[from=1-1, to=1-2]
	\arrow[from=2-1, to=2-2]
\end{tikzcd}\]
is in $\overline{\M}$.
\end{lemma}
\begin{proof}
Let $a$ be a globular sum.
Remark that the morphism $[a,\Sp_n]\to [a,d^0\cup d^n]$ is in $\overline{\M}$. By left cancellation, this implies that $[a,d^0\cup d^n]\to [a,n]$ is in $\overline{\M}$. For any presheaf $X$ on $\Theta$, $\Theta_{/X}$ is an elegant Reedy category, and $[X,d^0\cup d^n]\to [X,n]$ is then in $\overline{\M}$. In particular, $[\partial v,d^0\cup d^n]\to [\partial v,n]$ is in $\overline{\M}$, and so is $[v,d^0\cup d^n]\to [\partial v,n]\cup[v,d^0\cup d^n]$ by stability by coproduct. A last use of the stability by left cancellation then concludes the proof.
\end{proof}

\p 
Let  $[b,m]$ be an element of $\Delta[\Theta]$. We denote $\Hom^*(i([b,m]),[\textbf{a},n])$ the subset of $\Hom(i([b,m]),[\textbf{a},n])$ that consists of morphisms that preserve extremal objects. The explicit expression of morphism in $\Theta$ implies the bijection:
\begin{equation}
\label{eq:hom in theta}
\Hom_{\Theta}^*(i([b,m]),[\textbf{a},n])\cong\Hom_{\Delta}([n],[m])^*\times \prod_{i<n}\Hom_{\Theta}(b,a_i)
\end{equation}
where $\Hom_{\Delta}^*([n],[m])$ is the subset of $\Hom_{\Delta}([n],[m])$ consisting of morphisms that preserve extremal objects.

Let $\textbf{a}:=\{a_0,a_1,...,a_{n-1}\}$ be a finite sequence of globular sums. We define $\Theta^{\hookrightarrow}_{/\textbf{a}}$ as the category whose objects are collections of maps $\{b\to a_i\}_{ i< n}$ such that there exists no degenerate morphism $b\to b'$ factorizing all $b\to a_i$. Morphisms are monomorphisms $b\to b'$ making all induced triangles commute.

The bijection \eqref{eq:hom in theta}  induces a bijection between the objects of $\Theta^{\hookrightarrow}_{/\textbf{a}}$ and the morphisms $[b,n]\to i^*[\textbf{a},n]$ that are the identity on objects and that can not be factored through any 	degenerate morphism $[b,n]\to [\tilde{b},n]$.

\begin{lemma}
\label{lemma:i etoile of W is in M 1}
For any morphism $p:[b,m]\to i^*[\textbf{a},n]$ in $\Psh{\Delta[\Theta]}$ that preserves extremal objects, there exists a unique pair $(\{b'\to a_i\}_{i<n},[f,i]:[b,m]\to [b',n])$ where $\{b'\to a_i\}_{i<n}$ is an element of $\Theta^{\hookrightarrow}_{/\textbf{a}}$, $f$ is a degenerate morphism, and such that the induced triangle
\[\begin{tikzcd}
	{[b,m]} & {[b',n]} \\
	& {i^*[\textbf{a},n]}
	\arrow["{[f,i]}", from=1-1, to=1-2]
	\arrow["{p'}", from=1-2, to=2-2]
	\arrow["p"', from=1-1, to=2-2]
\end{tikzcd}\]
commutes.
\end{lemma}
\begin{proof}
By adjunction and thanks to the bijection \eqref{eq:hom in theta}, $p$ corresponds to a pair $(j:[m]\to [n], \{b\to a_i\}_{i<n})$, and $i$ has to be equal to $j$.

Using once again this bijection, and the fact that degeneracies are epimorphisms, we have to show that there exists a unique degenerate morphism $g:b\to b'$ that factors the morphisms $b\to a_i$ for all $i<n$, and such that the induced family of morphisms $\{b'\to a_i\}_{i<n}$ is an element of $\Theta^{\hookrightarrow}_{/\textbf{a}}$.

As any infinite sequence of degenerate morphisms is constant at some point, the existence is immediate.

Suppose given two morphisms $b\to b'$, $b\to b''$ fulfilling the previous condition.
The proposition 3.8 of \cite{Bergner_reedy_category_and_the_theta_construction} implies that there exists a globular sum $\tilde{b}$ and two degenerate morphisms $b'\to \tilde{b}$ and $b''\to \tilde{b}$  such that the induced square
\[\begin{tikzcd}
	b & {b'} \\
	{b''} & {\tilde{b}}
	\arrow[from=1-1, to=2-1]
	\arrow[from=2-1, to=2-2]
	\arrow[from=1-1, to=1-2]
	\arrow[from=1-2, to=2-2]
\end{tikzcd}\]
is cartesian. The universal property of pushout implies that $b\to \tilde{b}$ also fulfills the previous condition. By definition of $b'$ and $b''$, this implies that they are equal to $\tilde{b}$, and this shows the uniqueness.
\end{proof}

\begin{lemma}
\label{lemma:i etoile of W is in M 0.5}
Let $\{b\to a_i\}_{ i< n}$ be an element of  $\Theta^{\hookrightarrow}_{/\textbf{a}}$  and $i:b'\to b$ a monomorphism of $\Theta$. The induced family $\{b'\to b\to a_i\}_{ i< n}$ is an object of $\Theta^{\hookrightarrow}_{/\textbf{a}}$. 
\end{lemma}
\begin{proof}
The lemma \ref{lemma:i etoile of W is in M 1} implies that there exists  a  unique degenerate morphism $j:b'\to \tilde{b}$ that factors all the morphism $b'\to b\to  a_i$ for $i<n$, and such the induced family of morphisms $\{\tilde{b}\to a_i\}_{i<n}$ is an element of $\Theta^{\hookrightarrow}_{/\textbf{a}}$.  We proceed by contradiction, and we then suppose that $j$ is different from the identity.

We then have, for any $i<n$, a commutative square
\[\begin{tikzcd}
	{b'} & b \\
	{\tilde{b}} & {a_i}
	\arrow["i", from=1-1, to=1-2]
	\arrow[from=1-2, to=2-2]
	\arrow["j"', from=1-1, to=2-1]
	\arrow[from=2-1, to=2-2]
\end{tikzcd}\]
As the morphism $j$ is degenerate and different of the identity, there exists an integer $k$ and a non trivial $k$-cell $d$ of $b'$ that is sent to an identity by $j$. Now, let $d'$ be a $k$-generator  of the polygraph $b$ that appears in the decomposition of $i(d)$. The commutativity of the previous square and the fact that the $\zo$-categories $a_i$ are polygraphs implies that for any $i$, the $k$-cell $a'$ is sent to an identity by the morphism $b\to a_i$.  As for any $i< n$ and any $l\geq k$,  there is no non trivial $l$-cell in $a_i$ whose $(k-1)$-source and $(k-1)$-target are the same, this implies that every $l$-cell of $b$ that is $(k-1)$-parallel with $d'$ is send to the identity by the morphism $b\to a_i$.

We denote $\bar{b}$ the globular sum obtained by crushing all $l$-cells of $b$ that are $(k-1)$-parallel with $d'$. The induced degenerate morphism $b\to \bar{b}$ factors all the morphisms $b\to a_i$ which is in contradiction with the fact that  $\{{b}\to a_i\}_{i<n}$ is an element of $\Theta^{\hookrightarrow}_{/\textbf{a}}$.
\end{proof}

\p 
We say that an element $\{v\to a_i\}_{i<n}$ in the category $\Theta^{\hookrightarrow}_{/\textbf{a}}$ is \textit{of height $0$} if $v\to a_0$ factors through $\partial a_0$ or $v\to a_{n-1}$ factors through $\partial a_{n-1}$. The \textit{height of an element $w$} is the maximal integer $m$ such that there exists a sequence 
$v_0\to v_1\to...\to v_m=w$ in $\Theta^{\hookrightarrow}_{/\textbf{a}}$ with $v_i\neq v_{i+1}$ for any $i<m$ and such that $v_0$ is of height $0$ and $v_1$ is not. As $\Theta$ is a Reedy category, all elements have finite height.

\begin{lemma}
\label{lemma:i etoile of W is in M 1.5}
For any morphism $p:[b,m]\to i^*[\textbf{a},n]$ that preserves extremal objects, there exists a unique integer $k$, a unique element $\{b'\to a_i\}_{i<n}$ of height $k$, and a unique morphism $[f,i]:[b,m]\to [b',n]$ that doesn't factors through $[\partial b',n]$, and such that the induced triangle
\[\begin{tikzcd}
	{[b,m]} & {[b',n]} \\
	& {i^*[\textbf{a},n]}
	\arrow["{p'}", from=1-2, to=2-2]
	\arrow["{[f,i]}", from=1-1, to=1-2]
	\arrow[from=1-1, to=2-2]
\end{tikzcd}\]
commutes.

If $\{\tilde{b}\to a_i\}_{i<n}$ is any other object of non negative height, and $[\tilde{f},j]:[b,m]\to [\tilde{b},n]$ is a morphism that make the induced triangle
\[\begin{tikzcd}
	{[b,m]} & {[\tilde{b},n]} \\
	& {i^*[\textbf{a},n]}
	\arrow["{\tilde{p}}", from=1-2, to=2-2]
	\arrow["{[\tilde{f},j]}", from=1-1, to=1-2]
	\arrow[from=1-1, to=2-2]
\end{tikzcd}\]
commutative, then $\{\tilde{b}\to a_i\}_{i<n}$ is of height strictly superior to $k$ and $[\tilde{f},j]$ factors through $[\partial\tilde{b},n]$.
\end{lemma}
\begin{proof}
The lemma \ref{lemma:i etoile of W is in M 1} implies the first assertion. For the second one, suppose given an object $\{\tilde{b}\to a_i\}_{i<n}$ of non negative height and a morphism  $[\tilde{f},j]:[b,m]\to [\tilde{b},n]$ fulfilling the desired condition. The bijection \eqref{eq:hom in theta} directly implies that $j$ is equal to $i$, and the first assertion  implies that $\tilde{f}$ is non degenerate.

We can then factor $\tilde{f}:b\to \tilde{b}$ in a degenerate morphism $b\to \bar{b}$ followed by a monomorphism $ \bar{b}\to \tilde{b}$ which is not the identity. The lemma \ref{lemma:i etoile of W is in M 0.5} then implies that $\{\bar{b}\to \tilde{b}\to a_i\}_{i<n}$ is an element of  $\Theta^{\hookrightarrow}_{/\textbf{a}}$. The first assertion then implies that the two morphisms $[b,m]\to [b',n]$ and $[b,m]\to {[\bar{b},n]}$ are equals. As the  monomorphism $  {[{b}',n]}={[\bar{b},n]}\to [\tilde{b},n]$ is not the identity, this concludes the proof.
\end{proof}

\begin{lemma}
\label{lemma:i etoile of W is in M 2}
The morphism $i^*[\partial^0\textbf{a},n]\cup i^*[\partial^{n-1}\textbf{a},n]\to i^*[\textbf{a},n]$ is in $\overline{\M}$, where $\partial^j\textbf{a}$ corresponds to the sequence $\{a_1,..,\partial a_j,..,a_n\}$. 
\end{lemma}
\begin{proof}
For $k\in\Nb\cup\{\infty\}$, we define $x_k$ as the smallest sub object of $i^*[\textbf{a},n]$ such that for any element 
of height inferior or equal to $k$ of $\Theta^{\hookrightarrow}_{/\textbf{a}}$, the corresponding morphism $[b,n]\to i^*[\textbf{a},n]$ factors through $x_k$. In particular we have $x_0= i^*[\partial^0\textbf{a},n]\cup i^*[\partial^{n-1}\textbf{a},n]$, and the lemma \ref{lemma:i etoile of W is in M 1} implies that $x_{\infty} =i^*[\textbf{a},n]$. 

Every morphism $[b,m]\to i^*[\textbf{a},n]$ that does not preserve extremal points then factors through $x_0$. 
The lemma \ref{lemma:i etoile of W is in M 1.5} implies that for any integer $k$, the canonical square 
\begin{equation}
\label{eq:lemma:i etoile of W is in M 2}
\begin{tikzcd}
	{\coprod_{(\Theta^{\hookrightarrow}_{/\textbf{a}})_{k+1}}[b,d^0\cup d^n]\cup[\partial b,n]} & {x_k} \\
	{\coprod_{(\Theta^{\hookrightarrow}_{/\textbf{a}})_{k+1}}[b,n]} & {x_{k+1}}
	\arrow[from=1-1, to=2-1]
	\arrow[from=2-1, to=2-2]
	\arrow[from=1-1, to=1-2]
	\arrow[from=1-2, to=2-2]
\end{tikzcd}
\end{equation}
is cocartesian.  The lemma \ref{lemma:i etoile of W is in M 0} and the stability under pushout of $\overline{\M}$  imply that $x_k\to x_{k+1}$ is in $\overline{\M}$.
 As $i^*[\textbf{a},n]$ is the transfinite composition of the sequence $x_0\to x_1\to...$, this implies that $x_0\to i^*[\textbf{a},n]$ is in $\overline{\M}$ which conclude the proof.
\end{proof}

\begin{lemma}
The morphism $i^*\Sp_a\to i^*a$ is in $\overline{\M}$ for any globular sum $a$.
\end{lemma}
\begin{proof}
Let $[\textbf{a},n]:= a$. As $\overline{\M}$ is closed under pushouts and composition, lemma \ref{lemma:i etoile of W is in M 2} implies that the morphism 
$$i^*[\{a_0,...,a_{n-2}\},n-1]\cup i^*[\{a_1,...,a_{n-1}\},n-1]\to i^*[\textbf{a},n]$$
is in $\widehat{\M}$. 
An easy induction on $n$ shows that this is also the case for the morphism 
$$[a_0,1]\cup... \cup [a_{n-1},1]= i^*[a_0,1]\cup... \cup i^*[a_{n-1},1]\to i^*[\textbf{a},n].$$
Now remark that $i^*\Sp_{[\textbf{a},n]}$ is equivalent to 
$$[\Sp_{a_0},1]\cup... \cup [\Sp_{a_{n-1}},1].$$
As the morphisms $[\Sp_i,1]\to [a_i,1]$ are by definition in $\M$, this concludes the proof.
\end{proof}

\begin{prop}
\label{prop:i etoile of W is in M}
There is an inclusion $i^*\W\subset \overline{\M}$.
\end{prop}
\begin{proof}
For Segal extensions, this is precisely the content of the last lemma. For saturation extensions, remark that $i^*\Wsat = \Msat$.
\end{proof}

\begin{proof}[Proof of theorem \ref{theo:unit and counit are in W}]
Let $a$ be a globe. We then have $i_!i^*a = a$. Suppose now that $a$ is any globular sum. We then have a commutative diagram
\[\begin{tikzcd}
	{i_!i^*\Sp_a} & {\Sp_a} \\
	{i_!i^*a} & a
	\arrow[from=1-1, to=2-1]
	\arrow[Rightarrow, no head, from=1-1, to=1-2]
	\arrow[from=1-2, to=2-2]
	\arrow[from=2-1, to=2-2]
\end{tikzcd}\]
where the upper horizontal morphism is an identity.
The proposition \ref{prop:i etoile of W is in M} and the fact that $i_!(\M)\subset \W$ implies that the vertical morphisms of the previous diagram are in $\overline{\W}$. By left cancellation, this implies that $i_!i^*a\to a$ belongs to $\overline{\W}$ for any globular sum. We proceed analogously to show that for any $b\in \Delta[\Theta]$, $b\to i^*i_! b$ is in $\overline{\M}$.
\end{proof}

\section{Gray Operations}
\subsection{Recollection on Steiner theory}
\label{section:Steiner thery} 

We present here the Steiner theory developed in \cite{Steiner_omega_categories_and_chain_complexes}.

\p
An augmented directed complex $(K,K^*,e)$ is given by a complex of abelian groups $K$, with an augmentation $e$: $$\Zb \xleftarrow{e} K_0 \xleftarrow{\partial_0} K_1 \xleftarrow{\partial_1} K_2 \xleftarrow{\partial_2} K_3 \xleftarrow{\partial_3}. .. $$
and a graded set $K^* = (K^*_n)_{n\in\Nb}$ such that for any $n$, $K_n^*$ is a submonoid of $K_n$. A morphism of directed complexes between $(K,K^*,e)$ and $(L,L^*,e')$ is given by a morphism of augmented complexes of abelian groups $f : (K,e)\to (L,e')$ such that $f(K^*_n)\subset L^*_n$ for any $n$. We note by \wcnotation{$\CDA$}{(adc@$\CDA$} the category of augmented directed complexes. 

Steiner then constructs an adjunction
\[\begin{tikzcd}
	{\lambda:\omegacat} & {\CDA:\nu}
	\arrow[""{name=0, anchor=center, inner sep=0}, shift left=2, from=1-1, to=1-2]
	\arrow[""{name=1, anchor=center, inner sep=0}, shift left=2, from=1-2, to=1-1]
	\arrow["\dashv"{anchor=center, rotate=-90}, draw=none, from=0, to=1]
\end{tikzcd}\]
The functor $\lambda$ is the simplest to define: \sym{(lambda@$\lambda:\omegacat\to \CDA$}

\begin{definition}
Let $C$ be a $\omega$-category.
We denote by $(\lambda C)_n$ the abelian group generated by the set $\{[x]_n: x\in C_n\}$ and the relations
$$[x*_m y]_n \sim [x]_n + [y]_n \mbox{ for $m<n$ }.$$
We define the morphism $\partial_n: (\lambda C)_{n+1}\to (\lambda C)_n$ on generators by the formula:
$$\partial_n([x]_{n+1}) := [d_n^+ x]_{n} - [d_n^- x]_{n}.$$
\end{definition}
We can easily check that the morphism $\partial$ is a differential. We define an augmentation $e:(\lambda C)_{0}\to \Zb$ by setting $e([x]_0) = 1$ on generators. 
We denote by $(\lambda C)_n^*$ the additive submonoid generated by the elements $[x]_n$. We then set:
$$\lambda C := (\{(\lambda C)_n \}_{n\in \Nb},\{(\lambda C)^*_n \}_{n\in \Nb},e ).$$ This assignation lifts to a functor:
$$\begin{array}{ccccc}
\lambda &:& \omegacat&\to&\CDA\\
&&C&\mapsto&\lambda C.
\end{array}$$
\begin{example}~
\begin{enumerate}
\item
For any integer $n$, $\lambda\Db_n$ is the augmented directed complex whose underlying chain complex is given by:
$$
\Zb\xleftarrow{e}
\Zb[e_0^-,e_0^+] \xleftarrow{\partial_0}
... \xleftarrow{\partial_{n-2}}
\Zb[e_{n-1}^-,e_{n-1}^+] \xleftarrow{\partial_{n-1}}
\Zb[e_{n}] \xleftarrow{\partial_{n}}
0\leftarrow ...$$
where for any $0<k<n$ and $\alpha\in\{-,+\}$
$$e(e_0^\alpha)=1~~~\partial_{k-1}(e_k^\alpha)= e_{k-1}^+-e_{k-1}^-~~~\partial_{n-1}(e_n)= e_{n-1}^+-e_{n-1}^-.$$
\item
The augmented directed complex $\lambda[n]$ has for underlying chain complex:
$$
\Zb\xleftarrow{e}
\Zb[v_0,v_1,...,v_n] \xleftarrow{\partial_0}
\Zb[v_{0,1},v_{1,2}...,v_{n-1,n}] \xleftarrow{\partial_{1}}
0\leftarrow ...$$
where for any $k<n$ and $\alpha\in\{-,+\}$
$$e(v_k)=e(v_n)=1~~~ \partial_{1}(v_{k,k+1})=v_{k+1}-v_k.$$
\end{enumerate}
\end{example}

\p We now define the functor $\nu:\CDA\to \omegacat$. Throughout, we fix an augmented directed complex $(K,K^*,e)$.
A \textit{Steiner array} (or simply a \notion{array}) of dimension $n$ is the data of a finite double sequence: \sym{(nu@$\nu:\CDA\to \omegacat$}
$$\left(\begin{matrix}
x^-_0 &x^-_1&x^-_2&x^-_3 &...&x_n^-\\
x^+_0 &x^+_1&x^+_2&x^+_3 &...&x_n^+
\end{matrix}\right)$$
such that
\begin{enumerate}
\item $x^-_n=x^+_n$;
\item For any $i\leq n$ and $\alpha\in\{-,+\}$, $x_i^\alpha$ is an element of $K^*_i$;
\item For any $0<i\leq n$, $\partial_{i-1}(x_i^\alpha)= x_{i-1}^+ - x_{i-1}^-$;
\end{enumerate}
An array is said to be \wcnotion{coherent}{coherent array} if $e(x^+_0) = e(x^-_0) = 1$.
\begin{definition}
We define the globular set $\nu K$, whose $n$-cells are the coherent arrays of dimension $n$. The source and target maps are defined for $k<n$ by the formula: 

$$d^\alpha_k\begin{pmatrix}
x^-_0 &x^-_1&x^-_2&...&x^-_n\\
x^+_0 &x^+_1&x^+_2&...& x^+_n
\end{pmatrix} = \begin{pmatrix}
x^-_0 &x^-_1&x^-_2&...& x^-_{k-1}&x^\alpha_k\\
x^+_0 &x^+_1&x^+_2&...& x^+_{k-1}&x^\alpha_k\end{pmatrix}$$

There is an obvious group structure on the arrays:
$$\begin{pmatrix}
x^-_0 &x^-_1&...& x^-_n\\
x^+_0 &x^+_1&...& x^+_n
\end{pmatrix}
+
\begin{pmatrix}
y^-_0 &y^-_1&...& y^-_n\\
y^+_0 &y^+_1&...& y^+_n
\end{pmatrix}
=
\begin{pmatrix}
x^-_0+y^-_0 &x^-_1+ y^-_1&...&x^-_n+ y^-_n \\
x^+_0+y^+_0 &x^+_1+ y^+_1&...&x^+_n +y^+_n 
\end{pmatrix}
$$
\label{defi:definition of composition and units of nu k}

\begin{itemize}
\item[$-$]
For two coherent arrays $x$ and $y$ such that $d^-_k(x) =d^+_k(y) = z$, we define their $k$-composition by the following formula: 
$$x*_k y := x- z + y .$$ More explicitly:
$$\begin{pmatrix}
x^-_0 &...& x^-_n\\
x^+_0 &...& x^+_n
\end{pmatrix}
*_k
\begin{pmatrix}
y^-_0 &...& y^-_n\\
y^+_0 &...& y^+_n
\end{pmatrix}
 := 
\begin{pmatrix}
y^-_0&...&y_k^-& y_{k+1}^- + x_{k+1}^- & ...& y_{n}^- + x_{n}^-\\
x^+_0 &...&x_k^+& y_{k+1}^+ + x_{k+1}^+ & ...& y_{n}^+ + x_{n}^+ 
\end{pmatrix}
$$
\item[$-$]
For an integer $m>n$, we define the $m$-sized array $1^m_x$ as follows:
$$1^m_x :=
\begin{pmatrix}
x^-_0 &...& x^-_n& 0 &...&0\\
x^+_0 &...& x^+_n& 0 &...&0	
\end{pmatrix}$$
\end{itemize}
The globular set $\nu K$, equipped with these compositions and units is an $\omega$-category.
\end{definition}

\begin{definition}
We define the functor $\nu: \CDA \to \omegacat$ which associates to an augmented directed complex $K$, the $\omega$-category $\nu K$, and to a morphism of augmented directed complexes $f: K \to L$, the morphism of $\omega$-categories.
$$
\begin{array}{rccc}
\nu f : &\nu K &\to& \nu L\\
& \left(\begin{matrix}
x^-_0 &...&x_n^-\\
x^+_0&...&x_n^+
\end{matrix}\right) 
&\mapsto&
\left(\begin{matrix}
f_0(x^-_0) &...&f_n(x_n^-)\\
f_0(x^+_0)&...&f_n(x_n^+)
\end{matrix}\right) 
\end{array}
$$
\end{definition}

\begin{theorem}[Steiner]
\label{theo:ajdonction de steiner avec unite et counite explicite}
The functors $\lambda$ and $\nu$ form an adjoint pair 
\[\begin{tikzcd}
	{\lambda:\omegacat} & {\CDA:\nu}
	\arrow[""{name=0, anchor=center, inner sep=0}, shift left=2, from=1-1, to=1-2]
	\arrow[""{name=1, anchor=center, inner sep=0}, shift left=2, from=1-2, to=1-1]
	\arrow["\dashv"{anchor=center, rotate=-90}, draw=none, from=0, to=1]
\end{tikzcd}\]
For a $\omega$-category $C$, the unit of the adjunction is given by:
$$\begin{array}{rrcl}
~~~~~\eta :& C &\to & \nu \lambda C \\
& x\in C_n &\mapsto & 
\begin{pmatrix}
[d^-_0(x)]_0&...&[d^-_{n-1}(x)]_{n-1}&[x]_n\\
[d^+_0(x)]_0&...& [d^+_{n-1}(x)]_{n-1}&[x]_n
\end{pmatrix}
\end{array}
$$
For an augmented directed complex $K$, the counit is given by:
$$\begin{array}{rrcl}
\pi :& \lambda \nu K &\to & K~~~~~~~~~~~~~~~~ \\
& [x ]_n \in (\lambda \nu K)_n&\mapsto & x_n^+ = x_n^-
\end{array}
$$
\end{theorem}
\begin{proof}
This is \cite[theorem 2.11]{Steiner_omega_categories_and_chain_complexes}.
\end{proof}

\p 
A \snotion{basis}{for augmented directed complexes} for an augmented directed complex $(K,K^*,e)$ is a graded set $B = (B_n)_{n\in\Nb}$ such that for every $n$, $B_n$ is both a basis for the monoid $K_n^*$ and for the group $K_n$.
\begin{remark}
The elements of $B_n$ can be characterized as the minimal elements of $K_n^*\backslash{0}$ for the following order relation:
	$$x\leq y \mbox{ iff } y-x \in K_n^*$$
This shows that if a basis exists, it is unique.
\end{remark}
\p Any element of $K_n$ can then be written uniquely as a sum $\sum_{b\in B_n} \lambda_b b$. This leads us to define new operations:
For an element $x := \sum_{b\in B_n} \lambda_b b$ of $K_n$, we define the \textit{positive part} and the \textit{negative part}:
$$
\begin{array}{rcl}
(x)_+ &:=& \sum_{b\in B_n, \lambda_b> 0} ~\lambda_bb\\
(x)_- &:=& \sum_{b\in B_n, \lambda_b< 0} -\lambda_bb
\end{array}
$$
We then have $x = (x)_+ - (x)_-$. An element $x$ is \textit{positive} (resp. \textit{negative}) when $x =(x)_+$ (resp. when $x =-(x)_-$).
Let $y = \sum_{b\in B_n} \mu_b b$, we set : 
$$
\begin{array}{rcl}
x\wedge y &:=& \sum_{b\in B_n} \mbox{ min}(\lambda_b, \mu_b)~ b \\
\end{array}
$$
Eventually, we set \sym{(partialna@$\partial_n^+(\uvar)$}\sym{(partialnb@$\partial_n^-(\uvar)$}
$$
\begin{array}{rcl}
\partial_n^+(\uvar) &:=& (\partial_n(\uvar))_+ : K_{n+1}\to K^*_n\\
\partial_n^-(\uvar) &:= &(\partial_n(\uvar))_- : K_{n+1}\to K^*_n
\end{array}
$$

When an element $b$ of the basis is in the support of $x$, i.e $\lambda_b\neq 0$, we say that \textit{$b$ belongs to $x$}, which is denoted by $b\in x$.
\begin{example}
For any integer $n$, $\lambda\Db_n$ admits a basis, given by the graded set $B_{\lambda\Db_n}$ fulfilling:
$$(B_{\lambda\Db_n})_k:= \left\{ 
\begin{array}{ll}
\{e_k^-,e_k^+\}&\mbox{ if $k<n$}\\
\{e_n\}&\mbox{ if $k=n$}\\
\emptyset&\mbox{ if $k>n$}\\
\end{array}\right.$$ 
The augmented directed complex $\lambda[n]$ also admits a basis, given by the graded set $B_{\lambda\Db_n}$ fulfilling:
$$(B_{\lambda\Db_n})_k:= \left\{ 
\begin{array}{ll}
\{v_0,v_1,...,v_n\}&\mbox{ if $k=0$}\\
\{v_{0,1},v_{1,2}...,v_{n-1,n}\}&\mbox{ if $k=1$}\\
\emptyset&\mbox{ if k>1}\\
\end{array} \right.$$ 
\end{example}

\p
Let $a\in K^*_n$. We set by a decreasing induction on $k\leq n$ : 
 $$ \begin{array}{rclc}
 \langle a\rangle_k^\alpha &:= & a & \mbox{if $k = n$}\\
 &:= & \partial_k^\alpha\langle a\rangle^\alpha_{k+1} & \mbox{if not}
\end{array} 
$$
The array associated to $a$ is then: 
$$\langle a\rangle := \begin{pmatrix}
\langle a\rangle^-_0 &...&\langle a\rangle^-_{n-1}&a\\
\langle a\rangle^+_0 &...&\langle a\rangle^+_{n-1}&a
\end{pmatrix}$$
The basis is said to be \wcnotion{unitary}{unitary basis} when for any $b\in B$, the array $\langle b\rangle$ is coherent.

\p We define the relation $\odot$ on $B$ as being the smallest transitive and reflexive relation such that for any pair of elements of the basis $a,b$, 
$$a\odot_n b \mbox{ if } \mbox{($|a|>0$ and $b\in\langle a\rangle_{|a|-1}^-$)}~~\mbox{or}~~\mbox{($|b|>0$ and $a\in \langle b\rangle_{|b|-1}^+$)}$$
A basis is said to be \wcsnotion{loop free}{loop free basis}{for augmented directed complexes} when for any $n$, the relation $\odot_n$ is a (partial) order on $B$.
\begin{remark}
In \cite{Ara_Maltsiniotis_joint_et_tranche}, this notion is called \textit{strongly loop free}.
\end{remark}

\begin{example}
For any integer $n$, $\lambda\Db_n$ and $\lambda[n]$ admit a loop free and unitary basis.
\end{example}

\p We now define the subcategory \wcnotation{$\CDAB$}{(adcb@$\CDAB$} of $\CDA$ composed of augmented directed complexes which admit a unitary and loop free basis. 
We will now describe the analog of the notion of basis for $\omega$-categories. 

\begin{definition}
A $\omega$-category $C$ is \wcnotion{generated by composition}{generated by composition} by a set $E\subset C$ when any cell can be written as a composition of elements of $E$ and iterated units of elements of $E$. This set is a \snotion{basis}{for $\io$-categories} if $\{[e]_{d(e)}\}_{e\in E}$ is a basis of the augmented directed complex $\lambda C$. 
\end{definition}

\begin{prop}
An $\omega$-category $C$ that admits a basis is an $\zo$-category.
\end{prop}
\begin{proof}
Let $C$ be an $\omega$-category that admits a basis $E$. Suppose that there exists a non trivial $n$-cell $\alpha$ that admits an inverse $\beta$. We then have $[\alpha]_n+ [\beta]_n=[\alpha \circ_{n-1} \beta]_n =0$. As $\lambda C$ is free, we have $[\alpha]_n=0$. This implies the equality $[e]_n=0$ for any element $e\in E$ of dimension $n$ that appears in a decomposition of $\alpha$. This is obviously in contradiction with the fact that $\{[e]_{d(e)}\}_{e\in E}$ is a basis of the augmented directed complex $\lambda C$. 
\end{proof}

\begin{definition}
A basis $E$ of an $\zo$-category is : 
\begin{enumerate}
\item \wcsnotion{Loop free}{loop free basis}{for $\zo$-categories} when $\{[e]_{d(e)}\}_{e\in E}$ is.
\item \wcnotion{Atomic}{atomic basis} when $[d_n^+ e]_n \wedge [d_n^- e]_n = 0$ for any $e\in E$ and any natural number $n$ strictly smaller than the dimension of $e$. 
\end{enumerate}
\end{definition}

\begin{prop}
 If a loop free basis $E$ is atomic then $\{[e]\}_{e\in E}$ is unitary.
 \end{prop}
\begin{proof}
 This is \cite[proposition 4.6]{Steiner_omega_categories_and_chain_complexes}.
 \end{proof}

\begin{example}
For any integer $n$, $\Db_n$ and $[n]$ admit a loop free and atomic basis.
More generally, \cite[proposition 4.13]{Ara_Maltsiniotis_joint_et_tranche} states that 
any globular sum admits a loop free and atomic basis. 
\end{example}

\p Proposition $1.23$ of \cite{Ara_a_categorical_characterization_of_strong_Steiner_omega_categories} states that if an $\zo$-category admits a loop-free and atomic basis, it is unique.
We then define the category \wcnotation{$\zocatB$}{((a30@$\zocatB$} as the full subcategory of $\omegacat$ composed of $\zo$-categories admitting an atomic and loop-free basis.

 \begin{theorem}[Steiner]
 \label{theorem:steiner}
 Once restricted to $\zocat_B$ and $\CDAB$, the adjunction 
\[\begin{tikzcd}
	{\lambda:\omegacat} & {\CDA:\nu}
	\arrow[""{name=0, anchor=center, inner sep=0}, shift left=2, from=1-1, to=1-2]
	\arrow[""{name=1, anchor=center, inner sep=0}, shift left=2, from=1-2, to=1-1]
	\arrow["\dashv"{anchor=center, rotate=-90}, draw=none, from=0, to=1]
\end{tikzcd}\]
becomes an adjoint equivalence, i.e. :
$$ \lambda_{|\zocatB } \circ \nu_{|\CDAB} \cong id_{|\CDAB}~~~~~~~ id_{|\zocatB }\cong \nu_{|\CDAB} \circ \lambda_{|\zocatB }$$
\end{theorem}
\begin{proof}
See \cite[theorem 5.11]{Steiner_omega_categories_and_chain_complexes}.
\end{proof}

If $K$ is an augmented directed complex admitting a unitary and loop-free basis $B$, then the $\zo$-category $\nu K$ admits an atomic and loop-free basis given by the set $\langle B\rangle := \{\langle b\rangle,b\in B\}$. Conversely if an $\zo$-category $C$ admits an atomic and loop-free basis $E$, then the augmented directed complex $\lambda C$ admits a unitary and loop-free basis given by the family of sets $[E_n] := \{[e]_{d(e)}, e\in E_n\}$. 
The isomorphisms
$$\lambda \nu K\cong K \mbox{~~~ and ~~~} C\cong \nu\lambda C$$
induce isomorphisms:
$$[\langle B\rangle ]\cong B \mbox{~~~ and ~~~} E \cong \langle [E]\rangle.$$
\p We define the \snotion{full duality}{for augmented directed complexes} \ssym{((b80@$(\uvar)^{\circ}$}{for augmented directed complexes}
$$(\uvar)^\circ:\CDA\to \CDA$$
that sends a augmented directed complex $((K,\partial),K^*,e)$ to $((K, - \partial),K^*,e)$. We left the reader to check that $K^\circ$ admits a loop free and atomic basis when this is the case for $K$. This functor then induces a functor:
$$(\uvar)^\circ:\CDAB\to \CDAB.$$
Morever, we have a canonical equivalence: 
$$\lambda (C^\circ)\cong (\lambda C)^\circ$$
natural in $C$.

\p 
Let $f:M\to N$ be a morphism between two augmented directed complexes admitting unitary and loop-free bases $B_M$ and $B_N$. The morphism $f$ is \wcnotion{quasi-rigid}{quasi-rigid morphism} if for any $n$, and any $b\in (B_M)_n$,
$$f_n(b)\neq 0 ~\Rightarrow ~ f_n(b)\in B_N\mbox{ and }\nu(f)\langle b\rangle = \langle f_n(b)\rangle.$$

\begin{theorem}
\label{theo:Kan condition}
Suppose given a commutative square in $\CDAB$
\[\begin{tikzcd}
	K & {M_1} \\
	{M_0} & M
	\arrow["{k^0}", from=1-1, to=1-2]
	\arrow["{l^1}", from=1-2, to=2-2]
	\arrow["{k^0}"', from=1-1, to=2-1]
	\arrow["{l^0}"', from=2-1, to=2-2]
\end{tikzcd}\]
and such that all morphisms are quasi-rigid. Let $B_K,~B_{M_0},~B_{M_1},~B_{M}$ be the bases of $K,~M_0,~M_1,~ M$.

Then, this square is cocartesian if and only if for any $n$, the induced diagram of sets
\[\begin{tikzcd}
	{(B_{K})_n\cup\{0\}} & {(B_{M_1})_n\cup\{0\}} \\
	{(B_{M_0})_n\cup\{0\}} & {(B_{M})_n\cup\{0\}}
	\arrow["{k^0_n}", from=1-1, to=1-2]
	\arrow["{l^1_n}", from=1-2, to=2-2]
	\arrow["{k^0_n}"', from=1-1, to=2-1]
	\arrow["{l^0_n}"', from=2-1, to=2-2]
\end{tikzcd}\]
is cocartesian. Furthermore, the induced square in $\zocat$
\[\begin{tikzcd}
	{\nu K} & {\nu M_1} \\
	{\nu M_0} & {\nu M}
	\arrow["{\nu k^0}", from=1-1, to=1-2]
	\arrow["{\nu l^1}", from=1-2, to=2-2]
	\arrow["{\nu k^0}"', from=1-1, to=2-1]
	\arrow["{\nu l^0}"', from=2-1, to=2-2]
\end{tikzcd}\]
is cocartesian.
\end{theorem}
\begin{proof}
This is a combination of theorems 3.1.2 and 3.2.7 of \cite{Loubaton_condition_de_kan}.
\end{proof}

\subsection{Gray operations on augmented directed complexes}
We follow Steiner (\cite{Steiner_omega_categories_and_chain_complexes}) and Ara-Maltsiniotis (\cite{Ara_Maltsiniotis_joint_et_tranche}) for the definitions and first properties of Gray operations on augmented directed complexes.

\p 
Let $(K,K^*,e)$ and $(L,L^*,f)$ be two augmented directed complexes. We define the \snotion{Gray tensor product}{for augmented directed complexes} of $(K,K^*,e)$ and $(L,L^*,f)$ as the augmented directed complex
$$(K,K^*,e)\otimes (L,L^*,f):= (K\otimes L,(K\otimes L)^*,e\otimes f)$$
where 
\begin{enumerate}
\item[$-$] $K\otimes L$ is the chain complex whose value on $n$ is:
$$(K\otimes L)_n:= \oplus_{k+l=n}K_k\otimes L_l$$
and the differential is the unique graded group morphism fulfilling: 
$$\partial (x\otimes y):= \partial x\otimes y + (-1)^{|x|}x\otimes \partial y$$
where we set the convention $\partial x:=0$ if $|x|=0$.
\item[$-$] $(K\otimes L)^*$ is given on all integer $n$ by :
$$(K\otimes L)^*_n:= \oplus_{k+l=n}K_k^*\otimes L_l^*.$$
\item[$-$] $e\otimes f:K_0\otimes L_0\to \Zb$ is the unique morphism fulfilling 
$$(e\otimes f)(x\otimes y)= e(x)f(y).$$
\end{enumerate}
\p The Gray tensor product induces a monoidal structure on $\CDA$. Its unit is given by $\lambda \Db_0$. Furthermore, Steiner shows that if $K$ and $L$ admit loop free and unitary bases, so does $K\otimes L$. The monoidal structure then restricts to a monoidal structure on $\CDAB$. 
Eventually \cite[proposition A.20]{Ara_Maltsiniotis_joint_et_tranche} provides an equivalence 
\begin{equation}
\label{eq:dualities and otimes}
(K\otimes L)^\circ \cong K^\circ\otimes L^\circ
\end{equation}

\p To simplify notion, the augmented directed complex $\lambda[1]$ will simply be denoted by $[1]$. The induced functor 
$$\uvar\otimes [1]:\CDA\to \CDA$$
is called the \snotionsym{Gray cylinder}{((d30@$\uvar\otimes[1]$}{for augmented directed complexes}. 
For $(K,K^*,e)$ an augmented directed complex, we then have
$$(K,K^*,e)\otimes [1]:=(K\otimes [1] ,(K\otimes [1])^*,e)$$
where
\begin{enumerate}
\item[$-$] $K\otimes [1]$ is the chain complex whose value on $n$ is:
$$(K\otimes [1])_n:=\left\{
\begin{array}{ll}
\{x\otimes \{\epsilon\},x\in K_0,\epsilon=0,1\}&\mbox{if $n=0$}\\
\{x\otimes \{\epsilon\},x\in K_n,\epsilon=0,1\}\oplus \{x\otimes[1],x\in K_{n-1}\} &\mbox{if $n>0$}
\end{array}\right.$$
and the differential is the unique graded group morphism fulfilling: 
$$\partial (x\otimes [1]):= \partial x\otimes [1] + (-1)^{|x|}(x\otimes \{1\}-x\otimes \{0\} )~~~~~\partial (x\otimes\{\epsilon\}) = (\partial x)\otimes\{\epsilon\}$$
for $\epsilon\in\{0,1\}$, and
where we set the convention $\partial x:=0$ if $|x|=0$.
\item[$-$] $(K\otimes [1])^*$ is given on all integer $n$ by :
$$(K\otimes [1])^*_n:=\left\{
\begin{array}{ll}
\{x\otimes \{\epsilon\},x\in K^*_0,\epsilon=0,1\}&\mbox{if $n=0$}\\
\{x\otimes\{ \epsilon\},x\in K^*_n,\epsilon=0,1\}\oplus \{x\otimes[1],x\in K^*_{n-1}\} &\mbox{if $n>0$}
\end{array}\right.$$
\item[$-$] $e:(K\otimes [1])_0\to \Zb$ is the unique morphism fulfilling 
$$e(x\otimes \{0\})=e(x\otimes \{1\})= e(x).$$
\end{enumerate}

\p We define the \snotionsym{Gray cone}{((d40@$\uvar\star 1$}{for augmented directed complexes} and the \snotion{Gray $\circ$-cone}{for augmented directed complexes}\index[notation]{((d50@$1\overset{co}{\star}\_$!\textit{for augmented directed complexes}}:
$$\begin{array}{ccccccc}
\CDA &\to&\CDA&&\CDA &\to&\CDA\\
K&\mapsto &K\star 1 & &K &\mapsto &1\costar K
\end{array}
$$
where $K\star 1$ and $1\costar K$ are defined as the following pushout: 
\begin{equation}
\label{eq:defin of cstar costar CDA}
\begin{tikzcd}
	{K\otimes\{1\}} & {K\otimes [1]} & {K\otimes\{0\}} & {K\otimes [1]} \\
	1 & {K\star 1} & 1 & {1\costar K}
	\arrow[from=1-1, to=2-1]
	\arrow[from=1-1, to=1-2]
	\arrow[from=2-1, to=2-2]
	\arrow[from=1-2, to=2-2]
	\arrow["\lrcorner"{anchor=center, pos=0.125, rotate=180}, draw=none, from=2-2, to=1-1]
	\arrow[from=1-3, to=2-3]
	\arrow[from=2-3, to=2-4]
	\arrow[from=1-3, to=1-4]
	\arrow[from=1-4, to=2-4]
	\arrow["\lrcorner"{anchor=center, pos=0.125, rotate=180}, draw=none, from=2-4, to=1-3]
\end{tikzcd}
\end{equation}
The equation \eqref{eq:dualities and otimes} provides an equivalence
$$(C\star 1)^\circ\cong 1\costar C^\circ.$$
According to \cite[corollary 6.21]{Ara_Maltsiniotis_joint_et_tranche} and to the previous equivalence, if $K$ admits a loop free and unitary basis, this is also the case for $K\star 1$ and $1\costar K$. The {Gray cone} and the {Gray $\circ$-cone} then induce functors:
$$\begin{array}{ccccccc}
\CDAB&\to&\CDAB&&\CDAB &\to&\CDAB\\
K&\mapsto &K\star 1 & &K &\mapsto &1\costar K
\end{array}
$$

\p Unfolding the definition, we have
$$(K,K',e)\star 1:=(K\star 1, (K\star 1)^*,e)~~~~~1\costar(K,K',e):=(1\costar K, (1\costar K)^*,e)$$
where
\begin{enumerate}
\item[$-$] $K\star 1$ and $1\costar K$ are the chain complex whose value on $n$ are:
$$(K\star 1)_n:=\left\{
\begin{array}{ll}
\Zb[\emptyset\star 1]\oplus \{x\star \emptyset,x\in K_0\}&\mbox{if $n=0$}\\
\{\emptyset\star x,x\in K_n\}\oplus \{x\star 1,x\in K_{n-1}\} &\mbox{if $n>0$}
\end{array}\right.$$
$$(1\costar K)^n:=\left\{
\begin{array}{ll}
\Zb[1\costar\emptyset]\oplus \{\emptyset\costar x,x\in K_0\}&\mbox{if $n=0$}\\
\{\emptyset\costar x,x\in K_n\}\oplus \{1\costar x,x\in K_{n-1}\} &\mbox{if $n>0$}
\end{array}\right.$$
and the differentials are the unique graded group morphisms fulfilling: 
$$\begin{array}{rr}
\partial (x\star 1)= \partial x\star 1 + (-1)^{|x|} x\star \emptyset&\partial( x \star \emptyset )=\partial x\star \emptyset \\
 \partial (1\costar x)= 1\costar \partial x + (-1)^{|x|} \emptyset\costar x& \partial( \emptyset \costar x )= \emptyset \costar x\\
\end{array}$$
where we set the convention $\partial x:=0$ if $|x|=0$.
\item[$-$] The graded monoids $(K\star 1)^*$ and $(1\costar K)^*$ are given on all integer $n$ by :
$$(K\star 1)^*:=\left\{
\begin{array}{ll}
\Nb[\emptyset\star 1]\oplus \{x\star \emptyset,x\in K^*_0\}&\mbox{if $n=0$}\\
\{\emptyset\star x,x\in K^*_n\}\oplus \{x\star 1,x\in K^*_{n-1}\} &\mbox{if $n>0$}
\end{array}\right.$$
$$(1\costar K)^*:=\left\{
\begin{array}{ll}
\Nb[1\costar\emptyset]\oplus \{\emptyset\costar x,x\in K^*_0\}&\mbox{if $n=0$}\\
\{\emptyset\costar x,x\in K^*_n\}\oplus \{1\costar x,x\in K^*_{n-1}\} &\mbox{if $n>0$}
\end{array}\right. .$$
\item[$-$] The augmentations $e:(K\star 1)_0\to \Zb$ and $e:(1\costar K)_0\to\Zb$ are the unique ones fulfilling 
$$
\begin{array}{cc}
e( \emptyset \star 1) =1 & e(x\star \emptyset)=e(x)\\
e( 1\costar \emptyset ) =1 & e( \emptyset\costar x)=e(x).
\end{array}$$
\end{enumerate}

\begin{prop}
\label{prop:non trivial automorphisme 1}
Let $A$ be an augmented directed complex admitting no non-trivial automorphisms. Then the augmented directed complexes $A\star 1$ and $1\costar A$ have no non-trivial automorphisms.
\end{prop}
\begin{proof}
Let $\phi:A\star 1\to A\star 1$ be an automorphism. The morphism $\phi$ then induces a bijection on the elements of the basis of $A\star 1$.

 As the element $\emptyset\star 1\in (A\star 1)_0$ is the only element of the basis such that for all $v\in (A\star 1)_1$  $\partial_0^-(v)\neq \emptyset\star 1$, it is preserved by $\phi$. As a consequence, for any element $x$ of the basis of $A_0$, $\phi(x\star \emptyset)$ is of shape $x'\star \emptyset$. The morphism $\phi$ then preserves $(A\star \emptyset)_0$.

Now, remark that for any element $e\in (A\star 1)^*_{n+1}$, there exists $x\in (A\star 1)^*_n$ such that $x\star 1\leq e$ if and only if there exists $y\in (A\star 1)^*_{n-1}$ such that $y\star 1\leq \partial^+(e)$. By a direct induction, this implies that there exists $x\in (A\star 1)^*_n$ such that $x\star 1\leq e$ if and only if $\partial^+_0(e)\in \Zb[\emptyset\star 1]$.

Combined with the previous observation, this implies that for any element $x$ of the basis of $A_{n+1}$, $\phi(x\star \emptyset)$ is of shape $x'\star \emptyset$.
The automorphism $\phi$ then induces by restriction an automorphism $\phi_{|A\star\emptyset}:A\to A$, and the hypothesis implies that it is the identity.

We now show by induction on $n$ that $\phi_n:(A\star 1)_n\to (A\star 1)_n$ is the identity. Suppose the result true at the stage $n$. For any element $x$ of the basis of $A_{n}$, we then have 
$$\partial \phi(x\star 1) = \phi(\partial (x\star 1)) = \partial (x\star 1).$$
By the definition of the derivative of $A\star 1$, and as $\phi$ preserves the basis, this forces the equality $\phi(x\star 1)=x\star 1$. As we already know that for any element $x$ of the basis of $A_{n+1}$ we have $\phi(x\star \emptyset)=x\star \emptyset$, this concludes the induction.

We then have $\phi=id$ and $A\star 1$ has no non trivial automorphisms.
The case $1\costar A$ follows directly by using the fact that dualities preserve  augmented directed complexes admitting no non-trivial automorphisms. 

\end{proof}

\p 
We define the \snotionsym{suspension}{((d60@$[\uvar,1]$}{for augmented directed complexes} as the functor 
$$[\uvar,1]:\CDA\to \CDA$$
where $[K,1]$ is defined as the following pushout:
\begin{equation}
\label{eq:def of suspension cda}
\begin{tikzcd}
	{K\otimes \{0,1\}} & {K\otimes [1]} \\
	{1\coprod 1} & {[K,1]}
	\arrow[from=1-1, to=2-1]
	\arrow[from=2-1, to=2-2]
	\arrow[from=1-1, to=1-2]
	\arrow[from=1-2, to=2-2]
	\arrow["\lrcorner"{anchor=center, pos=0.125, rotate=180}, draw=none, from=2-2, to=1-1]
\end{tikzcd}
\end{equation}
We leave to the reader to check that $[K,1]$ admits a loop free and unitary basis when this is the case for $K$. This functor then induces a functor:
$$[\uvar,1]:\CDAB\to \CDAB$$

\p Unfolding the definition, we have
$$[(K,K',e),1]:=([K,1] ,([K,1])^*,e)$$
where
\begin{enumerate}
\item[$-$] $[K,1]$ is the chain complex whose value on $n$ is:
$$[K,1]:=\left\{
\begin{array}{ll}
 \Zb[\{0\},\{1\}]&\mbox{if $n=0$}\\
 \{[x,1],x\in K_{n-1}\} &\mbox{if $n>0$}
\end{array}\right.$$
and the differential is the unique graded group morphism fulfilling: 
$$\partial([x,1]):= \left\{
 \begin{array}{lll} 
 \{1\}-\{0\}&\mbox{if $|x|=0$}\\ 
 ~[\partial x,1]&\mbox{if $|x|>0$}
 \end{array}\right.
$$
\item[$-$] $([K,1])^*$ is given on all integer $n$ by:
$$([K,1])^*_n:=\left\{
\begin{array}{ll}
\Nb[0,1]&\mbox{if $n=0$}\\
 \{[x,1],x\in K^*_{n-1}\} &\mbox{if $n>0$}
\end{array}\right.$$
\item[$-$] $e:([K,1])_0\to \Zb$ is the unique morphism	 fulfilling 
$$e( 0)=e( 1)= e(x).$$
\end{enumerate}

\begin{prop}
\label{prop:non trivial automorphisme 2}
Let $A$ be a non null augmented directed complex admitting no non-trivial automorphisms. Then the augmented directed complex $[A,1]$ has no non-trivial automorphisms.
\end{prop}
\begin{proof}
Let $\phi:[A,1]\to [A,1]$ be an automorphism. As the element $\{1\}\in ([A,1])_0$ is the only element of the basis such that for all $v\in [A,1]_1$  $\partial_0^-(v)\neq \{1\}$, it is preserved by $\phi$. As a consequence, $\phi$ also preserves $\{0\}$. The induced morphism $\phi_0:[A,1]_0\to [A,1]_0$ is then the identity. 

Now, remark that $(\phi_{n+1})_{n\in \Nb}:A\to A$ is an automorphism and is then the identity. This implies that for all $n>0$, $\phi_n:[A,1]_n\to [A,1]_n$ is then identity, which concludes the proof.
\end{proof}

\p 
We define the \textit{wedges} as the functors
$$[\uvar,1]\vee[1]:\CDA\to \CDA~~~~~~ [1]\vee[\uvar,1]:\CDA\to \CDA$$
where $[K,1]\vee [1]$ and $[1]\vee[K,1]$ are defined as the following pushouts:
\[\begin{tikzcd}
	{\lambda [0]} & {[1]} && {\lambda [0]} & {[K,1]} \\
	{[K,1]} & { [K,1]\vee[1]} && {[1]} & {[1]\vee[K,1]}
	\arrow["{\{1\}}"', from=1-1, to=2-1]
	\arrow[from=2-1, to=2-2]
	\arrow["{\{0\}}", from=1-1, to=1-2]
	\arrow[from=1-2, to=2-2]
	\arrow["\lrcorner"{anchor=center, pos=0.125, rotate=180}, draw=none, from=2-2, to=1-1]
	\arrow["{\{0\}}", from=1-4, to=1-5]
	\arrow["{\{1\}}"', from=1-4, to=2-4]
	\arrow[from=1-5, to=2-5]
	\arrow[from=2-4, to=2-5]
	\arrow["\lrcorner"{anchor=center, pos=0.125, rotate=180}, draw=none, from=2-5, to=1-4]
\end{tikzcd}\]
Once again, we can easily check that $[K,1]\vee[1]$ and $[1]\vee[K,1]$ have a loop free and unitary basis when this is the case for $K$. These functors then induce functors
$$[\uvar,1]\vee[1]:\CDAB\to \CDAB~~~~~~ [1]\vee[\uvar,1]:\CDAB\to \CDAB$$
\p Unfolding the definition, we have
$$[(K,K',e),1]\vee [1]:=([K,1]\vee [1] ,([K,1]\vee [1])^*,e)$$ $$
[1]\vee(K,K',e),1]:=([1]\vee[K,1] ,([1]\vee[K,1])^*,e)$$
where
\begin{enumerate}
\item[$-$] $[K,1]\vee [1]$ and $[1]\vee[K,1]$ are the chain complexes whose value on $n$ are:
$$[K,1]\vee[1]:=\left\{
\begin{array}{ll}
\Zb[\{0\},\{1\},\{2\}]&\mbox{if $n=0$}\\
 \{[x,1],x\in K_{0}\}\oplus \Zb[e_1] &\mbox{if $n=1$}\\
 \{[x,1],x\in K_{n-1}\} &\mbox{if $n>1$}
\end{array}\right.$$
$$[1]\vee[K,1]:=\left\{
\begin{array}{ll}
\Zb[\{0\},\{1\},\{2\}]&\mbox{if $n=0$}\\
\Zb[e_1] \oplus \{[x,1],x\in K_{0}\} &\mbox{if $n=1$}\\
 \{[x,1],x\in K_{n-1}\} &\mbox{if $n>1$}
\end{array}\right.$$
and the differentials are the unique graded group morphism fulfilling: 
$$\partial_{[K,1]\vee[1]} (e_1):= \{2\}-\{1\}
~~~
\partial_{[K,1]\vee[1]} ([x,1]):=
\left\{
\begin{array}{ll}
 \{1\}-\{0\}&\mbox{if $|x|=0$}\\
 ~[\partial x,1]&\mbox{if $|x|>0$}\\
\end{array}\right.
$$
$$
\partial_{[1]\vee[K,1]} (e_1):= \{1\}-\{0\}
~~~
\partial_{[1]\vee[K,1]} ([x,1]):=
\left\{
\begin{array}{ll}
 \{2\}-\{1\}&\mbox{if $|x|=0$}\\
~ [\partial x,1]&\mbox{if $|x|>0$}\\
\end{array}\right.
$$
\item[$-$] $([K,1]\vee [1])^*$ and $([1]\vee[K,1])^*$ are given on all integer $n$ by:
$$([K,1]\vee[1])^*:=\left\{
\begin{array}{ll}
\{\{0\},\{1\},\{2\}\}&\mbox{if $n=0$}\\
 \{[x,1],x\in K_0^*\}\oplus \Nb[e_1] &\mbox{if $n=1$}\\
 \{[x,1],x\in K_{n-1}\} &\mbox{if $n>1$}
\end{array}\right.$$
$$([1]\vee[K,1])^*:=\left\{
\begin{array}{ll}
\{\{0\},\{1\},\{2\}\}&\mbox{if $n=0$}\\
\Nb[e_1]\oplus\cup \{[x,1],x\in K^*_{0}\} &\mbox{if $n=1$}\\
 \{[x,1],x\in K^*_{n-1}\} &\mbox{if $n>1$}
\end{array}\right.$$
\item[$-$] The augmentations $e$ are the unique morphism fulfilling 
$$e( \{0\})=e(\{ 1\})= e(\{2\})=1.$$
\end{enumerate}

\begin{prop}
\label{prop:non trivial automorphisme 3}
Let $A$ be a non null augmented directed complex admitting no non-trivial automorphisms. Then the augmented directed complexes $[A,1]\vee[1]$ and $[1]\vee[A,1]$ have no non-trivial automorphisms.
\end{prop}
\begin{proof}
The proof is similar to the one of proposition \ref{prop:non trivial automorphisme 2} and we leave it to the reader.
\end{proof}

\p 
There are two canonical morphisms 
$$\triangledown:\Sigma K\to \Sigma K \vee [1]
~~~~~~~ \triangledown:\Sigma K\to [1]\vee \Sigma K $$
that are the unique ones fulfilling
$$\triangledown(\{0\}):= \{0\}~~~\triangledown(\{1\}):= \{2\}~~~
\triangledown([x,1]):=\left\{ 
\begin{array}{ll}
~[x,1]+e_1&\mbox{if $|x|=0$}\\
~[x,1]&\mbox{if $|x|>0$}\\
\end{array}\right.$$
When we write $ \Sigma K\to \Sigma K \vee [1]$ and $\Sigma K\to [1]\vee \Sigma K$ and nothing more is specified, it will always mean that we considered the morphisms $\triangledown$.

\begin{prop}
 \label{prop:appendice formula for otimes cda}
 Let $K$ be an augmented directed complex. 
 There is a natural transformation between the colimit of the following diagram
$$
\begin{tikzcd}
	{[1]\vee [K,1]} & {[K\otimes\{0\},1]} & {[K\otimes [1],1]} & {[K\otimes\{1\},1]} & {[K,1]\vee [1]}
	\arrow[from=1-2, to=1-1]
	\arrow[from=1-2, to=1-3]
	\arrow[from=1-4, to=1-3]
	\arrow[from=1-4, to=1-5]
\end{tikzcd}$$
and $[K,1]\otimes [1]$.
\end{prop}
\begin{proof}
The cone is induced by morphisms
$$
\begin{array}{rl}
&[1]\vee [K,1]\to [K,1]\otimes [1]\\
(\mbox{resp}.&[ K,1]\vee[1]\to [ K,1] \otimes [1])
\end{array}
$$ sending an element $x$ in the basis of $[1]$ to $\{0\}\otimes x$ (resp. $\{1\}\otimes x$), an element $y$ in the basis of $[ K,1]$ to $y\otimes\{1\}$ (resp. $y\otimes\{0\}$), 
and by the morphism 
$$f:[K\otimes [1],1]\to [K,1]\otimes [1]$$
defined by the formula 
$$f([x\otimes y,1]):= [ x,1]\otimes y$$ 
for $x$ in the basis of $K$ and $y$ in the basis of $[1]$.
We leave it to the reader to check the compatibilities of this three morphisms.
\end{proof}

\subsection{Gray operations on $\zo$-categories}
\label{section:definition of Gray operations}
We follow Ara-Maltsiniotis \cite{Ara_Maltsiniotis_joint_et_tranche} for the definitions and first properties of Gray operations on $\zo$-categories. Originally, these authors work with $\omega$-categories, and not with $\zo$-categories. However, this modification does not affect proof, and we then allow ourselves to use their results in our framework.

\begin{theorem}[Steiner, Ara-Maltsiniotis]
There is a unique colimit preserving monoidal structure on $\zocat$,
up to a unique monoidal isomorphism, making the functor
$\nu_{|\CDAB}:\CDAB\to \zocat$
a monoidal functor, when $\CDAB$ is endowed with the monoidal structure given by the Gray tensor product.
\end{theorem}
\begin{proof}
This is \cite[theorem A.15]{Ara_Maltsiniotis_joint_et_tranche}.
\end{proof}

\p The monoidal product on $\zocat$ induced by the previous theorem is called the \snotionsym{Gray tensor product}{((d00@$\otimes$}{for $\zo$-categories} and is denoted by $\otimes$. It's unit is $ \Db_0$. If $C$ and $D$ are $\zo$-categories with an atomic and loop free basis, we have by construction
$$C\otimes D := \nu(\lambda C\otimes \lambda D).$$
The induced functor 
$$\uvar\otimes[1]:\zocat\to \zocat$$
is called the \snotionsym{Gray cylinder}{((d30@$\uvar\otimes[1]$}{for $\zo$-categories}.

\begin{prop}
Let $C$ be an $\io$-category.
The following canonical square 
\[\begin{tikzcd}
	{C\otimes\{0,1\}} & {C\otimes[1]} \\
	{1\coprod 1} & {[C,1]}
	\arrow[from=1-1, to=2-1]
	\arrow[from=1-1, to=1-2]
	\arrow[from=2-1, to=2-2]
	\arrow[from=1-2, to=2-2]
	\arrow["\lrcorner"{anchor=center, pos=0.125, rotate=180}, draw=none, from=2-2, to=1-1]
\end{tikzcd}\]
is cocartesian
\end{prop}
\begin{proof}
As all these functors commute with colimits, it is sufficient to demonstrate this assertion when $C$ is a globular sum, and \textit{a fortiori} when $C$ admits a loop free and atomic basis. In this case, remark that all the morphisms appearing in canonical cartesian square
\[\begin{tikzcd}
	{\lambda C\otimes\{0,1\}} & {\lambda C\otimes[1]} \\
	{1\coprod 1} & {[\lambda C,1]}
	\arrow[from=1-1, to=2-1]
	\arrow[from=1-1, to=1-2]
	\arrow[from=2-1, to=2-2]
	\arrow[from=1-2, to=2-2]
	\arrow["\lrcorner"{anchor=center, pos=0.125, rotate=180}, draw=none, from=2-2, to=1-1]
\end{tikzcd}\]
 are quasi-rigid. 
The results then follow from an application of theorem \ref{theo:Kan condition}.
\end{proof}

\p 
\label{para:explicit Dbn otiomes [1]}
Applying the duality $(\uvar)^{op}$ to the computation achieved in appendix B.1 of \cite{Ara_Maltsiniotis_joint_et_tranche}, we can give an explicit expression of $\Db_n\otimes [ 1]$. As a polygraph, the generating arrows of $\Db_n\otimes [1]$ are:
$$ e^\epsilon_k\otimes\{0\}~~~~~e^\epsilon_k\otimes\{1\}~~~~~e^\epsilon_k\otimes[1]$$
 \[ a^-_0 \otimes e^\epsilon_k \qquad a^+_0 \otimes e^\epsilon_k \qquad a \otimes e^\epsilon_k \]
 where $\epsilon$ is either $+$ or $-$, $k \leqslant n$ and $e^+_n = e^-_n$. Their source and target are given as follows:
 \[ \pi^-( e^\epsilon_k \otimes\{0\}) = e^-_{k-1} \otimes\{0\} \qquad\qquad\qquad \pi^+(e^\epsilon_k \otimes\{0\}) = e^+_{k-1}\otimes\{0\} \]
 \[ \pi^-(e^\epsilon_k \otimes\{1\} ) = e^-_{k-1} \otimes\{1\}\qquad\qquad\qquad \pi^+(e^\epsilon_k\otimes\{1\} ) = e^+_{k-1}\otimes\{1\} \]
$$\pi^{-}(e^\epsilon_{2k}\otimes[1]) =...\circ_2(e^+_0\otimes[1])\circ_0(e^\epsilon_{2k}\otimes\{0\})\circ_1 (e^-_1\otimes[1])\circ_3... \circ_{2k-1}(e_{2k-1}^-\otimes[1])$$
$$\pi^{+}(e^\epsilon_{2k}\otimes[1]) = (e_{2k-1}^+\otimes[1])\circ_{2k-1}...\circ_3(e^+_1\otimes[1])\circ_1(e^\epsilon_{2k}\otimes\{1\})\circ_0 (e^-_0\otimes[1])\circ_2...$$
$$\pi^{-}(e^\epsilon_{2k+1}\otimes[1]) = ...\circ_3(e^+_1\otimes[1])\circ_1(e^\epsilon_{2k+1}\otimes\{1\})\circ_0 (e^-_0\otimes[1])\circ_2...\circ_{2k}(e_{2k}^-\otimes[1])$$
$$\pi^{+}(e^\epsilon_{2k+1}\otimes[1]) = (e_{2k}^+\otimes[1])\circ_{2k}...\circ_2(e^+_0\otimes[1])\circ_0(e^\epsilon_{2k+1}\otimes\{0\})\circ_1 (e^-_1\otimes[1])\circ_3...$$
 We did not put parenthesis in the expression above, to keep them shorter, the default convention is to do the composition $\circ_i$ in order of increasing values of $i$.
 
\begin{example}
The $\zo$-category $\Db_1\otimes[1]$ is the polygraph: 
\[\begin{tikzcd}
	00 & 01 \\
	10 & 11
	\arrow[from=1-1, to=2-1]
	\arrow[from=2-1, to=2-2]
	\arrow[from=1-1, to=1-2]
	\arrow[from=1-2, to=2-2]
	\arrow[shorten <=4pt, shorten >=4pt, Rightarrow, from=1-2, to=2-1]
\end{tikzcd}\]
The $\zo$-category $\Db_2\otimes[1]$ is the polygraph: 
\[\begin{tikzcd}
	00 & 01 & 00 & 01 \\
	10 & 11 & 10 & 11
	\arrow[from=1-1, to=1-2]
	\arrow[""{name=0, anchor=center, inner sep=0}, from=1-1, to=2-1]
	\arrow[from=2-1, to=2-2]
	\arrow[""{name=1, anchor=center, inner sep=0}, from=1-2, to=2-2]
	\arrow[shorten <=4pt, shorten >=4pt, Rightarrow, from=1-2, to=2-1]
	\arrow[""{name=2, anchor=center, inner sep=0}, from=1-3, to=2-3]
	\arrow[from=1-3, to=1-4]
	\arrow[""{name=3, anchor=center, inner sep=0}, from=1-4, to=2-4]
	\arrow[shorten <=4pt, shorten >=4pt, Rightarrow, from=1-4, to=2-3]
	\arrow[""{name=4, anchor=center, inner sep=0}, curve={height=30pt}, from=1-1, to=2-1]
	\arrow[from=2-3, to=2-4]
	\arrow[""{name=5, anchor=center, inner sep=0}, curve={height=-30pt}, from=1-4, to=2-4]
	\arrow["{ }"', shorten <=6pt, shorten >=6pt, Rightarrow, from=0, to=4]
	\arrow["{ }"', shorten <=6pt, shorten >=6pt, Rightarrow, from=5, to=3]
	\arrow[shift left=0.7, shorten <=6pt, shorten >=8pt, no head, from=1, to=2]
	\arrow[shift right=0.7, shorten <=6pt, shorten >=8pt, no head, from=1, to=2]
	\arrow[shorten <=6pt, shorten >=6pt, from=1, to=2]
\end{tikzcd}\]
\end{example}

\p We define the \snotionsym{Gray cone}{((d40@$\uvar\star 1$}{for $\zo$-categories} and the \snotion{Gray $\circ$-cone}{for $\zo$-categories}\index[notation]{((d50@$1\overset{co}{\star}\_$!\textit{for $\zo$-categories}}:
$$\begin{array}{ccccccc}
\zocat &\to&\zocat_{\cdot}&&\zocat &\to&\zocat_{\cdot}\\
C&\mapsto &C\star 1 & &C &\mapsto &1\costar C
\end{array}
$$
where $C\star 1$ and $1\costar C$ are defined as the following pushout: 
\[\begin{tikzcd}
	{C\otimes\{1\}} & {C\otimes [1]} & {C\otimes\{0\}} & {C\otimes [1]} \\
	1 & {C\star 1} & 1 & {1\costar C}
	\arrow[from=1-1, to=2-1]
	\arrow[from=1-1, to=1-2]
	\arrow[from=2-1, to=2-2]
	\arrow[from=1-2, to=2-2]
	\arrow["\lrcorner"{anchor=center, pos=0.125, rotate=180}, draw=none, from=2-2, to=1-1]
	\arrow[from=1-3, to=2-3]
	\arrow[from=2-3, to=2-4]
	\arrow[from=1-3, to=1-4]
	\arrow[from=1-4, to=2-4]
	\arrow["\lrcorner"{anchor=center, pos=0.125, rotate=180}, draw=none, from=2-4, to=1-3]
\end{tikzcd}\]

\begin{example}
The $\zo$-categories $\Db_1\star 1$ and $1\costar \Db_1$ correspond respectively to the polygraphs: 
\[\begin{tikzcd}
	0 &&&& 0 \\
	1 & \star && \star & 1
	\arrow[from=1-1, to=2-1]
	\arrow[from=2-1, to=2-2]
	\arrow[""{name=0, anchor=center, inner sep=0}, from=1-1, to=2-2]
	\arrow[""{name=1, anchor=center, inner sep=0}, from=1-5, to=2-5]
	\arrow[from=2-4, to=1-5]
	\arrow[""{name=2, anchor=center, inner sep=0}, from=2-4, to=2-5]
	\arrow[shorten <=2pt, Rightarrow, from=0, to=2-1]
	\arrow[shift right=2, shorten <=4pt, shorten >=4pt, Rightarrow, from=1, to=2]
\end{tikzcd}\]
The $\zo$-categories $\Db_2\star 1$ and $1\costar \Db_2$ correspond respectively to the polygraphs: 
\[\begin{tikzcd}
	0 & {~} & 0 &&& 0 & {~} & 0 \\
	1 & \star & 1 & \star & \star & 1 & \star & 1
	\arrow[""{name=0, anchor=center, inner sep=0}, from=1-1, to=2-1]
	\arrow[from=2-1, to=2-2]
	\arrow[""{name=1, anchor=center, inner sep=0}, from=1-3, to=2-3]
	\arrow[""{name=2, anchor=center, inner sep=0}, curve={height=30pt}, from=1-1, to=2-1]
	\arrow[from=2-3, to=2-4]
	\arrow[""{name=3, anchor=center, inner sep=0}, from=1-1, to=2-2]
	\arrow[""{name=4, anchor=center, inner sep=0}, draw=none, from=1-2, to=2-2]
	\arrow[""{name=5, anchor=center, inner sep=0}, from=1-3, to=2-4]
	\arrow[from=1-6, to=2-5]
	\arrow[""{name=6, anchor=center, inner sep=0}, from=1-6, to=2-6]
	\arrow[""{name=7, anchor=center, inner sep=0}, from=2-5, to=2-6]
	\arrow[from=1-8, to=2-7]
	\arrow[""{name=8, anchor=center, inner sep=0}, from=1-8, to=2-8]
	\arrow[""{name=9, anchor=center, inner sep=0}, from=2-8, to=2-7]
	\arrow[""{name=10, anchor=center, inner sep=0}, curve={height=-30pt}, from=1-8, to=2-8]
	\arrow[""{name=11, anchor=center, inner sep=0}, draw=none, from=1-7, to=2-7]
	\arrow["{ }"', shorten <=6pt, shorten >=6pt, Rightarrow, from=0, to=2]
	\arrow[shorten <=2pt, shorten >=2pt, Rightarrow, from=3, to=2-1]
	\arrow[shift left=0.7, shorten <=6pt, shorten >=8pt, no head, from=4, to=1]
	\arrow[shift right=0.7, shorten <=6pt, shorten >=8pt, no head, from=4, to=1]
	\arrow[shorten <=6pt, shorten >=6pt, from=4, to=1]
	\arrow[shorten <=2pt, Rightarrow, from=5, to=2-3]
	\arrow[shorten <=6pt, shorten >=6pt, Rightarrow, from=10, to=8]
	\arrow[shift right=2, shorten <=4pt, shorten >=4pt, Rightarrow, from=8, to=9]
	\arrow[shift right=2, shorten <=4pt, shorten >=4pt, Rightarrow, from=6, to=7]
	\arrow[shift right=0.7, shorten <=6pt, shorten >=8pt, no head, from=6, to=11]
	\arrow[shorten <=6pt, shorten >=6pt, from=6, to=11]
	\arrow[shift left=0.7, shorten <=6pt, shorten >=8pt, no head, from=6, to=11]
\end{tikzcd}\]
\end{example}

\begin{prop}
Let $C$ be an $\zo$-category with an unitary and loop free basis. The canonical comparaisons
$$ (\lambda C)\star 1\to \lambda (C\star 1) ~~~~~~~ 1\costar (\lambda C)\to \lambda (1\costar C)$$
are equivalences.

Let $K$ be an augmented directed complex with a loop free and unitary basis. The canonical comparaisons
$$ (\nu K)\star 1\to \nu (K\star 1) ~~~~~~~ 1\costar (\nu K)\to \nu (1\costar K)$$
are equivalences. 
\end{prop}
\begin{proof}
The first assertion directly follows from the fact $\lambda$ commutes with colimits. For the second one,
we can easily check that all the morphisms appearing in the squares \eqref{eq:defin of cstar costar CDA} are quasi-rigid.
The results then follow from an application of theorem \ref{theo:Kan condition}.
\end{proof}

\p We now give some technical results that we will use later.
\begin{lemma}
\label{lemma:non trivial automorphisme 4}
Let $S$ be the smallest set of $\zo$-categories such that
\begin{enumerate}
\item $S$ is stable by isomorphisms,
\item the terminal $\zo$-category belong to $S$,
\item $S$ is stable by $\uvar\star 1$, $1\costar \uvar$, $[\uvar,1]$, $[\uvar,1]\vee[1]$ and $[1]\vee[\uvar,1]$.
\end{enumerate}
Then, the $\zo$-categories belonging to $S$ have non non-trivial automorphisms.
\end{lemma}
\begin{proof}
The set of $\zo$-categories admitting an atomic and loop free basis fulfills the three condition. As a consequence, every $\zo$-category in $S$ has an atomic and loop free basis. Using theorem \ref{theorem:steiner}, it is then sufficient to show that any augmented directed complex in $\lambda(S)$ has no non-trivial automorphisms. The result then follows from propositions \ref{prop:non trivial automorphisme 1}, \ref{prop:non trivial automorphisme 2} and \ref{prop:non trivial automorphisme 3}.
\end{proof}

\begin{prop}
\label{prop:the globes a non non trivial automorphisms}
Let $n$ be an integer $n$. The $\zo$-categories $\Db_n$ and $\underbrace{1\star 1\star ... \star 1}_{n}$ have no non-trivial automorphisms.
\end{prop}
\begin{proof}
This is a direct consequence of lemma \ref{lemma:non trivial automorphisme 4} as these two $\zo$-categories belong to $S$.
\end{proof}

\p The following propositions express the link between the Gray operations and the suspension. They will play a fundamental role in the rest of this work.
\begin{theorem}
 \label{theo:appendice formula for otimes} 
 Let $C$ be an $\zo$-category.
There is a natural identification between $[ C,1]\otimes [1]$ and the colimit of the following diagram
$$
\begin{tikzcd}
	{[1]\vee [ C,1]} & {[C\otimes\{0\},1]} & {[C\otimes [1],1]} & {[C\otimes\{1\},1]} & {[C,1]\vee[1]}
	\arrow[from=1-2, to=1-1]
	\arrow[from=1-2, to=1-3]
	\arrow[from=1-4, to=1-3]
	\arrow[from=1-4, to=1-5]
\end{tikzcd}$$
\end{theorem}
\begin{proof}
As all these functors preserve colimits, it is sufficient to construct the comparison when $C$ is a globular sum, and to show that it is an equivalence when $C$ is a globe. 
As globular sums have atomic and loop free bases, the comparison is induced by proposition \ref{prop:appendice formula for otimes cda}. Using the explicit description of the $\zo$-category $\Db_n\otimes[1]$ given in paragraph \ref{para:explicit Dbn otiomes [1]}, it is straightforward to see that it induces an equivalence on globes.
\end{proof}

The definitional cocartesian squares
\[\begin{tikzcd}
	{C\otimes\{1\}} & {C\otimes [1]} & {C\otimes\{0\}} & {C\otimes [1]} \\
	1 & {C\star 1} & 1 & {1\costar C}
	\arrow[from=1-1, to=2-1]
	\arrow[from=1-1, to=1-2]
	\arrow[from=2-1, to=2-2]
	\arrow[from=1-2, to=2-2]
	\arrow["\lrcorner"{anchor=center, pos=0.125, rotate=180}, draw=none, from=2-2, to=1-1]
	\arrow[from=1-3, to=2-3]
	\arrow[from=2-3, to=2-4]
	\arrow[from=1-3, to=1-4]
	\arrow[from=1-4, to=2-4]
	\arrow["\lrcorner"{anchor=center, pos=0.125, rotate=180}, draw=none, from=2-4, to=1-3]
\end{tikzcd}\]
 imply the following proposition:
\begin{theorem}
 \label{theo:appendice formula for star} 
There is a natural identification between $1\costar [C,1]$ and the colimit of the following diagram
\[\begin{tikzcd}
	{[1]\vee [C,1]} & {[C,1]} & {[C\star 1,1]}
	\arrow[from=1-2, to=1-3]
	\arrow[from=1-2, to=1-1]
\end{tikzcd}\]
There is a natural identification between $[C,1]\star 1$ and the colimit of the following diagram
\[\begin{tikzcd}
	{[1\costar C,1]} & {[C,1]} & {[C,1]\vee[1]}
	\arrow[from=1-2, to=1-3]
	\arrow[from=1-2, to=1-1]
\end{tikzcd}\]
\end{theorem}

\begin{prop}
\label{prop:cartesian squares}
Let $C$ be an $\zo$-category with an atomic and loop free basis. The two following canonical squares are cartesian:
\[\begin{tikzcd}
	1 & {1\costar C} & 1 & {C\star 1} \\
	{\{0\}} & {[C,1]} & {\{1\}} & {[C,1]}
	\arrow[from=1-1, to=1-2]
	\arrow[from=2-1, to=2-2]
	\arrow[from=1-1, to=2-1]
	\arrow[from=1-2, to=2-2]
	\arrow[from=1-3, to=1-4]
	\arrow[from=2-3, to=2-4]
	\arrow[from=1-3, to=2-3]
	\arrow[from=1-4, to=2-4]
\end{tikzcd}\]
The five squares appearing in the following canonical diagram are both cartesian and cocartesian:
\[\begin{tikzcd}
	& {C\otimes\{0\}} & 1 \\
	{C\otimes\{1\}} & {C\otimes[1]} & {C\star 1} \\
	1 & {1\costar C} & {[C,1]}
	\arrow[from=2-3, to=3-3]
	\arrow[from=3-2, to=3-3]
	\arrow[from=2-2, to=3-2]
	\arrow[from=2-2, to=2-3]
	\arrow[from=1-2, to=1-3]
	\arrow[from=1-3, to=2-3]
	\arrow[from=1-2, to=2-2]
	\arrow[from=2-1, to=2-2]
	\arrow[from=3-1, to=3-2]
	\arrow[from=2-1, to=3-1]
\end{tikzcd}\]
\end{prop}
\begin{proof}
The five squares are cocartesian by construction. 
Since the proofs of the cartesianess of all squares are identical, we will only show the proof for the square
\[\begin{tikzcd}
	{C\otimes[1]} & {C\star 1} \\
	{1\costar C} & {[C,1]}
	\arrow[from=1-2, to=2-2]
	\arrow[from=2-1, to=2-2]
	\arrow[from=1-1, to=2-1]
	\arrow[from=1-1, to=1-2]
\end{tikzcd}\]
To this extend, remark that for any integer $n$, the  following square is cartesian. 
\[\begin{tikzcd}
	{(B_{\lambda C\otimes[1]})_n\cup \{0\}} & {(B_{1\costar \lambda C})_n\cup \{0\}} \\
	{(B_{\lambda C\star 1})_n\cup \{0\}} & {(B_{[\lambda C,1]})_n\cup \{0\}}
	\arrow[from=1-1, to=2-1]
	\arrow[from=1-2, to=2-2]
	\arrow[from=1-1, to=1-2]
	\arrow[from=2-1, to=2-2]
\end{tikzcd}\]
This then implies that the following square in the category $\CDA$ is cartesian. 
\[\begin{tikzcd}
	{\lambda C\otimes[1]} & {1\costar \lambda C} \\
	{\lambda C\star 1} & {[\lambda C,1]}
	\arrow[from=1-1, to=2-1]
	\arrow[from=1-2, to=2-2]
	\arrow[from=1-1, to=1-2]
	\arrow[from=2-1, to=2-2]
\end{tikzcd}\]
As $\nu$ is a right adjoint, it preserves limits,  and as it commutes with Gray operation, this concludes the proof.
\end{proof}

\begin{lemma}
\label{lemma: pullback and sum}
Let $a$, $b$, $c$ and $d$ be four globular sums.
Suppose given a cartesian square:
\[\begin{tikzcd}
	a & b \\
	c & d
	\arrow[from=1-1, to=2-1]
	\arrow[from=2-1, to=2-2]
	\arrow[from=1-1, to=1-2]
	\arrow[from=1-2, to=2-2]
	\arrow["\lrcorner"{anchor=center, pos=0.125}, draw=none, from=1-1, to=2-2]
\end{tikzcd}\]
where the two horizontal morphisms are globular.
The two following squares are cartesian 
\[\begin{tikzcd}
	{b\coprod_aa\star 1} & {b\star 1} & {1\costar a\coprod_a b} & {1\costar b} \\
	{\Sigma c} & {\Sigma d} & {\Sigma c} & {\Sigma d}
	\arrow[from=1-1, to=2-1]
	\arrow[from=2-1, to=2-2]
	\arrow[from=1-1, to=1-2]
	\arrow[from=1-2, to=2-2]
	\arrow["\lrcorner"{anchor=center, pos=0.125}, draw=none, from=1-1, to=2-2]
	\arrow[from=1-3, to=2-3]
	\arrow[from=1-4, to=2-4]
	\arrow[from=2-3, to=2-4]
	\arrow[from=1-3, to=1-4]
	\arrow["\lrcorner"{anchor=center, pos=0.125}, draw=none, from=1-3, to=2-4]
\end{tikzcd}\]
\end{lemma}
\begin{proof}
We show only the cartesianess of the first square, as the cartesianess of the second one follows by applying the duality $(\uvar)^\circ$. A direct computation shows that for any integer $n$, the following square is cartesian
\[\begin{tikzcd}
	{\lambda b\coprod_{\lambda a}\lambda a\star 1} & {\lambda b\star 1} \\
	{\Sigma\lambda c} & {\Sigma \lambda d}
	\arrow[from=1-1, to=2-1]
	\arrow[from=2-1, to=2-2]
	\arrow[from=1-1, to=1-2]
	\arrow[from=1-2, to=2-2]
	\arrow["\lrcorner"{anchor=center, pos=0.125}, draw=none, from=1-1, to=2-2]
\end{tikzcd}\]
To conclude, one has to show that the canonical morphism
$$ \nu(\lambda b)\coprod_{\nu(\lambda a)}\nu(\lambda a\star 1)\to \nu (\lambda b\coprod_{\lambda a}\lambda a\star 1) $$
is an equivalence. 	
As $a\to b$ is globular, all the morphisms of the following cocartesian square are quasi-rigid. 
\[\begin{tikzcd}
	{\lambda a} & {\lambda b} \\
	{\lambda a\star 1 } & {\lambda b\coprod_{\lambda b}\lambda a\star 1 }
	\arrow[from=1-1, to=2-1]
	\arrow[from=1-1, to=1-2]
	\arrow[from=1-2, to=2-2]
	\arrow[from=2-1, to=2-2]
\end{tikzcd}\]
The results then follow from an application of theorem \ref{theo:Kan condition}.
\end{proof}

\p The end of this section is devoted to proving the following theorem: 
\begin{theorem}
\label{theo:appendince unicity of operation}
Let $F$ be an endofunctor of $\zocat$ such that the induced functor $\zocat\to \zocat_{F(\emptyset)/}$ is colimit preserving and $\psi$ an invertible natural transformation between $\Gb\cup \{\emptyset\}\to \zocat\xrightarrow{F}\zocat$ and $\Gb\cup \{\emptyset\}\to \zocat\xrightarrow{G}\zocat$ where $G$ is either the Gray cylinder, the Gray cone, the Gray $\circ$-cone or an iterated suspension.

Then, the natural transformation $\psi$ can be extended to an invertible natural transformation between $F$ and $G$.
\end{theorem}
The previous theorem implies that the equations given in theorem \ref{theo:appendice formula for otimes} and \ref{theo:appendice formula for star} characterize respectively the Gray cylinder, the Gray cone, and the Gray $\circ$-cone.
We also have the following corollary:
\begin{cor}
\label{cor:crushing of Gray tensor is identitye strict case}
The colimit preserving endofunctor $F:\zocat\to \zocat$, sending $[a,n]$ to the colimit of the span
$$\coprod_{k\leq n}\{k\}\leftarrow \coprod_{k\leq n}a\otimes\{k\}\to a\otimes[n]$$
is equivalent to the identity.
\end{cor}
\begin{proof}
The theorem \ref{theo:appendice formula for otimes} implies that the restriction of $F$ to globes is equivalent to the restriction of the identity to globes. As the identity is the $0$-iterated suspension, we can apply theorem \ref{theo:appendince unicity of operation}.
\end{proof}

\begin{lemma}
\label{lemma:sub categgory of Theta}
A sub category $\Theta'$ of $\Theta$, stable by colimit and containing globular morphisms is equal to $\Theta$ iff
\begin{enumerate}
\item for any integer $n$, $i_n^{-}:\Db_n\to \Db_{n+1}$ belongs to $\Theta'$.
\item For any integer $n$, the unit $\Ib_n:\Db_{n+1}\to \Db_n$ belongs to $\Theta'$.
\item For any pair of integers $k<n$, the composition $\triangledown_{k,n}:\Db_n\to \Db_n\coprod_{k}\Db_n$ belongs to $\Theta'$.
\end{enumerate}
\end{lemma}
\begin{proof}
Suppose that $\Theta'$ fulfills these conditions.
As globular morphisms are compositions of pushouts along morphisms of shape $i_n^{-}$, they belong to $\Theta'$.
 As algebraic morphisms are compositions of colimits of morphism of shape $\triangledown_{k,n}$ or $\Ib_n$, they belong to $\Theta'$.
The result then follows from \cite[proposition 3.3.10]{Ara_thesis} that states that every morphism factors as an algebraic morphism followed by a globular morphism.
\end{proof}

\begin{lemma}
\label{lemma:unit forced}
Let $n$ be an integer, and $G$ be either the Gray cylinder, the Gray cone, the Gray $\circ$-cone or an iterated suspension, and suppose
given a square 
\[\begin{tikzcd}
	{G(\Db_n)} \\
	& {G(\Db_{n+1})} & {G(\Db_n)} \\
	{G(\Db_n)}
	\arrow["f", from=2-2, to=2-3]
	\arrow["{G(i_n^-)}", from=1-1, to=2-2]
	\arrow["{G(i_n^+)}"', from=3-1, to=2-2]
	\arrow["id", curve={height=-18pt}, from=1-1, to=2-3]
	\arrow["id"', curve={height=18pt}, from=3-1, to=2-3]
\end{tikzcd}\]
Then, the morphism $f$ is $G(\Ib_n)$.
\end{lemma}
\begin{proof}
As the proof for any possibilities of $G$ are similar, we will show only the case $G:=\uvar\otimes [1]$.
As for any integer $n$, $\Db_n\otimes[1]$ admits a loop free and atomic basis, we can then show the desired assertion after applying the functor $\lambda$.
Remark first that the assumption implies that $\partial f((e_{n+1}\otimes \{\alpha\})=0$, and so $f((e_{n+1}\otimes \{\alpha\}) =0$. We also have $f(e_{n+1}\otimes[1])=0$ as $(\lambda (\Db_n\otimes[1])_{n+2} =0$. This implies that $f$ is equal to $\lambda(G(\Ib_n))$.
\end{proof} 
\begin{lemma}
\label{lemma:comp forced}
Let $k<n$ be two integers, and $G$ be either the Gray cylinder, the Gray cone, the Gray $\circ$-cone or an iterated suspension, and suppose
given a square 
\[\begin{tikzcd}
	{G(\Db_{n-1})} && { G(\Db_{n-1}\coprod_k\Db_{n-1})} \\
	& {G(\Db_n)} && { G(\Db_{n}\coprod_k\Db_n)} \\
	{G(\Db_{n-1})} && { G(\Db_{n-1}\coprod_k\Db_{n-1})}
	\arrow["f", from=2-2, to=2-4]
	\arrow["{G(i_n^-)}"{description}, from=1-1, to=2-2]
	\arrow["{G(i_n^+)}"{description}, from=3-1, to=2-2]
	\arrow["{G(i_n^+)\coprod_k G(i_n^+)}"{description}, from=3-3, to=2-4]
	\arrow["{G(i_n^-)\coprod_k G(i_n^-)}"{description}, from=1-3, to=2-4]
	\arrow["{\triangledown_{n-1,k}}"', from=3-1, to=3-3]
	\arrow["{\triangledown_{n-1,k}}", from=1-1, to=1-3]
\end{tikzcd}\]
where we set $\triangledown_{n,n}:=id$.
Then, the morphism $f$ is $G(\triangledown_{n,k})$.
\end{lemma}
\begin{proof}
As the proof for any possibilities of $G$ are similar, we will show only the case $G:=\uvar\otimes [1]$.
As for any integer $n$, $\Db_n\otimes[1]$ admits a loop free and atomic basis, we can then show the desired assertion after applying the functor $\lambda$. Suppose first that $k<n-1$.
By assumption, we have 
$$
\begin{array}{rcl}
\partial f(e_n\otimes \{\alpha\})&=& \partial (e_n^0\otimes \{\alpha\} +e_n^1\otimes \{\alpha\})\\
\partial f(e_n\otimes [1])&=& \partial (e_n^0\otimes [1]) + \partial (e_n^1\otimes [1]) \\
\end{array}
$$
This forces the equalities
$$
\begin{array}{rcl}
 f(e_n\otimes \{\alpha\})&=& e_n^0\otimes \{\alpha\} +e_n^1\otimes \{\alpha\}\\
 f(e_n\otimes [1])&=& e_n^0\otimes [1] + e_n^1\otimes [1] \\
\end{array}
$$
and $f$ is then equal to $\triangledown_{n,k}\otimes[1]$. The case $k=n-1$ is similar.
\end{proof}

\begin{proof}[Proof of theorem \ref{theo:appendince unicity of operation}]
As every globular sum is a colimit of globes, we can extend $\psi$ to a (\textit{a priori} non natural) transformation, 
$\psi:F_{|\Theta}\to G_{|\Theta}$.
Let $\Theta'$ be the maximal sub category of $\Theta$ such that $\psi_{\Theta'}$ is an equality.
As $G(\Db_n)$ does not have non trivial automorphisms, the assumption implies that	 $\Theta'$ fulfills the first condition of lemma \ref{lemma:sub categgory of Theta}. The lemma \ref{lemma:unit forced} implies that it fulfills the second condition, and an easy induction on $(n-k)$ using lemma \ref{lemma:comp forced} implies that it fulfills the last condition. Applying the lemma \ref{lemma:sub categgory of Theta}, this concludes the proof.
\end{proof}

%

\part{On the side of models}

\chapter{Study of the complicial model}
\label{chapter:Studies of the complicial model}

\minitoc
\vspace{1cm}
%
%
%
%
%
%
%
%
%
%
%
%
%

This chapter is devoted to the study of Verity's complicial sets (\cite{Verity_complicial_set_characterising_the_simplicial_nerve}).
One of the benefits of complicial sets is that they admit a simple definition of the Gray tensor product. Being strongly linked to $\zo$-categories by the Street nerve, they are also a privileged framework for stating and proving strictification results, as done in \cite{Ozornova_Fundamental_pushouts_of_n_complical_set}, \cite{Gagna_Nerves_and_cones_of_free_loop_free_omega_categories}, \cite{Ozornova_a_quillen_adjunction_between_globular_and_complicial} and \cite{Maehara_oriental_as_free_weak_omega_categories}. 
However, they do not interact \textit{a priori} well with the globular language. The goal of this chapter is to show that, with some computation, it is possible to have a globular point of view in this model. 

The first section is a recollection of usual results and definitions about complicial sets. 
In the second section, we aim to prove an analogue of the formula given in \ref{theo:appendice formula for otimes} to the complicial setting.
We also have a suspension in this category, which is denoted by $X\mapsto \Sigma X$. Objects $[1]\fwedge \Sigma X$ and $\Sigma X\fwedge [1]$ are defined in \ref{subsection:wedge}, but for now, we can suppose that they are fibrant replacements of respectively $[1]\coprod_{[0]}\Sigma X$ and $\Sigma X\coprod_{[0]}[1]$.
They come along with morphisms that are analogue to whiskerings, and that we also note by $\triangledown$: 
$$\triangledown:\Sigma X\to [1]\fwedge\Sigma X ~~~~\mbox{and}~~~~ \triangledown:\Sigma X\to\Sigma X\fwedge [1].$$ 
We then show the following theorem:
\begin{itheorem}[\ref{theo:interval_first_formula}]
There exists a zigzag of acyclic cofibrations, natural in $X$, between $(\Sigma X)\otimes [1]$ and the colimit of the following diagram:
 $$\Sigma X\fwedge [1]\xleftarrow{\triangledown} \Sigma (X\otimes\{0\}) \hookrightarrow \Sigma (X\otimes[1])\hookleftarrow \Sigma (X\otimes\{1\})\xrightarrow{\triangledown} [1]\fwedge \Sigma X.$$
\end{itheorem}
We also provide similar formulas for the \textit{Gray cone} and Gray \textit{$\circ$-cone}:
\begin{itheorem}[\ref{theo:cyl_formula}]
There exists a zigzag of acyclic cofibrations, natural in $X$, between $\Sigma X \star[0]$ and the colimit of the following diagram: 
$$ \Sigma X\fwedge [1]\leftarrow \Sigma X\to \Sigma([0]\costar X).$$
There exists a zigzag of acyclic cofibrations, natural in $X$, between  $[0]\costar \Sigma X$ and the colimit of the following diagram: 
$$\Sigma(X\star[0]) \leftarrow \Sigma X\to [1]\fwedge\Sigma X.$$
\end{itheorem}
The third section uses this formula and the strictification result of Gagna, Ozornova and	 Rovelli (\cite{Gagna_Nerves_and_cones_of_free_loop_free_omega_categories}) to demonstrate a criterion for detecting autoequivalences of complicial sets by their behavior on globes.
Indeed, in section \ref{section:Globular equivalences}, by iterating the suspension, we construct a globular object: 
\[\begin{tikzcd}
	{\Db_0} & {\Db_1} & {\Db_2} & {...}
	\arrow["{i_0^+}", shift left=2, from=1-1, to=1-2]
	\arrow["{i_1^+}", shift left=2, from=1-2, to=1-3]
	\arrow["{i_3^+}", shift left=2, from=1-3, to=1-4]
	\arrow["{i_0^-}"', shift right=2, from=1-1, to=1-2]
	\arrow["{i_1^-}"', shift right=2, from=1-2, to=1-3]
	\arrow["{i_3^-}"', shift right=2, from=1-3, to=1-4]
\end{tikzcd}\]
\begin{itheorem}[\ref{theo:criterion_to_be_linked_to_identity}]
Let $i$ be a left Quillen endofunctor for the model category for complicial sets. Suppose that there exists a zigzag of weakly invertible natural transformations:
$$i(\Db_{\uvar}) \leftrightsquigarrow \Db_{\uvar}.$$
Then, there exists a zigzag of weakly invertible natural transformations between $i$ and $id$.
\end{itheorem} 
Proposition 15.10 of \cite{Barwick_on_the_unicity_of_the_theory_of_higher_categories} provides a similar result for models of $(\infty,n)$-categories.

\section{Preliminaries}

\subsection{Generalities on model categories}
\label{chapter:Generalities on model categories}
For this chapter, we fix a model category $C$ whose cofibrations are monomorphisms.

\p We give first some results on homotopy colimits. These results will be used freely throughout these first two chapters.

\begin{prop}
\label{prop:hom colimit 1}
Suppose given a square
\[\begin{tikzcd}
	a & b \\
	c & d
	\arrow[from=1-1, to=2-1]
	\arrow[from=1-2, to=2-2]
	\arrow[from=2-1, to=2-2]
	\arrow[from=1-1, to=1-2]
\end{tikzcd}\]
such that the two horizontal morphisms are weak equivalences. Then this square is homotopy cocartesian. 
\end{prop}
\begin{proof}
This is \cite[proposition 2.3.26]{Cisinski_Higher_categories_and_homotopical_algebra}.
\end{proof}
\begin{prop}
\label{prop:hom colimit 2}
Suppose given a cocartesian square
\[\begin{tikzcd}
	a & b \\
	c & d
	\arrow[from=1-1, to=2-1]
	\arrow[from=1-2, to=2-2]
	\arrow[from=2-1, to=2-2]
	\arrow[from=1-1, to=1-2]
	\arrow["\lrcorner"{anchor=center, pos=0.125, rotate=180}, draw=none, from=2-2, to=1-1]
\end{tikzcd}\]
where the left vertical morphism is a cofibration. Then this square is homotopy cocartesian. 
\end{prop}
\begin{proof}
This is \cite[corollary 2.3.28]{Cisinski_Higher_categories_and_homotopical_algebra}.
\end{proof}

\begin{prop}
\label{prop:hom colimit 3}
Let $F:\alpha\to C$ be a diagram indexed by an ordinal. The transfinite composition $\colim_\alpha F$ is the homotopy colimit of the diagram $F$.
\end{prop}
\begin{proof}
This is \cite[proposition 2.3.13]{Cisinski_Higher_categories_and_homotopical_algebra}.
\end{proof}

\begin{prop}
\label{prop:hom colimit 4}
Suppose given a diagram 
\[\begin{tikzcd}
	& {b_0} && {...} && {b_{n-1}} \\
	{a_0} && {a_1} && {a_{n-1}} && {a_{n}}
	\arrow[from=1-2, to=2-1]
	\arrow[hook, from=1-2, to=2-3]
	\arrow[from=1-4, to=2-3]
	\arrow[hook, from=1-4, to=2-5]
	\arrow[from=1-6, to=2-5]
	\arrow[hook, from=1-6, to=2-7]
\end{tikzcd}\]
where all morphisms labelled by $\hookrightarrow$ are cofibrations.
The colimit of this diagram is also the homotopy colimit of this diagram.
\end{prop}
\begin{proof}
Let $I_n$ be the category indexing the previous diagram. We denote $i_0$, $j_0$,..., $i_{n-1}$, $j_{n-1}$, $i_{n}$ it's objects.
The projective model structure on $\Fun(I_n,C)$ is given by functor $G$ such that for any $k<n$, $F(j_k)\to F(i_k)$, $F(j_k)\to F(i_{k+1})$ are monomorphisms, and such that for any $0<k<n$, $F(j_k)\coprod F(j_{k+1}) \to F(i_k)$ is a monomorphism. Remark that such presheaves verify the condition given in the statement of the proposition.

 We will show on induction on $n$ that a natural transformation $\psi$ between two diagrams $F,G:I_n\to C$ that fulfills the desired condition induces a weak equivalence between their colimits. As we can always chose $F$ to be the cofibrant replacement of $G$ in the projective model structure on $\Fun(I_n,C)$, it will imply the desired result. 

The case $n=1$ is proposition \ref{prop:hom colimit 2}. Suppose now the result is true at the stage $(n-1)$ and let $\psi$ be a weakly invertible natural transformation between two diagram $F,G:I_n\to C$ that fulfills the desired condition. We denote by $\iota:I_{n-1}\to I_n$ the canonical inclusion that sends $i_k$(resp. $j_k$) on $i_k$(resp. $j_k$) for $k<n$ (resp. $k<n-1$).
We then have a diagram 
\[\begin{tikzcd}
	{\colim_{I_{n-1}}F\circ\iota} & {F(j_{n-1})} & {F(i_n)} \\
	{\colim_{I_{n-1}}G\circ\iota} & {G(j_{n-1})} & {G(i_n)}
	\arrow[hook, from=1-2, to=1-3]
	\arrow[from=1-2, to=1-1]
	\arrow["\sim"', from=1-1, to=2-1]
	\arrow["\sim"', from=1-3, to=2-3]
	\arrow["\sim"', from=1-2, to=2-2]
	\arrow[from=2-2, to=2-1]
	\arrow[hook, from=2-2, to=2-3]
\end{tikzcd}\]
where all arrows labeled by $\sim$ are weak equivalences. Remark furthermore that the limit of the two lines are respectively $\colim_{I_{n}}F$ and $\colim_{I_{n}}G$. A last application of proposition \ref{prop:hom colimit 2} concludes the proof.
\end{proof}

\p 
The definition of elegant Reedy category is given in paragraph \ref{para:reedy}.
As all the presheaves categories that we will encounter through this text are presheaves on elegant Reedy categories, we will use freely the following theorem:
\begin{theorem}[Hirschhorn]
\label{theo:hom colimi}
We suppose that $C$ is a simplicial model category.
Let $A$ be a elegant Reedy category, and $F:A\to C$ a functor such that the induced morphism $\colim_{\partial a}F\to F(a)$ is a monomorphism for any object $a$. The object $\colim_A F$ is the homotopy colimit of $F$. In particular, if $C$ is $\Psh{A}$, every object $X$ is the homotopy colimit of the diagram $A_{/X}\to A\to \Psh{A}$.
\end{theorem}
\begin{proof}
Using the characterization of elegant Reedy category given by proposition 3.8 of \cite{Bergner_reedy_category_and_the_theta_construction}, and \cite[proposition 15.10.2]{Hirschhorn_Model_categories_and_their_localizations},
it's easy to see that they have fibrant constant in the sens of \cite[definition 15.10.1]{Hirschhorn_Model_categories_and_their_localizations}. We can then apply the theorem 19.9.1	 of \cite{Hirschhorn_Model_categories_and_their_localizations}.
\end{proof}

\p
\label{para:nice model structure}
A model structure is \wcnotion{nice}{nice model structure} if it is simplicial, combinatorial, cartesian and its cofibrations are monomorphisms. 

\begin{notation}
Let $\uvar\square\uvar:C\times D\to E$ be a bifunctor. 
If $f:a\to b$ and $g:x\to y$ are respectively morphisms of $C$ and $D$, we will note by $f~\hat{\square}~g$ the induced morphism $a\square y\coprod_{a\square x} b\square x\to b\square y$.\sym{((g38@$\hat{\square}$}
\end{notation}
\begin{prop}[{\cite[proposition A.3.7.3]{Lurie_Htt}}]
\label{prop:left_boosfiled_localization}
Let $A$ be a nice model structure and $S$ a set of cofibrations. There exists a model structure $A_S$ on the same category, and a left Quillen adjoint $L:A\to A_S$, such that an object is fibrant in $A_S$ if and only if it is fibrant in $A$ and has the right lifting property against all morphisms of shape $i\htimes f$ where $i$ is a cofibration and $f$ in $S$. Moreover, a left Quillen functor $F:A\to C$ lifts to $A_{S}$ if and only if for any cofibration $i$ and morphism $f\in S$, $F(i\htimes f)$ is a weak equivalence.
\end{prop}

\begin{cor}
\label{cor:left_boosfiled_localization}
Let $A$, $C$ be two nice model categories, $F:A\to C$ a left Quillen functor, $S$ a set of cofibrations and $T$ a set of morphisms such that for any cofibrations $i$ and morphisms $f\in S$, the morphism $i\htimes f$ is included in the smallest saturated class stable by two out of three, containing weak equivalences and $T$. Then a left Quillen functor $F:A\to C$ lifts to $A$ if and only if it sends morphisms of $T$ to weak equivalences.
\end{cor}
\begin{proof}
Let $U$ be the class of morphisms in $A$ that are sent to weak equivalences by $F$. This class is obviously stable by two out of three, retracts and contains weak equivalences. As the model structure on $C$ is combinatorial and left proper, it is saturated. The class $U$ then includes all morphisms of shape $i\htimes f$ for $i$ a cofibration and $f\in S$, which implies that $F$ can be lifted to $A_S$.
\end{proof}

\p Let $i:A\to B$ and $i':A'\to B'$ be two cofibrations. A \notion{zigzag of acyclic cofibration} between $i$ and $i'$, denoted $i\leftrightsquigarrow i'$ is a zigzag in the category of arrows such that all the horizontal maps are acyclic cofibrations, and all the vertical maps are cofibrations.

\begin{lemma}
Let $i$ and $j$ be two cofibrations, and $f:X\to Y$ a fibration between fibrant objects. Suppose that we have a morphism in the category of arrows $i\to j$ which is  pointwise an acyclic cofibration. Then, if $j$ has the left lifting property against $f$, so has $i$.
\end{lemma}
\begin{proof}
We consider a diagram of the following shape:
\[\begin{tikzcd}
	A & {A'} & X \\
	B & {B'} & Y.
	\arrow["i", from=1-1, to=2-1]
	\arrow["\sim", from=1-1, to=1-2]
	\arrow["\sim"', from=2-1, to=2-2]
	\arrow["j"', from=1-2, to=2-2]
	\arrow[from=1-3, to=2-3]
	\arrow[curve={height=-18pt}, from=1-1, to=1-3]
	\arrow[curve={height=18pt}, from=2-1, to=2-3]
	\arrow["{l_0}"', dotted, from=2-2, to=2-3]
	\arrow["{l_1}"{description}, dotted, from=1-2, to=1-3]
	\arrow["{l_2}"{description}, dotted, from=2-2, to=1-3]
\end{tikzcd}\]
We construct, one after the other, the lifting $l_0$, $l_1$ and $l_2$.
\end{proof}

\begin{lemma}
Let $i$ and $j$ be two cofibrations, and $f:X\to Y$ a fibration between fibrant objects. Suppose that we have a morphism in the category of arrows $i\to j$ which is  pointwise an acyclic cofibration. Then, if $i$ has the right lifting property against $f$, so has $j$.
\end{lemma}
\begin{proof}
We consider a diagram of the following shape:
\[\begin{tikzcd}
	A & {A'} &&& X \\
	B & {B\coprod_A A'} \\
	&& {B'} && Y.
	\arrow[from=1-1, to=2-1]
	\arrow["\sim"', from=1-1, to=1-2]
	\arrow["\sim", from=2-1, to=2-2]
	\arrow["\sim"{description}, from=2-2, to=3-3]
	\arrow["\sim"', curve={height=6pt}, from=2-1, to=3-3]
	\arrow[curve={height=-6pt}, from=1-2, to=3-3]
	\arrow[from=1-2, to=2-2]
	\arrow[from=1-2, to=1-5]
	\arrow[from=3-3, to=3-5]
	\arrow[from=1-5, to=3-5]
	\arrow["{l_0}"{description}, dotted, from=2-2, to=1-5]
	\arrow["{l_1}"{description}, dotted, from=3-3, to=1-5]
\end{tikzcd}\]
We construct, one after the other, the lifting $l_0$, $l_1$.
\end{proof}

\begin{prop}
\label{prop:lifting_property_zigzag_of_acyclic_cofibration}
Let $f$ be a fibration between fibrant objects and $i$ and $j$ two cofibrations such that there exists a zigzag of acyclic cofibrations $i\leftrightsquigarrow j$. Then $f$ has the right lifting property against $i$ if and only if it has the right lifting property against $j$. 
\end{prop} 
\begin{proof}
This is a direct consequence of the last two lemmas.
\end{proof}

\subsection{Marked and stratified presheaves}
\label{section:Marked and stratified presheaves}
\p 
Let $B$ be an elegant Reedy category and $M$ a subset of the set of objects of $B$. A \wcnotion{$M$-stratified presheaf on $B$}{stratified presheaf on $B$}, or just a \textit{stratified prehsheaf on $B$} when the subset $M$ will be non-ambiguous, is a pair $(X,tX)$ where $X$ is a presheaf on $B$ and $tX:=\coprod_{a\in M} tX_a$ is the disjoint union of sets, such that for any $a\in M$, $tX_a$ is a subset of $X_a$ including degeneracies, i.e the image of morphisms $X_p:X_b\to X_a$ for $p:b\to a$ in $B_-$.

A \notion{stratified morphism} $f:(X,tX)\to (Y,tY)$ is the data of a morphism on the underlying presheaf such that $f(tX_n)\subset tY_n$.
The category of stratified presheaves is denoted by \wcnotation{$\tPshM{B}$}{(tpsh@$\tPshM{B}$}. 

A morphism between two stratified presheaves is \wcnotion{entire}{entire morphism} if it is the identity on the underlying presheaves.

We then have an adjunction
\[\begin{tikzcd}
	{(\uvar)^\flat:\Psh{B}} & {\tPshM{B}:(\uvar)^\natural}
	\arrow[""{name=0, anchor=center, inner sep=0}, shift left=2, from=1-1, to=1-2]
	\arrow[""{name=1, anchor=center, inner sep=0}, shift left=2, from=1-2, to=1-1]
	\arrow["\dashv"{anchor=center, rotate=-90}, draw=none, from=0, to=1]
\end{tikzcd}\]
where the left adjoint is a fully faithful inclusion that sends a presheaf $X$ onto $(X,S)$ where $S$ is the smaller stratification on $X$, and where the right adjoint is the obvious forgetful functor. We will identify presheaves on $B$ with their image by the functor $(\uvar)^\flat$.

\p If $b$ is an object of $M$, we denote by $b_t$ the stratifed presheaf $(b,S)$, where $S$ is the smaller stratification that includes $id:b\to b$.

We then define $t_MB$ as the full subcategory of $\tPshM{B}$ spanned by the objects of shape $a$ or $b_t$ with $a\in B$ and $b\in M$. We then have equalities:
$$\begin{array}{rcl}
\Hom_{t_MB}(a,b)&:=& \Hom_B(a,b),\\
\Hom_{t_MB}(a,b_t)&:=& \Hom_B(a,b),\\
\Hom_{t_MB}(a_t,b)&:=& \Hom_B(a,b)\cap B_- \diagdown \{id_a\}, \\
\Hom_{t_MB}(a_t,b_t)&:=&\Hom_B(a,b)\cap B_-.\\
\end{array}$$
The canonical functor $B\to t_MB$ is then fully faithful and we will identify object of $B$ with their image through this functor.
\begin{prop}
\label{prop:reedy structure on tB}
The category $t_MB$ admits a structure of elegant Reedy category, that makes the inclusion $B\to t_MB$ a morphism of Reedy category. There is no non trivial negative morphism whose codomain is of shape $b_t$ for $b\in M$. There is no non trivial positive morphism whose domain is of shape $b_t$ for $b\in M$. 
\end{prop}
\begin{proof}
We define the degree degree function $ob(t_MB)\to \Nb$ by the assignment 
$$d'(b):= 2 d(b)~~~~~d'(b_t):= 2d(b)+1$$
The category $(t_MB)_+$ is the smallest that includes $B_+$ and morphisms of shape $a\to a_t$. The category $(t_MB)_-$ is the smallest that includes $B_-$ and morphisms of shape $b_t\to a$.

To prove the axioms of Reedy category, we can replicate the strategy used in proposition C.2 of \cite{Ozornova_model_structure_for_infini_n_categories} with obvious modification to this more general framework.

We still have to show that $tB$ is elegant. Let $X$ be a presheaf on $t_MB$, $a$ an element of $t_MB$, $f:a\to a'$ and $g:a\to a'$ two negative morphisms, an element $x$ of $X(a)$, two non degenerate elements $y\in X(a')$ and $z\in X(a'')$ such that $f^*y=x$, $g^*z=x$. 

Suppose first that $a$ is in $B$. In this case, $f$ and $g$ are also in $B$, and as this Reedy category is elegant by assumption, this implies $f=g$ and $y=z$. Suppose now that $a$ is of shape $b_t$ for $b\in B$. We denote $\alpha$ the canonical morphism $\alpha:b\to b_t$. By definition of negative morphism, the codomain of $f$ and $g$ are in $B$. The morphisms $\alpha f $ and $\alpha g$ then are in $B$. Moreover, these two morphisms are negative, and we have $(\alpha f)^*y=\alpha^* x$, $(\alpha g)^*z=\alpha^* x$. As $B$ is elegant, $\alpha f=\alpha g$ and $y=z$. Eventually, remark that the first equality implies that $f$ is equal to $g$. 
\end{proof}
A cellular model for $t_MB$ is given by $C\cup\{b\to b_t,b\in M\}$ where $C$ is a cellular model for $B$.

\p The category of $M$-stratified presheaves is then equivalent to the fully faithful subcategory of presheaves $X$ on $t_MB$ such that for any $b\in M$, $X(b_t)\to X(b)$ is a monomorphism. 
In particular, we have an adjunction 
\begin{equation}
\label{eq:entre presheaveds on tB et stratified presheages}
\begin{tikzcd}
	{\pi:\Psh{t_MB}} & {\tPshM{B}:\iota}
	\arrow[""{name=0, anchor=center, inner sep=0}, shift left=2, from=1-1, to=1-2]
	\arrow[""{name=1, anchor=center, inner sep=0}, shift left=2, from=1-2, to=1-1]
	\arrow["\dashv"{anchor=center, rotate=-90}, draw=none, from=0, to=1]
\end{tikzcd}
\end{equation}
Remark furthermore that the unit $X\to \iota \pi X$ is a trivial fibration. Indeed, the cellular model is given $C\cup\{b\to b_t,b\in M\}$, where  $C$ is a cellular model for $B$, and the unit obviously has the right lifting property against it.

\begin{prop}
\label{prop:transfert from presheaves on tB to stratified presheaves}
Suppose given a combinatorial on $\Psh{t_MB}$ whose cofibrations are monomorphisms. Then there exists a combinatorial model structure on $\tPshM{B}$ making the adjunction \ref{eq:entre presheaveds on tB et stratified presheages} a Quillen equivalence.

A morphism of $\tPshM{B}$ is a cofibration if and only if it is a monomorphism. A morphism is a fibration (resp. a weak equivalence) if and only if its image by $\iota$ is.
\end{prop}
\begin{proof}
We are willing to apply \cite[theorem 11.3.2]{Hirschhorn_Model_categories_and_their_localizations}. As two adjoints of \eqref{eq:entre presheaveds on tB et stratified presheages} preserve smallness, the first condition is obviously fulfilled. Using the fact that $\iota$ is fully faithful, the second condition of theorem \textit{op cit} is equivalent to asking that for any acyclic cofibration $i$ of $\Psh{t_MB}$, the morphism $\iota\pi i$ is a weak equivalence. As the unit $id\to \iota \pi $ is pointwise a trivial fibration, this directly follows from the stability of weak equivalences by two out of three.

This provides the model structure. As the unit is pointwise a trivial fibration and the counit is the identity, the adjunction
\eqref{eq:entre presheaveds on tB et stratified presheages} induces a Quillen equivalence. 
\end{proof}

\p We now fix a Reedy category $B$, a subset $M$ of objects of $B$, and we suppose given a nice model structure on $\tPshM{B}$ (as defined in paragraph \ref{para:nice model structure}).
A \wcnotion{$M$-marked presheaf on $B$}{marked presheaf on $B$} is a stratified presheaf having the unique right lifting property against all entire acyclic cofibrations. In particular, any fibrant objects is marked. 

We denote by \wcnotation{$\mPshM{B}$}{(mpsh@$\mPshM{\uvar}$} the full subcategory of marked presheaves on $B$. We then have an adjunction: \sym{((b91@$(\uvar)_{\mk}$}
\begin{equation}
\label{eq:adj beetwen stratified and marked}
\begin{tikzcd}
	{(\uvar)_{\mk}:\tPshM{B}} & {\mPshM{B}:\iota}
	\arrow[""{name=0, anchor=center, inner sep=0}, shift left=2, from=1-2, to=1-1]
	\arrow[""{name=1, anchor=center, inner sep=0}, shift left=2, from=1-1, to=1-2]
	\arrow["\dashv"{anchor=center, rotate=-90}, draw=none, from=1, to=0]
\end{tikzcd}
\end{equation}
where the left adjoint $(\uvar)_{\mk}$ sends a stratified presheaf $(X,tX)$ to the marked presheaf $(X,\overline{tX})$, where $\overline{tX}$ is the smaller stratification that includes $tX$ and makes $(X,\overline{tX})$ a marked presheaf, and where the right adjoint is a fully faithful inclusion.
Remark furthermore that at the level of presheaves, these two adjoints are the identity. 

\begin{prop}
\label{prop:X to Xmk is acycli cof}
Let $X$ be a $M$-stratified presheaf on $B$.
The canonical morphism $X\to \iota (X_{\mk})$ is an entire acyclic cofibration.
\end{prop}
\begin{proof}
Let $\kappa$ be a regular cardinal such that $X$ is $\kappa$-small. Remark first the domain of a entire monomorphism is $\kappa$-small if and only if its codomain is.

Let $I$ be the set of entire acyclic cofibrations with $\kappa$-small codomains and domains. This set generates via the small object argument a weak factorization system, and we denote by $X\to X'\to 1$ the factorization of $X\to 1$. We are willing to show that $X'$ is $M$-marked. As $X\to X'$ is an entire acyclic cofibration by construction, this will directly imply that $X'$ is equal to $\iota (X_{\mk})$ and so demonstrate the desired result.

Suppose then given a diagram
\[\begin{tikzcd}
	K & {X'} \\
	L & 1
	\arrow[from=2-1, to=2-2]
	\arrow[from=1-1, to=1-2]
	\arrow[from=1-2, to=2-2]
	\arrow["i"', from=1-1, to=2-1]
\end{tikzcd}\]
with $i$ an entire acyclic cofibration. We have to show that it admits a lift.
Remark that this square factors as:
\[\begin{tikzcd}
	K & {X'} & {X'} \\
	L & {X'\coprod_KL} & 1
	\arrow[from=1-1, to=1-2]
	\arrow["i"', from=1-1, to=2-1]
	\arrow[from=1-3, to=2-3]
	\arrow[Rightarrow, no head, from=1-2, to=1-3]
	\arrow[from=2-1, to=2-2]
	\arrow[from=2-2, to=2-3]
	\arrow["{i'}", from=1-2, to=2-2]
	\arrow["\lrcorner"{anchor=center, pos=0.125, rotate=180}, draw=none, from=2-2, to=1-1]
\end{tikzcd}\]
The morphism $i'$ is an entire acyclic cofibration with $\kappa$-small codomain and domain and then belongs to $i$. The right square of the previous diagram then admits a lift. This induces a lift in the in the original square, and this concludes the proof.
\end{proof}
\begin{prop} 
\label{prop:model structure on marked presheaves}
Suppose given a nice model structure on $\tPshM{B}$.
This induces a nice model structure on $\mPshM{B}$, making the adjunction \eqref{eq:adj beetwen stratified and marked} a Quillen equivalence. A morphism between two marked presheaves is a cofibration (resp. a fibration) (resp. a weak equivalence) if it is a cofibration (resp. a fibration) (resp. a weak equivalence) when seen as a morphism of $\tPshM{B}$. 
\end{prop} 
\begin{proof} Let $f:X\to Y$ be a fibration between stratified presheaves. If $Y$ is marked, so is $X$. The two weak factorization systems on $\mPshM{B}$ are then induced by the one of $\tPshM{B}$. We leave it to the reader to check that this model structure is nice. 

The unit is pointwise a weak equivalence according to proposition \ref{prop:X to Xmk is acycli cof} and the counit is the identity. 
The adjunction \eqref{eq:adj beetwen stratified and marked} is then a Quillen equivalence.
\end{proof}

\section{The complicial model}

\subsection{Model structure on marked simplicial sets}
This section is a recollection of  the principal results of \cite{Ozornova_model_structure_for_infini_n_categories}. We refer to \cite{Rhiel_Complicial_sets_an_ouverture} for an introduction to complicial sets.

\p
A \notion{stratified simplicial set} is a pair $(X,tX)$ where $X$ is a simplicial set and $tX := \cup_{n>0}tX_n$
 a graded set such that for any $n\geq 1$, $tX_n$ is a subset of $X_n$ that includes all degenerate simplices. A simplex in $tX$ is called \wcnotion{thin}{thin simplex}.

A \textit{stratified morphism} $f:(X,tX)\to (Y,tY)$ is the data of a morphism on the underlying simplicial set such that $f(tX_n)\subset tY_n$.
The category of stratified simplicial sets is denoted by \wcnotation{$\stratSset$}{(tpshdelta@$\stratSset$}.

Given a functor $i:I\mapsto (F(i),tF(i))$ with value in stratified simplicial sets, its colimit is given by $(\colim F(i),M)$ where $M$ is the smaller stratification that includes the image of $tF(i)\to \colim F(i)$ for any $i:I$.

We can extend the join to stratified simplicial sets as follows: 
If $(X,tX)$ and $(Y,tY)$ are two stratified simplicial sets, we define $tX\star tY$ as the set of simplices of $X\star Y$ of shape $x\star y$ where either $x$ or $y$ are thin. We then define 
$$(X,tX)\star (Y,tY) := (X\star Y, tX\star tY).$$

\begin{definition}
A stratified monomorphism $f:X\to Y$ is 
\begin{enumerate}
\item \textit{entire} if it is an identity on underlying simplicial sets.
\item \wcnotion{regular}{regular morphism} if for every $n\geq 1$ the following diagram is a pullback:
\[\begin{tikzcd}
	{tX_n} & {X_n} \\
	{tY_n} & {Y_n}.
	\arrow[from=2-1, to=2-2]
	\arrow[from=1-2, to=2-2]
	\arrow[from=1-1, to=1-2]
	\arrow[from=1-1, to=2-1]
	\arrow["\lrcorner"{anchor=center, pos=0.125}, draw=none, from=1-1, to=2-2]
\end{tikzcd}\]
\end{enumerate}
\end{definition}

\begin{definition}
We define several stratified structures on $[n]$. 
\begin{enumerate}
\item \wcnotation{$[n]_t$}{((g31@$[n]_t$}. The top $n$-simplex is thin. All degeneracies are thin.
\item \wcnotation{$[n]^k$}{((g32@$[n]^k$}. All simplices that include $\{k-1,k,k+1\}\cap[n]$ are thin. All degeneracies are thin.
\item \wcnotation{$([n]^k)'$}{((g33@$([n]^k)'$}. All simplices that include $\{k-1,k,k+1\}\cap[n]$, together with the $(k-1)$-face and the $(k+1)$ face are thin. All degeneracies are thin.
\item \wcnotation{$([n]^k)''$}{((g34@$([n]^k)''$}. All simplices that include $\{k-1,k,k+1\}\cap[n]$, together with the $(k-1)$-face, the $k$-face and the $(k+1)$ face are thin. All degeneracies are thin.
\item \wcnotation{$[3]^{eq}$}{((g35@$[3]^{eq}$}. All simplices of dimension strictly higher than $2$, together with $[0,2]$ and $[1,3]$ are thin. All degeneracies are thin.
\item \wcnotation{$[n]^\sharp$}{((g36@$[n]^\sharp$}. All simplices are thin.
\end{enumerate}
\end{definition}
\begin{definition}	
An \snotion{elementary anodyne extension}{for stratified simplicial sets} is one of the following:
\begin{enumerate}
\item The \notion{complicial horn inclusions} are the regular extensions:
$$\Lambda^k[n]\to [n]^k,~n\geq 1,~ n\geq k\geq 0.$$
\item The \notion{complicial thinness extensions}:
$$([n]^k)'\to ([n]^k)'',~n\geq 2,~ n\geq k\geq 0.$$
\item The \notion{saturation extensions}:
$$[n]\star[3]^{eq}\star[m]\to [n]\star[3]^{\sharp}\star[m],~ n,m\geq -1.$$
\end{enumerate}
The set of complicial horn inclusions is $\Lambda$ and the reunion of \textit{complicial thinness extensions} and of \textit{saturation extensions} is $S$.
\end{definition}

\begin{definition}
\label{defi:complicial set}
Let $n\in \Nb\cup\{\omega\}$.	
A \wcnotion{$n$-complicial set}{complicial set} is a stratified set having the right lifting property against all elementary anodyne extensions and against all morphisms $[k]\to [k]_t$ for $k>n$. 
\end{definition}

\begin{theorem}[Ozornova, Rovelli, Verity]
\label{theo:model structure on complicial set}
Let $n\in \Nb\cup\{\omega\}$.	
There exists a nice model structure on stratified simplicial sets, denoted by $\stratSset^n$, whose fibrant objects are $n$-complicial sets. 

A left adjoint $F:\stratSset\to D$ to a model category is a left Quillen functor if it preserves cofibrations and sends all elementary anodyne extensions and morphisms $[k]\to [k]_t$, for $k>n$, to weak equivalences. \sym{(tpshdeltan@$\stratSset^n$}
\end{theorem}
\begin{proof}
This is \cite[theorem 1.25]{Ozornova_model_structure_for_infini_n_categories}.
\end{proof}
During this chapter, we will only be interested in the model structure for $\omega$-complicial sets, and we will therefore drop the index $\omega$. The $\omega$-complicial sets will then just be called \textit{complicial sets} and we will denote by $\stratSset$ the model category $\stratSset^{\omega}$.

\p A \notion{marked simplicial set} is a stratified simplicial set that has the right lifting property against entire acyclic cofibrations. In particular, all complicial sets are marked. The category of marked simplicial sets is denoted by \wcnotation{$\mSset$}{(mpsh@$\mSset$}. There is an adjunction:
\begin{equation}
\label{adj:beetwen marked an stratified}
\begin{tikzcd}
	{(\uvar)_{\mk}:\stratSset} & {\mSset:\iota}
	\arrow[""{name=0, anchor=center, inner sep=0}, shift left=2, from=1-1, to=1-2]
	\arrow[""{name=1, anchor=center, inner sep=0}, "i", shift left=2, from=1-2, to=1-1]
	\arrow["\dashv"{anchor=center, rotate=-90}, draw=none, from=0, to=1]
\end{tikzcd}
\end{equation}
The left adjoint $(\uvar)_{\mk}$ sends a stratified simplicial set $(X,tX)$ to the marked simplicial set $(X,\overline{tX})$, where $\overline{tX}$ is the smaller stratification that includes $tX$ and makes $(X,\overline{tX})$ a marked simplicial set. Moreover, the proposition \ref{prop:X to Xmk is acycli cof} implies that the canonical morphism $X\to \iota (X)_{\mk}$ is an entire acyclic cofibration.

Given a functor $i:I\mapsto (F(i),tF(i))$ with value in marked simplicial sets, its colimit is given by $(\colim F(i),\overline{M})$ where $M$ is the smaller stratification that includes the image of $tF(i)\to \colim F(i)$ for any $i:I$.

\begin{prop}
\label{prop:model structure on marked simplicial set}
The category $\mSset$ admits a nice model structure that makes the adjunction \ref{adj:beetwen marked an stratified} a Quillen equivalence.
\end{prop}
\begin{proof}
This is a direct consequence of proposition \ref{prop:model structure on marked presheaves} and theorem \ref{theo:model structure on complicial set}.
\end{proof}

\p
\label{para:inteligentr trucation for simplicial set}
Let $n$ be an integer, and $(X,tX)$ a marked simplicial set. We define $\tau^i_n(tX)$ as the reunion of $tX$ and all simplices of dimension strictly superior to $n$. This induces a functor, called the \snotionsym{intelligent $n$-truncation}{(taui@$\tau^i_n$}{for marked simplicial sets}:
$$\begin{array}{rcll}
\tau^i_n :& \mSset&\mapsto &\mSset\\
 &(X,tX)&\mapsto &(X, \overline{\tau^i_n(tX)}).
\end{array}$$
This functor preserves cofibrations.	
Given the explicit description of colimits in marked simplicial sets, it is easy to see that $\tau^i_n$ preserves colimits. 
For every elementary anodyne extension $i:K\to L$, we have a pushout 
\[\begin{tikzcd}
	K & L \\
	{\tau^i_n(K)} & {\tau^i_n(L).}
	\arrow[from=1-1, to=2-1]
	\arrow[from=2-1, to=2-2]
	\arrow[from=1-2, to=2-2]
	\arrow[from=1-1, to=1-2]
	\arrow["\lrcorner"{anchor=center, pos=0.125, rotate=180}, draw=none, from=2-2, to=1-1]
\end{tikzcd}\]
The intelligent $n$-truncation is then a left Quillen functor.

It's associated right adjoint is called the \wcsnotionsym{$n$-truncation}{(tau@$\tau_n$}{truncation@$n$-truncation}{for marked simplicial sets} and is denoted by 
$$\tau_n:\mSset\to \mSset.$$

\subsection{Gray tensor product}

\begin{construction}[{\cite[Notation 5]{Verity_weak_complicial_sets_I}}] For any $n,p,q\geq 0$ such that $n=p+q$, we define:
\begin{itemize}
\item the \notion{degeneration partition operator}:
$$
\begin{array}{rclllrrclll}
\invamalg^1_{p,q}:&[n]&\to&[p]&&~~~~~~&\invamalg^2_{p,q}:&[n]&\to&[q]&\\
&k&\mapsto &k &\mbox{if $k\leq p$} &&&k&\mapsto &0& \mbox{if $k\leq p$}\\
&k&\mapsto &p 	&\mbox{if $k>p$} &&&k&\mapsto &k-p& \mbox{if $k> p$}.
\end{array}
$$
\item the \notion{face partition operator}:
$$
\begin{array}{rcllrrcll}
\amalg^1_{p,q}:&[p]&\to&[n]&~~~~~~&\amalg^2_{p,q}:&[q]&\to&[n]\\
&k&\mapsto &k &&&k&\mapsto &k+p.
\end{array}
$$
\end{itemize}
\end{construction}

\begin{definition}[{\cite[Definition 128]{Verity_weak_complicial_sets_I}}]
Let $(X,tX)$ and $(Y,tY)$ be two stratified simplicial sets. 
We define the \snotionsym{Gray tensor product}{((d00@$\otimes$}{for stratified simplicial sets} of $(X,tX)$ and $(Y,tY)$ as the stratified simplicial set 
$$(X,tX)\otimes (Y,tY):=(X\times Y,tX\otimes tY)$$ where $tX\otimes tY$ is the set of pairs $(x,y)$ such that for any partitions $(p,q)$ of $n$ either $\amalg^1_{p,q}x$ or $\amalg^2_{p,q}y$ is thin. 
\end{definition}

\begin{remark}
Let $X,Y$ be two stratified simplicial sets such that all simplices of $X$ are thin. The morphism 
$X\otimes Y\to X\times Y$ is then an isomorphism.
\end{remark}

\p In \cite{Verity_weak_complicial_sets_I}, it is shown that the Gray tensor is associative. The problem of this operation comes from the fact that it doesn't commute with colimits. Verity then defines an other binary operation, which is cocontinuous, the \textit{Gray pretensor} (\cite[definition 135]{Verity_weak_complicial_sets_I}) $(X,tX)\boxtimes(Y,tY):=(X\times Y, tX\boxtimes tY)$, together with a natural transformation: 
$$\uvar\boxtimes\uvar\to \uvar\otimes\uvar$$
that is pointwise an entire acyclic cofibration (\cite[lemma 149]{Verity_complicial_set_characterising_the_simplicial_nerve}). Moreover, in \cite{Ozornova_Gray_tensor_product_and_saturated_n_complicia}, it is shown that this pretensor is a Quillen bifunctor for the model structure on $\stratSset$. 

\begin{definition}[Gray tensor product for marked simplicial sets]
Let $X$ and $Y$ be two marked simplicial sets. We define the \snotionsym{Gray tensor product}{((d00@$\otimes$}{for marked simplicial sets} of $X$ and $Y$ as the marked simplicial set 
$$X\otimes Y:= (\iota(X)\otimes \iota(Y))_{\mk}$$
 where $((\uvar)_{\mk},\iota)$ is the adjunction \ref{adj:beetwen marked an stratified}.
As $\uvar\boxtimes\uvar\to \uvar\otimes\uvar$ is pointwise a entire acyclic cofibration, we have an equality: 
$$X\otimes Y:= (\iota(X)\boxtimes \iota(Y))_{\mk}.$$
\end{definition}

\begin{prop}
\label{prop:R_commutes_with_gray_tensor}
We have equalities
$$(\uvar\boxtimes \uvar)_{\mk}=(\uvar\otimes \uvar)_{\mk}= (\uvar)_{\mk}\otimes (\uvar)_{\mk}.$$
\end{prop}
\begin{proof}
The first equality is a consequence of the fact that $\uvar\boxtimes\uvar\to \uvar\otimes\uvar$ is pointwise a entire acyclic cofibration.

 For the second one, we have to show that $(X\otimes Y)_{\mk}=(\iota(X_{\mk})\otimes \iota(Y_{\mk}))_{\mk}$.
The unit of the adjunction $(\iota,(\uvar)_{\mk})$ induces a morphism $h:(X\otimes Y)_{\mk}\to (\iota(X_{\mk})\otimes \iota(Y_{\mk}))_{\mk}$. This morphism is an entire acyclic cofibration according to proposition \ref{prop:X to Xmk is acycli cof}, 
and the corollary 2.2 of \cite{Ozornova_Gray_tensor_product_and_saturated_n_complicia} and the fact that $(\uvar)_{\mk}$ is a left Quillen functor.

 We then have lifts in the following diagram:
\[\begin{tikzcd}
	{(X\otimes Y)_{\mk}} & {(X\otimes Y)_{\mk}} \\
	{(\iota(X_{\mk})\otimes \iota(Y_{\mk}))_{\mk}}
	\arrow["id", from=1-1, to=1-2]
	\arrow["h"', from=1-1, to=2-1]
	\arrow["k"', from=2-1, to=1-2]
\end{tikzcd}\]
As both $k$ and $h$ are the identity on the underlying simplicial sets, this implies that the stratifications of $(X\otimes Y)_{\mk}$ and $(X\otimes Y)_{\mk}$ coincide, and this two objects are then equal. 
\end{proof}

We can then deduce the following proposition:
\begin{prop}
\label{prop:gray_product_is_a_left_Quillen_bifunctor}
The Gray tensor product is associative, and is a left Quillen bifunctor in $\mSset$.
\end{prop}
\begin{proof}
The first assertion is a consequence of proposition \ref{prop:R_commutes_with_gray_tensor} and the fact that the binary operation $\otimes$ on $\stratSset$ is associative. The second one is a consequence of proposition \ref{prop:R_commutes_with_gray_tensor} and \cite[Theorem 2.1]{Ozornova_Gray_tensor_product_and_saturated_n_complicia}.
\end{proof}

We now give a lemma investigating the interaction between the truncation, the intelligent truncation and the Gray tensor product.
\begin{lemma}
\label{lemma:technique marked oicategoros}
Let $C$ and $D$ be two stratified simplicial sets.
\begin{enumerate}
\item
The following canonical square is cocartesian
\[\begin{tikzcd}
	{\coprod_{n} \tau_nC\otimes \tau_nD} & {C\otimes D} \\
	{\coprod_{n} \tau^i_n(\tau_nC\otimes \tau_nD)} & {C\times D}
	\arrow[from=1-1, to=1-2]
	\arrow[from=1-1, to=2-1]
	\arrow[from=2-1, to=2-2]
	\arrow[from=1-2, to=2-2]
	\arrow["\lrcorner"{anchor=center, pos=0.125, rotate=180}, draw=none, from=2-2, to=1-1]
\end{tikzcd}\]

\item 
If $D$ is invariant under $\tau^i_2$, the following canonical square is cocartesian
\[\begin{tikzcd}
	{\coprod_{ n} \tau_{n}C\otimes D} & {C\otimes D} \\
	{\coprod_{n} \tau^i_{n+1}(\tau_{n}C\otimes D)} & {C\otimes \tau^i_1D}
	\arrow[from=1-1, to=2-1]
	\arrow[from=2-1, to=2-2]
	\arrow[from=1-2, to=2-2]
	\arrow["\lrcorner"{anchor=center, pos=0.125, rotate=180}, draw=none, from=2-2, to=1-1]
	\arrow[from=1-1, to=1-2]
\end{tikzcd}\]
\end{enumerate}
\end{lemma}
\begin{proof}
Let $C^\natural$ and $D^\natural$ be the underlying simplicial sets of $C$ and $D$.
Remark first that the two vertical morphisms of the first square are the identity. 
The induced morphism 
\begin{equation}
\label{eq:felkjzfezoifjezoi}
\coprod_{n} \tau^i_n(\tau_nC\otimes \tau_nD)\coprod_{\coprod_{n} \tau_nC\otimes \tau_nD}C\otimes D\to C\times D
\end{equation}
is then the identity of $C^\natural\times D^\natural$ at the level of underlying simplicial sets. To conclude, one has to show that every simplex $C^\natural\times D^\natural$ that is marked in the right term of \eqref{eq:felkjzfezoifjezoi} is also marked in the left term.
For this, let $n$ be a non negative integer, $x\in C^\natural_k$ and $y\in D^\natural_k$, such that $x$ is marked in $C$ and $y$ is marked in $D$.
The $k$-simplex $(x,y)$ then is in the image of $\tau^i_{k-1}(\tau_{k-1}C\otimes \tau_{k-1}D)$ and is then marked in the left term of \eqref{eq:felkjzfezoifjezoi}. This concludes the proof of the first assertion. 

The two vertical morphisms of the second square also are the identity and the induced morphism 
\begin{equation}
\label{eq:felkjzfezoifjezoibfdbd}
\coprod_{n} \tau^i_{n+ 1}(\tau_{n}C\otimes D)\coprod_{\coprod_{ n} \tau_{n}C\otimes D}C\otimes D\to C\otimes \tau^i_1	D
\end{equation}
is then once again the identity of $C^\natural\times D^\natural$ at the level of underlying simplicial sets.
Unfolding the definition, the marking of the left term is the smaller one that includes the one of $C\otimes D$ and every 
$k$-simplex $(x,y)$ such that both $x$ and $d^kx$ are marked in $C$.

Let $(x,y)$ be a $k$-simplex of $C^\natural\times D^\natural$. Suppose first that it is marked in $C\otimes D$. Remark that $(x,y)$ is then marked in $\tau_{k}C\otimes D$, and so is in the left term of \eqref{eq:felkjzfezoifjezoibfdbd}. Suppose now that both $x$ and $d^kx$ are marked in $C$. This implies that $s^{k-1}d^kx$ is in the image of $\tau_{k-1} C$. The simplex $(s^{k-1}d^kx,y)$ is then in the image of $\tau^i_{k}(\tau_{k-1}C\otimes D)$ and is then marked in the left term of \eqref{eq:felkjzfezoifjezoibfdbd}.

Now remark that we have 
$$ d^{k-1}(s^{k-1}x,s^ky)=(x,s^{k-1}d^{k-1}y) ~~~~~~~~~~ d^{k}(s^{k-1}x,s^ky)= (x,y)$$
$$d^{k+1}(s^{k-1}x,s^ky)= (s^{k-1}d^kx,y)$$
and both the $(k-1)$ and $(k+1)$ faces of $(s^{k-1}x,s^ky)$ are marked.
We leave it to the reader to check that by definition every sub $l$-simplex $z$ of $(s^{k-1}x,s^ky)$ containing the points $k-1$, $k$ and $k+1$ is marked in $C\otimes D$, and so in $\tau_{k}C\otimes D$, and, therefore, in the left term of \eqref{eq:felkjzfezoifjezoibfdbd}. As the marking is stable by complicial thinness extension, this implies that $(x,y)$ is also marked in the left term of \eqref{eq:felkjzfezoifjezoibfdbd}.

The marking of the right term of \eqref{eq:felkjzfezoifjezoibfdbd} is then included in the marking of the left term. They then coincide, which concludes the proof.
\end{proof}

\begin{remark}

The reason for including the assumption that $D$ is invariant under $\tau^i_2$ is solely because it will be the only relevant case. If we remove this assumption, the statement remains true, but the proof becomes a little bit more technical.
\end{remark}

\p
Let $X$ be a marked simplicial set. We define the \snotion{suspension}{for marked simplicial sets} of $X$, noted by \wcnotation{$\Sigma X$}{(sigma@$\Sigma\uvar$}, as the following pushout:
\[\begin{tikzcd}
	{X\otimes\partial [1]} & {X\otimes [1]} \\
	{\partial[1]} & {\Sigma X}
	\arrow[from=1-2, to=2-2]
	\arrow["\lrcorner"{anchor=center, pos=0.125, rotate=180}, draw=none, from=2-2, to=1-1]
	\arrow[from=1-1, to=2-1]
	\arrow[from=2-1, to=2-2]
	\arrow[from=1-1, to=1-2]
\end{tikzcd}\]
This assignation defines a cocontinuous functor $\Sigma:\mSset\to \mSset_{\partial[1]/}.$ For every acyclic cofibration $K\to L$, we have cartesian squares
\[\begin{tikzcd}
	{L\otimes\partial[1]} & {K\otimes[1]\cup L\otimes\partial[1]} & {L\otimes[1]} \\
	{\partial[1]} & {\Sigma K} & {\Sigma L}
	\arrow[from=1-1, to=2-1]
	\arrow[""{name=0, anchor=center, inner sep=0}, from=1-1, to=1-2]
	\arrow[from=2-1, to=2-2]
	\arrow[from=1-2, to=2-2]
	\arrow[from=1-3, to=2-3]
	\arrow[from=2-2, to=2-3]
	\arrow[""{name=1, anchor=center, inner sep=0}, from=1-2, to=1-3]
	\arrow["\lrcorner"{anchor=center, pos=0.125, rotate=180}, draw=none, from=2-2, to=0]
	\arrow["\lrcorner"{anchor=center, pos=0.125, rotate=180}, draw=none, from=2-3, to=1]
\end{tikzcd}\]
The suspension then preserves acyclic cofibration and is then a left Quillen functor.

This functor admits a right adjoint, that sends a pair $(a,b,C)$ to \wcnotation{$C(a,b)$}{(cab@$C(a,b)$} where $a,b$ are two $0$-simplices of $C$. If $p:C\to D$ is a morphism between complicial sets, and $a,b$ two $0$-simplices of $C$, we denote by 
$$p(a,b):C(a,b)\to D(pa,pb)$$
the induced morphism.

\p
We introduce an other operation, the \notion{diamond product}, that makes the link between the Gray tensor product and the join. 
Let $X$ and $Y$ be two marked simplicial sets. We define \sym{((d21@$\diamond$}$X\diamond Y$ as the colimit of the diagram:
\[\begin{tikzcd}
	X & {X\otimes \{0\}\otimes Y} & {X\otimes[1]\otimes Y} & {X\otimes \{1\}\otimes Y} & Y
	\arrow[from=1-4, to=1-3]
	\arrow[from=1-4, to=1-5]
	\arrow[from=1-2, to=1-1]
	\arrow[from=1-2, to=1-3]
\end{tikzcd}\]
The functors 
$$\uvar\diamond X:\mSset\to \mSset_{/X} ~~~~\mbox{and}~~~~ X\diamond \uvar:\mSset\to \mSset_{/X}$$
are colimit preserving. Furthermore, for every acyclic cofibration $K\to L$, the morphism $K\diamond X\to L\diamond X$ is the horizontal colimit of the diagram: 
\[\begin{tikzcd}
	{K\amalg X} & {K\otimes \partial[1]\otimes X} & {K\otimes [1]\otimes X} \\
	{L\amalg X} & {L\otimes \partial[1]\otimes X} & {L\otimes [1] \otimes X}
	\arrow[from=1-2, to=1-1]
	\arrow[from=1-2, to=1-3]
	\arrow[from=2-2, to=2-3]
	\arrow[from=2-2, to=2-1]
	\arrow[from=1-2, to=2-2]
	\arrow[from=1-1, to=2-1]
	\arrow[from=1-3, to=2-3]
\end{tikzcd}\]
However, these two horizontal colimits are homotopy colimits, and all the horizontal maps of the previous diagram are weak equivalences. This morphism is then an acyclic cofibration. This shows that 
 $\uvar\diamond X$ is a left Quillen functor. We show analogously that $X\diamond \uvar$ is a left Quillen functor.

\begin{lemma}
There exists a unique natural transformation $\gamma_{X,Y}:X\diamond Y\to X\star Y$ that fits in the following diagram: 
\[\begin{tikzcd}
	{X\coprod Y} & {X\star Y} \\
	{X\diamond Y} & {[1]}
	\arrow[from=1-1, to=2-1]
	\arrow[from=1-1, to=1-2]
	\arrow[from=1-2, to=2-2]
	\arrow[from=2-1, to=2-2]
	\arrow["{\gamma_{X,Y}}", from=2-1, to=1-2]
\end{tikzcd}\]
\end{lemma}
\begin{proof}
We begin by defining this morphism on simplicial sets, and for this we can suppose that both $X$ and $Y$ are representables, ie $X:=[n]$, $Y:=[m]$.
On object, this morphism is induced by the assignation:
$$p(k,0,l) := k~~~p(k,1,l) := l.$$ 

We need to verify that this morphism preserves thin cells. Suppose now that $(x,v,y)$ is a thin $n$-simplex of $X\diamond Y$. There are several cases to consider. \textbf{Case $v_n=0$.} The simplex $x$ is then thin, and is sent to $x\star \emptyset$ which is also thin. \textbf{Case $v_0=1$.} Similar. \textbf{Case $v_0=0$ and $v_n=1$.} Let $p$ be the smaller integer such that $v_p=1$. Either $\amalg_{p-1,n-p+1}^1(x)$ or $\amalg_{p,n-p}^2(y)$ is thin. This implies that $\phi_{X,Y}(x,v,y)= \amalg_{p-1,n-p+1}^1(x)\star \amalg_{p,n-p}^2(y)$ is thin. 
\end{proof}

\begin{prop}
\label{prop:equivalence between diamond and join product}
For any $X,Y$, the morphism $\gamma_{X,Y}$ is a weak equivalence. 
\end{prop}
\begin{proof}
The set of couples $(X,Y)$ such that $\gamma_{X,Y}$ is a weak equivalence is saturated by monomorphisms. It is then enough to show the result for any couples of representables. 

Let's start by the case $(X,Y)=([n],[m])$. Let $s:X\star Y\to X\diamond Y$ be the morphism defined on objects by the formula: 
$$s(k\star \emptyset) := (k,0,0)~~~s(\emptyset \star l) := (n,1,l)$$
We have
$$\gamma_{X,Y}s = id ~~~~s\gamma_{X,Y} (k,\epsilon,l) =(k + \epsilon (n-k), \epsilon,\epsilon l).$$

Let $\eta:[n]\diamond [m]\to [n]\diamond [m]$ be induced by the application
$$(k,\epsilon,l)\mapsto (k,\epsilon,\epsilon l).$$
We are now going to construct two morphisms
$$\epsilon_0: ([n]\diamond[m])\times [1]_t\to [n]\diamond[m]~~~~\mbox{ and }~~~~\epsilon_1: ([n]\diamond[m])\times [1]_t\to [n]\diamond[m]$$
such that $$
\begin{array}{rrl}
\epsilon_0(\uvar,0)=\eta&&\epsilon_0(\uvar,1)=s\gamma_{X,Y}\\
\epsilon_1(\uvar,0)=\eta&~~~~&\epsilon_1(\uvar,1)=id\\
\end{array}$$
The first one is induced on the level of simplicial sets by
$$(k,\epsilon,l,\alpha)\mapsto (k + \alpha\epsilon (n-k),\epsilon,\epsilon l ),$$
and the second one by
$$(k,\epsilon,l,\alpha)\mapsto (k,\epsilon,(\epsilon\vee\alpha)l),$$
where $\epsilon\vee\alpha := \epsilon+\alpha - \epsilon\alpha.$
These two morphisms extend to marked simplicial sets. 

We proceed in a similar way with cases $(X,Y) = ([n]_t,[m]), ([n],[m]_t)$ or $([n]_t,[m]_t)$. 
\end{proof}

As we already now that functors $\uvar\diamond X$ and $X\diamond \uvar$ preserve weak equivalences, the previous proposition implies that for any marked simplicial sets $X$, functors $\uvar\star X$ and $X\star \uvar$ preserves weak equivalences and are then
 left Quillen functors.

\p 
\label{para:sigma star}
Let $X$ be a marked simplicial set. We now describe an variation on the suspension. We define \wcnotation{$\Sigma^\star X$}{(sigmastar@$\Sigma^\star\uvar$}, as the following pushout:
\[\begin{tikzcd}
	X & {X\star [0]} \\
	1 & {\Sigma^\star X}
	\arrow[from=1-2, to=2-2]
	\arrow["\lrcorner"{anchor=center, pos=0.125, rotate=180}, draw=none, from=2-2, to=1-1]
	\arrow[from=1-1, to=2-1]
	\arrow[from=1-1, to=1-2]
	\arrow[from=2-1, to=2-2]
\end{tikzcd}\]
This assignation defines a cocontinuous functor $\Sigma^\star:\mSset\to \mSset_{\partial[1]/}.$ Using proposition \ref{prop:equivalence between diamond and join product}, all the vertical morphisms of the following diagram are weak equivalences:
\[\begin{tikzcd}
	1 & X & {X\diamond 1} \\
	1 & X & X\star1
	\arrow[from=1-2, to=1-3]
	\arrow[from=1-2, to=1-1]
	\arrow[from=2-2, to=2-1]
	\arrow[from=1-3, to=2-3]
	\arrow[from=2-2, to=2-3]
	\arrow[from=1-2, to=2-2]
	\arrow[from=1-1, to=2-1]
\end{tikzcd}\]
Remark furthermore that the colimits of these lines are also homotopy colimits. Taking the horizontal colimit, this induces a weak equivalence
\begin{equation}
\label{eq:sigma et sigam star}
\Sigma X\to \Sigma^{\star}X
\end{equation}
natural in $X$.

\p
We define the \notion{co-join} of $X$ and $Y$, denoted by \index[notation]{((d22@$\overset{co}{\star}$}$X\costar Y$, as the colimit of the following diagram:
\[\begin{tikzcd}
	Y & {Y\otimes \{1\}\otimes X} & {Y\otimes [1]\otimes X} & {Y\otimes \{0\}\otimes X} & X
	\arrow[from=1-2, to=1-1]
	\arrow[from=1-2, to=1-3]
	\arrow[from=1-4, to=1-5]
	\arrow[from=1-4, to=1-3]
\end{tikzcd}\]
The functors 
$$\uvar\costar X:\mSset\to \mSset_{/X} ~~\mbox{and}~~ X\costar \uvar:\mSset\to \mSset_{/X}$$
are colimit preserving. Furthermore, for every acyclic cofibration $K\to L$, the morphism $K\costar X\to L\costar X$ is the horizontal colimit of the diagram:
\[\begin{tikzcd}
	{K\amalg X} & {X\otimes \partial[1]\otimes K} & {X\otimes [1]\otimes K} \\
	{L\amalg X} & {X\otimes \partial[1]\otimes L} & {X\otimes [1] \otimes K}
	\arrow[from=1-2, to=1-1]
	\arrow[from=1-2, to=1-3]
	\arrow[from=2-2, to=2-3]
	\arrow[from=2-2, to=2-1]
	\arrow[from=1-2, to=2-2]
	\arrow[from=1-3, to=2-3]
	\arrow[from=1-1, to=2-1]
\end{tikzcd}\]
However, these two horizontal colimits are homotopy colimits, and all the horizontal maps of the previous diagram are weak equivalences. This morphism is then an acyclic cofibration.
This shows that $\uvar\costar X$ is a left Quillen functor. We show analogously that $X\costar \uvar$ is a left Quillen functor.

\p
\label{subsection:wedge} Let $X$ be a simplicial set. We define the \textit{wedge} of $\Sigma X$ and $[1]$, noted by \sym{(sigmavee@$\Sigma X~\rotatebox[origin=c]{270}{$\gtrdot$}~[1]$}\sym{(sigmave@$[1]~\rotatebox[origin=c]{270}{$\gtrdot$}\Sigma X$}$\Sigma X\fwedge [1]$, as the colimit of the following diagram:
\[\begin{tikzcd}
	{X\otimes[0,1]} & {X\otimes[2]_t} & {X\otimes[1,2]} \\
	{\Sigma X} & {X\fwedge[1]} & {[1,2]}
	\arrow[from=1-1, to=2-1]
	\arrow[from=1-3, to=2-3]
	\arrow[from=1-1, to=1-2]
	\arrow[from=1-3, to=1-2]
	\arrow[from=1-2, to=2-2]
	\arrow[from=2-3, to=2-2]
	\arrow[from=2-1, to=2-2]
\end{tikzcd}\]
This assignation defines a cocontinuous functor $\uvar\fwedge [1]:\mSset\to \mSset_{[0]\amalg [1]/}.$ For every acyclic cofibration $K\to L$, the morphism $K\fwedge [1]\to L\fwedge [1]$ is the horizontal colimit of the diagram:
\[\begin{tikzcd}
	{[0]\coprod[1]} & {K\otimes([0]\coprod[1,2])} & {K\otimes[2]_t} \\
	{K\otimes[2]_t} & {L\otimes[2]_t} & {L\otimes[2]_t}
	\arrow[from=1-2, to=1-1]
	\arrow[from=1-2, to=1-3]
	\arrow[from=1-1, to=2-1]
	\arrow[from=2-2, to=2-1]
	\arrow[from=2-2, to=2-3]
	\arrow[from=1-3, to=2-3]
	\arrow[from=1-2, to=2-2]
\end{tikzcd}\]
However, these two horizontal colimits are homotopy colimits, and all the horizontal maps of the previous diagram are weak equivalences. This morphism is then an acyclic cofibration.
This shows that this functor is a left Quillen functor. We denote by $$\triangledown:\Sigma X\to \Sigma X\fwedge [1]$$ the morphism induced by the inclusion $X\otimes [0,2]\subset X\otimes [2]_t$ and 
$$\Sigma X\hookrightarrow \Sigma X\fwedge [1]$$
the morphism induced by the inclusion $X\otimes [1,2]\subset X\otimes [2]_t$.
We define similarly the left Quillen functor $$[1]\fwedge\uvar:\mSset\to \mSset_{[1]\amalg [0]/}$$ and the morphisms
$$\triangledown:\Sigma X\to [1]\fwedge\Sigma X~~~\mbox{and}~~~\Sigma X\hookrightarrow [1]\fwedge\Sigma X .$$

\begin{prop}
Morphisms 
$$ \Sigma X\coprod_{[0]}[1]\to \Sigma X\fwedge [1]~~~~\mbox{and}~~~~ [1]\coprod_{[0]}\Sigma X\to [1]\fwedge \Sigma X$$
are acyclic cofibrations. 
\end{prop}
\begin{proof}
We have cartesian squares:
\[\begin{tikzcd}
	{X\otimes([0]\coprod[1,2])} & {X\otimes \Lambda^{1}[2]} & {X\otimes[2]_t} \\
	{[0]\coprod[1]} & {\Sigma X\coprod_{[0]} [1]} & {\Sigma X\fwedge [1].}
	\arrow[from=1-1, to=2-1]
	\arrow[from=2-1, to=2-2]
	\arrow[from=1-1, to=1-2]
	\arrow[from=1-2, to=1-3]
	\arrow[from=1-3, to=2-3]
	\arrow[from=1-2, to=2-2]
	\arrow["\lrcorner"{anchor=center, pos=0.125, rotate=180}, draw=none, from=2-3, to=1-2]
	\arrow["\lrcorner"{anchor=center, pos=0.125, rotate=180}, draw=none, from=2-2, to=1-1]
	\arrow[from=2-2, to=2-3]
\end{tikzcd}\]
The upper right horizontal morphism is an acyclic cofibration, and so is the downer right horizontal one. We proceed similarly for the other morphism.
\end{proof}

\subsection{Gray cylinder, Gray cone and Gray $\circ$-cone}
\p The Gray tensor product induced a left Quillen functor 
$$\uvar\otimes[1]:\mSset\to \mSset$$
called the \snotionsym{Gray cylinder}{((d30@$\uvar\otimes[1]$}{for marked simplicial sets}. 
The join and the co-join also incuce two left Quillen functors
$$\uvar\star [0]:\mSset\to \mSset_{[0]/}~~~~~[0]\costar \uvar:\mSset\to \mSset_{[0]/}$$
called the \snotionsym{Gray cone}{((d40@$\uvar\star 1$}{for marked simplicial sets} and the \snotion{Gray $\circ$-cone}{for marked simplicial sets}\index[notation]{((d50@$1\overset{co}{\star}\_$!\textit{for marked simplicial sets}}. We denote by 
$$
\begin{array}{rclcrcl}
 \mSset_{\cdot} &\to &\mSset & & \mSset_{\cdot}&\to & \mSset\\
(X,x)&\mapsto & X_{/x} &~~~~~ & (X,x)&\mapsto & X_{x/}\\
\end{array}
$$
respectively called the \wcsnotionsym{slice of $X$ over $x$}{(cc@$C_{c/}$}{slice over}{for marked simplicial sets} and the \wcsnotionsym{slice of $X$ under $x$}{(cc@$C_{/c}$}{slice under}{for marked simplicial sets}, the right adjoints of the Gray cone and the Gray $\circ$-cone.

Remark furthermore that we have canonical natural transformation $X_{x/}\to X$ and $X_{/x}\to X$, induced by the natural transformation $X\to X\star [0]$ and $X\to [0]\costar X$.
\p 
The category of endomorphisms of marked simplicial sets has a monoidal structure given by the composition. The endomorphism $[0]\costar \uvar$ admits a monoid structure, where the multiplication is the natural transformation:
$[0]\costar ([0]\costar X)\to [0]\costar X$, induced by the pairing: 
$$
\begin{array}{rcl}
X\otimes[1]\otimes[1]&\to& X\otimes[1]\\
(x,i,j)&\mapsto& (x,i\wedge j).
\end{array}$$

This defines a cosimplicial object in $\End(\mSset)$, which evaluated on $\emptyset$, provides a cosimplicial object in $\mSset$: 
$$\begin{array}{rcl}
\Delta &\to & \mSset\\
n&\mapsto&[n]_{\circ}:=[0]\costar (... ([0]\costar[0])).
\end{array}$$
Eventually, we set $([n]_t)_{\circ} := \tau^i_{n-1}([n]_{\circ})$. We then have defined a functor:
$$(\uvar)_{\circ}:t\Delta \to \mSset.$$ 
\sym{((b90@$(\uvar)_{\circ}$}

\subsection{Street nerve}
\label{section:Street nerve}

We recall that $\zo$-categories are defined in section \ref{section:zocategories}. The Gray operations on $\zo$-categories - 
$\uvar\otimes[1]$, $\uvar\star 1$, $1\costar \uvar$ -
are defined in section \ref{section:definition of Gray operations}.

In \cite{Street_algebra_of_orianted_simplexes}, Street defines a cosimplicial object in $\zocat$, that associates to $n$, the $n^{th}$ \notion{oriental} $O_n$. 
The original construction of this object is complicated, but Ara and Maltsiniotis have shown that it can be easily defined using Gray operations. Indeed, in \cite[Corollaire 7.10]{Ara_Maltsiniotis_joint_et_tranche}, these authors construct an isomorphism
$$O_n\cong \overbrace{1\star...\star 1}^{n+1}$$
natural in $n$.

We can extend the functor $O_{\uvar}:\Delta\to \zocat$ to $t\Delta$ by defining
$$(O_n)_t:=\tau^i_{n-1}(O_n).$$
By extention by colimit, this induces a functor 
$$\R:\stratSset\to \zocat.$$
As explained in example 11 of \cite{Verity_weak_complicial_set_part2_nerve_of_complicial_Gray_categories}, $\R$ preserves the Gray tensor product, and so also the suspension, the wedge, the Gray cone and the Gray $\circ$-cone. 
 Moreover, \cite[Theorem 249]{Verity_complicial_set} states that this functor sends complicial horn inclusions and complicial thinness extensions to isomorphisms. It obviously also sends saturation extensions to isomorphisms. This functor then sends every weak equivalences to isomorphisms, and then lifts to a colimit preserving functor $\R:\mSset\to \zocat$ and induces an adjoint pair: \sym{(r@$\R:\mSset\to \zocat$}\sym{(n@$\N:\zocat\to \mSset$}
\[\begin{tikzcd}
	{\R:\mSset} & {\zocat:\N}
	\arrow[""{name=0, anchor=center, inner sep=0}, shift left=2, from=1-1, to=1-2]
	\arrow[""{name=1, anchor=center, inner sep=0}, shift left=2, from=1-2, to=1-1]
	\arrow["\dashv"{anchor=center, rotate=-90}, draw=none, from=0, to=1]
\end{tikzcd}\]

We now recall two fundamental results of strictification:
\begin{theorem}[Gagna, Ozornova, Rovelli]
\label{theo:strict representable}
Let $n$ be an integer. The canonical morphism
$$[n]\to \N(\R([n]))$$
is an acyclic cofibration.
\end{theorem}
\begin{proof}
This is \cite[corollary 5.4]{Gagna_Nerves_and_cones_of_free_loop_free_omega_categories}.
\end{proof}
\begin{theorem}[Ozornova, Rovelli]
\label{theo:strict susension}
Let $C$ be an $\zo$-category.
The canonical morphism
$$\Sigma \N C \to \N([C,1])$$
is an acyclic cofibration.
\end{theorem}
\begin{proof}
The morphism \eqref{eq:sigma et sigam star} provides a weak equivalence
$\Sigma \N C\to \Sigma^{\star} \N C$.
As this morphism is sent to an isomorphism by $R$, it induces a commutative triangle
\[\begin{tikzcd}
	& {\Sigma^{\star} \N C} \\
	{\Sigma \N C} && {\N([C,1])}
	\arrow[from=2-1, to=2-3]
	\arrow["\sim", from=2-1, to=1-2]
	\arrow[from=1-2, to=2-3]
\end{tikzcd}\]
The theorem 3.22 of \cite{Ozornova_a_quillen_adjunction_between_globular_and_complicial} stipulates that $\Sigma^{\star} \N C\to \N([C,1])$ is a weak equivalence, which concludes the proof.
\end{proof}

\begin{definition}
We define the \notion{Street endofunctor} \wcnotation{$i_{str}$}{(istr@$i_{str}$} to be the colimit preserving functor defined on representables by: 
$$i_{str}([n]) := \N(\R([n]))~~~\mbox{ and }~~~i_{str}([n]_t) :=\tau^i_{n-1} (i_{str}([n]))$$
\end{definition}

\begin{prop}
\label{prop:i_str_is_Quillen}
 The functor $i_{srt}$ is left Quillen and 
the natural transformation 
$$id \to i_{srt}$$ 
is weakly invertible.
\end{prop}
\begin{proof}
As noticed earlier, for any integer $n$, the map $[n]\to i_{srt}([n])$ is a weak equivalence.
We recall that the intelligent truncation functor $\tau^i_{n-1}:\mSset\to \mSset$ is a left Quillen functor, and so preserves weak equivalences between cofibrant objects. The morphism $[n]_t\to i_{str}([n]_t)$ is then a weak equivalence.
The set of objects $X$ such that the morphism $X\to i_{srt}X$ is a weak equivalence is closed by homotopy colimits and includes all representables. As $i_{srt}$ preserves monomorphisms, it then consists of all marked simplicial sets. Now let $K\to L$ be an acyclic cofibration. We have a commutative square:
\[\begin{tikzcd}
	K & {i_{str}(K)} \\
	L & {i_{str}(L)}
	\arrow["\sim"', from=1-1, to=2-1]
	\arrow["\sim", from=1-1, to=1-2]
	\arrow["\sim"', from=2-1, to=2-2]
	\arrow[from=1-2, to=2-2]
\end{tikzcd}\]
By two out of three, $i_{str}(K)\to i_{str}(L)$ is then an acyclic cofibration. The functor $i_{srt}$ is then left Quillen. 
\end{proof}

\section{Suspension and Gray operations}
\label{section:Suspension and Gray operation}
\subsection{Formula for the Gray cylinder}
The aim of this subsection is to demonstrate the following theorem:
\begin{theorem}
\label{theo:interval_first_formula}
There is a zigzag of acyclic cofibrations, natural in $X$, between the colimit of the diagram
$$[1]\fwedge\Sigma X\xleftarrow{\triangledown} \Sigma (X\otimes\{0\})\hookrightarrow \Sigma(X\otimes[1])\hookleftarrow \Sigma (X\otimes\{1\})\xrightarrow{\triangledown} \Sigma X\fwedge[1]$$
and $(\Sigma X)\otimes [1]$.
\end{theorem}

\begin{construction}

Let $C$ be the following colimit:
\[\begin{tikzcd}
	{[3]\times\{0\}\coprod [3]\times\{1\}} & {[3]\times[1]} \\
	{[1]\coprod[1]} & {C.}
	\arrow["{s^0s^0\coprod s^2s^3}"', from=1-1, to=2-1]
	\arrow[""{name=0, anchor=center, inner sep=0}, from=1-1, to=1-2]
	\arrow[from=1-2, to=2-2]
	\arrow[from=2-1, to=2-2]
	\arrow["\lrcorner"{anchor=center, pos=0.125, rotate=180}, draw=none, from=2-2, to=0]
\end{tikzcd}\]

We define several marked simplicial sets whose underlying simplicial sets are sub objects of C: 
\[\begin{tikzcd}
	{} & 00 & 01 && 00 & 01 \\
	& 10 & 11 && 20 & 21 \\
	& 20 & 21 && 00 & 01 \\
	& 30 & 31 & {} & 30 & 31
	\arrow[from=1-3, to=2-3]
	\arrow["{\large{A_1:=~~~~~~}}"', Rightarrow, no head, from=2-2, to=3-2]
	\arrow[Rightarrow, no head, from=2-3, to=3-3]
	\arrow["{\large{A_0:=~~~~~~}}"', Rightarrow, no head, from=1-2, to=2-2]
	\arrow[from=1-2, to=1-3]
	\arrow[from=3-2, to=3-3]
	\arrow[from=2-2, to=2-3]
	\arrow[""{name=0, anchor=center, inner sep=0}, from=1-2, to=2-3]
	\arrow[""{name=1, anchor=center, inner sep=0}, from=2-2, to=3-3]
	\arrow["{\large{A_3:=~~~~~~}}"', Rightarrow, no head, from=1-5, to=2-5]
	\arrow[from=1-6, to=2-6]
	\arrow[from=2-5, to=2-6]
	\arrow[""{name=2, anchor=center, inner sep=0}, from=1-5, to=2-6]
	\arrow[from=4-2, to=4-3]
	\arrow["{\large{A_4:=~~~~~~}}"', from=3-5, to=4-5]
	\arrow[from=4-5, to=4-6]
	\arrow[from=3-6, to=4-6]
	\arrow[from=3-5, to=3-6]
	\arrow[""{name=3, anchor=center, inner sep=0}, from=3-5, to=4-6]
	\arrow[from=1-5, to=1-6]
	\arrow["{\large{A_2:=~~~~~~}}"', from=3-2, to=4-2]
	\arrow[""{name=4, anchor=center, inner sep=0}, from=3-2, to=4-3]
	\arrow[Rightarrow, no head, from=3-3, to=4-3]
	\arrow["\sim"{description}, Rightarrow, draw=none, from=0, to=1-3]
	\arrow["\sim"{description}, Rightarrow, draw=none, from=1, to=2-3]
	\arrow["\sim"{description}, Rightarrow, draw=none, from=0, to=2-2]
	\arrow["\sim"{description}, Rightarrow, draw=none, from=2, to=1-6]
	\arrow[shorten <=2pt, Rightarrow, from=2, to=2-5]
	\arrow["\sim"{description}, Rightarrow, draw=none, from=3, to=3-6]
	\arrow[shorten <=2pt, Rightarrow, from=3, to=4-5]
	\arrow[shorten <=2pt, Rightarrow, from=1, to=3-2]
	\arrow["\sim"{description}, Rightarrow, draw=none, from=4, to=4-2]
	\arrow["\sim"{description}, Rightarrow, draw=none, from=4, to=3-3]
\end{tikzcd}\]
where arrows labeled by $=$ are degenerate and simplicies labeled by $\sim$ are thin.

Let $B_0$ be the sub object corresponding to the image of $[0,1,2]\times[0,1]$ where the marking includes all cells of dimension $\leq 2$, except $[10,20,21]$ and $[00,20,21]$.

Let $B_1$ be the sub object corresponding to the image of $[0,2,3]\times[0,1]$ where the marking includes all cells of dimension $\leq 2$, except $[00,20,21]$, $[00,30,31]$ and $[00,20,31]$.

Let $B$ be the reunion of $[0,1,2]\times[0,1]$ and $[0,2,3]\times[0,1]$ where the marking is the reunion of $B_0$ and $B_1$.
\end{construction}

\begin{lemma}
Morphisms $A_0\cup A_1\to B_0$ and $A_3\to B_0$ are acyclic cofibrations. 
\end{lemma}
\begin{proof}
The cofibration $A_0\cup A_1\to B_0$ fits in the following pushout square:
\[\begin{tikzcd}
	{\Lambda^{1}[2]\otimes [1]\cup[2]_t\otimes \partial[1]} & {A_1\cup A_2} \\
	{[2]_t\otimes [1]} & {B_0}
	\arrow[""{name=0, anchor=center, inner sep=0}, from=1-1, to=1-2]
	\arrow[from=1-1, to=2-1]
	\arrow[from=1-2, to=2-2]
	\arrow["{[0,1,2]\times[0,1]}"', from=2-1, to=2-2]
	\arrow["\lrcorner"{anchor=center, pos=0.125, rotate=180}, draw=none, from=2-2, to=0]
\end{tikzcd}\]

The cofibration $A_3\to B_0$ is a sequence of inclusions:
$$A_3=:(D_0,M_0)\subset (D_1,M_1)\subset (D_2,M_2)\subset(D_3,M_3)\subset(D_4,M_4)\subset (D_5,M_5)\subset (D_6,M_6):= B_0,$$ where 

\begin{itemize}[leftmargin=* ,parsep=0cm,itemsep=0cm,topsep=0cm]
\item $D_1 = D_0\cup [00,{01},11]$ ;
\item $D_2 = D_1\cup [ {00},10,11]$ ;
\item $D_2 = D_1\cup [ {00},10,21]$ ;
\item $D_4 = D_3\cup [00, {01},11,21]$; 
\item $D_5 = D_4\cup [ {00},10,11,21]$; 
\item $D_6 = D_5\cup [ {00},10,20,21]$; 
\end{itemize} and
\begin{itemize}[leftmargin=* ,parsep=0cm,itemsep=0cm,topsep=0cm]
\item $(D_0,M_0)\to (D_1,M_1)$ is a pushout of $\Lambda^1[2]\to [2]^1$;
\item $(D_1,M_1)\to (D_2,M_2)$ is a pushout of $\Lambda^0[2]\to [2]^0$;
\item $(D_2,M_2)\to (D_3,M_3)$ is a pushout of $\Lambda^0[2]\to [2]^0$;
\item $(D_3,M_3)\to (D_4,M_4)$ is a pushout of $\Lambda^1[3]\to [3]^1$;
\item $(D_4,M_4)\to (D_5,M_5)$ is a pushout of $\Lambda^0[3]\to [3]^0$;
\item $(D_5,M_5)\to (D_6,M_6)$ is a pushout of $\Lambda^0[3]\to [3]^0$.
\end{itemize}
\end{proof}

\begin{lemma}
Morphisms $A_2\cup A_3\to B_1$ and $A_4\to B_1$ are acyclic cofibrations. 
\end{lemma}
\begin{proof}
The cofibration $A_2\cup A_3\to B_1$ fits in the pushout square:
\[\begin{tikzcd}
	{\Lambda^{1}[2]\otimes [1]\cup [2]_t\otimes \partial[1]} & {A_2\cup A_3} \\
	{[2]_t\otimes [1]} & {B_1}
	\arrow[from=1-1, to=1-2]
	\arrow["{[0,2,3]\times[0,1]}"', from=2-1, to=2-2]
	\arrow[from=1-1, to=2-1]
	\arrow[from=1-2, to=2-2]
\end{tikzcd}\]
The cofibration $A_4\to B_1$ is a sequence of inclusions:
$$A_4=:(D_0,M_0)\subset (D_1,M_1)\subset (D_2,M_2)\subset(D_3,M_3)\subset(D_4,M_4)\subset (D_5,M_5)\subset (D_6,M_6):= B_1$$ where 
\begin{itemize}[leftmargin=* ,parsep=0cm,itemsep=0cm,topsep=0cm]
\item $D_1 = D_0\cup [00,21, {31}]$ ;
\item $D_2 = D_1\cup [20, {30},31]$ ;
\item $D_3 = D_2\cup [20,21, {31}]$;
\item $D_4 = D_3\cup [00,01,21, {31}]$;
\item $D_5 = D_4\cup [00,20, {30},31]$ ;
\item $D_6 = D_5\cup [00,20,21, {31}]$ ;
\end{itemize}
and
\begin{itemize}[leftmargin=* ,parsep=0cm,itemsep=0cm,topsep=0cm]
\item $(D_0,M_0)\to (D_1,M_1)$ is a pushout of $\Lambda^2[2]\to [2]^2$;
\item $(D_1,M_1)\to (D_2,M_2)$ is a pushout of $\Lambda^1[2]\to [2]^1$;
\item $(D_2,M_2)\to (D_3,M_3)$ is a pushout of $\Lambda^2[2]\to [2]^2$;
\item $(D_3,M_3)\to (D_4,M_4)$ is a pushout of $\Lambda^3[3]\to [3]^3$;
\item $(D_4,M_4)\to (D_5,M_5)$ is a pushout of $\Lambda^2[3]\to [3]^2$;
\item $(D_5,M_5)\to (D_6,M_6)$ is a pushout of $\Lambda^3[3]\to [3]^3$.
\end{itemize}
\end{proof}

\begin{lemma}
\label{lemma:formula for gray 0}
The maps $A_0\cup A_1\cup A_2\to B$ and $A_4\to B$ are acyclic cofibrations. 
\end{lemma}
\begin{proof}
This is a direct consequence of the last two lemmas.
\end{proof}

\begin{construction}
The marked simplicial set
$\overline{X\otimes B}$ is the pushout:
\[\begin{tikzcd}
	{X\otimes([00,01]\coprod [30,31])} & {X\otimes B} \\
	{[00,01]\coprod [30,31]} & {\overline{X\otimes B}.}
	\arrow[from=1-1, to=2-1]
	\arrow[from=2-1, to=2-2]
	\arrow[from=1-2, to=2-2]
	\arrow[""{name=0, anchor=center, inner sep=0}, from=1-1, to=1-2]
	\arrow["\lrcorner"{anchor=center, pos=0.125, rotate=180}, draw=none, from=2-2, to=0]
\end{tikzcd}\]

Let $\overline{X\otimes A_i}$ and $\overline{X\otimes B_i}$ be the sub-objects of $\overline{X\otimes B}$ corresponding to image of ${X\otimes A_i}$ and $	{X\otimes B_i}$.
\end{construction}

\begin{lemma}
\label{lemma:formula for gray 1}
The inclusion 
$	\overline{X\otimes A_{0}}\cup \overline{X\otimes A_{1}}\cup \overline{X\otimes A_{2}}\to	\overline{X\otimes B}$ and 
$\overline{X\otimes A_{4}}\to \overline{X\otimes B}$ are acyclic cofibrations.
\end{lemma}
\begin{proof}
Remark that we have cocartesian squares
\[\begin{tikzcd}
	{X\otimes([00,01]\coprod [30,31])} & {{X\otimes A_{0}}\cup {X\otimes A_{1}}\cup {X\otimes A_{2}}} & {{X\otimes B}} \\
	{[00,01]\coprod [30,31]} & {\overline{X\otimes A_{0}}\cup \overline{X\otimes A_{1}}\cup \overline{X\otimes A_{2}}} & {\overline{X\otimes B}}
	\arrow[""{name=0, anchor=center, inner sep=0}, from=1-2, to=1-3]
	\arrow[from=1-1, to=2-1]
	\arrow[from=2-1, to=2-2]
	\arrow[""{name=1, anchor=center, inner sep=0}, from=1-1, to=1-2]
	\arrow[from=1-2, to=2-2]
	\arrow[from=2-2, to=2-3]
	\arrow[from=1-3, to=2-3]
	\arrow["\lrcorner"{anchor=center, pos=0.125, rotate=180}, draw=none, from=2-2, to=1]
	\arrow["\lrcorner"{anchor=center, pos=0.125, rotate=180}, draw=none, from=2-3, to=0]
\end{tikzcd}\]
and 
\[\begin{tikzcd}
	{X\otimes([00,01]\coprod [30,31])} & {{X\otimes A_{4}}} & {{X\otimes B}} \\
	{[00,01]\coprod [30,31]} & {\overline{X\otimes A_{4}}} & {\overline{X\otimes B}}
	\arrow[""{name=0, anchor=center, inner sep=0}, from=1-2, to=1-3]
	\arrow[from=1-1, to=2-1]
	\arrow[from=2-1, to=2-2]
	\arrow[""{name=1, anchor=center, inner sep=0}, from=1-1, to=1-2]
	\arrow[from=1-2, to=2-2]
	\arrow[from=2-2, to=2-3]
	\arrow[from=1-3, to=2-3]
	\arrow["\lrcorner"{anchor=center, pos=0.125, rotate=180}, draw=none, from=2-2, to=1]
	\arrow["\lrcorner"{anchor=center, pos=0.125, rotate=180}, draw=none, from=2-3, to=0]
\end{tikzcd}\]
The result then follows from lemma \ref{lemma:formula for gray 0}.
\end{proof}
\begin{lemma}
\label{lemma:formula for gray 2}
The morphisms 
$\overline{X\otimes A_0} \to [1]\fwedge \Sigma X$ and $\overline{X\otimes A_2} \to \Sigma X\fwedge [1],$
induced by the morphism $A_0\to [00,01,11]_t$ and $A_2\to [20,30,31]_t$, are acyclic cofibrations. 
\end{lemma}
\begin{proof}
We have cocartesian squares
\[\begin{tikzcd}
	{X\otimes ([00,01]\coprod \{11\})} & {X\otimes [00,01]\coprod_{X\otimes[01]} X\otimes[01,11]} & {X\otimes A_0} \\
	{[00,01]\coprod \{11\}} & {[1]\coprod_{[0]}\Sigma X} & {\overline{X\otimes A_0}}
	\arrow[from=1-1, to=2-1]
	\arrow[""{name=0, anchor=center, inner sep=0}, "\sim", from=1-2, to=1-3]
	\arrow[""{name=1, anchor=center, inner sep=0}, from=1-1, to=1-2]
	\arrow["\sim", from=2-2, to=2-3]
	\arrow[from=1-2, to=2-2]
	\arrow[from=1-3, to=2-3]
	\arrow[from=2-1, to=2-2]
	\arrow["\lrcorner"{anchor=center, pos=0.125, rotate=180}, draw=none, from=2-2, to=1]
	\arrow["\lrcorner"{anchor=center, pos=0.125, rotate=180}, draw=none, from=2-3, to=0]
\end{tikzcd}\]
That shows that $[1]\coprod_{[0]} \Sigma X \to \overline{X\otimes A_0}$ is an acyclic cofibration. We then have a commutative diagram: 
\[\begin{tikzcd}
	{[1]\coprod_{[0]}\Sigma X} & {\overline{X\otimes A_0}} & {[1]\fwedge\Sigma X}
	\arrow["\sim", from=1-1, to=1-2]
	\arrow[from=1-2, to=1-3]
	\arrow["\sim", curve={height=-24pt}, from=1-1, to=1-3]
\end{tikzcd}\]
and by two out of three, this shows that $\overline{X\otimes A_0} \to [1]\fwedge \Sigma X$ is an acyclic cofibration. 
We proceed similarly for the second morphism. 
\end{proof}

\begin{lemma}
\label{lemma:formula for gray 3}
Marked simplicial sets $\overline{X\otimes A_1}$ and $\overline{X\otimes A_4}$ are respectively equal to $\Sigma (X\otimes [1])$ and $(\Sigma X)\otimes [1]$.
\end{lemma}
\begin{proof}
This is true by the definition of these objects.
\end{proof} 

\begin{proof}[Proof of theorem \ref{theo:interval_first_formula}]
According to lemma \ref{lemma:formula for gray 3} we have a cocartesian square
\[\begin{tikzcd}
	{\overline{X\otimes A_{0}}\coprod\overline{X\otimes A_{2}}} & {\overline{X\otimes A_{0}}\cup \overline{X\otimes A_{1}}\cup \overline{X\otimes A_{2}}} \\
	{[1]\fwedge\Sigma X\coprod \Sigma X\fwedge[1]} & {[1]\fwedge\Sigma X\coprod_{\Sigma (X\otimes\{0\})} \Sigma(X\otimes[1])\coprod_{\Sigma (X\otimes\{1\})} \Sigma X\fwedge[1]}
	\arrow[from=1-1, to=2-1]
	\arrow[from=1-1, to=1-2]
	\arrow[from=1-2, to=2-2]
	\arrow[from=2-1, to=2-2]
\end{tikzcd}\]
The left vertical morphism is a weak equivalence according to lemma \ref{lemma:formula for gray 2}, and the horizontal morphisms are cofibrations. By left properness, the right vertical morphism is a weak equivalence. Combined with lemmas \ref{lemma:formula for gray 1} and \ref{lemma:formula for gray 3}, this provides a zigzag of weak equivalences between 
$
[1]\fwedge\Sigma X\coprod_{\Sigma (X\otimes\{0\})} \Sigma(X\otimes[1])\coprod_{\Sigma (X\otimes\{1\})} \Sigma X\fwedge[1]$
 and $(\Sigma X)\otimes[1].$
\end{proof}

\subsection{Formulas for the Gray cone and the Gray $\circ$-cone}
\begin{theorem}
\label{theo:cyl_formula}
There is a zigzag of acyclic cofibrations, natural in $X$, between the colimit of the diagram 
$$ \Sigma X\fwedge [1]\leftarrow \Sigma X\to \Sigma([0]\costar X)$$
and $\Sigma X \star[0]$.

There is a zigzag of acyclic cofibrations, natural in $X$, between the colimit of the diagram 
$$\Sigma(X\star[0]) \leftarrow \Sigma X\to [1]\fwedge\Sigma X$$
and $[0]\costar \Sigma X$.
\end{theorem}

\begin{proof}
We consider the diagram:
\[\begin{tikzcd}
	{[1]} & {[1]\coprod_{[0]}\Sigma X} & {\Sigma X\fwedge[1]\coprod_{\Sigma X} \Sigma( X\otimes[1]) \coprod_{\Sigma X}[1]\fwedge\Sigma X} \\
	{[1]} & {[1]\fwedge \Sigma X} & {\Sigma X\fwedge[1]\coprod_{\Sigma X} \Sigma( X\otimes[1]) \coprod_{\Sigma X}[1]\fwedge\Sigma X}
	\arrow["id", from=1-3, to=2-3]
	\arrow[from=1-2, to=1-1]
	\arrow[from=2-2, to=2-1]
	\arrow[from=1-2, to=1-3]
	\arrow[from=2-2, to=2-3]
	\arrow["\sim"', from=1-2, to=2-2]
	\arrow["id"', from=1-1, to=2-1]
\end{tikzcd}\]
All vertical morphisms are weak equivalences.
We denote by $A$ the colimit of the first line. The theorem \ref{theo:interval_first_formula} implies that there is a zigzag of acyclic cofibrations between $A$ and $X\diamond [0]$. Colimits of the two lines are homotopy colimits, and the comparison morphism is then an acyclic cofibration. 
We then have a zigzag of acyclic cofibrations: 
$$
X\star [0]\leftarrow X\diamond[0] \leftrightsquigarrow A\to \Sigma X\fwedge [1]\coprod_{\Sigma X} \Sigma([0]\costar X)
$$

The second assertion is demonstrated similarly.
\end{proof}

\begin{cor}
\label{cor:star and zigzag}
Let $f:C\to D$ be a fibration between complicial sets, and $K\to L$ a cofibration. It $f$
has the right lifting property against $$\Sigma( [0]\costar K\cup \emptyset \star L )\to \Sigma([0]\costar L),$$ then $f$
 has the right lifting property against $$(\Sigma K)\star [0]\cup (\Sigma L)\star \emptyset \to \Sigma K\star [0].$$
 
If $f$ has the right lifting property against $\Sigma [1]\to \Sigma[1]_t$, then $f$ has the right lifting property against
$$[1]_t\star\emptyset \cup[1]\star [0] \to [1]_t\star [0]$$
\end{cor}
\begin{proof}
Suppose that $f$ fulfills the condition. The class of cofibration having the right lifting property against $f$ is closed by pushouts and, according to \ref{prop:lifting_property_zigzag_of_acyclic_cofibration}, by zigzag of acyclic cofibration. The morphism 
$$\alpha:\Sigma L\fwedge [1]\coprod\limits_{\Sigma L} \Sigma([0]\costar K\coprod\limits_{\emptyset \star K}\emptyset\star L )\to
 \Sigma L\fwedge [1]\coprod\limits_{\Sigma L} \Sigma([0]\costar L)$$ is then in this class. 
Remark that we have a cocartesian square
\[\begin{tikzcd}
	{\Sigma L \cup[1]\coprod\limits_{\Sigma K \cup[1]}\Sigma K\fwedge [1]} & {\Sigma L \cup[1]\coprod\limits_{\Sigma K \cup[1]}\Sigma K\fwedge [1]\coprod\limits_{\Sigma L} \Sigma([0]\costar K)} \\
	{\Sigma L\fwedge [1]} & {\Sigma L\fwedge [1]\coprod\limits_{\Sigma L} \Sigma([0]\costar K\coprod\limits_{\emptyset \star K}\emptyset\star L )}
	\arrow[from=1-1, to=1-2]
	\arrow[from=1-1, to=2-1]
	\arrow[from=1-2, to=2-2]
	\arrow[from=2-1, to=2-2]
\end{tikzcd}\]
where the left vertical morphism, and so also the right vertical morphism, is an acyclic cofibration. This induces a zigzag of acyclic cofibration between $\alpha$ and $\beta$ where $\beta$ is 
$$\Sigma L \cup[1]\coprod\limits_{\Sigma K \cup[1]}\Sigma K\fwedge [1]\coprod\limits_{\Sigma L} \Sigma([0]\costar K)\to 
 \Sigma L\fwedge [1]\coprod\limits_{\Sigma L} \Sigma([0]\costar L)$$
Eventually, the theorem \ref{theo:cyl_formula} induces a zigzag of acyclic cofibration between $\beta$ and 
$(\Sigma K)\star [0]\cup (\Sigma L)\star \emptyset \to \Sigma K\star [0]$ which concludes the proof of the first assertion.

For the second assertion, remark that $[1]_t\star [0]$ is $\tau^i_1([1]_t\star\emptyset \cup [1]\star [0])$. As $\tau^i_1$ is a left Quillen functor, 
the theorem \ref{theo:cyl_formula} induces a zigzag of acyclic cofibration between $[1]_t\star\emptyset \cup[1]\star [0] \to [1]_t\star [0]$ and 
$$[1]_t\fwedge [1]\coprod_{[1]}\Sigma [1]\to [1]_t\fwedge [1]\coprod_{[1]}\Sigma [1]_t.$$
As this cofibration is a pushout of $\Sigma [1]\to \Sigma [1]_t$, this concludes the proof.
\end{proof}

\begin{cor}
\label{cor:costar and zigzag}
Let $f:C\to D$ be a fibration between complicial sets, and $K\to L$ a cofibration. It $f$
has the right lifting property against
$$\Sigma (L\star \emptyset \cup K \star [0])\to \Sigma (L\star [0]),$$
then $f$ has the right lifting property against 
$$ [0]\costar \Sigma K \cup \emptyset \star \Sigma L \to [0]\costar \Sigma L.$$

If $f$ has the right lifting property against $\Sigma [1]\to \Sigma[1]_t$, then $f$ has the right lifting property against
$$[0]\costar [1]\cup \emptyset \star [1]_t\to [0]\costar [1]_t$$
\end{cor}
\begin{proof}
The proof is similar to the one of corollary \ref{cor:star and zigzag}.

\end{proof}

\section{Globular equivalences}
\label{section:Globular equivalences}

\subsection{Homotopy categories}

\p
The \wcsnotionsym{$n$-globe}{(da@$\Db_n$}{globe@$n$-globe}{for marked simplicial sets} is the marked simplicial set $\Db_n:=\Sigma^n [0]$. We then have 
$\Db_0:=[0]$ and $\Db_{n+1}:= \Sigma \Db_n$.
This defines a globular object in $\mSset$:
\[\begin{tikzcd}
	{\Db_0} & {\Db_1} & {\Db_2} & {...}
	\arrow["{i_0^+}", shift left=2, from=1-1, to=1-2]
	\arrow["{i_1^+}", shift left=2, from=1-2, to=1-3]
	\arrow["{i_3^+}", shift left=2, from=1-3, to=1-4]
	\arrow["{i_0^-}"', shift right=2, from=1-1, to=1-2]
	\arrow["{i_1^-}"', shift right=2, from=1-2, to=1-3]
	\arrow["{i_3^-}"', shift right=2, from=1-3, to=1-4]
\end{tikzcd}\]
and we have equalities:
$$i_{n+1}^- i^+_n=i^+_{n+1} i^-_n~~~~i^+_{n+1} i^-_n=i^+_{n+1} i^+_n.$$
 We also set $(\Db_n)_t:= \tau^i_{n-1}(\Db_n)$ for $n>0$ and $\partial\Db_n:=\Sigma^n \emptyset$. 
 We then have a canonical inclusions 
 $$\partial \Db_0\to \Db_0$$ and
 for any $n>0$, we have canonical inclusions 
 $$\partial \Db_n\to \Db_n\to (\Db_n)_t.$$

Let $C$ be a complicial set. A \wcsnotion{$n$-cell}{cell@$n$-cell}{for marked simplicial sets} $a$ of $C$ is a morphism $a:\Db_n\to C$. If $n$ is non null, the \textit{source} of $a$ (resp. the \textit{target} of $a$)
is the $(n-1)$-cell $a\circ i^-_{n-1}$ (resp. $a\circ i^+_{n-1}$). The cell $a$ is thin if the corresponding morphism $\Db_n\to C$ factorizes via $(\Db_n)_t$.

\p From now on, and until the end of this section, we fix a complicial set $C$. All considered cells are cells of $C$.

Let $n$ be a non null integer, and $a,b$ two $n$-cells.
Cells $a$ and $b$ are \textit{parallel} if they share the same source and the same target. They are \textit{composable} if the source of $a$ is the target of $b$.

Let $a$ and $b$ be two parallel cells. The cell $a$ is \wcnotion{equivalent}{equivalent $n$-cells} to the cell $b$ if there exists a thin $(n+1)$-cell $d:a\to b$, or equivalently, if there exists a homotopy $\Db_n\times [1]_t$ between $a$ and $b$, and constant on $\partial \Db_n\times [1]_t$. This relation is denoted by $\sim$.

\begin{lemma}
The relation $\sim$ is reflexive, symmetric and transitive.
\end{lemma}
\begin{proof}
This comes from usual properties of fibrant objects.
\end{proof}

\begin{lemma}
Let $a$, $b$ be two equivalent cells. If $a$ is thin, so is $b$.
\end{lemma}
\begin{proof}
As $\{0\}\to [1]_t$ is a weak equivalence, so is $\Db_n\times [1]_t\cup (\Db_n)_t\times \{0\}\to (\Db_n)_t\times [1]_t$. As $C$ is fibrant, this directly implies the result. 
\end{proof}

\begin{construction}
\label{cons:composition_homotopy_category}
Let $a,b$ be two composable $n$-cells . A composition of ${a}$ and ${b}$ is a $n$-cell $a\circ b$ that fits in a diagram:
\[\begin{tikzcd}
	{\Db_n\coprod_{\Db_{n-1}}\Db_n} \\
	{\Sigma^{n-1}([2]_t)} & C \\
	{\Db_{n}}
	\arrow[from=1-1, to=2-1]
	\arrow["{a\coprod b}", from=1-1, to=2-2]
	\arrow["{a\circ b}"', from=3-1, to=2-2]
	\arrow[from=3-1, to=2-1]
	\arrow[from=2-1, to=2-2]
\end{tikzcd}\]
As $C$ is a fibrant object, if $(a\circ b)'$ is any other composition, $(a\circ b)'\sim a\circ b$.
\end{construction}

\begin{lemma}
\label{lemma:associativity_of_composition_in_homotopy_category}
Let $a,b,c$ be three composable cells. There exists compositions such that $(a\circ b)\circ c = a\circ (b\circ c)$.
\end{lemma}
\begin{proof}
Let $M$ be the marking on $[3]$ that includes all simplices of dimension superior or equal to $2$. We define $\Sp_{[3]}$ as the simplicial set $[1]\coprod_{[0]} [1]\coprod_{[0]} [1]$. Remark that the cofibration $\Sp_{[3]}\to ([3],M)$ is acyclic. We then have a lift $f$ in the following diagram
\[\begin{tikzcd}
	{\Sigma^{n-1}\Sp_{[3]}} & C \\
	{\Sigma^{n-1}([3],M)}
	\arrow[from=1-1, to=2-1]
	\arrow["{a\coprod b\coprod c}", from=1-1, to=1-2]
	\arrow["f"', dashed, from=2-1, to=1-2]
\end{tikzcd}\]
The morphism $f$ provides all the desired compositions.
\end{proof}

\begin{definition}
We define the category $\pi_0(C)$ whose objects are $0$-cells $x:s\to t$, and edges between $x,y:s\to t$ are equivalence classes of the set of $1$-cells $f:x\to y$ quotiented by the relation $\sim$. The composition is given by construction \ref{cons:composition_homotopy_category} which is associative according to lemma \ref{lemma:associativity_of_composition_in_homotopy_category}.

 Let $n>0$ be an integer, and $s,t$ two parallel $(n-1)$-cells. We define the category $\pi_n(s,t,C)$ whose objects are $n$-cells $x:s\to t$, and edges between $x,y:s\to t$ are equivalence classes of the set of $(n+1)$-cells $f:x\to y$ quotiented by the relation $\sim$. 
The composition is given by construction \ref{cons:composition_homotopy_category} which is associative according to lemma \ref{lemma:associativity_of_composition_in_homotopy_category}.
\end{definition}

\begin{prop}
\label{prop:in_the_homotopy_category_thin_is_iso}
Let $x,y:s\to t$ be two parallel $n$-cells, and $f:x\to y$ a $n+1$-cell. The cell $f$ is thin if and only if $[f]:x\to y$ is an isomorphism in $\pi_n(s,t,C)$.
\end{prop}
\begin{proof}
Suppose first that $f$ is thin. There are liftings in the following diagrams:
\[\begin{tikzcd}
	{\Sigma^{n}\Lambda^0[2]} & C & {\Sigma^{n}\Lambda^2[2]} & C \\
	{\Sigma^{n}[2]^0} && {\Sigma^{n}[2]^0}
	\arrow["{f\amalg id}", from=1-1, to=1-2]
	\arrow[from=1-1, to=2-1]
	\arrow["h"', dotted, from=2-1, to=1-2]
	\arrow["{id\amalg f}", from=1-3, to=1-4]
	\arrow["k"', dotted, from=2-3, to=1-4]
	\arrow[from=1-3, to=2-3]
\end{tikzcd}\]
Let $g:y\to z$ be the restriction of $h$ to $\Sigma^{n}[1,2]$ and $l:y\to z$ be the restriction of $k$ to $\Sigma^{n}[0,1]$. We then have $[f][g]= id$, and $[h][f]=id$, and $[f]$ is then an isomorphism. 

For the other direction, suppose that $[f]$ is an isomorphism. Let $M$ be the marking on $[3]$ that includes all simplices of dimension superior or equal to $2$. As $\Sp_{[3]}\to ([3],M)$ is a weak equivalence, there is a lifting in the following diagram:
\[\begin{tikzcd}
	{\Sigma^n([0,1]\coprod_{\{1\}}[1,2]\coprod_{\{2\}}[2,3])} && C \\
	{\Sigma^n([3],M)}
	\arrow["{f^{-1}\amalg f\amalg f^{-1}}", from=1-1, to=1-3]
	\arrow[from=1-1, to=2-1]
	\arrow["h"', dotted, from=2-1, to=1-3]
\end{tikzcd}\]
Now $h(\Sigma^n[0,3])$ and $h(\Sigma^n[0,2])$ are respectively compositions of $(f,f^{-1})$ and $(f^{-1},f)$. Hypotheses imply that these compositions are equivalent to identities, and so are thin. The morphism then lifts to $\Sigma^n [3]^{eq}$. The object $C$ being fibrant, $h$ lifts to $\Sigma^n [3]^\sharp$, and $f$ is then thin.
\end{proof}

\begin{lemma}
\label{lemma:homotopycategory_are_idenpendant_of}
Let $s,t$ and $s',t'$ be two pairs of parallel cells, and $\psi:\partial \Db_n\times [1]_t\to C$ a homotopy between $s\cup t:\partial \Db_n\to C$ and $s'\cup t':\partial \Db_n\to C$. Then 
$$\pi_n(s,t,C)\cong \pi_n(s',t',C)$$
\end{lemma}
\begin{proof}
For each $x:s\to t$, there exists a lifting $h_x$ in the following diagram:
\[\begin{tikzcd}
	{\Db_n\times\{0\}\cup\partial\Db_n\times [1]_t} & C \\
	{\Db_n\times [1]_t}
	\arrow[from=1-1, to=2-1]
	\arrow["x\cup\psi", from=1-1, to=1-2]
	\arrow["h"', dotted, from=2-1, to=1-2]
\end{tikzcd}\]
and we define $F(x)$ as the restriction of $h_x$ to $\Db_n\times \{1\}$. For a $(n+1)$-cell $f:x\to y$, there exists a lifting $h_f$ in the following diagram:
\[\begin{tikzcd}
	{\Db_{n+1}\times \{0\}\cup \partial\Db_{n+1}\times [1]_t} & C \\
	{\Db_{n+1}\times [1]_t}
	\arrow["{f\cup h_x\cup h_y}", from=1-1, to=1-2]
	\arrow[from=1-1, to=2-1]
	\arrow["{h_f}"', from=2-1, to=1-2]
\end{tikzcd}\]
and we define $F(f)$ as the restriction of $h_f$ to $\Db_{n+1}\times \{1\}$. Furthermore, the unicity up to homotopy of lifting implies that $[F(f)]$ is independent of the choice of the lifting, and that $f\sim g$ implies $[F(f)]=[F(g)]$.
If $g:y\to z$ is an other morphism, and $\psi:\Sigma^n[2]_t \to C$ corresponds to the composition of $f$ and $g$, 
there is a lift in the following diagram:
\[\begin{tikzcd}
	{\Sigma^n [2]_t\cup (\Sigma^n\partial[2])\times [1]_t} && C \\
	{\Sigma^n [2]_t\times [1]_t}
	\arrow["{ \phi \cup h_f\cup h_g\cup h_{f\circ g}}", from=1-1, to=1-3]
	\arrow[from=1-1, to=2-1]
	\arrow[dotted, no head, from=2-1, to=1-3]
\end{tikzcd}\]
Restricted to $\Sigma^n [2]_t\times \{1\}$ this shows that $F$ commutes with compositions. We then have defined a functor 
$$F:\pi_n(s,t,C)\to \pi_n(s',t',C).$$

Using exactly the same procedure, where we just invert $0$ and $1$, we define a functor:
$$G:\pi_n(s',t',C)\to \pi_n(s,t,C).$$
Now, we have a lift in the following diagram:
\[\begin{tikzcd}
	{\Db_{n}\times \Lambda^{2}[2]^\sharp\cup\partial\Db_n\times[2]^\sharp} &&& C \\
	{\Db_n\times[2]^\sharp}
	\arrow["{h_x\cup h_{F(x)}\cup\psi(id\times s^0)}", from=1-1, to=1-4]
	\arrow[from=1-1, to=2-1]
	\arrow["{k_x}"', dotted, from=2-1, to=1-4]
\end{tikzcd}\]
The restriction of $k_x$ to $\Db_n\times [0,1]_t$ provides a thin cell $x\to G(F(x))$, which corresponds to an isomorphism in $\pi_n(s,t,C)$ according to proposition \ref{prop:in_the_homotopy_category_thin_is_iso}. If $f:x\to y$ is a $(n+1)$-cell, there is a lifting in the following diagram:
\[\begin{tikzcd}
	{\Db_{n+1}\times \Lambda^{2}[2]^\sharp\cup\partial\Db_{n+1}\times[2]^\sharp} &&& C \\
	{\Db_{n+1}\times[2]^\sharp}
	\arrow["{h_f\cup h_{F(f)}\cup k_x\cup k_y}", from=1-1, to=1-4]
	\arrow["{k_f}"', dotted, from=2-1, to=1-4]
	\arrow[from=1-1, to=2-1]
\end{tikzcd}\]
The restriction of $k_f$ to $\Db_{n+1}\times[0,1]_t$ induces in $\pi_n(s,t,C)$ a commutative diagram:
\[\begin{tikzcd}
	x & GFx \\
	y & GFy.
	\arrow["{[GFf]}", from=1-2, to=2-2]
	\arrow["{[f]}"', from=1-1, to=2-1]
	\arrow[from=2-1, to=2-2]
	\arrow[from=1-1, to=1-2]
\end{tikzcd}\]
We then have an invertible natural transformation $\psi: id\to GF$. Similarly we can construct an other natural transformation $id\to GF$, which shows the desired equivalence of categories.
\end{proof}
\p
Let $a$ be an element of  $\Hom_{ho(\mSset)}(\partial \Db_n,C)$. We define 
\begin{equation}
\label{eq:def of pi a}
\pi_n(a, C) := \pi_n(s,t,C)
\end{equation}
where $s,t$ is a pair of parallel arrows such that $s\cup t$ represents $a$.
The previous proposition shows that this is well defined.

\subsection{A criterion to be a weak equivalence}
\p
A morphism $p:C\to D$ between complicial sets is a \wcnotion{$\Db$-equivalence}{Dequivalence@$\Db$-equivalence} if 
$$\pi_0(C)\to \pi_0(D)$$
is an equivalence of categories, and for any $n>0$ and pair of parallel arrow $s,t$, the induced functor
$$\pi_n(s,t,C)\to \pi_n(ps,pt,D)$$
is an equivalence of categories. 

A \wcnotion{$\Db$-trivial fibration}{Dtrivial fibration@$\Db$-trivial fibration} is a fibration having the right lifting property against $\partial\Db_n\to \Db_n$ and $\Db_n\to (\Db_{n})_t$.

\begin{lemma}
\label{lemma:fibration_are_isofibration}
Let $\alpha\in\{-,+\}$.
The morphism $i^\alpha_{n+1}:\Db_n\to (\Db_{n+1})_t$ is an acyclic cofibration. 
\end{lemma}
\begin{proof}
We have a pushout diagram
\[\begin{tikzcd}
	{\Db_n\times\{\alpha\}\cup\partial\Db_n\times [1]_t} & {\Db_n\times \{\alpha\}} \\
	{\Db_n\times [1]_t} & {(\Db_n)_t}
	\arrow[""{name=0, anchor=center, inner sep=0}, "{id\cup \partial\times s^0}", from=1-1, to=1-2]
	\arrow[from=1-1, to=2-1]
	\arrow[from=2-1, to=2-2]
	\arrow["{i^\alpha_{n+1}}", from=1-2, to=2-2]
	\arrow["\lrcorner"{anchor=center, pos=0.125, rotate=180}, draw=none, from=2-2, to=0]
\end{tikzcd}\]
The left hand morphism being an acyclic cofibration, this concludes the proof.
\end{proof}

\begin{lemma}
\label{lemma:acyclic_cofibration_are_G_equivalence}
Acyclic cofibrations between complicial sets are $\Db$-equivalences.
\end{lemma}
\begin{proof}
Let $i:A\to B$ be an acyclic cofibration. The morphism $i$ admits a retraction $r:B\to A$: 
\[\begin{tikzcd}
	A & A \\
	B.
	\arrow["id", from=1-1, to=1-2]
	\arrow["i"', from=1-1, to=2-1]
	\arrow["r"', from=2-1, to=1-2]
\end{tikzcd}\]
and a homotopy $\psi$ between $id_B$ and $ir$ which is constant on the image of $i$, obtained as the lift in the following diagram:
\[\begin{tikzcd}
	{B\times\{0\}\coprod _{A\times\{0\}}A\times[1]_t} & B \\
	{B\times[1]_t}
	\arrow[from=1-1, to=2-1]
	\arrow["\phi"', dashed, from=2-1, to=1-2]
	\arrow[from=1-1, to=1-2]
\end{tikzcd}\]
Let $n>0$ be an integer, and $s$, $t$ be two $(n-1)$-cells of $C$. The retraction implies that $i_!$ is an injection on morphisms. For any $n$-cell $y:i(s)\to i(t)$ in $B$, the homotopy $\psi$ induces a thin cell $y\to ir(y)$ which corresponds to an isomorphism in $\pi_n(is,it,B)$ according to proposition \ref{prop:in_the_homotopy_category_thin_is_iso}. The functor $i_!$ is then essentially surjective. For any $(n+1)$-cell $f:i(x)\to i(y)$, the homotopy $\psi$ induces an equivalence $[ir(f)]\sim [f]$. The morphism $i_!$ is a surjection on morphisms. All put together, $i_!$ is fully faithfull and essentially surjective, and is then an equivalence. We proceed similarly to show that $i_!:\pi_0(A)\to \pi_0(B)$ is an equivalence.
\end{proof}

\begin{lemma}
\label{lemma:2_out_of_3_for_G_equivalence}
Suppose given a commutative triangle between complicial sets
\[\begin{tikzcd}
	& B \\
	A && C
	\arrow["g"', from=2-1, to=2-3]
	\arrow["f", from=1-2, to=2-3]
	\arrow["i", from=2-1, to=1-2]
\end{tikzcd}\]
If $i$ is an acyclic cofibration, and $g$ is a $\Db$-equivalence, then $f$ is a $\Db$-equivalence.
\end{lemma}
\begin{proof}
Let $s,t$ be any pair of parallel arrows in $B$. There exists a pair of parallel arrows $s',t'$ in $A$ such that $s\cup t$ and $is'\cup it'$ correspond to the same element in $[\partial\Db_n,B]$. We then have a diagram:
\[\begin{tikzcd}
	& { \pi(s,t,B)} & {\pi(fs,ft,C)} \\
	{\pi(s,t,B)} & { \pi(is,it,B)} & {\pi(gs,gt,C).}
	\arrow["\sim", from=2-1, to=2-2]
	\arrow[from=2-2, to=2-3]
	\arrow["\sim", from=1-3, to=2-3]
	\arrow["\sim", from=1-2, to=2-2]
	\arrow[from=1-2, to=1-3]
	\arrow["\sim"', curve={height=18pt}, from=2-1, to=2-3]
\end{tikzcd}\]
where arrows labeled by $\sim$ are isomorphisms according to lemmas \ref{lemma:homotopycategory_are_idenpendant_of} and \ref{lemma:acyclic_cofibration_are_G_equivalence}. By two out of three, this shows that $ \pi(s,t,B)\to \pi(fs,ft,C)$ is an isomorphism, and $f$ is then a $\Db$ equivalence.
\end{proof}

\begin{prop}
\label{prop:caracterisation_of_G_fibration}
Let $p:C\to D$ be a fibration between complicial sets.
The morphism $p$ is a $\Db$-trivial fibration if and only if it is a $\Db$-equivalence.
\end{prop}
\begin{proof}
If $p$ is a $\Db$-trivial fibration, it is obvious that it is a $\Db$-equivalence. For the converse, suppose $p$ is a fibration and a $\Db$-equivalence, and consider a diagram 
\[\begin{tikzcd}
	{\partial\Db_n} & C \\
	{\Db_n} & D
	\arrow[from=1-1, to=2-1]
	\arrow[from=1-1, to=1-2]
	\arrow["x"', from=2-1, to=2-2]
	\arrow["p", from=1-2, to=2-2]
\end{tikzcd}\]
As $p$ is a $\Db$-equivalence this implies that there exists a cell $\overline{x}:\Db_n\to C$ together with a thin $(n+1)$-cell $y:p(\overline{x})\to y$. All this data corresponds to a diagram:
\[\begin{tikzcd}
	{\Db_n} & C \\
	{(\Db_{n+1})_t} & D
	\arrow["p", from=1-2, to=2-2]
	\arrow["{\delta^0_{n+1}}"', from=1-1, to=2-1]
	\arrow["{\bar{x}}", from=1-1, to=1-2]
	\arrow["y"', from=2-1, to=2-2]
\end{tikzcd}\]
The left hand morphism being an acyclic cofibration according to \ref{lemma:fibration_are_isofibration}, this diagram admits a lift $h:(\Db_{n+1})_t\to C$. The restriction of $h$ to $i^+_{n+1}$ provides a lift in the first diagram. Now, we consider a diagram of shape: 
\[\begin{tikzcd}
	{\Db_n} & C \\
	{(\Db_n)_t} & D
	\arrow[from=1-1, to=2-1]
	\arrow["g", from=1-1, to=1-2]
	\arrow[from=2-1, to=2-2]
	\arrow["p", from=1-2, to=2-2]
\end{tikzcd}\]
with $n>1$.
Let $s,t$ be respectively the $(n-1)$-source and the $(n-1)$-target of $g$. Hypotheses imply that $[p(g)]$ is an isomorphism in $\pi_n(s,t,D)$ and because $p$ is a $\Db$-equivalence, so is $[g]$. According to lemma \ref{prop:in_the_homotopy_category_thin_is_iso}, this implies that $g$ is thin. There exists then a lifting in the previous diagram. The case $n=1$ is similar.
 The morphism $f$ is then a $\Db$-trivial fibration.
\end{proof}

\begin{lemma}
\label{lemma:slice_G_fibrations}
Let $p:X\to Y$ be a $\Db$-trivial fibration between complicial sets. Then for any $x\in X_0$, the induced fibrations
$$X_{/x}\to X\times_Y Y_{/p(x)} ~~\mbox{and}~~ X_{x/}\to X\times_Y Y_{p(x)/}$$
are $\Db$-trivial fibrations.
\end{lemma}
\begin{proof}
We define $\mathbb{P}(p,n)$ to be the statement that $p$ has the right lifting property against 
$$ \Db_n\cup \partial\Db_n\star [0]\to \Db_{n+1}\star[0] \mbox{ and }(\Db_n)_t\cup \Db_n\star [0]\to (\Db_{n})_t\star[0]$$
and against
$$[0]\costar \partial \Db_n\cup \Db_n\to [0]\costar \Db_{n+1} \mbox{ and }[0]\star \Db_n\cup (\Db_n)_t\to [0]\costar (\Db_n)_t$$
We then have to show that for any $n$, $\mathbb{P}(p,n)$ holds.

First, it is obvious that each $\Db$-equivalence $p$ satisfies $\mathbb{P}(p,0)$. As $p$ is a fibration, the corollaries \ref{cor:star and zigzag} and \ref{cor:costar and zigzag} then imply that $\mathbb{P}(p,n+1)$ is equivalent to $\mathbb{P}(p(a,b),n)$ for any $a,b\in X_0$, where $p(a,b)$ is the induced morphism: $X(a,b)\to Y(p(a),p(b))$. 

Using the fact that $p(a,b)$ is a $\Db$-trivial fibration as soon as $p$ is, this shows the desired result.
\end{proof}

\begin{lemma}
\label{lemma:G_fibration_right lifting property_against_partial}
 $\Db$-Trivial fibrations between complicial sets have the right lifting property against $\partial[n]\to [n]$.
\end{lemma}
\begin{proof}
Let $C$ be the class of cofibrations having the right lifting property against $\Db$-equivalences. The lemma \ref{lemma:slice_G_fibrations} implies that for any 
 $K\to L$ in $C$, the induced morphism:
$$L\cup K\star[0]\to L\star[0]$$
is in $C$. 
The class $C$ is then closed under Leibniz join. Furthermore, it includes $\partial[1]\to [1]$, and then, by induction, it includes $\partial[n]\to[n]$ for any integer $n$.
\end{proof}

\begin{lemma}
\label{lemma:G_fibration_right lifting property_against_sat}
$\Db$-Trivial fibrations between complicial sets have the right lifting property against $[n]\to [n]_t$.
\end{lemma}
\begin{proof}
Let $p$ be $\Db$-trivial fibrations between complicial sets, and
 $C_{n,p}$ be the set of objects $A$ such that $p$ has the right lifting property against:
$$A\to \tau^i_{n-1}(A).$$
This set is then closed under colimits, and by zigzags of acyclic cofibrations.
Let $k\leq n$ be two integers. We define $\mathbb{P}(k,n,p)$ to be the statement that 
$$ \Sigma [n-k]_{\circ}\star[k-1]~~~\mbox{and}~~~ [k-1]_{\circ}\costar\Sigma [n-k] $$
are in $C_{n+1,p}$.
The statement $\mathbb{P}(0,0,f)$ corresponds to the belonging of $\Db_1$ to $C_{1,p}$, which is obviously true. Suppose that $0<k$ and $\mathbb{P}(k-1,n,p)$. 
According to theorem \ref{theo:cyl_formula}, the object $\Sigma [n-k]_{\circ}\star[k-1]$ is linked by a zigzag of acyclic cofibrations to the colimit of 
$$
(\Sigma [n-k]_{\circ} \fwedge [1])\star [k-2] \leftarrow (\Sigma [n-k]_{\circ})\star [k-2] \to (\Sigma [n-k+1]_{\circ})\star [k-2]
$$
The center object and the left hand object are in $C_{n+1,p}$ because there are invariant under $\tau^i_{n}$, and the 
 right hand object is in $C_{n+1,p}$ by induction hypothesis. The object $\Sigma [n-k]_{\circ}\star[k-1]$ is then in $C_{n+1,p}$.
We demonstrate similarly that $[k-1]_{\circ}\costar\Sigma [n-k]$ is in $C_{n+1,p}$.

This then implies $\mathbb{P}(k,n,p)$. Eventually, $\mathbb{P}(0,n+1,p)$ is equivalent to $\mathbb{P}(n,n,p(a,b))$ for any pair of objects $(a,b)\in X_0$.
The statement $\mathbb{P}(k,n,p)$ is then true for any $k,n$ and $\Db$-trivial fibrations between complicial sets $p$. This implies that $p$ has the right lifting property against $[n]\to [n]_t$.
\end{proof}

\begin{theorem}
\label{theo:f_weak_equivalence_ssi_f_G_equivalence}
Let $p$ be a map between complicial sets. Then $p$ is a weak equivalence if and only if it is a $\Db$-equivalence.
\end{theorem}
\begin{proof}
According to lemmas \ref{lemma:acyclic_cofibration_are_G_equivalence} and \ref{lemma:2_out_of_3_for_G_equivalence} we can restrict ourselves to the case where $p$ is a fibration. If it is a weak equivalence, $p$ is then a trivial fibration and is then a $\Db$-equivalence. Suppose now that $p$ is a $\Db$-equivalence. According to proposition \ref{prop:caracterisation_of_G_fibration}, $p$ is then a $\Db$-trivial fibration. Lemmas \ref{lemma:G_fibration_right lifting property_against_partial} and \ref{lemma:G_fibration_right lifting property_against_sat} imply that $p$ is a trivial fibration.
\end{proof}

\begin{definition}
Let $p:X\to Y$ be a morphism between complicial sets. The morphism $p$ is \snotion{essentially surjective}{for marked simplicial sets} if for any $x\in Y_0$, there exists $\bar{x}\in X_0$ together with a thin cell $\bar{x}\to x$.
The morphism $f$ is \snotion{fully faithful}{for marked simplicial sets} if the induced morphisms: 
$$X(a,b)\to Y(pa,pb)$$
are weak equivalences for any $a,b\in X_0$.
\end{definition}

\begin{cor}
Let $p$ be a map between complicial sets. Then $p$ is a weak equivalence if and only if it is fully faithfull and essentially surjective.
\end{cor}
\begin{proof}
If $p$ is a weak equivalence, it is then fully faithfull and essentially surjective. Conversely, suppose $p$ is fully faithfull and essentially surjective. 
The morphism $\pi_0(X)\to \pi_0(Y)$ is fully faithfull and essentially surjective, and then an equivalence of category. For $(a,b)$ a pair of $0$-cells, we have equalities:
\[\begin{tikzcd}
	{\pi_1(a,b,X)} & {\pi_0(X(a,b))} \\
	{\pi_1(pa,pb,Y)} & {\pi_0(Y(pa,pb)).}
	\arrow["{\pi_0p(a,b)}", from=1-2, to=2-2]
	\arrow["{\pi_1p}"', from=1-1, to=2-1]
	\arrow[Rightarrow, no head, from=2-1, to=2-2]
	\arrow[Rightarrow, no head, from=1-1, to=1-2]
\end{tikzcd}\]
The morphism $\pi_1(a,b,p)$ is then an equivalence of categories. For $(s,t)$ a pair of parallel arrows of dimension $>1$, if we denote by $a$ and $b$ the $0$-source and the $0$-target of $s$ and $t$, we have a diagram:
\[\begin{tikzcd}
	{\pi_n(s,t,X)} & {\pi_{n-1}(s,t,X(a,b))} \\
	{\pi_n(pa,pb,Y)} & {\pi_{n-1}(s,t,Y(pa,pb)).}
	\arrow["{\pi_{n-1}(s,t ,p(a,b))}", from=1-2, to=2-2]
	\arrow["{\pi_np}"', from=1-1, to=2-1]
	\arrow[Rightarrow, no head, from=2-1, to=2-2]
	\arrow[Rightarrow, no head, from=1-1, to=1-2]
\end{tikzcd}\]
The morphism $\pi_n(a,b,p)$ is then an equivalence of categories.
The morphism $p$ is then a $\Db$-equivalence, and according to \ref{theo:f_weak_equivalence_ssi_f_G_equivalence}, a weak equivalence.
\end{proof}

\subsection{A criterion to be a weakly invertible transformation}
\label{section:A criterion to be a weakly invertible transformation}
The purpose of this section is to show the following proposition:
\begin{prop}
\label{prop:criterimu_to_be_an_weak_equivalence}
Let $i:\mSset\to \mSset$ and $j:\mSset\to \mSset$ be two left Quillen functors and $\psi:i\to j$ a natural transformation. If 
$\psi(\Db_n):i (\Db_n) \to j (\Db_n)$ is a weak equivalence for any $n$, then $\psi(X):i(X)\to j(X)$ is a weak equivalence for any $X$.
\end{prop}
For the remaining of this section, we fix two left Quillen functors $i$, $j$ and a natural transformation $\psi:i\to j$ satisfying the previous hypothesis. We denote by $N_i$ and $N_j$ the right adjoints of $i$ and $j$.

\begin{lemma}
\label{lemma:psipartial_is_a_weak_equivalence}
Morphisms $\psi(\partial\Db_n):i(\partial\Db_n)\to j(\partial\Db_n)$ are weak equivalences. 
\end{lemma}
\begin{proof}
We proceed by induction on $n$. The case $n=0$ is trivial. Suppose then the result true at the stage $n-1$. Remark then that $\partial \Db_n$ is the colimit and the homotopy colimit of the span
$$\Db_{n-1}\leftarrow \partial\Db_{n-1}\to \Db_{n-1}$$
As $i$ and $j$ are left Quillen functors, the induction hypothesis implies that $\psi(\partial\Db_n):i(\partial\Db_n)\to j(\partial\Db_n)$ is a weak equivalence.
\end{proof}

\begin{lemma}
\label{lemma:psisat_is_a_weak_equivalence}
Morphisms $\psi((\Db_n)_t):i((\Db_n)_t)\to j((\Db_n)_t)$ are weak equivalences. 
\end{lemma}
\begin{proof}
There is a diagram:
\[\begin{tikzcd}
	{i_!\Db_{n-1}} && {j_!\Db_{n-1}} \\
	{i_!(\Db_n)_t} && {j_!(\Db_n)_t}
	\arrow["{\psi(\Db_n)}", from=1-1, to=1-3]
	\arrow["{i_!(i^-_n)}"', from=1-1, to=2-1]
	\arrow["{j_!(i^-_n)}", from=1-3, to=2-3]
	\arrow["{\psi((\Db_n)_t)}"', from=2-1, to=2-3]
	\arrow["\sim"', draw=none, from=1-1, to=1-3]
	\arrow["\sim", draw=none, from=1-1, to=2-1]
	\arrow["\sim"', draw=none, from=1-3, to=2-3]
\end{tikzcd}\]
By two out of three, this shows that $\psi((\Db_n)_t)$ is a weak equivalence.
\end{proof}

\begin{lemma}
\label{lemma:j*is_a_trivial_fibration}
For any complicial set $Y$, the canonical morphism $N_jY\to N_i Y$ is a weak equivalence.
\end{lemma}
\begin{proof}
Let $Y$ be a complicial set. For any integer $n$, we have by adjunction a bijection
$$\Hom_{ho(\mSset)}(\Db_n, N_jY)\cong \Hom_{ho(\mSset)}(\Db_n, N_iY)$$
and according to lemmas \ref{lemma:psipartial_is_a_weak_equivalence} and \ref{lemma:psisat_is_a_weak_equivalence}, we have bijections
$$\Hom_{ho(\mSset)}(\partial \Db_n, N_jY)\cong \Hom_{ho(\mSset)}(\partial\Db_n, N_iY)$$
$$\Hom_{ho(\mSset)}((\Db_n)_t, N_jY)\cong \Hom_{ho(\mSset)}((\Db_n)_t, N_iY).$$
Let $a$ be an element of $\Hom_{ho(\mSset)}(\partial \Db_n, N_jY)$. We recall that the category $\pi_n(a,N_jY)$ is defined in \ref{eq:def of pi a}. The previous equivalences implies that we have an isomorphism of category
$$\pi_n(a,N_jY)\cong \pi_n(a,N_jY).$$
which concludes the proof according to theorem \ref{theo:f_weak_equivalence_ssi_f_G_equivalence}.
\end{proof}

\begin{proof}[Proof of the proposition \ref{prop:criterimu_to_be_an_weak_equivalence}]
Let $X$ be any marked simplicial set and $Y$ a complicial set. We have equalities:
\[\begin{tikzcd}
	{\Hom_{ho(\mSset)}(j_!X,Y)} & {\Hom_{ho(\mSset)}(X,j^*Y)} \\
	{\Hom_{ho(\mSset)}(i_!X,Y)} & {\Hom_{ho(\mSset)}(X,i^*Y)}
	\arrow[from=1-2, to=2-2]
	\arrow[from=1-1, to=2-1]
	\arrow[Rightarrow, no head, from=1-1, to=1-2]
	\arrow[Rightarrow, no head, from=2-1, to=2-2]
\end{tikzcd}\]
Lemma \ref{lemma:j*is_a_trivial_fibration} implies that the right hand morphism is a bijection, and so is the left hand morphism. 
For any $X$, $\psi(X)$ is then a weak equivalence.
\end{proof}

\subsection{Weak characterization of the identity}
 For the rest of this section, we fix a left Quillen functor $i:\mSset\to \mSset$ such that there exists a zigzag of weakly invertible natural transformations:
$$i(\Db_{\uvar}) \leftrightsquigarrow \Db_{\uvar}.$$

\begin{lemma}
\label{lemma:weak characteroeiation 1}
Let $n$ be any integer, the following natural transformations are pointwise acyclic cofibrations:
$$i\tau^i_n\to \tau^i_{n}i\tau^i_n \leftarrow \tau^i_{n}i.$$
\end{lemma}
\begin{proof}
These are natural transformations between left Quillen functors. The hypothesis implies that they induce weak equivalences on globes of dimension inferior or equal to $n$. Remark that for any $k> n$, as $i_{k-1}^-:\Db_{k-1}\to (\Db_{k})_t$ is an acyclic cofibration and $\tau^i_n$ preserves them, 
$\tau^i_n\Db_{k-1}\to \tau^i_n\Db_{k}$ is an acyclic cofibration. A direct induction implies that $\Db_{n}= \tau^i_n\Db_{n}\to \tau^i_n\Db_{k}$ is an acyclic cofibration.
We then have a commutative diagram: 
\[\begin{tikzcd}
	{i{\tau^i_n}(\Db_k)} & {{\tau^i_n}i {\tau^i_n}(\Db_k)} & {{\tau^i_n}i(\Db_k)} \\
	& {i(\Db_{n})}
	\arrow["\sim"', from=2-2, to=1-1]
	\arrow["\sim"', from=2-2, to=1-2]
	\arrow["\sim", from=2-2, to=1-3]
	\arrow[from=1-1, to=1-2]
	\arrow[from=1-3, to=1-2]
	\arrow[draw=none, from=2-2, to=1-1]
	\arrow[draw=none, from=2-2, to=1-2]
	\arrow[draw=none, from=2-2, to=1-3]
\end{tikzcd}\]
where all morphisms labelled by $\sim$ are weak equivalences.

By two out of three, this implies that theses natural transformations induce weak equivalences on all globes, and proposition \ref{prop:criterimu_to_be_an_weak_equivalence} concludes the proof.
\end{proof}

\begin{prop}
\label{prop:modification_of_the_value_on_thin_representables}
 There exists a zigzag of weakly invertible natural transformations 
$$i\leftrightsquigarrow j$$
where $j$ is a left Quillen functor such that $j([n])=i([n])$ and $j([n]_t)=\tau^i_{n-1}i([n])$, and such that the image of $[n]\to [n]_t$ by $j$ is induced by the canonical morphism $id\to \tau^i_{n-1}(id)$.
\end{prop} 
\begin{proof}
We define $\tilde{i}$ (resp. $j$) to be the colimit preserving functor defined on representables by $\tilde{i}([n]):=i([n])$ and $\tilde{i}:=([n]_t)=\tau^i_{n-1}i([n]_t)$ (resp. $j([n]):=i([n])$ and $j([n]_t):=\tau^i_{n-1}i([n])$). We then have a zigzag of natural transformations 
$$i\xrightarrow{\sim} \tilde{i}\xleftarrow{\sim} j.$$
that are pointwise acyclic cofibrations according to \ref{lemma:weak characteroeiation 1}.
This implies that both $\tilde{i}$ and $j$ are left Quillen functors.
\end{proof}

\p
In the following lemmas, we use the Steiner theory recalled in section \ref{section:Steiner thery}. 
\begin{lemma}
\label{lemma:unicity_of_composition}
Let $m$ be an integer and $X$ and $Y$ be two $\zo$-categories admitting a loop free and atomic basis. We denote by $0$, $1$ and $t$ the three points of $\Sigma X\vee[1]$.
 Let $$f: \Sigma^m ([X,1]\star Y)\to \Sigma^m( ([X,1]\vee[1])\star Y)$$ be a morphism fitting in the following diagram:
\[\begin{tikzcd}
	{\Sigma^m((\{0\}\coprod\{1\})\star Y)} && {\Sigma^m( ([X,1]\vee[1])\star Y)} \\
	{\Sigma^m([X,1]\star Y)} && {\Sigma^m([X,1]\star Y)}
	\arrow["f", from=2-1, to=1-3]
	\arrow["{\Sigma^m(g\star Y)}", from=1-1, to=1-3]
	\arrow[from=1-3, to=2-3]
	\arrow["id"', from=2-1, to=2-3]
	\arrow[from=1-1, to=2-1]
\end{tikzcd}\]
where $g$ sends $0$ on $0$, and sends $1$ on $t$ and the right vertical morphism induced by the retraction $[X,1] \vee[1]\to [X,1]$.

Then $f$ is $\Sigma^m(\triangledown\star Y)$. 
\end{lemma}
\begin{proof}
All these categories admit loop free and atomic basis. We can then show this lemma in the category of augmented directed complexes. Furthermore, in this category, the suspension only makes an index shift, so we can assume without loss of generality that $m=0$.

The commutativity of the diagram implies that 
$$
\begin{array}{rclr}
f(0\star x)&=& 0\star x\\
f(1\star x)&=&t\star x\\
f([x,1] \star y)&=& [x,1] \star y +r_{x ,y} &\\
\end{array}
$$
where $r_{x,y}$ is a positive sum of elements of $(B_{[1]\star Y})_{|x|+|y|+1}$.
We show by induction on $|x|+|y|$ that: 
$$\begin{array}{rcll}
r_{x,y}&=& [1]\star y&\mbox{ if $|x|= 0$}\\
&=&0&\mbox{ if $|x|> 0$} .\\
\end{array}$$

Suppose the result true when the sum of dimensions of $x$ and $y$ is $(k-1)$. Let $x, y$ be two cells such that $|x|+|y|=k$.
\textbf{Case $|x|=0$.} The commutativity of $f$ with $\partial$ and the induction hypothesis imply that 
$$\begin{array}{rcl}
\partial r_{x,y} &=& f(\partial ([x,1]\star y)) - \partial ([x,1]\star y)\\
 &=& \{t\}\star y - \{0\}\star y + f([x,1]\star \partial y) - \{1\}\star y + \{0\}\star y - [x,1]\star \partial y\\
 &=& \{t\}\star y - \{1\}\star y + [1]\star \partial y\\
\end{array}$$
and $r_{x,y}$ is then equal to $[1]\star y$. \textbf{Case $|x|>0$.} The commutativity of $f$ with $\partial$ implies that 
$$ \partial r_{x,y} = 0$$
and $r_{x,y}$ is then equal to $0$.
\end{proof}

\begin{lemma}
\label{lemma:unicity_of_composition 2}
Let $m$ be an integer and $X$ and $Y$ be two $\zo$-categories admitting a loop free and atomic basis. We denote by $0$, $1$ and $t$ the three points of $\Sigma X\vee[1]$. 
 Let $$f: \Sigma^m ([X,1]\star Y)\to \Sigma^m( ([X,1]\vee[1])\star Y)$$ be a morphism fitting in the following diagram:
\[\begin{tikzcd}
	{\Sigma^m(\{t\}\star Y)} & {\Sigma^m(\{1\}\star Y)} \\
	& {\Sigma^m(([X,1]\vee[1])\star Y)} & {\Sigma^m([X,1]\star Y)} \\
	{\Sigma^m([X,1]\star Y)}
	\arrow["f", from=2-2, to=2-3]
	\arrow["id"', curve={height=12pt}, from=3-1, to=2-3]
	\arrow[hook, from=3-1, to=2-2]
	\arrow["\cong", from=1-1, to=1-2]
	\arrow[from=1-2, to=2-3]
	\arrow[from=1-1, to=2-2]
\end{tikzcd}\]
Then $f$ is the morphism induced by the retraction $[X,1]\vee[1]\to [X,1]$.
\end{lemma}
\begin{proof}
The proof is an easy computation using Steiner theory, similar to the one done in lemma \ref{lemma:unicity_of_composition}, and left to the reader.
\end{proof}

\p 
\label{para:def of C}
Let $C$ be the subcategory of marked simplicial sets whose
\begin{enumerate}
\item[$-$] objects are the marked simplicial sets $X$ such that $\R(X)$ has no non-trivial automorphisms, and such that there exists a (necessary unique) isomorphism $$\phi_X:\R(iX) \to\R(X),$$
\item[$-$] morphisms are the maps $f:X\to Y$ making the induced diagram
\[\begin{tikzcd}
	{\R(i(X))} & {\R(X)} \\
	{\R(i(Y))} & {\R(Y)}
	\arrow["{\phi_X}", from=1-1, to=1-2]
	\arrow["{\phi_Y}"', from=2-1, to=2-2]
	\arrow["{\R(f)}", from=1-2, to=2-2]
	\arrow["{\R(i(f))}"', from=1-1, to=2-1]
\end{tikzcd}\]
commutative.
\end{enumerate}

\begin{remark}
\label{rem:about_P_modified_3}
As $\R$ sends acyclic cofibrations to isomorphisms, $C$ is stable by zigzags of acyclic cofibrations. Moreover, as $\R$ and $i$ preserve colimits, for any diagram $F:I\to C$ such that the $\zo$-category $\R(\colim_IF)$ has no non-trivial automorphisms, $\colim_IF$ is in $C$. Eventually, the colimit of any natural transformation between two such diagrams is in $C$.
\end{remark}

\begin{lemma}
\label{lemma:about_P_modified_2}
Let $(k,n)$ be a couple of integers such that $k\leq n$.  We set the convention $[-1]:=\emptyset$.
For any integer $m$, the following assertion holds:
\begin{enumerate}
\item $\Sigma^m( \Sigma [n-k]_{\circ}\star[k-1])$ and $\Sigma^m( [k-1]_{\circ}\costar \Sigma [n-k])$ are in $C$.
\item For any $-1\leq l\leq k-1$ and $0\leq p\leq n-k$, and any monomorphisms $[l]\to [k-1]$ and $[p]\to [n-k]$, the morphisms
$$\Sigma^m( \Sigma [p]_{\circ}\star[l]) \to \Sigma^m( \Sigma [n-k]_{\circ}\star[k-1])~~\mbox{and}~~
\Sigma^m( [l]_{\circ} \costar \Sigma[p])\to \Sigma^m( [k-1]_{\circ} \costar \Sigma[n-k])$$
are in $C$.
\item For any $\epsilon\in \{0,1\}$, the morphisms
$$\Sigma^m( \{\epsilon\}\star[k-1]) \to \Sigma^m( \Sigma [n-k]_{\circ}\star[k-1])~~\mbox{and}~~
\Sigma^m( [k-1]_{\circ} \costar \{\epsilon\})\to \Sigma^m( [k-1]_{\circ} \costar \Sigma[n-k])$$
are in $C$.
\item If $k>0$, the morphisms
$$\Sigma^m( \emptyset\star[k-1]) \to \Sigma^m( \Sigma [n-k]_{\circ}\star[k-1])~~\mbox{and}~~
\Sigma^m( [k-1]_{\circ} \costar \emptyset)\to \Sigma^m( [k-1]_{\circ} \costar \Sigma[n-k])$$
are in $C$.
\end{enumerate}
\end{lemma}
\begin{proof}
We will proceed by induction on $(k,n)$.

- The case $(0,0)$ corresponds to the belonging of globes to $C$, which is true by the assumptions we made on the functor $i$ and by the proposition \ref{prop:the globes a non non trivial automorphisms} that assert that the globes have no non-trivial automorphisms.

- We now suppose that the case $(n-1,n-1)$ holds and we are willing to show the case $(0,n)$. The assertions $(1)$ and $(2)$ are direct consequences of the case $(n-1,n-1)$ after remarking the isomorphisms:
$$\Sigma^m \Sigma [n]\cong \Sigma^{m+1}((\Sigma[0]_{\circ}) \star[n-2])~~~~~
\Sigma^m \Sigma [n]_{\circ}\cong \Sigma^{m+1}([n-2]_{\circ}\costar(\Sigma[0]))
$$
It remains to show the third assertion.  Let $m$ be any integer and $\epsilon\in \{0,1\}$. 
By induction hypothesis and by the belonging of globes to $C$, the following morphism
$$\Sigma^m( \{\epsilon\})\to \Sigma^m(\Sigma \{0\})\cong \Sigma^{m+1}\{0\}\to \Sigma^{m+1}((\Sigma[0]_{\circ}) \star[n-2])\cong\Sigma^m \Sigma [n]$$
is in $C$. As the morphism $\Sigma^m( \{\epsilon\})\to \Sigma^m \Sigma [n]$ is their composite, it belongs to $C$. We proceed similarly to show that
$\Sigma^m( \{\epsilon\})\to \Sigma^m \Sigma [n]_{\circ}$ belongs to $C$. This concludes the proof of the case  $(0,n)$.

- Suppose the result true for the couples $(k-1,n)$, $(k-1,n-1)$ and $(k-1,k-1)$ for an integer $k$ strictly superior to $0$ and inferior or equal to $n$. We are willing to show the case $(k,n)$. Let $m$ be any integer.

As $R$ commutes with Gray operations and pushouts, the lemma \ref{lemma:non trivial automorphisme 4} implies that 
$\Sigma^m ( (\Sigma[n-k]_{\circ}\coprod_{[0]}[1])\star [k-2])$
together with all the objects appearing in the statement of this lemma are sent by $\R$ to $\zo$-categories with loop free and atomic basis and with no non-trivial automorphisms. Remark \ref{rem:about_P_modified_3} implies that for one of these objects (resp. a morphism between them) to belong to $C$, it is sufficient to show that it is linked by a zigzag of acyclic cofibrations to the colimit, computed in $\mSset$, of a diagram with value in $C$ (resp. in the arrow category of $C$).

As $\Sigma [0]_\circ=[1]$, the case $(k-1,k-1)$ implies that the morphism 
$$\Sigma^m(\{0\}\star [k-1])\to \Sigma^m([1]\star [k-1])$$
is in $C$. Combined with the case $(k-1,n-1)$, this implies that the diagram
\[\begin{tikzcd}
	{\Sigma^m ( (\Sigma[n-k]_{\circ})\star [k-2])} & {\Sigma^m ( (\Sigma[n-k]_{\circ})\star [k-2])} \\
	{\Sigma^m ( [0]\star [k-2])} & {\Sigma^m ( [0]\star [k-2])} \\
	{\Sigma^m ( [0]\star [k-2])} & {\Sigma^m ( [1]\star [k-2])}
	\arrow[from=2-2, to=1-2]
	\arrow[from=2-2, to=3-2]
	\arrow[from=1-1, to=1-2]
	\arrow["id"', from=2-1, to=3-1]
	\arrow[from=2-1, to=1-1]
	\arrow["id", from=2-1, to=2-2]
	\arrow[from=3-1, to=3-2]
\end{tikzcd}\]
is in $C$, and so is it's vertical colimits. 
As the codomain is weakly equivalent to $\Sigma^m ( (\Sigma[n-k]_{\circ}\fwedge[1])\star [k-2])$, this implies that $C$ includes the canonical morphism
\begin{equation}
\label{eq:eqlemmaunicity_of_composition1}
\Sigma^m ( (\Sigma[n-k]_{\circ})\star [k-2]) \hookrightarrow \Sigma^m ( (\Sigma[n-k]_{\circ}\fwedge[1])\star [k-2]).
\end{equation}
We can show similarly that the canonical morphism
\begin{equation}
\label{eq:eqlemmaunicity_of_composition1.5}
\Sigma^m ( [1]\star [k-2]) \hookrightarrow \Sigma^m ( (\Sigma[n-k]_{\circ}\fwedge[1])\star [k-2]).
\end{equation}
is in $C$.

The image by $\R$ of the canonical morphism 
$$\Sigma^m ( (\Sigma[n-k]_{\circ}\fwedge[1])\star [k-2])\to \Sigma^m ( (\Sigma[n-k]_{\circ})\star [k-2]) $$
induced by the retraction $\Sigma[n-k]_{\circ}\fwedge[1]\to \Sigma[n-k]_{\circ}$ fulfills the condition of lemma \ref{lemma:unicity_of_composition 2} and then belongs to $C$. 
The lemma \ref{lemma:unicity_of_composition} then implies that the morphism 
\begin{equation}
\label{eq:eqlemmaunicity_of_composition2}
\Sigma^m ( \triangledown\star [k-2]):\Sigma^m ( (\Sigma[n-k]_{\circ})\star [k-2])\to \Sigma^m ( (\Sigma[n-k]_{\circ}\fwedge[1])\star [k-2])
\end{equation}
is in $C$. 
We will use freely in the rest of the proof that morphisms \eqref{eq:eqlemmaunicity_of_composition1}, \eqref{eq:eqlemmaunicity_of_composition1.5} and \eqref{eq:eqlemmaunicity_of_composition2} are in $C$.
 
Theorem \ref{theo:cyl_formula} implies that the object $\Sigma^m( \Sigma [n-k]_{\circ}\star[k-1])$ is linked by a zigzag of acyclic cofibrations to the colimit of 
$$
\Sigma^m ( (\Sigma[n-k]_{\circ}\fwedge[1])\star [k-2]) \leftarrow
\Sigma^m ( \Sigma[n-k]_{\circ}\star [k-2]) \to
\Sigma^m ( \Sigma[n-k+1]_{\circ}\star [k-2])
$$
and the induction hypothesis implies that it belongs to $C$. 
We proceed similarly to show that $\Sigma^m( [k-1]_{\circ}\costar \Sigma [n-k])$ belongs to $C$.

Let $0\leq l\leq k-1$ and $-1\leq p\leq n-k$ be two integers, and $f:[l]\to [k-1]$ and $g:[p]\to [n-k]$ be two monomorphisms. Suppose first that $f$ is of shape $[0]\star f'$ for $f':[l-1]\to [k-2]$.
In this case, $\Sigma^m( \Sigma [p]_{\circ}\star[l]) \to \Sigma^m( \Sigma [n-k]_{\circ}\star[k-1])$ is linked by a zigzag of acyclic cofibrations to the vertical colimit of the diagram
\[\begin{tikzcd}
	{ \Sigma^m ( (\Sigma[p]_{\circ}\fwedge[1])\star [l-1])} & {\Sigma^m ( (\Sigma[n-k]_{\circ}\fwedge[1])\star [k-2])} \\
	{ \Sigma^m ( \Sigma[p]_{\circ}\star [l-1])} & { \Sigma^m ( \Sigma[n-k]_{\circ}\star [k-2])} \\
	{ \Sigma^m ( \Sigma[p+1]_{\circ}\star [l-1])} & { \Sigma^m ( \Sigma[n-k+1]_{\circ}\star [k-2])}
	\arrow[from=2-2, to=1-2]
	\arrow[from=2-2, to=3-2]
	\arrow[from=2-1, to=2-2]
	\arrow[from=1-1, to=1-2]
	\arrow[from=2-1, to=1-1]
	\arrow[from=2-1, to=3-1]
	\arrow[from=3-1, to=3-2]
\end{tikzcd}\]
and the induction hypothesis implies that it belongs to $C$. 
Suppose now that $f$ avoids the initial object of $[k-1]$. In this case, the morphism $\Sigma^m( \Sigma [p]_{\circ}\star[l]) \to \Sigma^m( \Sigma [n-k]_{\circ}\star[k-1])$
 is linked by a zigzag of acyclic cofibrations to the vertical colimit of the diagram
\[\begin{tikzcd}
	{\Sigma^m( \Sigma [p]_{\circ}\star[l])} & {\Sigma^m ( (\Sigma[n-k]_{\circ})\star [k-2])} & {\Sigma^m ( (\Sigma[n-k]_{\circ}\fwedge[1])\star [k-2])} \\
	&& { \Sigma^m ( \Sigma[n-k]_{\circ}\star [k-2])} \\
	&& { \Sigma^m ( \Sigma[n-k+1]_{\circ}\star [k-2])}
	\arrow[from=2-3, to=1-3]
	\arrow[from=2-3, to=3-3]
	\arrow[hook, from=1-2, to=1-3]
	\arrow[from=1-1, to=1-2]
\end{tikzcd}\]
and the induction hypothesis implies that it belongs to $C$.
We prove similarly that $$\Sigma^m( [l]_{\circ} \costar \Sigma[p])\to \Sigma^m( [k-1]_{\circ} \costar \Sigma[n-k])$$ belongs to $C$.

The morphism $\Sigma^m( \{0\}\star[k-1]) \to \Sigma^m( \Sigma [n-k]_{\circ}\star[k-1])$  is linked by a zigzag of acyclic cofibrations to the vertical colimit of the diagram
\[\begin{tikzcd}
	& {\Sigma^m ( (\Sigma[n-k]_{\circ}\fwedge[1])\star [k-2])} \\
	& { \Sigma^m ( \Sigma[n-k]_{\circ}\star [k-2])} \\
	{\Sigma^m( \{0\}\star[k-1])\cong \Sigma^m ( (\Sigma\{n-k+1\})\star [k-2])} & { \Sigma^m ( \Sigma[n-k+1]_{\circ}\star [k-2])}
	\arrow[from=2-2, to=1-2]
	\arrow[from=2-2, to=3-2]
	\arrow[from=3-1, to=3-2]
\end{tikzcd}\]
and the induction hypothesis implies that it belongs to $C$.  The morphism $\Sigma^m( \{1\}\star[k-1]) \to \Sigma^m( \Sigma [n-k]_{\circ}\star[k-1])$   
is linked by a zigzag of acyclic cofibrations to the vertical colimit of the diagram
\[\begin{tikzcd}
	{\Sigma^m( \{1\}\star[k-1])\cong \Sigma^m ( [1]\star [k-2])} & {\Sigma^m ( (\Sigma[n-k]_{\circ}\fwedge[1])\star [k-2])} \\
	& { \Sigma^m ( \Sigma[n-k]_{\circ}\star [k-2])} \\
	& { \Sigma^m ( \Sigma[n-k+1]_{\circ}\star [k-2])}
	\arrow[from=2-2, to=1-2]
	\arrow[from=2-2, to=3-2]
	\arrow[hook, from=1-1, to=1-2]
\end{tikzcd}\]
and the induction hypothesis implies that it belongs to $C$.
We prove similarly that for any $\epsilon\in \{0,1\}$, $$\Sigma^m( [k-1]_{\circ} \costar \{\epsilon\})\to \Sigma^m( [k-1]_{\circ} \costar \Sigma[n-k])$$ belongs to $C$.

Eventually, the morphism $\Sigma^m( \emptyset\star[k-1]) \to \Sigma^m( \Sigma [n-k]_{\circ}\star[k-1])$  is 
is linked by a zigzag of acyclic cofibrations to the vertical colimit of the diagram
\[\begin{tikzcd}
	{\Sigma^m( \{1\}\star[k-2])} & { \Sigma^m ( [1]\star [k-2])} & {\Sigma^m ( (\Sigma[n-k]_{\circ}\fwedge[1])\star [k-2])} \\
	&& { \Sigma^m ( \Sigma[n-k]_{\circ}\star [k-2])} \\
	&& { \Sigma^m ( \Sigma[n-k+1]_{\circ}\star [k-2])}
	\arrow[from=2-3, to=1-3]
	\arrow[from=2-3, to=3-3]
	\arrow[hook, from=1-2, to=1-3]
	\arrow[from=1-1, to=1-2]
\end{tikzcd}\]
and the induction hypothesis implies that it belongs to $C$. We prove similarly that $$\Sigma^m( [k-1]_{\circ} \costar \emptyset)\to \Sigma^m( [k-1]_{\circ} \costar \Sigma[n-k])$$ belongs to $C$.

We have then proven the case $(k,n)$, and this concludes the proof.
\end{proof}

\begin{lemma}
\label{lemma:about_P_modified}
Let $F:\Delta\to \zocat$ be a functor and $\phi:F\to \R$ be a invertible transformation such that for any monomorphism $i:[k]\to [n]$, the induced square
\[\begin{tikzcd}
	{F([k])} & {\R([k])} \\
	{F([n])} & {\R([n])}
	\arrow["{\phi_{[k]}}", from=1-1, to=1-2]
	\arrow["{\phi_{[n]}}"', from=2-1, to=2-2]
	\arrow["{R(i)}", from=1-2, to=2-2]
	\arrow["{F(i)}"', from=1-1, to=2-1]
\end{tikzcd}\]
commutes. Then $\phi$ is an invertible natural transformation between $F$ and $\R$.
\end{lemma}
\begin{proof}
We can suppose without loss of generality that for all integer $n$, $F([n])=\R([n])$. The hypotheses implies that for any monomorphism $i:[n]\to [m]$, $F(i)=\R(i)$ and it then remains to show that for any degeneracy $p:[n]\to [m]$, $F(p)=\R(p)$.

We proceed by induction and we then suppose that for any $0<k\leq n$ and any degeneracy $s:[k]\to [k-1]$, $F(s)=\R(s)$. As any  morphism of $\Delta$ factors as a degeneracy followed by a monomorphism, the induction hypothesis implies that for any $f:[k]\to [n]$ with $k\leq n$,  $F(f)=\R(f)$.

Let $s:[n+1]\to [n]$ be a degeneracy. 
We have a \textit{a priori} non commutative diagram:
\[\begin{tikzcd}[ampersand replacement=\&]
	{\colim_{[k]\underset{\neq id}{\hookrightarrow}[n+1] }\R([k])} \& {\colim_{[k]\underset{\neq id}{\hookrightarrow}[n+1] }\R([k])} \\
	{{\R}([n+1])} \& {{\R}([n+1])} \\
	{{\R}([n])} \& {{\R}([n])}
	\arrow["{F(s)}"', from=2-1, to=3-1]
	\arrow[from=1-1, to=2-1]
	\arrow[from=1-2, to=2-2]
	\arrow["{\R(s)}", from=2-2, to=3-2]
	\arrow[Rightarrow, no head, from=2-1, to=2-2]
	\arrow[Rightarrow, no head, from=1-1, to=1-2]
	\arrow[Rightarrow, no head, from=3-1, to=3-2]
\end{tikzcd}\]
The induction hypothesis implies that the outer and the upper square commute. As $R$ commutes with colimits, $\colim_{[k]\to\partial [n]}\R([k])$ is equivalent to $\R(\partial[n])$. Moreover, the inclusion $\R(\partial[n])\to \R([n])$ induces an isomorphisms on cells of dimension lower or equal to $n$. 
For the lower square to commutes, we then only have to check that the top cell of $\R([n+1])$ is sent on the same element on ${\R}([n])$. That is the case because the two paths send it to an unity as there is no non trivial $(n+1)$-cells in $\R([n])$. 

We then have $F(s)=\R(s)$, which concludes the induction and then the proof.
\end{proof}

\begin{prop}
\label{prop:existence_of_comparaison_with_street}
There exists an invertible natural transformation $\R i\to \R$.
\end{prop} 
\begin{proof}
As $\Sigma[0]_{\circ}$ is isomorphic to $[1]$, the case $(n,n)$ for any integer $n$ of the lemma \ref{lemma:about_P_modified_2} imply that there exists an invertible transformation $\phi:(\R i)_{|\Delta}\to \R_{|\Delta}$ which is natural when restricted to the full subcategory of $\Delta$ whose morphisms are the monomorphisms. 

The lemma \ref{lemma:about_P_modified} then implies that $\phi:(\R i)_{|\Delta}\to \R_{|\Delta}$ is natural. We can extend it to a natural transformation $\phi':(\R i)_{|t\Delta}\to \R_{|t\Delta}$ thanks to the proposition \ref{prop:modification_of_the_value_on_thin_representables}. 

Eventually, as both $\R i$ and $\R$ preserves colimits, we can extend $\phi'$ to a invertible natural transformation between $\R i$ and $\R$.
\end{proof}

\begin{theorem}
\label{theo:criterion_to_be_linked_to_identity}
Let $i: \mSset\to \mSset$ be a left Quillen functor. Suppose that there exists a zigzag of weakly invertible natural transformations:
$$i(\Db_{\uvar}) \leftrightsquigarrow \Db_{\uvar}.$$
Then, there exists a zigzag of weakly invertible natural transformations between $i$ and $id$. In particular, $i$ is a left Quillen equivalence.
\end{theorem} 
\begin{proof} 
The proposition \ref{prop:existence_of_comparaison_with_street} implies that we have a natural transformation $\psi:i\to i_{str}$. 
Furthermore, hypotheses imply that this natural transformation is a weak equivalence on globes. According to proposition \ref{prop:criterimu_to_be_an_weak_equivalence}, $\psi$ is then a weakly invertible natural transformation.
We then have a zigzag of weakly invertible natural transformations: 
$$i\xrightarrow{\sim} i_{str}\xleftarrow{\sim} id.$$
\end{proof}

\begin{cor}
\label{cor:criterion_to_be_linked_to_identity_case stratified}
Let $i: \stratSset\to \stratSset$ be a left Quillen functor. Suppose that there exists a zigzag of weakly invertible natural transformations:
$$i(\Db_{\uvar}) \leftrightsquigarrow \Db_{\uvar}.$$
Then, there exists a zigzag of weakly invertible natural transformations between $i$ and $id$. In particular, $i$ is a left Quillen equivalence.
\end{cor} 
\begin{proof} 
We recall that the adjunction between stratified and marked simplicial sets is denoted by:
\[\begin{tikzcd}
	{(\uvar)_{\mk}:\stratSset} & {\mSset:\iota}
	\arrow[""{name=0, anchor=center, inner sep=0}, shift left=2, from=1-1, to=1-2]
	\arrow[""{name=1, anchor=center, inner sep=0}, shift left=2, from=1-2, to=1-1]
	\arrow["\dashv"{anchor=center, rotate=-90}, draw=none, from=0, to=1]
\end{tikzcd}\]
The proposition \ref{prop:transfert from presheaves on tB to stratified presheaves} states that this adjonction is a Quillen equivalence and that the functor $\iota$ preserves acyclic cofibrations.

Remark now that the functor $(\uvar)_{\mk}\circ i\circ \iota:\mSset\to \mSset$ verifies the hypothesis of theorem \ref{theo:criterion_to_be_linked_to_identity} and we then have a zigzag of of weakly invertible natural transformations:
$$(\uvar)_{\mk}\circ i\circ \iota  \leftrightsquigarrow id$$
This induces a zigzag of of weakly invertible natural transformations:
$$i\to \iota\circ(\uvar)_{\mk}\circ i\circ \iota \circ (\uvar)_{\mk} \leftrightsquigarrow \iota\circ(\uvar)_{\mk}\leftarrow id$$
\end{proof}

%

\chapter{Complicial sets as a model of $\io$-categories}	
\label{chapter:complicial set as a model of io categories}

\minitoc
\vspace{1cm}

%
%
%
%
%
%
%
%
%
%
%

Let $n\in \Nb\cup\{\omega\}$.
Following the terminology of Barwick and Schommer-Pries (\cite{Barwick_on_the_unicity_of_the_theory_of_higher_categories}), we call \textit{model of $(\infty,n)$-categories} any model category whose corresponding $(\infty, 1)$-category is $\ncat{n}$.

With the definition of $(\infty,n)$-categories given in the introduction, we have a natural model for the $\iun$-category $\ncat{n}$, given by Rezk's complete Segal $\Theta_n$-spaces, i.e. space valued presheaves on $\Theta_n$ satisfying the (homotopical) Segal conditions and (homotopical) completeness conditions. However, there are many other models, see for instance \cite{Ara_Higher_quasi_cat}, \cite{Bergner_Comparison_of_model_of_infini_n_categories}, \cite{Bergner_Comparison_of_model_for_infini_n_categories_II}, \cite{Bergner_reedy_category_and_the_theta_construction} (we refer to  \cite{Barwick_on_the_unicity_of_the_theory_of_higher_categories}
for a comprehensive presentation of these models and their equivalence). For example, one can mention $n$-fold Segal spaces and Simpson's and Tamsamani's Segal $n$-categories among others.

It was conjectured (\cite{Street_algebra_of_orianted_simplexes}, \cite{Verity_a_complicial_compendium}, \cite{Barwick_on_the_unicity_of_the_theory_of_higher_categories}) that Verity's $n$-complicial sets were also a model of $(\infty,n)$-categories. This would imply that Campion-Kapulkin-Maehara's $n$-comical sets also are, as they are shown to be Quillen equivalent to $n$-complicial sets in \cite{Doherty_Equivalence_of_cubical_and_simplicial_approaches}.

Results of Bergner, Gagna, Harpaz, Joyal, Lanari, Lurie, Rezk and Tierney (\cite{Bergner_Comparison_of_model_of_infini_n_categories},\cite{Bergner_Comparison_of_model_for_infini_n_categories_II}, \cite{Rezk_a_cartesian_of_weak_n_categories}, \cite{Lurie_Htt},\cite{Lurie_goodwillie_calculus}, \cite{Gagna_on_the_equivallence_of_all_model_for_infini2_cat}, \cite{Joyal_Quasi-categories_vs_Segal_spaces}) imply that $2$-complicial sets are a model of $(\infty,2)$-categories (see \cite{Gagna_on_the_equivallence_of_all_model_for_infini2_cat} to understand how to use all this source to obtained the desired result and \cite{Bergner_explicit_comparaison_bt_theta_2_space_and_2_complicial_set} for a direct comparison between complete Segal $\Theta_2$-spaces and $2$-complicial sets).
The purpose of this chapter is to generalize this result to any $n\in \Nb\cup\{\omega\}$.

To this extend, we first address the more general problem of finding sufficient conditions on a model category $A$ to build a \textit{Gray cylinder} $C\mapsto I\otimes C$ and a \textit{Gray cone} $C\mapsto e\star C$ on Segal precategories enriched in $A$. These two operations should be linked by the following homotopy cocartesian square
\[\begin{tikzcd}
	{\{0\}\otimes C} & {I\otimes C} \\
	e & {e\star C}
	\arrow[from=1-2, to=2-2]
	\arrow[from=1-1, to=2-1]
	\arrow[from=2-1, to=2-2]
	\arrow[from=1-1, to=1-2]
\end{tikzcd}\]
where $e$ is the terminal object. The conditions that $A$ has to	 fulfill are encapsulated in the notion of \textit{Gray module} (paragraph \ref{para:Gray module}). Thanks to the Gray cylinder and cone, we can show the following theorem:

\begin{itheorem}[\ref{theo:Quillen adjunction}]
If $A$ is a Gray module, there is a Quillen adjunction between the Ozornova-Rovelli model structure for $\omega$-complicial sets on stratified simplicial sets and stratified Segal precategories enriched in $A$ where the left adjoint sends $[n]$ to $e\star e\star ... \star e\star \emptyset$
\end{itheorem} 

We will apply this theorem to the case where $A$ is the category of stratified simplicial sets endowed with the model structure for $\omega$-complicial sets, and after tedious work, we get
\begin{itheorem}[\ref{theo:letheo}]
Let $n\in \Nb$.
The model structure for $n$-complicial sets is a model of $(\infty,n)$-categories.
\end{itheorem}
As a corollary we have
\begin{itheorem}[\ref{theo:lecorozo}]
The adjunction between the model structure for complete Segal $\Theta$-spaces and $\omega$-complicial set constructed in \cite{Ozornova_a_quillen_adjunction_between_globular_and_complicial} is a Quillen equivalence.
\end{itheorem}

\section{Preliminaries}

\subsection{Segal $A$-precategories}
Let $A$ be a category of stratified presheaves on a elegant Reedy category (as defined in paragraph \ref{para:reedy} and section \ref{section:Marked and stratified presheaves}), endowed with a nice model structure (as defined in paragraph \ref{para:nice model structure}). We suppose furthermore that the terminal element of $A$, denoted by $e$, is representable. We then have an adjunction 
\begin{equation}
\label{eq:ob adj}
\begin{tikzcd}
	{\iota:\Set} & {A:ob}
	\arrow[""{name=0, anchor=center, inner sep=0}, shift left=2, from=1-1, to=1-2]
	\arrow[""{name=1, anchor=center, inner sep=0}, shift left=2, from=1-2, to=1-1]
	\arrow["\dashv"{anchor=center, rotate=-90}, draw=none, from=0, to=1]
\end{tikzcd}
\end{equation}
where the left adjoint sends a set $S$ onto $\coprod_S e$ and the right adjoint is the evaluation at $e$.
The objects lying in the image of $\iota$ are called \notion{discrete objects}.

An object $C$ of $\Fun( \Delta^{op},A)$ is a \notion{Segal $A$-precatagory} if $C_0$ is discrete. We denote by \wcnotation{$\Seg(A)$}{(seg@$\Seg(A)$} the full subcategory of $\Fun(\Delta^{op},A)$ spanned by the Segal $A$-precategories.

\p We consider the functor $A\times \Delta \to \Fun(\Delta^{op},A)$ defined by the assignation $a\times [n]\to |[a,n]|$ where $|[a,n]|([m]):=a\times \iota (\Hom_\Delta([m],[n]))$. We define the Segal $A$-precategory \wcsnotation{$[a,n]$}{((g10@$[a,n]$}{for $A$-Segal precategories} as the pushout: 
\[\begin{tikzcd}
	{\underset{k\leq n}{\cup}{|[a,\{k\}]|}} & {|[a,n]|} \\
	{|[e,0]|} & {[a,n]}
	\arrow[from=1-1, to=1-2]
	\arrow[from=1-1, to=2-1]
	\arrow[from=2-1, to=2-2]
	\arrow[from=1-2, to=2-2]
	\arrow["\lrcorner"{anchor=center, pos=0.125, rotate=180}, draw=none, from=2-2, to=1-1]
\end{tikzcd}\]
The object $[e,0]$ is simply denoted by $[0]$. Remark that this object is the terminal Segal $A$-precategory.

The assignation $(a,n)\mapsto [a,n]$ induces by left Kan extension a colimit preserving functor 
$$[\uvar,\uvar]:A\times \Sset \to \Seg(A).$$
The image of this functor is dense in $\Seg(A)$.

For $\{n_i\}_{i\leq k}$ and $\{a\to a_i\}_{i\leq k}$ two finite sequences, we denote by \wcsnotation{$[a_0,n_0]\vee[a_1,n_1]\vee...\vee [a_k,n_k]$}{((g20@$[a_0,n_0]\vee[a_1,n_1]\vee...\vee [a_k,n_k]$}{for Segal $A$-precategories} the Segal $A$-precategory fitting in the following pushout:
\[\begin{tikzcd}
	{\amalg_{i\leq k}[a,n_i]} & {[a,\Sigma_{i\leq k}n_i]} \\
	{\amalg_{i\leq k}[a_i,n_i]} & {[a_0,n_0]\vee[a_1,n_1]\vee...[a_k,n_k]}
	\arrow[""{name=0, anchor=center, inner sep=0}, from=1-1, to=1-2]
	\arrow[from=1-2, to=2-2]
	\arrow[from=1-1, to=2-1]
	\arrow[from=2-1, to=2-2]
	\arrow["\lrcorner"{anchor=center, pos=0.125, rotate=180}, draw=none, from=2-2, to=0]
\end{tikzcd}\]

The case we will use the most is the one of the Segal $A$-precategories $[e,1]\vee[a,n]$ and $[a,n]\vee[e,1]$ corresponding to the sequence $((1,n),(a\to e,a\to a))$ and $((n,1),(a\to a,a\to e))$.

\p
\label{para:defi of delta[B]}
Let $B$ be the Reedy category and $M$ the subset of objects of $B$ such that $A$ is the category of $M$-stratified presheaves on $B$. We define the category $\Delta[B]$ as the fully faithful subcategory of $\Seg(A)$ whose objects are of shape $[b,n]$ for $b\in B$ and $n$ an integer. Eventually, we define $\Delta[M]$ as the set of objects of shape $[b,n]$ for $b\in M$ and $n>0$. 
We can easily check that the category $\Seg(A)$ is the category of $\Delta[M]$-stratified presheaves on $\Delta[B]$.

A cellular model for $\stratSeg(A)$ is given by the set of morphisms $[b,\partial n]\cup [a,n]\to [b,n]$ for $n$ an integer, and $a\to b$ a generating cofibration of $A$.

Eventually, for any Segal $A$-precategory $C$, we have an isomorphism $$C\cong \colim_{\Delta[tB]_{/C}}[b,n].$$

Following the definition of section \ref{section:Marked and stratified presheaves}, a morphism between Segal precategories is \textit{entire} if it is the identity on the underlying $\Delta[B]$-presheaves.

\begin{prop}
\label{prop:delta[B] is reedy}
The category $\Delta[B]$ as a structure of elegant Reedy category.
\end{prop}
\begin{proof}
Remark first that $\Hom_{\Delta[B]}([a,n],[b,m])$ fits in the following cocartesian square:
\[\begin{tikzcd}
	{\coprod_{k\leq m}\Hom_{B}(a,b)\times \Hom_{\Delta}([n],\{k\})} & {\Hom_{B}(a,b)\times \Hom_{\Delta}([n],[m])} \\
	{\coprod_{k\leq m} \Hom_{\Delta}([n],\{k\})} & {\Hom_{\Delta[B]}([a,n],[b,m])}
	\arrow[from=1-2, to=2-2]
	\arrow[from=1-1, to=2-1]
	\arrow[from=2-1, to=2-2]
	\arrow[from=1-1, to=1-2]
\end{tikzcd}\]
We then define the degree functor $ob(\Delta[B])\to \Nb$ by the formula $d([b,n])=d(b)d(n)$. The subcategory $(\Delta[B])_{+}$ is the image of $\Delta_+\times B_+$, and the subcategory $(\Delta[B])_{-}$ is the image of $\Delta_-\times B_-$.

We recall that we suppose that the Reedy category $B$ is elegant. 
Let $X$ be a presheaf on $\Delta[B]$, $[a,n]$ an element of $\Delta[A]$, $[f,g]:[a,n]\to [a',n']$ and $[h,i]:[a,n]\to [a',n']$ two negative morphisms, an element $x$ of $X([a,n])$, two non degenerate elements $y\in X([a',n'])$ and $z\in X([a'',n''])$ such that $[f,g]^*y=x$, $[h,i]^*z=x$.

We suppose first that $n\neq 0$.
 We denote $\pi:B\times \Delta\to \Delta[B]$ the canonical projection and 
$$\pi^*:\Psh{\Delta[B]}\to \Psh{\Delta\times B}$$ the functor obtained by precomposing.
Remark that for any $a,n$, $(\pi^*X)(a,n)=X([a,n])$.
Furthermore, we  have again equalities $(f,g)^*y=x$, $(h,i)^*z=x$.
As $\Delta\times B$ is Reedy elegant, this implies that $f=h$, $g=i$ and $y=z$. 

If $n=0$, then $[f,g]$ and $[h,i]$ are the identity, and we directly have $y=z$. 
The Reedy category $\Delta[B]$ is then elegant.
\end{proof}

\begin{definition}
\label{defi:generating_acyclic_cofibration}
We define the simplicial set \wcnotation{$E^{\cong}$}{(econg@$E^{\cong}$} as the colimit of the diagram:
$$[e,0]\leftarrow [e,1]\xrightarrow{[e,d^1d^3]} [e,3]\xleftarrow{[e,d^0d^2]} [e,1]\to [e,0].$$
An \snotion{elementary anodyne extension}{for Segal $A$-precategory} is one of the following:
\begin{enumerate}
 \item The \notion{generating Reedy cofibrations}:
 $$[a,n]\cup [b,\partial[n]]\to [b,n],~~\mbox{for $a\to b$ a generating acyclic cofibration of A.}$$
\item The \notion{Segal extensions}: 
$$[a,1]\cup[a,1]\cup ...\cup [a,1]\to [a,n],~~\mbox{for $a$ an object of $A$ and $n>0$.}$$
\item The \notion{completeness extensions}: 
$$\{0\}\to E^{\cong}.$$
\end{enumerate}
\end{definition}

\p 
\label{para:def sega a cat}
A \notion{Segal $A$-category} is a Segal $A$-precategory having the right lifting property against all elementary anodyne extensions.

Let $C$ be a Segal $A$-categories. We define the presheaf $ho(C):\Delta^{op}\to \textbf{Set}$ sending $[n]$ to $\Hom_{ho(A)}(e,C_n)$. As explained in \cite[$\S $ 14.5]{Simpson_Homotopy_theory_of_higher_categories}, this simplicial set has the unique right lifting property against Segal's maps, and is then the nerve of a category that we also note by $ho(C)$.
An arrow $x:[e,1]\to C$ is an \wcnotion{isomorphism}{isomorphism for an arrow $x:[e,1]\to C$} if its image in $ho(C)$ is. 

We can give an other characterization of isomorphisms in Segal $A$-categories. An arrow $x:[e,1]\to C$ is an isomorphism if and only if there exists a lifting in the following diagram: 
\[\begin{tikzcd}
	{[e,1]} & C \\
	{E^{\cong}}
	\arrow["x", from=1-1, to=1-2]
	\arrow[from=1-1, to=2-1]
	\arrow[dashed, from=2-1, to=1-2]
\end{tikzcd}\]

 A morphism $f:C\to D$ between Segal $A$-categories is an \textit{equivalence of Segal $A$-categories} if $C_1\to D_1$ is a weak equivalence in $A$, and for any element $x\in ob(D)$, there exists $y\in ob(C)$ and an isomorphism in $D$ between $f(y)$ and $x$.

\begin{theorem}[{\cite[21.2.1]{Simpson_Homotopy_theory_of_higher_categories}}]
\label{theo:carlos}
There exists a nice model structure on $\Seg(A)$ where fibrant objects are Segal $A$-categories and weak equivalences between Segal $A$-categories are equivalences of Segal $A$-categories. 

A left adjoint from $\Seg(A)$ to a model category $C$ is a left Quillen functor if it preserves cofibrations, and sends elementary anodyne extensions to weak equivalences.
\end{theorem}

\begin{prop}
Any Segal $A$-precategory is a homotopy colimit of objects of shape $[a,n]$.
\end{prop}
\begin{proof}
Let $C$ be a Segal $A$-precategory. We have $C\cong \colim_{\Delta[tB]_{/C}}\uvar.$
The result then follows from propositions \ref{prop:elelangat stable by slice}, \ref{prop:reedy structure on tB} and \ref{prop:delta[B] is reedy}.
\end{proof}

\subsection{Stratified Segal $A$-precategories}

\p A \notion{stratified Segal $A$-precatagory} is a pair $(C,tC)$ where $tC$ is a subset of $ob(C_1)$ that factors $s^0: C_0\to ob(C_1)$. A \textit{morphism of stratified Segal $A$-precatagory} $(C,tC)\to (D,tD)$ is the data of a morphism $f:C\to D$ such that $f(tC)\subset tD$.
The category of stratified Segal $A$-precategories is denoted by \wcnotation{$\stratSeg(A)$}{(tseg@$\stratSeg(A)$}.

We have an adjunction \ssym{((b30@$(\uvar)^\natural$}{for stratified Segal $A$-precategories}
\begin{equation}
\label{eq:adj u flat}
\begin{tikzcd}
	{(\uvar)^\flat:\Seg(A)} & {\stratSeg(A):(\uvar)^\natural}
	\arrow[""{name=0, anchor=center, inner sep=0}, shift left=2, from=1-2, to=1-1]
	\arrow[""{name=1, anchor=center, inner sep=0}, shift left=2, from=1-1, to=1-2]
	\arrow["\dashv"{anchor=center, rotate=-90}, draw=none, from=1, to=0]
\end{tikzcd}
\end{equation}
where the left adjoint is a fully faithful inclusion that sends $C$ to $C^\flat:=(C,Im(s^0))$. The right adjoint is the obvious forgetful functor.
We will identify Segal $A$-precategories with their images in stratified Segal $A$-precategories under the left adjoint.

\p
We define $[e,1]_t:= ([e,1],[e,1]_1)$. The subcategory of objects of shape $[a,n]$ or $[e,1]_t$ is then dense in $\stratSeg(A)$.

Let $B$ be the Reedy category and $M$ the subset of objects of $B$ such that $A$ is the category of $M$-stratified presheaves on $B$. We recall that we defined the category $\Delta[B]$ and the set of morphism $\Delta[M]$ in paragraph \ref{para:defi of delta[B]}. We set $t\Delta[M]$ as the reunion of $\Delta[M]$ and the singleton $\{[e,1]_t\}$.
We can easily check that the category $\stratSeg(A)$ is the category of $t\Delta[M]$-stratified presheaves on $\Delta[B]$. The set of generating cofibrations for $\stratSeg(A)$ then consists of morphisms of shape $[e,1]\to [e,1]_t$ or $[a,n]\cup[b,\partial n]\to [b,n]$ where $a\to b $ is a generating cofibration of $A$. \sym{((g21@$[e,1]_t$}	
For any stratified Segal $A$-precategory $C$, we then have an isomorphism 
$$C\cong \colim_{t\Delta[tB]_{/C}}\uvar.$$
where $t\Delta[tB]$ is the full subcategory of $\stratSeg(A)$ whose objects are of in $\Delta[B]$ or $t\Delta[M]$.

Following the definition of section \ref{section:Marked and stratified presheaves}, a morphism between stratified Segal precategories is \textit{entire} if it is the identity on the underlying $\Delta[B]$-presheaves.

\p A \notion{marked Segal $A$-category} is a pair $(C,C^{\cong})$ where $C$ is a Segal $A$-category and $C^{\cong}$ is the subset of $ob(C_1)$ consisting of all isomorphisms. A morphism $f:(C,C^{\cong})\to (D,D^{\cong})$ between marked Segal $A$-categories is an \notion{equivalence of marked Segal $A$-categories} if $C_1\to D_1$ is a weak equivalence in $A$, and for any element $x\in ob(D)$, there exists $y\in ob(C)$ and $v:f(y)\to x\in D^{\cong}$.

\p We are now willing to endow $\stratSeg(A)$ with a nice model structure whose fibrant objects are marked Segal $A$-category and weak equivalences between fibrant objects are equivalences of marked Segal $A$-categories. 
We define the stratified Segal $A$-precategories $(E^{\cong})'$ as the following pushout: 
\[\begin{tikzcd}
	{[e,1]} & {E^{\cong}} \\
	{[e,1]_t} & {(E^{\cong})'}
	\arrow["{d^0d^3}", from=1-1, to=1-2]
	\arrow[from=2-1, to=2-2]
	\arrow[from=1-1, to=2-1]
	\arrow[from=1-2, to=2-2]
	\arrow["\lrcorner"{anchor=center, pos=0.125, rotate=180}, draw=none, from=2-2, to=1-1]
\end{tikzcd}\]
We define the set of map $J$ as the reunion of the set of generating acyclic cofibration of $\Seg(A)$ and of $\{[e,1]_t\to (E^{\cong})'\}$ and $\{E^{\cong}\to (E^{\cong})'\}$. We suppose furthermore that $J$ includes the acyclic cofibrations $\{0\}\to E^{\cong}$ and $\{1\}\to E^{\cong}$.

\begin{lemma}
\label{lemma:fib a marked segal cat}
A morphism $f$ has the right lifting property against $J$ if and only if $f^\natural$ is a fibration and  $f$ has the right lifting property against $[e,1]_t\to (E^{\cong})'$ and $E^{\cong}\to (E^{\cong})'$. An object $X$ has the right lifting property against $J$ if and only if it is a marked Segal $A$-category.
\end{lemma}
\begin{proof}
Straightforward.
\end{proof}

\begin{lemma}
\label{lemma:invertivble natural trans are detect pointwise}
Let $i:K\to L$ be a cofibration that induces an isomorphism on objects. 
The morphism 
$$K\times E^{\cong}\coprod_{K\times [e,1]}L\times [e,1]\to L\times E^{\cong}$$
is an acyclic cofibration of the model strucure on $\Seg(A)$.
\end{lemma}
\begin{proof}
By two out of three, and some diagram chasing, is it sufficent to demonstrate the result for $K$ being $L_0$. We then have to show that the square
\[\begin{tikzcd}
	{L_0\times [e,1]} & {L\times [e,1]} \\
	{L_0\times E^{\cong}} & {L\times E^{\cong}}
	\arrow[from=2-1, to=2-2]
	\arrow[from=1-1, to=2-1]
	\arrow[from=1-2, to=2-2]
	\arrow[from=1-1, to=1-2]
\end{tikzcd}\]
is homotopy cocoartesian. As the model structure is cartesian, and as $E^{\cong}\to 1$ is a weak equivalence, this is suffisent to show that the following square is homotopy cocartesian:
\[\begin{tikzcd}
	{L_0\times [e,1]} & {L\times [e,1]} \\
	{L_0} & L
	\arrow[from=2-1, to=2-2]
	\arrow[from=1-1, to=2-1]
	\arrow[from=1-2, to=2-2]
	\arrow[from=1-1, to=1-2]
\end{tikzcd}\]
 As $\uvar\times [e,1]$ and $\uvar\times E^{\cong}$ are left Quillen functors, we can reduce to the case where $L$ is $[a,n]$ and using Segal extension, to the case where $L$ is $[a,1]$. We then have to show that the following square is homotopy cocartesian
\begin{equation}
\label{eq:cartesian square9870}
\begin{tikzcd}
	{(\{0\}\cup\{1\})\times [e,1]} & {[a,1]\times [e,1]} \\
	{\{0\}\cup\{1\}} & {[a,1]}
	\arrow[from=2-1, to=2-2]
	\arrow[from=1-1, to=2-1]
	\arrow[from=1-2, to=2-2]
	\arrow[from=1-1, to=1-2]
\end{tikzcd}
\end{equation}
 Remark then that $[a,1]\times[e,1]$ is the colimit of the following span:
\[\begin{tikzcd}
	{[e,1]\vee[a,1]} & {[a,1]} & {[a,1]\vee[e,1]}
	\arrow["{[a,d^1]}"', from=1-2, to=1-1]
	\arrow["{[a,d^1]}", from=1-2, to=1-3]
\end{tikzcd}\]
The pushout of the span of \eqref{eq:cartesian square9870} is then the (homotopy) colimit of 
\[\begin{tikzcd}
	{[0]\coprod\limits_{[e,1]}[e,1]\vee[a,1]} & {[a,1]} & {[a,1]\vee[e,1]\coprod\limits_{[e,1]}[0]}
	\arrow["{[a,d^1]}"', from=1-2, to=1-1]
	\arrow["{[a,d^1]}", from=1-2, to=1-3]
\end{tikzcd}\]
By two out of three, and using Segal extensions, the two morphisms 
$$[0]\coprod_{[e,1]}[e,1]\vee[a,1]\to [a,1] ~~~~~\mbox{ and }~~~~~[a,1]\vee[e,1]\coprod_{[e,1]}[0]\to [a,1]$$ induced by $[a,d^0]$ and $[a,d^2]$ are weak equivalences.
In particular, this implies that the canonical morphism from the pushout of the span of \eqref{eq:cartesian square9870} to $[a,1]$ is a weak equivalence. As the upper horizontal vertical morphisms of \eqref{eq:cartesian square9870} is a cofibration, this implies that this square is homotopy cocartesian which concludes the proof.
\end{proof}

\begin{lemma}
\label{lemma:j stable by leibtnis}
Let $i:K\to L$ be a monomorphism and $f:X\to Y$ a morphism having the right lifting property against $J$. The induced morphism 
$$f^i:X^{L}\to X^{K}\times_{Y^{K}} Y^{L}$$ has the right lifting property against $J$.
\end{lemma}
\begin{proof}
As the model structure on $\Seg(A)$ is cartesian, $(f^i)^\natural$ is a fibration. We then have to show that this morphism has the right lifting property against $[e,1]_t\to (E^{\cong})'$ and $E^{\cong}\to (E^{\cong})'$. We can reduce to the case where $i$ is a generating acyclic cofibration. If $i$ is $\emptyset\to [0]$, this is obvious. We then suppose that $i$ is $[e,1]\to [e,1]_t$ or $[a,\partial n]\cup [b,n]\to [b,n]$ for $a\to b$ a generating acyclic cofibration of $A$. In both case, $i$ induces an equivalence on objects. The morphism $i\hat{\times}(E^{\cong}\to (E^{\cong})')$ is then the identity. Moreover, $i\hat{\times}([e,1]_t\to (E^{\cong})')$ fits in the following cocartesian square
\[\begin{tikzcd}
	{L^\natural\times [e,1]\coprod_{K^\natural\times [e,1]}K^\natural\times (E^{\cong})} & {L\times [e,1]_t\coprod_{K\times [e,1]_t}K\times (E^{\cong})'} \\
	{L^\natural\times E^{\cong}} & {L\times (E^{\cong})'}
	\arrow[from=1-1, to=2-1]
	\arrow[from=1-1, to=1-2]
	\arrow[from=2-1, to=2-2]
	\arrow[from=1-2, to=2-2]
\end{tikzcd}\]
The lemma \ref{lemma:invertivble natural trans are detect pointwise} implies $f$ has the right lifting property against the left vertical morphism, and so also against the right vertical one.
 By adjunction, this implies that $f^i$ has the desired lifting property.
\end{proof}

\begin{prop}
\label{prop:model structure on stratified Segal category}
There exists a nice model structure on $\stratSeg(A)$ where fibrant objects are stratified Segal $A$-categories and weak equivalences between marked Segal $A$-categories are stratified equivalences. The adjunction 
\[\begin{tikzcd}
	{(\uvar)^\flat:\Seg(A)} & {\stratSeg(A):(\uvar)^\natural}
	\arrow[""{name=0, anchor=center, inner sep=0}, shift left=2, from=1-2, to=1-1]
	\arrow[""{name=1, anchor=center, inner sep=0}, shift left=2, from=1-1, to=1-2]
	\arrow["\dashv"{anchor=center, rotate=-90}, draw=none, from=1, to=0]
\end{tikzcd}\]
induces a Quillen equivalence.

 A left adjoint from $\stratSeg(A)$ to a model category $C$ is a left Quillen functor if it preserves cofibrations, and sends elementary anodyne extensions and morphisms $[e,1]_t\to 1$, $E^{\cong}\to (E^{\cong})'$ to weak equivalences.
\end{prop}
\begin{proof}
We recall that we define $J$ as the reunion of the set of generating acyclic cofibrations of $\Seg(A)$ and of $\{[e,1]_t\to (E^{\cong})'\}$ and $\{E^{\cong}\to (E^{\cong})'\}$ and we suppose that it includes the trivial cofibrations $\{0\}\to E^{\cong}$ and $\{1\}\to E^{\cong}$. We denote $I$ a cellular model for $\Psh{t\Delta[tB]}$.

As $\stratSeg(A)$ is the category of $t\Delta[M]$ stratified presheaves on $\Delta[B]$, we have an adjunction
\[\begin{tikzcd}
	{\pi:\Psh{t\Delta[tB]}} & {\stratSeg(A):\iota}
	\arrow[""{name=0, anchor=center, inner sep=0}, shift left=2, from=1-1, to=1-2]
	\arrow[""{name=1, anchor=center, inner sep=0}, shift left=2, from=1-2, to=1-1]
	\arrow["\dashv"{anchor=center, rotate=-90}, draw=none, from=0, to=1]
\end{tikzcd}\]
where the right adjoint is fully faithfull.

The set $l(r(\iota(\J)\hat{\times}I))$ is a class of anodyne extension relative to the interval $\uvar\times E^{\cong}$ as defined in \cite[paragraph 1.3.12]{cisinski_prefaisceaux_comme_modele}. We then consider $\Psh{t\Delta[tB]}$ endowed with the model structure induced by \cite[théorème 1.3.22]{cisinski_prefaisceaux_comme_modele}. An object is fibrant if and only if it has the right lifting property against $\iota(\J)\hat{\times}I$. A morphism between fibrant objects is a fibration if and only if it has the right lifting property against $\iota(\J)\hat{\times}I$.

According to proposition \ref{prop:transfert from presheaves on tB to stratified presheaves}, this induces a model structure on $\stratSeg(A)$. By adjunction and using lemma \ref{lemma:j stable by leibtnis}, an object is fibrant if and only if it has the right lifting property against $J$ and a morphism between fibrant objects is a fibration if and only if it has the right lifting property against $J$. According to lemma \ref{lemma:fib a marked segal cat}, the fibrant objects correspond to marked Segal $A$-categories.

The theorem \ref{theo:carlos} implies that the adjunction \eqref{eq:adj u flat} is a Quillen adjunction.
It's unit is the identity, and lemma \ref{lemma:fib a marked segal cat} implies that the counit, computed on a fibrant object $(C,C^{\cong})$, is the canonical inclusion $(C,C^\flat)\to (C,C^{\cong})$. As this morphism is a transfinite composite of $E^{\cong}\to (E^{\cong})'$, it is a weak equivalence. The Quillen pair \ref{lemma:fib a marked segal cat} is then a Quillen equivalence.
As a consequence, the model structure on $\stratSeg(A)$ is cartesian and simplicial, and weak equivalences between fibrant objects are stratified equivalences.

It then remains to prove the last assertion. Suppose given a left adjoint $F:\stratSeg(A)\to C$ that preserves cofibrations, and sends elementary anodyne extensions and morphisms $[e,1]_t\to 1$, $E^{\cong}\to (E^{\cong})'$ to weak equivalences. The theorem \ref{theo:carlos} implies that the restriction of $F$ to $\Seg(A)$ is a left Quillen functor, and this functors then sends any acyclic cofibration of $\Seg(A)$ to a weak equivalence. 
As we have a commutative diagram,
\[\begin{tikzcd}
	{(E^{\cong})'} & {[1]_t} \\
	{E^{\cong}} & {[0]}
	\arrow[from=2-1, to=1-1]
	\arrow[from=1-2, to=1-1]
	\arrow[from=2-1, to=2-2]
	\arrow[from=1-2, to=2-2]
	\arrow[from=1-1, to=2-2]
\end{tikzcd}\]
we deduce by two out of three that $F$ sends $[1]_t\to (E^{\cong})'$ to a weak equivalence. The functor $F$ then sends any morphism of $J$ to a weak equivalence. 

As fibrant objects and fibrations between fibrant objects are detected by right lifting property against $J$, the right adjoint of $F$ preserves them. The corollary A.2 of \cite{Dugger_Replacing_model_categories_with_simplicial_ones} implies that $F$ is a left Quillen functor.
\end{proof}

\begin{prop}
Any stratified Segal $A$-precategory is a homotopy colimit of objects of shape $[a,n]$ or $[e,1]_t$.
\end{prop}
\begin{proof}
Let $C$ be a stratified Segal $A$-precategory. We have 
$C\cong \colim_{t\Delta[tB]_{/C}}\uvar.$
The result then follows from propositions \ref{prop:elelangat stable by slice}, \ref{prop:reedy structure on tB} and \ref{prop:delta[B] is reedy}.
\end{proof}

\p
We now present the main way of constructing functors whose codomain is $\stratSeg(A)$. 
\begin{construction}
\label{cons:lifting of a functor from A times Delta}
Suppose given a colimit preserving functor $G:A\times \Delta\to D$ in a complete category, an object $G(e,1)'$ and a morphism $p:G(e,1)\to G(e,1)'$ such that for any object $d$ of $D$, $\Hom(p,d)$ is a monomorphism. We define the functor 
$\overline{G}:\stratSeg(A)\to D$
as the unique colimit preserving functor such that $\overline{G}([e,1]_t):= G(e,1)'$ and for any $a,n$, $\overline{G}([a,n])$ fits in the following cocartesian square: 
\[\begin{tikzcd}
	{\coprod_{i\in[n]} G(a,\{i\})} & {G(a,[n])} \\
	{\coprod_{i\in[n]} G(e,\{i\})} & {\overline{G}([a,n])}
	\arrow[from=1-1, to=1-2]
	\arrow[from=1-1, to=2-1]
	\arrow[from=2-1, to=2-2]
	\arrow[from=1-2, to=2-2]
\end{tikzcd}\]
Remark that if the top horizontal morphism is a cofibration, the previous square is homotopy cocartesian.
\end{construction}

\p In this model structure, the morphism $[e,1]_t\to 1$ is a weak equivalence. For any $a\in A$ and $n\in \Nb$, we define $[e,1]_t\vee[a,n]$ as the pushout:
\[\begin{tikzcd}
	{[e,1]} & {[e,1]\vee[a,n]} \\
	{[e,1]_t} & {[e,1]_t\vee[a,n]}
	\arrow[from=1-1, to=2-1]
	\arrow[from=2-1, to=2-2]
	\arrow[from=1-1, to=1-2]
	\arrow[from=1-2, to=2-2]
\end{tikzcd}\]
The canonical morphism $[e,1]_t\cup [a,1]\cup...\cup [a,1]\to [e,1]_t\vee[a,n]$ is then a weak equivalence. By two out of three, and using the weak equivalence $[e,1]_t\to 1$, this implies that $[e,1]_t\vee[a,n]\to [a,n]$ is a weak equivalence. \sym{((g22@$[e,1]_t\vee[a,n]$}

We define similarly the object $[a,n]\vee[e,1]_t$ that comes along with a weak equivalence $[a,n]\vee[e,1]_t\to [a,n]$.

\subsection{Gray module}
\p
Let $A$ be a category of stratified presheaves on an elegant Reedy category (as defined in paragraph \ref{para:reedy} and section \ref{section:Marked and stratified presheaves}), endowed with a nice model structure (as defined in paragraph \ref{para:nice model structure}). We suppose furthermore that the terminal element of $A$, denoted by $e$, is representable. We also suppose that $A$ is endowed with \textit{intelligent $n$-truncation} for any $n\in \mathbb{N}\cup\{\omega\}$, i.e a family of left Quillen functors $\tau^i_{\uvar}:(\mathbb{N}\cup\{\omega\})^{op}\to \End(A)$ such that
\begin{enumerate}[itemsep=0mm]
\item[$-$] $\tau^i_\omega = id$,
\item[$-$] for any $n\leq m$, $\tau^i_n\tau^i_m=\tau^i_n$,
\item[$-$] for any $n\leq m$, the natural transformation $\tau^i_m\to \tau^i_n$ is an entire monomorphism,
\end{enumerate}
and a left Quillen bifunctor $\uvar\otimes \uvar: \stratSset^1\times A \to A$ such that 
\begin{enumerate}
\item[$-$] for $K$ and $L$ two stratified simplicial sets, and $a\in A$, there is a morphism $K\otimes (L\otimes a)\to (K\times L)\otimes a$ natural in $K,L$ and $a$, such that the following square commutes
\[\begin{tikzcd}
	{K\otimes (L\otimes (M\otimes a))} & {(K\times L)\otimes (M\otimes a)} \\
	{K\otimes ((L\times M)\otimes a)} & {(K\times L\times M) \otimes a}
	\arrow[from=1-1, to=1-2]
	\arrow[from=1-2, to=2-2]
	\arrow[from=2-1, to=2-2]
	\arrow[from=1-1, to=2-1]
\end{tikzcd}\]
for any stratified simplicial sets $M$.
\item[$-$] The functor $[0]\otimes \uvar: A\to A$ is the identity.
\item[$-$] For any integer $n$, for any object $a$ invariant under $\tau^i_n$, and for any stratified simplicial set $K$, the object $K\otimes a$ is invariant under $\tau^i_{n+1}$.
\end{enumerate}
Here, the model category $\stratSset^1$ corresponds to the model structure for $1$-complicial sets on stratified simplicial sets given in theorem \ref{theo:model structure on complicial set}.

\p
We define $e\star a$ as the pushout:
\[\begin{tikzcd}
	{\{0\}\times a} & {[1]\otimes a} \\
	e & {e\star a}
	\arrow[from=1-1, to=2-1]
	\arrow[from=1-1, to=1-2]
	\arrow[from=1-2, to=2-2]
	\arrow[from=2-1, to=2-2]
	\arrow["\lrcorner"{anchor=center, pos=0.125, rotate=180}, draw=none, from=2-2, to=1-1]
\end{tikzcd}\]
We consider the natural transformations $s^0\star a:e\star e\star a\to e\star a$ and $d^0\star a:a\to e\star a$,
induced respectively by the morphism
$$
\begin{array}{cclcccccccc}
[1]\otimes [1]\otimes a&\to & ([1]\times [1])\otimes a &\to & [1]\otimes a\\
&&(\{i\}\times \{j\})\otimes a&\mapsto & \{i\wedge j\}\otimes a.
\end{array}$$
and the morphism 
$$\{1\}\otimes a \to [1]\otimes a.$$
These natural transformations induce commutative diagrams:
\[\begin{tikzcd}
	{e\star e\star e\star a } & { e\star e\star a } && {e\star a} & { e\star e\star a } & {e\star a} \\
	{ e\star e\star a } & {e\star a} &&& {e\star a}
	\arrow["{s^0\star a}", from=1-2, to=2-2]
	\arrow["{s^0\star a}"', from=2-1, to=2-2]
	\arrow["{s^0\star (e\star a)}", from=1-1, to=1-2]
	\arrow["{e\star (s^0\star a)}"', from=1-1, to=2-1]
	\arrow["{e\star d^0}", from=1-4, to=1-5]
	\arrow["{d^0\star (e\star a)}", from=1-5, to=1-6]
	\arrow["{s^0\star a}", from=1-5, to=2-5]
	\arrow["id", curve={height=-6pt}, from=1-6, to=2-5]
	\arrow["id"', curve={height=6pt}, from=1-4, to=2-5]
\end{tikzcd}\]
The (inverted) composition $g,f\mapsto g\circ f$ is a monoidal structure on the category of endomorphisms of $A$ and the natural transformation $s^0:e\star e \star\uvar \to e\star \uvar$ defines a structure of monoid for $e\star\uvar$.
This induces a functor $\Delta\times A\to A$ sending $([n],a)$ to $e\star e\star ....\star a$. We extend this to a functor $\Delta_t\times A\to A$ in defining $[n]_t\star a$ as the pushout:
\[\begin{tikzcd}
	{\underset{k\geq -1}{\coprod}~~\underset{b,~\tau^i_k(b)=b}{\coprod}~~\underset{b\to a}{\coprod}[n]\star b} & {[n]\star a} \\
	{\underset{k\geq -1}{\coprod}~~\underset{b,~\tau^i_k(b)=b}{\coprod}~~\underset{b\to a}{\coprod}\tau^i_{n+k}([n]\star b)} & {[n]_t\star a}
	\arrow[from=2-1, to=2-2]
	\arrow[""{name=0, anchor=center, inner sep=0}, from=1-1, to=1-2]
	\arrow[from=1-2, to=2-2]
	\arrow[from=1-1, to=2-1]
	\arrow["\lrcorner"{anchor=center, pos=0.125, rotate=180}, draw=none, from=2-2, to=0]
\end{tikzcd}\]
where $\tau^i_{-1}$ is the constant functor with value $\emptyset$.

\p 
\label{para:Gray module}
Such model category $A$ is a \notion{Gray module} if for any $a$, the induced functor $\uvar\star a:\Delta_t\to A_{a/}$ lifts to a left Quillen functor $\uvar\star a:\stratSset^\omega\to A_{a/}$.

We recall that $\stratSset^\omega$ denotes the model structure for $\omega$-complicial sets given in theorem \ref{theo:model structure on complicial set}.

For the rest of this chapter, we fix a Gray module $A$. For a stratified simplicial set $K\in \stratSset$, the object $K\star \emptyset\in A$ is simply noted by $K$.

\begin{remark}
In general, $[n]\otimes e$ and $[n]\star \emptyset$ are two very different objects. Indeed $[n]\otimes e$ has to be invariant up to homotopy under $\tau^i_1$ which is not the case for $[n]\star \emptyset$. Analogously $[k]\otimes ([l]\otimes [a])$ and $([k]\otimes [l])\otimes [a]$ have \textit{a priori} no links. When we write $[n_0]\otimes[n_1]\otimes..[n_k]\otimes a$, we will always mean $[n_0]\otimes([n_1]\otimes..([n_k]\otimes a))$.
 \end{remark}

\begin{example}
\label{example:stratsset is gray module}
For any $d\in \mathbb{N}\cup \{\omega\}$,
the model category $\stratSset^d$, corresponding to the model structure for $d$-complicial sets on stratified simplicial sets, and where $K\otimes L := \tau^i_1(K)\boxtimes L$, is an example of Gray module.

Indeed, if $n$ is any integer, we define $[n]^{\diamond}:=[0]\diamond [0]\diamond...\diamond [0]$ and $[n]_t^{\diamond}:= \tau^i_n([n]^{\diamond})$. This induces a colimit preserving functor $K\mapsto K^{\diamond}$. The join coming from $\tau^i_1(\uvar)\boxtimes \uvar$ then corresponds to the functor $(K,L)\mapsto K^\diamond\diamond L$. The proposition \ref{prop:equivalence between diamond and join product} provides a natural transformation $K^{\diamond}\diamond L\to K\star L$, wich implies that the first functor is left Quillen. 
\end{example}

\section{Gray constructions for stratified Segal $A$-categories}
 We now construct a Gray cylinder and a Gray cone on $\stratSeg(A)$, using the structure of Gray module that $A$ has. We denote by $\Delta_+$ the augmented simplex category and $d^0$ the unique morphism $\emptyset\to [0]$.
\subsection{Gray cylinder}
\p
We define the functor 
$$
\begin{array}{ccl}
\Delta^3\times A &\to& \Seg(A)\\
~[n_0],[n_1],[n_2],a&\mapsto &[a,n_0]\vee[[n_1]\otimes a,1]\vee[a,n_2]
\end{array}$$
where $[a,n_0]\vee[[n_1]\otimes a,1]\vee[a,n_2]$ fits in the following pushout:
\[\begin{tikzcd}
	{[[n_1]\otimes a,n_0]\amalg [[n_1]\otimes a,n_2]} & {[[n_1]\otimes a,n_0+1+n_2]} \\
	{[[0]\otimes a,n_0]\amalg [[0]\otimes a,n_2]} & {[a,n_0]\vee[[n_1]\otimes a,1]\vee[a,n_2]}
	\arrow[from=1-1, to=2-1]
	\arrow[from=1-2, to=2-2]
	\arrow[from=2-1, to=2-2]
	\arrow[""{name=0, anchor=center, inner sep=0}, from=1-1, to=1-2]
	\arrow["\lrcorner"{anchor=center, pos=0.125, rotate=180}, draw=none, from=2-2, to=0]
\end{tikzcd}\]
If $n$ is an integer, \wcnotation{$\Delta^3_{/[n]}$}{(delta3@$\Delta^3_{/[n]}$} is the pullback: 
\[\begin{tikzcd}
	{\Delta^3_{/[n]}} & {\Delta^3} \\
	{\Delta_{/[n]}} & \Delta
	\arrow[from=1-1, to=1-2]
	\arrow[from=1-2, to=2-2]
	\arrow[from=2-1, to=2-2]
	\arrow[from=1-1, to=2-1]
	\arrow["\lrcorner"{anchor=center, pos=0.125}, draw=none, from=1-1, to=2-2]
\end{tikzcd}\]
where the right hand functor sends $([n_0],[n_1],[n_2])$ to $[n_0]\star [n_1]^{op}\star [n_2]$.
\begin{prop}
\label{prop:delta 3 n is reedy elegant}
The category $\Delta^3_{/[n]}$ is an elegant Reedy category. 
\end{prop}
\begin{proof}
We denote $X$ the trisimplicial set whose value on $[n_0],[n_1],[n_2]$ is $\Hom_{\Delta}([n_0]\star [n_1]^{op}\star [n_2],[n])$. The category $\Delta^3_{/[n]}$ fits in the pullback
\[\begin{tikzcd}
	{\Delta^3_{/[n]}} & {\Psh{\Delta^3}_{/X}} \\
	{\Delta^3} & {\Psh{\Delta^3}}
	\arrow[from=1-2, to=2-2]
	\arrow[from=1-1, to=2-1]
	\arrow[from=2-1, to=2-2]
	\arrow[from=1-1, to=1-2]
\end{tikzcd}\]
and is then an elegant Reedy category according to proposition \ref{prop:elelangat stable by slice}.
\end{proof}

\p
We define the functor 
$$
\begin{array}{rcl}
A\times \Delta&\to& \Seg(A)\\
~[n],a&\mapsto &F(a,n) 
\end{array}$$
by the formula 
$ F(a,n) :=\underset{\Delta^3_{/[n]}}{\colim}~[a,n_0]\vee[[n_1]\otimes a,1]\vee [a,n_2].$

In order to extend this functor to stratified Segal $A$-precategories with construction \ref{cons:lifting of a functor from A times Delta}, we will need to define the value on $[e,1]_t$, i.e. to choose an object $F(e,1)'$ and an entire cofibration $F(e,1)\to F(e,1)'$. It will be useful to have a more explicit description of this object.
\begin{example}
\label{exe:explicit Gray cycinder 1}
The sub-category of $\Delta^3_{/[1]}$ composed of non degenerate objects can be pictured by the graph:
\[\begin{tikzcd}
	{[0]\star[0]^{op}\star[0]} & {[0]\star[1]^{op}\star[0]} & {[0]\star[0]^{op}\star[0]} \\
	{[0]\star[0]^{op}\star[1]} & {[1]} & {[1]\star[0]^{op}\star[0]} \\
	{[0]\star[0]^{op}\star[0]} && {[0]\star[0]^{op}\star[0]}
	\arrow["{d^0}"', from=3-3, to=2-3]
	\arrow["{d^3}", from=3-1, to=2-1]
	\arrow["{d^2}"', from=1-1, to=2-1]
	\arrow["{d^1}"', from=1-3, to=1-2]
	\arrow["{d^2}", from=1-1, to=1-2]
	\arrow["{s^0s^2}"{description}, from=1-2, to=2-2]
	\arrow["{d^1}", from=1-3, to=2-3]
	\arrow["{s^1}"{description}, from=1-3, to=2-2]
	\arrow["{d^0s^0s^1}", from=3-3, to=2-2]
	\arrow["{s^1s^2}"{description}, from=2-3, to=2-2]
	\arrow["{s^0}"{description}, from=1-1, to=2-2]
	\arrow["{s^0s^0}"{description}, from=2-1, to=2-2]
	\arrow["{d^1s^0s^1}"', from=3-1, to=2-2]
\end{tikzcd}\]
The Segal $A$-precategory $F(e,1)$ is then the colimit of the following diagram:
\[\begin{tikzcd}
	{[e,2]} & {[e,1]} & {[[1],1]} & {[e,1]} & {[e,2]}
	\arrow["{[e,d^1]}"', from=1-2, to=1-1]
	\arrow["{[d^1,1]}"', from=1-4, to=1-3]
	\arrow["{[d^0,1]}", from=1-2, to=1-3]
	\arrow["{[e,d^1]}", from=1-4, to=1-5]
\end{tikzcd}\]
\end{example}
\p \sym{(iotimes@$I\otimes\uvar$}
We define the functor $$ I \otimes\uvar: \stratSeg(A)\to \stratSeg(A)$$ induced, as in the construction \ref{cons:lifting of a functor from A times Delta}, by $F$ and with $F(e,1)'$ as the colimit of the following diagram:
\[\begin{tikzcd}[sep = small]
	{[e,1]_t} & {[e,1]} & {[e,2]} & {[e,1]} & {[[1]_t,1]} & {[e,1]} & {[e,2]} & {[e,1]} & {[e,1]_t}
	\arrow["{[e,d^0]}"', from=1-8, to=1-7]
	\arrow["{[e,d^2]}", from=1-2, to=1-3]
	\arrow["{[e,d^1]}"', from=1-4, to=1-3]
	\arrow["{[d^1,1]}"', from=1-6, to=1-5]
	\arrow["{[d^0,1]}", from=1-4, to=1-5]
	\arrow["{[e,d^1]}", from=1-6, to=1-7]
	\arrow[from=1-2, to=1-1]
	\arrow[from=1-8, to=1-9]
\end{tikzcd}\]

The two objects of $\Delta^3_{[n]}$, $s^ns^{n+1}:[n]\star [0]^{op}\star[0]\to [n]$ and $s^0s^0:[0]\star[0]^{op}\star[n]\to [n]$, induce two morphisms:
$d^1\otimes [a,n]:\{0\}\otimes[a,n]:= [a,n]\hookrightarrow [a,n]\vee[e,1]\to I\otimes [a,n]$ and $d^0\otimes [a,n]:\{1\}\otimes[a,n]:= [a,n]\hookrightarrow [e,1]\vee[a,n]\to I\otimes [a,n]$. By extending them by colimits we get two maps 
$$d^1\otimes C:\{0\}\otimes C := C\to I\otimes C~~~\mbox{and}~~~
d^0\otimes C: \{1\}\otimes C := C\to I\otimes C.$$

\begin{prop}
The Segal $A$-precategory $I \otimes[a,1]$ is the colimit and the homotopy colimit of the diagram:
\[\begin{tikzcd}
	{[e,1]\vee[a,1]} & {[a,1]} & {[[1]\otimes a,1]} & {[a,1]} & {[a,1]\vee[e,1]}
	\arrow["{[e,d^1]}"', from=1-2, to=1-1]
	\arrow["{[d^1,1]}"', from=1-4, to=1-3]
	\arrow["{[d^0,1]}", from=1-2, to=1-3]
	\arrow["{[e,d^1]}", from=1-4, to=1-5]
\end{tikzcd}\]
\end{prop}
\begin{proof}
 The description of $\Delta^{3}_{/[1]}$ is given in the example \ref{exe:explicit Gray cycinder 1}.
The stratified Segal $A$-precategory $F(a,1)$ is then the colimit of the following diagram:
\[\begin{tikzcd}
	{[a,2]} & {[a,1]} & {[[1]\otimes a,1]} & {[a,1]} & {[a,2]}
	\arrow["{[e,d^1]}"', from=1-2, to=1-1]
	\arrow["{[d^1,1]}"', from=1-4, to=1-3]
	\arrow["{[d^0,1]}", from=1-2, to=1-3]
	\arrow["{[e,d^1]}", from=1-4, to=1-5]
\end{tikzcd}\]
and the Segal $A$-precategory $I \otimes[a,1]$ is the colimit of the given diagram. As all the morphisms are cofibrations, this colimit is a homotopy colimit.
\end{proof}

\begin{remark}
\label{rem:link with thestrict case cyinder}
To justify why this definition of the Gray interval is the good one, let's study the case of $\zo$-categories.
We denote by $I$ the $\zo$-category generated by the graph $0\to 1$.
If $C$ is an $\zo$-category, we denote by $[C,1]$ the $\zo$-category with two objects - denoted by $0$ and $1$ - and verifying: 
$$\Hom_{[C,1]}(0,1) :=C,~~~\Hom_{[C,1]}(1,0) := \emptyset,~~~\Hom_{[C,1]}(0,0)=\Hom_{[C,1]}(1,1):=\{id\}.$$
We denote by $e$ the terminal $\zo$-category. For example $[e,1] = I$. 
Applying the duality $(\uvar)^{op}$ to the formula given in theorem \ref{theo:appendice formula for otimes}, the $\zo$-category $I\otimes [C,1]$ is the colimit of the following diagram:
\[\begin{tikzcd}
	{[e,1]\vee[C,1]} & {[C,1]} & {[[1]\otimes C,1]} & {[C,1]} & {[C,1]\vee[e,1]}
	\arrow["\triangledown"', from=1-2, to=1-1]
	\arrow["{ [d^1\otimes C,1]}"', from=1-4, to=1-3]
	\arrow["{ [d^0\otimes C,1]}", from=1-2, to=1-3]
	\arrow["\triangledown", from=1-4, to=1-5]
\end{tikzcd}\]
where $\triangledown$ denotes the whiskerings.
\end{remark}
\subsection{Gray cone}
\p 
We define the functor 
$$
\begin{array}{ccl}
\Delta^2\times A &\to& \Seg(A)\\
~[n_0],[n_1] ,a&\mapsto &[[n_0]\otimes a,1]\vee[a,n_1]
\end{array}$$
where $[[n_0]\otimes a,1]\vee[a,n_1]$ fits in the following pushouts:
\[\begin{tikzcd}
	{ [[n_0]\otimes a,n_1]} & {[[n_0]\otimes a,1+n_1]} \\
	{[a,n_1]} & {[[n_0]\otimes a,1]\vee[a,n_1]}
	\arrow[from=1-1, to=2-1]
	\arrow[from=1-2, to=2-2]
	\arrow[from=2-1, to=2-2]
	\arrow[""{name=0, anchor=center, inner sep=0}, from=1-1, to=1-2]
	\arrow["\lrcorner"{anchor=center, pos=0.125, rotate=180}, draw=none, from=2-2, to=0]
\end{tikzcd}\]

If $n$ is an integer, \wcnotation{$\Delta^2_{/[n]}$}{(delta2@$\Delta^2_{/[n]}$} is the pullback: 
\[\begin{tikzcd}
	{\Delta^2_{/[n]}} & {\Delta^2} \\
	{\Delta_{/[n]}} & \Delta
	\arrow[from=1-1, to=1-2]
	\arrow[from=1-2, to=2-2]
	\arrow[from=2-1, to=2-2]
	\arrow[from=1-1, to=2-1]
	\arrow["\lrcorner"{anchor=center, pos=0.125}, draw=none, from=1-1, to=2-2]
\end{tikzcd}\]
where the right hand functor sends $([n_0],[n_1])$ to $[n_0]^{op}\star [n_1]$.
\begin{prop}
\label{prop:delta 2 n is reedy elegant}
The category $\Delta^2_{/[n]}$ is an elegant Reedy category. 
\end{prop}
\begin{proof}
The proof is analogue to the one of proposition \ref{prop:delta 3 n is reedy elegant}.
\end{proof}

\p 
We define the functor 
$$
\begin{array}{rcl}
A\times \Delta&\to& \Seg(A)\\
~[n],a&\mapsto &H(a,n)
\end{array}$$
by the formula 
$H(a,n):=\colim_{\Delta^2_{/[n]}} [[n_0]\otimes a,1]\vee[a,n_1]$.
 
In order to extend this functor to stratified Segal $A$-precategories with construction \ref{cons:lifting of a functor from A times Delta}, we will need to define the value on $[e,1]_t$, i.e. to choose an object $H(e,1)'$ and an entire cofibration $H(e,1)\to H(e,1)'$. It will be useful to have a more explicit description of this object. 
\begin{example}
\label{exe:explicit Gray cone 1}
The sub-category of $\Delta^2_{/[1]}$ composed of non degenerate objects can be pictured by the graph:
\[\begin{tikzcd}
	{[0]^{op}\star[0]} & {[1]^{op}\star[0]} & {[0]^{op}\star[0]} \\
	{[0]^{op}\star[1]} & {[1]} \\
	{[0]^{op}\star[0]}
	\arrow["{d^2}", from=3-1, to=2-1]
	\arrow["{d^1}"', from=1-1, to=2-1]
	\arrow["{d^1}"', from=1-3, to=1-2]
	\arrow["{d^2}", from=1-1, to=1-2]
	\arrow["{s^1}"{description}, from=1-2, to=2-2]
	\arrow["{d^0s^0}", from=1-3, to=2-2]
	\arrow["id"{description}, from=1-1, to=2-2]
	\arrow["{s^0}"{description}, from=2-1, to=2-2]
	\arrow["{d^1s^0}"', from=3-1, to=2-2]
\end{tikzcd}\]
The Segal $A$-precategory $ H(e,1)$ is then the colimit of the following diagram:
\[\begin{tikzcd}
	{[e,2]} & {[e,1]} & {[[1],1]}
	\arrow["{[e,d^1]}"', from=1-2, to=1-1]
	\arrow["{[d^0,1]}", from=1-2, to=1-3]
\end{tikzcd}\]
\end{example}

\p \sym{(estar@$e\star\uvar$}
We define the functor $$ e \star\uvar: \stratSeg(A)\to \stratSeg(A)$$ induced, as in the construction \ref{cons:lifting of a functor from A times Delta} by $ H$ and with $H(e,1)'$ as the colimit of the following diagram:
\[\begin{tikzcd}
	{[e,1]_t} & {[e,1]} & {[e,2]} & {[e,1]} & {[[1]_t,1]}
	\arrow["{[e,d^0]}", from=1-2, to=1-3]
	\arrow["{[e,d^1]}"', from=1-4, to=1-3]
	\arrow["{[d^0,1]}", from=1-4, to=1-5]
	\arrow[from=1-2, to=1-1]
\end{tikzcd}\]

The object $s^0:[0]^{op}\star[n]\to [n]$ induces a composite morphism $d^0\star[a,n]:\emptyset\star[a,n]:= [a,n]\hookrightarrow [1,1]\vee[a,n]\to e\star [a,n]$, which induces by extension by colimit a natural transformation $$d^0\star C:\emptyset\star C:= C\to e\star C.$$

\begin{prop}
\label{prop:explicit expression of e star a,1}
The Segal $A$-precategory
 $e\star [a,1]$ is the colimit and the homotopy colimit of the following diagram:
$$\begin{tikzcd}
	{[e,1]\vee[a,1]} & {[a,1]} & {[e\star a,1]}
	\arrow["{[e,d^1]}"', from=1-2, to=1-1]
	\arrow["{[d^{0}\star a,1]}", from=1-2, to=1-3]
\end{tikzcd}
$$
\end{prop}
\begin{proof}
We have already given the description of $\Delta^2_{/[1]}$ in
the example \ref{exe:explicit Gray cone 1}.
The Segal $A$-precategory $H(a,1)$ is the colimit of the following diagram:
\[\begin{tikzcd}
	{[a,2]} & {[a,1]} & {[[1]\otimes a,1]}
	\arrow["{[e,d^1]}"', from=1-2, to=1-1]
	\arrow["{[d^0\otimes a,1]}", from=1-2, to=1-3]
\end{tikzcd}\]
and $e\star [a,1]$ is the colimit of the given diagram.
As all the morphisms are cofibrations, this colimit is a homotopy colimit.
\end{proof}

\begin{remark}
Using again notations of remark \ref{rem:link with thestrict case cyinder}, if $C$ is an $\zo$-category, the $\zo$-category $e\star C$ is the colimit of the following diagram:
\[\begin{tikzcd}
	{[e,1]\vee[C,1]} & {[C,1]} & {[e\star C,1]}
	\arrow["\triangledown"', from=1-2, to=1-1]
	\arrow["{[d^{0}\star C,1]}", from=1-2, to=1-3]
\end{tikzcd}\]
where $\triangledown$ is the whiskering. Our definition of the join is therefore analogous to that of the strict world.
\end{remark}

\begin{prop}
\label{prop:explicit expression of e star e star a,1}
The Segal $A$-precategory
 $[1]\star [a,1]$ is the colimit of the following diagram:
$$
\begin{tikzcd}[column sep=0.7 cm]
	{[[2]\botimes a,1]} & {[[1]\otimes a,1]} & {[[1],1]\vee[a,1]} & {[e,1]\vee[a,1]} & {[e,2]\vee[a,1]} \\
	{[e\star a,1]} &&& {[a,1]} & {[e,1]\vee[a,1]} \\
	{[[1]\star a,1]} &&& {[e\star a,1]} & {[e,1]\vee[e\star a,1]}
	\arrow["{[d^0\otimes a,1]}"', from=1-2, to=1-1]
	\arrow["{[[1]\otimes a,d^1]}", from=1-2, to=1-3]
	\arrow["{[d^0\otimes a,2]}"', from=1-4, to=1-3]
	\arrow["{[a,d^1]}", from=1-4, to=1-5]
	\arrow["{[a,d^1]}", from=2-4, to=2-5]
	\arrow["{[d^1\botimes a,1]}", from=2-1, to=1-1]
	\arrow["{[a,d^1]}", from=2-4, to=1-4]
	\arrow["{[a,d^2]}"', from=2-5, to=1-5]
	\arrow["{[d^1\star a,1]}"', from=2-1, to=3-1]
	\arrow["{[d^{0}\star a,2]}", from=2-5, to=3-5]
	\arrow["{[d^{0}\star a,1]}"', from=2-4, to=2-1]
	\arrow["{[d^0\star a,1]}", from=3-4, to=3-1]
	\arrow["{[d^{0}\star a,1]}"', from=2-4, to=3-4]
	\arrow["{[e\star a,d^1]}"', from=3-4, to=3-5]
\end{tikzcd}$$
where $[2]\botimes a$ and $[[1],1]\vee[a,1]$ are the pushouts:
\[\begin{tikzcd}
	{[1]\otimes a\amalg [1]\otimes a} && {[2]\otimes a} & {[[1]\otimes a,1]\amalg [[1]\otimes a,2]} & {[[1]\otimes a,2]} \\
	{e\star a\amalg e\star a} && {[2]\botimes a} & {[[1],1]\amalg[a,1]} & {[[1],1]\vee[a,1]}
	\arrow[""{name=0, anchor=center, inner sep=0}, "{d^1\otimes a\amalg d^2\otimes a}", from=1-1, to=1-3]
	\arrow["{d^1\botimes a\amalg d^2\botimes a}"', from=2-1, to=2-3]
	\arrow[from=1-1, to=2-1]
	\arrow[from=1-3, to=2-3]
	\arrow["{[[1]\otimes a,d^2\amalg d^1]}", from=1-4, to=1-5]
	\arrow[from=1-4, to=2-4]
	\arrow[from=2-4, to=2-5]
	\arrow[from=1-5, to=2-5]
	\arrow["\lrcorner"{anchor=center, pos=0.125, rotate=180}, draw=none, from=2-5, to=1-4]
	\arrow["\lrcorner"{anchor=center, pos=0.125, rotate=180}, draw=none, from=2-3, to=0]
\end{tikzcd}\]
\end{prop}
\begin{proof}
Let's start by studying the object $H(a,2)$. Here is a final subcategory of $\Delta^2_{/[2]}$:
\[\begin{tikzcd}
	{[1]^{op}\star[0]} & {[1]^{op}\star[1]} & {[0]^{op}\star[1]} \\
	{[2]^{op}\star[0]} & {[2]} & {[0]^{op}\star[2]}
	\arrow["{s^2}"', from=2-1, to=2-2]
	\arrow["{d^2}"', from=1-1, to=2-1]
	\arrow["{d^2}", from=1-1, to=1-2]
	\arrow["{s^1}", from=1-2, to=2-2]
	\arrow["{d^1}"', from=1-3, to=1-2]
	\arrow["{d^1}", from=1-3, to=2-3]
	\arrow["{s^0}", from=2-3, to=2-2]
\end{tikzcd}\]
The Segal $A$-precategory $H(a,2)$ is then the colimit of the following diagram:
\[\begin{tikzcd}
	{[[2]\otimes a,1]} & {[[1]\otimes a,1]} & {[[1]\otimes a,1]\vee[a,1]} & {[a,2]} & {[a,3]}
	\arrow["{[d^0\otimes a,1]}"', from=1-2, to=1-1]
	\arrow["{[[1]\otimes a,d^1]}", from=1-2, to=1-3]
	\arrow["{[d^0\otimes a,2]}"', from=1-4, to=1-3]
	\arrow["{[a,d^1]}", from=1-4, to=1-5]
\end{tikzcd}\]
The Segal $A$-precategory $e\star([e,1]\vee[a,1])$ is then the colimit of the following diagram:
\[\begin{tikzcd}[column sep=0.7 cm]
	{[[2]\botimes a,1]} & {[[1]\otimes a,1]} & {[[1],1]\vee[a,1]} & {[e,1]\vee[a,1]} & {[e,2]\vee[a,1]}
	\arrow["{[d^0\otimes a,1]}"', from=1-2, to=1-1]
	\arrow["{[[1]\otimes a,d^1]}", from=1-2, to=1-3]
	\arrow["{[d^0\otimes a,2]}"', from=1-4, to=1-3]
	\arrow["{[a,d^1]}", from=1-4, to=1-5]
\end{tikzcd}\]
The fact that $[1]\star[a,1]$ is the colimit of the given diagram then follows from the equality $[1]\star[a,1]=e\star(e\star[a,1])$ and from the explicit expression of 
$e\star [a,1]$ given in proposition \ref{prop:explicit expression of e star a,1}.
\end{proof}

\subsection{Link between the Gray cylinder and Gray cone}
\p There is a canonical morphism $I\otimes [a,n]\to e\star [a,n]$ sending $[a,n_0]\vee[[n_1]\otimes a,1]\vee[a,n_2]$ to $[[n_1]\otimes a,1]\vee[a,n_2]$.
Note that the induced morphism $I\otimes [e,1]\to e\star [e,1]\to e\star [e,1]_t$ factors through $I\otimes [e,1]_t$. We can then extend it by colimit to a natural transformation $I\otimes C\to e\star C$.

We now define $(I\otimes [a,n])_{/\{0\}\otimes [a,n]}$ and $[a,n_0]\vee[[n_1]\otimes a,1]\vee[a,n_2]_{/[a,n_0]}$ as the pushouts:
\[\begin{tikzcd}[column sep = 0.5cm]
	{[a,n]\otimes\{0\}} & {I\otimes[a,n]} & { [a,n_0]} & { [a,n_0]\vee[[n_1]\otimes a,1]\vee[a,n_2]} \\
	e & {(I\otimes [a,n])_{/\{0\}\otimes [a,n]}} & e & { [a,n_0]\vee[[n_1]\otimes a,1]\vee[a,n_2]_{/[a,n_0]}}
	\arrow[""{name=0, anchor=center, inner sep=0}, from=1-1, to=1-2]
	\arrow[from=2-1, to=2-2]
	\arrow[from=1-1, to=2-1]
	\arrow[from=1-2, to=2-2]
	\arrow[from=1-3, to=2-3]
	\arrow[from=1-4, to=2-4]
	\arrow[from=2-3, to=2-4]
	\arrow[""{name=1, anchor=center, inner sep=0}, from=1-3, to=1-4]
	\arrow["\lrcorner"{anchor=center, pos=0.125, rotate=180}, draw=none, from=2-4, to=1]
	\arrow["\lrcorner"{anchor=center, pos=0.125, rotate=180}, draw=none, from=2-2, to=0]
\end{tikzcd}\]
By Segal extensions and by two out of three, the following canonical morphism 
$$ [a,n_0]\vee[[n_1]\otimes a,1]\vee[a,n_2]_{/[a,n_0]}\to [[n_1]\otimes a,1]\vee[a,n_2]$$
is a weak equivalence. As $\Delta^3_{/[n]}$ is Reedy elegant, this induces a weak equivalence
$$\colim_{\Delta^3_{/[n]}} [a,n_0]\vee[[n_1]\otimes a,1]\vee[a,n_2]_{/[a,n_0]}\to \colim_{\Delta^3_{/[n]}} [[n_1]\otimes a,1]\vee[a,n_2].$$
Remark furthermore that the left hand object is equivalent to $(I\otimes [a,n])_{/\{0\}\otimes [a,n]}$ and the right one to 
 $H(a,n)$. As the construction \ref{cons:lifting of a functor from A times Delta} preserves weakly invertible natural transformations between functors that preserve cofibration, this induces a weakly invertible natural transformation $(I\otimes [a,n])_{/\{0\}\otimes [a,n]}\to e \star[a,n]$. This directly implies that squares
\[\begin{tikzcd}
	{\{0\}\otimes [a,n]} & {I\otimes[a,n]} && {\{0\}\otimes [e,1]_t} & {I\otimes[e,1]_t} \\
	e & {e\star[a,n]} && e & {e\star[e,1]_t}
	\arrow[from=1-1, to=2-1]
	\arrow[from=2-1, to=2-2]
	\arrow[from=1-1, to=1-2]
	\arrow[from=1-2, to=2-2]
	\arrow[from=1-4, to=2-4]
	\arrow[from=2-4, to=2-5]
	\arrow[from=1-5, to=2-5]
	\arrow[from=1-4, to=1-5]
\end{tikzcd}\]
are homotopy cocartesian. As every stratified Segal $A$-precategory is a homotopy colimit of objects of shape $[a,n]$ and $[e,1]_t$, and as $I\otimes \uvar$ and $e\star \uvar$ preserves monomorphisms, this implies the following proposition:

\begin{prop}
\label{labe:Link between the Gray cylinder and cone}
For any stratified Segal $A$-precategory $C$, the natural transformation $I\otimes \uvar\to e\star \uvar$ fits into a homotopy cocartesian square:
\[\begin{tikzcd}
	{\{0\}\otimes C} & {I\otimes C} \\
	e & {e\star C}
	\arrow[from=1-1, to=1-2]
	\arrow[from=2-1, to=2-2]
	\arrow[from=1-1, to=2-1]
	\arrow[from=1-2, to=2-2]
\end{tikzcd}\]
\end{prop}

\p
We define the functor 
$$
\begin{array}{rcl}
A\times \Delta&\to& \Seg(A)\\
~[n],a&\mapsto &T(a,n)
\end{array}$$
by the formula 
$T(a,n) := [[n]\otimes a,1]$.

Eventually
we define the functor $ \Sigma^{\circ}[a,n]: \stratSeg(A)\to \stratSeg(A)$ \sym{(sigmacirc@$\Sigma^{\circ}$} induced, as in the construction \ref{cons:lifting of a functor from A times Delta}, by $ T$ and with $T(e,1):= [[1]_t\otimes e,1]$. This functor is called the \wcnotion{$\circ$-suspension}{suspensions@$\circ$-suspension}.
With a proof similar to the on of proposition \ref{labe:Link between the Gray cylinder and cone}, one can show:
\begin{prop}
\label{labe:Link between the Gray cylinder and cosuspension}
There exists a natural transformation $e\star \uvar\to\Sigma^{\circ}(\uvar)$ such that
for any marked Segal $A$-precategory $C$, $e\star C\to \Sigma^{\circ}C$ induces a homotopy cocartesian square:
\[\begin{tikzcd}
	{ C} & {e\star C} \\
	{ e} & {\Sigma^{\circ}C}
	\arrow[from=1-1, to=1-2]
	\arrow[from=2-1, to=2-2]
	\arrow[from=1-1, to=2-1]
	\arrow[from=1-2, to=2-2]
\end{tikzcd}\]
\end{prop}

\subsection{Gray constructions are left Quillen}
In this section, we show that the Gray cylinder is a Quillen functor. Combined with the proposition \ref{labe:Link between the Gray cylinder and cone}, this will imply that the Gray cone is Quillen.
\p
Let $x:[k_0]\star[k_1]^{op}\star [k_2] \to[n]$ be an element of $\Delta^3_{/[n]}$. The \wcnotion{degree}{degree of an element of $\Delta^3_{/[n]}$} of $x$, is $f(0)-f(k_1)$ where $f$ is the composite morphism:
$$f:[k_1]^{op}\to [k_0]\star[k_1]^{op}\star [k_2] \to [n]$$
We will denote by $K_{\leq i}$ the full subcategory of $\Delta^3_{/[n]}$ whose objects are of degree inferior or equal to $i$.

An element $x:[k_0]\star[k_1]^{op}\star [k_2] \to[n]$ of degree $d$ is \wcnotion{regular}{regular elements of $\Delta^3_{/[n]}$} if $k_1 = d$, $k_0+k_1+k_2 = n$ and 
$$x(l):=
\left\{
\begin{array}{ll}
l &\mbox{if $l\leq k_0$}\\
l-1 &\mbox{if $k_0<l\leq k_0+k_1$}\\
l-2 &\mbox{if $k_0+k_1<l$}\\
\end{array} \right.$$
Remark that the regular object $x$ is characterized by the triple $(k_0,k_1,k_2)$.

\p
Let $x:[k_0]\star[k_1]^{op}\star [k_2] \to[n]$ be an element $\Delta^3_{/[n]}$, and $i:[0]\to [k_0]\star[k_1]^{op}\star [k_2]$ a morphism. We denote by $d^ix :=[k_0']\star[k_1']^{op}\star [k_2'] \xrightarrow{d} [k_0]\star[k_1]^{op}\star [k_2] \to[n]$ the morphism that avoids $i$, and where $k_j' := k_j-1$ if $i$ factors through $[k_j]$ and $k_j' := k_j$ if not. We then define
$(\Delta^3_{/[n]})_{/\Lambda^ix}$ as the full subcategory of $(\Delta^3_{/[n]})_{/x}$ that includes any non negative object $x'\to x$ that are different of $d^ix\to x$ and $id:x\to x$.

\begin{lemma}
\label{lem:Gray cinder is Quillen 2}
For any regular object $x:[k_0]\star[k_1]^{op}\star [k_2] \to[n]$ and for any $i:[0]\to [k_0]\star[k_1]^{op}\star [k_2]$ which is neither $k_0+1$ nor $k_0+k_1+1$, the morphism
$$\underset{(\Delta^3_{/[n]})_{\Lambda^ix}}{\colim} [a,\uvar]\vee[\uvar \otimes a,1]\vee[a,\uvar] 
\to [a,k_0]\vee[[k_1]\otimes a,1]\vee[a,k_2]$$
is an acyclic cofibration. 
\end{lemma} 
\begin{proof}
Suppose first that the image of $i$ is in $[k_0]$. There is a cocartesian square:
\[\begin{tikzcd}[column sep=0.38cm]
	{[[k_1]\otimes a, \Lambda^{i}[k_0+1+k_2]]\cup [\partial[k_1]\otimes a,  [k_0+1+k_2]]} & {\underset{(\Delta^3_{/[n]})_{/\Lambda^ix}}{\colim} [a,\uvar]\vee[\uvar \otimes a,1]\vee[a,\uvar] } \\
	{[[k_1]\otimes a, [k_0+1+k_2]]} & {[a,k_0]\vee[[k_1]\otimes a,1]\vee[a,k_2]}
	\arrow[from=1-1, to=1-2]
	\arrow[from=2-1, to=2-2]
	\arrow[from=1-1, to=2-1]
	\arrow[from=1-2, to=2-2]
\end{tikzcd}\]
where the left-hand morphism is an acyclic cofibration. The case where the image of $i$ is in $[k_2]$ is similar. 
Suppose now that $i$ lands in $[k_1]$. We then define $i':=i-k_0-1$, and there is a cocartesian square:
\[\begin{tikzcd}[column sep=0.33cm]
	{[\Lambda^{i'}[k_1]\otimes a, [k_0+1+k_2]]\cup[ [k_1]\otimes a,\partial [k_0+1+k_2]]} & {\underset{(\Delta^3_{/[n]})_{/\Lambda^ix}}{\colim} [a,\uvar]\vee[\uvar \otimes a,1]\vee[a,\uvar] } \\
	{[ [k_1]\otimes a, [k_0+1+k_2]]} & {[a,k_0]\vee[[k_1]\otimes a,1]\vee[a,k_2]}
	\arrow[from=1-1, to=1-2]
	\arrow[from=2-1, to=2-2]
	\arrow[from=1-1, to=2-1]
	\arrow[from=1-2, to=2-2]
\end{tikzcd}\]
where the left-hand morphism is an acyclic cofibration.
\end{proof}

\begin{lemma}
\label{lem:Gray cinder is Quillen 3}
Let $0<k<n$ be two integers. The morphism
$$\underset{\Delta^3_{/\Lambda^k[n]} \cup K_{\leq d}}{\colim} [a,\uvar]\vee[\uvar \otimes a,1]\vee[a,\uvar] \to \underset{\Delta^3_{/\Lambda^k[n]} \cup K_{\leq d+1}}{\colim} [a,\uvar]\vee[\uvar \otimes a,1]\vee[a,\uvar] $$
is an acyclic cofibration
\end{lemma}
\begin{proof}
For $x:=[k_0]\star [k_1]^{op}\star[k_2]\to [n]$ a regular element of degree $d+1$, we denote by $s_x$ the section of $x$ that avoids $k_0+1$ and $k_0+k_1+1$. We denote $R_{d+1}$ the set of regular elements of degree $d+1$.
We claim that we have a cocartesian square
\begin{equation}
\label{eq:Gray cinder is Quillen 3}
\begin{tikzcd}
	{\coprod_{x\in R_{d+1}}(\Delta^3_{/[n]})_{/\Lambda^{s_k(x)} x}} & {\Delta^3_{/\Lambda^k[n]} \cup K_{\leq d}} \\
	{\coprod_{x\in R_{d+1}}(\Delta^3_{/[n]})_{/x}} & {\Delta^3_{/\Lambda^k[n]}\cup K_{\leq d+1}}
	\arrow[from=1-1, to=2-1]
	\arrow[from=1-1, to=1-2]
	\arrow[from=2-1, to=2-2]
	\arrow[from=1-2, to=2-2]
\end{tikzcd}
\end{equation}
This will induce a cocartesian square:
\[\begin{tikzcd}[column sep =0.3cm]
	{\coprod_{x\in R_{d+1}}\underset{(\Delta^3_{/[n]})_{/\Lambda^{s_x(k)}x}}{\colim} [a,\uvar]\vee[\uvar \otimes a,1]\vee[a,\uvar] } & {\underset{\Delta^3_{/\Lambda^k[n]} \cup K_{\leq d }}{\colim} [a,\uvar]\vee[\uvar \otimes a,1]\vee[a,\uvar] } \\
	{\coprod_{x\in R_{d+1}}\ [a,k_0]\vee[[k_1]\otimes a,1]\vee[a,k_2] } & {\underset{\Delta^3_{/\Lambda^k[n]} \cup K_{\leq d +1}}{\colim} [a,\uvar]\vee[\uvar \otimes a,1]\vee[a,\uvar] }
	\arrow[from=1-1, to=2-1]
	\arrow[from=1-2, to=2-2]
	\arrow[from=2-1, to=2-2]
	\arrow[from=1-1, to=1-2]
\end{tikzcd}\]
where the left vertical morphism is an acyclic cofibration according to lemma \ref{lem:Gray cinder is Quillen 2}, which will conclude the proof.

We then have to justify the cocartesianess of the square \eqref{eq:Gray cinder is Quillen 3}. We denote by $D$ the colimit of the underlying span of this square and $\psi:D\to \Delta^3_{/\Lambda^k[n]}\cup K_{\leq d+1}$ the induced morphism. We will construct an inverse $\phi$ of this functor.

Let $x:[k_0]\star[k_1]^{op}\star [k_2] \to[n]$ be an element of $\Delta^3_{/[n]}$ of degree $(d+1)$. We denote by $x_r$ the regular element characterized by the triple $(x(k_1),d+1,n-x(k_0+k_1+1))$. There is a unique morphism $x\to x_r$. Furthermore, for any other regular element $x'$, $\Hom(x,x')=\emptyset$. We then set $$\phi(x):= x\to x_r\in (\Delta^3_{/[n]})_{/x_r}.$$ If $x:[k_0]\star[k_1]^{op}\star [k_2] \to[n]$ is an element of $\Delta^3_{/\Lambda^k[n]}$, we set $$\phi(x):=x\in \Delta^3_{/\Lambda^k[n]}\cup K_{\leq d}.$$ To justify that this is well defined, remark that for any object $x$ of $\Delta^3_{\Lambda^k[n]}$ of degree $d+1$, the morphism $x\to x_r$ factors through $\Lambda^{s_k(x_r)}x_r$. This assignation lifts to a functor $\phi:\Delta^3_{/\Lambda^k[n]}\cup K_{\leq d+1}\to D$ that is an inverse of $\psi$.
\end{proof}

\begin{prop}
\label{prop:Gray cinder is Quillen 5}
The morphism 
$I\otimes([a,1]\cup[a,1]\cup... \cup [a,1])\to I\otimes [a,n]$ is an acyclic cofibration.
\end{prop}
\begin{proof}
Let $0<k< n$ be two integers.
Let's demonstrate first that morphisms $I\otimes[a,\Lambda^k[n]]\to I\otimes [a,n]$ are acyclic cofibrations.
We set
$$P_d:=\underset{\Delta^3_{/\Lambda^k[n]} \cup K_{\leq d}}{\colim} [a,\uvar]\vee[\uvar \otimes a,1]\vee[a,\uvar].$$
According to lemma \ref{lem:Gray cinder is Quillen 3} , we have a sequence of acyclic cofibrations 
$I\otimes[a,\Lambda^k[n]]=P_0\to P_1...\to P_n= I\otimes [a,n] $. This implies that the functor $I\otimes [a,\uvar]:\Sset\to \stratSeg(A)$ sends inner anodyne extensions to weak equivalences. 

Eventually, proposition 3.7.4 of \cite{Cisinski_Higher_categories_and_homotopical_algebra} states that the inclusion $[1]\cup...\cup[1]\cup[1]\to[n]$ is an inner anodyne extension, which concludes the proof.
\end{proof}

\begin{lemma}
\label{lem:Gray cinder is Quillen 6}
Let $a\to b$ be a generating acyclic cofibration. The morphism 
$I\otimes ([a,n]\cup [b,\partial[n]])\to I\otimes [b,n]$ is an acyclic cofibration.
\end{lemma}
\begin{proof}
It is obvious that $I\otimes [a,n]\to I\otimes [b,n]$ is an acyclic cofibration. As $I\otimes [\uvar,\partial[n]]$ is the homotopy colimit of element of shape $I\otimes [\uvar,[k]]$, the morphism $I\otimes [a,\partial[n]]\to I\otimes [b,\partial[n]]$ also is an acyclic cofibration. Now, we consider the diagram: 
\[\begin{tikzcd}
	{I\otimes [a,\partial[n]]} & {I\otimes [a,n]} \\
	{I\otimes [b,\partial[n]]} & {I\otimes [a,n]\cup [b,\partial[n]]} \\
	&& {I\otimes [b,n]}
	\arrow[from=1-1, to=2-1]
	\arrow[from=2-1, to=2-2]
	\arrow[from=2-2, to=3-3]
	\arrow[from=1-2, to=2-2]
	\arrow[from=1-1, to=1-2]
	\arrow["\lrcorner"{anchor=center, pos=0.125, rotate=180}, draw=none, from=2-2, to=1-1]
	\arrow[curve={height=-12pt}, from=1-2, to=3-3]
	\arrow[curve={height=12pt}, from=2-1, to=3-3]
\end{tikzcd}\]
By stability of acyclic cofibration by pushouts and by two out of three, this implies the result.
\end{proof}

\begin{lemma}
\label{lem:Gray cinder is Quillen 7}
The morphism $I\otimes E^{\cong}\to I\otimes (E^{\cong})'$ is an acyclic cofibration. 
\end{lemma}
\begin{proof}
First of all, remark that $E^{\cong}\to [0]$ is a weak equivalence in $\stratSeg(A)$. 
According to the proposition \ref{labe:Link between the Gray cylinder and cone}, we then have a commutative square:
\[\begin{tikzcd}
	{I\otimes E^{\cong}} & {I\otimes (E^{\cong})'} \\
	{[E^{\cong}\otimes e,1]} & {[(E^{\cong})'\otimes e,1]}
	\arrow["\sim"', from=1-1, to=2-1]
	\arrow["\sim", from=1-2, to=2-2]
	\arrow["\sim", from=2-1, to=2-2]
	\arrow[from=1-1, to=1-2]
\end{tikzcd}\]
where all arrows labelled by $\sim$ are weak equivalences. By two out of three, this implies the result.
\end{proof}

\begin{lemma}
\label{lem:Gray cinder is Quillen 8}
The morphism $ I\otimes [e,1]_t\to I\otimes e$ is a weak equivalence.
\end{lemma}
\begin{proof}
This morphism is the horizontal colimit of the diagram
\[\begin{tikzcd}
	{[e,1]_t\vee[e,1]} & {[e,1]} & {[[1]_t,1]} & {[e,1]} & {[e,1]\vee[e,1]_t} \\
	{[e,1]} & {[e,1]} & {[e,1]} & {[e,1]} & {[e,1]}
	\arrow[from=2-2, to=2-1]
	\arrow[from=1-4, to=1-5]
	\arrow[from=2-4, to=2-5]
	\arrow[from=2-4, to=2-3]
	\arrow[from=2-2, to=2-3]
	\arrow[from=1-2, to=1-3]
	\arrow[from=1-4, to=1-3]
	\arrow[from=1-1, to=2-1]
	\arrow[from=1-2, to=1-1]
	\arrow[from=1-2, to=2-2]
	\arrow[from=1-3, to=2-3]
	\arrow[from=1-5, to=2-5]
	\arrow[from=1-4, to=2-4]
\end{tikzcd}\]
As all the vertical morphisms are weak equivalences, and as these colimits are homotopy colimits, this concludes the proof.
\end{proof}

\begin{prop}
\label{prop:cylinder is Quillen}
The functor $I\otimes\uvar:\stratSeg(A)\to \stratSeg(A)$ is a left Quillen functor. 
\end{prop}
\begin{proof}
It is obvious that this functor preserves cofibrations. Proposition \ref{prop:Gray cinder is Quillen 5} and lemmas \ref{lem:Gray cinder is Quillen 6}, \ref{lem:Gray cinder is Quillen 7} and \ref{lem:Gray cinder is Quillen 8} imply that it sends elementary anodyne extensions, and morphisms $E^{\cong}\to (E^{\cong})'$, $[e,1]_t\to 1$ to weak equivalences. According to proposition \ref{prop:model structure on stratified Segal category}, this implies the result.
\end{proof}

\begin{cor}
\label{cor:cone is Quillen}
The functor $e\star \uvar:\stratSeg(A)\to \stratSeg(A)_{e/}$ is a left Quillen functor.
\end{cor}
\begin{proof}
First of all, it is obvious that this functor preserves cofibrations. It is then enough to show that it preserves weak equivalences.
Proposition \ref{labe:Link between the Gray cylinder and cone} implies that $e\star \uvar$ is the homotopy colimit of the diagram of functors
$e\leftarrow id \xrightarrow{i_0} I\otimes \uvar.$
Each of these functors preserves weak equivalences, and so does $e\star\uvar$.
\end{proof}

\section{Quillen Adjunction with $\stratSset$}
The purpose of this section is to construct a Quillen adjunction
\[\begin{tikzcd}
	\stratSset & {\stratSeg(A)}
	\arrow[""{name=0, anchor=center, inner sep=0}, shift left=2, from=1-1, to=1-2]
	\arrow[""{name=1, anchor=center, inner sep=0}, shift left=2, from=1-2, to=1-1]
	\arrow["\dashv"{anchor=center, rotate=-90}, draw=none, from=0, to=1]
\end{tikzcd}\]
where the left adjoint sends $[n]$ to $e\star e\star ...\star e$.

In section \ref{section:Cosimplicial object}, we show that this assignment extends to a left adjoint. In sections \ref{section:Complicial horn inclusion}, \ref{section:Complicial thinness extensions}, and \ref{section:Saturation extensions}, we show that this left adjoint sends complicial horn inclusions, complicial thinness extensions, and saturation extensions to weak equivalences.
\subsection{Cosimplicial object}
\label{section:Cosimplicial object}
\p 
We consider the following span:
\[\begin{tikzcd}
	{\Delta^2_{/[n]}} & {\underset{\Delta^2_{/[n]}}{\colim}~\Delta^2_{/[n_1]}} & {\underset{\Delta^2_{/[n]}}{\colim}~\Delta^2_{/[1+n_1]}}
	\arrow[from=1-2, to=1-3]
	\arrow[from=1-2, to=1-1]
\end{tikzcd}\]
where the right functor is induced by $1+\uvar:[n_1]\to[1+n_1]$ and where the left one sends an element 
$([n_0]^{op}\star [n_1]\to [n], [n_2]^{op}\star[n_3]\to [n_1])$ to the composite:
$h:[n_2]^{op}\star[n_3]\to [n_1]\to [n].$	
We define $H^2(a,n)$ as the pushout: 
\[\begin{tikzcd}
	{\underset{\Delta^2_{/[n]}}{\colim}~\underset{\Delta^2_{/[n_1]}}{\colim}~[[n_2]\otimes[n_0]\otimes a,1]\vee[[n_0]\otimes a,n_3]} & {\underset{\Delta^2_{/[n]}}{\colim}~[[n_2]\otimes a,1]\vee[ a,n_3]} \\
	{\underset{\Delta^2_{/[n]}}{\colim}~\underset{\Delta^2_{/[1+n_1]}}{\colim}~[[n_2]\otimes[n_0]\otimes a,1]\vee[[n_0]\otimes a,n_3]} & {H^2(a,n)}
	\arrow[from=1-1, to=2-1]
	\arrow[""{name=0, anchor=center, inner sep=0}, from=1-1, to=1-2]
	\arrow[from=1-2, to=2-2]
	\arrow[from=2-1, to=2-2]
	\arrow["\lrcorner"{anchor=center, pos=0.125, rotate=180}, draw=none, from=2-2, to=0]
\end{tikzcd}\]
By construction, we have a cocartesian square
\begin{equation}
\label{eq:lin H2 avec ee an}
\begin{tikzcd}[column sep = 0.3cm]                                                                                                                                                                                                                                                                                                                                                                                                                                                                                   
	{\coprod\limits_{l\leq 1+n_1}\underset{\Delta^2_{/[n]}}{\colim}~\underset{\Delta^2_{/\{l\}}}{\colim}~[[n_2]\otimes[n_0]\otimes a,1]\vee[[n_0]\otimes a,n_3]} & {H^2(a,n)\coprod\limits_{H^2(a,\amalg_{p\leq n}\{p\})}H^2(e,\amalg_{p\leq n}\{p\})} \\
	{\coprod\limits_{l\leq 1+n_1}\underset{\Delta^2_{/[n]}}{\colim}~\underset{\Delta^2_{/\{l\}}}{\colim}~[[n_2]\otimes e,1]\vee[e,n_3]} & {e\star e\star[a,n]}
	\arrow[from=1-1, to=2-1]
	\arrow[""{name=0, anchor=center, inner sep=0}, from=1-1, to=1-2]
	\arrow[from=2-1, to=2-2]
	\arrow[from=1-2, to=2-2]
	\arrow["\lrcorner"{anchor=center, pos=0.125, rotate=180}, draw=none, from=2-2, to=0]
\end{tikzcd}
\end{equation}
Let $x:= ([n_0]^{op}\star [n_1]\to [n], [n_2]^{op}\star[n_3]\to [1+n_1])$ be an element of $\underset{\Delta^2_{/[n]}}{\colim}~\Delta^2_{/[1+n_1]}$. We define two integers $-1\leq \tilde{n}_2\leq n_2$ and $-1\leq \tilde{n}_3\leq n_3$ as the ones fitting in the following pullbacks in $\Delta_+$
\[\begin{tikzcd}
	{ [\tilde{n}_2]^{op}} && {[n_1]} && {[\tilde{n}_3]} \\
	{[n_2]^{op}} & { [n_2]^{op}\star[n_3]} & {[1+n_1]} & { [n_2]^{op}\star[n_3]} & {[n_3]}
	\arrow[from=1-1, to=1-3]
	\arrow[from=1-3, to=2-3]
	\arrow[from=2-1, to=2-2]
	\arrow[from=2-2, to=2-3]
	\arrow[from=2-4, to=2-3]
	\arrow[from=2-5, to=2-4]
	\arrow[from=1-5, to=2-5]
	\arrow[from=1-5, to=1-3]
	\arrow[from=1-1, to=2-1]
	\arrow["\lrcorner"{anchor=center, pos=0.125}, draw=none, from=1-1, to=2-2]
	\arrow["\lrcorner"{anchor=center, pos=0.125, rotate=-90}, draw=none, from=1-5, to=2-4]
\end{tikzcd}\]
where we set the convention $[-1]=\emptyset$. This induces a cartesian square
\[\begin{tikzcd}
	{[n_0]^{op}\star [\tilde{n}_2]^{op}\star[\tilde{n}_3]} & {[n_0]^{op}\star [n_2]^{op}\star[n_3]} \\
	{[n_0]^{op}\star[n_1]} & {[n_0]^{op}\star[1+n_1]} \\
	{[n]}
	\arrow[from=1-1, to=2-1]
	\arrow[from=2-1, to=3-1]
	\arrow[""{name=0, anchor=center, inner sep=0}, from=2-1, to=2-2]
	\arrow[from=1-1, to=1-2]
	\arrow[from=1-2, to=2-2]
	\arrow["\lrcorner"{anchor=center, pos=0.125}, draw=none, from=1-1, to=0]
\end{tikzcd}\]

We consider the morphism $j:[n_2]\otimes [n_0]\otimes a \to ([n_2]\times [n_0])\otimes a\to ([\tilde{n}_2]\star [n_0])\otimes a$ where the right-hand morphism sends $\{(k,l)\}\otimes a$ to $(\{k\}\star \emptyset)\otimes a$ if $k\leq \tilde{n_2}$ and to $(\emptyset\star \{l\})\otimes a$ if not.
The inclusion $[\tilde{n}_3]\to [n_3]$ induces an inclusion $i:[1+\tilde{n}_3]\to [1+n_3]$. We denote $r$ the unique retraction of this inclusion that verifies $r(k) = 0$ if $k\notin Im(i)$.
Put together, $j$ and $r$ induce a morphism:
$$\psi_x:[[n_2]\otimes [n_0]\otimes a,1]\vee[[n_0]\otimes a,n_3] \to [([\tilde{n}_2]\star [n_0])\otimes a,1]\vee[a,\tilde{n}_3]$$
where we set the convention $[([\tilde{n}_2]\star [n_0])\otimes a,1]\vee[a,-1] := [0]$.

Remark that if $[n_2]^{op}\star [n_3]\to [1+n_1]$ factors through $[n_1]\to [1+n_1]$, we have $\tilde{n}_2=n_2$ and $\tilde{n}_3=n_3$, and a unique arrow fitting in a commutative triangle
\[\begin{tikzcd}
	& {[([\tilde{n}_2]\star \emptyset)\otimes a,1]\vee[a,\tilde{n}_3]} \\
	{[[n_2]\otimes [n_0]\otimes a,1]\vee[[n_0]\otimes a,n_3] } & {[([\tilde{n}_2]\star [n_0])\otimes a,1]\vee[a,\tilde{n}_3]}
	\arrow["{\psi_x}"', from=2-1, to=2-2]
	\arrow[from=1-2, to=2-2]
	\arrow[dashed, from=2-1, to=1-2]
\end{tikzcd}\]

Considering the canonical morphism
$$[([\tilde{n}_2]\star [n_0])\otimes a,1]\vee[a,\tilde{n}_3]\to e\star[a,n]$$
if $\tilde{n}_3\geq 0$ (coming from the fact that $([n_0]^{op}\star [\tilde{n}_2]^{op})\star[\tilde{n}_3]\to [n]$ is an element of $\Delta^2_{[n]}$),
and the morphism 
$$[([\tilde{n}_2]\star [n_0])\otimes a,1]\vee[a,\tilde{n}_3]\to e\star \emptyset\to e\star[a,n]$$
if $\tilde{n}_3=-1$, this induces a natural transformation 
$$H^{s^0}(a,n):H^2(a,n)\to e\star[a,n]$$
induced by $\psi_{\uvar}$ on $\underset{\Delta^2_{/[n]}}{\colim}~\underset{\Delta^2_{/[1+n_1]}}{\colim}~[[n_2]\otimes[n_0]\otimes a,1]\vee[[n_0]\otimes a,n_3]$ 
and by the identity on $\underset{\Delta^2_{/[n]}}{\colim}~[[n_2]\otimes a,1]\vee[ a,n_3]$.

By construction, if $[n_0]^{op}\star [n_1]\to [n]$ factor through $\{p\}$ for $p\leq n$ we have a commutative diagram
\[\begin{tikzcd}
	{[[n_2]\otimes[n_0]\otimes a,1]\vee[[n_0]\otimes a,n_3]} && {H^2(a,n)} \\
	{[[n_2]\otimes[n_0]\otimes e,1]\vee[[n_0]\otimes e,n_3]} & {e\star\{p\}} & { e\star[a,n]}
	\arrow[from=1-1, to=2-1]
	\arrow[from=1-1, to=1-3]
	\arrow[from=1-3, to=2-3]
	\arrow[from=2-1, to=2-2]
	\arrow[from=2-2, to=2-3]
\end{tikzcd}\]
If $[n_2]^{op}\star[n_3]\to [1+n_1]$ factors through $\{0\}$, $\tilde{n}_3$ is equal to $-1$, and we have a commutative diagram
\[\begin{tikzcd}
	{[[n_2]\otimes[n_0]\otimes a,1]\vee[[n_0]\otimes a,n_3]} && {H^2(a,n)} \\
	{[[n_2]\otimes e,1]\vee[e,n_3]} & e\star\emptyset & { e\star[a,n]}
	\arrow[from=1-1, to=2-1]
	\arrow[from=1-1, to=1-3]
	\arrow[from=2-1, to=2-2]
	\arrow[from=1-3, to=2-3]
	\arrow[from=2-2, to=2-3]
\end{tikzcd}\]
and if  $[n_2]^{op}\star[n_3]\to [1+n_1]$ factors through any other point, $\tilde{n}_3$ is equal to $0$, and we have a commutative diagram
\[\begin{tikzcd}
	{[[n_2]\otimes[n_0]\otimes a,1]\vee[[n_0]\otimes a,n_3]} && {H^2(a,n)} \\
	{[[n_2]\otimes e,1]\vee[e,n_3]} & {e\star\{k\}} & { e\star[a,n]}
	\arrow[from=1-1, to=2-1]
	\arrow[from=1-1, to=1-3]
	\arrow[from=2-1, to=2-2]
	\arrow[from=1-3, to=2-3]
	\arrow[from=2-2, to=2-3]
\end{tikzcd}\]
where $k$ is the image of the composite morphism $[\tilde{n}_2]^{op}\star[\tilde{n}_3]\to [n_1]\to[n]$.
The cocartesian square \eqref{eq:lin H2 avec ee an} then implies that $H^2(a,n)$ 
lifts to a natural transformation 
$$s^0\star [a,n]: e\star e\star [a,n]\to e \star [a,n].$$
By extension by colimits, this induces a natural transformation
$$C\mapsto \big(s^0\star C: e\star e\star C\to e\star C\big).$$

To define the cosimplicial object, we will need to show the commutativity of several diagrams whose initial objects are of shape $e\star..\star e\star [a,n]$. To this extend, it is enough to find coverings of these objects by easier one, and to show that the induced diagrams commute. 
\begin{lemma}
\label{lem:cover of n star a}
We set $\Pi^0_{/[n]} := \Delta^{2}_{/[n]}$ and 
$$\Pi^{k+1}_{/[n]}:= \colim_{\Delta^2_{/[n]}} \colim_{\Delta^2_{/[n_1+1]}}...\colim_{\Delta^2_{/[n_{2k-1}+1]}}\Delta^2_{/[n_{2k+1}+1]}$$
There is an epimorphism:
$$\colim_{\Pi^k_{/[n]}\times A}[[n_{2k}]\otimes [n_{2k-2}]\otimes...\otimes [n_0]\otimes a,1+ n_{2k-1}]\to \underbrace{e\star e\star ... \star e}_{k+1}\star [a,n]$$
\end{lemma}
\begin{proof}
This is an easy proof by induction, after remarking that $${[[n_0]\otimes a, 1+n_1]\to [[n_0]\otimes a, 1]\vee[a,n_1]}$$
 is an epimorphism.
\end{proof}

\begin{lemma}
\label{lemma:monoid 1}
The following triangles commute:
\[\begin{tikzcd}
	{e\star[a,n]} & {e\star e\star[a,n]} & {e\star [a,n]} \\
	& {e\star[a,n]}
	\arrow["{d^{0}\star{e\star[a,n]}}", from=1-1, to=1-2]
	\arrow["{s^0\star{[a,n]}}", from=1-2, to=2-2]
	\arrow["id"', from=1-1, to=2-2]
	\arrow["id", from=1-3, to=2-2]
	\arrow["{e\star d^{0}\star{[a,n]}}"', from=1-3, to=1-2]
\end{tikzcd}\]
\end{lemma}
\begin{proof}
We will prove only the left triangle and we leave the other to the reader. Let $x:= ([n_0]^{op}\star [n_1]\to [n], [n_2]^{op}\star[n_3]\to [1+n_1])$ be an element of $\colim_{\Delta^2_{/[n]}}~\Delta^2_{/[1+n_1]}$. We have a diagram:
\[\begin{tikzcd}
	& {e\star[a,n]} & {e\star e\star[a,n]} \\
	{[[n_0]\otimes a, 1+n_1]} & {[[0]\otimes[n_0]\otimes a, 1]\vee[[n_0]\otimes a, 1+n_1]} & {e\star[a,n]} \\
	& {[[n_0]\otimes a, 1+n_1]}
	\arrow["{d^0\star{e\star[a,n]}}", from=1-2, to=1-3]
	\arrow["{s^0\star{[a,n]}}", from=1-3, to=2-3]
	\arrow[from=2-1, to=1-2]
	\arrow["{[[n_0]\otimes a, d^0]}", from=2-1, to=2-2]
	\arrow[from=2-2, to=1-3]
	\arrow["{\psi_x}", from=2-2, to=3-2]
	\arrow[from=3-2, to=2-3]
	\arrow["id"', from=2-1, to=3-2]
	\arrow[dotted, from=1-2, to=2-3]
\end{tikzcd}\]
where we know that everything except the right triangle commutes. As this is true for any $x$, lemma \ref{lem:cover of n star a} implies the desired commutativity.
\end{proof}

\begin{lemma}
\label{lemma:monoid 2}
The following square commutes 
\[\begin{tikzcd}
	{e\star e\star e\star[a,n]} & {e\star e\star[a,n]} \\
	{e\star e\star[a,n]} & {e\star[a,n]}
	\arrow["{e\star s^1\star{[a,n]}}"', from=1-1, to=2-1]
	\arrow["{s^1\star{[a,n]}}"', from=2-1, to=2-2]
	\arrow["{s^1\star{e\star[a,n]}}", from=1-1, to=1-2]
	\arrow["{s^1\star{[a,n]}}", from=1-2, to=2-2]
\end{tikzcd}\]
\end{lemma}
\begin{proof}
Let 
$x = (f:[n_0]^{op}\star[n_1]\to[n],g:[n_2]^{op}\star[n_3]\to[1+n_1],h:[n_4]^{op}\star[n_5]\to[n_3+1]) $ be an object of $\Pi^2_k$. We define integers $-1\leq \bar{n}_4\leq n_4$ and $-1\leq \bar{n}_5\leq n_5$ as the one fitting in the following pullbacks in $\Delta_+$.
\[\begin{tikzcd}
	{[\bar{n}_4]^{op}} & {[1+\tilde{n}_3]} & {[\bar{n}_5]} \\
	{[n_4]^{op}} & { [1+n_3]} & {[n_5]}
	\arrow[""{name=0, anchor=center, inner sep=0}, from=1-2, to=2-2]
	\arrow[from=2-1, to=2-2]
	\arrow[from=1-1, to=1-2]
	\arrow[from=1-1, to=2-1]
	\arrow[from=1-3, to=2-3]
	\arrow[from=2-3, to=2-2]
	\arrow[from=1-3, to=1-2]
	\arrow["\lrcorner"{anchor=center, pos=0.125, rotate=-90}, draw=none, from=1-3, to=2-2]
	\arrow["\lrcorner"{anchor=center, pos=0.125}, draw=none, from=1-1, to=0]
\end{tikzcd}\]
This induces cartesian squares
\[\begin{tikzcd}[column sep=0.5cm]
	{[n_0]^{op}\star[\tilde{n}_2]^{op}\star[\tilde{\bar{n}}_4]^{op}\star[\tilde{\bar{n}}_5]} & {[n_0]^{op}\star[\tilde{n}_2]^{op}\star[\bar{n}_4]^{op}\star[\bar{n}_5]} & {[n_0]^{op}\star[n_2]^{op}\star[n_4]^{op}\star[n_5]} \\
	{[n_0]^{op}\star[\tilde{n}_2]^{op}\star[\tilde{n}_3]} & {[n_0]^{op}\star[\tilde{n}_2]^{op}\star[1+ \tilde{n}_3]} & {[n_0]^{op}\star[n_2]^{op}\star[1+ n_3]}
	\arrow[from=2-2, to=2-3]
	\arrow[from=2-1, to=2-2]
	\arrow[from=1-1, to=2-1]
	\arrow[from=1-1, to=1-2]
	\arrow[from=1-2, to=2-2]
	\arrow[from=1-3, to=2-3]
	\arrow[from=1-2, to=1-3]
\end{tikzcd}\]
The outer squares fits in the following cartesian squares:
\[\begin{tikzcd}[column sep=0.5cm]
	{[n_0]^{op}\star[\tilde{n}_2]^{op}\star[\tilde{\bar{n}}_4]^{op}\star[\tilde{\bar{n}}_5]} & {[n_0]^{op}\star[n_2]^{op}\star[\tilde{n}_4]^{op}\star[\tilde{n}_5]} & {[n_0]^{op}\star[n_2]^{op}\star[n_4]^{op}\star[n_5]} \\
	{[n_0]^{op}\star[\tilde{n}_2]^{op}\star[\tilde{n}_3]} & {[n_0]^{op}\star[n_2]^{op}\star[n_3]} & {[n_0]^{op}\star[n_2]^{op}\star[1+ n_3]} \\
	{[n_0]^{op}\star[n_1]} & {[n_0]^{op}\star[1+n_1]} \\
	{[n]}
	\arrow[from=2-2, to=2-3]
	\arrow[from=2-1, to=2-2]
	\arrow[from=3-1, to=3-2]
	\arrow[from=3-1, to=4-1]
	\arrow[from=2-1, to=3-1]
	\arrow[from=2-2, to=3-2]
	\arrow[from=1-3, to=2-3]
	\arrow[from=1-1, to=2-1]
	\arrow[from=1-2, to=2-2]
	\arrow[from=1-1, to=1-2]
	\arrow[from=1-2, to=1-3]
\end{tikzcd}\]
This induces a diagram:
\[\begin{tikzcd}[column sep=0.5cm]
	& {e\star e\star e\star[a,n]} & {e\star e\star[a,n]} \\
	& {e\star e\star[a,n]} & {e\star[a,n]} \\
	{[[n_{4}]\otimes [n_{2}]\otimes [n_0]\otimes a,1+ n_{5}]} & {[([\tilde{n}_{4}]\star [n_{2}])\otimes [n_0]\otimes a,1+ \tilde{n}_{5}]} \\
	{[[\bar{n}_{4}]\otimes ( [\tilde{n}_{2}]\star [n_0])\otimes a,1+ \bar{n}_{5}]} & {[[\tilde{\bar{n}}_4]\star [\tilde{n}_{2}]\star [n_0]\otimes a,1+ \tilde{\bar{n}}_5]}
	\arrow[from=4-1, to=4-2]
	\arrow[from=3-1, to=4-1]
	\arrow[from=3-2, to=4-2]
	\arrow[from=3-1, to=3-2]
	\arrow["{s^1\star{[a,n]}}", from=1-3, to=2-3]
	\arrow["{s^1\star{[a,n]}}"', dotted, from=2-2, to=2-3]
	\arrow["{s^1\star{e\star[a,n]}}", from=1-2, to=1-3]
	\arrow["{e\star s^1\star{[a,n]}}", from=1-2, to=2-2]
	\arrow[from=3-1, to=1-2]
	\arrow[from=3-2, to=1-3]
	\arrow[from=4-2, to=2-3]
	\arrow[dotted, from=4-1, to=2-2]
\end{tikzcd}\]
where we know that everything except the behind square commutes. As this is true for any $x$, lemma \ref{lem:cover of n star a} implies the desired commutativity.

\end{proof}

\begin{definition}
For $k\leq 1$, the \wcsnotionsym{intelligent $k$-truncation functor}{(taui@$\tau^i_n$}{truncation@$n$-truncation}{for stratified Segal $A$-precategories}, noted by $\tau^i_k$, is the colimit preserving functor such that $\tau^i_k([a,n]) = [\tau^i_{k-1}(a),n]$ and $\tau^i_k[e,1]_t = [e,1]_t$. The intelligent \textit{$0$-truncation functor}, denoted by $\tau^i_0$, is the colimit preserving functor such that $\tau^i_0([a,n])$ fits in the following pushout 
\[\begin{tikzcd}
	{\underset{ob(a)\times Hom([1],[n])}{\coprod}[e,1]} & {[\tau^i_0(a),n]} \\
	{\underset{ob(a)\times Hom([1],[n])}{\coprod}[e,1]_t} & {\tau^i_0([a,n])}
	\arrow[from=1-1, to=1-2]
	\arrow[from=1-2, to=2-2]
	\arrow[from=1-1, to=2-1]
	\arrow[from=2-1, to=2-2]
\end{tikzcd}\]
and such that $\tau^i_0[e,1]_t = [e,1]_t$.
As the intelligent $k$-truncations on $A$ are left Quillen, the intelligent $k$-truncations on $\stratSeg(A)$ preserve generating Reedy cofibrations and Segal extensions. It is staightforward that they also send $[e,1]_t\to [0]$ and $E^{\cong}\to (E^{\cong})'$ to weak equivalences. According to theorem \ref{prop:model structure on stratified Segal category}, they are left Quillen functors. 
\end{definition}

\p 
\label{para:definition of the cosimplicial object}
The (inverted) composition $g,f\mapsto g\circ f$ is a monoidal structure on the category of endomorphisms of $\stratSeg(A)$. Lemmas \ref{lemma:monoid 1} and \ref{lemma:monoid 2} show that $e\star \uvar$ is a monoid for this monoidal structure. This induces a cosimplicial object: 
$$\begin{array}{rcl}
\Delta &\to & \End(\stratSeg(A))\\
~[n] &\mapsto & [n]\star\uvar :=\underbrace{e\star e\star...\star e}_{n+1}\star \uvar
\end{array}$$
We extend this functor to $\Delta_t$ in setting for a stratified Segal $A$-precategory $C$ and an integer $n>0$:
\[\begin{tikzcd}
	{\underset{k\geq -1}{\coprod}~~\underset{D,~\tau^i_k(D)=D}{\coprod}~~\underset{D\to C}{\coprod}[n]\star D} & {[n]\star C} \\
	{\underset{k\geq -1}{\coprod}~~\underset{D,~\tau^i_k(D)=D}{\coprod}~~\underset{D\to C}{\coprod}\tau^i_{n+k}([n]\star D)} & {[n]_t\star C}
	\arrow[from=2-1, to=2-2]
	\arrow[""{name=0, anchor=center, inner sep=0}, from=1-1, to=1-2]
	\arrow[from=1-2, to=2-2]
	\arrow[from=1-1, to=2-1]
	\arrow["\lrcorner"{anchor=center, pos=0.125, rotate=180}, draw=none, from=2-2, to=0]
\end{tikzcd}\]
where $\tau^i_{-1}$ is the constant functor with value $\emptyset$.
Evaluated on the empty Segal $A$-category, and by extension under colimits, this gives a functor 
\begin{equation}
\stratSset\to \stratSeg(A).
\end{equation}
The image of $[n]$ (resp. $[n]_t$) is also noted by $[n]$ (resp. $[n]_t$).

By construction, for $K,L$ two stratified sets and $D$ a stratified Segal $A$-precategory, we have $K\star (L\star C)\cong (K\star L)\star C$.

\begin{lemma}
\label{lemma:leibnizt joint is Quillen}
Let $K$ be a stratified simplicial set. The morphism $K\star\uvar$ is a left Quillen functor. 
Moreover, if
 $i$ is a cofibration of stratified simplicial sets and $g$ an acyclic cofibration of stratified Segal $A$-precategories, the morphism $i\hstar g$ is an acyclic cofibration. 
\end{lemma}
\begin{proof}
As every simplicial set is a homotopy colimit of representables and as $\star$ preserves monomorphisms, it is enough to show the first assertion for $K= [n]$. In this case, this is a repeated application of the corollary \ref{cor:cone is Quillen}.
By diagram chasing and the use of two out of three, this implies the second assertion.
\end{proof}

\subsection{Complicial horn inclusions}
\label{section:Complicial horn inclusion}

\begin{notation*}
In this section, we will often consider morphisms $\tilde{a}\to \tilde{b}$ that fit into cocartesian squares:
\[\begin{tikzcd}
	a & b \\
	{\tilde{a}} & {\tilde{b}}
	\arrow[from=1-1, to=2-1]
	\arrow["i", from=1-1, to=1-2]
	\arrow[from=2-1, to=2-2]
	\arrow[from=1-2, to=2-2]
	\arrow["\lrcorner"{anchor=center, pos=0.125}, draw=none, from=1-1, to=2-2]
\end{tikzcd}\]
where $a\to \tilde{a}$ and $b\to \tilde{b}$ are epimorphisms.
To avoid complicating the notations unnecessarily, the induced morphism $\tilde{a}\to \tilde{b}$ will just be denoted $i$.
\end{notation*}

\p 
\label{para:marked segal}
A \notion{marked Segal $A$-precategory} is a stratified Segal $A$-precategory having the right lifting property against all entire acyclic cofibrations. We denote by \wcnotation{$\mSeg(A)$}{(mseg@$\mSeg(A)$} the full subcategory of marked Segal $A$-precategory. We then have an adjunction: 
\[\begin{tikzcd}
	{(\uvar)_{\mk}:\stratSeg(A)} & {\mSeg(A):\iota}
	\arrow[""{name=0, anchor=center, inner sep=0}, shift left=2, from=1-2, to=1-1]
	\arrow[""{name=1, anchor=center, inner sep=0}, shift left=2, from=1-1, to=1-2]
	\arrow["\dashv"{anchor=center, rotate=-90}, draw=none, from=1, to=0]
\end{tikzcd}\]
where the left adjoint $(\uvar)_{\mk}$ sends a stratified Segal $A$-precategory $(C,tC)$ to the marked Segal $A$-precategory $(C,\overline{tC})$, where $\overline{tC}$ is the smaller stratification that includes $tC$ and makes $(C,\overline{tC})$ a marked Segal $A$-precategory, and where the right adjoint is a fully faithful inclusion.
Remark furthermore that at the level of preshaves, these two adjoints are the identity. We denote \wcnotation{$r_C:C\to C_{\mk}$}{(rc@$r_C:C\to C_{\mk}$} the canonical inclusion. The proposition \ref{prop:X to Xmk is acycli cof} states that $r_C$ is an entire acyclic cofibration.

There is an isomorphism $(e\star C_{\mk})_{\mk}\cong (e\star C)_{\mk}$. Indeed $e\star\uvar$ preserves both entire cofibrations and weak equivalences, we have two entire acyclic cofibration $e\star C\to (e\star C)_{\mk}$ and $e \star C\to (e\star C_{\mk})_{\mk}$.  As the two codomain are marked, they are isomorphic.

The fact that will be used the most with the marked Segal $A$-precategory is their right lifting property with respect to morphisms of shape $[\tau^i_n(a),\Lambda^1[2]]\cup [a,2]\to [\tau^i_n(a),2]$. This fact will  be used freely.

\p We recall that $[2]\botimes a$ is the following pushout:
\[\begin{tikzcd}
	{[1]\otimes a\amalg [1]\otimes a} && {[2]\otimes a} \\
	{e\star a\amalg e\star a} && {[2]\botimes a}
	\arrow[""{name=0, anchor=center, inner sep=0}, "{d^1\otimes a\amalg d^2\otimes a}", from=1-1, to=1-3]
	\arrow["{d^1\botimes a\amalg d^2\botimes a}"', from=2-1, to=2-3]
	\arrow[from=1-1, to=2-1]
	\arrow[from=1-3, to=2-3]
	\arrow["\lrcorner"{anchor=center, pos=0.125, rotate=180}, draw=none, from=2-3, to=0]
\end{tikzcd}\]
We define $[e,1]\vee(e\star[a,1])$ as the colimit of the following diagram
\[\begin{tikzcd}
	{[e,1]\vee[e\star a,1]} & {[e,1]\vee[a,1]} & {[e,2]\vee[a,1]}
	\arrow["{[d^0\star a,2]}"', from=1-2, to=1-1]
	\arrow["{[a,d^2]}", from=1-2, to=1-3]
\end{tikzcd}\]
The canonical composite morphism 
$$[e\star a,1]\xrightarrow{[e\star a,d^1]}[e,1]\vee[e\star a,1]\to [e,1]\vee(e\star[a,1])$$ 
is also denoted by $[e\star a,d^1]$. Eventually, we define $\overline{[1]\star[a,1]}$ as the following pushout
\[\begin{tikzcd}
	{[1]\star\{0\}} & {{[1]\star[a,1]}} \\
	{[2]_t} & {\overline{[1]\star[a,1]}}
	\arrow[from=1-1, to=2-1]
	\arrow[from=2-1, to=2-2]
	\arrow[from=1-2, to=2-2]
	\arrow[from=1-1, to=1-2]
	\arrow["\lrcorner"{anchor=center, pos=0.125, rotate=180}, draw=none, from=2-2, to=1-1]
\end{tikzcd}\] 

\begin{lemma}
\label{lemma:le lemme quon voulais pas faire}
There is a weak equivalence from $\overline{[1]\star[a,1]}$ to the colimit of the diagram
\[\begin{tikzcd}
	{[[1]\star a,1]} & {[e\star a,1]} & {[e,1]\vee(e\star[ a,1])}
	\arrow["{[e\star a,d^1]}", from=1-2, to=1-3]
	\arrow["{[d^0\star a,1]}"', from=1-2, to=1-1]
\end{tikzcd}\]
making $\overline{[1]\star[a,1]}$ the homotopy colimit of the previous diagram.
\end{lemma}
\begin{proof}
The proposition \ref{prop:explicit expression of e star e star a,1} implies that $(\overline{[1]\star [a,1]})_{\mk}$ is the colimit of the diagram
\begin{equation}
\label{eq:changemeet markage2}
\begin{tikzcd}[column sep=0.7cm]
	{[[2]^2\botimes a,1]} & {[[1]_t\otimes a,1]} & {[[1]_t,1]\vee[a,1]} & {[e,1]\vee[a,1]} & {[e,2]\vee[a,1]} \\
	{[e\star a,1]} &&& {[a,1]} & {[e,1]\vee[a,1]} \\
	{[[1]\star a,1]} &&& {[e\star a,1]} & {[e,1]\vee[e\star a,1]}
	\arrow["{[d^0\otimes a,1]}"', from=1-2, to=1-1]
	\arrow["{[[1]\otimes a,d^1]}", from=1-2, to=1-3]
	\arrow["{[d^0\otimes a,2]}"', from=1-4, to=1-3]
	\arrow["{[a,d^1]}", from=1-4, to=1-5]
	\arrow["{[a,d^1]}", from=2-4, to=2-5]
	\arrow["{[d^1\botimes a,1]}", from=2-1, to=1-1]
	\arrow["{[a,d^1]}", from=2-4, to=1-4]
	\arrow["{[a,d^2]}"', from=2-5, to=1-5]
	\arrow["{[d^{0}\star a,2]}", from=2-5, to=3-5]
	\arrow["{[d^{0}\star a,1]}"', from=2-4, to=2-1]
	\arrow["{[d^{0}\star a,1]}"', from=2-4, to=3-4]
	\arrow["{[e\star a,d^1]}"', from=3-4, to=3-5]
	\arrow["{[d^1\star a,1]}"', from=2-1, to=3-1]
	\arrow["{[d^0\star a,1]}", from=3-4, to=3-1]
\end{tikzcd}
\end{equation}
In the previous diagram, the fact that we have $[[1]_t\otimes a,1]$  instead of $[[1]\otimes a,1]$ comes from the fact that we have considered $(\overline{[1]\star [a,1]})_{\mk}$ instead of $\overline{[1]\star [a,1]}$.

Consider now the morphism
\begin{equation}
\label{eq:changemeet markage}
[[2]^2\botimes a,1]\coprod_{[[1]_t\otimes a,1]}   [[1]_t,1]\vee[a,1]\to e\star[a,1]
\end{equation}
induces by the vertical colimit of the diagram
\begin{equation}
\label{eq:changemeet markage5}
\begin{tikzcd}
	{[[2]^2\botimes a,1]} & {[[1]_t\otimes a,1]} & {[[1]_t,1]\vee[a,1]} \\
	{[e\star a,1]} & {[a,1]} & {[e,1]\vee[a,1]}
	\arrow["{[d^0\otimes a,1]}"', from=1-2, to=1-1]
	\arrow["{[[1]\otimes a,d^1]}", from=1-2, to=1-3]
	\arrow["{[s^0,1]\vee[a,1]}", from=1-3, to=2-3]
	\arrow[from=2-2, to=2-3]
	\arrow[from=2-2, to=2-1]
	\arrow["{[s^0\otimes a,1]}", from=1-2, to=2-2]
	\arrow["{[s^0\botimes a,1]}"', from=1-1, to=2-1]
\end{tikzcd}
\end{equation}
As all the vertical morphisms of \eqref{eq:changemeet markage5} are cofibrations, the colimit of each line is a homotopy colimit. As all the horizontal morphisms of \eqref{eq:changemeet markage5} are weak equivalences, the morphism \eqref{eq:changemeet markage} also is a weak equivalence.  

Consider now the span
\begin{equation}
\label{eq:deojfoizejfgorezj}
 e\star[a,1]\xleftarrow{\eqref{eq:changemeet markage}} [[2]^2\botimes a,1]\coprod_{[[1]_t\otimes a,1]}   [[1]_t,1]\vee[a,1]\to (\overline{[1]\star [a,1]})_{\mk}
 \end{equation}
As the right hand morphism is a cofibration, and as  \eqref{eq:changemeet markage} is a weak equivalence, the canonical morphism from 
$(\overline{[1]\star [a,1]})_{\mk}$ to the colimit of \eqref{eq:deojfoizejfgorezj} is a weak equivalence.
Using the diagram \eqref{eq:changemeet markage2}, the colimit of \eqref{eq:deojfoizejfgorezj} is also the colimit of the following diagram
\[\begin{tikzcd}
	{e\star[a,1]} & {[e,1]\vee[a,1]} & {[e,2]\vee[a,1]} \\
	{[e\star a,1]} & {[a,1]} & {[e,1]\vee[a,1]} \\
	{[[1]\star a,1]} & {[e\star a,1]} & {[e,1]\vee[e\star a,1]}
	\arrow["{[a,d^1]}", from=1-2, to=1-3]
	\arrow["{[a,d^1]}", from=2-2, to=2-3]
	\arrow["{[a,d^1]}", from=2-2, to=1-2]
	\arrow["{[a,d^2]}"', from=2-3, to=1-3]
	\arrow["{[d^{0}\star a,2]}", from=2-3, to=3-3]
	\arrow["{[d^{0}\star a,1]}"', from=2-2, to=2-1]
	\arrow["{[d^{0}\star a,1]}"', from=2-2, to=3-2]
	\arrow["{[e\star a,d^1]}"', from=3-2, to=3-3]
	\arrow["{[d^1\star a,1]}"', from=2-1, to=3-1]
	\arrow["{[d^0\star a,1]}", from=3-2, to=3-1]
	\arrow[from=1-2, to=1-1]
	\arrow[from=2-1, to=1-1]
	\arrow["\lrcorner"{anchor=center, pos=0.125}, draw=none, from=1-1, to=2-2]
\end{tikzcd}\]
As the  upper left square is cocartesian, the colimit of the previous diagram is equivalent to the colimit of the  given diagram.
All put together, we have demonstrated the assertion.
\end{proof}

\begin{lemma}
\label{lemma:le lemme quon voulais pas faire2}
The morphism 
$$[e,1]\vee(e\star[a,1])\cup \{1\}\star [e\star a,1]\to[e,1]\vee(e\star[e\star a,1])$$
is a weak equivalence. 
\end{lemma}
\begin{proof}
We have a cocartesian square
\begin{equation}
\label{eq:lemma:le lemme quon voulais pas faire2}
\begin{tikzcd}
	{[e,1]\cup e\star[ a,1]} & {[e,1]\cup e\star[e\star a,1]} \\
	{[e,1]\vee(e\star[a,1])} & {[e,1]\vee(e\star[a,1])\cup \{1\}\star [e\star a,1]}
	\arrow[from=1-1, to=2-1]
	\arrow[from=2-1, to=2-2]
	\arrow["{[e,1]\cup e\star[d^0\star a,1]}", from=1-1, to=1-2]
	\arrow[from=1-2, to=2-2]
\end{tikzcd}
\end{equation}
Remark that the left vertical morphism is the vertical colimit and homotopy colimit of the diagram
\[\begin{tikzcd}
	{[e,1]\cup[e\star a,1]} & {[e,1]\cup[a,1]} & {[e,1]\cup[e,1]\vee[a,1]} \\
	{[e,1]\vee[e\star a,1]} & {[e,1]\vee[a,1]} & {[e,2]\vee[a,1]}
	\arrow[from=2-2, to=2-1]
	\arrow[from=2-2, to=2-3]
	\arrow[from=1-2, to=1-1]
	\arrow[from=1-2, to=1-3]
	\arrow[from=1-2, to=2-2]
	\arrow[from=1-3, to=2-3]
	\arrow[from=1-1, to=2-1]
\end{tikzcd}\]
and is then a weak equivalence. Similarly, $[e,1]\cup e\star[e\star a,1]\to [e,1]\vee(e\star[e\star a,1])$ is a weak equivalence. 
This implies that the right vertical morphism of \eqref{eq:lemma:le lemme quon voulais pas faire2} is a weak equivalence. By two out of three this concludes the proof.
\end{proof}

\begin{lemma}
\label{lem:outer horn inclusion2}
The morphism $\{1\}\star [0]\to [1]_t\star [0]$ is an acyclic cofibration.
\end{lemma}
\begin{proof}
Using proposition \ref{prop:explicit expression of e star a,1} we deduce that $[1]_t\star [0]$ is the colimit of the diagram 
\[\begin{tikzcd}
	{[[1]_t,1]} & {[e,1]} & {[e,1]_t\vee[e,1]}
	\arrow[from=1-2, to=1-3]
	\arrow[from=1-2, to=1-1]
\end{tikzcd}\]
The inclusion $\{1\}\star [0]\to [1]_t\star [0]$ is then the composite of the following sequence
\[\begin{tikzcd}
	& {[e,1]} & {[[1]_t,1]} \\
	{[e,1]} & {[e,1]_t\vee[e,1]} & {[1]_t\star [0]}
	\arrow["{[e,d^0]}", from=2-1, to=2-2]
	\arrow[from=1-2, to=2-2]
	\arrow[from=1-3, to=2-3]
	\arrow["{[d^0,1]}", from=1-2, to=1-3]
	\arrow[from=2-2, to=2-3]
	\arrow["\lrcorner"{anchor=center, pos=0.125, rotate=180}, draw=none, from=2-3, to=1-2]
\end{tikzcd}\]
As the morphism $[e,d^0]$ and $[d^0,1]$ are acyclic cofibrations, this concludes the proof.
\end{proof}

\begin{lemma}
\label{lem:outer horn inclusion1}
The morphism $\{1\}\star[a,1]\to [1]_t\star[a,1]$ is an acyclic cofibration.
\end{lemma}
\begin{proof}
The Segal $A$-precategory $[1]_t\star [a,1]$ is the colimit  and the homotopy colimit of the diagram
\[\begin{tikzcd}
	{[1]\star\emptyset} && {[a\star[1],1]} \\
	{[1]_t\star\emptyset} & {\overline{[1]\star[a,1]}} & {[a\star[1]_t,1]}
	\arrow[from=1-3, to=2-2]
	\arrow[from=1-3, to=2-3]
	\arrow[from=1-1, to=2-1]
	\arrow[from=1-1, to=2-2]
\end{tikzcd}\]
The lemma \ref{lemma:le lemme quon voulais pas faire} then implies that we have a weak equivalence from $[1]_t\star [a,1]$ to the colimit, denoted by $K$, of the diagram 
\[\begin{tikzcd}
	{[[1]_t\star a,1]} & {[e\star a,1]} & {[e,1]_t\vee(e\star [ a,1])}
	\arrow["{[e\star a,d^1]}", from=1-2, to=1-3]
	\arrow["{[d^0\star a,1]}"', from=1-2, to=1-1]
\end{tikzcd}\]
As all the morphisms are cofibrations, $K$ is also the homotopy colimit of the previous diagram.

The morphism $[e,1]_t\vee(e\star [ a,1])\to e\star [ a,1]$ is a  weak equivalence as it is a homotopy colimit of weak equivalences. 
Moreover, the morphism $[[1]_t\star a,1]\to [e\star a,1]$ is also a weak equivalence. This implies that the composite $s^0\star[a,1]:[1]_t\star [a,1]\to K\to [0]\star[a,1]$ is a weak equivalence.
The morphism $\{1\}\star[a,1]\to [1]_t\star[a,1]$ is a section of $s^0\star[a,1]$ and is then also a weak equivalence.
\end{proof}

\begin{lemma}
\label{lem:horn_inclusion_2}
The morphism 
$\Lambda^1[2]\star [0]\to [2]_t\star [0]$
is an acyclic cofibration.
\end{lemma}
\begin{proof}
The Segal $A$-precategory $[2]_t\star [0]$ is the colimit of the following diagram
\[\begin{tikzcd}
	{[[2]_t,1]} & {[[2],1]} & {\overline{[1]\star[1]}}
	\arrow[from=1-2, to=1-1]
	\arrow[from=1-2, to=1-3]
\end{tikzcd}\]
The lemma \ref{lemma:le lemme quon voulais pas faire} then implies that we have a weak equivalence from $[2]_t\star [0]$ to the colimit, denoted by $K$, of the diagram 
\[\begin{tikzcd}
	{[[2]_t,1]} & {[[1],1]} & {[e,1]\vee(e\star[e,1])}
	\arrow["{[[1],d^1]}", from=1-2, to=1-3]
	\arrow["{[d^{0},1]}"', from=1-2, to=1-1]
\end{tikzcd}\]
On the other side, $\Lambda^1[2]\star [0]$ is the colimit of the diagram 
\[\begin{tikzcd}
	&&& {[e,1]} \\
	{[[1],1]} & {[e,1]} & {[e,2]} && {[[1],1]} & {[e,1]} & {[e,2]}
	\arrow["{[d^0,1]}"', from=2-6, to=2-5]
	\arrow["{[d^1,1]}", from=1-4, to=2-5]
	\arrow["{[e,d^1]}", from=2-3, to=2-2]
	\arrow["{[d^0,1]}", from=2-2, to=2-1]
	\arrow["{[e,d^0]}"', from=1-4, to=2-3]
	\arrow["{[e,d^1]}", from=2-6, to=2-7]
\end{tikzcd}\]
The composite $\Lambda^1[2]\star [0]\to [2]_t\star [0]\to K$ fits in the sequence of acyclic cofibrations
\[\begin{tikzcd}[column sep =0.3cm]
	{[e,d^0]\cup[e,d^2]} & {[e,3]} & {[\Lambda^1[2],1]} & {[[2]_t,1]} \\
	{\Lambda^1[2]\star [0]} & \bullet & \bullet & K \\
	& {([e,1]\cup[[1],1])\cup[e,1]\vee[\partial[1],1]} & {[e,1]\vee[[1],1]}
	\arrow[from=1-1, to=2-1]
	\arrow[""{name=0, anchor=center, inner sep=0}, from=1-1, to=1-2]
	\arrow[from=1-2, to=2-2]
	\arrow[from=2-1, to=2-2]
	\arrow[from=1-3, to=1-4]
	\arrow["\lrcorner"{anchor=center, pos=0.125, rotate=180}, draw=none, from=2-4, to=1-3]
	\arrow[from=2-3, to=2-4]
	\arrow[""{name=1, anchor=center, inner sep=0}, from=2-2, to=2-3]
	\arrow[from=1-4, to=2-4]
	\arrow[from=1-3, to=2-3]
	\arrow[from=3-2, to=2-2]
	\arrow[from=3-2, to=3-3]
	\arrow[from=3-3, to=2-3]
	\arrow["\lrcorner"{anchor=center, pos=0.125, rotate=180}, draw=none, from=2-2, to=0]
	\arrow["\lrcorner"{anchor=center, pos=0.125, rotate=180}, draw=none, from=3-3, to=1]
\end{tikzcd}\]
and is then a weak equivalence.
By two out of three, this concludes the proof.
 \end{proof}

\begin{lemma}
\label{lem:horn_inclusion_3}
The morphism 
$\Lambda^1[2]\star [a,1]\to [2]_t\star [a,1]$
is an acyclic cofibration.
\end{lemma}
\begin{proof}
The lemma \ref{lem:horn_inclusion_2} implies that the inclusion $\Lambda^1[2]\star [a,1]\to \Lambda^1[2]\star [a,1]\cup [2]_t\star \{0\}$ is an acyclic cofibration.
Using proposition \ref{prop:explicit expression of e star a,1}, we deduce that the
 Segal $A$-precategory $[2]_t\star[a,1]$ is the colimit of the diagram
\[\begin{tikzcd}
	{[1]\star[e,1]} && {[1]\star[a,1]} \\
	{\overline{\overline{[1]\star[e,1]}}} & {[1]\star([e,1]\vee[a,1])} & {\overline{\overline{[1]\star[e\star a,1]}}}
	\arrow["{[1]\star[d^0\star a,1]}", from=1-3, to=2-3]
	\arrow["{[1]\star[a,d^1]}"', from=1-3, to=2-2]
	\arrow[from=1-1, to=2-1]
	\arrow[from=1-1, to=2-2]
\end{tikzcd}\]
while $ \Lambda^1[2]\star [a,1]\cup [2]_t\star \{0\}$ is the colimit  of the diagram
\[\begin{tikzcd}
	{\{1\}\star[e,1]} && {\{1\}\star[a,1]} \\
	{\overline{\overline{[1]\star[e,1]}}} & {\{1\}\star([e,1]\vee[a,1])\cup[1]\star[a,1]} & {\{1\}\star[e\star a,1]}
	\arrow["{\{1\}\star[a,d^1]}"', from=1-3, to=2-2]
	\arrow["{\{1\}\star[d^0\star a,1]}", from=1-3, to=2-3]
	\arrow[from=1-1, to=2-1]
	\arrow[from=1-1, to=2-2]
\end{tikzcd}\]
where $\overline{\overline{[1]\star[e,1]}}:=[2]_t\star [0]$ and where $\overline{\overline{[1]\star[e\star a,1]}}$ is the following pushout:
\[\begin{tikzcd}
	{ [[2]\star a,1]} & {e\star [[1]\star a,1]} & {\overline{[1]\star [e\star a,1]}} \\
	{ [[2]_t\star a,1]} && {\overline{\overline{[1]\star[e\star a,1]}}}
	\arrow[from=1-1, to=1-2]
	\arrow[from=1-2, to=1-3]
	\arrow[from=1-1, to=2-1]
	\arrow[from=2-1, to=2-3]
	\arrow[from=1-3, to=2-3]
	\arrow["\lrcorner"{anchor=center, pos=0.125, rotate=180}, draw=none, from=2-3, to=1-2]
\end{tikzcd}\]
Let $K_1$ be the following pushout:
\[\begin{tikzcd}
	{\{1\}\star ([e,1]\vee [a,1]) \cup [1]\star ([e,1]\cup [a,1]) } & {\Lambda^1[2]\star[a,1]\cup [2]_t\star\{0\}} \\
	{[1]\star ([e,1]\vee[a,1])} & {K_1}
	\arrow[from=1-1, to=2-1]
	\arrow[""{name=0, anchor=center, inner sep=0}, from=1-1, to=1-2]
	\arrow[from=2-1, to=2-2]
	\arrow[from=1-2, to=2-2]
	\arrow["\lrcorner"{anchor=center, pos=0.125, rotate=180}, draw=none, from=2-2, to=0]
\end{tikzcd}\]
The left-hand morphism is equal to $(d^0:[0]\to [1])\hstar ([e,1]\cup[a,1]\to [e,1]\vee[a,1])$ which is an acyclic cofibration according to lemma \ref{lemma:leibnizt joint is Quillen}.
Furthermore, the morphism $K_1\to [2]_t\star [a,1]$ fits in the following pushout:
\[\begin{tikzcd}
	{\overline{[1]\star[ a,1]}\cup\{1\}\star[e\star a,1]} & {K_1} \\
	{\overline{\overline{[1]\star[e\star a,1]}}} & {[2]_t\star[a,1]}
	\arrow[from=1-2, to=2-2]
	\arrow[""{name=0, anchor=center, inner sep=0}, from=1-1, to=1-2]
	\arrow[from=1-1, to=2-1]
	\arrow[from=2-1, to=2-2]
	\arrow["\lrcorner"{anchor=center, pos=0.125, rotate=180}, draw=none, from=2-2, to=0]
\end{tikzcd}\]
The lemma \ref{lemma:le lemme quon voulais pas faire} implies that we have a weak equivalence from $\overline{[1]\star[ a,1]}\cup\{1\}\star[e\star a,1]$ to the colimit, denoted by $K_2$, of the diagram 
\[\begin{tikzcd}
	{[[1]\star a,1]} & {[e\star a,1]} & {[e,1]\vee(e\star[a,1])\cup \{1\}\star[e\star a,1]}
	\arrow["{[e\star a,d^1]}", from=1-2, to=1-3]
	\arrow["{[d^0\star a,1]}"', from=1-2, to=1-1]
\end{tikzcd}\]
As all the morphisms are cofibrations, $K_2$ is also the homotopy colimit of the previous diagram. We now define $K_3$ as the colimit of the diagram
\[\begin{tikzcd}
	{[\Lambda^1[2]\star  a,1]} & {[[1]\star a,1]} & {[e,1]\vee(e\star[e\star a,1])}
	\arrow["{[d^0\star a,1]}"', from=1-2, to=1-1]
	\arrow["{[[1]\star a,d^1]}", from=1-2, to=1-3]
\end{tikzcd}\]
The canonical morphism $K_2\to K_3$ fits in the cocartesian square
\[\begin{tikzcd}
	{[e,1]\vee(e\star[a,1])\cup \{1\}\star[e\star a,1]} & {K_2} \\
	{[e,1]\vee(e\star[e\star a,1])} & {K_3}
	\arrow[from=1-1, to=2-1]
	\arrow[from=2-1, to=2-2]
	\arrow[""{name=0, anchor=center, inner sep=0}, from=1-1, to=1-2]
	\arrow[from=1-2, to=2-2]
	\arrow["\lrcorner"{anchor=center, pos=0.125, rotate=180}, draw=none, from=2-2, to=0]
\end{tikzcd}\]
and is then a weak equivalence according to the lemma \ref{lemma:le lemme quon voulais pas faire2}.

On the other side, the lemma \ref{lemma:le lemme quon voulais pas faire} also implies that we have a weak equivalence from $\overline{\overline{[1]\star[e\star a,1]}}$ to the colimit, denoted by $K_4$, of the diagram
\[\begin{tikzcd}
	{[[2]_t\star  a,1]} & {[[1]\star a,1]} & {[e,1]\vee(e\star[e\star a,1])}
	\arrow["{[d^0\star a,1]}"', from=1-2, to=1-1]
	\arrow["{[[1]\star a,d^1]}", from=1-2, to=1-3]
\end{tikzcd}\]
As all the morphisms are cofibrations, $K_4$ is also the homotopy colimit of the previous diagram.
As $\Lambda^1[2]\star a \to [2]_t\star a$ is a weak equivalence in $A$, this implies that the canonical morphism $K_3\to K_4$ is also a weak equivalence. We then have  commutative diagram:
\[\begin{tikzcd}
	{[1]\star[a,1]\cup \{1\}\star[e\star a,1]} && {\overline{[1]\star[e\star a,1]}} \\
	{K_2} & {K_3} & {K_4}
	\arrow["\sim"', from=2-2, to=2-3]
	\arrow[from=1-1, to=1-3]
	\arrow["\sim", from=1-3, to=2-3]
	\arrow["\sim"', from=1-1, to=2-1]
	\arrow["\sim"', from=2-1, to=2-2]
\end{tikzcd}\]
where all arrows labelled by $\sim$ are weak equivalences. By two out of three, this implies the result.
\end{proof}

\begin{lemma}
\label{lem:horn_inclusion_4}
For any stratified Segal $A$-precategory $C$, the morphisms $\Lambda^1[2]\star C\to [2]_t\star C$ and $\{1\}\star C\to [1]_t\star C$ are
 acyclic cofibrations.
Moreover, for any cofibration of stratified Segal $A$-precategory $i$, and $j$ being either $\{1\}\to [1]_t$ or $\Lambda^1[2]\to [2]_t$, the morphism $j\hstar i$ is an acyclic cofibration.
\end{lemma}
\begin{proof}
We begin with the first assertion.
The lemma \ref{lemma:leibnizt joint is Quillen} implies that  $\Lambda^1[2]\star\uvar$ and $[2]_t\star\uvar$ are left Quillen functors.  As every object is a homotopy colimits of objects of shape $[a,n]$ or $[e,1]_t$, we can reduce to the case where $C$ is of this shape.
Using Segal extensions, we can reduce to the case where $C$ is $[a,1]$, $[0]$ or $[e,1]_t$.

If $C$ is $[a,1]$ or $[0]$, the result follows from lemmas \ref{lem:outer horn inclusion2}, \ref{lem:outer horn inclusion1}, \ref{lem:horn_inclusion_2} and \ref{lem:horn_inclusion_3}.

Eventually, for $C := [e,1]_t$, we have a diagram:
\[\begin{tikzcd}
	{\{1\}\star[e,1]_t} & {\{0\}\star [0]} && {\Lambda^1[2]\star [e,1]_t} & {\Lambda^1[2]\star [0]} \\
	{[1]_t\star[e,1]_t} & {[1]_t\star [0]} && {[2]_t\star[e,1]_t} & {[2]_t\star [0]}
	\arrow[from=1-4, to=2-4]
	\arrow[from=1-4, to=1-5]
	\arrow[from=2-4, to=2-5]
	\arrow[from=1-5, to=2-5]
	\arrow[from=2-1, to=2-2]
	\arrow[from=1-1, to=2-1]
	\arrow[from=1-2, to=2-2]
	\arrow[from=1-1, to=1-2]
\end{tikzcd}\]
Lemmas \ref{lemma:leibnizt joint is Quillen}, \ref{lem:outer horn inclusion2} and \ref{lem:horn_inclusion_2} imply that all
 horizontal morphisms and right vertical morphisms are weak equivalences. By two out of three, this implies that the left vertical morphisms are weak equivalences.

This concludes the proof of the first assertion. The second one is obtained with some diagram chasing.
\end{proof}

\begin{prop}
\label{prop:horn_inclusion}
The functor $\stratSset\to \stratSeg(A)$ sends complicial horn inclusions to weak equivalences.
\end{prop}
\begin{proof}
Let $k\leq n$ be two integers. First, we suppose that $0<k<n$. We then have an equality $$(\Lambda^k[n]\to [n]^k)= (\partial[k-2]\to [k-2])\hstar (\Lambda^1[2]\to [2]_t)\hstar (\partial [n-k-2]\to [n-k-2]).$$ This is an acyclic cofibration according to lemmas \ref{lemma:leibnizt joint is Quillen} and \ref{lem:horn_inclusion_4}. If $k=0$, we have an equality $$(\Lambda^0[n]\to [n]^0) = (\{1\}\to [e,1]_t)\hstar (\partial[n-2]\to [n-2])$$
and the right hand morphism is an acyclic cofibration again thanks to lemma \ref{lem:horn_inclusion_4}. Eventually, for $k=n$, note that $$(\Lambda^n[n]\to [n]^n) = (\partial[n-2]\to [n-2])\hstar (\{0\}\to [e,1]_t).$$ This morphism is an acyclic cofibration according to lemma \ref{lemma:leibnizt joint is Quillen}.
\end{proof}

\subsection{Complicial thinness extensions}
\label{section:Complicial thinness extensions}
\begin{notation*}
In this section, we will often consider morphisms $\tilde{a}\to \tilde{b}$ that fit into cocartesian squares:
\[\begin{tikzcd}
	a & b \\
	{\tilde{a}} & {\tilde{b}}
	\arrow[from=1-1, to=2-1]
	\arrow["i", from=1-1, to=1-2]
	\arrow[from=2-1, to=2-2]
	\arrow[from=1-2, to=2-2]
	\arrow["\lrcorner"{anchor=center, pos=0.125}, draw=none, from=1-1, to=2-2]
\end{tikzcd}\]
where $a\to \tilde{a}$ and $b\to \tilde{b}$ are epimorphisms.
To avoid complicating the notations unnecessarily, the induced morphism $\tilde{a}\to \tilde{b}$ will just be denoted $i$.
\end{notation*}

\begin{lemma}
\label{lemma:thinnes extension case 0 and n}
Morphisms $([n]^0)'\to ([n]^0)''$ and $([n]^n)' \to ([n]^n)''$ are acyclic cofibrations.
\end{lemma}
\begin{proof}
For $k$ equal to $0$ or $n$, we have pushout diagrams: 
\[\begin{tikzcd}
	{[n]^k} & {([n]^k)'} & {([n]^k)''} \\
	{[n-1]} & {[n-1]_t} & {[n-1]_t}
	\arrow[from=1-1, to=2-1]
	\arrow[from=1-1, to=1-2]
	\arrow[from=1-2, to=2-2]
	\arrow[from=2-1, to=2-2]
	\arrow[from=1-2, to=1-3]
	\arrow[from=1-3, to=2-3]
	\arrow["id"', from=2-2, to=2-3]
	\arrow["\lrcorner"{anchor=center, pos=0.125, rotate=180}, draw=none, from=2-2, to=1-1]
	\arrow["\lrcorner"{anchor=center, pos=0.125, rotate=180}, draw=none, from=2-3, to=1-2]
\end{tikzcd}\]
Lemmas \ref{lemma:leibnizt joint is Quillen} and \ref{lem:horn_inclusion_4} imply that both $s^0:[n]^0\to [n-1]$ and $s^{n-1}:[n]^{n-1}\to [n-1]$ are weak equivalences. As horizontal morphisms are cofibrations, the left properness imply that all the vertical morphisms are weak equivalences.
 By two out of three, this shows that $([n]^k)' \to ([n]^k)''$ is a weak equivalence. 
\end{proof}

\begin{construction}
\label{cons:the big construction}
We consider these objects of $\Delta^2_{/[1]}$ and $\Delta^2_{/[2]}$:
$$\begin{array}{cc}
s^1:[1]^{op}\star [0]\to[1]&s^0:[0]^{op}\star [1]\to[1]
\\s^1:[1]^{op}\star[1]\to[2]&s^2:[2]^{op}\star[0]\to [2].
\end{array}$$
They induce morphisms:
$$\begin{array}{cc}
\alpha_a:[e\star a,1]\to e\star [a,1]&\beta_a:[e,1]\vee [ a,1]\to e\star [a,1] \\
 \delta_a:[e\star a,1]\vee[a,1]\to e\star([a,2])& \epsilon_a:[[2]\botimes a,1]\to e\star([a,2])\end{array}$$
where $[2]\botimes a$ and $[e\star a,1]\vee[a,1]$ are the following pushouts:
\[\begin{tikzcd}
	{[1]\otimes a\amalg [1]\otimes a} && {[2]\otimes a} && {[[1]\otimes a,1]\amalg[[1]\otimes a,1]} & {[[1]\otimes a,2]} \\
	{e\star a\amalg e\star a} && {[2]\botimes a} && {[e\star a,1]\amalg [a,1]} & {[e\star a,1]\vee[a,1]}
	\arrow[""{name=0, anchor=center, inner sep=0}, "{d^1\otimes a\amalg d^2\otimes a}", from=1-1, to=1-3]
	\arrow["{d^1\botimes a\amalg d^2\botimes a}"', from=2-1, to=2-3]
	\arrow[from=1-1, to=2-1]
	\arrow[from=1-3, to=2-3]
	\arrow[from=1-5, to=2-5]
	\arrow[""{name=1, anchor=center, inner sep=0}, "{[[1]\otimes a,d^2\amalg d^0]}", from=1-5, to=1-6]
	\arrow[from=2-5, to=2-6]
	\arrow[from=1-6, to=2-6]
	\arrow["\lrcorner"{anchor=center, pos=0.125, rotate=180}, draw=none, from=2-3, to=0]
	\arrow["\lrcorner"{anchor=center, pos=0.125, rotate=180}, draw=none, from=2-6, to=1]
\end{tikzcd}\]
Moreover there are commutative diagrams:
\[\begin{tikzcd}
	{[1]} & {[0]^{op}\star[1]} & {[0]^{op}\star[0]} & {[1]^{op}\star[0]} \\
	& { [1]} & {[0]^{op}\star[1]} & {[1] } \\
	{[1]^{op}\star[0]} & {[1]^{op}\star[1]} & {[1]^{op}\star[0]} & {[2]^{op}\star[0]} \\
	{[1]} & {[2]} & {[1]} & {[2]} \\
	{[1]^{op}\star[0]} & {[2]^{op}\star[0]} & {[1]^{op}\star[0]} & {[2]^{op}\star[0]} \\
	{[1]^{op}\star[1]} & {[2]} & {[1]} & {[2]}
	\arrow["{s^1}", from=3-2, to=4-2]
	\arrow["{s^2}", from=3-4, to=4-4]
	\arrow["{d^1}", from=3-3, to=3-4]
	\arrow["{s^1}"', from=3-3, to=4-3]
	\arrow["{d^1}"', from=4-3, to=4-4]
	\arrow["{s^1}"', from=3-1, to=4-1]
	\arrow["{d^2}", from=3-1, to=3-2]
	\arrow["{d^2}"', from=4-1, to=4-2]
	\arrow["{s^0}"', from=2-3, to=2-4]
	\arrow["{d^1}"', from=1-3, to=2-3]
	\arrow["{d^1}", from=1-3, to=1-4]
	\arrow["{s^1}", from=1-4, to=2-4]
	\arrow["{s^0}", from=1-2, to=2-2]
	\arrow["{d^0}", from=1-1, to=1-2]
	\arrow["{s^1}"', from=6-1, to=6-2]
	\arrow["{s^2}", from=5-2, to=6-2]
	\arrow["{d^2}"', from=5-1, to=6-1]
	\arrow["{d^2}", from=5-1, to=5-2]
	\arrow["id"', from=1-1, to=2-2]
	\arrow["{d^0}", from=5-3, to=5-4]
	\arrow["{s^1}"', from=5-3, to=6-3]
	\arrow["{d^0}"', from=6-3, to=6-4]
	\arrow[from=5-4, to=6-4]
\end{tikzcd}\]
which induce commutative diagrams:
\[\begin{tikzcd}
	{[a,1]} & {[e,1]\vee[a,1]} & {[a,1]} & {[e\star a,1]} \\
	\textcolor{white}{e\star [a,1]} & {e\star [a,1]} & {[e,1]\vee[a,1]} & {e\star [a,1]} \\
	{[e\star a,1]} & {[e\star a,1]\vee[a,1]} & {[e\star a,1]} & {[[2]\botimes a,1]} \\
	{e\star[ a,1]} & {e\star[a,2]} & {e\star[ a,1]} & {e\star[a,2]} \\
	{[[1]\otimes a,1]} & {[[2]\botimes a,1]} & {[e\star a,1]} & {[[2]\botimes a,1]} \\
	{[e\star a,1]\vee[a,1]} & {e\star[a,2]} & {e\star[ a,1]} & {e\star[a,2]}
	\arrow["{\delta_a}", from=3-2, to=4-2]
	\arrow["{\alpha_a}"', from=3-3, to=4-3]
	\arrow["{[d^1\botimes a,1]}", from=3-3, to=3-4]
	\arrow["{\epsilon_a}", from=3-4, to=4-4]
	\arrow["{e\star [a,d^1]}"', from=4-3, to=4-4]
	\arrow["{\alpha_a}"', from=3-1, to=4-1]
	\arrow["{e\star[a,d^2]}"', from=4-1, to=4-2]
	\arrow["{[e\star a,d^2]}", from=3-1, to=3-2]
	\arrow["{[a,d^1]}"', from=1-3, to=2-3]
	\arrow["{[d^{0}\star a,1]}", from=1-3, to=1-4]
	\arrow["{\alpha_a}", from=1-4, to=2-4]
	\arrow["{\beta_a}", from=2-3, to=2-4]
	\arrow["{\beta_a}", from=1-2, to=2-2]
	\arrow["{[a,d^0]}", from=1-1, to=1-2]
	\arrow["{(3):}"{description}, shift right=5, curve={height=30pt}, draw=none, from=3-1, to=4-1]
	\arrow["{[[1]\otimes a,d^1]}"', from=5-1, to=6-1]
	\arrow["{\epsilon_a}", from=5-2, to=6-2]
	\arrow["{[d^0\otimes a,1]}", from=5-1, to=5-2]
	\arrow["{(5):}"{description}, shift right=5, curve={height=30pt}, draw=none, from=5-1, to=6-1]
	\arrow["{\delta_a}"', from=6-1, to=6-2]
	\arrow["{:(4)}"{description}, shift left=5, curve={height=-30pt}, draw=none, from=3-4, to=4-4]
	\arrow["{d^{0}\star {[a,1]}}"', from=1-1, to=2-2]
	\arrow["{(1):}"{description}, shift right=5, curve={height=30pt}, draw=none, from=1-1, to=2-1]
	\arrow["{[d^2\botimes a,1]}", from=5-3, to=5-4]
	\arrow["{e\star[a,d^0]}"', from=6-3, to=6-4]
	\arrow["{\alpha_a}"', from=5-3, to=6-3]
	\arrow["{\epsilon_a}", from=5-4, to=6-4]
	\arrow["{:(6)}"{description}, shift left=5, curve={height=-30pt}, draw=none, from=5-4, to=6-4]
	\arrow["{:(2)}"{description}, shift left=5, curve={height=-30pt}, draw=none, from=1-4, to=2-4]
\end{tikzcd}\]
\end{construction}
\begin{definition}
Let $b$ be an object of $A$ and $x:a\to b,~x':a'\to b$ two morphisms. The element $b$ is \wcnotion{$n$-relying on $x$}{relying on $x$@$n$-relying on $x$} if for any $k\geq -1$, the following square is homotopy cocartesian:
\[\begin{tikzcd}
	{[k]\star a} & {[k]\star b} \\
	{\tau^i_{n+k+1}([k]\star a)} & {\tau^i_{n+k+1}([k]\star b)}
	\arrow[from=1-1, to=1-2]
	\arrow[from=2-1, to=2-2]
	\arrow[from=1-1, to=2-1]
	\arrow[from=1-2, to=2-2]
\end{tikzcd}\]
The element $b$ is \wcnotion{$n$-relying on $x$ and $x'$}{relying on $x$ and $x'$@$n$-relying on $x$ and $x'$} if for any $k\geq -1$, the following square is homotopy cocartesian:
\[\begin{tikzcd}
	{[k]\star a\amalg [k]\star a'} & {[k]\star b} \\
	{\tau^i_{n+k+1}([k]\star a)\amalg \tau^i_{n+k+1}([k]\star a')} & {\tau^i_{n+k+1}([k]\star b)}
	\arrow[tail reversed, no head, from=2-1, to=1-1]
	\arrow[from=1-1, to=1-2]
	\arrow[from=2-1, to=2-2]
	\arrow[tail reversed, no head, from=2-2, to=1-2]
\end{tikzcd}\]
\end{definition}

\p We recall that we denote by $C_{\mk}$ the marked Segal $A$-precategory associated to a stratified Segal $A$-precategory $C$. The canonical inclusion $C\to C_{\mk}$ is denoted $r_C$ and is an acyclic cofibration according to he proposition \ref{prop:X to Xmk is acycli cof}. These notions and notations are defined in paragraph \ref{para:marked segal}.
The fact that will be used the most with the marked Segal $A$-precategory is their right lifting property with respect to morphisms of shape $[\tau^i_n(a),\Lambda^1[2]]\cup [a,2]\to [\tau^i_n(a),2]$. This fact will  be used freely.

\begin{definition}
\label{def:order relation case 1}
Let $C$ be a Segal $A$-precategory. We define the relation \wcnotation{$\geq_{n}$}{((g37@$\geq_{n}$} on morphisms of shape $[a,1]\to C$ for $a$ verifying $\tau^i_{n}a = a$, as the smallest reflexive and transitive relation such that $(x:[a,1]\to C) \geq_n (x':[a',1]\to C)$ whenever one of the three following conditions is verified:
\begin{enumerate}
\item The elements $a$ and $a'$ are equal and there exists a lifting the following diagram:
\[\begin{tikzcd}
	{[a,1]} \\
	& {[a,1]\vee[\tau^i_{n-1}a,1]} & C \\
	{[a,1]}
	\arrow["{[a,d^2]}"', from=1-1, to=2-2]
	\arrow["x", curve={height=-12pt}, from=1-1, to=2-3]
	\arrow["{[a,d^1]}", from=3-1, to=2-2]
	\arrow[dotted, from=2-2, to=2-3]
	\arrow["{x'}"', curve={height=12pt}, from=3-1, to=2-3]
\end{tikzcd}\]
\item The elements $a$ and $a'$ are equal and there exists a lifting in the following diagram:
\[\begin{tikzcd}
	{[a,1]} \\
	& {[\tau^i_{n-1}a,1]\vee[a,1]} & C \\
	{[a,1]}
	\arrow["{[a,d^0]}"', from=1-1, to=2-2]
	\arrow["x", curve={height=-12pt}, from=1-1, to=2-3]
	\arrow["{[a,d^1]}", from=3-1, to=2-2]
	\arrow[dotted, from=2-2, to=2-3]
	\arrow["{x'}"', curve={height=12pt}, from=3-1, to=2-3]
\end{tikzcd}\]
\item There exists an element $b$ which is $(n-1)$-relying on $a\to b$ and dotted arrows in the following diagram:
\[\begin{tikzcd}
	{[a,1]} \\
	& {[b,1]} & {C_{\mk}} \\
	{[a',1]}
	\arrow["{ }"', from=1-1, to=2-2]
	\arrow[dotted, from=3-1, to=2-2]
	\arrow["{r_C\circ x}", curve={height=-12pt}, from=1-1, to=2-3]
	\arrow["{r_C\circ x'}"', curve={height=12pt}, from=3-1, to=2-3]
	\arrow[dotted, from=2-2, to=2-3]
\end{tikzcd}\]
\end{enumerate}
\end{definition}

\begin{definition}
\label{def:order relation case 2}
We also set $(\bar{x}:[\bar{a},1]\to C,\bar{x}':[\bar{a}',1]\to C)\geq_n \bar{x}'':[\bar{a}'',1]\to C$ if there exists three elements $x:[a,1]\to C$, $x':[a',1]\to C$ and $x'':[a'',1]\to C$ such that $\bar{x}\geq_n x$, $\bar{x}'\geq_n x'$, $x''\geq_n \bar{x}''$ and one of the two following conditions is verified:
\begin{enumerate}
\item The elements $a$, $a'$ and $a''$ are equal and there exists a dotted arrow:
\[\begin{tikzcd}
	{[a,1]\cup[a,1]} \\
	& {[a,2]} & C \\
	{[a,1]}
	\arrow["{[a,d^2\cup d^0]}"'{pos=0.1}, from=1-1, to=2-2]
	\arrow["{[a,d^1]}"{pos=0.4}, from=3-1, to=2-2]
	\arrow["{x\cup x'}", curve={height=-12pt}, from=1-1, to=2-3]
	\arrow["{x''}"', curve={height=12pt}, from=3-1, to=2-3]
	\arrow[dotted, from=2-2, to=2-3]
\end{tikzcd}\]
\item There exists an element $b$ which is $(n-1)$-relying on $a\to b$ and $a'\to b$ and dotted arrows in the following diagram:
\[\begin{tikzcd}
	{[a,1]\amalg[a',1]} \\
	& {[b,1]} & {C_{mk}} \\
	{[a'',1]}
	\arrow[from=1-1, to=2-2]
	\arrow[dotted, from=3-1, to=2-2]
	\arrow["{r_C\circ x\amalg r_C\circ x'}", curve={height=-12pt}, from=1-1, to=2-3]
	\arrow["{r_C\circ x''}"', curve={height=12pt}, from=3-1, to=2-3]
	\arrow[dotted, from=2-2, to=2-3]
\end{tikzcd}\]
\end{enumerate}
\end{definition}

\begin{prop}
\label{prop:meaning of geq case 1}
Let $C$ be a stratified Segal $A$-precategory and $x:[a,1]\to C$, $y:[a',1]\to C$ two morphisms such that $x\geq_n y$. The morphism 
$$ C\coprod_{[a,1]} \tau^i_n( [a,1])\to \tau^i_n( [a',1])\coprod_{[a',1]}C\coprod_{[a,1]} \tau^i_n( [a,1])$$
is an acyclic cofibration. 
\end{prop}
\begin{proof}
By two out of three, we can suppose without loss of generality that $C$ is already a marked Segal $A$-precategory. We suppose first that $x$ and $y$ fulfill one of the three cases of definition \ref{def:order relation case 1}. The following square is then homotopy cartesian: 
\[\begin{tikzcd}
	{[a,1]} & C \\
	{\tau^i_n[a,1]} & {\tau^i_n[a,1]\amalg_{[a,1]} C\coprod_{[a',1]} \tau^i_n[a',1]}
	\arrow["x", from=1-1, to=1-2]
	\arrow[from=1-1, to=2-1]
	\arrow[from=1-2, to=2-2]
	\arrow[from=2-1, to=2-2]
\end{tikzcd}\]
As the cocartesian square:
\[\begin{tikzcd}
	{[a,1]} & C \\
	{\tau^i_n[a,1]} & {\tau^i_n[a,1]\amalg_{[a,1]} C}
	\arrow["x", from=1-1, to=1-2]
	\arrow[from=1-1, to=2-1]
	\arrow[from=1-2, to=2-2]
	\arrow[from=2-1, to=2-2]
	\arrow["\lrcorner"{anchor=center, pos=0.125, rotate=180}, draw=none, from=2-2, to=1-1]
\end{tikzcd}\]
is also homotopy cocartesian, this
 implies that $$C\coprod_{[a,1]} \tau^i_n( [a,1])\to \tau^i_n( [a',1])\coprod_{[a',1]}C\coprod_{[a,1]} \tau^i_n( [a,1])$$ is an acyclic cofibration. Suppose now that there exists a familly of morphisms $(x_k:[a_k,1])_{k\leq m}\to C$ such that $x_0=x$, $x_m=y$ and for any $k$, $x_k$ and $x_{k+1}$ fullfill one of the three cases of definition \ref{def:order relation case 1}. We then have two homotopy cocartesian squares:
\[\begin{tikzcd}
	{C\coprod_{[a',1]} \tau^i_n[a',1]} & {[a,1]} & C \\
	{C\coprod_{\coprod_{k\leq m}[a_k,1]}\coprod_{k\leq m}\tau^i_n[a_k,1]} & {\tau^i_n[a,1]} & {C\coprod_{\coprod_{k\leq m}[a_k,1]}\coprod_{k\leq m}\tau^i_n[a_k,1]}
	\arrow[from=1-2, to=1-3]
	\arrow[from=1-2, to=2-2]
	\arrow[from=1-3, to=2-3]
	\arrow[from=2-2, to=2-3]
	\arrow[from=1-2, to=1-1]
	\arrow[from=2-2, to=2-1]
	\arrow[from=1-1, to=2-1]
\end{tikzcd}\]
As before, this implies that $$C\coprod_{[a,1]} \tau^i_n( [a,1])\to C\coprod_{\coprod_{k\leq m}[a_k,1]}\coprod_{k\leq m}\tau^i_n[a_k,1]$$ and $$\tau^i_n( [a',1])\coprod_{[a',1]}C\coprod_{[a,1]} \tau^i_n( [a,1])\to C\coprod_{\coprod_{k\leq m}[a_k,1]}\coprod_{k\leq m}\tau^i_n[a_k,1]$$ are acyclic cofibrations. 
By two out of three, this implies the result.
\end{proof}

One can show similarly:

\begin{prop}
\label{prop:meaning of geq case 2}
Let $C$ be a stratified Segal $A$-precategory, and $x:[a,1]\to C$, $y:[a',1]\to C$ and $z:[a'',1]\to C$ three morphisms such that $(x,y)\geq_n z$. The morphism 
$$ \tau^i_n( [a',1])\coprod_{[a',1]}C\coprod_{[a,1]} \tau^i_n( [a,1])\to \tau^i_n( [a',1])\coprod_{[a',1]}C\coprod_{[a,1]} \tau^i_n( [a,1]) \coprod_{[a'',1]} \tau^i_n( [a'',1])$$
is an acyclic cofibration. 
\end{prop}

\begin{lemma}
\label{lem:abstract thinness 0}
Let $n$ be a non null integer and $a$ an element such that $\tau^i_{n}(a)=a$. The object $[2]^2\otimes a$ is $n$-relying on $d^1\botimes a:e\star a\to [2]^2\botimes a$.
\end{lemma}
\begin{proof}
As the morphism $d^1\botimes a:e\star a\to [2]^2\botimes a$ is a weak equivalence, so are the horizontal morphisms of the following diagram:
\[\begin{tikzcd}
	{[k]\star e\star a} & {[k]\star([2]^2\botimes a)} \\
	{\tau^i_{n+k+1}([k]\star e\star a)} & {\tau^i_{n+k+1}([k]\star([2]^2\botimes a))}
	\arrow["\sim", from=1-1, to=1-2]
	\arrow[from=1-1, to=2-1]
	\arrow[from=1-2, to=2-2]
	\arrow["\sim"', from=2-1, to=2-2]
\end{tikzcd}\]
As the vertical morphisms are cofibrations, this implies that this square is homotopy cocartesian.
\end{proof}

\begin{lemma}
\label{lem:abstract thinness 1}
Let $n$ be a non null integer and $a$ an element such that $\tau^i_{n}(a)=a$. The object $[2]\botimes a$ is $n$-relying on $d^0\otimes a:[1]\otimes a\to [2]\botimes a$ and $d^2\otimes a:e\star a\to [2]\otimes a$. Moreover, $[2]\botimes a\coprod_{d^0\otimes a} \tau^i_{n}([1]\otimes a)$ (resp. $[2]\botimes a\coprod_{d^2\botimes a} \tau^i_{n}(e\star a)$) is $n$-relying on $d^{2}\otimes a$ (resp. $d^0\botimes a$).
\end{lemma}
\begin{proof}
Consider the following diagram:
\[\begin{tikzcd}[column sep=0.3cm]
	{[k]\star([1]\otimes a)\amalg [k]\star([1]\otimes a)} & {[k]\star(\Lambda^1[2]\otimes a)} & {[k]\star([2]\otimes a)} \\
	{\tau^i_{n+k+1}([k]\star([1]\otimes a))\amalg \tau^i_{n+k+1}([k]\star([1]\otimes a))} & {\tau^i_{n+k+1}([k]\star(\Lambda^1[2]\otimes a))} & {\tau^i_{n+k+1}([k]\star([2]\otimes a))}
	\arrow[from=1-1, to=2-1]
	\arrow[""{name=0, anchor=center, inner sep=0}, from=1-1, to=1-2]
	\arrow[from=1-2, to=2-2]
	\arrow[from=2-1, to=2-2]
	\arrow["\sim", from=1-2, to=1-3]
	\arrow["\sim"', from=2-2, to=2-3]
	\arrow[from=1-3, to=2-3]
	\arrow["\lrcorner"{anchor=center, pos=0.125, rotate=180}, draw=none, from=2-2, to=0]
\end{tikzcd}\]
The left square is cocartesian and so homotopy cocartesian. Horizontal morphisms of the right square are weak equivalences, so this square is also homotopy cocartesian.
The outer square is then homotopy cocartesian and this implies that 
$[[2]\otimes a,1]$ is $n$-relying on $d^0\otimes a$ and $d^2\otimes a$. We then have a diagram:
\[\begin{tikzcd}[column sep=0.3cm]
	{[k]\star([1]\otimes a)\amalg [k]\star([1]\otimes a)} & {[k]\star([2]\otimes a)} & {[k]\star([2]\botimes a)} \\
	{\tau^i_{n+k+1}([k]\star([1]\otimes a))\amalg \tau^i_{n+k+1}([k]\star([1]\otimes a))} & {\tau^i_{n+k+1}([k]\star([2]\otimes a))} & {\tau^i_{n+k+1}([k]\star([2]\botimes a))}
	\arrow[""{name=0, anchor=center, inner sep=0}, from=1-2, to=1-3]
	\arrow[from=1-2, to=2-2]
	\arrow[from=2-2, to=2-3]
	\arrow[from=1-3, to=2-3]
	\arrow[from=2-1, to=2-2]
	\arrow[from=1-1, to=1-2]
	\arrow[from=1-1, to=2-1]
	\arrow["\lrcorner"{anchor=center, pos=0.125, rotate=180}, draw=none, from=2-3, to=0]
\end{tikzcd}\]
where the two squares are homotopy cocartesian and so is the outer one. This implies the first assertion and the two others follow easily.
\end{proof}

\begin{lemma}
\label{lem:abstract thinness 2}
Let $n$ be an integer strictly superior to $1$ and $a$ such that $\tau^i_{n}(a) = a$.
We consider the projection $\pi:[a,2]\to [a,1]\vee[\tau^i_{n-1}(a),1]$ and $\pi':[a,2]\to [\tau^i_{n-1}(a),1]\vee[a,1]$. We then have inequalities
$$ e\star\pi\circ \epsilon_a\circ [d^0\otimes a, 1]\geq_{n+1} e\star\pi\circ \epsilon_a\circ [d^1\botimes a, 1]$$ and 
$$ e\star\pi'\circ \epsilon_a\circ [d^2\botimes a, 1]\geq_{n+1}e\star \pi\circ \epsilon_a\circ [d^1\botimes a, 1].$$
\end{lemma}
\begin{proof}
Using the diagram $(6).\ref{cons:the big construction}$ we get a diagram
\[\begin{tikzcd}
	{[e\star a,1]} & {[[2]\botimes a,1]} \\
	{[\tau^i_{n}(e\star a),1]} & {e\star[ a,1]} & {e\star[a,2]} \\
	& {e\star[\tau^i_{n-1}(a),1]} & {e\star([a,1]\vee[\tau^i_{n-1}(a),1])}
	\arrow["{[d^2\botimes a,1]}", from=1-1, to=1-2]
	\arrow["{e\star[a,d^0]}"', from=2-2, to=2-3]
	\arrow["{\alpha_a}", from=1-1, to=2-2]
	\arrow["{\epsilon_a}", from=1-2, to=2-3]
	\arrow["e\star\pi", from=2-3, to=3-3]
	\arrow[from=2-2, to=3-2]
	\arrow[from=3-2, to=3-3]
	\arrow[from=1-1, to=2-1]
	\arrow[from=2-1, to=3-2]
\end{tikzcd}\]
The morphism $r_{e\star([a,1]\vee[\tau^i_{n-1}(a),1])}\circ e\star\pi\circ \epsilon_a$ then factors through $[[2]\botimes a\coprod_{d^2\botimes a} \tau^i_{n}(e\star a),1]$. According to lemma \ref{lem:abstract thinness 1}, we then get the first inequalities.

For the second inequality, using the diagrams $(3).\ref{cons:the big construction}$ and $(5).\ref{cons:the big construction}$, we have a diagram:
\[\begin{tikzcd}
	& {[[1]\otimes a,1]} & {[[2]\botimes a,1]} \\
	& {[e\star a,1]\vee[a,1]} & {e\star[a,2]} \\
	{[e\star a,1]} & {e\star[ a,1]} & {e\star([\tau^i_{n-1}(a),1]\vee[a,1])} \\
	{[\tau^i_{n}(e\star a),1]} & {e\star[ \tau^i_{n-1}(a),1]}
	\arrow["{\alpha_a}"', from=3-1, to=3-2]
	\arrow["{[[1]\otimes a,d^1]}"', from=1-2, to=2-2]
	\arrow["{\epsilon_a}", from=1-3, to=2-3]
	\arrow["{[d^0\otimes a,1]}", from=1-2, to=1-3]
	\arrow["{\delta_a}", from=2-2, to=2-3]
	\arrow["{e\star \pi'}", from=2-3, to=3-3]
	\arrow["{[e\star a,d^2]}", from=3-1, to=2-2]
	\arrow["{e\star[a,d^2]}"{description}, from=3-2, to=2-3]
	\arrow[from=3-2, to=4-2]
	\arrow[from=4-2, to=3-3]
	\arrow["{\alpha_{\tau^i_{n-1}(a)}}"', from=4-1, to=4-2]
	\arrow[from=3-1, to=4-1]
\end{tikzcd}\]
This implies that $r_{e\star([\tau^i_{n-1}(a),1]\vee[a,1])}\circ e\star\pi'\circ e\star[a,d^2]\circ \alpha_a$ factors through $[\tau^i_n(e\star a),1]$. The morphism $r_{e\star([\tau^i_{n-1}(a),1]\vee[a,1])}\circ e\star\pi\circ \epsilon_a$ then factors through $[[2]\otimes a\coprod_{d^0\otimes a} \tau^i_{n}([1]\otimes a),1]$. According to lemma \ref{lem:abstract thinness 1}, we then get the second inequality.

\end{proof}

\begin{lemma}
\label{lem:abstract thinness 3}
Let $n$ be an integer strictly superior to $1$ and $a$ such that $\tau^i_{n}(a) = a$.
We then have $\delta_a\circ [e\star a,d^2]\geq_{n+1} \delta_a\circ [[1]\otimes a,d^1]$. 
\end{lemma}
\begin{proof}
There is a diagram:
\[\begin{tikzcd}
	{[e\star a,1]} & {[e\star a,1]} & {[[1]\otimes a,1]} \\
	{e\star[a,2]} & {[e\star a,1]\vee[a,1]} & {[[1]\otimes a,1]\vee[a,1]} \\
	& {[[1]\otimes a,1]} & {[[1]\otimes a,1]}
	\arrow["{\delta_a}", from=2-2, to=2-1]
	\arrow["{[e\star a,d^2]}", from=1-2, to=2-2]
	\arrow["{[[1]\otimes a,d^1]}"', from=3-2, to=2-2]
	\arrow["{[[1]\otimes a,d^2]}", from=1-3, to=2-3]
	\arrow["id", from=3-3, to=3-2]
	\arrow["{[[1]\otimes a,d^1]}"', from=3-3, to=2-3]
	\arrow[from=1-3, to=1-2]
	\arrow[from=2-3, to=2-2]
	\arrow["id", from=1-1, to=1-2]
\end{tikzcd}\]
As the morphism $[[1]\otimes a,1]\vee[a,1]\to [e\star a,1]\vee[a,1]$ factors through $[[1]\otimes a,1]\vee[\tau^i_{n}([1]\otimes a),1]$, we get the desired inequality.
\end{proof}

\begin{prop}
\label{prop:geqn stable by star}
Let $a$ be an object such that $\tau^i_{n}(a)=a$.
Let $x:[a,1]\to C,y:[a',1]\to C$ be two morphisms, such that $x\geq_ny$, then if we denote by $\bar{x} := e\star x\circ \alpha_a$ and $\bar{y} :=e\star y\circ \alpha_{a'} $, we have $\bar{x}\geq_{n+1}\bar{y}$.
\end{prop}
\begin{proof}
First, we suppose that we are in the first case of the definition \ref{def:order relation case 1}. We can then suppose without loss of generality that $C= [a,1]\vee[\tau^i_{n-1}(a),1]$. We denote by $\pi$ the projection of $[a,2]$ on $[a,1]\vee[\tau^i_{n-1}(a),1]$.
 Using the diagrams $(3).\ref{cons:the big construction}$, $(4).\ref{cons:the big construction}$ and $(5).\ref{cons:the big construction}$, we have a diagram:
\[\begin{tikzcd}
	{[[1]\otimes a,1]} & {[[2]\botimes a,1]} & {[e\star a,1]} \\
	{[e\star a,1]\vee[a,1]} & {e\star[a,2]} & {e\star[ a,1]} \\
	{[e\star a,1]} & {e\star[ a,1]} & {e\star ( [a,1]\vee[\tau^i_{n-1}(a),1])}
	\arrow["{\delta_a}", from=2-1, to=2-2]
	\arrow["{\alpha_a}", from=1-3, to=2-3]
	\arrow["{\alpha_a}"', from=3-1, to=3-2]
	\arrow["{e\star[a,d^2]}", from=3-2, to=2-2]
	\arrow["{[e\star a,d^2]}", from=3-1, to=2-1]
	\arrow["{[d^0\otimes a,1]}", from=1-1, to=1-2]
	\arrow["e\star\pi", from=2-2, to=3-3]
	\arrow["{[[1]\otimes a,d^1]}"', from=1-1, to=2-1]
	\arrow["{\epsilon_a}", from=1-2, to=2-2]
	\arrow["{[d^1\botimes a,1]}"', from=1-3, to=1-2]
	\arrow["{e\star [a,d^1]}"', from=2-3, to=2-2]
\end{tikzcd}\]
Thanks to lemmas \ref{lem:abstract thinness 2} and \ref{lem:abstract thinness 3}, this implies the result.

If we are in the second case of \ref{def:order relation case 1}, we can suppose that $C = [\tau^i_{n-1}(a),1]\vee[a,1]$, and we note by $\pi'$ the projection from $[a,2]\to[\tau^i_{n-1}(a),1]\vee[a,1]$ . Using the diagrams $(4).\ref{cons:the big construction}$ and $(6).\ref{cons:the big construction}$, we have a diagram:
\[\begin{tikzcd}
	{[e\star a,1]} & {[[2]\botimes a,1]} & {[e\star a,1]} \\
	{e\star[ a,1]} & {e\star[a,2]} & {e\star[ a,1]} \\
	& {e\star([\tau^i_{n-1}(a),1]\vee[a,1])}
	\arrow["{\alpha_a}", from=1-3, to=2-3]
	\arrow["{\epsilon_a}", from=1-2, to=2-2]
	\arrow["{[d^2\botimes a,1]}", from=1-1, to=1-2]
	\arrow["{e\star[a,d^0]}"', from=2-1, to=2-2]
	\arrow["{\alpha_a}"', from=1-1, to=2-1]
	\arrow["{e\star\pi'}", from=2-2, to=3-2]
	\arrow["{e\star [a,d^1]}", from=2-3, to=2-2]
	\arrow["{[d^1\botimes a,1]}"', from=1-3, to=1-2]
\end{tikzcd}\]
Thanks to lemmas \ref{lem:abstract thinness 2}, this implies the result.

If we are in the third case, it is a direct consequence of the naturality of $\alpha$, of the definition of $n$-reliability and of the fact that $(e\star C)_{\mk}\cong (e\star C_{\mk})_{\mk}$ as remarked in \ref{para:marked segal}.
\end{proof}

\begin{prop}
\label{prop:2geqn stable by star}
Let $x:[a,1]\to C$, $y:[a',1]\to C$ and $z:[a'',1]$ be three morphisms, such that $(x,y)\geq_nz$, then if we denote by $\bar{x} := e\star x\circ \alpha_a$, $\bar{y} :=e\star y\circ \alpha_{a'} $ and $\bar{z} :=e\star z\circ \alpha_{a''} $, we have $(\bar{x},\bar{y})\geq_{n+1}\bar{z}$.
\end{prop}
\begin{proof}
Suppose first that we are in the first case of the definition \ref{def:order relation case 2}. We can then suppose without loss of generality that $C= [a,2]$.
We define $\tilde{x}:=\epsilon_a\circ [d^0\otimes a,1]$. Diagram $(6).\ref{cons:the big construction}$ and
lemma \ref{lem:abstract thinness 2} imply that $(\tilde{x},\bar{y})\geq_{n+1} \bar{z}$. Eventually, diagrams $(3).\ref{cons:the big construction}$ and $(5).\ref{cons:the big construction}$ induce a diagram:
\[\begin{tikzcd}
	{[e\star a,1]} & {[e\star a,1]\vee[a,1]} & {[[1]\otimes a,1]} \\
	{e\star[ a,1]} & {e\star[a,2]} & {[[2]\botimes a,1]}
	\arrow["{\delta_a}", from=1-2, to=2-2]
	\arrow["{\alpha_a}"', from=1-1, to=2-1]
	\arrow["{e\star[a,d^2]}"', from=2-1, to=2-2]
	\arrow["{[e\star a,d^2]}", from=1-1, to=1-2]
	\arrow["{[d^0\otimes a,1]}", from=1-3, to=2-3]
	\arrow["{[[1]\otimes a,d^1]}"', from=1-3, to=1-2]
	\arrow["{\epsilon_a}", from=2-3, to=2-2]
\end{tikzcd}\]
wich implies that $\bar{x}\geq_{n+1}\tilde{x}$.

If we are in the second case of the definition, it is a direct consequence of the naturality of $\alpha$, of the definition of $n$-reliability and of the fact that $(e\star C)_{\mk}\cong (e\star C_{\mk})_{\mk}$ as remarked in paragraph \ref{para:marked segal}.
\end{proof}

\begin{lemma}
\label{lemma:case k 0}
For any $a$ such that $\tau^i_na=a$ and $x:[a,1]\to C$, if we denote by $\bar{x} := e\star x \circ d^0\star[a,1]$ and $\tilde{x}:= e\star x \circ \alpha_a\circ [d^0\star a,1]$, then $\bar{x}\geq_{n+1}\tilde{x}$.
\end{lemma}
\begin{proof}
Using the diagrams $(1).\ref{cons:the big construction}$ and $(2).\ref{cons:the big construction}$, we have a diagram:
\[\begin{tikzcd}
	{[a,1]} & {[e\star a,1]} \\
	{[e,1]\vee[a,1]} & {e\star [a,1]} & C \\
	{[a,1]}
	\arrow["{[a,d^0]}", from=3-1, to=2-1]
	\arrow["{d^{0}\star {[a,1]}}"', from=3-1, to=2-2]
	\arrow["{\beta_a}", from=2-1, to=2-2]
	\arrow["{[d^{0}\star a,1]}", from=1-1, to=1-2]
	\arrow["{[a,d^1]}"', from=1-1, to=2-1]
	\arrow["{\alpha_a}", from=1-2, to=2-2]
	\arrow["{e\star x}", from=2-2, to=2-3]
\end{tikzcd}\]
which implies the desired inequality.
\end{proof}

\p
\label{para:a cocartesian square for intelingent truncation} 
We now use these results to show that the thinness extensions are weak equivalences. We define by induction the morphism $\iota_n:[[n-1],1]\to[n]$ where $\iota_2:= \alpha_{[0]}$ and $\iota_{n+1} := e\star\iota_n\circ \alpha_{[n-1]}$.

 We can easily show by induction that $[n]$ is a colimit of terms which are all invariant under $\tau^i_{n-1}$ except the one corresponding to $\iota_n$.
 For any $n$ we then have a pushout square:
\[\begin{tikzcd}
	{[[n-1],1]} & {[n]} \\
	{[[n-1]_t,1]} & {[n]_t}
	\arrow["{\iota_n}", from=1-1, to=1-2]
	\arrow[from=1-1, to=2-1]
	\arrow[from=1-2, to=2-2]
	\arrow[from=2-1, to=2-2]
	\arrow["\lrcorner"{anchor=center, pos=0.125, rotate=180}, draw=none, from=2-2, to=1-1]
\end{tikzcd}\]

\begin{lemma}
\label{lem:thinness extension last step0}
For any $n$ and for any $k<n$, such that $k\neq n-2$, we have inequalities
$d^k\circ\iota_{n-1}\geq_{n-1}\iota_n \circ[d^k,1]$ and $(d^n\circ\iota_{n-1},d^{n-2}\circ \iota_{n-1})\geq_{n-1}\iota_n\circ [d^{n-2},1]$
\end{lemma}
\begin{proof}
We start by showing the first inequality by induction on $n$. If $n=2$, the only case is $k=1$, and the two morphisms are equal.

Suppose now the result true at the stage $n$. 
If $k>0$, we have
$$
\begin{array}{rllc}
d^{k}\circ\iota_{n}&=& e\star d^{k-1}\circ e\star \iota_{n-1}\circ \alpha_{[n-2]}\\
&\geq_{n}&e\star \iota_n \circ e\star [d^{k-1},1] \circ \alpha_{[n-2]}&\mbox{(induction hypothesis and \ref{prop:geqn stable by star})}\\
&=& e\star \iota_n \circ \alpha_{[n-1]} \circ [e\star d^{k-1},1]\\
&=& \iota_{n+1} \circ \alpha_{[n-1]} \circ [d^{k},1]
\end{array}
$$
We still have to deal with the case $k=0$. As $d^0:[n]\to [n+1]$ (resp $[d^0,1]:[[n-1],1]\to [[n],1]$) is equal to $d^0\star [n]$ (resp. $[d^0\star [n-1],1]$), this is exactly the content of lemma \ref{lemma:case k 0}.

For the second inequality, we proceed again by induction. We remark that this is true for $n=2$. Suppose now the result true at the stage $n$. We have
$$
\begin{array}{rllc}
(d^{n+1}\circ\iota_{n},d^{n-1}\iota_{n})&=& (e\star d^{n}\circ e\star \iota_{n-1}\circ\alpha_{[n-2]},e\star d^{n-2}\circ e\star \iota_{n-1}\circ\alpha_{[n-2]})\\
&\geq_{n-1}& e\star \iota_n\circ e\star[d^{n-2},1]\circ \alpha_{[n-2]}~~~~~\mbox{(induction hypothesis and \ref{prop:2geqn stable by star})}\\
&=& e\star \iota_n\circ e\star \alpha_{[n-1]}\circ[e\star d^{n-2},1]\\ 
&=& \iota_{n+1}\circ [d^{n-1},1]
\end{array}
$$
\end{proof}

\begin{lemma}
\label{lem:thinness extension last step1}
Let $0<k<n$ be two integers.
We denote by $\tau^k$ the projection $[n]\to [n]^k$. We then have $$\tau^k\circ \iota_n\circ [d^k,1]\geq_{n-1}\tau^k\circ d^k\circ\iota_{n-1}.$$
\end{lemma}
\begin{proof}
We demonstrate the result by induction on $n$. For the initialization, the only case is $n=2$ and $k=1$, and is obvious. 
Suppose now the result true at the stage $n$, and let $k>1$. 
We have inequalities:
$$\begin{array}{rlll}
\tau^k\circ \iota_{n+1}\circ [d^k,1] &=& e\star \tau^k\circ e\star \iota_n \circ \alpha_{[n-1]}\circ [d^k,1]\\
&=&\star \tau^k\circ e\star \iota_n \circ e\star [d^{k-1},1]\circ \alpha_{[n-2]}\\
&\geq_n & e\star \tau_k\circ e\star d^{k-1}\circ e\star \iota_{n-1}\circ \alpha_{[n-2]}& \mbox{(induction hypothesis and \ref{prop:geqn stable by star})}\\
&=& \tau_k\circ d^k\circ \iota_n
\end{array}$$
We still have to deal with the case $k=1$. Using diagrams $(1)$, $(2)$, $(4)$ and $(5)$, of construction \ref{cons:the big construction}, we get a diagram:
\[\begin{tikzcd}
	{[[n-1],1]} & {e\star[[n-2],1]} & {[n]} \\
	{[[2]\botimes [n-2],1]} & {e\star([e,1]\vee[[n-2],1])} & {[n+1]} & {[n+1]^1} \\
	{[[n-1],1]} & {e\star[[n-2],1]} & {e\star[[n-1],1]}
	\arrow["{d^1}", from=1-3, to=2-3]
	\arrow["{e\star\iota_n}"', from=3-3, to=2-3]
	\arrow["{e\star[d^0,1]}"', from=3-2, to=3-3]
	\arrow["{e\star [[n-1],d^1]}"', from=3-2, to=2-2]
	\arrow["{e\star [[n-1],d^0]}", from=1-2, to=2-2]
	\arrow["{e\star \iota_{n-1}}", from=1-2, to=1-3]
	\arrow["{e\star\beta_{[n-1]}}", from=2-2, to=2-3]
	\arrow["{\alpha_{[n-2]}}"', from=3-1, to=3-2]
	\arrow["{\alpha_{[n-2]}}", from=1-1, to=1-2]
	\arrow["{e\star\pi\circ \epsilon_{[n-2]}}", from=2-1, to=2-2]
	\arrow["{[d^2\botimes [n-2],1]}"', from=1-1, to=2-1]
	\arrow["{[d^1\botimes [n-2],1]}", from=3-1, to=2-1]
	\arrow["{\tau^1}", from=2-3, to=2-4]
\end{tikzcd}\]
where $\pi$ is the projection $[[n-2],2]\to [e,1]\vee[[n-2],1]$.
However, according to the diagrams $(5)$ and $(3)$ of \ref{cons:the big construction}, there is a diagram:
\[\begin{tikzcd}
	{[[1]\otimes [n-2],1]} & {[e\star [n-2],1]\vee[[n-2],1]} & {[e\star[n-2],1]} \\
	{[[2]\botimes [n-2],1]} & {[[n-2],2]} & {e\star[[n-2],1]} \\
	& {e\star([e,1]\vee[[n-2],1])} & {e\star[e,1]} \\
	& {[n+1]^1} & {[2]_t}
	\arrow["{[d^0\otimes [n-2],1]}"', from=1-1, to=2-1]
	\arrow["{[[1]\otimes[n-2],d^1]}", from=1-1, to=1-2]
	\arrow["{d^3\circ...\circ d^{n+1}}", from=4-3, to=4-2]
	\arrow["{\delta_{[n-2]}}", from=1-2, to=2-2]
	\arrow["{ \epsilon_{[n-2]}}"', from=2-1, to=2-2]
	\arrow["{e\star \pi}", from=2-2, to=3-2]
	\arrow["{[e\star[n-2],d^2]}"', from=1-3, to=1-2]
	\arrow["{\alpha_{[n-2]}}", from=1-3, to=2-3]
	\arrow[from=2-3, to=3-3]
	\arrow[from=3-3, to=3-2]
	\arrow[from=2-3, to=2-2]
	\arrow["{\tau_1\circ e\star\beta_{[n-1]}}"', from=3-2, to=4-2]
	\arrow[from=3-3, to=4-3]
\end{tikzcd}\]
This implies that $[[2]\botimes [n-2],1]\to [n+1]^{k}\to ([n+1]^{k})_{\mk} $ factors through $[[2]\botimes[n-2]\coprod_{d^0\otimes a}\tau^i_{n-1}([1]\otimes [n-2]),1]$. We can then apply lemma \ref{lem:abstract thinness 1}. 	
\end{proof}

\begin{lemma}
\label{lem:thinness extension last step2}
Let $0<k<n-1$ be two integers.
We denote by $\tau^k$ the projection $[n]\to [n]^k$. We then have $$(\tau^k\circ \iota_n\circ [d^{k-1},1], \tau^k\circ \iota_n\circ [d^{k+1},1])\geq_{n-1}\tau^k\circ \iota_n\circ [d^{k},1]$$ and $$\tau^{n-1}\circ \iota_n\circ [d^{n-2},1]\geq_{n-1}\tau^k\circ \iota_n\circ [d^{n-1},1].$$
\end{lemma}
\begin{proof}
By construction, for any $a$, the morphism $[[2]\star a,1]\to [2]\star[a,1]\to [2]_t\star[a,1]$ factors through $[[2]_t\star a,1]$. By induction, this implies that the composite morphism $[[n-1],1]\xrightarrow{\iota_n}[n]\to [n]^k$ factors through $[[n-1]^k,1]$ for any $k<n-1$. This implies the first assertion. 

For the second one, note that $[[1],e]\to [2]\to [2]_t$ factors through $[[1]_t,e]$. By induction, this implies that the composite morphism $[[n-1],1]\xrightarrow{\iota_n}[n]\to [n]^{n-1}$ factors through $[[n-1]^{n-2},1]$ which gives the second one.
\end{proof}

\begin{prop}
\label{prop:thinness extension}
For any $0\leq k\leq n$, the morphism $([n]^k)' \to ([n]^k)''$ is a weak equivalence.
\end{prop}
\begin{proof}
The case $k=0$ and $k=n$ are demonstrated in lemma \ref{lemma:thinnes extension case 0 and n}. For the case $0<k<n$, lemmas \ref{lem:thinness extension last step0}, \ref{lem:thinness extension last step1} and \ref{lem:thinness extension last step2} imply that if we denote by $\tau_k$ the projection $[n]\to [n]^k$, we have an inequality: $(\tau_k\circ d^{k-1}\circ \iota_{n-1},\tau_k\circ d^{k+1}\circ \iota_{n-1})\geq_{n-1}\tau_k\circ d^k\circ \iota_{n-1}$. Together with the proposition \ref{prop:meaning of geq case 2}, this implies that the following square is homotopy cartesian:
\[\begin{tikzcd}
	{[n-1]\cup[n-1]} & {[n]^k} \\
	{[n-1]_t\cup[n-1]_t} & {([n]^k)''}
	\arrow["{d^{k+1}\cup d^{k-1}}", from=1-1, to=1-2]
	\arrow[from=1-1, to=2-1]
	\arrow[from=2-1, to=2-2]
	\arrow[from=1-2, to=2-2]
\end{tikzcd}\]
The morphism $([n]^k)' \to ([n]^k)''$ is then a weak equivalence.
\end{proof}

\subsection{Saturation extensions}
\label{section:Saturation extensions}

Let $\Lambda[3]^{eq}\to [3]^{eq}$ be the entire inclusion generated by $Im(d^3)\cup Im(d^0)\subset [3]$. This inclusion fits in the following sequence:
\[\begin{tikzcd}
	{\Lambda^1[2]} & {[2]_t} & {([3]^1)'} & {([3]^1)''} \\
	{\Lambda[3]^{eq}} & \bullet & \bullet & {[3]^{eq}} \\
	& {\Lambda^1[3]} & {[3]}
	\arrow[from=1-1, to=2-1]
	\arrow["{d^2}", from=1-2, to=2-2]
	\arrow[from=1-1, to=1-2]
	\arrow[from=2-1, to=2-2]
	\arrow["\lrcorner"{anchor=center, pos=0.125, rotate=180}, draw=none, from=2-2, to=1-1]
	\arrow[from=3-2, to=2-2]
	\arrow[from=3-3, to=2-3]
	\arrow[from=3-2, to=3-3]
	\arrow[from=2-2, to=2-3]
	\arrow["\lrcorner"{anchor=center, pos=0.125, rotate=-90}, draw=none, from=2-3, to=3-2]
	\arrow[from=1-3, to=2-3]
	\arrow[from=1-4, to=2-4]
	\arrow[from=1-3, to=1-4]
	\arrow[from=2-3, to=2-4]
	\arrow["\lrcorner"{anchor=center, pos=0.125, rotate=180}, draw=none, from=2-4, to=1-3]
\end{tikzcd}\]
This inclusion is then a weak equivalence according to propositions \ref{prop:horn_inclusion} and \ref{prop:thinness extension}.
Now, note that we have a pushout:
\[\begin{tikzcd}
	{[2]_t\amalg [2]_t} & {\Lambda[3]^{eq}} \\
	{[e,2]\coprod [e,2]} & {[e,[3]^{eq}]}
	\arrow[from=1-1, to=2-1]
	\arrow[from=1-1, to=1-2]
	\arrow[from=1-2, to=2-2]
	\arrow[from=2-1, to=2-2]
	\arrow["\lrcorner"{anchor=center, pos=0.125, rotate=180}, draw=none, from=2-2, to=1-1]
\end{tikzcd}\]
As the left vertical morphism is a weak equivalence, so is the right one. 
Let $\Lambda[3]^{\sharp}\to [3]^{\sharp}$ be the entire inclusion generated by $Im(d^3)\cup Im(d^0)\subset [3]$.
Using the same reasoning, we show that this cofibration is acyclic and that there is a weak equivalence $\Lambda[3]^\sharp \to [e,[3]^\sharp]$. We then have a commutative square:
\[\begin{tikzcd}
	{[e,[3]^{eq}]} & {\Lambda[3]^{eq}} & {[3]^{eq}} \\
	{[e,[3]^{\sharp}]} & {\Lambda[3]^{{\sharp}}} & {[3]^{{\sharp}}}
	\arrow["\sim"', from=1-2, to=1-1]
	\arrow["\sim", from=1-2, to=1-3]
	\arrow[from=1-3, to=2-3]
	\arrow["\sim", from=2-2, to=2-1]
	\arrow["\sim"', from=2-2, to=2-3]
	\arrow["\sim"', from=1-1, to=2-1]
	\arrow[from=1-2, to=2-2]
\end{tikzcd}\]
where all arrows labelled by $\sim$ are weak equivalences. By two out of three, this implies that $[3]^{eq}\to [3]^\sharp$ is a weak equivalence. Combined with the lemma \ref{lemma:leibnizt joint is Quillen}, this implies the following proposition:

\begin{prop}
\label{prop:saturation extension}
For any $n\geq -1$, the morphism $[n]\star [3]^{eq}\to [n]\star [3]^{\sharp}$ is an acyclic cofibration.
\end{prop}

\begin{theorem}
\label{theo:Quillen adjunction}
The stratified cosimplicial object constructed in paragraph
\ref{para:definition of the cosimplicial object} induces a Quillen adjunction $\stratSset^\omega\to \stratSeg(A)$.
\end{theorem}
\begin{proof}
It is a direct consequence of theorem \ref{theo:model structure on complicial set} and propositions \ref{prop:horn_inclusion}, \ref{prop:thinness extension}, and \ref{prop:saturation extension}.
\end{proof}

\section{The case $A:=\stratSset^n$}
For $n\in \Nb\cup\{\omega\}$, we denote by $\stratSset^n$ the category of stratified simplicial set endowed with the model structure for $n$-complicial set given in theorem \ref{theo:model structure on complicial set}. As remarked in example \ref{example:stratsset is gray module}, these model categories are Gray modules. The functor $\stratSset\to\stratSeg(\stratSset^n)$ defined in \ref{para:definition of the cosimplicial object} is left Quillen according to theorem \ref{theo:Quillen adjunction}.
 It was noted in paragraph \ref{para:a cocartesian square for intelingent truncation} that for $k>0$, $[k]\to [k]_t$ fits in the following cocartesian square: 
\[\begin{tikzcd}
	{[[k-1],1]} & {[k]} \\
	{[[k-1]_t,1]} & {[k]_t}
	\arrow["{\iota_k}", from=1-1, to=1-2]
	\arrow[from=1-1, to=2-1]
	\arrow[from=2-1, to=2-2]
	\arrow[from=1-2, to=2-2]
\end{tikzcd}\]
The functor $\stratSset\to\stratSeg(\stratSset^n)$ then sends $[k]\to [k]_t$ to an acyclic cofibration for $k>n+1$, and then induces a left Quillen functor 
\begin{equation}
\label{eq:defi of in}
i^{n+1}:\stratSset^{n+1}\to\stratSeg(\stratSset^n)
\end{equation}

\subsection{Comparison with $\zocat$}
We denote by 
\[\begin{tikzcd}
	{\R:\stratSset^\omega} & {\zocat:\N}
	\arrow[""{name=0, anchor=center, inner sep=0}, shift left=2, from=1-1, to=1-2]
	\arrow[""{name=1, anchor=center, inner sep=0}, shift left=2, from=1-2, to=1-1]
	\arrow["\dashv"{anchor=center, rotate=-90}, draw=none, from=0, to=1]
\end{tikzcd}\]
the adjunction between stratified simplicial sets and $\zo$-categories described in section 
\ref{section:Street nerve}. For an $\zo$-category $C$ and an integer $n$, the $\zo$-category $[C,n]$ is defined as the colimit of the following diagram
\[\begin{tikzcd}
	& {[0]} && {[0]} && {[0]} \\
	{\Sigma C} && {\Sigma C} && {...} && {\Sigma C}
	\arrow["{i_0^+}"', from=1-2, to=2-1]
	\arrow["{i_0^-}", from=1-2, to=2-3]
	\arrow["{i_0^+}"', from=1-6, to=2-5]
	\arrow["{i_0^-}", from=1-4, to=2-5]
	\arrow["{i_0^+}"', from=1-4, to=2-3]
	\arrow["{i_0^-}", from=1-6, to=2-7]
\end{tikzcd}\]
This induces an adjunction
\[\begin{tikzcd}
	{\R:\stratSeg(\stratSset)} & {\zocat:\N}
	\arrow[""{name=0, anchor=center, inner sep=0}, shift left=2, from=1-1, to=1-2]
	\arrow[""{name=1, anchor=center, inner sep=0}, shift left=2, from=1-2, to=1-1]
	\arrow["\dashv"{anchor=center, rotate=-90}, draw=none, from=0, to=1]
\end{tikzcd}\]
where the left adjoint sends $[K,n]$ to $[\R(K),n]$ and $[e,1]_t$ on $[0]$.
\sym{(r@$\R:\stratSeg(\stratSset)\to \zocat$}\sym{(n@$\N:\zocat\to \stratSeg(\stratSset)$}
\begin{lemma}
\label{lemma: nerve commute is suspension}
For any $\zo$-category $C$, the canonical morphism 
$$[\N C,1]\to \N[C,1]$$ is an isomorphism. 
\end{lemma}

\begin{proof}
Let $K$ be a stratified simplicial set, $n$ an integer. By construction, we have two cartesian squares
\[\begin{tikzcd}
	{\coprod\limits_{\epsilon\in\{0,1\}}\Hom_{\Delta}([n],\{\epsilon\})\times \Hom_{\stratSset}(K,\N C)} & {\Hom_{\Delta}([n],[1])\times \Hom_{\stratSset}(K,\N C)} \\
	{\coprod\limits_{\epsilon\in\{0,1\}}\Hom_{\Delta}([n],\{\epsilon\})} & {\Hom_{\stratSeg(\stratSset)}([K,n],[\N C,1])}
	\arrow[from=1-1, to=2-1]
	\arrow[from=1-1, to=1-2]
	\arrow[from=2-1, to=2-2]
	\arrow[from=1-2, to=2-2]
\end{tikzcd}\]
\[\begin{tikzcd}
	{\coprod\limits_{\epsilon\in\{0,1\}}\Hom_{\Delta}([n],\{\epsilon\})\times \Hom_{\zocat}(\R(K), C)} & {\Hom_{\Delta}([n],[1])\times \Hom_{\zocat}(\R(K), C)} \\
	{\coprod\limits_{\epsilon\in\{0,1\}}\Hom_{\Delta}([n],\{\epsilon\})} & {\Hom_{\zocat}(\R([K,n]),[C,1])}
	\arrow[from=1-1, to=2-1]
	\arrow[from=1-1, to=1-2]
	\arrow[from=2-1, to=2-2]
	\arrow[from=1-2, to=2-2]
\end{tikzcd}\]
which directly concludes the proof.
\end{proof}

\begin{lemma}
\label{lemma:compairaon beetwen join and the formula of chap 2}
Let $C$ be an $\zo$-category and $n$ an integer. There is a canonical commutative square in $\zocat$: 
\[\begin{tikzcd}
	{\coprod\limits_{k\leq n}\colim_{\Delta^2_{/\{k\}}}[[n_0]\otimes C,1]\vee[C,n_1]} & {\colim_{\Delta^2_{/[n]}}[[n_0]\otimes C,1]\vee[C,n_1]} \\
	{\coprod\limits_{k\leq n}\colim_{\Delta^2_{/\{k\}}}[[n_0],1]\vee[n_1]} & {1\star [C,n]}
	\arrow[from=1-2, to=2-2]
	\arrow[from=1-1, to=2-1]
	\arrow[from=2-1, to=2-2]
	\arrow[from=1-1, to=1-2]
\end{tikzcd}\]
natural in $C:\zocat$ and $[n]:\Delta$.
\end{lemma}
\begin{proof}
In this proof, we use the Steiner theory recalled in section \ref{section:Steiner thery}.
It is sufficient to show the assertion when $C$ is a globular form, and then \textit{a fortiori}, an $\zo$-category with an atomic and loop free basis. Using the equivalence between $\zocatB$ and $\CDAB$ given in \ref{theorem:steiner} and the equivalences
$$(K\otimes L)^{op} \sim L^{op}\otimes K^{op}~~~(K\otimes L)^{co} \sim L^{co}\otimes K^{co} ~~~ (1\star K)^{op}\sim K^{op}\star 1$$
provided by propositions A.20 and 6.10 of \cite{Ara_Maltsiniotis_joint_et_tranche}, it is sufficient to construct for every augmented direct complex $K$ a natural commutative square: 
\[\begin{tikzcd}
	{\coprod_{k\leq n}\colim_{[n_1]\star [n_0]\to \{k\}}[K,n_1]\vee[K\otimes\lambda[n_0],1]} & {\colim_{[n_1]\star [n_0]\to [n]}[K,n_1]\vee[K\otimes \lambda[n_0],1]} \\
	{\coprod_{k\leq n}\colim_{[n_1]\star [n_0]\to \{k\}}\lambda[n_1]\vee[\lambda[n_0],1]} & {[K,n]\star 1}
	\arrow[from=2-1, to=2-2]
	\arrow[from=1-1, to=2-1]
	\arrow[from=1-2, to=2-2]
	\arrow[from=1-1, to=1-2]
\end{tikzcd}\]

For an element $f:[n_0]\star[n_1]\to [n]$ of $\Delta^2_{/[n]}$, we considere the morphism 
$\phi_f:[K,n_1]\vee[K\otimes \lambda[n_0],1]\to [K,n]\star 1$ as the unique morphism fulfilling 
$$\phi_f(	[x,v_{i,i+1}]):= [x,v_{f_0(i),f_0(i)+1}]\star \emptyset+...+ [x,v_{f_0(i)-1,f_0(i+1)}]\star \emptyset$$
$$\phi_f(	[x\otimes v_i,1]):= 0$$
$$\phi_f(	[x\otimes v_{i,i+1},1]):= [x,v_{f_1(i),f_1(i)+1}]\star 1+...+ [x,v_{f_1(i)-1,f_1(i+1)}]\star 1$$
for $x$ an element of $K$ and where we denote by $f_0$ and $f_1$ the induced morphisms $[n_0]\to [n_0]\star [n_1]\to [n]$ and $[n_1]\to [n_0]\star [n_1]\to [n]$.

Peforming this for any such $f:[n_0]\star[n_1]\to [n]$ of $\Delta^2_{/[n]}$, this induces a morphism 
$$\psi:\colim_{\Delta^2_{/[n]}}[[n_0]\otimes a,1]\vee[a,n_1]\to 1\star [a,n]$$
whose restriction to $\coprod\limits_{k\leq n}\colim_{\Delta^2_{/\{k\}}}[[n_0]\otimes a,1]\vee[a,n_1]$ factors through $\coprod\limits_{k\leq n}\colim_{\Delta^2_{/\{k\}}}[[n_0],1]\vee[1,n_1]$ and this concludes the proof.
\end{proof}

\begin{lemma}
\label{lemma:joi commutes with realization}
There is an invertible natural transformation $\R(e\star\uvar)\to 1\star \R(\uvar)$ that firs in a commutative square
\[\begin{tikzcd}
	{\R(\emptyset\star\uvar)} & {\R(e\star\uvar)} \\
	{\emptyset\star \R(\uvar)} & { 1\star \R(\uvar)}
	\arrow[from=2-1, to=2-2]
	\arrow[from=1-1, to=1-2]
	\arrow[from=1-2, to=2-2]
	\arrow["id"', from=1-1, to=2-1]
\end{tikzcd}\]
\end{lemma}
\begin{proof}
The lemma \ref{lemma:compairaon beetwen join and the formula of chap 2} provides such natural transformation. As $\R$ sends weak equivalences to isomorphisms, it is sufficient to show that $\R(e\star [K,1])\to 1\star [\R(K),1]$ is an equivalence, which directly follows from the explicit description of these two objects provided by proposition \ref{prop:explicit expression of e star a,1} and by the example \ref{exe:explicit Gray cone 1}.
\end{proof}

\begin{prop}
\label{prop:first triangle}
The following triangle commutes up to an invertible natural transformation
\[\begin{tikzcd}
	& {\stratSeg(\stratSset^n)} \\
	{\stratSset^{n+1}} & \zocat
	\arrow["\R", from=1-2, to=2-2]
	\arrow["\R"', from=2-1, to=2-2]
	\arrow["{i^{n+1}}", from=2-1, to=1-2]
\end{tikzcd}\]
For any integer $k\leq n+1$, the induced morphism $i^{n+1}(\N \Db_k)\to \N(\Db_k)$ is a weak equivalence.
\end{prop}
\begin{proof}
It is sufficient to show the result for $n:=\omega$.
The lemma \ref{lemma:joi commutes with realization} provides an invertible transformation $\phi:(\R i^{\omega})_{|\Delta}\to \R_{|\Delta}$ which is natural when restricted to the full subcategory of $\Delta$ whose morphisms are the monomorphisms. 
The lemma \ref{lemma:about_P_modified} then implies that $\phi:(\R i^{\omega})_{|\Delta}\to \R_{|\Delta}$ is natural. As all these functors commute with the intelligent truncations, we can extend it to a natural transformation $\phi:(\R i^{\omega})_{|t\Delta}\to \R_{|t\Delta}$. 
Eventually, as all theses morphisms preserves colimits, we can extend $\phi$ to an invertible natural transformation $\phi:\R i^{\omega}\to \R$.

We now turn our attention to the second assertion. We define the functor $\Sigma^{\circ}:\stratSset\to \stratSset$ that sends a stratified simplicial set $K$ onto the following pushout:
\[\begin{tikzcd}
	K & {1\star K} \\
	1 & {\Sigma^{\circ}K}
	\arrow[from=1-1, to=2-1]
	\arrow[from=2-1, to=2-2]
	\arrow[from=1-1, to=1-2]
	\arrow[from=1-2, to=2-2]
	\arrow["\lrcorner"{anchor=center, pos=0.125, rotate=180}, draw=none, from=2-2, to=1-1]
\end{tikzcd}\]
Remark that we have a canonical equivalence
$$(\Sigma^{\circ}X)^{op} \sim \Sigma^\star X^{op}$$
where $\Sigma^\star$ is the functor defined in paragraph \ref{para:sigma star}.
As the nerve commutes with the op-dualities, and as globes are invariant under it, a repeated application of \cite[theorem 3.22]{Ozornova_a_quillen_adjunction_between_globular_and_complicial} imply that the following canonical morphism between stratified simplicial sets
$$(\Sigma^\circ)^k[0]\to \N(\Db_k)$$
is an acylic cofibration.
Furthermore, proposition \ref{labe:Link between the Gray cylinder and cosuspension} provides a weak equivalence
$$i^{n+1}(\Sigma^{\circ} K)\to \Sigma^{\circ} K.$$
A direct induction then induces a weak equivalence 
$$i^{n+1}((\Sigma^\circ)^k[0])\to (\Sigma^\circ)^k[0]$$

Otherwise, remark that by construction, $\Sigma^{\circ}[K,1]:=[[0]\diamond K\coprod_{K}[0],1]$. The weak equivalence $[0]\diamond K\to [0]\star K$ provided by proposition \ref{prop:equivalence between diamond and join product} induces a weak equivalence
$$\Sigma^{\circ}[K,1]\to [\Sigma^{\circ}K,1].$$
As $\Sigma^{\circ}[0] = [[0],1]$, a direct induction induces a weak equivalence 
$$(\Sigma^\circ)^k[0]\to [(\Sigma^\circ)^{k-1}([0]),1].$$

All put together, and using lemma \ref{lemma: nerve commute is suspension}, this induces two acyclic cofibrations
$$\begin{array}{rl}
\psi_k:&i^{n+1}((\Sigma^\circ)^k[0])\xrightarrow{\sim} \N \Db_{k}\\
\psi_k':&i^{n+1}((\Sigma^\circ)^k[0])\xrightarrow{\sim}(\Sigma^\circ)^k[0]\xrightarrow{\sim}[(\Sigma^\circ)^{k-1}[0],1]\xrightarrow{\sim}[\N\Db_{k-1},1]\cong \N \Db_k
\end{array}
$$

To concludes, one have to show that the induces diagram
\[\begin{tikzcd}
	{ i^{n+1}((\Sigma^\circ)^k[0])} & { i^{n+1}(\N \Db_{k} )} \\
	& {\N \Db_k }
	\arrow["{ \psi_k}", from=1-1, to=1-2]
	\arrow[from=1-2, to=2-2]
	\arrow["{\psi'_k}"', from=1-1, to=2-2]
\end{tikzcd}\]
commutes. By adjunction, this is sufficient to show that the diagram 
\[\begin{tikzcd}
	{\R  i^{n+1}((\Sigma^\circ)^k[0])} & {\R i^{n+1}(\N \Db_{k} )} \\
	& {\R \N \Db_k }
	\arrow["{\R  \psi_k}", from=1-1, to=1-2]
	\arrow["{\R \psi'_k}"', from=1-1, to=2-2]
	\arrow["{\phi_{\N \Db_k}}", from=1-2, to=2-2]
\end{tikzcd}\]
commutes. We claim that $\R \N \Db_k$ has no non-trivial automorphisms. This  directly implies the results as $\R$ sends acyclic cofibrations to isomorphisms.

It then remains to show that $\R \N \Db_k$ has no non-trivial automorphisms. If $k=0$, this is trivial as $\R \N \Db_0\cong \Db_0$.
We suppose now that $k>0$. As $\R$ commutes with the suspension and sends acyclic cofibration to isomorphism, the lemma  \ref{lemma: nerve commute is suspension} and a repeated application of the theorem \ref{theo:strict susension} imply that the morphism
$$\begin{array}{rcl}
\Db_k &= &[\Db_{k-1},1]\\
&\cong &[\Sigma^{k-1}\R \N \Db_0,1] \\
&\cong &\R [\Sigma^{k-1} \N \Db_0,1 ]\\
&\to& \R [\N \Sigma^{k-1} \Db_0,1]\\
&\cong &\R \N  [ \Sigma^{k-1} \Db_0,1]  \\
&= & \R \N \Db_k \\ 
\end{array}
$$
is an isomorphism. The result then follows from proposition \ref{prop:the globes a non non trivial automorphisms} that states that $\Db_k$ has no non-trivial automorphisms.
\end{proof}

\subsection{The other adjunction}
We define the colimit preserving functor 
\begin{equation}
\label{eq:defi of jn}
j:\stratSeg(\stratSset)\to \stratSset
\end{equation}
 sending $[K,n]$ to the pushout:
\[\begin{tikzcd}
	{\cup_{i\leq n}K\boxtimes\{i\}} & {K\boxtimes[n]} \\
	{\cup_{i\leq n}[0]} & {j([K,n])}
	\arrow[from=1-1, to=2-1]
	\arrow[from=2-1, to=2-2]
	\arrow[from=1-1, to=1-2]
	\arrow[from=1-2, to=2-2]
	\arrow["\lrcorner"{anchor=center, pos=0.125, rotate=180}, draw=none, from=2-2, to=1-1]
\end{tikzcd}\]
 and $[[0],1]_t$ to $[1]_t$. As $\uvar\boxtimes\uvar$ is a left Quillen bifunctor, and as $j([[0],1]_t\to [0])=[1]_t\to [0]$ and $j([[0],E^{\cong}]\to [[0],(E^{\cong})'])= E^{\cong}\to (E^{\cong})'$ are weak equivalences, 
the proposition \ref{prop:model structure on stratified Segal category} implies that the functor 
$$j^{\omega}:\stratSeg(\stratSset^{\omega})\to \stratSset^{\omega}$$ is a left Quillen functor. By definition of the Gray pre-tensor given in \cite[Definition 128]{Verity_weak_complicial_sets_I}, we remark that $j([[k],n]\to [[k]_t,n])$ is a pushout of a disjoint union of $[k+1]\to [k+1]_t$. This implies that for any $n\in \mathbb{N}$, 
$$j^{n+1}:\stratSeg(\stratSset^{n})\to \stratSset^{n+1}$$ 
 is a left Quillen functor.
\begin{prop}
\label{prop:second triangle}
The following triangle commutes up to an invertible natural transformation:
\[\begin{tikzcd}
	& {\stratSset^{n+1}} \\
	{\stratSeg(\stratSset^n)} & \zocat
	\arrow["\R"', from=2-1, to=2-2]
	\arrow["\R", from=1-2, to=2-2]
	\arrow["{j^{n+1}}", from=2-1, to=1-2]
\end{tikzcd}\]
For any integer $k\leq n+1$, the induced morphism $j^{n+1}(\N \Db_k)\to \N(\Db_k)$ is a weak equivalence.
\end{prop}
\begin{proof}
The first assertion is a direct consequence of the definition of $\R:\stratSeg(\stratSset^n)\to \zocat$ and the corrolary \ref{cor:crushing of Gray tensor is identitye strict case}. We denote $\phi: \R j^{n+1}\to \R$ the corresponding invertible natural transformation. 

For the second assertion, remark that the case $k=0$ is trivial, and for $k>0$, lemma \ref{lemma: nerve commute is suspension}, theorem \ref{theo:strict susension} and the definition of $j^{n+1}$ induce a weak equivalence
$$\psi_k:j^{n+1}(\N \Db_k)\cong j^{n+1}([\N\Db_{k-1},1])= \Sigma \N \Db_{k-1}\to \N [\Db_{k-1},1] = \N \Db_k$$
To conclude, one have to show that $\phi_{\N \Db_k}$ is equal to $\R \psi_k$. We claim that $\R \N \Db_k$ has no non-trivial automorphisms. This  directly implies the results as $\R$ sends acyclic cofibrations to isomorphisms.

It then remains to show that $\R \N \Db_k$ has no non-trivial automorphisms. As $\R$ commutes with the suspension and sends acyclic cofibration to isomorphism, a repeated application of the theorem \ref{theo:strict susension} implies that the morphism
$$\Db_k = \Sigma^k \Db_0\cong \Sigma^k \R\N \Db_0 \cong \R \Sigma^k \N \Db_0\to \R \N \Sigma^k \Db_0\cong  \R\N \Db_k$$
is an isomorphism. The result then follows from proposition \ref{prop:the globes a non non trivial automorphisms} that states that $\Db_k$ has no non-trivial automorphisms.
\end{proof}

\subsection{Complicial sets as a model of $\io$-categories}

\begin{prop}
\label{prop:first equivalence}
For any $n\in \Nb\cup\{\omega\}$, the composite 
$$j^{n+1}\circ i^{n+1}:\stratSset^{n+1}\to \stratSset^{n+1}$$
is a Quillen equivalence.
\end{prop}
\begin{proof}
Using theorem \ref{theo:strict susension}, and propositions \ref{prop:first triangle} and \ref{prop:second triangle},
we have a zigzag of weak equivalences
$$j^{\omega}\circ i^{\omega}(\Db_n)\to j^{\omega}\circ i^{\omega}(\N(\Db_n))\to \N(\Db_n)\leftarrow \Db_n$$
natural in $n$.
The corollary \ref{cor:criterion_to_be_linked_to_identity_case stratified} then provides a zigzag of weakly invertible natural transformations
$$j^{\omega}\circ i^{\omega}\leftrightsquigarrow id_{\stratSset^{\omega}}.$$
This also induces for any integer $n$ a zigzag of weakly invertible natural transformations
$$j^{n+1}\circ i^{n+1}\leftrightsquigarrow id_{\stratSset^{n+1}}.$$
\end{proof}

\begin{theorem}
\label{theo:letheo}
For $n<\omega$, the model category $\stratSset^{n}$ is a model of $(\infty,n)$-categories.
\end{theorem}
\begin{proof}
To demonstrate the theorem, we will proceed by induction. The initialization is exactly the theorem 2.14 of \cite{Bergner_explicit_comparaison_bt_theta_2_space_and_2_complicial_set}. Suppose now the result is true at the stage $n$. We can apply \cite[example 15.8]{Barwick_on_the_unicity_of_the_theory_of_higher_categories} which implies that the $(\infty,1)$-category represented by $\Seg(\stratSset^n)$ is a model of $(\infty,n+1)$-categories, and according to \ref{prop:model structure on stratified Segal category}, so is $\stratSeg(\stratSset^n)$. Eventually, the proposition \ref{prop:first triangle} and \ref{prop:second triangle} imply that the functor
$$i^{n+1}\circ j^{n+1}:\stratSeg(\stratSset^n)\to \stratSeg(\stratSset^n)$$ preserves globes up to homotopy.
Proposition 15.10 of \cite{Barwick_on_the_unicity_of_the_theory_of_higher_categories} states that $i^{n+1}\circ j^{n+1}$ is a Quillen equivalence, and proposition \ref{prop:first equivalence} implies that $j^{n+1}\circ i^{n+1}$ is a Quillen equivalence. The functor $i^{n+1}$ is then a Quillen equivalence, and $\stratSset^{n+1}$ is a model of $(\infty,n+1)$-categories.
\end{proof}

\p For an integer $n$, we consider the model structure on $\ssetPsh{\Theta_n}$ (resp. $\ssetPsh{\Theta}$) obtained as the left Bousfield localization of the projective model structure along the set of map $\W_n$ (resp. $\W$ ) defined in paragraph \ref{para:definition of W}.
For any $n<\omega$, the inclusion $\Theta_n\to \Theta$ induces a Quillen adjunction
\begin{equation}
\label{eq:adjunction beetwen theta n and theta}
\begin{tikzcd}
	{\iota^n:\ssetPsh{\Theta_n}} & {\ssetPsh{\Theta}:\tau_n}
	\arrow[""{name=0, anchor=center, inner sep=0}, shift left=2, from=1-1, to=1-2]
	\arrow[""{name=1, anchor=center, inner sep=0}, shift left=2, from=1-2, to=1-1]
	\arrow["\dashv"{anchor=center, rotate=-90}, draw=none, from=0, to=1]
\end{tikzcd}
\end{equation}

\p Let $n\in\Nb\cup\{\omega\}$.
We consider the functor
$$\Theta_n\times \Delta\to \stratSset$$
sending a pair $(a,[n])$ onto $\N(a)\times \tau^i_0([n])$. By left Kan extension, this induces an adjunction
\begin{equation}
\label{eq:adjunction betwen theta and complicial}
\begin{tikzcd}
	{L_n:\ssetPsh{\Theta_n}} & {\stratSset:N_{L_n}}
	\arrow[""{name=0, anchor=center, inner sep=0}, shift left=2, from=1-1, to=1-2]
	\arrow[""{name=1, anchor=center, inner sep=0}, shift left=2, from=1-2, to=1-1]
	\arrow["\dashv"{anchor=center, rotate=-90}, draw=none, from=0, to=1]
\end{tikzcd}
\end{equation}

\begin{theorem}[Ozornova-Rovelli]
\label{theo:fondamental adj}
The adjunction 
\[\begin{tikzcd}
	{L_n:\ssetPsh{\Theta_n}} & {\stratSset^n:N_{L_n}}
	\arrow[""{name=0, anchor=center, inner sep=0}, shift left=2, from=1-1, to=1-2]
	\arrow[""{name=1, anchor=center, inner sep=0}, shift left=2, from=1-2, to=1-1]
	\arrow["\dashv"{anchor=center, rotate=-90}, draw=none, from=0, to=1]
\end{tikzcd}\]
is a Quillen adjunction.
\end{theorem}
\begin{proof}
This is \cite[theorem 4.16]{Ozornova_a_quillen_adjunction_between_globular_and_complicial}.
\end{proof}

\begin{remark}
The two authors demonstrate this result when $\stratSset$ is endowed with the model structure for $n$-complicial sets with $n<\omega$. However, their argument generalizes directly to the case $n=\omega$.
\end{remark}

A direct induction using \cite[theorem 3.22]{Ozornova_a_quillen_adjunction_between_globular_and_complicial} implies that the left adjoint preserves globes.

\begin{prop}
\label{prop:intermedeire}
For any $n\in \Nb$, the adjunction given in theorem \ref{theo:fondamental adj}
is a Quillen equivalence. 
\end{prop}
\begin{proof}
This is an adjunction between two models of $(\infty,n)$-categories. As the left adjoint preserves globes up to homotopy, the result follows from \cite[proposition 15.10]{Barwick_on_the_unicity_of_the_theory_of_higher_categories}.
\end{proof}
 
 \p If $C$ is a model category, we denote by $C^{\iun}$ the corresponding $\iun$-category.

\begin{lemma}
\label{lemma:iota homotpically fully faitfhfull}
For any integer $n$, the $\iun$-functor 
$$\iota^n:(\ssetPsh{\Theta_n})^{\iun}\to (\ssetPsh{\Theta})^{\iun}$$
is fully faithful.
\end{lemma}
\begin{proof}
This is proposition \ref{ref:infini n a full sub cat}.
\end{proof}

\begin{lemma}
\label{lemma:the composite is ff}
For any integer $n$, the $\iun$-functor $$\tau^i_n:(\stratSset^n)^{\iun}\to (\stratSset^\omega)^{\iun}$$
is fully faithful.
\end{lemma}
\begin{proof}
This is a direct consequence of the fact that  $\stratSset^n$ is the left Bousfield localization of $\stratSset^\omega$ along morphisms $[m]\to [m]_t$ for $m>n$.
\end{proof}

\begin{lemma}
\label{lemma:L ff}
The $\iun$-functor $L_\omega: (\ssetPsh{\Theta})^{\iun}\to (\stratSset^\omega)^{\iun}$ is fully faithful.
\end{lemma}
\begin{proof}
We have to show that for any pair of $\Theta$-spaces $X$ and $Y$, the induced morphism of $\infty$-groupoids
$$\Hom_{(\ssetPsh{\Theta})^{\iun}}(X,Y)\to \Hom_{(\stratSset^\omega)^{\iun}}(L_\omega(X),L_\omega(Y))$$
is an equivalence. 
As every  $\Theta$-space is a $\iun$-colimit of globular sums, which are themself $\iun$-colimits of globes, we can suppose that $X$ is of shape $\Db_n$. In this case $\Db_n$ is $\omega$-small. As $L(\Db_n)$ has a finite presentation, given by the $n$-times interated suspension of $[0]$, it is also $\omega$-small.

Eventually, proposition \ref{prop:infini omega a limit of infini n} implies that every $\Theta$-spaces is a directed colimit of objects that are in the image of  $\iota_n$ for an integer $n$. We can then restrict ourselves to the case where $Y$ is in the image of $\iota_n$. As we have an equivalences $L_{\omega}\circ \iota_n\sim \tau^i_n  \circ L_n$, the results follows from
proposition \ref{prop:intermedeire}, and lemmas \ref{lemma:iota homotpically fully faitfhfull} and \ref{lemma:the composite is ff}.
\end{proof}

\begin{theorem}
\label{theo:lecorozo}
For any $n\in \Nb\cup\{\omega\}$, the adjunction
\[\begin{tikzcd}
	{L_n:\ssetPsh{\Theta_n}} & {\stratSset^\omega:N_{L_n}}
	\arrow[""{name=0, anchor=center, inner sep=0}, shift left=2, from=1-1, to=1-2]
	\arrow[""{name=1, anchor=center, inner sep=0}, shift left=2, from=1-2, to=1-1]
	\arrow["\dashv"{anchor=center, rotate=-90}, draw=none, from=0, to=1]
\end{tikzcd}\]
is a Quillen equivalence. 
The two induced diagrams
\[\begin{tikzcd}
	\zocat & {\stratSset^\omega} & \zocat & {\ssetPsh{\Theta}} \\
	{\ssetPsh{\Theta}} & \zocat & {\stratSset^\omega} & \zocat
	\arrow["{\pi_0}"', from=2-1, to=2-2]
	\arrow["{L_\omega}"{description}, from=2-1, to=1-2]
	\arrow["\R", from=1-2, to=2-2]
	\arrow["\iota"', from=1-1, to=2-1]
	\arrow["\N", from=1-1, to=1-2]
	\arrow["{N_{L_\omega}}"{description}, from=2-3, to=1-4]
	\arrow["\iota", from=1-3, to=1-4]
	\arrow["\N"', from=1-3, to=2-3]
	\arrow["\R"', from=2-3, to=2-4]
	\arrow["{\pi_0}", from=1-4, to=2-4]
\end{tikzcd}\]
commute up to homotopy.
\end{theorem}
\begin{proof}
If $n<\omega$, the first assertion is a consequence of proposition  \ref{prop:intermedeire}. Suppose now that $n=\omega$.
The lemma \ref{lemma:L ff} implies that the left adjoint is homotopically fully faithful. It then remains to show that the right adjoint is conservative. This is a direct consequence of the preservation of globes by $L_{\omega}$ up to homotopy and theorem \ref{theo:f_weak_equivalence_ssi_f_G_equivalence}.

For the second assertion, it it sufficient to demonstrate that the restriction to $\Theta$ of the canonical natural transformation $\R \circ L_{\omega} \to \pi_0$ is an isomorphism. As these two functors send Segal extensions on isomorphisms, it it sufficient to show the result on globes where it directly follows from the preservation of globes by $L_{\omega}$ up to homotopy.
\end{proof}

%
%
%

%
%
%
%
%
%
%
%
%
%

\part{On the side of theory}

\chapter{The $(\infty,1)$-category of $\io$-categories}
\label{chapter:the infini 1 categorory of io categories}

\minitoc
\vspace{2cm}
This chapter is dedicated to the basic definition of $\io$-categories. In the first section, we recall some results on factorization systems in presentable $\iun$-categories. In the second section, we define $\io$-categories and give some basic properties. 
We also define and study \textit{discrete Conduché functor}, which are morphisms having the unique right lifting property against 
units $\Ib_{n+1}:\Db_{n+1}\to \Db_n$ for any integer $n$, and against compositions $\triangledown_{k,n}:\Db_n\to \Db_n\coprod_{\Db_k}\Db_n$ for any pair of integers $k\leq n$. This notion was originally defined and studied in the context of strict $\omega$-category by Guetta in \cite{Guetta_conduche}.
\begin{itheorem}[\ref{theo:pullback along conduche preserves colimits}]
Let $f:C\to D$ be a discrete Conduché functor. The pullback functor $f^*:\ocat_{/D}\to \ocat_{/C}$ preserves colimits.
\end{itheorem}

 In the third section, we study Gray operations for $\io$-categories. We conclude this chapter by proving results of strictification. In particular, we demonstrate the following theorem:
\begin{itheorem}[\ref{prop:strict stuff are pushout}]
Let $C$ be an $\io$-category, $b$ a globular sum, and $f:b\to C$ any morphism. The $\io$-categories $$1\costar b\coprod_b C,~C\coprod_b b\otimes[1]~\mbox{and}~C\coprod_b b\star 1$$
are strict whenever $C$ is.
\end{itheorem}
We will also prove the following theorem:
\begin{itheorem}[\ref{theo:strictness}]
If $C$ is strict, so are $C\star 1$, $1\costar C$ and $C\otimes [1]$.
\end{itheorem}
In the process, we will demonstrate another fundamental equation combining $C\otimes[1]$, $1\costar C$, $C\star 1$, and $[C,1]$.
\begin{itheorem}[\ref{theo:formula between pullback of slice and tensor}]
Let $C$ be an $\io$-category. The five squares appearing in the following canonical diagram are both cartesian and cocartesian:
\[\begin{tikzcd}
	& {C\otimes\{0\}} & 1 \\
	{C\otimes\{1\}} & {C\otimes[1]} & {C\star 1} \\
	1 & {1\costar C} & {[C,1]}
	\arrow[from=2-3, to=3-3]
	\arrow[from=3-2, to=3-3]
	\arrow[from=2-2, to=3-2]
	\arrow[from=2-2, to=2-3]
	\arrow[from=1-2, to=1-3]
	\arrow[from=1-3, to=2-3]
	\arrow[from=1-2, to=2-2]
	\arrow[from=2-1, to=2-2]
	\arrow[from=3-1, to=3-2]
	\arrow[from=2-1, to=3-1]
\end{tikzcd}\]
where $[C,1]$ is the \textit{suspension of $C$}.
\end{itheorem}

\paragraph{About the use of the language of $(\infty,1)$-categories.}
In this chapter and the two following, we will freely use the language of $\iun$-categories\footnote{  As there are currently several directions for the formalization of the language of $\iun$-categories (\cite{Riehl_element_of_infini_categories}, \cite{Riehl_A_type_theory_for_synthetic_-categories}, \cite{North_Towards_a_directed_homotopy_type_theory}, \cite{Cisinski-Univalent-Directed-Type-theory}), talking about "the" language of (infinite,1)-categories may be confusing.

In such case, the reader may consider that we are working within the quasi-category $\qcat$ of $\Tb$-small quasi-categories for $\Tb$ a Grothendieck universe. This quasi-category may be obtained either using the coherent nerve as described in \cite[chapter 3]{Lurie_Htt}, or by considering it as the codomain of the universal cocartesian fibration with $\Tb$-small fibers as done in \cite{Cisinski_The_universal_coCartesian_fibration}. In both cases, the straightening/unstraightening correspondence provides a morphism
$$\N(\Sset_{\Tb})\to \qcat$$
that exhibits $\qcat$ as the quasi-categorical localization of $\N(\Sset_{\Tb})$ with respect to the weak equivalences of the Joyal's model structure (\cite[theorem 8.13]{Cisinski_The_universal_coCartesian_fibration}). 

The constructions we use to build new objects - (co)limits of functor between quasi-categories, quasi-categories of functor, localization of quasi-categories, sub maximal Kan complex, full sub quasi-category, adjunction, left and right Kan extension, Yoneda lemma - are well documented in the Joyal model structure (see \cite{Lurie_Htt} or \cite{Cisinski_Higher_categories_and_homotopical_algebra})
, and therefore have direct incarnation in the quasi-category $\qcat$. }.

We allow ourselves the following abuse of language: when a $\infty$-groupoid $X$ is contractible, we will use the expression \textit{the element of $X$} to refer to any element of $X$. For example, we'll talk about \textit{the} composition of two functors, or \textit{the} colimit/limit of a functor. The adjective \textit{unique} should be understood as \textit{the $\infty$-groupoid of choice is contractible}.

An equivalence $v$ in a $\iun$-category $C$ between an object $a$ and an object $b$ is denoted by $v:a\sim b$.

The maximal sub $\infty$-groupoid of an $\iun$-category $C$ is denoted by $\tau_0(C)$.

Eventually, we will identify (strict) categories with the $\iun$-categories obtained by applying the simplicial nerve.

\paragraph{Cardinality hypothesis.}
We fix during this chapter three Grothendieck universes $\U \in \V\in\Wcard$, such that $\omega\in \U$. 
All defined notions depend on a choice of cardinality. When nothing is specified, this corresponds to the implicit choice of the cardinality $\V$.
With this convention in mind, we denote by {$\Set$} the $\Wcard$-small $1$-category of $\V$-small sets, {$\igrd$} the $\Wcard$-small $\iun$-category of $\V$-small $\infty$-groupoids and {$\icat$} the $\Wcard$-small $\iun$-category of $\V$-small $\iun$-categories.

\section{Preliminaries}
\subsection{Explicit computation of some colimits}

\p
We have an adjunction:
\begin{equation}
\label{eq:adj betwen set and space}
\begin{tikzcd}
	{\pi_0:\igrd} & {\Set:\iota}
	\arrow[""{name=0, anchor=center, inner sep=0}, shift left=2, from=1-2, to=1-1]
	\arrow[""{name=1, anchor=center, inner sep=0}, shift left=2, from=1-1, to=1-2]
	\arrow["\dashv"{anchor=center, rotate=-90}, draw=none, from=1, to=0]
\end{tikzcd}
\end{equation}
For a category $B$, we denote by {$\Psh{B}$} the category of functors $B^{op}\to \Set$.
For a $\iun$-category $A$, we denote by \wcnotation{$\iPsh{A}$}{(psh@$\iPsh{\uvar}$} the $\iun$-category of functors $A^{op}\to \igrd$. A presheaf on $B$, (resp. a $\infty$-presheaves on $A$) is \textit{$\U$-small} if it is pointwise a $\U$-small set (resp. a $\U$-small $\infty$-groupoid).

\p If $A$ is a $1$-category, the adjunction \eqref{eq:adj betwen set and space} induces an adjunction:
\begin{equation}
\label{eq:adj betwen A set and A space}
\begin{tikzcd}
	{\pi_0:\iPsh{A}} & {\Psh{A}:\iota}
	\arrow[""{name=0, anchor=center, inner sep=0}, shift left=2, from=1-2, to=1-1]
	\arrow[""{name=1, anchor=center, inner sep=0}, shift left=2, from=1-1, to=1-2]
	\arrow["\dashv"{anchor=center, rotate=-90}, draw=none, from=1, to=0]
\end{tikzcd}
\end{equation}

\p We recall that the notion of {elegant Reedy category} is defined in paragraph \ref{para:reedy}.
The following lemma provides a powerful way to compute simple colimits in $\iun$-categories by reducing to computations in (stricts) categories. These techniques will be used freely in the rest of this text.

\begin{lemma}
\label{lemma:colimit computed in set presheaves}
Let $A$ be a $\V$-small category. We denote $\iota:\Psh{A}\to \iPsh{A}$ the canonical inclusion.
\begin{enumerate}
\item 
The functor $\iota$ preserves cocartesian square 
\[\begin{tikzcd}
	a & b \\
	c & d
	\arrow[from=1-1, to=2-1]
	\arrow[from=1-2, to=2-2]
	\arrow[from=1-1, to=1-2]
	\arrow[from=2-1, to=2-2]
	\arrow["\lrcorner"{anchor=center, pos=0.125, rotate=180}, draw=none, from=2-2, to=1-1]
\end{tikzcd}\]
where the left vertical morphism is a monomorphism.
\item 
The functor $\iota$ preserves colimit of finite diagrams of shape: 
\[\begin{tikzcd}
	& \bullet && {...} && \bullet \\
	\bullet && \bullet && \bullet && \bullet
	\arrow[from=1-2, to=2-1]
	\arrow[hook, from=1-2, to=2-3]
	\arrow[from=1-4, to=2-3]
	\arrow[hook, from=1-4, to=2-5]
	\arrow[from=1-6, to=2-5]
	\arrow[hook, from=1-6, to=2-7]
\end{tikzcd}\]
where morphisms labeled $\hookrightarrow$ are monomorphisms.
\item The functor $\iota$ preserves transfinite composition. 
\item For any $\V$-small elegant Reedy category, and any functor $F:I\to \Psh{A}$ that is Reedy cofibrant, i.e such that for any $i\in I$, $\colim_{\partial i}F\to F(i)$ is a monomorphism,
the canonical comparison 
$$\iota \colim F\to \colim \iota F$$
is an isomorphism. In particular, if $A$ is itself an elegant Reedy category, for any set-valued presheaf $X$ on $A$, there is an equivalence 
$$\iota(X)\sim \colim_{A_{/X}}a.$$ 
\end{enumerate}
\end{lemma}
\begin{proof}
For this result, we use model categories. We consider the interval induces by the constant functor $I:A\to \Psh{\Delta}$ with value $[1]$. We then consider the model structure on $\Psh{A\times \Delta}$ produced by \cite[theorem 1.3.22]{cisinski_prefaisceaux_comme_modele} and induces by the homotopical data $(I\times \uvar,\emptyset)$. This model structure represent $\iPsh{A}$.
To conclude, we then have to show that all the given colimits, seen as (simplicialy constant) presheaves on $\Delta\times A$ are also homotopy colimits of the same diagrams. This then follows from proposition \ref{prop:hom colimit 2}, \ref{prop:hom colimit 3}, \ref{prop:hom colimit 4} and theorem \ref{theo:hom colimi}.
\end{proof}

\subsection{Factorization sytems}
\label{section:Factorization system}
\p For the rest of the section, we fix a \textit{presentable $\iun$-category} $C$, i.e a $\iun$-category $C$ that is a reflexive and $\V$-accessible localization of a $\iun$-category of $\infty$-presheaves on a $\V$-small $\iun$-category.

A full sub $\infty$-groupoid of the $\infty$-groupoid of arrows of $C$ is \wcnotionsym{cocomplete}{(s@$\widehat{S}$}{cocomplete $\infty$-groupoid of arrows} if it is closed under colimit and composition and contains the equivalences. For a $\infty$-groupoid $S$, we define $\widehat{S}$ as the smallest cocomplete full sub $\infty$-groupoid of the $\infty$-groupoid of arrows containing $S$. 

\begin{remark}
A cocomplete full sub $\infty$-groupoid $U$ is closed by pushouts along any morphism. Indeed, suppose given a cocartesian square
\[\begin{tikzcd}[row sep=scriptsize]
	a & b \\
	c & d
	\arrow["f"', from=1-1, to=2-1]
	\arrow[from=1-1, to=1-2]
	\arrow["{f'}", from=1-2, to=2-2]
	\arrow[from=2-1, to=2-2]
	\arrow["\lrcorner"{anchor=center, pos=0.125, rotate=180}, draw=none, from=2-2, to=1-1]
\end{tikzcd}\]
with $f$ in $U$. Remark that $f'$ is the horizontal colimit of the diagram
\[\begin{tikzcd}[row sep=scriptsize]
	a & a & b \\
	c & a & b
	\arrow[from=2-2, to=2-1]
	\arrow[from=2-2, to=2-3]
	\arrow["id", from=1-3, to=2-3]
	\arrow["id", from=1-2, to=2-2]
	\arrow[from=1-2, to=1-3]
	\arrow["id"', from=1-2, to=1-1]
	\arrow["f"', from=1-1, to=2-1]
\end{tikzcd}\]
and then is in $U$.
\end{remark}

We say that an $\infty$-groupoid of morphisms $T$ is \wcnotion{closed under left cancellation}{closed under left or right cancellation} (resp. \textit{closed under right cancellation}), if for any pair of composable morphisms $f$ and $g$, if $gf$ and $f$ are in $T$, so is $g$ (resp. if $gf$ and $g$ are in $T$, so is $f$).

\begin{prop}
\label{prop:closed under colimit imply saturated}
Let $U$ be a cocomplete $\infty$-groupoid of arrows of $C$. The $\infty$-groupoid $U$ is closed under left cancellation.
\end{prop}
\begin{proof}
Suppose given $f:a\to b$, $g:b\to c$ such that $gf$ and $f$ are in $S$. As $g$ is the horizontal colimit of the following diagram
\[\begin{tikzcd}
	b & a & a \\
	b & b & c
	\arrow["g", from=1-3, to=2-3]
	\arrow["f", from=1-2, to=2-2]
	\arrow[from=2-2, to=2-1]
	\arrow[from=2-2, to=2-3]
	\arrow[from=1-2, to=1-1]
	\arrow[from=1-2, to=1-3]
	\arrow["{id_b}", from=1-1, to=2-1]
\end{tikzcd}\]
it is in $U$.
\end{proof}

\p We recall some standard results on factorization systems, which appear in many places in the literature, such as in section 5.5.5 of \cite{Lurie_Htt} for the $\iun$-case and  \cite{Joyal_factorisation} for the strict case.

Let $S$ be a $\V$-small $\infty$-groupoid of maps of $C$. We denote by $\Arr_S(C)$ the full sub $\iun$-category of $\Arr(C)$ whose objects correspond to arrows of $S$.

A \notion{weak factorization system in $(L,R)$} is the data of two full sub $\infty$-groupoids $L$ and $R$ of the $\infty$-groupoid of arrows of $C$, stable under composition and containing equivalences, and of section 
$\Arr_R(C)\to \Arr_L(C)\times_C \Arr_R(C)$ of the functor $ \Arr_L(C)\times_C \Arr_R(C)\to \Arr(C)$ sending two arrows onto their composite.
This is a \wcnotion{factorization system}{factorization system in $(L,R)$} if the functor $\Arr(C)\to \Arr_L(C)\times_C \Arr_R(C)$ is an equivalence.

Until the end of this section, we suppose given such factorization system in $(L,R)$.

\begin{definition}
Let $i$ and $p$ be two morphisms, and consider a square of shape:
\[\begin{tikzcd}
	a & b \\
	c & d
	\arrow["i"', from=1-1, to=2-1]
	\arrow["p", from=1-2, to=2-2]
	\arrow[from=2-1, to=2-2]
	\arrow[from=1-1, to=1-2]
\end{tikzcd}\]
A \wcnotion{lift}{lift in a square} in such square is the data of a morphism $h:c\to b$ and of two commutative triangles
\[\begin{tikzcd}
	a & b && b \\
	c && c & d
	\arrow[from=1-1, to=1-2]
	\arrow["i"', from=1-1, to=2-1]
	\arrow["h"', from=2-1, to=1-2]
	\arrow["h", from=2-3, to=1-4]
	\arrow[from=2-3, to=2-4]
	\arrow["p", from=1-4, to=2-4]
\end{tikzcd}\]

Equivalently, we can see a square of the previous shape as a morphism $s:1\to \Sq({i,p}):=\Hom(a,b)\times_{\Hom(a,d)}\Hom(c,d)$\sym{(sq@$\Sq(i,p)$} and a lift as the data of a morphism $h:1\to \Hom(c,d)$ and of a commutative triangle
\[\begin{tikzcd}
	& {\Hom(c,b)} \\
	1 & {\Sq(i,p)}
	\arrow["s"', from=2-1, to=2-2]
	\arrow["h", from=2-1, to=1-2]
	\arrow[from=1-2, to=2-2]
\end{tikzcd}\]

The \textit{$\infty$-groupoid of lift of $s$} is the fibers of $\Hom(c,b)\to \Sq(i,p)$ at $s$.
\end{definition}

\begin{definition}
Let $i$ and $p$ be two morphisms. The morphism \wcnotion{$i$ has the unique left lifting property against $p$}{unique left or right lifting property}, or equivalently, \textit{$p$ has the unique right lifting property against $i$}, if for any square $s\in \Sq(i,p)$, the $\infty$-groupoid of lift of $s$ is contractible. This is equivalent to asking for the morphism $\Hom(c,d)\to \Sq(i,p)$ to be an equivalence.
\end{definition}

\begin{lemma}
\label{lemma:when weak factorization system are factoryzation system}
Suppose that we have a weak factorization system in $(L',R')$ such that morphisms in $R'$ have the unique right lifting property against morphisms of $L'$. The weak factorization system is a factorization system.
\end{lemma}
\begin{proof}
Our goal is to demonstrate that the fibers of $\Arr_{L'}(C)\times_C\Arr_{R'}(C)\to \Arr(C)$ are contractible. Let $f$ be a morphism of $C$. As we have a weak factorization system, there exists an element in the fiber at $f$. Suppose given two elements in this fiber. This corresponds to a square
\[\begin{tikzcd}
	\cdot & \cdot \\
	\cdot & \cdot
	\arrow["i"', from=1-1, to=2-1]
	\arrow["p"', from=2-1, to=2-2]
	\arrow["{i'}", from=1-1, to=1-2]
	\arrow["{p'}", from=1-2, to=2-2]
\end{tikzcd}\]
Morphisms between these two factorizations correspond to lifts in the previous square, which are contractible by assumption, and the fiber is then contractible. 
\end{proof}
We recall that in this section, we suppose that we have a factorization system in $(L,R)$.
\begin{lemma}
\label{lemma:caracterisation of L and R with lifting property 1}
Morphisms in $L$ have the unique left lifting property with respect to morphisms in $R$. 
\end{lemma}
\begin{proof}
Let $i:a\to c$ be a morphim of $L$ and $p:b\to d$ a morphism of $R$.
The factorization functor induces an equivalence between squares $s\in \Sq(i,p)$ and diagrams of shape
\[\begin{tikzcd}[row sep=tiny]
	a && b \\
	& e \\
	c && d
	\arrow[from=1-1, to=3-1]
	\arrow[from=1-3, to=3-3]
	\arrow[from=1-1, to=2-2]
	\arrow[from=3-1, to=2-2]
	\arrow[from=2-2, to=3-3]
	\arrow[from=2-2, to=1-3]
\end{tikzcd}\]
where all the morphisms of the left triangle are in $L$ and the ones of the right triangle are in $R$.
Such diagrams are then in equivalence between composite $c\to e\to b$ where the first morphism is in $S$ and the second in $R$. Using once again the factorization functor, we can see that this data is exactly equivalent to a lift in the square $s$.
\end{proof}

We now show the converse of the previous lemma.

\begin{lemma}
\label{lemma:caracterisation of L and R with lifting property 2}
A morphism having the unique left lifting property against morphisms of $R$ is in $L$. Analogously, a morphism having the unique right lifting property against morphisms of $L$ is in $R$.
\end{lemma}
\begin{proof}
Let $f$ be a morphism having the unique left lifting property against morphisms in $R$. We factorize the morphism $f$ in $i\in L$ followed by $p\in R$ and we want to produce an equivalence $f\sim i$. The previous data induces by construction a square
\[\begin{tikzcd}
	a & b \\
	c & d
	\arrow["f"', from=1-1, to=2-1]
	\arrow["i", from=1-1, to=1-2]
	\arrow["p", from=1-2, to=2-2]
	\arrow[Rightarrow, no head, from=2-1, to=2-2]
	\arrow[dashed, from=2-1, to=1-2]
\end{tikzcd}\]
By hypothesis, this square admits a lift $l:c\to b$, that we factorize in a morphism $r'\in L$ followed by a morphism $p'\in R$. The commutativity of the lower triangle implies equivalences $pl'\sim pp'r'\sim id$, and by unicity, $r'\sim id$ and $pp'\sim id$. The lift $l$ is equivalent to $p'$ and is then in $R$. The commutativity of the upper triangle implies $lf\sim lpi \sim i$ and by unicity again, $p'p\sim id$. The morphism $p$ is then an isomorphism, this implies that $f\sim i$, and $f$ is then in $L$. We proceed similarly for the dual assertion.
\end{proof}

\begin{prop}
\label{prop:caracterisation of L and R with lifting property}
A morphism is in $L$ (resp. in $R$) if and only if it has the unique left lifting property against morphisms of $R$ (resp. the unique right lifting property against the morphisms of $R$).
\end{prop}
\begin{proof}
This is the content of lemma \ref{lemma:caracterisation of L and R with lifting property 1} and \ref{lemma:caracterisation of L and R with lifting property 2}.
\end{proof}

\begin{prop}
\label{prop:fonctorialite des relevement}
The forgetful functor from the $\iun$-category of squares with lifts, and whose left (resp. right) vertical morphism is in $L$ (resp. in $R$), to the $\iun$-category of squares whose left (resp. right) vertical morphism is in $L$ (resp. in $R$), is an equivalence.

Roughly speaking, the formation of the lift in squares whose left (resp. right) vertical morphism is in $L$ (resp. in $R$) is functorial.
\end{prop}
\begin{proof}
The $\iun$-category of squares with lifts, and whose left (resp. right) vertical morphism is in $L$ (resp. in $R$), is the $\iun$-category
$$ \mbox{$\Arr_L(C)$}\times_C  \mbox{$\Arr(C)$}\times_C  \mbox{$\Arr_R(C)$}$$
and the  $\iun$-category   whose left (resp. right) vertical morphism is in $L$ (resp. in $R$) of squares is the limit of the diagram
\[\begin{tikzcd}
	{\Arr_L(C)\times_C \Arr(C)} & {\Arr(C)} & {\Arr(C)\times_C\Arr_R(C)}
	\arrow["\triangledown"', from=1-3, to=1-2]
	\arrow["\triangledown", from=1-1, to=1-2]
\end{tikzcd}\]
The forgetful functor is induced by the commutative diagram
\[\begin{tikzcd}
	{ \Arr_L(C)\times_C  \Arr(C)\times_C  \Arr_R(C)} && {\Arr(C)\times_C\Arr_R(C)} \\
	{\Arr_L(C)\times_C \Arr(C)} && {\Arr(C)}
	\arrow["\triangledown", from=1-3, to=2-3]
	\arrow["{ \Arr_L(C)\times_C \triangledown}"', from=1-1, to=2-1]
	\arrow["{ \triangledown\times_C  \Arr_R(C)}", from=1-1, to=1-3]
	\arrow["\triangledown"', from=2-1, to=2-3]
\end{tikzcd}\]
and we then have to show that it is cartesian.

By definition of factorization system, the morphism 
$$\triangledown:  \mbox{$\Arr_L(C)$}\times_C  \mbox{$\Arr_R(C)$}\to \Arr(C)$$ is an equivalence. The previous square is then equivalent to the square
\[\begin{tikzcd}
	{ \Arr_L(C)\times_C \Arr(C)_L\times_C\Arr_R(C)\times_C  \Arr_R(C)} && {\Arr(C)_L\times_C\Arr_R(C)\times_C\Arr_R(C)} \\
	{\Arr_L(C)\times_C \Arr_L(C)\times_C  \Arr_R(C)} && {\Arr_L(C)\times_C  \Arr_R(C)}
	\arrow["{\Arr_L(C)\times_C\triangledown}", from=1-3, to=2-3]
	\arrow["{ \Arr_L(C)\times_C  \Arr_L(C)\times_C\triangledown}"', from=1-1, to=2-1]
	\arrow["{ \triangledown\times_C  \Arr_R(C)\times_C  \Arr_R(C)}", shift left=2, draw=none, from=1-1, to=1-3]
	\arrow["{ \triangledown\times_C  \Arr_R(C)}"', from=2-1, to=2-3]
	\arrow[from=1-1, to=1-3]
\end{tikzcd}\]
which is obviously cartesian.
\end{proof}

\begin{prop}
\label{prop:cloture of L recap}
The $\infty$-groupoid $L$ is stable under colimit, retract, composition, and left cancellation. The $\infty$-groupoid $R$ is stable under limit, retract, composition, and right cancellation. 
\end{prop}
\begin{proof}
Let $p:b\to d$ be a morphism of $R$ and $\{i_j:a_j\to c_j\}_{j:J}$ a family of morphisms of $L$ indexed by a functor $J\to \Arr_L(C)$, admitting a colimit $\bar{i}:\bar{a}\to \bar{c}$. Both functors $r\mapsto \Sq(r,p)$ and $c\mapsto\Hom(c,b)$ send colimits on limits. This implies that the morphism \[\Hom(\bar{c},b)\to\Sq(\bar{i},p)\] is the limit in $\Arr(\Sp)$ of the family of morphisms 
$$\Hom(c_j,b)\to\Sq({i_j,p}).$$
Each of these morphisms is an equivalence by assumption, so that implies that $\Hom(\bar{c},b)\to\Sq({\bar{i},p})$ is an equivalence. As this is true for any $p$ in $R$, proposition \ref{prop:caracterisation of L and R with lifting property} implies that $\bar{i}$ is in $L$.

Consider now a retract diagram:
\[\begin{tikzcd}
	a & {a'} & a \\
	c & {c'} & c
	\arrow["i", from=1-1, to=2-1]
	\arrow["i", from=1-3, to=2-3]
	\arrow[from=2-1, to=2-2]
	\arrow["{i'}", from=1-2, to=2-2]
	\arrow[from=1-1, to=1-2]
	\arrow[from=1-2, to=1-3]
	\arrow[from=2-2, to=2-3]
	\arrow["id"', curve={height=12pt}, from=2-1, to=2-3]
	\arrow["id", curve={height=-12pt}, from=1-1, to=1-3]
\end{tikzcd}\]
such that $i'$ is in $L$. For any morphism $p:b\to d$ of $R$, this induces a retract diagram
\[\begin{tikzcd}
	{\Hom(c,b)} & {\Hom(c',b)} & {\Hom(c,b)} \\
	{\Sq(i,p)} & {\Sq(i',p)} & {\Sq(i,p)}
	\arrow["{ }", from=1-1, to=2-1]
	\arrow["{ }", from=1-3, to=2-3]
	\arrow[from=2-1, to=2-2]
	\arrow["{ }", from=1-2, to=2-2]
	\arrow[from=1-1, to=1-2]
	\arrow[from=1-2, to=1-3]
	\arrow[from=2-2, to=2-3]
	\arrow["id"', curve={height=12pt}, from=2-1, to=2-3]
	\arrow["id", curve={height=-12pt}, from=1-1, to=1-3]
\end{tikzcd}\]
As equivalences are stable under retract, $\Hom(c,b)\to \Sq(i,p)$ is an equivalence, and as it is true for any $p$ in $R$, $i$ is in $L$.

For the cloture under left cancellation, this is proposition \ref{prop:closed under colimit imply saturated}.

We proceed similarly for the dual assertion.
\end{proof}

\p We fix an $\infty$-groupoid $S$ of arrows of $C$ with $\U$-small domain and codomain. We define \sym{(ls@$L_S$}$L_S := \widehat{S}$, i.e as the smallest full sub $\infty$-groupoid of arrows of $C$ stable under colimits, composition and including $S$, and \wcnotation{$R_S$}{(rs@$R_S$} as the full sub $\infty$-groupoid of arrows of $C$ having the unique right lifting property against morphisms of $S$. 
\begin{construction}[Small object Argument]
\label{cons:small object argument}
Let $f:x\to y$ be an arrow. We define by induction on $\U$ a sequence $\{x_\alpha\}_{\alpha<\U}$ sending $\emptyset$ on $x$.
For a limit ordinal $\alpha<\U$, we set $x_{\alpha}:= \colim_{\alpha'<\alpha}{x_{\alpha'}}$. For a successor ordinal, we define $x_{\alpha+1}$ as the pushout:
\[\begin{tikzcd}
	{\colim_{a\to b\in S}\big(\colim_{\Sq(a\to b,x_\alpha\to y)}a\underset{\colim_{\Hom(b,x_\alpha)}a}{\coprod} \colim_{\Hom(b,x_\alpha)}b\big)} & {x_\alpha} \\
	{\colim_{a\to b\in S}\big(\colim_{\Sq(a\to b,x_\alpha\to y)} b\big)} & {x_{\alpha+1}} \\
	&& y
	\arrow[""{name=0, anchor=center, inner sep=0}, from=1-1, to=1-2]
	\arrow[from=2-1, to=2-2]
	\arrow[from=1-2, to=2-2]
	\arrow[from=1-1, to=2-1]
	\arrow[curve={height=12pt}, from=2-1, to=3-3]
	\arrow[curve={height=-12pt}, from=1-2, to=3-3]
	\arrow[dashed, from=2-2, to=3-3]
	\arrow["\lrcorner"{anchor=center, pos=0.125, rotate=180}, draw=none, from=2-2, to=0]
\end{tikzcd}\]
Let $i:x\to\tilde{x}$ be the transfinite composition of this sequence. There is an induced morphism $p:\tilde{x}\to y$, and an equivalence $f\sim pi$. 
\end{construction}

\begin{prop}
\label{prop:factorization system from S}
The previous construction defines a factorization system between $L_S$ and $R_S$. 
\end{prop}
\begin{proof}
Let $f:x\to y $ be any morphism. The previous construction produces functorially morphisms $i:x\to \tilde{x}$ and $p:\tilde{x}\to y $ whose composite is $f$. The morphism $i$ is obviously in $L_S$. We then need to show that $p$ has the unique right lifting property against any morphism of $L_S$. Let $j:a\to b$ be any morphism in $L_S$, $n$ an integer and consider a commutative square
\[\begin{tikzcd}
	{a\coprod_{\colim_{\Sb_n} a} \colim_{\Sb_n} b} & {\tilde{x}} \\
	b & y
	\arrow["j"', from=1-1, to=2-1]
	\arrow["p", from=1-2, to=2-2]
	\arrow[from=2-1, to=2-2]
	\arrow[from=1-1, to=1-2]
\end{tikzcd}\]
By stability by $\omega$-small colimits, the object $a\coprod_{\colim_{\Sb_n} a} \colim_{\Sb_n} b$ is $\U$-small. There exists then $\alpha<\U$ such that the top morphism factors through $x_\alpha$, and by construction there exists a morphism $l:b\to x_{\alpha+1}$ and a comutative square
\[\begin{tikzcd}
	{a\coprod_{\colim_{\Sb_n} a} \colim_{\Sb_n} b} & {x_\alpha} & {\tilde{x}} \\
	& {x_{\alpha+1}} \\
	b && y
	\arrow["j"', from=1-1, to=3-1]
	\arrow["p", from=1-3, to=3-3]
	\arrow[from=3-1, to=3-3]
	\arrow[from=1-1, to=1-2]
	\arrow[from=1-2, to=1-3]
	\arrow["l", dotted, from=3-1, to=2-2]
	\arrow[from=1-2, to=2-2]
	\arrow[from=2-2, to=3-3]
	\arrow[from=2-2, to=1-3]
\end{tikzcd}\]
The induced diagonal is a lift in the first square. This implies that $\Hom(b,x)\to \Sq(j,p)$ has the right lifting property against $\Sb_n\to 1$. Eventually, this implies that $\Hom(b,x)\to \Sq(j,p)$ is an equivalence of $\infty$-groupoid, and $p$ then has the unique right lifting property against $i$. We then have a weak factorization system, which is a factorization system according to lemma \ref{lemma:when weak factorization system are factoryzation system}. 
\end{proof}

\subsection{Reflexive localization}

\p An object $x$ is \wcnotion{$S$-local}{local@$S$-local} if for any $i:a\to b\in S$, the induced functor $\Hom(i,x):\Hom(b,x)\to \Hom(a,x)$ is an equivalence. 
We define \wcnotation{$C_{S}$}{(cs@$C_S$} as the full sub $\iun$-category of $C$ composed of $S$-local objects.

\begin{lemma}
\label{lemma:object is $S$ local if fibrant}
An object is $S$-local if and only if $x\to 1$ is in $R_S$.
\end{lemma}
\begin{proof}
Let $i\in S$.
Remark that the functor $\Hom(b,x)\to \Sq(i,x\to 1)\sim\Hom(a,x)$ is $\Hom(i,f)$. The proposition \ref{prop:caracterisation of L and R with lifting property} then implies the desired result.
\end{proof}

\begin{theorem}
\label{theo:adjunction between presheaves and local presheaves}
The inclusion $\iota:C_S\to C$ is part of an adjunction
\[\begin{tikzcd}
	{\Fb_S:C} & {C_S:\iota}
	\arrow[""{name=0, anchor=center, inner sep=0}, shift left=2, from=1-1, to=1-2]
	\arrow[""{name=1, anchor=center, inner sep=0}, shift left=2, from=1-2, to=1-1]
	\arrow["\dashv"{anchor=center, rotate=-90}, draw=none, from=0, to=1]
\end{tikzcd}\]
Moreover, $\Fb_S:C\to C_S$ is the localization of $C$ by $\widehat{S}$.\sym{(f@$\Fb$}
\end{theorem}
\begin{proof}
For an object $x$, the small object argument provides a factorization of $x\to 1$ into a morphism $x\to \Fb_S x$ of $L_S$ followed by a morphism $\Fb_S x\to 1$ in $R_S$. According to lemma \ref{lemma:object is $S$ local if fibrant}, $\Fb_Sx$ is in $C_S$. As the factorization is functorial, this defines a functor $\Fb_S:C\to C_S$, and a natural transformation $\nu:id\to \Fb_S$ constant on $S$-local objects. As $\Fb_S\iota$ is equivalent to the identity, this induces the claimed adjunction. 

For the second proposition, let $F:C\to D$ be a functor sending morphisms of $L_S$ on equivalences. We define $\Db(F):= F\circ \iota$, and we have a diagram
\[\begin{tikzcd}
	C && D \\
	& {C_S}
	\arrow["{\Fb_S}"', from=1-1, to=2-2]
	\arrow[""{name=0, anchor=center, inner sep=0}, "F", from=1-1, to=1-3]
	\arrow["{\Db(F)}"', from=2-2, to=1-3]
	\arrow[shorten <=5pt, shorten >=5pt, Rightarrow, from=0, to=2-2]
\end{tikzcd}\]
that commutes up to the natural transformation $F\circ_0 \nu:F\to D(F)\circ \Fb_S$. However, the natural transformation $\nu$ is pointwise in $L_S$, which implies that $F\circ \nu$ is pointwise an equivalence, and the previous diagram then commutes. Now, let $G:C_S\to D$ be any other functor such that $G\circ\Fb_S \sim F$. By precomposing with iota, this implies that $G\sim F\circ \iota$.
\end{proof}

\begin{cor}
 \label{cor:derived colimit preserving functor}
The $\iun$-category $C_S$ is cocomplete. Moreover, if $F:C\to D$ is a colimit preserving functor sending $S$ onto equivalences, the induced functor $\Db F:C_S\to D$ preserves colimits.
\end{cor}
\begin{proof}
The first assertion is a direct consequence of the adjunction given in theorem \ref{theo:adjunction between presheaves and local presheaves}.

This adjunction also implies that the colimit of a functor $G:A\to C_S$ is given by $\Fb_S(\colim_{a:A} \iota G(a))$.
As the canonical morphism $\colim_{a:A} \iota G(a)\to \Fb_S(\colim_{a:A} \iota G(a))$ is  by construction in $\widehat{S}$ this proves the second assertion.
\end{proof}

\p Suppose given an adjunction between two $\iun$-categories
\[\begin{tikzcd}
	{F:C} & {D:G}
	\arrow[""{name=0, anchor=center, inner sep=0}, shift left=2, from=1-1, to=1-2]
	\arrow[""{name=1, anchor=center, inner sep=0}, shift left=2, from=1-2, to=1-1]
	\arrow["\dashv"{anchor=center, rotate=-90}, draw=none, from=0, to=1]
\end{tikzcd}\]
with unit $\nu$ and counit $\epsilon$,
as well as an $\infty$-groupoid of morphisms $S$ of $C$ and $T$ of $D$ such that $F(S)\subset \widehat{T}$. 
By adjunction property, it implies that for any $T$-local object $d\in D$, $G(d)$ is $S$-local.
The previous adjunction induces a derived adjunction\sym{(lf@$\Lb F$} \sym{(rg@$\Rb G$}
\[\begin{tikzcd}
	{\Lb F:C_S} & {D_T:\Rb G}
	\arrow[""{name=0, anchor=center, inner sep=0}, shift left=2, from=1-1, to=1-2]
	\arrow[""{name=1, anchor=center, inner sep=0}, shift left=2, from=1-2, to=1-1]
	\arrow["\dashv"{anchor=center, rotate=-90}, draw=none, from=0, to=1]
\end{tikzcd}\]
where $\Lb F$ is defined by the formula $c\mapsto \Fb_T F(c)$ and $\Rb G$ is the restriction of $G$ to $D_T$. The unit is given by $\nu\circ \Fb_S$ and the counit by the restriction of $\epsilon$ to $D_T$.

\begin{example}
\label{exe:exe localization}
\index[notation]{(f0@$f_{\mbox{$\exclam$}}:C_{/c}\to C_{d/}$}
\index[notation]{(f1@$f^*:C_{/d}\to C_{c/}$}
\index[notation]{(f2@$f_*:C_{/c}\to C_{d/}$}
\index[notation]{(lf0@$\Lb f_{\mbox{$\exclam$}}:(C_{/c})_{S_{c/}}\to (C_{d/})_{S_{d/}}$}
\index[notation]{(rf1@$\Rb f^*:(C_{/d})_{S_{d/}}\to (C_{c/})_{S_{c/}}$}
\index[notation]{(lf2@$\Lb f^*:(C_{/d})_{S_{d/}}\to (C_{c/})_{S_{c/}}$}
\index[notation]{(rf3@$\Rb f_*:(C_{/c})_{S_{c/}}\to (C_{d/})_{S_{d/}}$}
Let $C$ be a presentable $\iun$-category, $S$ a full sub $\infty$-groupoid of morphisms of $\iPsh{A}$ with $\U$-small codomain and domain. 
Eventually, we set \wcnotation{$S_{/c}$}{(sc@$S_{/c}$} as the $\infty$-groupoid of morphisms of shape
\[\begin{tikzcd}
	& b \\
	a & c
	\arrow[from=2-1, to=2-2]
	\arrow[from=1-2, to=2-2]
	\arrow["s", from=2-1, to=1-2]
\end{tikzcd}\]
where $s:S$.

 A morphism $f:c\to d$ induces an adjunction
\[\begin{tikzcd}
	{f_!:C_{/c}} & {C_{/d}:f^*}
	\arrow[""{name=0, anchor=center, inner sep=0}, shift left=2, from=1-1, to=1-2]
	\arrow[""{name=1, anchor=center, inner sep=0}, shift left=2, from=1-2, to=1-1]
	\arrow["\dashv"{anchor=center, rotate=-90}, draw=none, from=0, to=1]
\end{tikzcd}\] 
where the left adjoint is the composition and the right adjoint is the pullback. By construction, $f_!(S_{/c})\subset S_{/d}$. The previous adjunction can then be derived, and induced an adjunction:
\[\begin{tikzcd}
	{\Lb f_!:(C_{/c})_{S_{/c}}} & {(C_{/d})_{S_{/d}}:\Rb f^*}
	\arrow[""{name=0, anchor=center, inner sep=0}, shift left=2, from=1-1, to=1-2]
	\arrow[""{name=1, anchor=center, inner sep=0}, shift left=2, from=1-2, to=1-1]
	\arrow["\dashv"{anchor=center, rotate=-90}, draw=none, from=0, to=1]
\end{tikzcd}\]
where the right adjoint is just the restriction of $f^*$ to $S_{/d}$-local objects.

If the functor $f^*:C_{/d}\to C_{/c}$ preserves colimits and $f^*(S_{/c})\subset S_{/d}$, the adjunction
\[\begin{tikzcd}
	{f^*:C_{/d}} & {C_{/c}:f_*}
	\arrow[""{name=0, anchor=center, inner sep=0}, shift left=2, from=1-1, to=1-2]
	\arrow[""{name=1, anchor=center, inner sep=0}, shift left=2, from=1-2, to=1-1]
	\arrow["\dashv"{anchor=center, rotate=-90}, draw=none, from=0, to=1]
\end{tikzcd}\]
induces an adjunction 
\[\begin{tikzcd}
	{\Lb f^*:(C_{/d})_{S_{/d}}} & {(C_{/c})_{S_{/c}}:\Rb f_*}
	\arrow[""{name=0, anchor=center, inner sep=0}, shift left=2, from=1-1, to=1-2]
	\arrow[""{name=1, anchor=center, inner sep=0}, shift left=2, from=1-2, to=1-1]
	\arrow["\dashv"{anchor=center, rotate=-90}, draw=none, from=0, to=1]
\end{tikzcd}\]

\end{example}

\section{Basic constructions}
\label{chapter:Basica construciton}
\subsection{$\io$-Categories}
\label{section:iocategories}
The definitions of section \ref{subsection:the categoru theta} will be used freely here.
\p
We denote by 
$$[\uvar,\uvar]: \iPsh{\Theta}\times \iPsh{\Delta}\to \iPsh{\Delta[\Theta]}$$
the extension by colimit of the functor $\Theta\times \Delta\to \iPsh{\Delta[\Theta]}$ sending $(a,n)$ onto $[a,n]$.
For an integer $n$, we denote
$$[\uvar,n]:\iPsh{\Theta}^n\to \iPsh{\Theta}$$ 
the extension by colimit of the functor 
$\Theta^n\to\iPsh{\Theta}$ sending $\textbf{a}:=\{a_1,...,a_n\}$ onto $[\textbf{a},n]$.

\p  We have an adjunction 
\begin{equation}
\label{eq:underived adjunction part}
\begin{tikzcd}
	{ i_!:\iPsh{\Delta[\Theta]}} & {\iPsh{\Theta}:i^*}
	\arrow[shift left=2, from=1-1, to=1-2]
	\arrow[shift left=2, from=1-2, to=1-1]
\end{tikzcd}
\end{equation}
where the left adjoint is the left Kan extension of the functor $\Delta[\Theta]\xrightarrow{i} \Theta\to \iPsh{\Theta}$. The sets of morphisms $\W$ and $\M$ are respectively defined in paragraphs \ref{para:definition of W} and \ref{para:defi of delta theta}.
There is an obvious inclusion $i_!(\M)\subset \W$. The previous adjunction then induced a derived adjunction
\begin{equation}
\label{eq:derived adjunction}
\begin{tikzcd}
	{\Lb i_!:\Psh{\Delta[\Theta]}_{\M}} & {\Psh{\Theta}_{\W}:\Rb i^*}
	\arrow[shift left=2, from=1-1, to=1-2]
	\arrow[shift left=2, from=1-2, to=1-1]
\end{tikzcd}
\end{equation}

\begin{prop}
\label{prop:infini changing theta}
The unit and counit of the adjunction \eqref{eq:underived adjunction part} are respectively in $\widehat{\M}$ and $\widehat{\W}$. As a consequence, the adjunction \eqref{eq:derived adjunction} is an adjoint equivalence.
\end{prop}
\begin{proof}
We denote by $\iota:\Psh{\Theta}\to \iPsh{\Theta}$ and $\iota:\Psh{\Delta[\Theta]}\to \iPsh{\Delta[\Theta]}$ the two canonical inclusions. By the definition of the smallest precocomplete class (paragraph \ref{para:precomplet}) and according to lemma \ref{lemma:colimit computed in set presheaves}, we have inclusions $\iota(\overline{\W})\subset \widehat{\W}$ and $\iota(\overline{\M})\subset \widehat{\M}$. The result then directly follows from theorem \ref{theo:unit and counit are in W}.  
\end{proof}

\p A \wcnotion{$\io$-category}{category4@$\io$-category} is a $\W$-local $\infty$-presheaf $C\in \iPsh{\Theta}$. We then define \index[notation]{((a60@$\ocat$}
$$\ocat := \iPsh{\Theta}_{\W}.$$
Proposition \ref{prop:infini changing theta} implies that $\ocat$ identifies itself with the full sub $\iun$-category of $\iPsh{\Delta[\Theta]}$ of $\M$-local objects:
$$\ocat \sim \iPsh{\Delta[\Theta]}_{\M}.$$
We recall that the sets of morphisms $\W$ and $\M$ are respectively defined in paragraphs \ref{para:definition of W} and \ref{para:defi of delta theta}.

\p
We denote by $\pi_0:\iPsh{\Theta}\to \Psh{\Theta}$ the functor sending an $\infty$-presheaf $X$ onto the presheaf
$$\pi_0X:a\mapsto\pi_0(X_a)$$
This functor admits a fully faithful right adjoint: $\N:\Psh{\Theta}\to \iPsh{\Theta}$. 
As $\pi_0$ preserves $\W$, it induces an adjoint pair: \sym{(pi@$\pi_0:\ocat\to \zocat$}\sym{n@$\N:\zocat\to \ocat$}
\[\begin{tikzcd}
	{\pi_0:\ocat} & {\zocat:\N}
	\arrow[""{name=0, anchor=center, inner sep=0}, shift left=2, from=1-1, to=1-2]
	\arrow[""{name=1, anchor=center, inner sep=0}, shift left=2, from=1-2, to=1-1]
	\arrow["\dashv"{anchor=center, rotate=-90}, draw=none, from=0, to=1]
\end{tikzcd}\]
where the right adjoint $\N$ is fully faithful.
Every $\zo$-category can then be seen as an $\io$-category and we will call \wcnotion{strict}{strict $\io$-category} the $\io$-categories lying in the image of this functor.

The inclusion $\Delta\to \Theta$ induces by extention by colimit a functor $\iPsh{\Delta}\to \iPsh{\Theta}$. Passing to the localization, this induces a fully faithful inclusion $\ncat{1}\to \ocat$.

The inclusion $\{[0]\}\to \Theta$ induces by extention by colimit a functor $\igrd \to \iPsh{\Theta}$. Passing to the localization, this induces a fully faithful inclusion $\igrd \to \ocat$. The $\io$-categories lying in the image of this functors will be also called \textit{$\infty$-groupoids}.

\p A \wcsnotion{$n$-cell}{cell@$n$-cell}{for $\io$-categories} of an $\io$-category is a morphism $\Db_n\to C$.
If $C$ is an $\io$-category, we denote by $C_n$ the value of $C$ on $\Db_n$. 

\begin{prop}
\label{prop:equivalences detected on globes}
Let $C,D$ be two $\io$-categories, and $f:C\to D$ any map. The morphism $f$ is an equivalence if and only if for any $n$, the induced morphism $f_n: C_n\to D_n$ is an equivalence. 
\end{prop}
\begin{proof}
This is a necessary condition. For the converse, let $f$ be a morphism fulfilling this condition. To show that $f$ is an equivalence, we have to show that for any globular sum $a$, $f_a: C_a\to D_a$ is an equivalence. This is true as 
$$f_a:C_a\to D_a~\sim~ \lim_{n\in\Sp_a}{f_n:C_n\to D_n}.$$	
\end{proof}

\begin{lemma}
\label{lemma:equivalence if unique right lifting property against globes.}
A functor is an equivalence if it has the unique right lifting property against $\emptyset\to \Db_n$ for any $n\geq 0$. 
\end{lemma}
\begin{proof}
This is a necessary condition. For the converse, let $f:C\to D$ be a morphism fulfilling this condition. By definition of left unique lifting property, it implies that the induced morphism
$f_n:C_n\to D_n$ is an equivalence for any $n\geq 0$. Using proposition \ref{prop:equivalences detected on globes}, $f$ is an equivalence.
 \end{proof}

\p Let $\iPsh{\Theta}_{\bullet,\bullet}$ be the $(\infty,1)$-category of $\infty$-presheaves on $\Theta$ with two distinguished points, i.e. of triples $(C,a,b)$ where $a$ and $b$ are elements of $C_0$.
The functor $[\uvar,1]:\Theta\to \iPsh{\Theta}_{\bullet,\bullet}$ that sends $a$ onto $([a,1],\{0\},\{1\})$ induces by extension by colimit an adjunction
\begin{equation}
\label{eq:suspesnion betweenpresheaves}
\begin{tikzcd}
	{[\uvar,1]:\iPsh{\Theta}} & {\iPsh{\Theta}_{\bullet,\bullet}:\hom_{\uvar}(\uvar,\uvar)}
	\arrow[""{name=0, anchor=center, inner sep=0}, shift left=2, from=1-1, to=1-2]
	\arrow[""{name=1, anchor=center, inner sep=0}, shift left=2, from=1-2, to=1-1]
	\arrow["\dashv"{anchor=center, rotate=-90}, draw=none, from=0, to=1]
\end{tikzcd}
\end{equation}
As the left adjoint preserves representables, the right adjoint commutes with colimit. It is then easy to check on representables that the unit of this adjunction is an equivalence. As a consequence, the left adjoint is fully faithful.

\begin{lemma}
\label{lemma:hom of the suspension prequel}
Let $C$ be an $\infty$-presheaves on $\Theta$. The canonical morphisms
$$C\to \hom_{[C,1]}(0,1)~~\hom_{[C,1]}(0,0)\to 1~~~ \hom_{[C,1]}(1,1) \sim 1~~~~\emptyset\to \hom_{[C,1]}(1,0)$$
are equivalences.
\end{lemma}
\begin{proof}
As both $\hom$ and $[\uvar,1]$ preserve colimits, it is sufficient to check this property on representables, where it is an easy computation.
\end{proof}

\begin{prop}
\label{prop:supspension preserves cat}
The functor $[\uvar,1]:\iPsh{\Theta}\to \iPsh{\Theta}$ preserves $\io$-categories.
\end{prop}
\begin{proof}
By construction, for any pair of integers $k<n$, and any pair of globular sums $([\textbf{a},n],b)$, we have cartesian squares
\[\begin{tikzcd}
	1 & {\Hom_{\Theta}([\textbf{a},n],[b,1])} & {\Hom_{\Theta}(a_k,b)} & {\Hom_{\Theta}([\textbf{a},n],[b,1])} \\
	{\{\epsilon\}} & {\Hom_{\Delta}([n],[1])} & {\{\alpha_k\}} & {\Hom_{\Delta}([n],[1])}
	\arrow[""{name=0, anchor=center, inner sep=0}, from=2-1, to=2-2]
	\arrow[from=1-1, to=1-2]
	\arrow[from=1-2, to=2-2]
	\arrow[from=1-1, to=2-1]
	\arrow[from=1-4, to=2-4]
	\arrow[""{name=1, anchor=center, inner sep=0}, from=2-3, to=2-4]
	\arrow[from=1-3, to=2-3]
	\arrow[from=1-3, to=1-4]
	\arrow["\lrcorner"{anchor=center, pos=0.125}, draw=none, from=1-1, to=0]
	\arrow["\lrcorner"{anchor=center, pos=0.125}, draw=none, from=1-3, to=1]
\end{tikzcd}\]
where $\epsilon$ denote any constant functor with value $0$ or $1$, and $\alpha_k$ the morphism that sends $k$ on $0$ and $k+1$ on $1$.
Let $C$ be an $\io$-category.
As the $\iun$-category $\igrd$ is locally cartesian closed,  we have  cartesian squares
\begin{equation}
\label{eq:prop:supspension preserves cat}
\begin{tikzcd}[column sep =0.3cm]
	1 & {\Hom_{\iPsh{\Theta}}([\textbf{a},n],[C,1])} & {\Hom_{\iPsh{\Theta}}(a_k,C)} & {\Hom_{\iPsh{\Theta}}([\textbf{a},n],[C,1])} \\
	{\{\epsilon\}} & {\Hom_{\Delta}([n],[1])} & {\{\alpha_k\}} & {\Hom_{\Delta}([n],[1])}
	\arrow[""{name=0, anchor=center, inner sep=0}, from=2-1, to=2-2]
	\arrow[from=1-1, to=1-2]
	\arrow[from=1-2, to=2-2]
	\arrow[from=1-1, to=2-1]
	\arrow[from=1-4, to=2-4]
	\arrow[""{name=1, anchor=center, inner sep=0}, from=2-3, to=2-4]
	\arrow[from=1-3, to=2-3]
	\arrow[from=1-3, to=1-4]
	\arrow["\lrcorner"{anchor=center, pos=0.125}, draw=none, from=1-1, to=0]
	\arrow["\lrcorner"{anchor=center, pos=0.125}, draw=none, from=1-3, to=1]
\end{tikzcd}
\end{equation}
which induces cartesian squares
\[\begin{tikzcd}[column sep =0.3cm]
	1 & {\Hom_{\iPsh{\Theta}}(\Sp_{[\textbf{a},n]},[C,1])} & {\Hom_{\iPsh{\Theta}}(\Sp_{a_k},C)} & {\Hom_{\iPsh{\Theta}}(\Sp_{[\textbf{a},n]},[C,1])} \\
	{\{\epsilon\}} & {\Hom_{\Delta}([n],[1])} & {\{\alpha_k\}} & {\Hom_{\Delta}([n],[1])}
	\arrow[""{name=0, anchor=center, inner sep=0}, from=2-1, to=2-2]
	\arrow[from=1-1, to=1-2]
	\arrow[from=1-2, to=2-2]
	\arrow[from=1-1, to=2-1]
	\arrow[from=1-4, to=2-4]
	\arrow[""{name=1, anchor=center, inner sep=0}, from=2-3, to=2-4]
	\arrow[from=1-3, to=2-3]
	\arrow[from=1-3, to=1-4]
	\arrow["\lrcorner"{anchor=center, pos=0.125}, draw=none, from=1-1, to=0]
	\arrow["\lrcorner"{anchor=center, pos=0.125}, draw=none, from=1-3, to=1]
\end{tikzcd}\]
This directly implies that $[C,1]$ is $\Wseg$-local.

Furthermore, for any integer $n>0$, the cartesian squares \eqref{eq:prop:supspension preserves cat} induces cartesian squares
\[\begin{tikzcd}[column sep =0.2cm]
	1 & {\Hom_{\iPsh{\Theta}}(\Sigma^nE^{eq},[C,1])} & {\Hom_{\iPsh{\Theta}}(\Sigma^{n-1}E^{eq},C)} & {\Hom_{\iPsh{\Theta}}(\Sigma^nE^{eq},[C,1])} \\
	{\{\epsilon\}} & {\Hom_{\Delta}([1],[1])} & {\{\alpha_k\}} & {\Hom_{\Delta}([1],[1])}
	\arrow[""{name=0, anchor=center, inner sep=0}, from=2-1, to=2-2]
	\arrow[from=1-1, to=1-2]
	\arrow[from=1-2, to=2-2]
	\arrow[from=1-1, to=2-1]
	\arrow[from=1-4, to=2-4]
	\arrow[""{name=1, anchor=center, inner sep=0}, from=2-3, to=2-4]
	\arrow[from=1-3, to=2-3]
	\arrow[from=1-3, to=1-4]
	\arrow["\lrcorner"{anchor=center, pos=0.125}, draw=none, from=1-1, to=0]
	\arrow["\lrcorner"{anchor=center, pos=0.125}, draw=none, from=1-3, to=1]
\end{tikzcd}\]
which implies that $[C,1]$  is local with respect to $\Sigma^nE^{eq}\to \Sigma^{n}1$.

Eventually, suppose given a diagram of shape
\begin{equation}
\label{eq:proof of sigma preserves omega cat 3}
\begin{tikzcd}
	{E^{eq}} & {[C,1]} \\
	{1}
	\arrow[from=1-1, to=2-1]
	\arrow[from=1-1, to=1-2]
\end{tikzcd}
\end{equation}
The canonical morphism $E^{eq}\to [C,1]\xrightarrow{\pi} [1]$ then factors through $0$ or $1$. As the two fibers of $\pi$ are trivial,  the diagram \eqref{eq:proof of sigma preserves omega cat 3} admits a unique lift, which concludes the proof.
\end{proof}

\p
\label{para:wiskering} As $[\uvar,1]$ sends $\W$ to a subset of $\M$, the functor $\hom_{\uvar,\uvar}(\uvar)$ preserves $\io$-categories. Combined with the last proposition, this implies that 
the adjunction \eqref{eq:suspesnion betweenpresheaves} restricts to an adjunction:
\begin{equation}
\label{eq:suspesnion between category}
\begin{tikzcd}
	{[\uvar,1]:\ocat} & {\ocat_{\bullet,\bullet}:\hom_{\uvar}(\uvar,\uvar)}
	\arrow[""{name=0, anchor=center, inner sep=0}, shift left=2, from=1-1, to=1-2]
	\arrow[""{name=1, anchor=center, inner sep=0}, shift left=2, from=1-2, to=1-1]
	\arrow["\dashv"{anchor=center, rotate=-90}, draw=none, from=0, to=1]
\end{tikzcd}
\end{equation}
The left adjoint is the \wcsnotionsym{suspension functor}{((d60@$[\uvar,1]$}{suspension}{for $\io$-categories}\ssym{(hom@$\hom_{\uvar}(\uvar,\uvar)$}{for $\io$-categories}.
\begin{prop}
\label{prop:hom of the suspension}
Let $C$ be an $\io$-categories. We have natural equivalences
$$\hom_{[C,1]}(0,1)\sim C~~~~\hom_{[C,1]}(0,0)\sim \hom_{[C,1]}(1,1) \sim 1~~~~\hom_{[C,1]}(1,0)\sim \emptyset.$$
\end{prop}
\begin{proof}
This is a direct consequence of lemma \ref{lemma:hom of the suspension prequel}.
\end{proof}

\p
Suppose given an $\io$-category $C$ and a $1$-cells $f:x'\to x$. 
As $C$ is an $\io$-category, for any globular sum $a$, the morphism
$$\Hom([1]\vee[a,1],C)\to \Hom([1], C)\times_{\Hom([0],C)}\Hom([a,1],C)$$
is an equivalence. 
This induces a morphism
$$\Hom(a,\hom_C(x,y))\to \Hom([1]\vee[a,1],(C,x',y))\to \Hom(a,\hom_{C}(x',y))$$	
where the two distinguished points of $[1]\vee[a,1]$ are the extremal ones, and where the left-hand morphism is the restriction of the inverse of the previous morphism. By the Yoneda lemma, this corresponds to a morphism
$$f_!:\hom_C(x',y)\to \hom_C(x,y).$$
Conversely, a $1$-cell $g:y\to y'$ induces a morphism
$$g_!:\hom_C(x,y)\to \hom_C(x,y').$$

\p \label{para: spetial colimits}
We denote by $\iota$ the inclusion of $\ocat$ into $\iPsh{\Theta}$.
A functor $F:I\to \ocat$ has a \snotion{special colimit}{for $\io$-categories} if the canonical morphism 
\begin{equation}
\label{eq:special colimit}
\colim_{i:I}\iota F(i)\to \iota(\colim_{i:I}F(i))
\end{equation}
is an equivalence of presheaves. 

Similarly, we say that a functor $\psi: I\to \Arr(\ocat)$ has a \textit{special colimit} if the canonical morphism 
$$\colim_{i:I}\iota \psi(i)\to \iota(\colim_{i:I}\psi(i))$$
is an equivalence in the arrow $\iun$-category of $\iPsh{\Theta}$.

\begin{example}
Let $C$ be an $\io$-category. The canonical diagram $\Theta_{/C}\to \ocat$ has a special colimit, given by $C$.
\end{example}
\begin{prop}
\label{prop:special colimit}
Let $F,G:I\to \ocat$ be two functors, and $\psi:F\to G$ a natural transformation. If $\psi$ is cartesian, and $G$ has a special colimit, then $\psi$ and $F$ have special colimits. 
\end{prop}
\begin{proof}
We have to show that $F$ has a special colimit, it will directly imply that $\psi$ also has one. The morphism \eqref{eq:special colimit} is always in $\widehat{\W}$. To conclude, one then has to show that $\colim_{i:I}\iota \psi(i)$ is $\W$-local. To this extend, it is enough to demonstrate that the canonical morphism 
$$\colim_{i:I}\iota \psi(i): \colim_{i:I}\iota F(i)\to \colim_{i:I}\iota G(i)$$ 
has the unique right lifting property against $\W$. We then consider a square
\begin{equation}
\label{eq:proof special colimit}
\begin{tikzcd}
	a & {\colim_{i:I}\iota F(i)} \\
	b & { \colim_{i:I}\iota G(i)}
	\arrow["{\colim_{i:I}\iota \psi(i)}", from=1-2, to=2-2]
	\arrow[from=2-1, to=2-2]
	\arrow[from=1-1, to=2-1]
	\arrow[from=1-1, to=1-2]
\end{tikzcd}
\end{equation}
where $f\in W$. As the domain of $f$ is representable, there always exists $j:I$, such that the bottom horizontal morphism factors through $G(j)$. As $\psi$ is cartesian, the square \eqref{eq:proof special colimit} factors in two squares, where the right one is cartesian. 
\[\begin{tikzcd}
	a & {F(i)} & {\colim_{i:I}\iota F(i)} \\
	b & {G(i)} & { \colim_{i:I}\iota G(i)}
	\arrow["{\colim_{i:I}\iota \psi(i)}", from=1-3, to=2-3]
	\arrow[from=1-1, to=2-1]
	\arrow["{\psi(i)}", from=1-2, to=2-2]
	\arrow[from=1-2, to=1-3]
	\arrow[from=2-2, to=2-3]
	\arrow["\lrcorner"{anchor=center, pos=0.125}, draw=none, from=1-2, to=2-3]
	\arrow[from=1-1, to=1-2]
	\arrow[from=2-1, to=2-2]
\end{tikzcd}\]
Lifts in the square \eqref{eq:proof special colimit} are then equivalent to lifts in the left square, which exist and are unique as $F(i)\to G(i)$ has the unique right lifting property against $\W$.
\end{proof}

\begin{prop}
\label{prop:example of a special colimit}
For any integer $n$, and globular sums $a$ and $b$, the equalizer diagram 
\[\begin{tikzcd}
	{\coprod_{k+l=n-1}[a,k]\vee[a\times b,1]\vee[a,l]} & {\coprod_{k+l=n}[a,k]\vee[ b,1]\vee[a,l]}
	\arrow[shift left=2, from=1-1, to=1-2]
	\arrow[shift right=2, from=1-1, to=1-2]
\end{tikzcd}\]
where the top diagram is induced by $[a\times b,1]\to [a,1]\vee[b,1]$ and to bottom one by $[a\times b,1]\to [b,1]\vee[a,1]$,
has a special colimit, which is $[a,n]\times [b,1]$.
\end{prop}
\begin{proof}
The lemma \ref{lemma:colimit computed in set presheaves} implies that the colimit of the previous diagram, computed in $\iPsh{\Theta}$ is strict. It is then enough to show that this colimit, computed in $\Psh{\Theta}$, is equivalent to $[a,n]\times [b,1]$. As this last object is $\W$-local, this will concludes the proof. The remaining combinatorial exercise is left to the reader. 
\end{proof}

\begin{prop}
\label{prop:example of a special colimit 2}
Any sequence of $\io$-categories has a special colimit. 
\end{prop}
\begin{proof}
Suppose given such sequence.
If the sequence is finite, this is obviously true. Suppose now that the sequence is non finite. As codomains and domains of morphism of $\W$ are $\omega$-small, the colimit of the sequence, computed in $\iPsh{\Theta}$ is $\W$-local, which concludes the proof.
\end{proof}

\begin{lemma}
\label{lemma:[ ,1] preserves spcial limits}
The functor $[\uvar,1]:\ocat\to \ocat_{\bullet,\bullet}$
preserves special colimits.
\end{lemma}
\begin{proof}
This is a direct consequence of proposition \ref{prop:supspension preserves cat}.
\end{proof}

\begin{lemma}
\label{lemma:[ ,1] vee [ ,1]preserves spcial limits}
We denote by $$
\begin{array}{cl}
 & [\uvar,1]\vee[1]:\ocat\to \ocat_{[0]\amalg[1]/}\\
 \mbox{(resp.} & [1]\vee[\uvar,1]:\ocat\to \ocat_{[1]\amalg[0]/})
\end{array}$$ the colimit preserving functor that sends an element $a$ of $\Theta$ onto the globular sum $[a,1]\vee[1]$ (resp. $[1]\vee[a,1]$).

The functors $[\uvar,1]\vee[\uvar,1]$  and $[1]\vee[\uvar,1]$ preserve special colimits.
\end{lemma}
\begin{proof}
To prove this, we establish a result analogous to the one given in the proposition \ref{prop:supspension preserves cat}. We omit its proof because it is long but essentially identical.
\end{proof}

\begin{prop}
\label{prop:example of a special colimit3}
Suppose given two cartesian squares
\[\begin{tikzcd}
	{ B} & C & D \\
	{\{0\}} & {[1]} & {\{1\}}
	\arrow[from=1-2, to=2-2]
	\arrow[from=1-1, to=1-2]
	\arrow[from=1-3, to=2-3]
	\arrow[from=1-1, to=2-1]
	\arrow[from=1-3, to=1-2]
	\arrow[from=2-1, to=2-2]
	\arrow[from=2-3, to=2-2]
	\arrow["\lrcorner"{anchor=center, pos=0.125}, draw=none, from=1-1, to=2-2]
	\arrow["\lrcorner"{anchor=center, pos=0.125, rotate=-90}, draw=none, from=1-3, to=2-2]
\end{tikzcd}\]
The diagram 
\[\begin{tikzcd}
	{[1]\vee[B,1]} & {[ B,1]} & {[C,1]} & {[D,1]} & {[D,1]\vee[1]}
	\arrow["\triangledown", from=1-4, to=1-5]
	\arrow["\triangledown"', from=1-2, to=1-1]
	\arrow[from=1-2, to=1-3]
	\arrow[from=1-4, to=1-3]
\end{tikzcd}\]
has a special colimit.
\end{prop}
\begin{proof}
Remark firsts that the colimit, computed in $\iPsh{\Theta}$, of the diagram
\[\begin{tikzcd}
	{[1]\vee[1]} & {[1]} & {[[1],1]} & {[1]} & {[1]\vee[1]}
	\arrow["\triangledown"', from=1-2, to=1-1]
	\arrow["\triangledown", from=1-4, to=1-5]
	\arrow["{[\{0\},1]}", from=1-2, to=1-3]
	\arrow["{[\{1\},1]}"', from=1-4, to=1-3]
\end{tikzcd}\]
is strict. We leave it to the reader to check that the previous diagram has a special colimit. 

Remark now that $\Theta$ is stable by pullback and $[\uvar,1]$ preserves cartesian squares in $\Theta$. 
The lemma \ref{lemma:[ ,1] preserves spcial limits} states that $[\uvar,1]$ preserves  special colimit, and as $\iPsh{\Theta}$ is locally cartesian closed, pullbacks also preserve them.
As every $\io$-category is a special colimit of representables, this implies that the squares
\[\begin{tikzcd}
	{[B,1]} & {[C,1]} & {[D,1]} \\
	{[1]} & {[[1],1]} & {[1]}
	\arrow[from=1-1, to=1-2]
	\arrow[from=1-3, to=1-2]
	\arrow["{[\{0\},1]}"', from=2-1, to=2-2]
	\arrow["{[\{1\},1]}", from=2-3, to=2-2]
	\arrow[from=1-1, to=2-1]
	\arrow[from=1-2, to=2-2]
	\arrow["\lrcorner"{anchor=center, pos=0.125}, draw=none, from=1-1, to=2-2]
	\arrow[from=1-3, to=2-3]
	\arrow["\lrcorner"{anchor=center, pos=0.125, rotate=-90}, draw=none, from=1-3, to=2-2]
\end{tikzcd}\]
are cartesian.
Furthermore, for any globular sum $b$, we have cartesian squares
\[\begin{tikzcd}
	{[b,1]} & {[b,1]\vee[1]} & {[b,1]} & {[1]\vee[b,1]} \\
	{[1]} & {[1]\vee[1]} & {[1]} & {[1]\vee[1]}
	\arrow["\triangledown"', from=2-3, to=2-4]
	\arrow[from=1-4, to=2-4]
	\arrow[from=1-3, to=2-3]
	\arrow[from=1-3, to=1-4]
	\arrow[from=1-2, to=2-2]
	\arrow["\lrcorner"{anchor=center, pos=0.125, rotate=-90}, draw=none, from=1-4, to=2-3]
	\arrow["\triangledown"', from=2-1, to=2-2]
	\arrow[from=1-1, to=2-1]
	\arrow["\lrcorner"{anchor=center, pos=0.125, rotate=-90}, draw=none, from=1-2, to=2-1]
	\arrow[from=1-1, to=1-2]
\end{tikzcd}\]
According to lemma \ref{lemma:[ ,1] vee [ ,1]preserves spcial limits}, $[\uvar,1]\vee[1]$ and $[1]\vee[\uvar,1]$ preserve special colimits. As every $\io$-category is a colimit of representables, this implies that the squares
\[\begin{tikzcd}
	{[B,1]} & {[1]\vee[B,1]} & {[D,1]} & {[D,1]\vee[1]} \\
	{[1]} & {[1]\vee[1]} & {[1]} & {[1]\vee[1]}
	\arrow[from=1-3, to=1-4]
	\arrow[from=1-1, to=1-2]
	\arrow["\triangledown"', from=2-1, to=2-2]
	\arrow["\triangledown"', from=2-3, to=2-4]
	\arrow[from=1-4, to=2-4]
	\arrow[from=1-3, to=2-3]
	\arrow[from=1-1, to=2-1]
	\arrow[from=1-2, to=2-2]
	\arrow["\lrcorner"{anchor=center, pos=0.125}, draw=none, from=1-1, to=2-2]
	\arrow["\lrcorner"{anchor=center, pos=0.125}, draw=none, from=1-3, to=2-4]
\end{tikzcd}\]
are cartesian.
The result then follows from proposition \ref{prop:special colimit}.
\end{proof}
\begin{prop}
\label{prop:example of a special colimit4}
Suppose given a cartesian square
\[\begin{tikzcd}
	{ B} & C \\
	{\{0\}} & {[1]}
	\arrow[from=1-2, to=2-2]
	\arrow[from=1-1, to=1-2]
	\arrow[from=1-1, to=2-1]
	\arrow[from=2-1, to=2-2]
	\arrow["\lrcorner"{anchor=center, pos=0.125}, draw=none, from=1-1, to=2-2]
\end{tikzcd}\]
The diagram 
\[\begin{tikzcd}
	{[1]\vee[B,1]} & {[ B,1]} & {[C,1]}
	\arrow["\triangledown"', from=1-2, to=1-1]
	\arrow[from=1-2, to=1-3]
\end{tikzcd}\]
has a special colimit.
\end{prop}
\begin{proof}
The proof is similar to the previous one. 
\end{proof}

\p  We have an adjunction 
\begin{equation}
\label{eq:underived adjunction case n}
\begin{tikzcd}
	{ i_!:\iPsh{\Delta[\Theta_{n-1}]}} & {\iPsh{\Theta_n}:i^*}
	\arrow[shift left=2, from=1-1, to=1-2]
	\arrow[shift left=2, from=1-2, to=1-1]
\end{tikzcd}
\end{equation}
where the left adjoint is the left Kan extension of the functor $\Delta[\Theta_{n-1}]\xrightarrow{i} \Theta_{n}\to \iPsh{\Theta_{n}}$. We recall that the sets of morphisms $\W_n$ and $\M_n$ are respectively defined in paragraphs \ref{para:definition of W} and \ref{para:defi of delta theta}.
Remark that there is an obvious inclusion $i_!(\M_n)\subset \W_n$. The previous adjunction then induced a derived adjunction
\begin{equation}
\label{eq:derived adjunction case n}
\begin{tikzcd}
	{\Lb i_!:\Psh{\Delta[\Theta_{n-1}]}_{\M}} & {\Psh{\Theta_{n}}_{\W}:\Rb i^*}
	\arrow[shift left=2, from=1-1, to=1-2]
	\arrow[shift left=2, from=1-2, to=1-1]
\end{tikzcd}
\end{equation}

\begin{prop}
\label{prop:infini changing theta n}
The unit and counit of the adjunction \eqref{eq:underived adjunction case n} are respectively in $\widehat{\M}_n$ and $\widehat{\W}_n$. As a consequence, the adjunction \eqref{eq:derived adjunction case n} is an adjoint equivalence.
\end{prop}
\begin{proof}
We denote by $\iota:\Psh{\Theta_n}\to \iPsh{\Theta_n}$ and $\iota:\Psh{\Delta[\Theta_{n-1}]}\to \iPsh{\Delta[\Theta_{n-1}]}$ the two canonical inclusions. By the definition of the smallest precocomplete class (paragraph \ref{para:precomplet}) and according to lemma \ref{lemma:colimit computed in set presheaves}, we have inclusions $\iota(\overline{\W_{n}})\subset \widehat{\W_{n}}$ and $\iota(\overline{\M_{n}})\subset \widehat{\M_{n}}$. The result then directly follows from theorem \ref{theo:unit and counit are in W}.  
\end{proof}

\p \label{para:truncation and inteligent trucation}
 Let $n>0$ be an integer. An \wcnotion{$(\infty,n)$-category}{category3@$(\infty,n)$-category} is a $\W_n$-local $\infty$-presheaf $C\in \iPsh{\Theta_n}$. We then define \sym{((a50@$\ncat{n}$}
$$\ncat{n} := \iPsh{\Theta_n}_{\W_n}.$$
Remark that the $\iun$-category $\ncat{0}$ is equivalent to $\igrd$.
Proposition \ref{prop:infini changing theta n} implies that $\ncat{n}$ identifies itself with the full sub $\iun$-category of $\iPsh{\Delta[\Theta_{n-1}]}$ of $\M_n$-local objects:
$$\ncat{n} \sim \iPsh{\Delta[\Theta_{n-1}]}_{\M_n}.$$
The inclusion $i_n:\Theta_n\to \Theta$ fits in an adjunction
\[\begin{tikzcd}
	{\tau^i_n:\Theta} & {\Theta_n:i_n}
	\arrow[""{name=0, anchor=center, inner sep=0}, shift left=2, from=1-1, to=1-2]
	\arrow[""{name=1, anchor=center, inner sep=0}, shift left=2, from=1-2, to=1-1]
	\arrow["\dashv"{anchor=center, rotate=-90}, draw=none, from=0, to=1]
\end{tikzcd}\]
where the left adjoint sends $\Db_k$ on $\Db_{\min{(n,k)}}$.
By extension by colimits, this induces an adjoint pair 
\begin{equation}
\label{eq:inclusion of n cat pre}
\begin{tikzcd}
	{\tau^i_n:\iPsh{\Theta}} & {\iPsh{\Theta_n}:i_n.}
	\arrow[""{name=0, anchor=center, inner sep=0}, shift left=2, from=1-1, to=1-2]
	\arrow[""{name=1, anchor=center, inner sep=0}, shift left=2, from=1-2, to=1-1]
	\arrow["\dashv"{anchor=center, rotate=-90}, draw=none, from=0, to=1]
\end{tikzcd}
\end{equation}
where the two functors are colimit preserving.
As the image of every morphism of $\W$ by $\tau^i_n$ is in $\W_n$ or is an equivalence, and as the image of $\W_n$ by $i_n$ is included in $\W$, the previous adjunction induces by localization an adjunction
\begin{equation}
\label{eq:inclusion of n cat}
\begin{tikzcd}
	{\tau^i_n:\ocat} & {\ncat{n}:i_n}
	\arrow[""{name=0, anchor=center, inner sep=0}, shift left=2, from=1-1, to=1-2]
	\arrow[""{name=1, anchor=center, inner sep=0}, shift left=2, from=1-2, to=1-1]
	\arrow["\dashv"{anchor=center, rotate=-90}, draw=none, from=0, to=1]
\end{tikzcd}
\end{equation}
where the two adjoints are colimit preserving.
The left adjoint is called the \snotionsym{intelligent $n$-truncation}{(taui@$\tau^i_n$}{for $\io$-categories}.
\begin{prop}
\label{ref:infini n a full sub cat}
The functor $i_n: \ncat{n}\to \ocat$ is fully faithful.
\end{prop}
\begin{proof}
We have to check that the unit of the adjunction \eqref{eq:inclusion of n cat} is an equivalence. As the two functors preserve colimits, we have to show that the restriction to $\Theta$ of the unit is an equivalence which is obvious.
\end{proof}
Being colimit preserving, the functor $i_n$ is also part of an adjunction
\begin{equation}
\begin{tikzcd}
	{i_n:\ncat{n}} & {\ocat:\tau_n}
	\arrow[""{name=0, anchor=center, inner sep=0}, shift left=2, from=1-2, to=1-1]
	\arrow[""{name=1, anchor=center, inner sep=0}, shift left=2, from=1-1, to=1-2]
	\arrow["\dashv"{anchor=center, rotate=-90}, draw=none, from=1, to=0]
\end{tikzcd}
\end{equation}
The right adjoint is called the \wcsnotionsym{$n$-truncation}{(tau@$\tau_n$}{truncation@$n$-truncation}{for $\io$-category}. 

We will identify objects of $\ncat{n}$ with their image in $\ocat$ and we will then also note by $\tau_n$ and $\tau^i_n$ the composites $i_n\tau^i_n$ and $i_n\tau^i_n$.

\begin{prop}
\label{prop:taun preserves special colimits}
The functor $\tau_n:\ocat\to \ocat$ preserves special colimits.
\end{prop}
\begin{proof}
As $i_n$ preserves representable objects, the functor $\tau_n:\ocat\to \ncat{n}$ preserves special colimits. As $i_n:\iPsh{\Theta_n}\to \iPsh{\Theta}$ preserves colimits and $\W$-local objects, this concludes the proof.
\end{proof}

\begin{prop}
\label{prop:inteligent trucatio and a particular colimit}
Let $C$ be an $\io$-category and $n$ an integer. The following canonical square is cartesian
\[\begin{tikzcd}
	C & {\tau_n^iC} \\
	{\tau_n^iC} & {\tau_n^iC}
	\arrow[from=1-1, to=2-1]
	\arrow[from=2-1, to=2-2]
	\arrow[from=1-2, to=2-2]
	\arrow[from=1-1, to=1-2]
\end{tikzcd}\]
\end{prop}
\begin{proof}
For this results we use model categories. The theorem \ref{theo:lecorozo} implies that the $\iun$-category $\ocat$ is presented by the category of marked simplicial sets $\mSset$ endowed with the model structure for $\omega$-complicial sets given by proposition \ref{prop:model structure on marked simplicial set}, and the functor $\tau^i_n:\ocat\to \ocat$ corresponds to the left Quillen functor $\tau^i_n:\mSset\to \mSset$ given in paragraph \ref{para:inteligentr trucation for simplicial set}. Remark that in this model category, for any marked simplicial set $X$, the following square is cocartesian
\[\begin{tikzcd}
	X & {\tau_n^iX} \\
	{\tau_n^iX} & {\tau_n^iX}
	\arrow[from=1-1, to=2-1]
	\arrow[from=2-1, to=2-2]
	\arrow[from=1-2, to=2-2]
	\arrow[from=1-1, to=1-2]
\end{tikzcd}\]
As all the morphisms are cofibrations, this square is also homotopy cocartesian which concludes the proof.
\end{proof}

\p The family of truncation functor induces a sequence 
$$...\to \ncat{n+1}\xrightarrow{\tau_{n}} \ncat{n}\to...\to \ncat{1}\xrightarrow{\tau_{0}}\ncat{0}$$
which induces an adjunction
\begin{equation}
\label{eq:inductivity}
\begin{tikzcd}
	{\colim_{n:\Nb}:\lim_{n:\Nb}\ncat{n}} & {\ocat:(\tau_n)_{n:\Nb}}
	\arrow[""{name=0, anchor=center, inner sep=0}, shift left=2, from=1-1, to=1-2]
	\arrow[""{name=1, anchor=center, inner sep=0}, shift left=2, from=1-2, to=1-1]
	\arrow["\dashv"{anchor=center, rotate=-90}, draw=none, from=0, to=1]
\end{tikzcd}
\end{equation}
where the left adjoint sends a sequence $(C_n, C_n\sim \tau_nC_{n+1})_{n:\Nb}$ to the colimit of the induced sequence
$$i_0C_0\to i_1C_1\to ... \to i_nC_n\to ..., $$
and the right adjoint sends an $\io$-category $C$ to the sequence $(\tau_nC,\tau_nC\sim \tau_{n}\tau_{n+1}C)_{n:\Nb}$. Indeed, we have equivalence
$$
\begin{array}{rcl}
\Hom(\colim_{n:\Nb}i_n C_n,D)&\sim& \lim_{n:\Nb}\Hom(C_n,\tau_n D)
\\&\sim& \Hom( (C_n, C_n\sim \tau_nC_{n+1})_{n:\Nb},(\tau_n D,\tau_n D\sim \tau_{n}\tau_{n+1}D)_{n:\Nb})
\end{array}$$
natural in $(C_n, C_n\sim \tau_nC_{n+1})_{n:\Nb}$ and $D$.

\begin{prop}
\label{prop:infini omega a limit of infini n}
The adjunction \eqref{eq:inductivity} is an adjoint equivalence. As a consequence, we have an equivalence
$$\ocat\sim \lim_{n:\Nb}\ncat{n}.$$
\end{prop}
\begin{proof}
According to proposition \ref{prop:example of a special colimit 2}, any sequence $(C_n)_{n:\Nb}:\lim_{n:\Nb}\ncat{n}$ has a special colimit.
Let $k$ be an integer. According to proposition \ref{prop:taun preserves special colimits}, this implies the equivalence
$$\tau_k(\colim_{n:\Nb}C_n) \sim \colim_{n:\Nb}(\tau_kC_n).$$
Furthermore, the sequence $(\tau_kC_n)_{n:\Nb}$ is constant after the rank $k$. We then have 
$$\tau_k\colim_{n:\Nb}C_n \sim \tau_k C_n.$$
This directly implies that the unit of the adjunction \eqref{eq:inductivity} is an equivalence. 

To conclude, one has to show that the right adjoint is conservative, i.e that a morphism $f$ is an equivalence if and only if for any $n$, $\tau_n f$ is an equivalence. This last statement is a direct consequence of proposition \ref{prop:equivalences detected on globes}.
\end{proof} 
\p
The following proposition states that the cartesian product preserves colimits in both variables. There exists then an internal hom functor that we denote by \wcnotation{$\uHom(\uvar,\uvar)$}{(hom@$\uHom(\uvar,\uvar)$}.

\begin{prop}
\label{prop:cartesian product preserves W}
The cartesian product in $\ocat$ preserves colimits in both variables.
\end{prop}
We first need several lemmas:

\begin{lemma}
\label{lemma:product of representable in preshaves on Delta Theta}
Let $a$, $b$ be two globular sums, and $n,m$ two integer. The colimit in $\iPsh{\Delta[\Theta]}$ of the diagram 
\[\begin{tikzcd}
	{\coprod_{k\leq n}[a\times b,\{k\}\times [m]]} && {\coprod_{l\leq m}[a\times b,[n]\times \{l\}]} \\
	{\coprod_{k\leq n}[b,m]} & {[a\times b,[n]\times [m]]} & {\coprod_{l\leq m}[a,n]}
	\arrow[from=1-1, to=2-1]
	\arrow[from=1-1, to=2-2]
	\arrow[from=1-3, to=2-2]
	\arrow[from=1-3, to=2-3]
\end{tikzcd}\]
is $[a,n]\times [b,m]$.
\end{lemma}
\begin{proof}
The lemma \ref{lemma:colimit computed in set presheaves} implies that the object 
$$K:=\coprod_{k\leq n}[b,m]\coprod_{\coprod_{k\leq n}[a\times b,\{k\}\times [m]]}[a\times b,[n]\times [m]]$$
is strict. As the induced morphism 
$\coprod_{l\leq m}[a\times b,[n]\times \{l\}]\to K$, is a monomorphism, the lemma \textit{op cit} implies that the colimit of the diagram given in the statement is strict. We can then show the result in the category of set valued presheaves on $ \Delta[\Theta]$ and we leave this combinatorial exercise to the reader.
\end{proof}

\begin{lemma}
\label{lemma:technical cartesian product preserves W}
Let $f$ be a morphism of $\W_1$ and $n$ an integer. The morphism $f\times [n]$ is in $\widehat{\W_1}$.
\end{lemma}
\begin{proof}
Suppose first that $f$ is of shape $\Sp_m\to [m]$. Remark first that for any $k$, $[k]\times [m]$ is $\W_1$-local as both $[k]$ and $[m]$ are. We then have $\Fb_{\W_1}([k]\times[m])\sim [k]\times [m]$.
As the fibrant replacement preserves colimits and as the cartesian product in $\iun$-categories preserves colimits, we have a sequence of equivalences in $\icat$:
$$
\begin{array}{rcl}
\Fb_{\W_1}(\Sp_m\times [n])&\sim& \Fb_{\W_1}([1]\times [n])\coprod_{ \Fb_{\W_1}([0]\times [n])}...\coprod_{ \Fb_{\W_1}([0]\times [n])}\Fb_{\W_1}([1]\times [n])\\
&\sim& [1]\times [n]\coprod_{ [0]\times [n]}...\coprod_{ [0]\times [n]} [1]\times [n]\\
&\sim & [m]\times [n]
\end{array}
$$
By construction, the morphism $\Sp_m\times [n]\to \Fb_{\W_1}(\Sp_m\times [n])$ is in $\widehat{\W_1}$. We proceed similarly for the case $f:=E^{eq}\to [0]$.
\end{proof}

\begin{proof}[Proof of proposition \ref{prop:cartesian product preserves W}]
As the cartesian product on $\iPsh{\Theta}$ preserves colimits in both variables, according to corollary \ref{cor:derived colimit preserving functor}, we then have to show that for any globular sum $a$, and any $f\in\W$, $f\times a$ is in $\widehat{\W}$.

We demonstrate by induction on $k$ that for any $f\in\W_k$ and any globular sum $a$, $f\times a$ is in $\W_k$. The case $k=0$ is trivial as $\W_0$ is the singleton $\{id_{[0]}\}$.

Suppose then the statement is true at this stage $k$. We recall that we denote $(i_!,i^*)$ the left and right adjoints between $\iPsh{\Delta[\Theta]}$ and $\iPsh{\Theta}$. As $i^*$ preserves cartesian product, proposition \ref{prop:infini changing theta} implies that it is enough to show that for any $f\in\M_{k+1}$ and any object $[b,n]$, $f\times [b,n]$ is in $\widehat{\M}$. 

Suppose first that $f$ is of shape $[a,1]\to [c,1]$ for $a\to c \in \W_k$. According to lemma \ref{lemma:product of representable in preshaves on Delta Theta}, the morphism $f\times[b,m]$ is the colimit in depth of the diagram 
\[\begin{tikzcd}[column sep =0.1cm]
	{\coprod_{k\leq 1}[a\times b,\{k\}\times [m]]} && {\coprod_{l\leq m}[a\times b,[1]\times \{l\}]} \\
	{\coprod_{k\leq 1}[b,m]} & {[a\times b,[1]\times [m]]} & {\coprod_{l\leq m}[a,1]} \\
	& {\coprod_{k\leq 1}[c\times b,\{k\}\times [m]]} && {\coprod_{l\leq m}[c\times b,[1]\times \{l\}]} \\
	& {\coprod_{k\leq 1}[b,m]} & {[c\times b,[1]\times [m]]} & {\coprod_{l\leq m}[c,1]}
	\arrow[from=1-1, to=2-1]
	\arrow[from=1-1, to=2-2]
	\arrow[from=1-3, to=2-2]
	\arrow[from=1-3, to=2-3]
	\arrow[from=2-1, to=4-2]
	\arrow[from=1-1, to=3-2]
	\arrow[from=2-2, to=4-3]
	\arrow[from=2-3, to=4-4]
	\arrow[from=1-3, to=3-4]
	\arrow[from=3-2, to=4-3]
	\arrow[from=3-4, to=4-3]
	\arrow[from=3-4, to=4-4]
	\arrow[from=3-2, to=4-2]
\end{tikzcd}\]
The lemma \ref{lemma:the functor [] preserves classes} and the induction hypothesis implies that all the depth morphisms are in $\widehat{M}$.
By stability by colimit, this implies that $f\times[b,m]$ belongs to $\widehat{\M}$.

Suppose now that $f$ is of shape $[a,\Sp_n]\to [a,n]$. According to lemma \ref{lemma:product of representable in preshaves on Delta Theta}, the morphism $f\times[b,m]$ is the colimit in depth of the diagram 
\[\begin{tikzcd}[column sep = 0.1cm]
	{\coprod_{k\leq n}[a\times b,\{k\}\times [m]]} && {\coprod_{l\leq m}[a\times b,\Sp_n\times \{l\}]} \\
	{\coprod_{k\leq n}[b,m]} & {[a\times b,\Sp_n\times [m]]} & {\coprod_{l\leq m}[a,\Sp_n]} \\
	& {\coprod_{k\leq n}[a\times b,\{k\}\times [m]]} && {\coprod_{l\leq m}[a\times b,[n]\times \{l\}]} \\
	& {\coprod_{k\leq n}[b,m]} & {[a\times b,[n]\times [m]]} & {\coprod_{l\leq m}[a,n]}
	\arrow[from=1-1, to=2-1]
	\arrow[from=1-1, to=2-2]
	\arrow[from=1-3, to=2-2]
	\arrow[from=1-3, to=2-3]
	\arrow[from=3-2, to=4-3]
	\arrow[from=2-3, to=4-4]
	\arrow[from=1-3, to=3-4]
	\arrow[from=3-4, to=4-4]
	\arrow[from=3-4, to=4-3]
	\arrow[from=3-2, to=4-2]
	\arrow[from=2-1, to=4-2]
	\arrow[from=1-1, to=3-2]
	\arrow[from=2-2, to=4-3]
\end{tikzcd}\]
The lemma \ref{lemma:technical cartesian product preserves W} implies that $\Sp_n\times [m]\to [n]\times [m]$ is in $\widehat{\W_1}$. Combined with lemma \ref{lemma:the functor [] preserves classes}, this implies that all the morphisms in depth are in $\widehat{\M}$. By stability by colimit, so is $f\times[b,m]$.

It remains to show the case $f= E^{eq}\to [0]$. According to lemma \ref{lemma:product of representable in preshaves on Delta Theta}, the morphism $f\times[b,m]$ is the horizontal colimit of the diagram
\[\begin{tikzcd}
	{\coprod_{k\leq m} E^{eq}} & {\coprod_{k\leq m} [b,E^{eq}\times \{k\}]} & { [b,E^{eq}\times [m]]} \\
	{\coprod_{k\leq m}[0]} & {\coprod_{k\leq m}[0]} & {[b,m]}
	\arrow[from=1-2, to=1-3]
	\arrow[from=1-2, to=1-1]
	\arrow[from=2-2, to=2-1]
	\arrow[from=2-2, to=2-3]
	\arrow[from=1-3, to=2-3]
	\arrow[from=1-2, to=2-2]
	\arrow[from=1-1, to=2-1]
\end{tikzcd}\]
The lemma \ref{lemma:technical cartesian product preserves W} implies that $ E^{eq}\times [m]\to [m]$ is in $\widehat{\W_1}$. Combined with lemma \ref{lemma:the functor [] preserves classes}, this implies that all the vertical morphisms are in $\widehat{\M}$. By stability by colimit, so is $f\times[b,m]$.
\end{proof}

\begin{cor}
\label{cor:if codomain a groupoid, then f is exponentiable}
Let $C$ be an $\io$-category, $S$ an $\infty$-groupoid, and $f:C\to S$ any morphism.
The functor $f^*:\ocat_{/S}\to \ocat_{/C}$ preserves colimits. 
\end{cor}
\begin{proof}
As $\iPsh{\Theta}$ is locally cartesian closed, we just have to verify that for any cartesian squares:
\[\begin{tikzcd}
	{C''} & {C'} & C \\
	a & b & S
	\arrow["i"', from=2-1, to=2-2]
	\arrow[from=1-3, to=2-3]
	\arrow["j", from=1-1, to=1-2]
	\arrow[from=1-2, to=1-3]
	\arrow[from=2-2, to=2-3]
	\arrow[from=1-1, to=2-1]
	\arrow[from=1-2, to=2-2]
	\arrow["\lrcorner"{anchor=center, pos=0.125}, draw=none, from=1-1, to=2-2]
	\arrow["\lrcorner"{anchor=center, pos=0.125}, draw=none, from=1-2, to=2-3]
\end{tikzcd}\]
if $i$ is in $\W$, then $j$ is in $\widehat{\W}$. Suppose given such cartesian squares. As $b$ is a globular form, $\tau^i_0(b)\sim 1$ and 
as $S$ is an $\infty$-groupoid, there exists an object $s$ of $S$ such that the morphism $b\to S$ factor through $\{s\}\to S$. If we denote by $C_s$ the fiber of $f$ in $\{s\}$, the morphisms $i$ and $j$ then fit in the following cartesian squares:
\[\begin{tikzcd}
	{C_s\times a} & {C_s\times b} & {C_s} & C \\
	a & b & {\{s\}} & S
	\arrow["i"', from=2-1, to=2-2]
	\arrow[from=1-3, to=2-3]
	\arrow["j", from=1-1, to=1-2]
	\arrow[from=1-2, to=1-3]
	\arrow[from=2-2, to=2-3]
	\arrow[from=1-1, to=2-1]
	\arrow[from=1-2, to=2-2]
	\arrow["\lrcorner"{anchor=center, pos=0.125}, draw=none, from=1-1, to=2-2]
	\arrow["\lrcorner"{anchor=center, pos=0.125}, draw=none, from=1-2, to=2-3]
	\arrow[from=2-3, to=2-4]
	\arrow[from=1-3, to=1-4]
	\arrow[from=1-4, to=2-4]
	\arrow["\lrcorner"{anchor=center, pos=0.125}, draw=none, from=1-3, to=2-4]
\end{tikzcd}\]
The proposition \ref{prop:cartesian product preserves W} implies that $j$ verifies the desired property, which concludes the proof.
\end{proof}

The following proposition implies that a natural transformation is an equivalence if and only if it is pointwise one. 
\begin{prop}
\label{prop:cartesian square and times}
For any $\io$-categories $X$ and $C$, the following natural square is cartesian:
\[\begin{tikzcd}
	{\tau_0\uHom(X,C)} & {\uHom(X,C)} \\
	{\uHom(\tau_0X,\tau_0C)} & {\uHom(\tau_0X,C)}
	\arrow[from=1-1, to=2-1]
	\arrow[from=1-2, to=2-2]
	\arrow[from=2-1, to=2-2]
	\arrow[from=1-1, to=1-2]
\end{tikzcd}\]
\end{prop}
\begin{proof}
As $\uHom(\uvar,C)$ sends colimits to limits, we can suppose that $X$ is of shape $\Db_n$ for $n\geq 0$. Eventually, proposition \ref{prop:equivalences detected on globes} implies that pullbacks are detected on globes. We then have to show that for any integer $m$, the following square is cartesian:
$$
\begin{tikzcd}
	{\tau_0\uHom(\Db_n,C)} & {\Hom(\Db_n\times\Db_m,C)} \\
	{\Hom(\tau_0\Db_n\times\Db_m,\tau_0C)} & {\Hom((\tau_0\Db_n)\times\Db_m,C)}
	\arrow[from=1-1, to=2-1]
	\arrow[from=1-2, to=2-2]
	\arrow[from=2-1, to=2-2]
	\arrow[from=1-1, to=1-2]
\end{tikzcd}$$
To this extent, we claim that the following square is cocartesian in $\ocat$:
\begin{equation}
\label{eq:proof of cartesian}
\begin{tikzcd}
	{(\tau_0\Db_n)\times\Db_m} & {\Db_n\times\Db_m} \\
	{\tau_0\Db_n} & {\Db_n}
	\arrow[from=1-1, to=2-1]
	\arrow[from=2-1, to=2-2]
	\arrow[from=1-1, to=1-2]
	\arrow[from=1-2, to=2-2]
\end{tikzcd}
\end{equation}
Applying the functor $\uHom(\uvar,C)$ it will prove the desired property.
To show the cocartesianess of \eqref{eq:proof of cartesian}, remark that if either $n$ or $m$ is null, this is trivial. If not, proposition \ref{prop:example of a special colimit} states that $\Db_n\times\Db_m$ is the colimit of the span:
$$[\Db_{n-1},1]\vee[\Db_{m-1},1]\leftarrow [\Db_{n-1}\times \Db_{m-1},1]\to [\Db_{m-1},1]\vee[\Db_{n-1},1]$$
Using the two cartesian squares
\[\begin{tikzcd}
	{[\Db_{m-1},1]} & {[\Db_{m-1},1]\vee[\Db_{n-1},1]} & {[\Db_{m-1},1]} & {[\Db_{n-1},1]\vee[\Db_{m-1},1]} \\
	{[0]} & {[\Db_{n-1},1]} & {[0]} & {[\Db_{n-1},1]}
	\arrow[from=1-3, to=2-3]
	\arrow[from=1-3, to=1-4]
	\arrow[from=2-3, to=2-4]
	\arrow[from=1-4, to=2-4]
	\arrow[from=1-2, to=2-2]
	\arrow["\lrcorner"{anchor=center, pos=0.125, rotate=180}, draw=none, from=2-4, to=1-3]
	\arrow[from=1-1, to=2-1]
	\arrow[from=2-1, to=2-2]
	\arrow[from=1-1, to=1-2]
	\arrow["\lrcorner"{anchor=center, pos=0.125, rotate=180}, draw=none, from=2-2, to=1-1]
\end{tikzcd}\]
this implies that the pushout of the upper span of \eqref{eq:proof of cartesian} is then the colimit of the diagram:
\begin{equation}
\label{eq:proof of cartesian2}
[\Db_{n-1},1]\leftarrow [\Db_{n-1}\times \Db_{m-1},1]\to [\Db_{n-1},1]
\end{equation}
The proposition \ref{prop:inteligent trucatio and a particular colimit} states that the square
\[\begin{tikzcd}
	{ \Db_{m-1}} & 1 \\
	1 & 1
	\arrow[from=2-1, to=2-2]
	\arrow[from=1-1, to=1-2]
	\arrow[from=1-1, to=2-1]
	\arrow[from=1-2, to=2-2]
\end{tikzcd}\]
is cocartesian. Combined with proposition \ref{prop:cartesian product preserves W}, this implies that the square
\[\begin{tikzcd}
	{\Db_{n-1}\times \Db_{m-1}} & {\Db_{n-1}} \\
	{\Db_{n-1}} & {\Db_{n-1}}
	\arrow[from=2-1, to=2-2]
	\arrow[from=1-1, to=1-2]
	\arrow[from=1-1, to=2-1]
	\arrow[from=1-2, to=2-2]
\end{tikzcd}\]
is cocartesian. 
As a consequence, the colimit of the span \eqref{eq:proof of cartesian2}, and so of the upper span of \eqref{eq:proof of cartesian}, is $ [\Db_{n-1},1]\sim \Db_n$, which concludes the proof. 
\end{proof}

\p
\label{para:dualities non strict case}
In paragraph \ref{para:dualities strict case}, for any subset $S$ of $\Nb^*$, we have defined the duality $(\uvar)^S:\zocat\to \zocat$.
 These functors restrict to functors $\Theta\to \Theta$ that induce by extension by colimit functors \ssym{((b49@$(\uvar)^S$}{for $\io$-categories}
$$(\uvar)^S:\iPsh{\Theta}\to \iPsh{\Theta}$$
which are once again called \snotion{dualities}{for $\io$-categories}.
It is easy to see that this functor preserves $\io$-categories and then induces functors
$$(\uvar)^S:\ocat\to \ocat.$$

In particular, we have the \snotionsym{odd duality}{((b60@$(\uvar)^{op}$}{for $\io$-categories} $(\uvar)^{op}$, corresponding to the set of odd integer, the \snotionsym{even duality}{((b50@$(\uvar)^{co}$}{for $\io$-categories} $(\uvar)^{co}$, corresponding to the subset of non negative even integer, the \snotionsym{full duality}{((b80@$(\uvar)^{\circ}$}{for $\io$-categories} $(\uvar)^{\circ}$, corresponding to $\Nb^*$ and the \snotionsym{transposition}{((b70@$(\uvar)^t$}{for $\io$-categories} $(\uvar)^t$, corresponding to the singleton $\{1\}$. Eventually, we have equivalences
$$((\uvar)^{co})^{op}\sim (\uvar)^{\circ} \sim ((\uvar)^{op})^{co}.$$

\p A morphism $f:C\to D$ is an \notion{epimorphism} if it is in the smallest cocomplete $\infty$-groupoid of arrows of $\ocat$ that includes the codiagonal $\Db_n\coprod\Db_n\to \Db_n$ for any $n\geq 0$. A morphism is a \notion{monomorphism} if it has the unique right lifting property against epimorphisms.

 A morphism $i:C\to D$ is then a monomorphism if and only if for any $n$, $C_n\to D_n$ is a monomorphism.
The small object argument induces a factorization system:
\begin{equation}
\label{eq:epimonomorphism factorization}
C\to \im i\to D
\end{equation}
of any morphism $i:C\to D$, where the left map is an epimorphism, and the right one is a monomorphism. The object \wcnotation{$\im i$}{(im@$\im$} is called the \wcnotion{image of $i$}{image of a morphism}. We then have by construction the following result:

\begin{prop}
A morphism is an equivalence if and only if it is both a monomorphism and a epimorphism.
\end{prop}

\begin{prop}
\label{prop:the image is stable under cartesian product}
The image is stable under the cartesian product.
\end{prop}
\begin{proof}
One has to show that both epimorphisms and monomorphisms are stable under the functor $\uvar\times A$ for $A$ any $\io$-category. For monomorphisms, it is a direct consequence of the fact that this notion has been defined with a right lifting property. For epimorphisms, as $\uvar\times A$ commutes with colimit, we can reduce to show that for any $n$, 
$$(\Db_n\coprod\Db_n)\times A \sim \Db_n\times A\coprod\Db_n\times A\to \Db_n \times A$$
is an epimorphism. 
However, the $\infty$-groupoid of object $B$ such that 
$B\coprod B\to B$ is an epimorphism is closed by colimits and contains globes. This $\infty$-groupoid then contains all the object and so in particular $\Db_n \times A$.
\end{proof}

\begin{lemma}
\label{lemma:id is an epi}
For any integer $n$, the projection $\Ib:\Db_{n+1}\to \Db_n$ is an epimorphism. 
\end{lemma}
\begin{proof}
Remark first that we have a cocartesian square:
\[\begin{tikzcd}
	{\partial\Db_n\coprod \partial\Db_n} & {\Db_n\coprod\Db_n} \\
	{\partial\Db_n} & {\partial\Db_{n+1}}
	\arrow[""{name=0, anchor=center, inner sep=0}, from=1-1, to=1-2]
	\arrow[from=1-1, to=2-1]
	\arrow[from=1-2, to=2-2]
	\arrow[from=2-1, to=2-2]
	\arrow["\lrcorner"{anchor=center, pos=0.125, rotate=180}, draw=none, from=2-2, to=0]
\end{tikzcd}\]
As the left hand morphism is an epimorphism, so is the right one. By stability by left cancellation, this implies that $\partial\Db_{n+1}\to \Db_n$ is an epimorphism.
Now, the map $\Ib$ can be factored as: 
\[\begin{tikzcd}
	{\partial\Db_{n+1}} & {\Db_n} \\
	{\Db_{n+1}} & {\Db_{n+1}\coprod_{\partial\Db_{n+1}}\Db_n} & {\Db_n} \\
	& {\partial\Db_{n+2}} & {\Db_{n+1}}
	\arrow[from=1-1, to=2-1]
	\arrow[""{name=0, anchor=center, inner sep=0}, from=1-1, to=1-2]
	\arrow[from=2-1, to=2-2]
	\arrow[from=1-2, to=2-2]
	\arrow[from=3-2, to=2-2]
	\arrow[from=2-2, to=2-3]
	\arrow[from=3-2, to=3-3]
	\arrow["\lrcorner"{anchor=center, pos=0.125, rotate=-90}, draw=none, from=2-3, to=3-2]
	\arrow[from=3-3, to=2-3]
	\arrow["\lrcorner"{anchor=center, pos=0.125, rotate=180}, draw=none, from=2-2, to=0]
\end{tikzcd}\]
which directly implies that $\Ib$ is an epimorphism. 
\end{proof}

\begin{prop}
\label{prop:intelignet truncation is poitwise an epi}
For any integer $n$, the canonical natural transformation $id\to \tau^i_n$ is pointwise an epimorphism. 
\end{prop}
\begin{proof}
This is a direct consequence of lemma \ref{lemma:id is an epi}.
\end{proof}

\begin{prop}
\label{prop:canonical epi}
For any integer $n$, any $(\infty,n)$-category $C$, and any $\io$-category $D$, the canonical morphisms
$$\alpha:\coprod_{C_n}\Db_n\to C~~~~~\beta:\coprod_{(n,D_n)}\Db_n\to D$$
are epimorphisms.
\end{prop}
\begin{proof}
Let $I$ be the image of $\alpha$. We are willing to show that the canonical morphism $j:I\to C$ is an equivalence.
According to lemma \ref{lemma:equivalence if unique right lifting property against globes.}, and as $j$ is a monomorphism, we have to show that $j$ has the (non unique) right lifting property against $\emptyset\to \Db_k$ for any $k\leq n$. It is sufficient to show that $\alpha$ has the (non unique) right lifting property against $\emptyset\to \Db_k$ for any $k\leq n$, which is obviously true. 
We proceed similarly for $\beta$.
\end{proof}

\begin{prop}
\label{prop:truncation of epimorphism is pushout}
Let $i:A\to B$ be an epimorphism and $n$ an integer. The canonical square 
\[\begin{tikzcd}
	A & B \\
	{\tau^i_n(A)} & {\tau^i_n(B)}
	\arrow[from=1-1, to=2-1]
	\arrow["{\tau^i_n(i)}"', from=2-1, to=2-2]
	\arrow["i", from=1-1, to=1-2]
	\arrow[from=1-2, to=2-2]
\end{tikzcd}\]
is cocartesian. 
\end{prop}
\begin{proof}
We can reduce to the case where $i$ is $\Db_k\coprod\Db_k\to \Db_k$. If $n\geq k$, it is directly true, and we then suppose $n<k$. In this case, the colimit of the span:
$$\Db_n\coprod \Db_n \leftarrow \Db_k\coprod\Db_k\to \Db_k$$
is $\Db_n\coprod_{\Db_k}\Db_n$. The proposition \ref{prop:inteligent trucatio and a particular colimit} implies that this pushout is $\Db_n$, which concludes the proof.
\end{proof}

\p
A functor $f:C\to D$ is \snotion{fully faithful}{for $\io$-categories} if for any pair of objects $a,b\in C$, the induced morphism 
$\hom_C(a,b)\to \hom_D(fa,fb)$ is an equivalence.

\begin{prop}
\label{prop:ff 1}
A functor is fully faithful if and only if it has the unique right lifting property against $\{0\}\coprod \{1\}\to \Db_n$ for $n>0$.
\end{prop}
\begin{proof}
Let $f$ be a functor having the unique right lifting property against $\{0\}\coprod \{1\}\to \Db_n$ for $n>0$. As $[\emptyset,1] =\{0\}\coprod \{1\}$ and $[\Db_n,1] = \Db_{n+1}$, 
this is equivalent to asking for any pair of objects $c,d$ and for any integer $n$, that $f(c,d)$ has the unique right lifting property against $\emptyset\to \Db_n$, which in turn is equivalent to $f$ being fully faithful according to lemma \ref{lemma:equivalence if unique right lifting property against globes.}.
\end{proof}

\begin{prop}
\label{prop:ff 2}
Fully faithful functors are stable under limits.
\end{prop}
\begin{proof}
This is a consequence of the fact that fully faithful functors are characterized by unique right lifting properties.
\end{proof}

\begin{lemma}
\label{lemma:ff 2}
Let $p:C\to D$ be a fully faithful functor. The induced morphism $C_0\to D_0$ is a monomorphism.
\end{lemma}
\begin{proof}
To this extent, we have to show that $p:C\to D$ has the unique right lifting property against $1\coprod 1\to 1$. This is equivalent to show that $p$ has the unique right lifting property against $\iota: 1\coprod 1 \to E^{eq}$.

The proposition \ref{prop:ff 1} implies that $p$ as the unique right lifting property against $1\coprod 1\to \Db_1$ and $1\coprod 1\to \Db_2$
By left cancellation, this implies that  $p$ has the unique right lifting property against $\Db_2\to \Db_1$. As $\iota$ is a composition of pushouts along  $1\coprod 1\to \Db_1$ and  $\Db_2\to \Db_1$, this directly concludes the proof.
\end{proof}

\begin{prop}
\label{prop:fully faithful plus surjective on objet}
A morphism $f:C\to D$ is an equivalence if and only if it is fully faithful and induces a surjection on objects.
\end{prop}
\begin{proof}
This is necessary. Suppose that $f$ is fully faithful. According to 	\ref{prop:ff 1}, for any $n>0$, $f_n:C_n\to D_n$ is an equivalence. If $f$ induces a surjection on objects, lemma \ref{lemma:ff 2} implies that $f_0:C_0\to D_0$ is an equivalence. We can then apply proposition \ref{prop:equivalences detected on globes}.
\end{proof}

\subsection{Discrete Conduché functors}
\label{section:conduche}
\p
We denote \wcnotation{$\triangledown_{k,n}$}{(nabla@$\triangledown_{k,n}$} the unique globular morphism between $\Db_n$ and $\Db_n\coprod_{\Db_k}\Db_n$.
A morphism $f:C\to D$ between $\io$-categories is a \snotion{discrete Conduché functor}{for $\io$-categories} if it has the unique right lifting property against 
units $\Ib_{n+1}:\Db_{n+1}\to \Db_n$ for any integer $n$, and against compositions $\triangledown_{k,n}:\Db_n\to \Db_n\coprod_{\Db_k}\Db_n$ for any pair of integers $k\leq n$.
\begin{lemma}
\label{lemma:technicalconduche have the rlp aginst alebraic morphism}
The two following full sub $\infty$-groupoids of morphisms of $\ocat$ are equivalent: 
\begin{enumerate}
\item The smallest cocomplete full sub $\infty$-groupoid of morphisms containing the family of morphism $\{\Ib_{n+1}:\Db_{n+1}\to \Db_n,\}$ and the family $\{\triangledown_{k,n}:\Db_n\to \Db_n\coprod_{\Db_k}\Db_n\, ~k\leq n\}$.
\item The smallest cocomplete full sub $\infty$-groupoid of morphisms containing algebraic morphisms of $\Theta$ (this notion is defined in paragraph \ref{para:algebraic and globular}). 
\end{enumerate}
\end{lemma}
\begin{proof}
For any pair of integers $k\leq n$, $\Ib_{n+1}$ and $\triangledown_{k,n}$ are algebraic morphisms. This directly induces the inclusion of the fist $\infty$-groupoid in the second one. To conclude, one has to show that every algebraic morphism $i:a\to b$ is contained in the first $\infty$-groupoid.

We proceed by induction on $|a|+|b|$. Suppose first that there exists $n$ such that $a=\Db_n$. In this case two cases have to be considered. Either $n>0$ and $i$ factors as $\Db_n\xrightarrow{\Ib_n} \Db_{n-1}\xrightarrow{j} b$. The result then follows by the induction hypothesis. Suppose now that $i$ does not factor though $\Ib_n$. In this case, there exists $k$ such that $i$ factors as $\Db_n\xrightarrow{\triangledown_{k,n}} \Db_n\coprod_{\Db_k}\Db_n\xrightarrow{j} b$. The unique factorization system between algebraic and globular morphisms given in proposition \ref{prop:algebraic ortho to globular} produces a diagram 
\[\begin{tikzcd}
	&& {\Db_n} &&& {b_2} \\
	& {\Db_k} &&& {b_1} \\
	{\Db_n} &&& {b_0} \\
	&& { \Db_n\coprod_{\Db_k}\Db_n } &&& b
	\arrow["j"{description}, from=4-3, to=4-6]
	\arrow[hook, from=3-1, to=4-3]
	\arrow[hook, from=1-3, to=4-3]
	\arrow["{j_2}"', from=1-3, to=1-6]
	\arrow["{j_0}", from=3-1, to=3-4]
	\arrow[hook, from=3-4, to=4-6]
	\arrow[hook, from=1-6, to=4-6]
	\arrow[hook, from=2-2, to=4-3]
	\arrow[hook, from=2-2, to=1-3]
	\arrow[hook', from=2-2, to=3-1]
	\arrow[hook', from=2-5, to=3-4]
	\arrow[hook, from=2-5, to=1-6]
	\arrow[hook, from=2-5, to=4-6]
	\arrow["{j_1}"{description}, from=2-2, to=2-5]
\end{tikzcd}\]
where arrows labeled by $\hookrightarrow$ are globular and the other ones are algebraic. Remark that we have a cocartesian square in $\iun$-category of arrows of $\ocat$:
\[\begin{tikzcd}
	{j_1} & {j_2} \\
	{j_0} & j
	\arrow[from=1-1, to=2-1]
	\arrow[from=1-1, to=1-2]
	\arrow[from=2-1, to=2-2]
	\arrow[from=1-2, to=2-2]
\end{tikzcd}\]
is cocartesian. As $j_0$, $j_1$ and $j_2$ are in the first $\infty$-groupoid by induction hypothesis, so is $j$. By stability by composition, the morphism $i$ is then in the first $\infty$-groupoid.

Suppose now that the domain of $i:a\to b$ is not a globe. Using once again the unique factorization system between algebraic and globular, we can construct a functor $\Sp_a\to \Arr(\Theta)$ whose value on $\Db_n\hookrightarrow a$ is given by the unique algebraic morphism $j$ fitting in a commutative square
\[\begin{tikzcd}
	{\Db_n} & {b'} \\
	a & b
	\arrow[hook, from=1-1, to=2-1]
	\arrow[hook, from=1-2, to=2-2]
	\arrow["i"', from=2-1, to=2-2]
	\arrow["j", from=1-1, to=1-2]
\end{tikzcd}\]
where arrows labeled by $\hookrightarrow$ are globular. By induction hypothesis, $j$ is in the first $\infty$-groupoid. The colimit of 
$\Sp_a\to \Arr(\Theta)\to \Arr(\ocat)$ is then in the first $\infty$-groupoid. As this colimit is $i$, this concludes the proof.
\end{proof}

\begin{prop}
\label{prop:conduche have the rlp aginst alebraic morphism}
A morphism $f:X\to Y$ is a discrete Conduché functor if and only if it as the unique right lifting property against algebraic morphism of $\Theta$ (this notion is defined in paragraph \ref{para:algebraic and globular}). 
\end{prop}
\begin{proof}
Given a morphism $f$, the full sub $\infty$-groupoid of morphisms having the unique left lifting property against $f$ is cocomplete. The result is then a direct implication of lemma 
\ref{lemma:technicalconduche have the rlp aginst alebraic morphism}.
\end{proof}

\begin{example}
The proposition \ref{prop:algebraic ortho to globular} implies that a morphism $a\to b$ between globular sums is a discrete Conduché functor if and only if it is globular.
\end{example}

\begin{lemma}
\label{lemma:conduche technical}
Let $p:C\to a$ be discrete Conduché functor with $a$ a globular sum. We denote by $(\Theta_{/p})^{Cd}$ the full sub $\iun$-category of $\Theta_{/p}$ whose objects are triangles 
\[\begin{tikzcd}
	b & C \\
	& a
	\arrow[from=1-1, to=1-2]
	\arrow["p", from=1-2, to=2-2]
	\arrow[from=1-1, to=2-2]
\end{tikzcd}\]
where every arrow is a discrete Conduché functor.
The canonical inclusion of $\iun$-category $\iota:(\Theta_{/p})^{Cd}\to \Theta_{/p}$ is final.
\end{lemma}
\begin{proof}
To prove this statement, we will endow $\iota$ with a structure of right deformation retract. We then first build a right inverse of $\iota$.
Any triangle 
\[\begin{tikzcd}
	b & C \\
	& a
	\arrow[from=1-1, to=1-2]
	\arrow["p", from=1-2, to=2-2]
	\arrow[from=1-1, to=2-2]
\end{tikzcd}\]
induces a diagram of shape
\[\begin{tikzcd}
	b & C \\
	{b'} & a
	\arrow[from=1-1, to=1-2]
	\arrow[from=1-1, to=2-1]
	\arrow[from=2-1, to=2-2]
	\arrow["l", from=2-1, to=1-2]
	\arrow["p", from=1-2, to=2-2]
\end{tikzcd}\]
where $b'$ is obtained in factorizing $b\to a$ in a algebraic morphism followed by a globular morphism, and $l$ comes from the unique right lifting property of $p$ against algebraic morphisms. By right cancellation, this implies that $l$ is a discrete Conduché functor.

 As these two operations are functorial, this defines a retraction $r: \Theta_{/p}\to (\Theta_{/p})^{Cd}$ sending the triangle spotted by $b,C$ and $a$ to the triangle spotted by $b',C$ and $a$. Moreover, this retraction comes along with a natural transformation $id\to r\iota$. As right deformation retracts are final, this concludes the proof.
\end{proof}

\begin{lemma}
\label{lemma:conduche preserves W}
Let $p:C\to D$ be a discrete Conduché functor. Then for any globular sums $a$, and any cartesian squares in $\iPsh{\Theta}$:
\[\begin{tikzcd}
	{C''} & {C'} & C \\
	{\Sp_a} & a & D
	\arrow["{p''}", from=1-1, to=2-1]
	\arrow["{p'}", from=1-2, to=2-2]
	\arrow["p", from=1-3, to=2-3]
	\arrow["j", from=1-1, to=1-2]
	\arrow[from=1-2, to=1-3]
	\arrow[from=2-1, to=2-2]
	\arrow[from=2-2, to=2-3]
	\arrow["\lrcorner"{anchor=center, pos=0.125}, draw=none, from=1-2, to=2-3]
	\arrow["\lrcorner"{anchor=center, pos=0.125}, draw=none, from=1-1, to=2-2]
\end{tikzcd}\]
the morphism $j$ is in $\widehat{\Wseg}$.
\end{lemma}
\begin{proof}
By stability under pullback,
the morphism $p'$ is a discrete Conduché functor. 
Taking the notations of lemma \ref{lemma:conduche technical}, $p'$ is equivalent to $\colim_{(\Theta_{/p})^{Cd}}b\to a$ where this colimit is taken in $\iPsh{\Theta}_{/a}$. As $\iPsh{\Theta}$ is locally cartesian closed and as $\widehat{\W}$ is by definition closed by colimits, we can then reduce to the case where $p'$ is a discrete Conduché functor between globular sums, i.e a globular morphism $b\to a$.
In this case, the following canonical square is a pullback
\[\begin{tikzcd}
	{\Sp_b} & b \\
	{\Sp_a} & a
	\arrow[from=2-1, to=2-2]
	\arrow["{p'}", from=1-2, to=2-2]
	\arrow[from=1-1, to=2-1]
	\arrow[from=1-1, to=1-2]
	\arrow["\lrcorner"{anchor=center, pos=0.125}, draw=none, from=1-1, to=2-2]
\end{tikzcd}\]
and this concludes the proof.
\end{proof}

\begin{lemma}
\label{lemma:pulback of Wsat preresult}
Consider a cartesian square 
\[\begin{tikzcd}
	X & Y \\
	{\Sigma^{n} E^{eq}} & {\Db_{n}}
	\arrow[from=1-1, to=2-1]
	\arrow["j", from=1-1, to=1-2]
	\arrow[from=2-1, to=2-2]
	\arrow[from=1-2, to=2-2]
	\arrow["\lrcorner"{anchor=center, pos=0.125}, draw=none, from=1-1, to=2-2]
\end{tikzcd}\]
in $\iPsh{\Theta}$. The morphism $j$ is in $\widehat{\W}$.
\end{lemma}
\begin{proof}
 If we are in the case $n=0$, this directly follows from the preservation of $\W$ by cartesian product, demonstrated in the proof of proposition \ref{prop:cartesian product preserves W}.
We now suppose the result is true at stage $n$, and we first show that for any square
\[\begin{tikzcd}
	X & Y \\
	{[\Sigma^{n} E^{eq},1]} & {[\Db_{n+1},1]}
	\arrow[from=1-1, to=2-1]
	\arrow["j", from=1-1, to=1-2]
	\arrow[from=2-1, to=2-2]
	\arrow["p", from=1-2, to=2-2]
	\arrow["\lrcorner"{anchor=center, pos=0.125}, draw=none, from=1-1, to=2-2]
\end{tikzcd}\]
in $\iPsh{\Delta[\Theta]}$, $j$ is in $\widehat{\M}$.
As $\iPsh{\Delta[\Theta]}$ is locally cartesian closed and $\widehat{\M}$ closed under colimits, one can suppose that $Y$ is of shape $[a,k]$ and we denote $f:[k]\to [1]$ the morphism induced by $p$. By stability under pullback, $X$ is then set-valued. Furthermore, we can then check in $\Psh{\Delta[\Theta]}$ that this presheaf fits in a cocartesian square:
\[\begin{tikzcd}
	{[\Sigma^nE^{eq}\times_{\Db_n} a,f^{-1}(0)]\coprod [\Sigma^nE^{eq}\times_{\Db_n}a,f^{-1}(1)]} & {[\Sigma^nE^{eq}\times_{\Db_n} a,k]} \\
	{[a,f^{-1}(0)]\coprod [a,f^{-1}(1)]} & X
	\arrow[from=1-1, to=2-1]
	\arrow[from=1-1, to=1-2]
	\arrow[from=2-1, to=2-2]
	\arrow[from=1-2, to=2-2]
\end{tikzcd}\]
By induction hypothesis $[\Sigma^nE^{eq}\times_{\Db_n} a,l]\to [a,l]$ is in $\widehat{\M}$ for any integer $l$. As $X\to [a,k]$ is the colimit in depth of the diagram
\[\begin{tikzcd}[column sep = 0.7cm]
	{[\Sigma^nE^{eq}\times_{\Db_n} a,f^{-1}(0)]} && { [\Sigma^nE^{eq}\times_{\Db_n}a,f^{-1}(1)]} \\
	{[a,f^{-1}(0)]} & {[\Sigma^nE^{eq}\times_{\Db_n} a,k]} & {[a,f^{-1}(1)]} \\
	& {[a,f^{-1}(0)]} && {[a,f^{-1}(1)]} \\
	& {[a,f^{-1}(0)]} & {[a,k]} & {[a,f^{-1}(1)]}
	\arrow[from=4-3, to=3-2]
	\arrow[from=3-2, to=4-2]
	\arrow[from=1-1, to=2-1]
	\arrow[from=1-1, to=2-2]
	\arrow[from=1-3, to=2-2]
	\arrow[from=1-3, to=2-3]
	\arrow[from=3-4, to=4-3]
	\arrow[from=3-4, to=4-4]
	\arrow[from=2-1, to=4-2]
	\arrow[from=1-1, to=3-2]
	\arrow[from=2-2, to=4-3]
	\arrow[from=2-3, to=4-4]
	\arrow[from=1-3, to=3-4]
\end{tikzcd}\]
this implies that this morphism is in $\widehat{\M}$.

We now return to $\infty$-presheaves on $\Theta$. We recall that we denote by $(i_!,i^*)$ the adjunction between $\iPsh{\Delta[\Theta]}$ and $\iPsh{\Theta}$. Suppose given a cartesian square:
\[\begin{tikzcd}
	X & Y \\
	{\Sigma^{n+1} E^{eq}} & {\Db_{n+1}}
	\arrow[from=1-1, to=2-1]
	\arrow["j", from=1-1, to=1-2]
	\arrow[from=2-1, to=2-2]
	\arrow[from=1-2, to=2-2]
	\arrow["\lrcorner"{anchor=center, pos=0.125}, draw=none, from=1-1, to=2-2]
\end{tikzcd}\]
This induces two squares
\[\begin{tikzcd}
	{i^*X} & {i^*Y} & {i_!i^*X} & {i_!i^*Y} \\
	{[\Sigma^{n}E^{eq},1]} & {[\Db_n,1]} & X & Y
	\arrow[from=2-3, to=2-4]
	\arrow[from=1-3, to=2-3]
	\arrow[from=1-4, to=2-4]
	\arrow["{i_!i^*j}", from=1-3, to=1-4]
	\arrow[from=1-1, to=2-1]
	\arrow[from=1-2, to=2-2]
	\arrow[from=2-1, to=2-2]
	\arrow["{i^*j}", from=1-1, to=1-2]
	\arrow["\lrcorner"{anchor=center, pos=0.125}, draw=none, from=1-1, to=2-2]
\end{tikzcd}\]
Where the cartesianess of the left square comes from the fact that $i^*$ preserves cartesian squares as it is a right adjoint. We just have demonstrated that $i^*j$ is in $\widehat{\M}$. Using proposition \ref{prop:infini changing theta}, and by left cancellation, the right square implies that $j$ is in $\widehat{W}$, which concludes the proof.
\end{proof}

\begin{prop}
\label{prop:pulback of Wsat}
Let $p:C\to D$ be a functor between $\io$-categories. Then for any globular sums $a$, and any cartesian squares in $\iPsh{\Theta}$:
\[\begin{tikzcd}
	{C''} & {C'} & C \\
	{\Sigma^nE^{eq}} & {\Db_n} & D
	\arrow["p", from=1-3, to=2-3]
	\arrow["j", from=1-1, to=1-2]
	\arrow[from=1-2, to=1-3]
	\arrow[from=2-1, to=2-2]
	\arrow[from=2-2, to=2-3]
	\arrow["\lrcorner"{anchor=center, pos=0.125}, draw=none, from=1-2, to=2-3]
	\arrow["\lrcorner"{anchor=center, pos=0.125}, draw=none, from=1-1, to=2-2]
	\arrow[from=1-1, to=2-1]
	\arrow[from=1-2, to=2-2]
\end{tikzcd}\]
the morphism $j$ is in $\widehat{\W}$.
\end{prop}
\begin{proof}
This is a direct consequence of lemma \ref{lemma:pulback of Wsat preresult}.
\end{proof}
\begin{theorem}
\label{theo:pullback along conduche preserves colimits}
Let $f:C\to D$ be a discrete Conduché functor. The pullback functor $f^*:\ocat_{/D}\to \ocat_{/C}$ preserves colimits.
\end{theorem}
\begin{proof}
As $\iPsh{\Theta}$ is locally cartesian closed, we can use the corollary \ref{cor:derived colimit preserving functor}. The hypotheses are provided by lemmas \ref{lemma:conduche preserves W} and proposition \ref{prop:pulback of Wsat}.
\end{proof}

\section{Gray Operations}
\subsection{Gray operations on $\io$-categories}

Theorem \ref{theo:lecorozo} states that the $(\infty,1)$-category $\ocat$ is represented by the model category of marked simplicial sets given in proposition \ref{prop:model structure on marked simplicial set} and the functor $\N:\zocat\to \ocat$ corresponds to the Street nerve $\N:\ocat\to \mSset$.

An important feature of this model category is that it admits a monoidal structure $\otimes$ given by the \snotionsym{Gray tensor product}{((d00@$\otimes$}{for $\io$-categories}. Furthermore, proposition \ref{prop:gray_product_is_a_left_Quillen_bifunctor} ensures that this operation commutes with colimits in both variables. 
The induced functor 
$$\uvar\otimes[1]:\ocat\to \ocat$$
is called the \snotionsym{Gray cylinder}{((d30@$\uvar\otimes[1]$}{for $\io$-categories}.
We will show later, in corollary \ref{cor:otimes et op}, that we have a natural diagram
\[\begin{tikzcd}
	{(C\otimes\{1\})^\circ} & {(C\otimes[1])^\circ} & {(C\otimes\{0\})^\circ} \\
	{C^\circ\otimes\{0\}} & {C^\circ\otimes[1]} & {C^\circ\otimes\{1\}}
	\arrow["\sim", from=1-2, to=2-2]
	\arrow["\sim", from=1-3, to=2-3]
	\arrow["\sim", from=1-1, to=2-1]
	\arrow[from=1-3, to=1-2]
	\arrow[from=1-1, to=1-2]
	\arrow[from=2-1, to=2-2]
	\arrow[from=2-3, to=2-2]
\end{tikzcd}\]

We denote by 
$$\begin{array}{rcl}
\ocat&\to&\ocat\\
C&\mapsto &C^{[1]}
\end{array}$$
the right adjoint of the Gray cylinder.\sym{(c@$C^{[1]}$}

Eventually, recall that we have a natural transformation $C\otimes [1]\to [C,1]$ whose restriction to $C\otimes\{0\}$ (resp. to $C\otimes\{1\}$) is constant on $\{0\}$ (resp. on $\{1\}$), and such that the following induced square is cocartesian:
\begin{equation}
\label{eq:liens entre Gray cylindre et suspension}
\begin{tikzcd}
	{C\otimes\{0,1\}} & {C\otimes [1]} \\
	{1\amalg 1} & {[C,1]}
	\arrow[from=1-1, to=2-1]
	\arrow[from=1-1, to=1-2]
	\arrow["\lrcorner"{anchor=center, pos=0.125, rotate=180}, draw=none, from=2-2, to=1-1]
	\arrow[from=2-1, to=2-2]
	\arrow[from=1-2, to=2-2]
\end{tikzcd}
\end{equation}

\p We define the \snotionsym{Gray cone}{((d40@$\uvar\star 1$}{for $\io$-categories} and the \snotion{Gray $\circ$-cone}{for $\io$-categories}\index[notation]{((d50@$1\overset{co}{\star}\_$!\textit{for $\io$-categories}}:
$$\begin{array}{ccccccc}
\ocat &\to&\ocat_{\bullet}&&\ocat &\to&\ocat_{\bullet}\\
C&\mapsto &C\star 1 & &C &\mapsto &1\costar C
\end{array}
$$
where $C\star 1$ and $1\costar C$ are defined as the following pushout: 
\[\begin{tikzcd}
	{C\otimes\{1\}} & {C\otimes [1]} & {C\otimes\{0\}} & {C\otimes [1]} \\
	1 & {C\star 1} & 1 & {1\costar C}
	\arrow[from=1-1, to=2-1]
	\arrow[from=1-1, to=1-2]
	\arrow[from=2-1, to=2-2]
	\arrow[from=1-2, to=2-2]
	\arrow["\lrcorner"{anchor=center, pos=0.125, rotate=180}, draw=none, from=2-2, to=1-1]
	\arrow[from=1-3, to=2-3]
	\arrow[from=2-3, to=2-4]
	\arrow[from=1-3, to=1-4]
	\arrow[from=1-4, to=2-4]
	\arrow["\lrcorner"{anchor=center, pos=0.125, rotate=180}, draw=none, from=2-4, to=1-3]
\end{tikzcd}\]
The corollary \ref{cor:otimes et op} will imply an
invertible natural transformation
$$ C\star 1\sim (1\costar C^{\circ})^\circ.$$

We will denote by 
$$\begin{array}{ccccccc}
\ocat_{\bullet} &\to& \ocat&&\ocat_{\bullet} &\to& \ocat\\
(C,c)&\mapsto &C_{/c} & &(C,c) &\mapsto &C_{c/}
\end{array}
$$
the right adjoints of the Gray cone and the Gray $\circ$-cone, respectively called the \wcsnotionsym{slice of $C$ over $c$}{(cc@$C_{/c}$}{slice over}{for $\io$-categories} and the \wcsnotionsym{slice of $C$ under $c$}{(cc@$C_{c/}$}{slice under}{for $\io$-categories}. 
The corollary \ref{cor:otimes et op} will imply an
invertible natural transformation
$$C_{/c}\sim (C^{\circ}_{c/})^\circ.$$
Given an $\io$-category $C$, and two objects $c,d$, we have by construction two cartesian squares:
\[\begin{tikzcd}
	{\hom_C(c,d)} & {C_{/d}} & {\hom_C(c,d)} & {C_{c/}} \\
	{\{c\}} & C & {\{d\}} & C
	\arrow[from=2-1, to=2-2]
	\arrow[from=1-1, to=2-1]
	\arrow[from=1-1, to=1-2]
	\arrow[from=1-2, to=2-2]
	\arrow[from=1-3, to=2-3]
	\arrow[from=1-3, to=1-4]
	\arrow[from=2-3, to=2-4]
	\arrow[from=1-4, to=2-4]
\end{tikzcd}\]

\p 
As explained in section \ref{section:Street nerve}, the functor $\pi_0$ induces canonical equivalences
$$\pi_0(C\otimes[1])\cong \pi_0(C)\otimes[1]~~~\pi_0(C\star 1)\cong \pi_0(C)\star 1~~~\pi_0(1\costar C)\cong 1\costar \pi_0(C)$$
natural in $C$.
We will show in theorem \ref{theo:strictness} that the nerve $\N:\zocat\to \ocat$ also preserves the Gray operations.
As a consequence, we obtain the following examples of Gray operations:

\begin{example}
The $\io$-category $\Db_1\otimes[1]$ corresponds to the polygraph
\[\begin{tikzcd}
	00 & 01 \\
	10 & 11
	\arrow[from=1-1, to=2-1]
	\arrow[from=2-1, to=2-2]
	\arrow[from=1-1, to=1-2]
	\arrow[from=1-2, to=2-2]
	\arrow[shorten <=4pt, shorten >=4pt, Rightarrow, from=1-2, to=2-1]
\end{tikzcd}\]
The $\io$-category $\Db_2\otimes[1]$ corresponds to the polygraph
\[\begin{tikzcd}
	00 & 01 & 00 & 01 \\
	10 & 11 & 10 & 11
	\arrow[from=1-1, to=1-2]
	\arrow[""{name=0, anchor=center, inner sep=0}, from=1-1, to=2-1]
	\arrow[from=2-1, to=2-2]
	\arrow[""{name=1, anchor=center, inner sep=0}, from=1-2, to=2-2]
	\arrow[shorten <=4pt, shorten >=4pt, Rightarrow, from=1-2, to=2-1]
	\arrow[""{name=2, anchor=center, inner sep=0}, from=1-3, to=2-3]
	\arrow[from=1-3, to=1-4]
	\arrow[""{name=3, anchor=center, inner sep=0}, from=1-4, to=2-4]
	\arrow[shorten <=4pt, shorten >=4pt, Rightarrow, from=1-4, to=2-3]
	\arrow[""{name=4, anchor=center, inner sep=0}, curve={height=30pt}, from=1-1, to=2-1]
	\arrow[from=2-3, to=2-4]
	\arrow[""{name=5, anchor=center, inner sep=0}, curve={height=-30pt}, from=1-4, to=2-4]
	\arrow["{ }"', shorten <=6pt, shorten >=6pt, Rightarrow, from=0, to=4]
	\arrow["{ }"', shorten <=6pt, shorten >=6pt, Rightarrow, from=5, to=3]
	\arrow[shift left=0.7, shorten <=6pt, shorten >=8pt, no head, from=1, to=2]
	\arrow[shift right=0.7, shorten <=6pt, shorten >=8pt, no head, from=1, to=2]
	\arrow[shorten <=6pt, shorten >=6pt, from=1, to=2]
\end{tikzcd}\]
\end{example}
\begin{example}
The $\io$-categories $\Db_1\star 1$ and $1\costar \Db_1$ correspond respectively to the polygraphs: 
\[\begin{tikzcd}
	0 &&&& 0 \\
	1 & \star && \star & 1
	\arrow[from=1-1, to=2-1]
	\arrow[from=2-1, to=2-2]
	\arrow[""{name=0, anchor=center, inner sep=0}, from=1-1, to=2-2]
	\arrow[""{name=1, anchor=center, inner sep=0}, from=1-5, to=2-5]
	\arrow[from=2-4, to=1-5]
	\arrow[""{name=2, anchor=center, inner sep=0}, from=2-4, to=2-5]
	\arrow[shorten <=2pt, Rightarrow, from=0, to=2-1]
	\arrow[shift right=2, shorten <=4pt, shorten >=4pt, Rightarrow, from=1, to=2]
\end{tikzcd}\]
The $\io$-categories $\Db_2\star 1$ and $1\costar \Db_2$ correspond respectively to the polygraphs: 
\[\begin{tikzcd}
	0 & {~} & 0 &&& 0 & {~} & 0 \\
	1 & \star & 1 & \star & \star & 1 & \star & 1
	\arrow[""{name=0, anchor=center, inner sep=0}, from=1-1, to=2-1]
	\arrow[from=2-1, to=2-2]
	\arrow[""{name=1, anchor=center, inner sep=0}, from=1-3, to=2-3]
	\arrow[""{name=2, anchor=center, inner sep=0}, curve={height=30pt}, from=1-1, to=2-1]
	\arrow[from=2-3, to=2-4]
	\arrow[""{name=3, anchor=center, inner sep=0}, from=1-1, to=2-2]
	\arrow[""{name=4, anchor=center, inner sep=0}, draw=none, from=1-2, to=2-2]
	\arrow[""{name=5, anchor=center, inner sep=0}, from=1-3, to=2-4]
	\arrow[from=1-6, to=2-5]
	\arrow[""{name=6, anchor=center, inner sep=0}, from=1-6, to=2-6]
	\arrow[""{name=7, anchor=center, inner sep=0}, from=2-5, to=2-6]
	\arrow[from=1-8, to=2-7]
	\arrow[""{name=8, anchor=center, inner sep=0}, from=1-8, to=2-8]
	\arrow[""{name=9, anchor=center, inner sep=0}, from=2-8, to=2-7]
	\arrow[""{name=10, anchor=center, inner sep=0}, curve={height=-30pt}, from=1-8, to=2-8]
	\arrow[""{name=11, anchor=center, inner sep=0}, draw=none, from=1-7, to=2-7]
	\arrow["{ }"', shorten <=6pt, shorten >=6pt, Rightarrow, from=0, to=2]
	\arrow[shorten <=2pt, shorten >=2pt, Rightarrow, from=3, to=2-1]
	\arrow[shift left=0.7, shorten <=6pt, shorten >=8pt, no head, from=4, to=1]
	\arrow[shift right=0.7, shorten <=6pt, shorten >=8pt, no head, from=4, to=1]
	\arrow[shorten <=6pt, shorten >=6pt, from=4, to=1]
	\arrow[shorten <=2pt, Rightarrow, from=5, to=2-3]
	\arrow[shorten <=6pt, shorten >=6pt, Rightarrow, from=10, to=8]
	\arrow[shift right=2, shorten <=4pt, shorten >=4pt, Rightarrow, from=8, to=9]
	\arrow[shift right=2, shorten <=4pt, shorten >=4pt, Rightarrow, from=6, to=7]
	\arrow[shift right=0.7, shorten <=6pt, shorten >=8pt, no head, from=6, to=11]
	\arrow[shorten <=6pt, shorten >=6pt, from=6, to=11]
	\arrow[shift left=0.7, shorten <=6pt, shorten >=8pt, no head, from=6, to=11]
\end{tikzcd}\]
\end{example}

\p 
\label{paragrap: equation fullfill by cylinder and join}
In section \ref{section:Suspension and Gray operation} are shown several equations fulfilled by the Gray cylinder, the Gray cone, and the Gray $\circ$-cone, that we recall here. For every $\io$-category $C$, there is a natural identification between $[C,1]\otimes [1]$ and the colimit of the following diagram
\begin{equation}
\label{eq:eq for cylinder}
\begin{tikzcd}
	{[1]\vee [ C,1]} & {[C\otimes\{0\},1]} & {[C\otimes [1],1]} & {[C\otimes\{1\},1]} & {[C,1]\vee[1]}
	\arrow[from=1-2, to=1-1]
	\arrow[from=1-2, to=1-3]
	\arrow[from=1-4, to=1-3]
	\arrow[from=1-4, to=1-5]
\end{tikzcd}
\end{equation}
 There is also a natural identification between
 $1\costar [C,1]$ and the colimit of the diagram
\begin{equation}
\label{eq:eq for Gray cone}
\begin{tikzcd}
	{[1]\vee [C,1]} & {[C,1]} & {[C\star 1,1]}
	\arrow[from=1-2, to=1-3]
	\arrow[from=1-2, to=1-1]
\end{tikzcd}
\end{equation}
and $[C,1] \star 1$ and the colimit of the diagram
\begin{equation}
\label{eq:eq for cojoin}
\begin{tikzcd}
	{[1\costar C,1]} & {[C,1]} & {[C,1]\vee[1]}
	\arrow[from=1-2, to=1-3]
	\arrow[from=1-2, to=1-1]
\end{tikzcd}
\end{equation}
In each of the three previous diagrams, morphisms $[C,1]\to [1]\vee[C,1]$ and $[C,1]\to [C,1]\vee[1]$ are the unique ones preserving extremal points.

\begin{remark}
It is worth noticing the great similarity of these equations with the one given in theorems \ref{theo:appendice formula for otimes} and \ref{theo:appendice formula for star}
\end{remark}

\p
Let $C$ be an $\io$-category and $K$ a $(\infty,1)$-category.
There is a canonical morphism $C\otimes K\to C\times K$. In a way, one can see $C\times K$ as an intelligent truncated version of the Gray tensor product $C\otimes K$. We will make this intuition precise by constructing a hierarchy of Gray tensor products with $(\infty,1)$-categories. 
For $k\in \Nb\cup\{\omega\}$, we define the functor
$$\begin{array}{ccl}
\ocat \times \ncat{1}&\to &\ocat\\
(C,K)&\mapsto &C\otimes_k K
\end{array}$$
where $C\otimes_kK$ fits in the cocartesian square
\[\begin{tikzcd}
	{\colim_{n\geq k}(\tau_nC)\otimes K} & {C\otimes K} \\
	{\colim_{n\geq k}\tau^i_n((\tau_nC)\otimes K)} & {C\otimes_{k} K}
	\arrow[from=1-1, to=2-1]
	\arrow[from=1-1, to=1-2]
	\arrow[from=1-2, to=2-2]
	\arrow[from=2-1, to=2-2]
	\arrow["\lrcorner"{anchor=center, pos=0.125, rotate=180}, draw=none, from=2-2, to=1-1]
\end{tikzcd}\]

The induced functors $\uvar\otimes_k[1]:\ocat\to\ocat$ are called the \wcnotionsym{$k$-Gray cylinder}{((d10@$\otimes_n$}{Gray cylindera@$n$-Gray cylinder}.
Formula \eqref{eq:eq for cylinder} implies that for every $\io$-category $C$,
there is a natural identification between $[C,1]\otimes_{k+1} [1]$ and the colimit of the following diagram
\begin{equation}
\label{eq:eq for k cylinder}
\begin{tikzcd}
	{[1]\vee [ C,1]} & {[C\otimes_k\{0\},1]} & {[C\otimes_k [1],1]} & {[C\otimes_k\{1\},1]} & {[C,1]\vee[1]}
	\arrow[from=1-2, to=1-1]
	\arrow[from=1-2, to=1-3]
	\arrow[from=1-4, to=1-3]
	\arrow[from=1-4, to=1-5]
\end{tikzcd}
\end{equation}
Remark that the endofunctor $\uvar\otimes_0[1]$ is the identity, 
the first assertion of lemma \ref{lemma:technique marked oicategoros} implies that the endofunctor $\uvar\otimes_1[1]$ is equivalent to $\uvar\times [1]$, and the endofunctor $\otimes_{\omega}[1]$ is just the normal Gray cylinder.

\begin{prop}
\label{prop:otimesk preserves colimits}
For any integer $k>0$, $\uvar\otimes_k[1]$ preserves colimits. 
\end{prop}
\begin{proof}
In order to simplify the notation, for a functor $F:\ocat\to \ocat$, the $\infty$-presheaves $\colim_{\Theta_{/\Sigma^nE^{eq}}}\iota F$, where $\iota$ in the inclusion $\ocat\to \iPsh{\Theta}$, will just be denoted by $F(\Sigma^nE^{eq})$.

As $\tau$ and $\tau^i$ preserves colimits in $\iPsh{\Theta}$ and $\widehat{\Wseg}$, and as $\uvar\otimes[1]$ preserves colimits, we just have to show that for any $n$, $(\Sigma^nE^{eq})\otimes_k[1]\to (\Sigma^n1)\otimes_k[1]$ is in $\widehat{\W}$. 

We then proceed by induction on $k$. The cases $k=0$ and $k=1$ are trivial as $\uvar\otimes_0[1]$ is the identity and $\uvar\otimes_1[1]$ is the tensor product with $[1]$.

Suppose the result is true at the stage $k$ for $k>1$. If $n=0$, remark that $E^{eq}\otimes_k[1]$ (resp. $ 1\otimes_k[1]$) is equivalent to $E^{eq}\otimes[1]$ (resp. $ 1\otimes[1]$) and the morphism is then in $\widehat{\W}$. Now, if $n>0$, formula \eqref{eq:eq for k cylinder} implies that $(\Sigma^nE^{eq})\otimes_k[1]\to (\Sigma^n1)\otimes_k[1]$ is the colimit in depth of the following diagram:
\[\begin{tikzcd}[column sep=0.2cm]
	{[ \Sigma^{n-1}E^{eq}\otimes_{k-1}\{0\},1]} && {[ \Sigma^{n-1}E^{eq}\otimes_{k-1}\{1\},1]} \\
	{[1]\vee [ \Sigma^{n-1}E^{eq},1]} & {[ \Sigma^{n-1}E^{eq}\otimes_{k-1}[1],1]} & {[ \Sigma^{n-1}E^{eq},1]\vee[1]} \\
	& {[ \Sigma^{n-1}1\otimes_{k-1}\{0\},1]} && {[ \Sigma^{n-1}1\otimes_{k-1}\{1\},1]} \\
	& {[1]\vee [ \Sigma^{n-1}1,1]} & {[ \Sigma^{n-1}1\otimes_{k-1}[1],1]} & {[ \Sigma^{n-1}1,1]\vee[1]}
	\arrow[from=1-1, to=2-1]
	\arrow[from=1-1, to=2-2]
	\arrow[from=1-3, to=2-2]
	\arrow[from=1-3, to=2-3]
	\arrow[from=3-2, to=4-2]
	\arrow[from=3-2, to=4-3]
	\arrow[from=3-4, to=4-3]
	\arrow[from=3-4, to=4-4]
	\arrow[from=1-1, to=3-2]
	\arrow[from=2-2, to=4-3]
	\arrow[from=1-3, to=3-4]
	\arrow[from=2-1, to=4-2]
	\arrow[from=2-3, to=4-4]
\end{tikzcd}\]
by induction hypothesis, and using lemma \ref{lemma:the functor [] preserves classes}, all the morphisms in depth are in $\widehat{\W}$, and so is their colimit.
\end{proof}
The functor $\uvar\otimes[1]_k$ then admits a right adjoint
$$(\uvar)^{[1]_k}:\ocat\to \ocat.$$

\p
We now describe a last operation that will play an essential role in the definition of lax colimit and lax limit. For any $C:\ocat$, we denote by \wcnotation{$m_C$}{(mc@$m_C$} the colimit preserving functor 
$\ocat\to\ocat$ whose value on a representable $[a,n]$ is $[a\times C,n]$. Remark that the assignation $C\mapsto m_C$ is natural in $C$ and that $m_1$ is the identity.
We define the colimit preserving functor: \ssym{((d20@$\ominus$}{for $\io$-categories}
$$\begin{array}{ccc}
\ocat\times\ocat &\to&\ocat\\
(X,Y)&\mapsto &X\ominus Y
\end{array}
$$
where for any $\io$-category $C$ and any element $[b,n]$ of $\Delta[\Theta]$, $X\ominus [b,n]$ is the following pushout: 
\[\begin{tikzcd}
	{\coprod\limits_{k\leq n}m_b(C\otimes\{k\})} & {m_b(C\otimes[n])} \\
	{\coprod\limits_{k\leq n}m_1(C\otimes\{k\})} & {C\ominus[b,n]}
	\arrow[from=1-1, to=2-1]
	\arrow[""{name=0, anchor=center, inner sep=0}, from=1-1, to=1-2]
	\arrow[from=1-2, to=2-2]
	\arrow[from=2-1, to=2-2]
	\arrow["\lrcorner"{anchor=center, pos=0.125, rotate=180}, draw=none, from=2-2, to=0]
\end{tikzcd}\]
By construction, the functor $\uvar\ominus \uvar$ commutes with colimits in both variables.
We also have the identification $C\ominus [1]:=C\otimes [1]$.

Eventually, formula \eqref{eq:eq for cylinder} induces a natural identification between $[C,1]\ominus[b,1]$ and the colimit of the following diagram
\begin{equation}
\label{eq:formula for the ominus}
\begin{tikzcd}[column sep = 0.3cm]
	{[b,1]\vee[C,1]} & {[C\otimes\{0\}\times b,1]} & {[(C\otimes[1])\times b),1]} & {[C\otimes\{1\}\times b,1]} & {[C,1]\vee[b,1]}
	\arrow[from=1-2, to=1-3]
	\arrow[from=1-4, to=1-3]
	\arrow[from=1-4, to=1-5]
	\arrow[from=1-2, to=1-1]
\end{tikzcd}
\end{equation}

\subsection{Gray deformation retract}
\label{subsection:Gray deformation retract}

\p
 A \wcnotion{left $k$-Gray deformation retract structure}{left or right $k$-Gray deformation retract structure} for a morphism $i:C\to D$ is the data of a \textit{retract}
 $r:D\to C$, a \textit{deformation} $\psi:D\otimes_k [1]\to D$, and equivalences
$$ri\sim id_C~~~~~\psi_{|D\otimes_k\{0\}}\sim ir~~~~~\psi_{|D\otimes_k\{1\}}\sim id_D~~~~~ \psi_{|C\otimes_k[1]}\sim i\cst_C
$$ 
A morphism $i:C\to D$ between $\io$-categories is a \wcnotion{left $k$-Gray deformation retract}{left or right $k$-Gray deformation retract} if it admits a left deformation retract structure. By abuse of language, such data will just be denoted by $(i,r,\psi)$.

We define dually the notion of \textit{right $k$-Gray deformation retract structure} and of \textit{right $k$-Gray deformation retract} in exchanging $0$ and $1$ in the previous definition.

\p
 A \textit{left $k$-Gray deformation retract structure} for a morphism $i:f\to g$ in the $\iun$-category of arrows of $\ocat$ is the data of a \textit{retract}
 $r:g\to f$, a \textit{deformation} $\psi:g\otimes_k [1]\to g$ and equivalences
$$ri\sim id_f~~~~~\psi_{|g\otimes_k\{0\}}\sim ir~~~~~\psi_{|g\otimes_k\{1\}}\sim id_D~~~~~ \psi_{|f\otimes_k[1]}\sim i\cst_C
$$ 
A morphism $i:C\to D$ between arrows of $\ocat$ is a \textit{left $k$-Gray deformation retract} if it admits a left deformation retract structure. By abuse of language, such data will just be denoted by $(i,r,\psi)$.

We define dually the notion of \textit{right $k$-Gray deformation retract structure} and of \textit{right $k$-Gray deformation retract} in exchanging $0$ and $1$ in the previous definition.

\begin{example}
\label{example:canonical example of left deformation retract unmarked}
Let $k\in \Nb\cup\{\omega\}$ and 
let $C$ be an $(\infty,k)$-category. We consider the morphism $i:C\otimes\{0\}\to C\otimes[1]$. We define $r:C\otimes[1]\xrightarrow{C\otimes\Ib} C\otimes\{0\}$. Eventually, we set 
$$\psi:C\otimes[1]\otimes[1]\to C\otimes([1]\times [1])\xrightarrow{C\otimes \phi}C\otimes[1]$$
where $\phi:[1]\times[1]$ is the morphism sending $(i,j)$ on the minimum of $i$ and $j$. 

As $C$ is an $(\infty,k)$-category, $\psi$ factors through $C\otimes[1]\to \tau^i_k(C\otimes[1])\sim C\otimes_k[1]$. We denote by $\phi:C\otimes_k[1]\to C\otimes\{0\}$ the induced morphism.
The triple $(i,r,\phi)$ is a left $k$-Gray deformation retract structure. Conversely, 
 $C\otimes\{1\}\to C\otimes[1]$ is a right deformation retract.

One can show similarly that $1\to 1\costar C$ is a left $k$-Gray deformation retract, and $1\to C\star 1$ is a right $k$-Gray deformation retract.
\end{example}

\p The $\infty$-groupoid of left and right Gray retracts enjoys many stability properties: 
\begin{prop}
\label{prop:left Gray deformation retract stable under pushout unmarked}
Let $(i_a,r_a,\psi_a)$ be a natural familly of left (resp. right) $k$-Gray deformation retract structures indexed by an $(\infty,1)$-category $A$.
The triple $(\colim_{A}i_a,\colim_{A}r_a,\colim_{A}\psi_a)$ is a left (resp. right) $k$-Gray deformation retract structure.
\end{prop}
\begin{proof}
This is an immediate consequence of the fact that $\uvar\otimes_k[1]$ preserves colimits.
\end{proof}
\begin{prop}
\label{prop:stability under pullback unmarked}
Suppose that we have a diagram 
\[\begin{tikzcd}
	X & Y & Z \\
	X & {Y'} & {Z'}
	\arrow[from=1-1, to=2-1]
	\arrow[from=1-2, to=2-2]
	\arrow[from=1-3, to=2-3]
	\arrow["p", from=1-1, to=1-2]
	\arrow["q"', from=1-3, to=1-2]
	\arrow["{p'}"', from=2-1, to=2-2]
	\arrow["{q'}", from=2-3, to=2-2]
\end{tikzcd}\]
such that $p\to p'$ and $q\to q'$ are left (resp. right) $k$-Gray deformation retract. The induced square $q^*p\to (q')^*p'$ is a left (resp. right) $k$-Gray deformation retract.
\end{prop}
\begin{proof}
The proof is an easy diagram chasing.
\end{proof}
\begin{prop}
\label{prop:stability by composition unmarked}
If $p\to p'$ and $p'\to p''$ are two left (resp. right) $k$-Gray deformation retracts, so is $p\to p''$.
\end{prop}
\begin{proof}
The proof is an easy diagram chasing.
\end{proof}

\p The two following propositions show how the shifting of dimension preserves Gray transformation retract.

\begin{prop}
\label{prop:Gray deformation retract and passage to hom unmarked}
Let $(i:C\to D,r,\psi)$ be a left (resp. right) $(k+1)$-Gray deformation structure. For any $x: C$ and $y:D$ (resp. $x: D$ and $y:C$), the morphism
$$\begin{array}{cc}
&\hom_C(x,ry)\xrightarrow{i} \hom_D(ix,iry)\xrightarrow{{\psi_y}_!} \hom_D(ix,y)\\
(resp. &\hom_C(rx,y)\xrightarrow{i} \hom_D(irx,iy)\xrightarrow{{\psi_x}_!} \hom_D(x,iy))
\end{array}
$$
is a right (resp. left) $k$-Gray deformation retract, whose retract is given by 
$$\begin{array}{cc}
&\hom_D(ix,y)\xrightarrow{r}\hom_C(x,ry)\\
(resp. &\hom_D(x,iy)\xrightarrow{r}\hom_C(rx,y))
\end{array}$$
\end{prop}
\begin{proof}
By currying $\psi$, this induces a diagram
\[\begin{tikzcd}
	& C & D \\
	D & {D^{[1]_{k+1}}} \\
	&& D
	\arrow["r", curve={height=-6pt}, from=2-1, to=1-2]
	\arrow["i", from=1-2, to=1-3]
	\arrow["id"', curve={height=12pt}, from=2-1, to=3-3]
	\arrow[from=2-2, to=3-3]
	\arrow[from=2-2, to=1-3]
	\arrow["\psi"{description}, from=2-1, to=2-2]
\end{tikzcd}\]
For any pair of objects $(z,y)$ of $D$, according to formula \eqref{eq:eq for k cylinder}, this induces a diagram
\[\begin{tikzcd}
	& {\hom_D(z,y)} \\
	{\hom_D(z,y)} & {\hom_D(irz,iry)\times_{\hom_D(irz,y)}\hom_D(irz,y)^{[1]_k}\times_{\hom_D(irz,y)} \hom_D(z,y)} \\
	{\hom_C(rz,ry)} & {\hom_D(irz,iry)}
	\arrow["r", from=2-1, to=3-1]
	\arrow["i", from=3-1, to=3-2]
	\arrow["id", curve={height=-12pt}, from=2-1, to=1-2]
	\arrow[from=2-2, to=1-2]
	\arrow[from=2-2, to=3-2]
	\arrow["\psi"{description}, from=2-1, to=2-2]
\end{tikzcd}\]
If $z$ is of shape $ix$, the diagram becomes
\[\begin{tikzcd}
	& {\hom_D(ix,y)} & {\hom_D(ix,y)} \\
	{\hom_D(ix,y)} & {\hom_D(ix,iry)\times_{\hom_D(ix,y)}\hom_D(ix,y)^{[1]_k}} & {\hom_D(ix,y)^{[1]_k}} \\
	{\hom_C(x,ry)} & {\hom_D(ix,iry)} & {\hom_D(ix,y)}
	\arrow["r"', from=2-1, to=3-1]
	\arrow["id", curve={height=-12pt}, from=2-1, to=1-2]
	\arrow[from=2-2, to=1-2]
	\arrow["\psi"{description}, from=2-1, to=2-2]
	\arrow["i"', from=3-1, to=3-2]
	\arrow[from=2-2, to=3-2]
	\arrow[""{name=0, anchor=center, inner sep=0}, "{{\psi_y}_!}"', from=3-2, to=3-3]
	\arrow[from=2-2, to=2-3]
	\arrow["id", from=1-2, to=1-3]
	\arrow[from=2-3, to=3-3]
	\arrow[from=2-3, to=1-3]
	\arrow["\lrcorner"{anchor=center, pos=0.125}, draw=none, from=2-2, to=0]
\end{tikzcd}\]
By decurrying, this induces a morphism $\phi:\hom_D(ix,y)\otimes_k[1]\to \hom_D(ix,y)$. We leave the reader verify that the triple $({\psi_y}_!i,r,\phi)$ is a right $k$-Gray deformation retract structure.
We proceed similarly for the other case.
\end{proof}

\begin{prop}
\label{prop:Gray deformation retract and passage to hom v2 unmarked}
For any left (resp. right) $(k+1)$-Gray deformation retract between $p$ and $p'$:
\[\begin{tikzcd}
	C & D \\
	{C'} & {D'}
	\arrow["p"', from=1-1, to=2-1]
	\arrow["i", from=1-1, to=1-2]
	\arrow["{p'}", from=1-2, to=2-2]
	\arrow["{i'}"', from=2-1, to=2-2]
\end{tikzcd}\]
and for any pair of objects $x: C$ and $y:D$ (resp. $x: D$ and $y:C$), the outer square of the following diagram
\[\begin{tikzcd}
	{\hom_{C}(x,ry)} & {\hom_{D}(ix,iry)} & {\hom_{D}(ix,y)} \\
	{\hom_{C'}(px,pr'y)} & {\hom_{D'}(p'i'x,p'i'r'y)} & {\hom_{D'}(p'i'x,p'y)}
	\arrow["{i'}"', from=2-1, to=2-2]
	\arrow["{{\psi'_{p'y}}_!}"', from=2-2, to=2-3]
	\arrow[from=1-1, to=2-1]
	\arrow[from=1-3, to=2-3]
	\arrow["{{\psi_y}_!}", from=1-2, to=1-3]
	\arrow["i", from=1-1, to=1-2]
	\arrow[from=1-2, to=2-2]
\end{tikzcd}\]
(resp.
\[\begin{tikzcd}
	{\hom_{C}(rx,y)} & {\hom_{D}(irx,iy)} & {\hom_{D}(x,iy)} \\
	{\hom_{C'}(pr'x,py)} & {\hom_{D'}(p'i'r'x,p'i'y)} & {\hom_{D'}(p'x,p'i'y)\big)}
	\arrow["{i'}"', from=2-1, to=2-2]
	\arrow["{{\psi'_{p'x}}_!}"', from=2-2, to=2-3]
	\arrow[from=1-1, to=2-1]
	\arrow[from=1-3, to=2-3]
	\arrow["{{\psi_x}_!}", from=1-2, to=1-3]
	\arrow["i", from=1-1, to=1-2]
	\arrow[from=1-2, to=2-2]
\end{tikzcd}\]
is a left (resp. right) $(k+1)$-Gray deformation retract, whose retract is given by
\[\begin{tikzcd}
	{\hom_{D}(ix,y)} & {\hom_{C}(x,ry)} & {(resp.\hom_{D}(x,iy)} & {\hom_{C}(rx,y)} \\
	{\hom_{D'}(p'i'x,p'y)\big)} & {\hom_{C'}(px,pr'y)} & {\hom_{D'}(p'x,p'i'y)} & {\hom_{C'}(pr'x,py)\big)}
	\arrow[from=1-3, to=2-3]
	\arrow["r", from=1-3, to=1-4]
	\arrow["{r'}"', from=2-3, to=2-4]
	\arrow[from=1-4, to=2-4]
	\arrow["{r'}"', from=2-1, to=2-2]
	\arrow["r", from=1-1, to=1-2]
	\arrow[from=1-2, to=2-2]
	\arrow[from=1-1, to=2-1]
\end{tikzcd}\]
\end{prop}
\begin{proof}
This comes from the fact that the construction of the retraction and the deformation in the previous proposition was functorial.
\end{proof}

\begin{prop}
\label{prop:suspension of left Gray deformation retract unmarked}
If $i$ is a left $k$-Gray deformation retract, $[i,1]$ is a right $(k+1)$-Gray deformation retract. Conversely, if $i$ is a right $k$-Gray deformation retract, $[i,1]$ is a left $(k+1)$-Gray deformation retract morphism.
\end{prop}
\begin{proof}
Let $(i:C\to D,r,\phi)$ be a left $k$-Gray deformation retract structure. We define the morphism $\psi:[D,1]\otimes_{k+1}[1]\to [D,1]$ as the horizontal colimit of the following diagram:
\[\begin{tikzcd}
	{[1]^{}\vee[D,1]} & {[D\otimes_k\{0\},1]} & {[D\otimes_k[1],1]} & {[D\otimes_k\{1\},1]} & {[D,1]\vee[1]^{}} \\
	& {[C,1]} & {[D,1]} & {[D,1]}
	\arrow[from=1-4, to=1-5]
	\arrow["{[\phi,1]}"', from=1-3, to=2-3]
	\arrow[from=1-2, to=1-1]
	\arrow[from=1-2, to=1-3]
	\arrow[from=1-4, to=1-3]
	\arrow["{[i,1]}"', from=2-2, to=2-3]
	\arrow["{[r,1]}"', from=1-2, to=2-2]
	\arrow["{[id,1]}", from=2-4, to=2-3]
	\arrow["{[id,1]}", from=1-4, to=2-4]
	\arrow[from=1-1, to=2-2]
	\arrow[from=1-5, to=2-4]
\end{tikzcd}\]
Eventually, remark that the triple $([i,1],[r,1],\psi)$ is a right $(k+1)$-Gray deformation retract. The other assertion is demonstrated similarly.
\end{proof}

\begin{prop}
\label{prop:of left Gray deformation retract unmarked}
For any integer $n$,
if $n$ is even, $i_{n}^-:\Db_{n}\to \Db_{n+1}$ is a left $n$-Gray deformation retract and $i_{n}^+:\Db_{n}\to \Db_{n+1}$ is a right $n$-Gray deformation retract, and if $n$ is odd, $i_{n}^-$ is a right $n$-Gray deformation retract and $i_{n}^+$ is a left $n$-Gray deformation retract.
\end{prop}
\begin{proof}
It is obvious that $\{0\}\to [1]$ is a left $1$-Gray deformation retract and $\{1\}\to [1]$ is a right $1$-Gray deformation retract. A repeated application of \ref{prop:suspension of left Gray deformation retract unmarked} proves the assertion.
\end{proof}

\begin{prop}
\label{prop:when glob inclusion are left Gray deformation unmarked}
Let $a$ be a globular sum of dimension $(n+1)$. We denote by $s_n(a)$ and $t_n(a)$ the globular sum defined in paragraph \ref{para:definition of source et but}. 

If $n$ is even, $s_n(a)\to a$ is a left $n$-Gray deformation retract and $t_n(a)\to a$ is a right $n$-Gray deformation retract, and if $n$ is odd, $s_n(a)\to a$ is a right $n$-Gray deformation retract and $t_n(a)\to a$ is a left $n$-Gray deformation retract.
\end{prop}
\begin{proof}
This is a direct consequence of proposition \ref{prop:of left Gray deformation retract unmarked} and \ref{prop:left Gray deformation retract stable under pushout unmarked} as $s_n(a)\to a$ is a composition of pushouts of $i_{n}^-:\Db_{n}\to (\Db_{n+1})_t$. The other assertion is proved similarly.
\end{proof}

\subsection{Gray operations and strict objects}

\label{section:on preservation of strict}
Recall that we have an adjunction
\[\begin{tikzcd}
	{\pi_0:\ocat} & {\zocat:\N}
	\arrow[""{name=0, anchor=center, inner sep=0}, shift left=2, from=1-1, to=1-2]
	\arrow[""{name=1, anchor=center, inner sep=0}, shift left=2, from=1-2, to=1-1]
	\arrow["\dashv"{anchor=center, rotate=-90}, draw=none, from=0, to=1]
\end{tikzcd}\]
An $\io$-category lying in the image of the nerve functor $\N$ is called strict. As explained in example 11 of \cite{Verity_weak_complicial_set_part2_nerve_of_complicial_Gray_categories}, $\pi_0$ preserves Gray tensor product, and so also the suspension, the Gray cone, and the Gray $\circ$-cone.

The strict categories play an important role as they allow us to make explicit calculations. In particular, it will be very useful to know which cocontinuous functors preserve them.
\begin{prop}
\label{prop:criter stricte easy}
An $\io$-category $C$ is strict if and only if $C_0$ is a set and for any pair of objects $x,y$, $\hom_C(x,y)$ is strict.
\end{prop}
\begin{proof}
By definition, an $\io$-category is strict if and only if, for any globular sum $[\textbf{b},n]$, $\Hom([\textbf{b},n],C)$ is a set. However, as $C$ is $\W$-local, we have an equivalence between $\Hom([\textbf{b},n],C)$ and 
$$\coprod_{x_0,x_1,...,x_n\in C_0}\Hom(b_1, \hom_C(x_0,x_1))\times...\times \Hom(b_n, \hom_C(x_{n-1},x_n))$$
As all the objects of the previous expression are set by hypothesis, and as the inclusion of set into $\infty$-groupoid is stable under coproduct and product, $\Hom([b,n],C)$ is a set.
\end{proof}

\begin{prop}
\label{prop:suspension preserves stricte}
If $C$ is a strict $\io$-category, so is $[C,1]$. 
\end{prop}
\begin{proof}
There is an obvious equivalence $[\N\uvar,1]\sim \N[\uvar,1]$ which directly implies the result.
\end{proof}

\begin{lemma}
\label{lemma:gray operation on globes are strict}
For any $n$, $\Db_n\otimes[1]$, $\Db_n\star 1$ and $1\costar \Db_n$ are strict.
\end{lemma}
\begin{proof}
We proceed by induction on $n$. The result is obviously true for $n=0$. Suppose it is true as the stage $n$. According to equation \eqref{eq:eq for cylinder}, $\Db_n\otimes[1]$ is the colimit of the following diagram
\begin{equation}
\label{eq:tensor of globuees is sitrict}
\begin{tikzcd}
	{[1]\vee \Db_n} & {\Db_n} & {[\Db_{n-1}\otimes [1],1]} & {\Db_n} & {\Db_n\vee[1]}
	\arrow[from=1-2, to=1-1]
	\arrow[from=1-2, to=1-3]
	\arrow[from=1-4, to=1-3]
	\arrow[from=1-4, to=1-5]
\end{tikzcd}
\end{equation}
The induction hypothesis and proposition \ref{prop:suspension preserves stricte} implies that all the objects are strict.
The proposition \ref{prop:cartesian squares} then implies that the diagram
\[\begin{tikzcd}
	{\Db_{n-1}} & {\Db_{n-1}\otimes[1]} & {\Db_{n-1}} \\
	{\{0\}} & {[1]} & {\{1\}}
	\arrow[from=1-1, to=2-1]
	\arrow[from=1-2, to=2-2]
	\arrow[from=1-3, to=2-3]
	\arrow[from=2-1, to=2-2]
	\arrow[from=1-1, to=1-2]
	\arrow[from=1-3, to=1-2]
	\arrow[from=2-3, to=2-2]
\end{tikzcd}\]
verifies the hypothesis of proposition \ref{prop:example of a special colimit3}.
The proposition \textit{op. cit.} then
states that the colimit of \eqref{eq:tensor of globuees is sitrict} is special, which implies, according to lemma \ref{lemma:colimit computed in set presheaves}, that  its colimit, which is $\Db_n\otimes[1]$,  is also strict.

We proceed similarly for the Gray cone and the Gray $\circ$-cone.
\end{proof}

We now recall the following fundamental result of strictification:
\begin{theorem}[Gagna, Ozornova, Rovelli]
\label{theo:join preserves stict VMG version}
For any globular sum $a$, $a\star 1$ and $1\costar a$ are stricts.
\end{theorem}
\begin{proof}
The fact that $a\star 1$ is strict is a particular case of theorem 5.2 of \cite{Gagna_Nerves_and_cones_of_free_loop_free_omega_categories}. For the second assertion, remark that we have a canonical comparison, natural in $a:\Theta$: 
$$1\costar a\to \N\pi_0(1\costar a)\sim \N \pi_0 (a^\circ \star 1)^\circ \sim (\N \pi_0 (a^\circ \star 1))^\circ\sim (a^\circ\star 1)^{\circ}$$
where the first equivalence is a consequence of \cite[proposition A.22]{Ara_Maltsiniotis_joint_et_tranche}, the second comes from the commutativity of $\pi_0$ and $\N$ with dualities, and the last one is the (already demonstrated) first assertion.
The subset of object of $\Theta$ making this comparison an equivalence is closed by colimits and, according to lemma \ref{lemma:gray operation on globes are strict}, contains globes. This subset then contains all the globular sums. As strict objects are stable by dualities, this concludes the proof of the second assertion.
\end{proof}

\begin{lemma}
\label{lemma:strictification2}
Let $\alpha$ be $-$ if $n$ is even (resp. odd) and $+$ if $n$ is odd (resp.even).
Consider a cartesian square
\begin{equation}
\label{eq:lemma:strictification2}
\begin{tikzcd}
	{C_0} & D \\
	{\Db_{n}} & {\Db_{n+1}}
	\arrow["p"', from=1-1, to=2-1]
	\arrow["{p'}", from=1-2, to=2-2]
	\arrow[from=1-1, to=1-2]
	\arrow["\lrcorner"{anchor=center, pos=0.125}, draw=none, from=1-1, to=2-2]
	\arrow["{i_{n}^\alpha}"', from=2-1, to=2-2]
\end{tikzcd}
\end{equation}
such that $p\to p'$ is a left $(n+1)$-Gray deformation retract (resp. a right $(n+1)$-Gray deformation retract). Let $C_1$ be the $\io$-category fitting in the pullback 
\begin{equation}
\label{eq:lemma:strictification3}
\begin{tikzcd}
	{C_1} & D \\
	{\Db_{n}} & {\Db_{n+1}}
	\arrow["p"', from=1-1, to=2-1]
	\arrow["{p'}", from=1-2, to=2-2]
	\arrow[from=1-1, to=1-2]
	\arrow["\lrcorner"{anchor=center, pos=0.125}, draw=none, from=1-1, to=2-2]
	\arrow["{i_{n}^{1-\alpha}}"', from=2-1, to=2-2]
\end{tikzcd}
\end{equation}
Then if $C_0$ and $C_1$ are strict, so is $D$.
\end{lemma}
\begin{proof}
We denote by $(i,r,\phi)$ the deformation retract structure corresponding to $C_0\to D$. 
We show this result by induction, and let's start with the case $n=0$. 
This corresponds to the case where $C_0\to D$ fits in a pullback diagram.
\[\begin{tikzcd}
	{C_0} & D \\
	{\{0\}} & {[1]}
	\arrow[from=1-1, to=1-2]
	\arrow[from=2-1, to=2-2]
	\arrow[from=1-1, to=2-1]
	\arrow[from=1-2, to=2-2]
\end{tikzcd}\]
 Let $x,y$ be two objects of $D$. Suppose first that $x$ and $y$ are over the same object of $[1]$. In this case, $\hom_D(x,y)$ is equivalent to either $\hom_{C_0}(x,y)$ or $\hom_{C_1}(x,y)$ and is then strict. If $x$ is over $1$ and $y$ over $0$, the $\infty$-groupoid $\hom_D(x,y)$ is empty. If $x$ is over $0$ and $y$ is over $1$, $\hom_D(x,y)$ is equivalent to $\hom_{C_0}(x,ry)$ according to \ref{prop:Gray deformation retract and passage to hom unmarked} and is then strict by hypothesis. Eventually, $\tau_0(D)$ is equivalent to $\tau_0(C_1)$ and is then a set. According to \ref{prop:criter stricte easy}, this implies that $D$ is strict.

Suppose now the result is true at stage $(n-1)$.
Let $p'\to p$ be a square verifying the condition.  Remark that, at the level of objects, the inclusion $C_0\to D$, its retract, and its deformation, are the identity.

Let $x$ and $y$ be two objects of $D$. 
As before, the only interesting case is when $x$ is over $0$ and $y$ is over $1$. In this case,
applying $\hom(\uvar,\uvar)$ to the square \eqref{eq:lemma:strictification2}, we get a cartesian square
\[\begin{tikzcd}
	{\hom_{C_0}(x,y)} & {\hom_D(x,y)} \\
	{\Db_{n-1}} & {\Db_{n}}
	\arrow[from=1-1, to=1-2]
	\arrow["{i_{n-1}^\alpha}"', from=2-1, to=2-2]
	\arrow[from=1-1, to=2-1]
	\arrow[from=1-2, to=2-2]
\end{tikzcd}\]
which is a right $n$-Gray deformation retract according to proposition \ref{prop:Gray deformation retract and passage to hom unmarked}.
Applying $\hom(\uvar,\uvar)$ to the square \eqref{eq:lemma:strictification3}, we get a cartesian square
\[\begin{tikzcd}
	{\hom_{C_1}(x,y)} & {\hom_D(x,y)} \\
	{\Db_{n-1}} & {\Db_{n}}
	\arrow[from=1-1, to=1-2]
	\arrow["{i_{n-1}^{1-\alpha}}"', from=2-1, to=2-2]
	\arrow[from=1-1, to=2-1]
	\arrow[from=1-2, to=2-2]
\end{tikzcd}\]
As $C_1$ is strict, so is $\hom_{C_1}(x,y)$.
 We can then apply the induction hypothesis, which implies that $\hom_D(x,y)$ is strict. As $\tau_0 D$ is equivalent to $\tau_0 C_{0}$, it is a set. We can apply proposition \ref{prop:criter stricte easy} which implies that $D$ is strict. 
\end{proof}

\p 
For an integer $n>0$,
we define by induction 
\begin{enumerate}
\item[$-$]
a left $(n+1)$-Gray retract structure for the inclusion 
\begin{equation}
\label{eq:Gray retract structurure for Gray cone}
\Db_n\star\emptyset \cup \Db_{n-1}\star 1 \to \Db_n\star 1
\end{equation}
where the gluing is performed along $i_n^\alpha:\Db_{n-1}\star\emptyset\to \Db_n\star\emptyset$ with $\alpha$ being $+$ if $n$ is odd and $-$ if not, 
\item[$-$]
a right $(n+1)$-Gray retract structure for the inclusion
\begin{equation}
\label{eq:Gray retract structurure for circ Gray cone}
1\costar\Db_{n-1} \cup \emptyset\costar\Db_{n}\to 1 \costar \Db_n
\end{equation}
where the gluing is performed along $i_n^\alpha:\emptyset\costar \Db_{n-1}\to \emptyset\costar \Db_n$ with $\alpha$ being $-$ if $n$ is odd and $+$ if not. 
\end{enumerate}
If $n=1$, the first morphism corresponds to the inclusion
\[\begin{tikzcd}
	\bullet & {} & \bullet \\
	\bullet & \bullet & \bullet & \bullet
	\arrow[""{name=0, anchor=center, inner sep=0}, from=1-3, to=2-3]
	\arrow[""{name=1, anchor=center, inner sep=0}, from=2-3, to=2-4]
	\arrow[""{name=2, anchor=center, inner sep=0}, from=1-3, to=2-4]
	\arrow[from=1-1, to=2-1]
	\arrow[from=2-1, to=2-2]
	\arrow[""{name=3, anchor=center, inner sep=0}, draw=none, from=1-2, to=2-2]
	\arrow[shift right=2, shorten <=2pt, shorten >=2pt, Rightarrow, from=2, to=1]
	\arrow[shorten <=6pt, shorten >=6pt, maps to, from=3, to=0]
\end{tikzcd}\]
and the second one to the inclusion:
\[\begin{tikzcd}
	& \bullet & {} & \bullet \\
	\bullet & \bullet & \bullet & \bullet
	\arrow[""{name=0, anchor=center, inner sep=0}, from=1-2, to=2-2]
	\arrow[""{name=1, anchor=center, inner sep=0}, from=2-3, to=2-4]
	\arrow[""{name=2, anchor=center, inner sep=0}, from=2-3, to=1-4]
	\arrow[from=1-4, to=2-4]
	\arrow[from=2-1, to=1-2]
	\arrow[""{name=3, anchor=center, inner sep=0}, draw=none, from=1-3, to=2-3]
	\arrow[shift left=2, shorten <=2pt, shorten >=2pt, Rightarrow, from=2, to=1]
	\arrow[shorten <=6pt, shorten >=6pt, maps to, from=0, to=3]
\end{tikzcd}\]
The propositions \ref{prop:of left Gray deformation retract unmarked} and \ref{prop:left Gray deformation retract stable under pushout unmarked} imply that the first morphism is a left $2$-Gray deformation retract and the second one a right $2$-Gray deformation retract. 
Suppose now that these two morphisms are constructed at stage $n$.
The formula \eqref{eq:eq for Gray cone} implies that $\Db_{n+1}\star\emptyset \cup \Db_{n}\star 1 \to \Db_{n+1}\star 1$ fits in the cocartesian square
\[\begin{tikzcd}
	{[1\costar\Db_{n-1}\cup \emptyset\star\Db_n,1]} & {\Db_n\star\emptyset\cup \Db_{n-1}\star 1} \\
	{[1\costar\Db_n,1]} & { \Db_n\star 1}
	\arrow[from=2-1, to=2-2]
	\arrow[from=1-1, to=1-2]
	\arrow[from=1-1, to=2-1]
	\arrow[from=1-2, to=2-2]
\end{tikzcd}\]
The induction hypothesis and the propositions \ref{prop:suspension of left Gray deformation retract unmarked} and \ref{prop:left Gray deformation retract stable under pushout unmarked} endow this morphism with a left $(n+2)$-Gray retract structure. We constructs similarly the right $(n+2)$-Gray retract structure for the inclusion $1\costar\Db_{n-1} \cup \emptyset\costar\Db_{n}\to 1 \costar \Db_n$.
\begin{prop}
\label{prop:strict stuff are stable under Gray cone}
Let $C$ be a strict $\io$-category, $a$ a globular sum, and $f:a\to C$ any morphism. The $\io$-categories $C\coprod_a a\star 1$ and $1\costar a\coprod_a C$ are strict.
\end{prop}
\begin{proof}
We will prove the result by induction on the number of non-identity cells of $a$. Remark that for any globular sum $b$, there exists a globular sum $a$, an integer $n$, and a cartesian square composed of globular morphism
\[\begin{tikzcd}
	{\Db_{n-1}} & a \\
	{\Db_{n}} & b
	\arrow["{i^\alpha_{n-1}}"', from=1-1, to=2-1]
	\arrow["l"', from=2-1, to=2-2]
	\arrow["\lrcorner"{anchor=center, pos=0.125, rotate=180}, draw=none, from=2-2, to=1-1]
	\arrow[from=1-1, to=1-2]
	\arrow[from=1-2, to=2-2]
\end{tikzcd}\]
with $\alpha=+$ if $n$ is odd, and $\alpha=-$ if $n$ is even, and such that $l$ admits a retract $r$. As $i^\alpha_{n-1}$ is globular, the pullback along this morphism preserves colimits according to theorem \ref{theo:pullback along conduche preserves colimits}. We then have a cartesian square:
\[\begin{tikzcd}
	a & b \\
	{\Db_{n-1}} & {\Db_{n}}
	\arrow["r", from=1-2, to=2-2]
	\arrow["{i^\alpha_{n-1}}"', from=2-1, to=2-2]
	\arrow[from=1-1, to=2-1]
	\arrow[from=1-1, to=1-2]
\end{tikzcd}\]
 We also define $a'$ as the pullback:
\[\begin{tikzcd}
	{a'} & b \\
	{\Db_{n-1}} & {\Db_{n}}
	\arrow["r", from=1-2, to=2-2]
	\arrow["{i^{-\alpha}_{n-1}}"', from=2-1, to=2-2]
	\arrow[from=1-1, to=2-1]
	\arrow[from=1-1, to=1-2]
\end{tikzcd}\]
and remark that $a'$ is a globular sum. Eventually, we fix a morphism $b\to C$. As $a$ and $a'$ are sub globular sum of $b$, the number of non-identity cells in each of them is strictly less than the one of $b$.
 We then suppose that for any strict $\io$-category $C$, and any morphism $b\to C$, the two induced $\io$-category 
 $C\coprod_a a\star 1$ and $C\coprod_{a'}a'\star 1$ are strict, and we are willing to show that $C\coprod_bb\star 1$ also is. 
We claim that the two following squares are cartesian
\[\begin{tikzcd}
	{b\coprod_{a}a\star1} & {b\star 1} & {b\coprod_{a'}a'\star1} & {b\star 1} \\
	{[\Db_{n-1},1]} & {[\Db_{n},1]} & {[\Db_{n-1},1]} & {[\Db_{n},1]}
	\arrow[from=1-1, to=2-1]
	\arrow[from=1-1, to=1-2]
	\arrow["{[i^\alpha_{n-1},1]}"', from=2-1, to=2-2]
	\arrow[from=1-2, to=2-2]
	\arrow[from=1-3, to=2-3]
	\arrow[from=1-3, to=1-4]
	\arrow["{[i^{-\alpha}_{n-1},1]}"', from=2-3, to=2-4]
	\arrow[from=1-4, to=2-4]
\end{tikzcd}\]
According to theorem \ref{theo:join preserves stict VMG version}, proposition \ref{prop:suspension preserves stricte}, and the induction hypothesis, all the objects of these squares are strict. We can show the cartesianess in $\zocat$, where it follows from lemma \ref{lemma: pullback and sum}. 
As morphism $[i^-_{n-1},1],[i^+_{n-1},1]$ are globular, the pullback functors $[i^-_{n-1},1]^*$, $[i^+_{n-1},1]^*$ preserve colimits according to theorem \ref{theo:pullback along conduche preserves colimits}. We then have two cartesian squares:
\begin{equation}
\label{eq:two square in the proof of strict}
\begin{tikzcd}
	{C\coprod_{a}a\star1} & {C\coprod_bb\star 1} & {C\coprod_{a'}a'\star1} & {C\coprod_bb\star 1} \\
	{[\Db_{n-1},1]} & {[\Db_{n},1]} & {[\Db_{n-1},1]} & {[\Db_{n},1]}
	\arrow[from=1-1, to=2-1]
	\arrow[from=1-1, to=1-2]
	\arrow["{[i^\alpha_{n-1},1]}"', from=2-1, to=2-2]
	\arrow[from=1-2, to=2-2]
	\arrow[from=1-3, to=2-3]
	\arrow[from=1-3, to=1-4]
	\arrow["{[i^{-\alpha}_{n-1},1]}"', from=2-3, to=2-4]
	\arrow[from=1-4, to=2-4]
\end{tikzcd}
\end{equation}
and by the induction hypothesis, the two top left objects are strict. 
Eventually, remark that we have a cocartesian square
\[\begin{tikzcd}
	{\Db_n\coprod_{\Db_{n-1}}\Db_{n-1}\star 1} & {\Db_{n}\star 1} \\
	{C\coprod_{a}a\star1} & {C\coprod_bb\star 1}
	\arrow[from=2-1, to=2-2]
	\arrow[from=1-1, to=2-1]
	\arrow[from=1-1, to=1-2]
	\arrow[from=1-2, to=2-2]
	\arrow["\lrcorner"{anchor=center, pos=0.125, rotate=180}, draw=none, from=2-2, to=1-1]
\end{tikzcd}\]
and the proposition \ref{prop:left Gray deformation retract stable under pushout unmarked} then implies that the left square of 
\eqref{eq:two square in the proof of strict} is a left $(n+1)$-Gray retract, and the lemma \ref{lemma:strictification2} implies that $C\coprod_bb\star 1$ is strict. This proves the first assertion. The second one is proved similarly.
\end{proof}

\p We now want to give an analogue of proposition \ref{prop:strict stuff are stable under Gray cone} for the Gray cylinder. In what follows, we will use the results of sections \ref{section:Colimit of left cartesian fibrations} and \ref{subsection:A criterion to be a left cartesian fibration} (more precisely the proposition \ref{prop:equivalence betwen slice and join strict word2}, the theorem \ref{theo:equivalence betwen slice and join} and the corollaries \ref{cor:cor of the past10}, \ref{cor:cor of the past3}).
We assure the reader that this is not a tautology, as the proofs of these results are not based on the following propositions and theorems

\begin{prop}
\label{prop:fibers of 1 star a}
Let $a$ be a globular sum. The two following canonical squares are cartesian
\[\begin{tikzcd}
	1 & {1\costar a} & 1 & {a\star 1} \\
	{\{0\}} & {[a,1]} & {\{1\}} & {[a,1]}
	\arrow[from=1-1, to=1-2]
	\arrow[from=2-1, to=2-2]
	\arrow[from=1-1, to=2-1]
	\arrow[from=1-2, to=2-2]
	\arrow[from=1-3, to=1-4]
	\arrow[from=2-3, to=2-4]
	\arrow[from=1-3, to=2-3]
	\arrow[from=1-4, to=2-4]
\end{tikzcd}\]
The five squares appearing in the following canonical diagram are both cartesian and cocartesian:
\[\begin{tikzcd}
	& {a\otimes\{0\}} & 1 \\
	{a\otimes\{1\}} & {a\otimes[1]} & {a\star 1} \\
	1 & {1\costar a} & {[a,1]}
	\arrow[from=2-3, to=3-3]
	\arrow[from=3-2, to=3-3]
	\arrow[from=2-2, to=3-2]
	\arrow[from=2-2, to=2-3]
	\arrow[from=1-2, to=1-3]
	\arrow[from=1-3, to=2-3]
	\arrow[from=1-2, to=2-2]
	\arrow[from=2-1, to=2-2]
	\arrow[from=3-1, to=3-2]
	\arrow[from=2-1, to=3-1]
\end{tikzcd}\]
\end{prop}
\begin{proof}
The five squares of the second diagram are cocartesian by construction. Furthermore, remark that all the objects appearing in the squares
\[\begin{tikzcd}
	a & {a\star 1} & a & {1\costar a} \\
	{\{0\}} & {[a,1]} & {\{1\}} & {[a,1]} \\
	{a\otimes\{1\}} & {a\star 1} & {a\otimes\{0\}} & {1\costar a} \\
	{\{1\}} & {[a,1]} & {\{0\}} & {[a,1]}
	\arrow[from=3-2, to=4-2]
	\arrow[from=3-1, to=4-1]
	\arrow[from=3-1, to=3-2]
	\arrow[from=4-1, to=4-2]
	\arrow[from=1-1, to=2-1]
	\arrow[from=1-2, to=2-2]
	\arrow[from=1-3, to=2-3]
	\arrow[from=1-4, to=2-4]
	\arrow[from=2-1, to=2-2]
	\arrow[from=2-3, to=2-4]
	\arrow[from=1-3, to=1-4]
	\arrow[from=1-1, to=1-2]
	\arrow[from=3-3, to=4-3]
	\arrow[from=3-4, to=4-4]
	\arrow[from=4-3, to=4-4]
	\arrow[from=3-3, to=3-4]
\end{tikzcd}\]
are strict according to theorem \ref{theo:join preserves stict VMG version} and proposition \ref{prop:suspension preserves stricte}. One can the show their cartesianess in $\zocat$, where it follows from proposition \ref{prop:cartesian squares}.

By stability by right cancellation of cartesian square, it remains to show that the square
\[\begin{tikzcd}
	{a\otimes[1]} & {1\costar a} \\
	{a\star 1} & {[a,1]}
	\arrow[from=1-1, to=2-1]
	\arrow[from=1-2, to=2-2]
	\arrow[from=2-1, to=2-2]
	\arrow[from=1-1, to=1-2]
\end{tikzcd}\]
is cartesian.
Using the fact that pullback along $1\costar a\to [a,1]$ preserves colimits as stated by corollary \ref{cor:cor of the past3}, it is sufficient to show that for any globular morphism $\Db_n\to a$, the outer square of the diagram
\[\begin{tikzcd}
	{\Db_n\otimes[1]\coprod_{\Db_n} a} & {a\otimes[1]} & {1\costar a} \\
	{\Db_n\star 1} & {a\star 1} & {[a,1]}
	\arrow[from=2-1, to=2-2]
	\arrow[from=1-2, to=2-2]
	\arrow[from=2-2, to=2-3]
	\arrow[from=1-1, to=2-1]
	\arrow[from=1-1, to=1-2]
	\arrow[from=1-2, to=1-3]
	\arrow[from=1-3, to=2-3]
\end{tikzcd}\]
is cartesian.
Remark that this outer square also factors as:
\[\begin{tikzcd}
	{\Db_n\otimes[1]\coprod_{\Db_n} a} & {1\costar \Db_n\coprod _{\Db_n}a} & {1\costar a} \\
	{\Db_n\star 1} & {[\Db_n,1]} & {[a,1]}
	\arrow[from=1-1, to=1-2]
	\arrow[from=1-1, to=2-1]
	\arrow[from=2-1, to=2-2]
	\arrow[from=1-2, to=2-2]
	\arrow[from=2-2, to=2-3]
	\arrow[from=1-2, to=1-3]
	\arrow[from=1-3, to=2-3]
\end{tikzcd}\]
The cartesianess of the left square is a consequence of the preservation of colimit of the pullback along the morphism $\Db_n\star 1\to [\Db_n,1]$, and of the cartesian square provided by proposition \ref{prop:cartesian squares}. We recall that we can indeed use the last proposition, as we already show in lemma \ref{lemma:gray operation on globes are strict} that $\Db_n\otimes[1]$, $1\costar \Db_n$ and $\Db_n\star 1$ are strict.

For the right hand square, all the objects are strict according to proposition \ref{prop:strict stuff are stable under Gray cone}. We can then show the cartesianess in $\zocat$, where it follows from lemma \ref{lemma: pullback and sum}.
\end{proof}

\begin{lemma}
\label{lemma:an other canonical square}
Let $C$ be an $\io$-category, $a$ a globular sum, and $a\to C$ any morphism.
The following canonical square is cartesian: 
\[\begin{tikzcd}
	{C\coprod_aa\otimes[1]} & {C\coprod_a a\star 1} \\
	{1\costar a} & {[a,1]}
	\arrow[from=2-1, to=2-2]
	\arrow[from=1-1, to=2-1]
	\arrow[from=1-1, to=1-2]
	\arrow[from=1-2, to=2-2]
\end{tikzcd}\]
\end{lemma}
\begin{proof}
For any $\io$-category $D$, the first square of proposition \ref{prop:fibers of 1 star a} implies that the following square is cartesian
\[\begin{tikzcd}
	{D\otimes\{0\}} & {D\otimes\{0\}} \\
	{1\costar a} & {[a,1]}
	\arrow[from=1-1, to=2-1]
	\arrow[from=1-2, to=2-2]
	\arrow[from=2-1, to=2-2]
	\arrow[from=1-1, to=1-2]
\end{tikzcd}\]
The statement then follows from proposition \textit{op cit} and the preservation of colimit of the pullback along the morphism $1\costar a\to [a,1]$ stated by corollary \ref{cor:cor of the past3}.
\end{proof}

\begin{prop}
\label{prop:strict stuff are stable under coproduc with cylinder}
Let $C$ be a strict $\io$-category, $a$ a globular sum, and $a\to C$ any morphism. The $\io$-category $C\coprod_a a\otimes [1]$ is strict. In particular $a\otimes[1]$ is strict.
\end{prop}
\begin{proof}
According to propositions \ref{prop:suspension preserves stricte} and \ref{prop:strict stuff are stable under Gray cone}, the two lower objects and the upper right one of the cartesian square of lemma \ref{lemma:an other canonical square} are strict whenever $C$ is. As strict object are stable under pullback, this concludes the proof.
\end{proof}

\p We combine the proposition \ref{prop:strict stuff are stable under Gray cone} and \ref{prop:strict stuff are stable under coproduc with cylinder} in the following theorem:

\begin{theorem}
\label{prop:strict stuff are pushout}
Let $C$ be an $\io$-category, $a$ a globular sum, and $f:a\to C$ any morphism. The $\io$-categories $$1\costar a\coprod_a C~~~~C\coprod_a a\otimes[1]~~~~C\coprod_a a\star 1$$
are strict whenever $C$ is. In particular, $a\otimes[1]$, $a\star 1$ and $1\costar a$ are strict.
\end{theorem}

\begin{cor}
\label{cor: a otimes n is strict}
Let $a$ be a globular sum, and $K$ an order set, viewed as an $\iun$-category. The $\io$-category $a\otimes K$ is strict.
\end{cor}
\begin{proof}
If $K$ is $[n]$, an easy induction using proposition \ref{prop:strict stuff are stable under coproduc with cylinder} shows the result.
In the general case, remark that $K$ is the special colimit of the diagram $\pi:\Delta^{\hookrightarrow}_{/K}\to \iPsh{\Delta}$ where $\Delta^{\hookrightarrow}_{/K}$ is the category whose objects are monomorphisms $[n]\to K$ and arrows are monomorphisms between domains making the induced triangle commutative, while $\pi$ sends $[n]\to K$ to $[n]$.
We claim that the natural transformation 
$$a\otimes \pi\to \pi$$
is cartesian. Proposition \ref{prop:special colimit} then implies that $a\otimes\pi$ has a special colimit. Moreover, $a\otimes \pi$
fulfills the hypotheses of the third assertion of lemma \ref{lemma:colimit computed in set presheaves}. Its colimit is then strict, and this concludes the proof of the first assertion. 

To demonstrate the cartesianess of the natural transformation $a\otimes \pi\to \pi$, one has to show that for any monomorphism $i:[k]\to [l]$,
the induced square 
\[\begin{tikzcd}
	{a\otimes [k]} & {a\otimes[l]} \\
	{[k]} & {[l]}
	\arrow[from=1-1, to=2-1]
	\arrow[from=1-2, to=2-2]
	\arrow[from=2-1, to=2-2]
	\arrow[from=1-1, to=1-2]
\end{tikzcd}\]
is cartesian.

As $[k]\to [l]$ is fully faithful, so is $[k]\times_{[l]}a\otimes[l]\to a\otimes[l]$. If we manage to show that $a\otimes[k]\to a\otimes[l]$ is fully faithful, it will imply by right cancelation that $a\otimes [k]\to [l]\coprod_{[k]}a\otimes[l]$ is also fully faithful, and as this morphism is obviously surjective on objects it will conclude the proof.

We then have to show that for any integer $n>0$, any square of shape
\[\begin{tikzcd}
	{ \partial\Db_n} & { a\otimes [k]} \\
	{ \Db_n} & { a\otimes [l]}
	\arrow[from=1-2, to=2-2]
	\arrow["f"', from=2-1, to=2-2]
	\arrow["g", from=1-1, to=1-2]
	\arrow[from=1-1, to=2-1]
\end{tikzcd}\]
admits a unique lifting.
Suppose given such square.
Using the Steiner theory recalled in \ref{section:Steiner thery}, it is equivalent show that the induced square of augmented directed complexes:
\[\begin{tikzcd}
	{\lambda \partial\Db_n} & {\lambda a\otimes \lambda [k]} \\
	{\lambda \Db_n} & {\lambda a\otimes\lambda [l]}
	\arrow[from=1-2, to=2-2]
	\arrow["f"', from=2-1, to=2-2]
	\arrow["g", from=1-1, to=1-2]
	\arrow[from=1-1, to=2-1]
\end{tikzcd}\]
admits a unique lifting. 
We recall that the basis of $\lambda\Db_n$ is given by the graded set:
$$(B_{\lambda\Db_n})_k:= \left\{ 
\begin{array}{ll}
\{e_k^-,e_k^+\}&\mbox{ if $k<n$}\\
\{e_n\}&\mbox{ if $k=n$}\\
\emptyset&\mbox{ if $k>n$}\\
\end{array}\right.$$ 
and that the basis of $\lambda[n]$ also admits is given by the graded set
$$(B_{\lambda\Db_n})_k:= \left\{ 
\begin{array}{ll}
\{v_0,v_1,...,v_n\}&\mbox{ if $k=0$}\\
\{v_{0,1},v_{1,2}...,v_{n-1,n}\}&\mbox{ if $k=1$}\\
\emptyset&\mbox{ if k>1}\\
\end{array} \right.$$ 

 We will suppose that $n$ is odd as the proof for $n$ even is similar. As the right vertical morphism is an injection, we just have to show the existence of the lifting.

There exists a unique sequence $\{b_0,...,b_{l-1}\}$ of element of $(\lambda b)_{n-1}$
and a unique sequence $\{c_0,...,c_{l}\}$ of element of $(\lambda b)_{n}$ such that 
$$f(e_n)= b_0\otimes v_{0,1}+...+ b_{l-1}\otimes v_{l-1,l}+c_0\otimes v_0+...+ c_l\otimes v_l$$
The commutativity of the square then implies that the cell
$$ \partial b_0\otimes v_{0,1}+...+ \partial b_{l-1}\otimes v_{l-1,l}+ (\partial c_0-b_0)\otimes v_0+(\partial c_1+b_0-b_1)\otimes v_1...+ (\partial c_l + b_l)\otimes v_l$$
is in the image of $\lambda a\otimes\lambda i$. 
As a consequence, for any $j<k$, we have
$$ \left\{
\begin{array}{ll}
\partial b_0=\partial b_1=...= \partial b_{i(0)-1}\\
\partial b_{i(j)}=\partial b_{i(j)+1}... =\partial b_{i(j+1)-1} &\mbox{for $j<k$}\\
\partial b_{i(k)}=\partial b_{i(k)+1}=...=\partial b_{l-1}\\
\end{array}\right.$$
and 
$$
\left\{
\begin{array}{ll}
\partial c_0-b_0=0 &\mbox{if $0$ is not in the image of $i$}\\
\partial c_p+b_{p-1}-b_p=0& \mbox{if $p>0$ is not in the image of $i$}\\
\partial c_l+b_{l-1}=0& \mbox{if $l$ is not in the image of $i$}\\
\end{array}\right.$$
The first set of equations forces the equalities: 
$$ \left\{
\begin{array}{ll}
 b_0= b_1=...= b_{i(0)-1}\\
 b_{i(j)}= b_{i(j)+1}... = b_{i(j+1)-1} &\mbox{for $j<k$}\\
 b_{i(k)}= b_{i(k)+1}=...= b_{l-1}\\
\end{array}\right.$$
Combined with the second set of equations this implies that $c_p$ is null whenever $p$ is not in the image of $i$.
We then have 
$$f(e_n)=b_{i(0)}\otimes \lambda i(v_{0,1})+...+b_{i(k)}\otimes \lambda i(v_{k-1,k})
+c_{i(0)}\otimes \lambda i (v_0)+...+ c_i(k)\otimes \lambda i (v_k)$$

We then define the morphism $l$ as the unique morphism extending $g$ and that fulfills 
$$l_n(e_n):= b_{i(0)}\otimes v_{0,1}+...+b_{i(k)}\otimes v_{k-1,k}
+c_{i(0)}\otimes v_0+...+ c_i(k)\otimes v_k$$ 
This morphism is the wanted lift.
\end{proof}

\begin{cor}
\label{cor:otimes et op}
There is a natural diagram
\[\begin{tikzcd}
	{(C\otimes\{1\})^\circ} & {(C\otimes[1])^\circ} & {(C\otimes\{0\})^\circ} \\
	{C^\circ\otimes\{0\}} & {C^\circ\otimes[1]} & {C^\circ\otimes\{1\}}
	\arrow["\sim", from=1-2, to=2-2]
	\arrow["\sim", from=1-3, to=2-3]
	\arrow["\sim", from=1-1, to=2-1]
	\arrow[from=1-3, to=1-2]
	\arrow[from=1-1, to=1-2]
	\arrow[from=2-1, to=2-2]
	\arrow[from=2-3, to=2-2]
\end{tikzcd}\]
where all vertical arrows are equivalences. There is an invertible natural transformation
$$ C\star 1\sim (1\costar C^{\circ})^\circ.$$
\end{cor}
\begin{proof}
As these functors preserve colimits, we can define this equivalence on representables. As cylinders (resp. cone) (resp. $\circ$-cone) of representables are strict according to theorem \ref{prop:strict stuff are pushout}, and as $(\uvar)^\circ$ preserves strict objects, it is enough to show these equivalences in $\zocat$, where it follows from \cite[proposition A.22]{Ara_Maltsiniotis_joint_et_tranche}.
\end{proof}

\begin{cor}
\label{cor:ominus et op}
Let $A$ and $B$ two $\io$-categories. There is an equivalence  
$$(A\ominus B)^\circ \sim A^\circ\ominus B^\circ$$
natural in $A$ and $B$.
\end{cor}
\begin{proof}
It is sufficient to construct the equivalence when $A$ is a globular sum $a$ and $B$ is of shape $[b,n]$. 
Remark first that the corollary \ref{cor: a otimes n is strict} implies that $(a\otimes[n])^\circ$ and $a^\circ\otimes[n]^\circ$ are strict objects. The proposition A.22 of \cite{Ara_Maltsiniotis_joint_et_tranche}  then implies that these two objects are isomorphic.  The results then directly follows from the definition of the operation $\ominus$ and from the equivalence $(m_b(\uvar))^\circ\sim m_{b^\circ}((\uvar)^\circ)$.
\end{proof}

\begin{cor}
\label{cor:characterisaiont of Gray operation}
Let $F$ be an endofunctor of $\ocat$ such that the induced functor $\ocat\to \ocat_{F(\emptyset)/}$ is colimit preserving, and $\psi$ is an invertible natural transformation between $\Gb^{+}\to \ocat\xrightarrow{F}\ocat$ and $\Gb^{+}\to \ocat\xrightarrow{H}\ocat$ where $\Gb^{+}$ is obtained from $\Gb$ by adding an initial element $\{\emptyset\}$, and $H$ is either the Gray cylinder, the Gray cone, the Gray $\circ$-cone or an iterated suspension.

Then, the natural transformation $\psi$ can be extended to an invertible natural transformation between $F$ and $H$.
\end{cor}
\begin{proof}
We denote by $\Theta^+$ the category obtained from $\Theta$ by adding an initial element $\emptyset$. 
Remark first that the theorem \ref{theo:appendince unicity of operation} implies that we have an invertible natural transformation 
$$\pi_0 \circ F_{|\Theta^+}\to \pi_0 \circ H_{|\Theta^+}.$$
The propositions \ref{prop:strict stuff are stable under Gray cone}, \ref{prop:strict stuff are stable under coproduc with cylinder} and \ref{prop:suspension preserves stricte} imply that the canonical morphism 
$$H_{|\Theta^+}\to \N\circ \pi_0\circ 	H_{|\Theta^+}$$ is an equivalence. The two previous morphisms then induce a comparison:
$$F_{|\Theta^+}\to \N\circ \pi_0 \circ F_{|\Theta^+}\to H_{|\Theta^+}$$
By extension by colimits, this produces a natural transformation $\phi:F\to H$ extending $\psi$. The full sub $\infty$-groupoid of objects $C$ such that $\phi_C:F(C)\to H(C)$ is an equivalence is closed by colimits, contains globes, and so is the maximal sub 	$\infty$-groupoid.
\end{proof}
The previous corollary implies that the equations \eqref{eq:eq for cylinder}, \eqref{eq:eq for Gray cone} and \eqref{eq:eq for cojoin} characterize respectively the Gray cylinder, the Gray cone, and the Gray $\circ$-cone. 

\begin{cor}
\label{cor:crushing of Gray tensor is identitye}
The colimit preserving endofunctor $F:\ocat\to \ocat$, sending $[a,n]$ to the colimit of the span
$$\coprod_{k\leq n}\{k\}\leftarrow \coprod_{k\leq n}a\otimes\{k\}\to a\otimes[n]$$
is equivalent to the identity.
\end{cor}
\begin{proof}
The proposition \ref{prop:fibers of 1 star a} implies that the restriction of $F$ to globes is equivalent to the restriction of the identity to globes. As the identity is the $0$-iterated suspension, we can apply corollary \ref{cor:characterisaiont of Gray operation}.
\end{proof}

The last corollary implies that for any $\io$-category $C$ and any globular sum $a$, the simplicial $\infty$-groupoid
$$\begin{array}{rcl}
\Delta^{op}&\to &\igrd\\
~[n]~&\mapsto &\Hom([a,n],C)
\end{array} $$
is a $\iun$-category.

\begin{theorem}
\label{theo:formula between pullback of slice and tensor}
Let $C$ be an $\io$-category. The two following canonical squares are cartesian:
\[\begin{tikzcd}
	1 & {1\costar C} & 1 & {C\star 1} \\
	{\{0\}} & {[C,1]} & {\{1\}} & {[C,1]}
	\arrow[from=1-1, to=1-2]
	\arrow[from=2-1, to=2-2]
	\arrow[from=1-1, to=2-1]
	\arrow[from=1-2, to=2-2]
	\arrow[from=1-3, to=1-4]
	\arrow[from=2-3, to=2-4]
	\arrow[from=1-3, to=2-3]
	\arrow[from=1-4, to=2-4]
\end{tikzcd}\]
The five squares appearing in the following canonical diagram are both cartesian and cocartesian:
\[\begin{tikzcd}
	& {C\otimes\{0\}} & 1 \\
	{C\otimes\{1\}} & {C\otimes[1]} & {C\star 1} \\
	1 & {1\costar C} & {[C,1]}
	\arrow[from=2-3, to=3-3]
	\arrow[from=3-2, to=3-3]
	\arrow[from=2-2, to=3-2]
	\arrow[from=2-2, to=2-3]
	\arrow[from=1-2, to=1-3]
	\arrow[from=1-3, to=2-3]
	\arrow[from=1-2, to=2-2]
	\arrow[from=2-1, to=2-2]
	\arrow[from=3-1, to=3-2]
	\arrow[from=2-1, to=3-1]
\end{tikzcd}\]
\end{theorem}
\begin{proof}
The five squares of the second diagram are cocartesian by construction.

If $C$ is empty, all the considered squares are cartesian. We can then suppose that there exists a globular sum $a$, and a morphism $a\to C$. We claim that the two following squares are cartesian.
\[\begin{tikzcd}
	{C\coprod 1} & {C\coprod_a a\star 1} & {C\star 1} \\
	{\{0\}\coprod \{1\}} & {[a,1]} & {[C,1]}
	\arrow[from=1-1, to=2-1]
	\arrow[from=1-3, to=2-3]
	\arrow[from=1-1, to=1-2]
	\arrow[from=1-2, to=1-3]
	\arrow[from=2-1, to=2-2]
	\arrow[from=2-2, to=2-3]
	\arrow[from=1-2, to=2-2]
\end{tikzcd}\]
The cartesianess of the left square is a consequence of proposition \ref{prop:fibers of 1 star a} and of the fact that 
$\{0\}\to [a,1]$ and $\{1\}\to [a,1]$ are discrete Conduché functors and so pullback along them preserves colimits. The cartesianess of the right square is a consequence of the preservation of Gray operations by the full duality stated in corollary \ref{cor:otimes et op}, and of the cartesian square provided by corollary \ref{cor:cor of the past10}.
The two following squares are then cartesian:
\[\begin{tikzcd}
	1 & C\star1 & C & C\star1 \\
	{\{1\}} & {[C,1]} & {\{0\}} & {[C,1]}
	\arrow[from=1-2, to=2-2]
	\arrow[from=1-3, to=2-3]
	\arrow[from=2-3, to=2-4]
	\arrow[from=1-4, to=2-4]
	\arrow[from=1-3, to=1-4]
	\arrow[from=1-1, to=1-2]
	\arrow[from=1-1, to=2-1]
	\arrow[from=2-1, to=2-2]
\end{tikzcd}\]
As the duality $(\uvar)^{\circ}$ preserves limits, and combined with corollary \ref{cor:otimes et op}, this implies that the two following squares are also cartesian:
\[\begin{tikzcd}
	1 & {1\costar C} & C & {1\costar C} \\
	{\{0\}} & {[C,1]} & {\{1\}} & {[C,1]}
	\arrow[from=1-2, to=2-2]
	\arrow[from=1-3, to=2-3]
	\arrow[from=2-3, to=2-4]
	\arrow[from=1-4, to=2-4]
	\arrow[from=1-3, to=1-4]
	\arrow[from=1-1, to=1-2]
	\arrow[from=1-1, to=2-1]
	\arrow[from=2-1, to=2-2]
\end{tikzcd}\]

By stability by right cancellation of cartesian square, it remains to show that the square
\[\begin{tikzcd}
	{C\otimes[1]} & {C\star 1} \\
	{1\costar C} & {[C,1]}
	\arrow[from=1-1, to=1-2]
	\arrow[from=2-1, to=2-2]
	\arrow[from=1-2, to=2-2]
	\arrow[from=1-1, to=2-1]
\end{tikzcd}\]
is cartesian. Consider the two following squares
\[\begin{tikzcd}
	{C\coprod_a a\otimes[1]} & {C\coprod_aa\star 1} & {C\star 1} \\
	{1\costar a} & {[a,1]} & {[C,1]}
	\arrow[from=1-1, to=2-1]
	\arrow[from=1-3, to=2-3]
	\arrow[from=1-2, to=2-2]
	\arrow[from=2-2, to=2-3]
	\arrow[from=2-1, to=2-2]
	\arrow[from=1-1, to=1-2]
	\arrow[from=1-2, to=1-3]
\end{tikzcd}\]	
We already demonstrate that the right one is cartesian and the lemma \ref{lemma:an other canonical square} states that the left one is also cartesian. The outer square is then cartesian.

Using that pulling back along $C\star 1\to [C,1]$ preserves colimits as shown in corollary \ref{cor:cor of the past3}, and the fact that $1\costar C$ (resp. $C\otimes[1]$) is the colimit of all the $1\costar a$ (resp. $a\otimes[1]$) for $a$ ranging over the morphisms $a\to C$, this concludes the proof.
\end{proof}

\begin{theorem}
\label{theo:strictness}
If $C$ is strict, so are $C\star 1$, $1\costar C$ and $C\otimes [1]$.
\end{theorem}
\begin{proof}
Forgetting the marking, the theorem \ref{theo:equivalence betwen slice and join} implies that $1\costar C$ is equivalent to $[C,1]_{0/}$ which is strict as $[C,1]$ is according to proposition \ref{prop:suspension preserves stricte}. The second assertion comes from the fact that the full duality preserves $\zo$-categories and that $1\costar C^{\circ} \sim (C\star 1)^{\circ}$.

The theorem \ref{theo:formula between pullback of slice and tensor} implies that we have a cartesian square
\[\begin{tikzcd}
	{C\otimes [1]} & {1\costar C} \\
	{C\star 1} & {[C,1]}
	\arrow[from=1-1, to=2-1]
	\arrow[from=1-1, to=1-2]
	\arrow[from=1-2, to=2-2]
	\arrow[from=2-1, to=2-2]
\end{tikzcd}\]
 As strict objects are stable under pullbacks, this concludes the proof.
\end{proof}

%
%

%
%
%
%
%
%
%
%
%
%

\chapter{The $\iun$-category of marked $\io$-categories}
\label{chapter:chapter the 1 category of marked categories}

\minitoc
\vspace{2cm}

This chapter is dedicated to the study of \textit{marked $\io$-categories}, which are pairs $(C,tC)$, where $C$ is an $\io$-category and $tC:=(tC_n)_{n>0}$ is a sequence of full sub $\infty$-groupoids of $C_n$ that include identities and are stable under composition and whiskering with (possibly unmarked) cells of lower dimensions. There are two canonical ways to mark an $\io$-category $C$. In the first, denoted by $C^\flat$, we mark as little as possible. In the second, denoted by $C^\sharp$, we mark everything.

The first section of the chapter defines these objects and establishes analogs of many results from section \ref{chapter:Basica construciton} to this new framework. In particular, the \textit{marked Gray cylinder} $\uvar\otimes [1]^\sharp$ is defined. If $A$ is an $\io$-category, the underlying $\io$-category of $A^\sharp\otimes[1]^\sharp$ is $A\times [1]$, and the underlying $\io$-category of $A^\flat\otimes[1]^\sharp$ is $A\otimes[1]$. By varying the marking, and at the level of underlying $\io$-categories, we "continuously" move from the cartesian product with the directed interval to the Gray tensor product with the directed interval.

The motivation for introducing markings comes from the notion of \textit{left (and right) cartesian fibrations}. A left cartesian fibration is a morphism between marked $\io$-categories such that only the marked cells of the codomain have cartesian lifting, and the marked cells of the domain correspond exactly to such cartesian lifting. For example, a left cartesian fibration $X\to A^\sharp$ is just a "usual" left cartesian fibration where we have marked the cartesian lifts of the domain, and every morphism $C^\flat \to D^\flat$ is a left cartesian fibration. This shows that marking plays a very different role here than in the case of marked simplicial sets, where it was there to represent (weak) invertibility. For example, if we had wanted to carry out this work in a complicial-like model category, we would have had to consider bimarked simplicial sets.

After defining and enumerating the stability properties enjoyed by this class of left (and right) cartesian fibration, we give several characterizations of this notion in theorem \ref{theo:other characterisation of left caresian fibration}. 

The more general subclass of left cartesian fibrations that still behaves well is the class of \textit{classified left cartesian fibrations}. 
This corresponds to left cartesian fibrations $X\to A$ such that there exists a cartesian square:
\[\begin{tikzcd}
	X & Y \\
	A & {A^\sharp}
	\arrow[from=1-1, to=2-1]
	\arrow[from=2-1, to=2-2]
	\arrow[from=1-1, to=1-2]
	\arrow[from=1-2, to=2-2]
	\arrow["\lrcorner"{anchor=center, pos=0.125}, draw=none, from=1-1, to=2-2]
\end{tikzcd}\]
 where the right vertical morphism is a left cartesian fibration and $A^\sharp$ is obtained from $A$ by marking all cells. In the second section, we prove the following fundamental result:

\begin{itheorem}[\ref{theo:pullback along un marked cartesian fibration}]
Let $p:X\to A$ be a classified left cartesian fibration. Then the functor $p^*:\ocatm_{/A}\to \ocatm_{/X}$ preserves colimits.
\end{itheorem}

The third subsection is devoted to the proof of the following theorem
\begin{itheorem}[\ref{theo:left cart stable by colimit}]
Let $A$ be an $\io$-category and $F:I\to \ocatm_{/A^\sharp}$ be a diagram that is pointwise a left cartesian fibration. The induced morphism 
$\colim_IF$ is a left cartesian fibration over $A^\sharp$.
\end{itheorem}

In the fourth subsection we study \textit{smooth} and \textit{proper} morphisms and we obtain the following expected result:
\begin{iprop}[\ref{prop:quillent theorem A}]
For a morphism $X\to A^\sharp$, and an object $a$ of $A$, we denote by $X_{/a}$ the marked $\io$-category fitting in the following cartesian squares. 
\[\begin{tikzcd}
	{X_{a/}} & X \\
	{A^\sharp_{a/}} & {A^\sharp}
	\arrow[from=2-1, to=2-2]
	\arrow[from=1-2, to=2-2]
	\arrow[from=1-1, to=1-2]
	\arrow[from=1-1, to=2-1]
	\arrow["\lrcorner"{anchor=center, pos=0.125}, draw=none, from=1-1, to=2-2]
\end{tikzcd}\]
We denote by $\bot:\ocatm\to \ocat$ the functor sending a marked $\io$-category to its localization by marked cells.
\begin{enumerate}
\item Let $E$, $F$ be two elements of $\ocatm_{/A^\sharp}$ corresponding to morphisms $X\to A^\sharp$, $Y\to A^\sharp$, and
 $\phi:E\to F$ a morphism between them. We denote by $\Fb E$ and $\Fb F$ the left cartesian fiborant replacement of $E$ and $F$. 
 
The induced morphism $\Fb\phi:\Fb E\to \Fb F$ is an equivalence if and only if for any object $a$ of $A$, the induced morphism 
$$\bot X_{/a}\to \bot Y_{/a}$$ 
is an equivalence of $\io$-categories.
\item A morphism $X\to A^\sharp$ is initial if and only if for any object $a$ of $A$, $\bot X_{/a}$ is the terminal $\io$-category.
\end{enumerate}
\end{iprop}

Finally, in the last subsection, for a marked $\io$-category $I$, we define and study a (huge) $\io$-category $\uLCartc(I)$ that has classified left cartesian fibrations as objects and morphisms between classified left cartesian fibrations as arrows.

\paragraph{Cardinality hypothesis.}
We fix during this chapter two Grothendieck universes $\V\in\Wcard$, such that $\omega\in \U$. When nothing is specified, this corresponds to the implicit choice of the cardinality $\V$.
We then denote by $\Set$ the $\Wcard$-small $1$-category of $\V$-small sets, $\igrd$ the $\Wcard$-small $\iun$-category of $\V$-small $\infty$-groupoids and $\icat$ the $\Wcard$-small $\iun$-category of $\V$-small $\iun$-categories.

\section{Marked $\io$-categories}
\subsection{Definition of marked $\io$-categories}
\p
A \notion{marked $\zo$-category} is a pair $(C,tC)$ where $C$ is an $\zo$-category and $tC:=(tC_n)_{n>0}$ is a sequence of subsets of $C_n$, containing identities, stable by composition and by whiskering with (possibly unmarked) cells of lower dimension.
A $n$-cell $a:\Db_n\to (C,tC)$ is \wcsnotion{marked}{marked $n$-cell}{for marked $\zo$-categories} if it belongs to $tC_n$.

A \notion{marked morphism} $f:(C,tC)\to (D,tT)$ is the data of a morphism on the underlying $\zo$-categories such that $f(tC_n)\subset tD_n$.
The category of marked $\zo$-categories is denoted by \wcnotation{$\zocatm$}{((a40@$\zocatm$}.

\p There are two canonical ways to mark an $\zo$-category. For $C\in \zocat$, we define $C^\sharp := (C,(C_n)_{n>0})$ and $C^\flat := (C,(\Ib(C_{n-1})_{n>0}))$. The first one corresponds to the case where all cells are marked, and the second one where only the identities are marked. These two functors fit in the following adjoint triple:
\ssym{((b10@$(\uvar)^\sharp$}{for (marked) $\zo$-categories}\ssym{((b20@$(\uvar)^\flat$}{for (marked) $\zo$-categories}\ssym{((b30@$(\uvar)^\natural$}{for (marked) $\zo$-categories}
\[\begin{tikzcd}
	{(\uvar)^{\flat}: \zocat} & {\zocatm:(\uvar)^\natural} & {(\uvar)^\natural:\zocatm} & { \zocat	:(\uvar)^\sharp}
	\arrow[""{name=0, anchor=center, inner sep=0}, shift left=2, from=1-1, to=1-2]
	\arrow[""{name=1, anchor=center, inner sep=0}, shift left=2, from=1-2, to=1-1]
	\arrow[""{name=2, anchor=center, inner sep=0}, shift left=2, from=1-3, to=1-4]
	\arrow[""{name=3, anchor=center, inner sep=0}, shift left=2, from=1-4, to=1-3]
	\arrow["\dashv"{anchor=center, rotate=-90}, draw=none, from=0, to=1]
	\arrow["\dashv"{anchor=center, rotate=-90}, draw=none, from=2, to=3]
\end{tikzcd}\]
where $(\uvar)^\natural$ is the obvious forgetfull functor.
To simplify notations, for a marked $\zo$-category $C$, the marked $\io$-categories $(C^\natural)^\flat$ and $(C^\natural)^\sharp$ will be simply denoted by $C^\flat$ and $C^\sharp$.

\begin{example}
For $n$ an integer, we denote by $(\Db_n)_t$ the marked $\zo$-category whose underlying $\zo$-category is $\Db_n$ and whose only non-trivial marked cell is the top dimensional one.
\end{example}

\begin{definition}
We define the category \wcnotation{$t\Theta$}{(tTheta@$t\Theta$} as the full subcategory of $\zocatm$ whose objects are of shape $a^\flat$ for $a$ a globular sum, or \wcnotation{$(\Db_n)_t$}{(dn@$(\Db_n)_t$} for an integer $n\in\Nb$. Remark that this subcategory is dense in $\zocatm$.
\end{definition}

\p We define the $\iun$-category of \wcnotion{stratified $\infty$-presheaves on $\Theta$}{stratified $\infty$-presheaf on $\Theta$}, noted by \wcnotation{$\tiPsh{\Theta}$}{(tPsh@$\tiPsh{\Theta}$}, as the full sub $\iun$-category of $\iPsh{t\Theta}$ whose objects correspond to $\infty$-presheaves $X$ such that the induced morphism
$X((\Db_n)_t)\to X(\Db_n)$
is a monomorphism. 

\begin{prop}
\label{prop:marked presheaves are locally cartesian closed}
The $\iun$-category $\tiPsh{\Theta}$ is locally cartesian closed. 
\end{prop}
\begin{proof}
The $\iun$-category $\tiPsh{\Theta}$ is the localization of the $\iun$-category $\iPsh{t\Theta}$ along the set of map $\widehat{I}$ with $$I:=\{(\Db_n)_t\coprod_{\Db_n}(\Db_n)_t\to(\Db_n)_t\}_n.$$
As $\iPsh{t\Theta}$ is locally cartesian closed, we have to show that for any integer $n>0$ and any cartesian square in $\iPsh{t\Theta}$:
\[\begin{tikzcd}
	{X'} & X \\
	{(\Db_n)_t\coprod_{\Db_n}(\Db_n)_t} & {(\Db_n)_t}
	\arrow["{ }", from=1-2, to=2-2]
	\arrow[from=1-1, to=2-1]
	\arrow["{ }", from=1-1, to=1-2]
	\arrow[from=2-1, to=2-2]
	\arrow["\lrcorner"{anchor=center, pos=0.125}, draw=none, from=1-1, to=2-2]
\end{tikzcd}\]
the top horizontal morphism is in $\widehat{I}$. Using once again the locally cartesian closeness of $\iPsh{t\Theta}$, it is sufficient to show that for any integer $n>0$ and for any morphism 
$j:b\to (\Db_n)_t$ between elements of $t\Theta$, the morphism $i$ appearing in the following cartesian square of $\iPsh{t\Theta}$ is an equivalence or is in $I$:
\[\begin{tikzcd}
	B & b \\
	{(\Db_n)_t\coprod_{\Db_n}(\Db_n)_t} & {(\Db_n)_t}
	\arrow["j", from=1-2, to=2-2]
	\arrow[from=1-1, to=2-1]
	\arrow["i", from=1-1, to=1-2]
	\arrow[from=2-1, to=2-2]
\end{tikzcd}\]
Two cases have to be considered. If $j$ is the identity this is trivially true. If $j$ is any other morphism, it factors through $\Db_n\to (\Db_n)_t$, 	and the following square is cartesian
\[\begin{tikzcd}
	b & b \\
	{\Db_n} & {(\Db_n)_t}
	\arrow["j", from=1-2, to=2-2]
	\arrow[from=1-1, to=2-1]
	\arrow["id", from=1-1, to=1-2]
	\arrow[from=2-1, to=2-2]
\end{tikzcd}\]
 This implies that $B$ is equivalent to $b\coprod_bb\sim b$, and $i$ is then the identity.
 \end{proof}

\p \label{para:defi of music}
For a stratified $\infty$-presheaf $X$ on $\Theta$, we denote by $tX_n$ the $\infty$-groupoid $X((\Db_n)_t)$.
A stratified $\infty$-presheaves on $\Theta$ is then the data of a pair $(X,tX)$ such that $X\in \iPsh{\Theta}$ and $tX:=(tX_n)_{n>0}$ is a sequence of $\infty$-groupoid such that for any $n>0$, $tX_n$ is a full sub $\infty$-groupoid of $X_n$ including all units.

For $X\in \iPsh{\Theta}$, we define $X^\sharp := (X,(X_n)_{n>0})$ and $X^\flat := (X,(\Ib (X_{n-1})_{n>0})$ and we have an adjoint triple 
\ssym{((b10@$(\uvar)^\sharp$}{for (marked) $\io$-categories}\ssym{((b20@$(\uvar)^\flat$}{for (marked) $\io$-categories}\ssym{((b30@$(\uvar)^\natural$}{for (marked) $\io$-categories}
\[\begin{tikzcd}
	{(\uvar)^{\flat}: \iPsh{\Theta}} & {\tiPsh{\Theta}:(\uvar)^\natural} & {(\uvar)^\natural:\tiPsh{\Theta}} & { \Psh{\Theta}:(\uvar)^\sharp}
	\arrow[""{name=0, anchor=center, inner sep=0}, shift left=2, from=1-1, to=1-2]
	\arrow[""{name=1, anchor=center, inner sep=0}, shift left=2, from=1-2, to=1-1]
	\arrow[""{name=2, anchor=center, inner sep=0}, shift left=2, from=1-3, to=1-4]
	\arrow[""{name=3, anchor=center, inner sep=0}, shift left=2, from=1-4, to=1-3]
	\arrow["\dashv"{anchor=center, rotate=-90}, draw=none, from=0, to=1]
	\arrow["\dashv"{anchor=center, rotate=-90}, draw=none, from=2, to=3]
\end{tikzcd}\]
where $(\uvar)^\natural$ is the obvious forgetful functor.

\p We define the category \wcnotation{$t\Delta[t\Theta]$}{(tDeltaTheta@$t\Delta[t\Theta]$} as the pullback
\[\begin{tikzcd}
	{t\Delta[t\Theta]} & t\Theta \\
	{\Delta[\Theta]} & \Theta
	\arrow["{(\uvar)^\natural}", from=1-2, to=2-2]
	\arrow[from=2-1, to=2-2]
	\arrow[from=1-1, to=2-1]
	\arrow["\lrcorner"{anchor=center, pos=0.125}, draw=none, from=1-1, to=2-2]
	\arrow[from=1-1, to=1-2]
\end{tikzcd}\]
The objects of $t\Delta[t\Theta]$ then are of shape $[1]^\sharp$ or $[a,n]$ with $a\in t\Theta$ and $n\in \Delta$.
The $(\infty,1)$-category of \wcnotion{stratified presheaves on $\Delta[\Theta]$}{stratified $\infty$-presheaf on $\Delta[\Theta]$}, denoted by $\tiPsh{\Delta[\Theta]}$, is the full sub $\iun$-category of $\iPsh{t\Delta[t\Theta]}$ whose objects correspond to $\infty$-presheaves $X$ such that the induced morphism
$X((\Db_n)_t)\to X(\Db_n)$
is a monomorphism. 

\begin{prop}
\label{prop:marked presheaves are locally cartesian closed2}
The $\iun$-category $\tiPsh{\Delta[\Theta]}$ is locally cartesian closed. 
\end{prop}
\begin{proof}
The proof is almost identical to the one of proposition \ref{prop:marked presheaves are locally cartesian closed}
\end{proof}

\p For a stratified $\infty$-presheaf $X$ on $\Delta[\Theta]$, we denote by $tX_1$ the   $\infty$-groupoid $X([1]^\sharp)$, and for any $n>1$, we denote by $tX_n$ the $\infty$-groupoid $X((\Db_n)_t)$.

A stratified $\infty$-presheaf on $\Delta[\Theta]$ is then the data of a pair $(X,tX)$ such that $X\in \iPsh{\Delta[\Theta]}$
and $tX:=(tX_n)_{n>0}$ is a sequence of $\infty$-groupoid such that for any $n>0$, $tX_n$ is a full sub $\infty$-groupoid of $X_n$ including all units.

 For $X\in \iPsh{\Delta[\Theta]}$, we define once again $X^\sharp := (X,( X_n)_{n>0})$ and $X^\flat := (X, (\Ib (X_{n-1}))_{n>0})$ and we still have an adjoint triple
\[\begin{tikzcd}[column sep=0.5cm]
	{(\uvar)^\natural\iPsh{\Delta[\Theta]}} & {\tiPsh{\Delta[\Theta]}:(\uvar)^\natural} & {(\uvar)^\natural:\tiPsh{\Delta[\Theta]}} & {\iPsh{\Delta[\Theta]}:(\uvar)^\sharp}
	\arrow[""{name=0, anchor=center, inner sep=0}, shift left=2, from=1-1, to=1-2]
	\arrow[""{name=1, anchor=center, inner sep=0}, shift left=2, from=1-2, to=1-1]
	\arrow[""{name=2, anchor=center, inner sep=0}, shift left=2, from=1-3, to=1-4]
	\arrow[""{name=3, anchor=center, inner sep=0}, shift left=2, from=1-4, to=1-3]
	\arrow["\dashv"{anchor=center, rotate=-90}, draw=none, from=0, to=1]
	\arrow["\dashv"{anchor=center, rotate=-90}, draw=none, from=2, to=3]
\end{tikzcd}\]
where $(\uvar)^\natural$ is the obvious forgetfull functor.

\p
We once again have an adjunction:
\[\begin{tikzcd}
	{i_!:\tiPsh{\Delta[\Theta]}} & {\tiPsh{\Theta}:i^*}
	\arrow[shift left=2, from=1-1, to=1-2]
	\arrow[shift left=2, from=1-2, to=1-1]
\end{tikzcd}\]
induced by the canonical inclusion $t\Delta[t\Theta]\to t\Theta$.
For an integer $n$, we define the functor \sym{((b40@$(\uvar)^{\sharp_n}$}$(\uvar)^{\sharp_n}:\iPsh{\Theta}\to \tiPsh{\Theta}$ and $(\uvar)^{\sharp_n}:\iPsh{\Delta[\Theta]}\to \tiPsh{\Delta[\Theta]}$ sending a $\infty$-presheaf $X$ onto $ (X, (X^n_k)_{k>0})$ where $X^n_k:= \Ib(X_{k-1})$ if $k<n$, and $X^n_k:=X_k$ if not. We eventually set \sym{(tw@$\Wm$}\sym{(tm@$\Mm$}
$$\Wm:= \coprod_{n}(\Wseg)^{\sharp_n}\coprod (\Wsat)^\flat~~~~~\Mm:= \coprod_{n}(\Mseg)^{\sharp_n}\coprod(\Msat)^\flat$$
As $i_!(\Mm)$ is contained in $\Wm$, the previous adjunction induces a derived one:
\begin{equation}
\label{eq:derived marked adjunction theta and delta theta}
\begin{tikzcd}
	{\Lb i_!:\tiPsh{\Delta[\Theta]}_{\Mm}} & {\tiPsh{\Theta}_{\Wm}:i^*\Rb}
	\arrow[""{name=0, anchor=center, inner sep=0}, shift left=2, from=1-1, to=1-2]
	\arrow[""{name=1, anchor=center, inner sep=0}, shift left=2, from=1-2, to=1-1]
	\arrow["\dashv"{anchor=center, rotate=-90}, draw=none, from=0, to=1]
\end{tikzcd}
\end{equation}

\begin{prop}
\label{prop:derived marked adjunction theta and delta theta}
The derived adjunction \eqref{eq:derived marked adjunction theta and delta theta}
is an adjoint equivalence.
\end{prop}
\begin{proof}
It is enough to show that for any element $a:t\Delta[t\Theta]$ and any $b:t\Theta$, $a\to i^*i_!a$ and $i_!i^*b\to b$ are respectively in $\widehat{\Mm}$ and $\widehat{\Wm}$. If $a$ is of shape $[b,n]^\flat$, this is a direct consequence of proposition \ref{prop:infini changing theta}, and if $a$ is $(\Db_n)_t$ the unit is the identity. We proceed similarly for $i_!i^*b\to b$.
\end{proof}
 The inclusion $t\Theta\to \zocatm$ induces an adjunction
\[\begin{tikzcd}
	{\tPsh{\Theta}} & \zocatm
	\arrow[""{name=0, anchor=center, inner sep=0}, shift left=2, from=1-1, to=1-2]
	\arrow[""{name=1, anchor=center, inner sep=0}, shift left=2, from=1-2, to=1-1]
	\arrow["\dashv"{anchor=center, rotate=-90}, draw=none, from=0, to=1]
\end{tikzcd}\]
and we can easily check that this induces an equivalence between $\zocatm$ and the sub-category of $\tPsh{\Theta}$ of $\Wm$-local objects.
Together with proposition \ref{prop:derived marked adjunction theta and delta theta}, this induces equivalences
$$\tPsh{\Theta}_{\Mm} \cong \tPsh{\Delta[\Theta]}_{\Wm}\cong \zocatm$$

\p A \notion{marked $\io$-category} is a $\Wm$-local stratified $\infty$-presheaves on $\Theta$. We denote by \wcnotation{$\ocatm$}{((a70@$\ocatm$} the $\iun$-category of marked $\io$-categories.
Unfolding the definition, a marked $\io$-category is a pair $(C,tC)$ where $C$ is an $\io$-category and $tC:=(tC_n)_{n>0}$ is a sequence of full sub $\infty$-groupoids of $C_n$, containing identities, stable by composition and by whiskering with cells of lower dimension.
A $n$-cell $a:\Db_n\to (C,tC)$ is \wcsnotion{marked}{marked $n$-cell}{for marked $\io$-categories} if it belongs to the image of $tC_n$.

There are two obvious ways to mark a $\io$-category. For $C\in \ocat$, we define $C^\sharp := (C,(C_n)_{n>0})$ and $C^\flat := (C,(\Ib(C_{n-1})_{n>0}))$. The first one corresponds to the case where all cells are marked, and the second one where only the identities are marked. These two functors fit in the following adjoint triple:
\[\begin{tikzcd}[column sep=0.7cm]
	{(\uvar)^{\flat}: \ocat} & {\ocatm:(\uvar)^\natural} & {(\uvar)^\natural:\ocatm} & { \ocat	:(\uvar)^\sharp}
	\arrow[""{name=0, anchor=center, inner sep=0}, shift left=2, from=1-1, to=1-2]
	\arrow[""{name=1, anchor=center, inner sep=0}, shift left=2, from=1-2, to=1-1]
	\arrow[""{name=2, anchor=center, inner sep=0}, shift left=2, from=1-3, to=1-4]
	\arrow[""{name=3, anchor=center, inner sep=0}, shift left=2, from=1-4, to=1-3]
	\arrow["\dashv"{anchor=center, rotate=-90}, draw=none, from=0, to=1]
	\arrow["\dashv"{anchor=center, rotate=-90}, draw=none, from=2, to=3]
\end{tikzcd}\]
where $(\uvar)^\natural$ is the obvious forgetful functor.	
To simplify notations, for a marked $\io$-category $C$, the marked $\io$-categories $(C^\natural)^\flat$ and $(C^\natural)^\sharp$ will be simply denoted by $C^\flat$ and $C^\sharp$.

\p Following paragraph \ref{para:dualities non strict case}, for any subset $S$ of $\Nb^*$, we define the duality\ssym{((b49@$(\uvar)^S$}{for marked $\io$-categories}
$$(\uvar)^S:\ocatm\to \ocatm$$
whose value on $(C,tC)$ is $(C^S,tC)$.
In particular, we have the \snotionsym{odd duality}{((b60@$(\uvar)^{op}$}{for marked $\io$-categories} $(\uvar)^{op}$, corresponding to the set of odd integer, the \snotionsym{even duality}{((b50@$(\uvar)^{co}$}{for marked $\io$-categories} $(\uvar)^{co}$, corresponding to the subset of non negative even integer, the \snotionsym{full duality}{((b80@$(\uvar)^{\circ}$}{for marked $\io$-categories} $(\uvar)^{\circ}$, corresponding to $\Nb^*$ and the \snotionsym{transposition}{((b70@$(\uvar)^t$}{for marked $\io$-categories} $(\uvar)^t$, corresponding to the singleton $\{1\}$. Eventually, we have equivalences
$$((\uvar)^{co})^{op}\sim (\uvar)^{\circ} \sim ((\uvar)^{op})^{co}.$$

\p Given a functor $F:I\to \ocatm$, the colimit of $F$ is given by the marked $\io$-category $(C,tC)$ with 
$$C:=\colim_{I}F^\natural$$
and for any $n$, $(tC)_n$ is the image of the morphism 
$$\colim_I tF_n\to (\colim_{I}F)^\natural_n.$$
The case of the limit is easier as we have 
$$\lim_{I}F := (\lim_{I}F^\natural,(\lim_{I}(tF_n)_{n>0}).$$
In particular, if $(C,tC)$ and $(D,tD)$ are two marked $\io$-categories, we have
$$(C,tC)\times (D,tD):= (C\times D, (tC_n\times tD_n)_{n>0}).$$
\begin{prop}
\label{prop:cartesian product preserves W marked version}
The cartesian product in $\ocatm$ preserves colimits in both variables.
\end{prop}
\begin{proof}
Let $F:I\to \ocatm$ be a diagram and $C$ a marked $\io$-category. The underlying $\io$-categories of $\colim_I (F\times C)$ and $(\colim_IF)\times C$ are the same as the cartesian product preserves colimits in $\ocat$. The equivalence of the two markings 
 is a direct consequence of the fact that the cartesian product in $\igrd$ preserves both colimits and the formation of image.
\end{proof}
This demonstrates the existence of an internal hom functor that we denote once again by $\uHom(\uvar,\uvar)$.

\p We denote again $\pi_0:\tiPsh{\Theta}\to \tPsh{\Theta}$ colimit preserving sending a stratified $\infty$-presheaf $X$ to the stratified presheaf $a\mapsto \pi_0(X_a)$. As this functor preserves $\Wm$, it induces an adjoint pair:
\sym{(pi@$\pi_0:\ocatm\to \zocatm$}\sym{n@$\N:\zocatm\to \ocatm$}
\[\begin{tikzcd}
	{\pi_0:\ocat} & {\zocat:\N}
	\arrow[""{name=0, anchor=center, inner sep=0}, shift left=2, from=1-1, to=1-2]
	\arrow[""{name=1, anchor=center, inner sep=0}, shift left=2, from=1-2, to=1-1]
	\arrow["\dashv"{anchor=center, rotate=-90}, draw=none, from=0, to=1]
\end{tikzcd}\]
where the right adjoint $\N$ is fully faithful.
A marked $\io$-category lying in the image of the nerve is called \wcnotion{strict}{strict marked $\io$-category}.
Remark eventually that the following square is cartesian
\[\begin{tikzcd}
	\zocatm & \ocatm \\
	\zocat & \ocat
	\arrow["\N", from=1-1, to=1-2]
	\arrow["{(\uvar)^\natural}", from=1-2, to=2-2]
	\arrow["{(\uvar)^\natural}"', from=1-1, to=2-1]
	\arrow["\N"', from=2-1, to=2-2]
\end{tikzcd}\]
A marked $\io$-category is then strict if and only if it's underlying $\io$-category is.

\p The \wcsnotionsym{marked suspension}{((d60@$[\uvar,1]$}{suspension}{for marked $\io$-categories} is the colimit preserving functor $$[\uvar,1]:\ocatm\to \ocatm_{\bullet,\bullet}$$ sending $a^\flat$ onto $[a,1]^\flat$ and $(\Db_n)_t$ to $([\Db_n,1])_t$. 
It then admits a right adjoint: \ssym{(hom@$\hom_{\uvar}(\uvar,\uvar)$}{for marked $\io$-categories}
$$\begin{array}{lll}
\ocatm_{\bullet,\bullet}&\to& \ocatm\\
(C,a,b)&\mapsto &\hom_C(a,b)
\end{array}
$$

With the same computation than the one of paragraph \ref{para:wiskering}, we show that for a marked $\io$-category $C$, any $1$-cell $f:x\to x'$ induces for any object $y$, a morphism
$$f_!:\hom_C(x',y)\to \hom_C(x,y).$$
Conversely, a $1$-cell $g:y\to y'$ induces for any object $x$ a morphism
$$g_!:\hom_C(x,y)\to \hom_C(x,y')$$

\p In section \ref{section:iocategories}, we define the notion of fully faithful morphism of $\io$-categories. There is an equivalent notion for marked $\io$-categories:
\begin{definition}
A morphism $f:C\to D$ is \snotion{fully faithful}{for marked $\io$-categories} if for any pair of objects $x,y$, the morphism of marked $\io$-categories $\hom_C(x,y)\to \hom_D(fx,fy)$ is an equivalence, and if a $1$-cell $v$ is marked whenever $f(v)$ is.
\end{definition}

 We now give some adaptation of the result on fully faithful functors to the case of marked $\io$-categories without proofs, as they are obvious modifications to this new framework.

\begin{prop}
\label{prop:ff 1 marked case}
A morphism is fully faithful if and only if it has the unique right lifting property against $\emptyset\to \Db_n$ and $\Db_n\to (\Db_n)_t$ for $n>0$.
\end{prop}

\begin{prop}
\label{prop:ff 2 marked case}
Fully faithful morphisms are stable under limits.
\end{prop}

\begin{prop}
\label{prop:fully faithful plus surjective on objet marked case}
A morphism $f:C\to D$ is an equivalence if and only if it is fully faithful and surjective on objects.
\end{prop}

\p A morphism $f:C\to D$ between marked $\io$-categories is a \snotion{discrete Conduché functor}{for marked $\io$-categories} if for any triplet of integers $k< n\leq m$, $f$ has the unique right lifting property against $$\Ib_{m+1}:\Db_{m+1}^\flat\to \Db_{m}^\flat ~~\mbox{ and }~~\triangledown^{\sharp_n}_{k,m}:\Db_{m}^{\sharp_n}\to \Db_{m}^{\sharp_n}\coprod_{ \Db_{k}^\flat} \Db_{m}^{\sharp_n}.$$

\begin{example}
If $f$ is a discrete Conduché functor between marked $\io$-categories, $f^\sharp$ is a discrete Conduché functor. Conversely, if $g$ is a discrete Conduché functor between $\io$-categories, so are $g^\sharp$, $g^\flat$ and $g^{\sharp_n}$ for any integer $n$.
\end{example}

\p
A \notion{marked globular sum} is a marked $\io$-category whose underlying $\io$-category is a globular sum and such that for any pair of integers $k\leq n$, and any pair of $k$-composable $n$-cells $(x,y)$, $x\circ_k y$ is marked if and only if $x$ and $y$ are marked.

A morphism $i:a\to b$ between marked globular sum is \wcsnotion{globular}{globular morphism}{for marked $\zo$-categories} if the morphism $i^\natural$ is globular.

The proposition \ref{prop:algebraic ortho to globular} implies that a morphism $a\to b$ between marked globular sums is a discrete Conduché functor if and only if it is globular.

\begin{lemma}
\label{lemma:pullback by conduch marked preserves colimitpre}
Let $p:C\to D^\flat$ be a discrete Conduché functor between marked $\io$-categories. The canonical morphism $(C^\natural)^\flat\to C$ is an equivalence. 
\end{lemma}
\begin{proof}
Suppose given a marked $n$-cell $v:\Db_n\to C^\natural$. As the marking on $C$ is trivial, this induces a commutative square
\[\begin{tikzcd}
	{\Db_n} & { C^\natural} \\
	{\Db_{n-1}} & D
	\arrow[from=1-1, to=2-1]
	\arrow["v", from=1-1, to=1-2]
	\arrow["{p^\natural}", from=1-2, to=2-2]
	\arrow[from=2-1, to=2-2]
	\arrow["l"{description}, dashed, from=2-1, to=1-2]
\end{tikzcd}\]
that admits a lift $l$ as $p^\natural$ is a discrete Conduché functor, which concludes the proof.
\end{proof}

\begin{prop}
\label{prop:pullback by conduch marked preserves colimit}
Let $p:C\to D$ be a discrete Conduché functor between marked $\io$-categories. The pullback functor $p^*$ preserves colimits.
\end{prop}
\begin{proof}
As $\tiPsh{\Theta}$ is locally cartesian closed, one has to show that for any pair of cartesian squares
\[\begin{tikzcd}
	{C''} & {C'} & C \\
	{D''} & {D'} & D
	\arrow["p", from=1-3, to=2-3]
	\arrow["i"', from=2-1, to=2-2]
	\arrow["j", from=1-1, to=1-2]
	\arrow[from=1-2, to=1-3]
	\arrow[from=2-2, to=2-3]
	\arrow[from=1-2, to=2-2]
	\arrow[from=1-1, to=2-1]
	\arrow["\lrcorner"{anchor=center, pos=0.125}, draw=none, from=1-1, to=2-2]
	\arrow["\lrcorner"{anchor=center, pos=0.125}, draw=none, from=1-2, to=2-3]
\end{tikzcd}\]
if $i$ is $\Wm$, then $j$ is in $\widehat{\Wm}$. Suppose first that $i$ is in $\Wsat^\flat$. According of the lemma \ref{lemma:pullback by conduch marked preserves colimitpre} the $\io$-categories $C'$ and $C''$ are  of shape $(E)^\flat$ and $(E')^\flat$ for $E$ and $E'$ two $\io$-categories. The proposition \ref{prop:pulback of Wsat} then implies that $i$ is in $\widehat{\W^\flat}\subset \widehat{\Wm}$. If $i$ is in $(\Wseg)^{\sharp_n}$ the proof is an easy adaptation of the one of lemma \ref{lemma:conduche preserves W}.
\end{proof}

\p 
 We now give some adaptation of the result on special colimits stated in paragraph \ref{para: spetial colimits} to the case of marked $\io$-categories without proofs, as they are easy modifications.

We denote by $\iota$ the inclusion of $\ocatm$ into $\tiPsh{\Theta}$.
A functor $F:I\to \ocatm$ has a \snotion{special colimit}{for marked $\io$-categories} if the canonical morphism 
\begin{equation}
\label{eq:special colimit marked case}
\colim_{i:I}\iota F(i)\to \iota(\colim_{i:I}F(i))
\end{equation}
is an equivalence of stratified presheaves. 

Similarly, we say that a functor $\psi: I\to \Arr(\ocatm)$ has a \textit{special colimit} if the canonical morphism 
$$\colim_{i:I}\iota \psi(i)\to \iota(\colim_{i:I}\psi(i))$$
is an equivalence in the arrow $\iun$-category of $\tiPsh{\Theta}$.

\begin{example}
Let $C$ be a marked $\io$-category. The canonical diagram $t\Theta_{/C}\to \ocat$ has a special colimit, given by $C$.
\end{example}
\begin{prop}
\label{prop:special colimit marked case}
Let $F,G:I\to \ocatm$ be two functors, and $\psi:F\to G$ a natural transformation. If $\psi$ is cartesian, and $G$ has a special colimit, then $\psi$ and $F$ have special colimits. 
\end{prop}

\begin{prop}
\label{prop:example of a special colimit marked case}
For any integer $n$, and element $a\in t\Theta$ and $b\in \Theta$, the equalizer diagram 
\[\begin{tikzcd}
	{\coprod_{k+l=n-1}[a,k]\vee[a\times b^\sharp,1]\vee[a,l]} & {\coprod_{k+l=n}[a,k]\vee[ b,1]^\sharp\vee[a,l]}
	\arrow[shift left=2, from=1-1, to=1-2]
	\arrow[shift right=2, from=1-1, to=1-2]
\end{tikzcd}\]
where the top diagram is induced by $[a\times b^\sharp,1]\to [a,1]\vee[b,1]^\sharp$ and to bottom one by $[a\times b^\sharp,1]\to [b,1]^\sharp\vee[a,1]$,
has a special colimit, which is $[a,n]\times [b,1]^\sharp$.
\end{prop}

\begin{prop}
\label{prop:example of a special colimit 2 marked case}
Any sequence of marked $\io$-categories has a special colimit. 
\end{prop}

\begin{prop}
\label{prop:example of a special colimit4 marked case}
Suppose given a cartesian square
\[\begin{tikzcd}
	{ B} & C \\
	{\{0\}} & {[1]^\sharp}
	\arrow[from=1-2, to=2-2]
	\arrow[from=1-1, to=1-2]
	\arrow[from=1-1, to=2-1]
	\arrow[from=2-1, to=2-2]
	\arrow["\lrcorner"{anchor=center, pos=0.125}, draw=none, from=1-1, to=2-2]
\end{tikzcd}\]
The diagram 
\[\begin{tikzcd}
	{[1]^\sharp\vee[B,1]} & {[B,1]} & {[C,1]}
	\arrow["\triangledown"', from=1-2, to=1-1]
	\arrow[from=1-2, to=1-3]
\end{tikzcd}\]
has a special colimit.
\end{prop}

\begin{prop}
\label{prop:example of a special colimit3 marked case}
Suppose given two cartesian squares
\[\begin{tikzcd}
	{ B} & C & D \\
	{\{0\}} & {[1]^\sharp} & {\{1\}}
	\arrow[from=1-2, to=2-2]
	\arrow[from=1-1, to=1-2]
	\arrow[from=1-3, to=2-3]
	\arrow[from=1-1, to=2-1]
	\arrow[from=1-3, to=1-2]
	\arrow[from=2-1, to=2-2]
	\arrow[from=2-3, to=2-2]
	\arrow["\lrcorner"{anchor=center, pos=0.125}, draw=none, from=1-1, to=2-2]
	\arrow["\lrcorner"{anchor=center, pos=0.125, rotate=-90}, draw=none, from=1-3, to=2-2]
\end{tikzcd}\]
The diagram 
\[\begin{tikzcd}
	{[1]^\sharp\vee[B,1]} & {[B,1]} & {[C,1]} & {[D,1]} & {[D,1]\vee[1]^\sharp}
	\arrow["\triangledown", from=1-4, to=1-5]
	\arrow["\triangledown"', from=1-2, to=1-1]
	\arrow[from=1-2, to=1-3]
	\arrow[from=1-4, to=1-3]
\end{tikzcd}\]
has a special colimit.
\end{prop}

\subsection{Gray tensor product of marked $\io$-categories}
We define the \wcsnotion{marked Gray tensor product}{Gray tensor product}{for marked $\io$-categories}
$$\uvar\otimes (\uvar)^\sharp:\ocatm\times \icat \to \ocatm$$
sending a marked $\io$-category $C$ and a $\iun$-category $K$ to the marked $\io$-category $C\otimes K^\sharp$, such that $(C\otimes K^\sharp)^\natural$ fits in the cocartesian square
\[\begin{tikzcd}
	{\coprod_{ tC}\Db_n\otimes K} & {C^\natural\otimes K} \\
	{\coprod_{ tC}\tau^i_n(\Db_n\otimes K)} & {(C\otimes K^\sharp)^\natural}
	\arrow[from=1-1, to=1-2]
	\arrow[from=1-1, to=2-1]
	\arrow[from=2-1, to=2-2]
	\arrow[from=1-2, to=2-2]
	\arrow["\lrcorner"{anchor=center, pos=0.125, rotate=180}, draw=none, from=2-2, to=1-1]
\end{tikzcd}\]
and such that $t(C\otimes K^\sharp)_n$ consists of $n$-cells lying in the image of the morphism 
$$\tau_{n-1}C\otimes K\coprod (tC)_n\otimes K_0\to (C\otimes K^\sharp)^\natural.$$

\begin{prop}
\label{prop:otimes marked preserves colimits}
The functor $\uvar\otimes(\uvar)^\sharp:\ocatm\times \icat\to \ocatm$ preserves colimits. 
\end{prop}
\begin{proof}
By construction, we have two cocartesian squares:
\[\begin{tikzcd}
	{\coprod_{\colim tF }\Db_n\otimes K} & {\colim (F^\natural\otimes K)} \\
	{\coprod_{\colim tF }\tau^i_{n}(\Db_n\otimes K)} & {\colim (F\otimes K^\sharp)^\natural} \\
	{\coprod_{t(\colim F)}\Db_n\otimes K} & {(\colim F^\natural)\otimes K} \\
	{\coprod_{ t(\colim F)}\tau^i_{n}(\Db_n\otimes K)} & {((\colim F)\otimes K^\sharp)^\natural}
	\arrow[from=1-1, to=2-1]
	\arrow[from=3-1, to=4-1]
	\arrow[from=1-2, to=2-2]
	\arrow[from=3-2, to=4-2]
	\arrow[""{name=0, anchor=center, inner sep=0}, from=3-1, to=3-2]
	\arrow[""{name=1, anchor=center, inner sep=0}, from=1-1, to=1-2]
	\arrow[from=2-1, to=2-2]
	\arrow[from=4-1, to=4-2]
	\arrow["\lrcorner"{anchor=center, pos=0.125, rotate=180}, draw=none, from=2-2, to=1]
	\arrow["\lrcorner"{anchor=center, pos=0.125, rotate=180}, draw=none, from=4-2, to=0]
\end{tikzcd}\]
By the preservation of colimit by the Gray tensor product for $\io$-categories and by the functor $(\uvar)^\natural$, we have an equivalence 
$$\colim (F^\natural\otimes K)\sim (\colim F^\natural)\otimes K$$
However, the canonical morphism $\colim tF \to t(\colim F)$ is an epimorphism, and according to proposition \ref{prop:truncation of epimorphism is pushout}, the following canonical square
is cocartesian
\[\begin{tikzcd}
	{\coprod_{\colim tF }\Db_n\otimes K} & {\coprod_{t(\colim F)}\Db_n\otimes K} \\
	{\coprod_{\colim tF }\tau^i_{n}(\Db_n\otimes K)} & {\coprod_{ t(\colim F)}\tau^i_{n}(\Db_n\otimes K)}
	\arrow[from=1-1, to=2-1]
	\arrow[from=1-1, to=1-2]
	\arrow[from=2-1, to=2-2]
	\arrow[from=1-2, to=2-2]
\end{tikzcd}\]
Combined with the first two cocartesian squares, this implies that that $\colim (F\otimes K^\sharp)^\natural$ and $((\colim F)\otimes K^\sharp)^\natural$ are equivalent.

According to proposition \ref{prop:intelignet truncation is poitwise an epi} and by construction, the morphisms
 $$\colim (\tau^i_n F^\natural\otimes K)\to \tau^i_n(\colim F^\natural\otimes K)~~\mbox{ and }~~\colim (tF\otimes K_0) \to t(\colim F\otimes K_0)$$ are epimorphisms. The marked $\io$-categories $\colim (F\otimes K^\sharp)$ and $(\colim F)\otimes K^\sharp$
 then have the same marked cells.
\end{proof}

\begin{prop}
\label{prop:associativity of Gray amput}
Let $C$ be a $\io$-category, $D$ a marked $\io$-category and $K,L$ two $(\infty,1)$-categories. 
\begin{enumerate}
\item The underlying $\io$-category of $C^\flat\otimes K^\sharp$ is $C\otimes K$.
\item The canonical morphism $C^\sharp\otimes K^\sharp\to C^\sharp\times K^\sharp$ is an equivalence.
\end{enumerate}
\end{prop}
\begin{proof}
The first assertion is obvious.

Let $a$ be a globular sum and $[k]$ an object of $\Delta$.
We claim that the following two squares are cocartesian:
\[\begin{tikzcd}
	{\coprod_n\coprod\limits_{\Db_n\to a}\Db_n\otimes [k]} & {\coprod_{n}\tau_n a\otimes [k]} & {a\otimes[k]} \\
	{\coprod_n\coprod\limits_{\Db_n\to a}\tau^i_{n}(\Db_n\otimes [k])} & {\coprod_{n}\tau^i_{n}(\tau_n a\otimes [k])} & {(a^\sharp\times [k]^\sharp)^\natural}
	\arrow[from=1-1, to=2-1]
	\arrow[from=1-2, to=2-2]
	\arrow[from=2-1, to=2-2]
	\arrow[from=1-1, to=1-2]
	\arrow[from=1-3, to=2-3]
	\arrow[from=2-2, to=2-3]
	\arrow[from=1-2, to=1-3]
\end{tikzcd}\]
The cocartesianess of the left square is a consequence of propositions \ref{prop:truncation of epimorphism is pushout} and \ref{prop:canonical epi}. The outer square is cocartesian by definition, and by left cancellation, this implies the cocartesianess of the right square.
The lemma \ref{lemma:technique marked oicategoros} then implies that the underlying category of $a^\sharp\otimes [k]^\sharp$ is $a\times[k]$. As every cell of $a^\sharp\otimes[k]^\sharp$ is marked, this concludes the proof of the second assertion.
\end{proof}

\begin{prop}
\label{prop:associativity of Gray2}
Let $D$ be an $\io$-category, $C$ a marked $\io$-category and $K$ an $\iun$-category.
The canonical morphism
$(D^\sharp\times C)\otimes K^\sharp\to D^\sharp\times (C\otimes K^\sharp)$ is an equivalence.
\end{prop}
\begin{proof}
As $\times$ and $\otimes$ preserve colimits, we can reduce to the case where $D$ is an element of $\Theta$, $C$ of $t\Theta$ and $K$ of $\Delta$, and we proceed by induction on the dimension of $D$. Remark first that if $D$ is $[0]$, the result is obvious, and if it is $(\Db_1)_t$, the result follows from the second assertion of proposition \ref{prop:associativity of Gray amput}. Suppose then the result is true at the stage $n$. Using once again the fact that $\times$ and $\otimes$ preserve colimits, we can reduce to the case where
 $D^\sharp$ is $[a,1]^\sharp$, $C$ is $[b,1]$ with $b$ an element of $\Theta_t$ of dimension $n$, and $K^\sharp$ is $[1]^\sharp$.
 
 The formula given in proposition \ref{prop:example of a special colimit marked case} implies that $([a,1]^\sharp\times[b,1])\otimes[1]^\sharp$ is the colimit of the sequence: 
\begin{equation}
\label{eq:prop:associativity of Gray2}
\begin{tikzcd}
	{([a,1]^\sharp\vee [b,1])\otimes[1]^\sharp} & {[a^\sharp\times b,1]\otimes[1]^\sharp} & {([b,1]\vee [a,1]^\sharp)\otimes[1]^\sharp}
	\arrow[from=1-2, to=1-1]
	\arrow[from=1-2, to=1-3]
\end{tikzcd}
\end{equation}
The marked $\io$-category $([a,1]^\sharp\vee [b,1])\otimes[1]^\sharp$ is then the colimit of the diagram 
\[\begin{tikzcd}
	{[a,1]^\sharp\times [1]^\sharp} & {[1]^\sharp} & {[b,1]\otimes[1]^\sharp}
	\arrow[from=1-2, to=1-1]
	\arrow[from=1-2, to=1-3]
\end{tikzcd}\]
and using the formulas \eqref{eq:eq for cylinder marked version} and \ref{prop:example of a special colimit marked case}, $([a,1]^\sharp\vee [b,1])\otimes[1]^\sharp$ is the colimit of the diagram
\[\begin{tikzcd}
	{[1]^\sharp\vee[a,1]^\sharp\vee[b,1]} & {[a,1]^\sharp\vee[b,1]} & {[a,1]^\sharp\vee[1]^\sharp\vee[b,1]} \\
	&& {[a,1]^\sharp\vee[b\otimes\{0\},1]} \\
	&& {[a,1]^\sharp\vee[b\otimes[1]^\sharp,1]} \\
	&& {[a,1]^\sharp\vee[b\otimes\{1\},1]} \\
	&& {[a,1]^\sharp\vee[b,1]\vee[1]^\sharp}
	\arrow[from=4-3, to=5-3]
	\arrow[from=4-3, to=3-3]
	\arrow[from=2-3, to=3-3]
	\arrow[from=2-3, to=1-3]
	\arrow[from=1-2, to=1-3]
	\arrow[from=1-2, to=1-1]
\end{tikzcd}\]
Similarly, $([b,1]\vee [a,1]^\sharp)\otimes[1]^\sharp$ is the colimit of the diagram
\[\begin{tikzcd}
	{[1]^\sharp\vee[b,1]\vee[a,1]^\sharp} \\
	{[b\otimes\{0\},1]\vee[a,1]^\sharp} \\
	{[b\otimes[1]^\sharp,1]\vee[a,1]^\sharp} \\
	{[b\otimes\{1\},1]\vee[a,1]^\sharp} \\
	{[b,1]\vee[1]^\sharp\vee[a,1]^\sharp} & {[b,1]\vee[a,1]} & {[b,1]\vee[a,1]^\sharp\vee[1]^\sharp}
	\arrow[from=2-1, to=1-1]
	\arrow[from=2-1, to=3-1]
	\arrow[from=4-1, to=3-1]
	\arrow[from=4-1, to=5-1]
	\arrow[from=5-2, to=5-1]
	\arrow[from=5-2, to=5-3]
\end{tikzcd}\]
Eventually, the formulas \eqref{eq:eq for cylinder marked version} and the induction hypothesis imply that $[a^\sharp\times b,1]\otimes[1]^\sharp$ is the colimit of the diagram
\[\begin{tikzcd}
	{[1]^\sharp\vee[a^\sharp\times b,1]} \\
	& {[a^\sharp\times b\otimes\{0\},1]} \\
	& {[a^\sharp\times (b\otimes[1]^\sharp),1]} \\
	& {[a^\sharp\times b\otimes\{1\},1]} \\
	&& {[a^\sharp\times b,1]\vee[1]^\sharp}
	\arrow[from=2-2, to=1-1]
	\arrow[from=2-2, to=3-2]
	\arrow[from=4-2, to=3-2]
	\arrow[from=4-2, to=5-3]
\end{tikzcd}\]
As all these colimits are special and composed of monomorphisms, the objects $([a,1]^\sharp\vee [b,1])\otimes[1]^\sharp$, $([b,1]^\sharp\vee [a,1]^\sharp)\otimes[1]^\sharp$ and $[a^\sharp\times b,1]\otimes[1]^\sharp$ are strict. As the colimit \eqref{eq:prop:associativity of Gray2} is also special, $([a,1]^\sharp\times[b,1])\otimes[1]^\sharp$ is strict.

All put together, $([a,1]^\sharp\times[b,1])\otimes[1]^\sharp$ is the colimit of the diagram
\[\begin{tikzcd}[column sep =0.1cm]
	{[1]^\sharp\vee[b,1]\vee[a,1]^\sharp} & {[1]^\sharp\vee[a^\sharp\times b,1]} & {[1]^\sharp\vee[a,1]^\sharp\vee[b,1]} & {[a,1]^\sharp\vee[b,1]} & {[a,1]^\sharp\vee[1]^\sharp\vee[b,1]} \\
	{[b\otimes\{0\},1]\vee[a,1]^\sharp} && {[a^\sharp\times b\otimes\{0\},1]} && {[a,1]^\sharp\vee[b\otimes\{0\},1]} \\
	{[b\otimes[1]^\sharp,1]\vee[a,1]^\sharp} && {[a^\sharp\times (b\otimes[1]^\sharp),1]} && {[a,1]^\sharp\vee[b\otimes[1]^\sharp,1]} \\
	{[b\otimes\{1\},1]\vee[a,1]^\sharp} && {[a^\sharp\times b\otimes\{1\},1]} && {[a,1]^\sharp\vee[b\otimes\{1\},1]} \\
	{[b,1]\vee[1]^\sharp\vee[a,1]^\sharp} & {[b,1]\vee[a,1]} & {[b,1]\vee[a,1]^\sharp\vee[1]^\sharp} & {[a^\sharp\times b,1]\vee[1]^\sharp} & {[a,1]^\sharp\vee[b,1]\vee[1]^\sharp}
	\arrow[from=4-5, to=5-5]
	\arrow[from=4-5, to=3-5]
	\arrow[from=2-5, to=3-5]
	\arrow[from=2-5, to=1-5]
	\arrow[from=5-4, to=5-5]
	\arrow[from=5-4, to=5-3]
	\arrow[from=2-3, to=2-5]
	\arrow[from=2-3, to=1-4]
	\arrow[from=1-4, to=1-5]
	\arrow[from=4-3, to=5-4]
	\arrow[from=4-3, to=3-3]
	\arrow[from=2-3, to=3-3]
	\arrow[from=3-3, to=3-5]
	\arrow[from=4-3, to=4-5]
	\arrow[from=1-4, to=1-3]
	\arrow[from=2-3, to=2-1]
	\arrow[from=3-3, to=3-1]
	\arrow[from=4-3, to=4-1]
	\arrow[from=4-1, to=3-1]
	\arrow[from=2-1, to=3-1]
	\arrow[from=1-2, to=1-3]
	\arrow[from=5-2, to=5-3]
	\arrow[from=5-2, to=5-1]
	\arrow[from=1-2, to=1-1]
	\arrow[from=4-1, to=5-1]
	\arrow[from=2-1, to=1-1]
	\arrow[from=2-3, to=1-2]
	\arrow[from=4-3, to=5-2]
\end{tikzcd}\]

Now, using the formula given in proposition \ref{prop:example of a special colimit marked case}, and taking the colimit line by line of the previous diagram, $([a,1]^\sharp\times[b,1])\otimes[1]^\sharp$ is the colimit of the diagram 
\[\begin{tikzcd}
	{[a,1]^\sharp\times([1]^\sharp\vee[b,1])} \\
	{[a,1]^\sharp\times[b\otimes\{0\},1]} \\
	{[a,1]^\sharp\times[b\otimes[1]^\sharp,1]} \\
	{[a,1]^\sharp\times[b\otimes\{1\},1]} \\
	{[a,1]^\sharp\times([b,1]\vee[1]^\sharp)}
	\arrow[from=2-1, to=1-1]
	\arrow[from=2-1, to=3-1]
	\arrow[from=4-1, to=3-1]
	\arrow[from=4-1, to=5-1]
\end{tikzcd}\]
Using for the last times formula \eqref{eq:eq for cylinder marked version}, $([a,1]^\sharp\times[b,1])\otimes[1]^\sharp$ is equivalent to $[a,1]^\sharp\times([b,1]\otimes[1]^\sharp)$.
\end{proof}

\begin{prop}
 \label{prop:associativity of Gray amput2}
Let $D$ be a marked $\io$-category and $K,L$ two $(\infty,1)$-categories. 
There is a natural equivalence
$(D\otimes K^\sharp)\otimes L^\sharp\to D\otimes(K\times L)^\sharp$.
\end{prop}
\begin{proof}
Suppose first that $D$ is of shape $C^\flat$.
The proposition \ref{prop:canonical epi} implies that $\coprod_{t(C^\flat\otimes K^\sharp)^\natural}\Db_n\to (C^\flat\otimes K^\sharp)^\natural$ and $(\coprod_n \tau_{n-1}C\otimes K) \to (C^\flat\otimes K^\sharp)^\natural$ have the same image. The proposition \ref{prop:truncation of epimorphism is pushout}, then implies that
the underlying $\io$-category of $(C^\flat\otimes K^\sharp)\otimes L^\sharp$ fits in the cocartesian square
\[\begin{tikzcd}
	{\coprod_n \tau_{n-1}C\otimes K\otimes L} & {C\otimes K\otimes L} \\
	{\coprod_n \tau^i_n(\tau_{n-1}C\otimes K\otimes L)} & {((C^\flat\otimes K^\sharp)\otimes L^\sharp)^\natural}
	\arrow[from=1-1, to=2-1]
	\arrow[from=1-1, to=1-2]
	\arrow[from=1-2, to=2-2]
	\arrow[from=2-1, to=2-2]
\end{tikzcd}\]
The second assertion of lemma \ref{lemma:technique marked oicategoros} then implies that $((C^\flat\otimes K^\sharp)\otimes L^\sharp)^\natural$ is equivalent to $C^\flat\otimes (K\times L)$.
For a general marked $\io$-category $D$, the underlying $\io$-category of $(D^\flat\otimes K^\sharp)\otimes L^\sharp$ then fits by construction in the cocartesian square
\[\begin{tikzcd}
	{\coprod_n\coprod_{tD_n}\Db_n\otimes K_0\otimes L\amalg\Db_n\otimes K\otimes L} & {D^\flat\otimes (K\times L)} \\
	{\coprod_n\coprod_{tD_n}\tau^i_n(\Db_n\otimes K_0\otimes L)\amalg\tau^i_n(\Db_n\otimes K)\otimes L} & {((D\otimes K^\sharp)\otimes L^\sharp)^\natural}
	\arrow[from=1-1, to=1-2]
	\arrow[from=1-1, to=2-1]
	\arrow[from=2-1, to=2-2]
	\arrow[from=1-2, to=2-2]
\end{tikzcd}\]

Furthermore, the underlying $\io$-category of $D\otimes(K\times L)^\sharp$ fits in the cocartesian square
\[\begin{tikzcd}
	{\coprod_n\coprod_{tC_n}\Db_n\otimes (K\times L)} & {D^\flat\otimes (K\times L)} \\
	{\coprod_n\coprod_{tC_n}\tau^i_n(\Db_n\otimes (K\times L))} & {(D\otimes(K\times L)^\sharp)^\natural}
	\arrow[from=1-1, to=1-2]
	\arrow[from=2-1, to=2-2]
	\arrow[from=1-2, to=2-2]
	\arrow[from=1-1, to=2-1]
\end{tikzcd}\]
Using the canonical morphism $\tau^i_n(\Db_n\otimes K)\otimes L\to  \tau^i_n (\Db_n\otimes K \otimes L)\to \tau^i_n(\Db_n\otimes (K\times L))$, we have a canonical morphism 
$$((D^\flat\otimes K^\sharp)\otimes L^\sharp)^\natural \to (D\otimes(K\times L)^\sharp)^\natural.$$
As all these functors preserves colimits, the full sub $\infty$-groupoid of elements $(D,K,L)$ of $\ocatm\times \icat\times \icat$ such that this comparison is an equivalence and preserves and detects marking is closed by colimits. It is then sufficient to show that it includes $([1]^\sharp,[1],[1])$ and $([a,1],[1],[1])$ for $a\in t\Theta$.
We can then proceed as in the proof of proposition \ref{prop:associativity of Gray2}, making these two objects explicit thanks to the equations given in paragraph \ref{paragrap: equation fullfill by cylinder and join marked version}. As the proof takes up a lot of space and is very similar to that of proposition \textit{op cit}, we leave it to the reader.
\end{proof}

\subsection{Gray operations on marked $\io$-categories}
\p
The Gray tensor product for marked $\io$-category restricts to a functor
$$\uvar\otimes[1]^\sharp:\ocatm \to\ocatm$$ called the \wcsnotion{marked Gray cylinder}{Gray cylinder}{for marked $\io$-categories}\sym{((d30@$\uvar\otimes[1]^\sharp$}.
We will denote by 
$$\begin{array}{rcl}
\ocatm&\to&\ocatm\\
C&\mapsto &C^{[1]^\sharp}
\end{array}$$
its right adjoint.\sym{(c@$C^{[1]^\sharp}$}
The equation \eqref{eq:liens entre Gray cylindre et suspension}, establishing a link between the suspension and the Gray cylinder implies that the following diagram is cocartesian for any $C:\ocat$:
\begin{equation}
\label{eq:liens entre Gray cylindre et suspension version marque}
\begin{tikzcd}
	{C^\flat\otimes\{0,1\}} & {C^\flat\otimes [1]^\sharp} \\
	{1\amalg 1} & {[C,1]^\sharp}
	\arrow[from=1-1, to=2-1]
	\arrow[from=1-1, to=1-2]
	\arrow["\lrcorner"{anchor=center, pos=0.125, rotate=180}, draw=none, from=2-2, to=1-1]
	\arrow[from=2-1, to=2-2]
	\arrow[from=1-2, to=2-2]
\end{tikzcd}
\end{equation}

\begin{prop}
\label{prop:otimes et op marked version}
There is diagram
\[\begin{tikzcd}
	{(C\otimes\{1\})^\circ} & {(C\otimes[1]^\sharp)^\circ} & {(C\otimes\{0\})^\circ} \\
	{C^\circ\otimes\{0\}} & {C^\circ\otimes[1]^\sharp} & {C^\circ\otimes\{1\}}
	\arrow["\sim", from=1-2, to=2-2]
	\arrow["\sim", from=1-3, to=2-3]
	\arrow["\sim", from=1-1, to=2-1]
	\arrow[from=1-3, to=1-2]
	\arrow[from=1-1, to=1-2]
	\arrow[from=2-1, to=2-2]
	\arrow[from=2-3, to=2-2]
\end{tikzcd}\]
natural in $C:\ocatm$,
where all vertical arrows are equivalences. There is an invertible natural transformation
$$ C\star 1\sim (1\costar C^{\circ})^\circ.$$
\end{prop}
\begin{proof}
The corollary \ref{cor:otimes et op} provides an invertible transformation 
$$(C^\natural\otimes [1])^\circ\sim (C^\natural)^\circ\otimes[1]$$
The first assertion then follows from the definition of the Gray tensor product for marked $\io$-categories.
The second assertion is a consequence of the definition of the marked Gray cone and $\circ$-cone.
\end{proof}

\begin{example}
In all the following diagrams, marked cells are represented by crossed-out arrows.

The object $\Db_1^\flat\otimes[1]^\sharp$ corresponds to the diagram
\[\begin{tikzcd}
	00 & 01 \\
	10 & 11
	\arrow[from=1-1, to=2-1]
	\arrow["{/}"{marking}, from=2-1, to=2-2]
	\arrow["{/}"{marking}, from=1-1, to=1-2]
	\arrow[from=1-2, to=2-2]
	\arrow["{/}"{marking}, shorten <=4pt, shorten >=4pt, Rightarrow, from=1-2, to=2-1]
\end{tikzcd}\]
the object $(\Db_1)^\sharp\otimes[1]^\sharp$ corresponds to the diagram
\[\begin{tikzcd}
	00 & 01 \\
	10 & 11
	\arrow["{/}"{marking}, from=1-1, to=2-1]
	\arrow["{/}"{marking}, from=2-1, to=2-2]
	\arrow["{/}"{marking}, from=1-1, to=1-2]
	\arrow["{/}"{marking}, from=1-2, to=2-2]
\end{tikzcd}\]
the object $\Db_2^\flat\otimes[1]^\sharp$ corresponds to the diagram
\[\begin{tikzcd}
	00 & 01 & 00 & 01 \\
	10 & 11 & 10 & 11
	\arrow["{/}"{marking}, from=1-1, to=1-2]
	\arrow[""{name=0, anchor=center, inner sep=0}, from=1-1, to=2-1]
	\arrow["{/}"{marking}, from=2-1, to=2-2]
	\arrow[""{name=1, anchor=center, inner sep=0}, from=1-2, to=2-2]
	\arrow["{/}"{marking}, shorten <=4pt, shorten >=4pt, Rightarrow, from=1-2, to=2-1]
	\arrow[""{name=2, anchor=center, inner sep=0}, from=1-3, to=2-3]
	\arrow["{/}"{marking}, from=1-3, to=1-4]
	\arrow[""{name=3, anchor=center, inner sep=0}, from=1-4, to=2-4]
	\arrow["{/}"{marking}, shorten <=4pt, shorten >=4pt, Rightarrow, from=1-4, to=2-3]
	\arrow[""{name=4, anchor=center, inner sep=0}, curve={height=30pt}, from=1-1, to=2-1]
	\arrow["{/}"{marking}, from=2-3, to=2-4]
	\arrow[""{name=5, anchor=center, inner sep=0}, curve={height=-30pt}, from=1-4, to=2-4]
	\arrow["{ }"', shorten <=6pt, shorten >=6pt, Rightarrow, from=0, to=4]
	\arrow["{ }"', shorten <=6pt, shorten >=6pt, Rightarrow, from=5, to=3]
	\arrow[shift left=0.7, shorten <=6pt, shorten >=8pt, no head, from=1, to=2]
	\arrow[shift right=0.7, shorten <=6pt, shorten >=8pt, no head, from=1, to=2]
	\arrow["{/}"{marking}, shorten <=6pt, shorten >=6pt, from=1, to=2]
\end{tikzcd}\]
and the object $(\Db_2)_t\otimes[1]^\sharp$ corresponds to the diagram
\[\begin{tikzcd}
	00 & 01 & 00 & 01 \\
	10 & 11 & 10 & 11
	\arrow["{/}"{marking}, from=1-1, to=1-2]
	\arrow[""{name=0, anchor=center, inner sep=0}, from=1-1, to=2-1]
	\arrow["{/}"{marking}, from=2-1, to=2-2]
	\arrow[""{name=1, anchor=center, inner sep=0}, from=1-2, to=2-2]
	\arrow["{/}"{marking}, shorten <=4pt, shorten >=4pt, Rightarrow, from=1-2, to=2-1]
	\arrow[""{name=2, anchor=center, inner sep=0}, from=1-3, to=2-3]
	\arrow["{/}"{marking}, from=1-3, to=1-4]
	\arrow[""{name=3, anchor=center, inner sep=0}, from=1-4, to=2-4]
	\arrow["{/}"{marking}, shorten <=4pt, shorten >=4pt, Rightarrow, from=1-4, to=2-3]
	\arrow[""{name=4, anchor=center, inner sep=0}, curve={height=30pt}, from=1-1, to=2-1]
	\arrow["{/}"{marking}, from=2-3, to=2-4]
	\arrow[""{name=5, anchor=center, inner sep=0}, curve={height=-30pt}, from=1-4, to=2-4]
	\arrow["{=}"{marking}, draw=none, from=1, to=2]
	\arrow["{ /}"{marking}, shorten <=6pt, shorten >=6pt, Rightarrow, from=5, to=3]
	\arrow["{ /}"{marking}, shorten <=6pt, shorten >=6pt, Rightarrow, from=0, to=4]
\end{tikzcd}\]
\end{example}

\p 
\label{para:slice and joint}
We also define the functors
$$\uvar\star 1:\ocatm\to\ocatm~~~~~~1\costar \uvar:\ocatm\to \ocatm,$$
respectively called the \wcsnotionsym{marked Gray cone}{((d40@$\uvar\star 1$}{Gray cone}{for marked $\io$-categories} and the \wcsnotion{marked Gray $\circ$-cone}{Gray $\circ$-cone}{for marked $\io$-categories}\index[notation]{((d50@$1\overset{co}{\star}\_$!\textit{for marked $\io$-categories}}, where for any marked $\io$-category $C$, $C\star 1$ and $1\costar C$, fit in the following cocartesian square
\[\begin{tikzcd}
	{C\otimes\{1\}} & {C\otimes [1]^\sharp} & {C\otimes\{0\}} & {C\otimes [1]^\sharp} \\
	1 & {C\star 1} & 1 & {1\costar C}
	\arrow[from=1-1, to=1-2]
	\arrow[from=1-3, to=1-4]
	\arrow[from=1-4, to=2-4]
	\arrow[from=1-3, to=2-3]
	\arrow[from=2-3, to=2-4]
	\arrow[from=1-2, to=2-2]
	\arrow[from=1-1, to=2-1]
	\arrow[from=2-1, to=2-2]
	\arrow["\lrcorner"{anchor=center, pos=0.125, rotate=180}, draw=none, from=2-2, to=1-1]
	\arrow["\lrcorner"{anchor=center, pos=0.125, rotate=180}, draw=none, from=2-4, to=1-3]
\end{tikzcd}\]
These two functors preserve colimit.
The proposition \ref{prop:otimes et op marked version} induces an
invertible natural transformation
$$ C\star 1\sim (1\costar C^{\circ})^\circ.$$
\begin{example}
In all the following diagrams, marked cells are represented by crossed-out arrows.

The objects $\Db_1^\flat\star 1$ and $1\costar \Db_1^\flat$ correspond respectively the diagrams
\[\begin{tikzcd}
	0 &&&& 0 \\
	1 & \star && \star & 1
	\arrow[from=1-1, to=2-1]
	\arrow["{/}"{marking}, from=2-1, to=2-2]
	\arrow[""{name=0, anchor=center, inner sep=0}, "{/}"{marking}, from=1-1, to=2-2]
	\arrow[""{name=1, anchor=center, inner sep=0}, from=1-5, to=2-5]
	\arrow["{/}"{marking}, from=2-4, to=1-5]
	\arrow[""{name=2, anchor=center, inner sep=0}, "{/}"{marking}, from=2-4, to=2-5]
	\arrow["{/}"{marking}, shift right=2, shorten <=4pt, shorten >=4pt, Rightarrow, from=1, to=2]
	\arrow["{/}"{marking}, shorten <=2pt, Rightarrow, from=0, to=2-1]
\end{tikzcd}\]
the objects $(\Db_1)_t\star 1$ and $1\costar (\Db_1)_t$ correspond respectively the diagrams
\[\begin{tikzcd}
	0 &&&& 0 \\
	1 & \star && \star & 1
	\arrow["{/}"{marking}, from=1-1, to=2-1]
	\arrow["{/}"{marking}, from=2-1, to=2-2]
	\arrow["{/}"{marking}, from=1-1, to=2-2]
	\arrow["{/}"{marking}, from=1-5, to=2-5]
	\arrow["{/}"{marking}, from=2-4, to=1-5]
	\arrow["{/}"{marking}, from=2-4, to=2-5]
\end{tikzcd}\]
the objects $\Db_2^\flat\star 1$ and $1\costar \Db_2^\flat$ correspond respectively the diagrams
\[\begin{tikzcd}
	0 & {} & 0 &&& 0 & {} & 0 \\
	1 & \star & 1 & \star & \star & 1 & \star & 1
	\arrow[""{name=0, anchor=center, inner sep=0}, from=1-1, to=2-1]
	\arrow["{/}"{marking}, from=2-1, to=2-2]
	\arrow[""{name=1, anchor=center, inner sep=0}, from=1-3, to=2-3]
	\arrow[""{name=2, anchor=center, inner sep=0}, curve={height=30pt}, from=1-1, to=2-1]
	\arrow["{/}"{marking}, from=2-3, to=2-4]
	\arrow[""{name=3, anchor=center, inner sep=0}, "{/}"{marking}, from=1-1, to=2-2]
	\arrow[""{name=4, anchor=center, inner sep=0}, draw=none, from=1-2, to=2-2]
	\arrow[""{name=5, anchor=center, inner sep=0}, "{/}"{marking}, from=1-3, to=2-4]
	\arrow["{/}"{marking}, from=1-6, to=2-5]
	\arrow[""{name=6, anchor=center, inner sep=0}, from=1-6, to=2-6]
	\arrow[""{name=7, anchor=center, inner sep=0}, "{/}"{marking}, from=2-5, to=2-6]
	\arrow["{/}"{marking}, from=1-8, to=2-7]
	\arrow[""{name=8, anchor=center, inner sep=0}, from=1-8, to=2-8]
	\arrow[""{name=9, anchor=center, inner sep=0}, "{/}"{marking}, from=2-8, to=2-7]
	\arrow[""{name=10, anchor=center, inner sep=0}, curve={height=-30pt}, from=1-8, to=2-8]
	\arrow[""{name=11, anchor=center, inner sep=0}, draw=none, from=1-7, to=2-7]
	\arrow["{ }"', shorten <=6pt, shorten >=6pt, Rightarrow, from=0, to=2]
	\arrow["{/}"{marking}, shorten <=2pt, shorten >=2pt, Rightarrow, from=3, to=2-1]
	\arrow[shift left=0.7, shorten <=6pt, shorten >=8pt, no head, from=4, to=1]
	\arrow[shift right=0.7, shorten <=6pt, shorten >=8pt, no head, from=4, to=1]
	\arrow["{/}"{marking}, shorten <=6pt, shorten >=6pt, from=4, to=1]
	\arrow["{/}"{marking}, shorten <=2pt, Rightarrow, from=5, to=2-3]
	\arrow[shorten <=6pt, shorten >=6pt, Rightarrow, from=10, to=8]
	\arrow["{/}"{marking}, shift right=2, shorten <=4pt, shorten >=4pt, Rightarrow, from=8, to=9]
	\arrow["{/}"{marking}, shift right=2, shorten <=4pt, shorten >=4pt, Rightarrow, from=6, to=7]
	\arrow[shift right=0.7, shorten <=6pt, shorten >=8pt, no head, from=6, to=11]
	\arrow["{/}"{marking}, shorten <=6pt, shorten >=6pt, from=6, to=11]
	\arrow[shift left=0.7, shorten <=6pt, shorten >=8pt, no head, from=6, to=11]
\end{tikzcd}\]
and the objects $(\Db_2)_t\star 1$ and $1\costar (\Db_2)_t$ correspond respectively the diagrams
\[\begin{tikzcd}
	0 & {} & 0 &&& 0 & {} & 0 \\
	1 & \star & 1 & \star & \star & 1 & \star & 1
	\arrow[""{name=0, anchor=center, inner sep=0}, from=1-1, to=2-1]
	\arrow["{/}"{marking}, from=2-1, to=2-2]
	\arrow[""{name=1, anchor=center, inner sep=0}, from=1-3, to=2-3]
	\arrow[""{name=2, anchor=center, inner sep=0}, curve={height=30pt}, from=1-1, to=2-1]
	\arrow["{/}"{marking}, from=2-3, to=2-4]
	\arrow[""{name=3, anchor=center, inner sep=0}, "{/}"{marking}, from=1-1, to=2-2]
	\arrow[""{name=4, anchor=center, inner sep=0}, draw=none, from=1-2, to=2-2]
	\arrow[""{name=5, anchor=center, inner sep=0}, "{/}"{marking}, from=1-3, to=2-4]
	\arrow["{/}"{marking}, from=1-6, to=2-5]
	\arrow[""{name=6, anchor=center, inner sep=0}, from=1-6, to=2-6]
	\arrow[""{name=7, anchor=center, inner sep=0}, "{/}"{marking}, from=2-5, to=2-6]
	\arrow["{/}"{marking}, from=1-8, to=2-7]
	\arrow[""{name=8, anchor=center, inner sep=0}, from=1-8, to=2-8]
	\arrow[""{name=9, anchor=center, inner sep=0}, "{/}"{marking}, from=2-8, to=2-7]
	\arrow[""{name=10, anchor=center, inner sep=0}, curve={height=-30pt}, from=1-8, to=2-8]
	\arrow[""{name=11, anchor=center, inner sep=0}, draw=none, from=1-7, to=2-7]
	\arrow["{/}"{marking}, shorten <=6pt, shorten >=6pt, Rightarrow, from=0, to=2]
	\arrow["{/}"{marking}, shorten <=2pt, shorten >=2pt, Rightarrow, from=3, to=2-1]
	\arrow["{/}"{marking}, shorten <=2pt, Rightarrow, from=5, to=2-3]
	\arrow["{/}"{marking}, shorten <=6pt, shorten >=6pt, Rightarrow, from=10, to=8]
	\arrow["{/}"{marking}, shift right=2, shorten <=4pt, shorten >=4pt, Rightarrow, from=8, to=9]
	\arrow["{/}"{marking}, shift right=2, shorten <=4pt, shorten >=4pt, Rightarrow, from=6, to=7]
	\arrow["{=}"{description}, draw=none, from=4, to=1]
	\arrow["{=}"{marking}, draw=none, from=6, to=11]
\end{tikzcd}\]
\end{example}

We will also denote by 
$$\begin{array}{ccccccc}
\ocatm_{\bullet} &\to&\ocatm&&\ocatm_{\bullet} &\to&\ocatm\\
(C,c)&\mapsto &C_{/c} & &(C,c) &\mapsto &C_{c/}
\end{array}
$$
the right adjoints of Gray cone and of the Gray $\circ$-cone, respectively called the \wcsnotionsym{slice of $C$ over $c$}{(cc@$C_{c/}$}{slice over}{for marked $\io$-categories} and the \wcsnotionsym{slice of $C$ under $c$}{(cc@$C_{/c}$}{slice under}{for marked $\io$-categories}.
The proposition \ref{prop:otimes et op marked version} induces an  invertible natural transformation:
$$C_{/c}\sim (C^{\circ}_{c/})^\circ.$$
Given an $\io$-category $C$, and $c,d$ two objects, the cocartesian square \eqref{eq:liens entre Gray cylindre et suspension version marque} induces two cartesian squares:
\begin{equation}
\label{eq:fiber of marked splices}
\begin{tikzcd}
	{\hom_C(c,d)^\flat} & {C^\sharp_{/d}} & {\hom_C(c,d)^\flat} & {C^\sharp_{c/}} \\
	{\{c\}} & {C^\sharp} & {\{d\}} & {C^\sharp}
	\arrow[from=1-1, to=2-1]
	\arrow[from=2-3, to=2-4]
	\arrow[from=1-4, to=2-4]
	\arrow[from=1-3, to=2-3]
	\arrow[from=1-2, to=2-2]
	\arrow[from=2-1, to=2-2]
	\arrow[from=1-1, to=1-2]
	\arrow["\lrcorner"{anchor=center, pos=0.125}, draw=none, from=1-1, to=2-2]
	\arrow["\lrcorner"{anchor=center, pos=0.125}, draw=none, from=1-3, to=2-4]
	\arrow[from=1-3, to=1-4]
\end{tikzcd}
\end{equation}

\p 
\label{paragrap: equation fullfill by cylinder and join marked version}
The equation given in paragraph \ref{paragrap: equation fullfill by cylinder and join}
induces similar ones for the marked version of these operations. For every marked $\io$-category $C$, there are a natural identification between $[C,1]\otimes [1]^\sharp$ and the colimit of the following diagram
\begin{equation}
\label{eq:eq for cylinder marked version}
\begin{tikzcd}
	{[1]^\sharp\vee [ C,1]} & {[C\otimes\{0\},1]} & {[C\otimes [1]^\sharp,1]} & {[C\otimes\{1\},1]} & {[C,1]\vee[1]^\sharp}
	\arrow[from=1-2, to=1-1]
	\arrow[from=1-2, to=1-3]
	\arrow[from=1-4, to=1-3]
	\arrow[from=1-4, to=1-5]
\end{tikzcd}
\end{equation}
There is also a natural identification between
 $1\costar [C,1]$ and the colimit of the diagram
\begin{equation}
\label{eq:eq for Gray cone marked version}
\begin{tikzcd}
	 {[1]^\sharp\vee [C,1]} & {[C,1]} & {[C\star 1,1]} 
	\arrow[from=1-2, to=1-1]
	\arrow[from=1-2, to=1-3]
\end{tikzcd}
\end{equation}
and between $[C,1] \star 1$ and the colimit of the diagram
\begin{equation}
\label{eq:eq for cojoin marked version}
\begin{tikzcd}
	 {[1\costar C,1]}& {[C,1]} & {[C,1]\vee[1]^\sharp} 
	\arrow[from=1-2, to=1-1]
	\arrow[from=1-2, to=1-3]
\end{tikzcd}
\end{equation}

\p
 For any $C:\ocat$, we denote by \wcnotation{$m_{C^\sharp}$}{(mc@$m_{C^\sharp}$} the colimit preserving functor 
$\ocatm\to\ocatm$ whose value on $[a,n]^\flat$ is $[a\times C^\sharp,n]$, on $[1]^\sharp$ is $[C,1]^\sharp$, and on $[(\Db_n)_t,1]$ is $
[(\Db_n)_t\times C^\sharp,1]$.
Remark that the assignation $C\mapsto m_{C^\sharp}$ is natural in $C$ and that $m_1$ is the identity.
We define the colimit preserving functor:
\ssym{((d20@$\ominus$}{for marked $\io$-categories}
$$\begin{array}{ccc}
\ocatm\times\ocatm &\to& \ocatm\\
(X,Y)&\mapsto &X\ominus Y^\sharp
\end{array}
$$
where for any marked $\io$-category $C$ and element $[b,n]$ of $\Delta[\Theta]$, $C\ominus [b,n]^\sharp$ is the following pushout: 
\begin{equation}
\label{eq: def of ominus marked}
\begin{tikzcd}
	{\coprod\limits_{k\leq n}m_{b^\sharp}(C\otimes\{k\})} & {m_{b^\sharp}(C\otimes[n]^\sharp)} \\
	{\coprod\limits_{k\leq n}m_1(C\otimes\{k\})} & {C\ominus[b,n]^\sharp}
	\arrow[from=1-1, to=2-1]
	\arrow[""{name=0, anchor=center, inner sep=0}, from=1-1, to=1-2]
	\arrow[from=1-2, to=2-2]
	\arrow[from=2-1, to=2-2]
	\arrow["\lrcorner"{anchor=center, pos=0.125, rotate=180}, draw=none, from=2-2, to=0]
\end{tikzcd}
\end{equation}
By construction, we then have $C\ominus [1]^\sharp:=C\otimes [1]^\sharp$.
The equation \eqref{eq:formula for the ominus} implies that for every marked $\io$-category $C$, 
there is a natural identification between $[C,1]\ominus[b,1]^\sharp$ and the colimit of the following diagram
\begin{equation}
\label{eq:formula for the ominus marked case}
\begin{tikzcd}[column sep = 0.3cm]
	{[b,1]^\sharp\vee[C,1]} & {[C\otimes\{0\}\times b^\sharp,1]} & {[(C\otimes[1]^\sharp)\times b^\sharp),1]} & {[C\otimes\{1\}\times b^\sharp,1]} & {[C,1]\vee[b,1]^\sharp}
	\arrow[from=1-2, to=1-3]
	\arrow[from=1-4, to=1-3]
	\arrow[from=1-4, to=1-5]
	\arrow[from=1-2, to=1-1]
\end{tikzcd}
\end{equation}

\begin{prop}
\label{prop:ominus and opmarked}
There is an equivalence 
$$(C\ominus B^\sharp)^\circ\sim C^\circ\ominus (B^\circ)^\sharp$$
natural in $C$ and $B$.
\end{prop}
\begin{proof}
It it sufficient to construct this equivalence when $B$ is of shape $[b,n]$.
The corollary \ref{cor:ominus et op} induces an equivalence
$$(C^\natural\otimes[n])^\circ\sim (C^\circ)^\natural\otimes [n]^\circ.$$
By the construction of the Gray tensor product of marked $\io$-categories, we have an equivalence 
$$(C\otimes[n]^\sharp)^\circ\sim C^\circ\otimes ([n]^\circ)^\sharp.$$
 The results then directly follows from the definition of the operation $\ominus$ and from the equivalence $(m_{b^\sharp}(\uvar))^\circ\sim m_{(b^\sharp)^\circ}((\uvar)^\circ)$.
\end{proof}

\begin{prop}
\label{prop:associativity of ominus}
Let $C$ be a $\io$-category, $D$ a marked $\io$-category and $[b,n]$ a globular sum.
\begin{enumerate}
\item The underlying $\io$-category of $C^\flat\ominus [b,n]^\sharp$ is $C\ominus [b,n]$.
\item The canonical morphism $C^\sharp\ominus [b,n]^\sharp\to C^\sharp\times [b,n]^\sharp$ is an equivalence.
\item The canonical morphism
$(C^\sharp\times D)\ominus [b,n]^\sharp\to C^\sharp\times (D\ominus K^\sharp)$ is an equivalence.
\end{enumerate}
\end{prop}
\begin{proof}
This is a consequence of propositions \ref{prop:cartesian square and times}, \ref{prop:associativity of Gray amput}, \ref{prop:associativity of Gray2} and \ref{prop:associativity of Gray amput2} and of the construction of $\ominus$.
\end{proof}

\p We now give some strictification results.

\begin{lemma}
\label{lemma:a otimes 1 is strict}
Let $C$ be a marked $\io$-category.
The canonical squares
\[\begin{tikzcd}
	C & {C\otimes[1]^\sharp} & C \\
	{\{0\}} & {[1]^\sharp} & {\{1\}}
	\arrow[from=1-1, to=2-1]
	\arrow[from=2-1, to=2-2]
	\arrow[from=2-3, to=2-2]
	\arrow[from=1-2, to=2-2]
	\arrow[from=1-3, to=1-2]
	\arrow[from=1-1, to=1-2]
	\arrow[from=1-3, to=2-3]
\end{tikzcd}\]
are cartesian. 
\end{lemma}
\begin{proof}
As the morphisms $\{\epsilon\}\to [1]$ for $\epsilon\leq 1$ are discrete Conduché functors, pullback along them preserves colimits, and we can then reduce to the case where $C$ is of the shape $[1]^\sharp$ or $[a,1]$ with  $a$ is an element of $t\Theta$. 
The case $C:=[1]^\sharp$ is obvious as we have $[1]^\sharp\otimes[1]^\sharp\sim [1]^\sharp\times[1]^\sharp$ according to  the first assertion of proposition \ref{prop:associativity of Gray amput}. We then focus on the case $C:=[a,1]$.

We claim that for any marked $\io$-category $D$, the square
\begin{equation}
\label{eq:lemma:a otimes 1 is strict}
\begin{tikzcd}
	{\{\epsilon\}} & {[D,1]} \\
	{\{\epsilon\}} & {[1]^\sharp}
	\arrow[from=1-1, to=2-1]
	\arrow[from=2-1, to=2-2]
	\arrow[from=1-1, to=1-2]
	\arrow[from=1-2, to=2-2]
\end{tikzcd}
\end{equation}
is cartesian. To show this, as  the morphisms $\{\epsilon\}\to [1]$, are discrete Conduché functors one can reduce to the case where $D$ is a globular sum, where it is obvious.

We now return to the proof of the assertion.  
Using the equation \eqref{eq:eq for cylinder marked version}, the morphism $[a,1]\otimes[1]^\sharp$ is the horizontal colimit of the following diagram:
\[\begin{tikzcd}
	{[1]^\sharp\vee[a,1]} & {[a\otimes\{0\},1]} & {[a\otimes[1]^\sharp,1]} & {[a\otimes\{1\},1]} & {[a,1]\vee[1]^\sharp} \\
	{[1]^\sharp} & {[1]^\sharp} & {[1]^\sharp} & {[1]^\sharp} & {[1]^\sharp}
	\arrow[from=1-2, to=1-1]
	\arrow[from=1-4, to=1-5]
	\arrow[from=1-4, to=1-3]
	\arrow[from=1-2, to=1-3]
	\arrow[from=2-2, to=2-1]
	\arrow[from=2-2, to=2-3]
	\arrow[from=2-4, to=2-3]
	\arrow[from=2-4, to=2-5]
	\arrow["{s^0}", from=1-5, to=2-5]
	\arrow[from=1-4, to=2-4]
	\arrow[from=1-2, to=2-2]
	\arrow[from=1-3, to=2-3]
	\arrow["{s^1}"', from=1-1, to=2-1]
\end{tikzcd}\]

The results is then a direct application of the cartesian square \eqref{eq:lemma:a otimes 1 is strict} and of the fact  that pullbacks along morphisms $\{\epsilon\}\to [1]$ for $\epsilon\leq 1$ preserves colimits.
\end{proof}

\begin{prop}
\label{prop:tensor of glboer are strics}
For any object $a$ of $t\Theta$, the marked $\io$-categories $a\otimes [1]^\sharp$, $a\star 1$ and $1\costar a$ are strict. 
\end{prop}
\begin{proof}
We will show only the the strictness of the object $a\otimes[1]^\sharp$, as the proofs for $a\star 1$ and $1\costar a$ are similar.

Suppose first that $a$ is of shape $b^\flat$.
The first assertion of proposition \ref{prop:associativity of Gray amput} implies that the underlying $\io$-categories of $b^\flat\otimes [1]^\sharp$ is $b\otimes [1]$ which is strict according to proposition \ref{prop:strict stuff are pushout}.

To conclude, we have to show that for any integer $n$, $(\Db_n)_t\otimes[1]^\sharp$ is strict. We proceed by induction.
Suppose first that $a$ is $(\Db_1)_t$. The second assertion of proposition \ref{prop:associativity of Gray amput} implies that 
$(\Db_1)_t \otimes[1]^\sharp$ is $([1]\times[1])^\sharp$ which is a strict object. 

Suppose now that $(\Db_n)_t\otimes[1]^\sharp$ is strict. 
The equation \eqref{eq:eq for cylinder marked version} stipulates that $(\Db_{n+1})_t\otimes[1]^\sharp$ is the colimit of the diagram.
$$\begin{tikzcd}[column sep = 0.2cm]
	{[1]^\sharp\vee [(\Db_n)_t,1]} & {[(\Db_n)_t\otimes\{0\},1]} & {[(\Db_n)_t\otimes [1]^\sharp,1]} & {[(\Db_n)_t\otimes\{1\},1]} & {[\Db_n)_t,1]\vee[1]^\sharp}
	\arrow[from=1-2, to=1-1]
	\arrow[from=1-2, to=1-3]
	\arrow[from=1-4, to=1-3]
	\arrow[from=1-4, to=1-5]
\end{tikzcd}$$
The induction hypothesis and the proposition \ref{prop:suspension preserves stricte} implies that all the objects are strict. According to proposition \ref{prop:example of a special colimit3 marked case}, whose hypotheses are provided by lemma \ref{lemma:a otimes 1 is strict}, this diagram admits a special colimit. As all the morphisms are monomorphism, this implies that $(\Db_{n+1})_t\otimes[1]^\sharp$ is strict, which concludes the proof.
\end{proof}

\begin{prop}
\label{prop:some strict marked io categories2}
If $C$ is a marked $\io$-category, $a$ a globular sum and $a^\flat\to C$ any morphism, the $\io$-categories $C\coprod_{a^\flat} a^\flat\otimes [1]^\sharp$, $C\coprod_{a^\flat} \star 1$ and $1\costar a^\flat \coprod_aC$ are strict.
\end{prop}
\begin{proof}
Using the first assertion of proposition \ref{prop:associativity of Gray amput}, the underlying $\io$-categories of $C\coprod_{a^\flat} a^\flat\otimes [1]^\sharp$, $C\coprod_{a^\flat}a^\flat \star 1$ and $1\costar a^\flat \coprod_aC$ are respectively $C^\natural\coprod_{a} a\otimes [1]$, $C^\natural\coprod_{a} a\star 1$ and $1\costar a \coprod_aC^\natural$, which are strict objects according to propositions \ref{prop:strict stuff are stable under Gray cone} and \ref{prop:strict stuff are stable under coproduc with cylinder}.
\end{proof}

\begin{theorem}
\label{theo:strictness marked}
If $C$ is strict $\io$-category, the marked $\io$-categories $C^\flat\star 1$, $1\costar C^\flat$ and $C^\flat\otimes [1]^\sharp$ are strict.
\end{theorem}
\begin{proof}
The first assertion of proposition \ref{prop:associativity of Gray amput} implies that the underlying $\io$-categories of these marked $\io$-categories respectively are $C\star 1$, $1\costar C$ and $C\otimes [1]$. As these objects are strict according to theorem \ref{theo:strictness}, this concludes the proof.
\end{proof}

\begin{prop}
\label{prop:crushing of Gray tensor is identitye marked case}
The colimit preserving endofunctor $F:\ocat\to \ocatm$, sending $[a,n]$ to the colimit of the span
$$\coprod_{k\leq n}\{k\}\leftarrow \coprod_{k\leq n}a^\flat\otimes\{k\}\to a^\flat\otimes[n]^\sharp$$
is equivalent to the functor $(\uvar)^\sharp:\ocat\to \ocatm$.
\end{prop}
\begin{proof}
This is a direct consequence of the first assertion of proposition \ref{prop:associativity of Gray amput}, of corollary \ref{cor:crushing of Gray tensor is identitye} and of the definition of the marking of the Gray tensor product for marked $\io$-categories.
\end{proof}
The last proposition implies that for any marked $\io$-category $C$ and any globular sum $a$, the simplicial $\infty$-groupoid
$$\begin{array}{rcl}
\Delta^{op}&\to &\igrd\\
~[n]~&\mapsto &\Hom([a,n]^\sharp,C)
\end{array} $$
is a $\iun$-category.

\begin{theorem}
\label{theo:formula between pullback of slice and tensor marked case}
Let $C$ be an $\io$-category. The two following canonical squares are cartesian:
\[\begin{tikzcd}
	1 & {1\costar C^\flat} & 1 & {C^\flat\star 1} \\
	{\{0\}} & {[C,1]^\sharp} & {\{1\}} & {[C,1]^\sharp}
	\arrow[from=1-1, to=1-2]
	\arrow[from=2-1, to=2-2]
	\arrow[from=1-1, to=2-1]
	\arrow[from=1-2, to=2-2]
	\arrow[from=1-3, to=1-4]
	\arrow[from=2-3, to=2-4]
	\arrow[from=1-3, to=2-3]
	\arrow[from=1-4, to=2-4]
\end{tikzcd}\]
The five squares appearing in the following canonical diagram are both cartesian and cocartesian:
\[\begin{tikzcd}
	& {C^\flat\otimes\{0\}} & 1 \\
	{C^\flat\otimes\{1\}} & {C^\flat\otimes[1]^\sharp} & {C^\flat\star 1} \\
	1 & {1\costar C^\flat} & {[C,1]^\sharp}
	\arrow[from=2-3, to=3-3]
	\arrow[from=3-2, to=3-3]
	\arrow[from=2-2, to=3-2]
	\arrow[from=2-2, to=2-3]
	\arrow[from=1-2, to=1-3]
	\arrow[from=1-3, to=2-3]
	\arrow[from=1-2, to=2-2]
	\arrow[from=2-1, to=2-2]
	\arrow[from=3-1, to=3-2]
	\arrow[from=2-1, to=3-1]
\end{tikzcd}\]
\end{theorem}
\begin{proof}
This is a direct consequence of the first assertion of proposition \ref{prop:associativity of Gray amput}, of theorem \ref{theo:formula between pullback of slice and tensor} and of the definition of the marking of the Gray tensor product for marked $\io$-categories.
\end{proof}

\subsection{Marked Gray deformation retract}
We provide analogous results for section \ref{subsection:Gray deformation retract}, with proofs that are entirely similar and, therefore, omitted.

\p A \wcnotion{left Gray deformation retract structure}{left or right Gray deformation retract structure} for a morphism $i:C\to D$ between marked $\io$-categories is the data of a \textit{retract}
 $r:D\to C$, a \textit{deformation} $\psi:D\otimes [1]^\sharp\to D$, and equivalences
$$ri\sim id_C~~~~~\psi_{|D\otimes\{0\}}\sim ir~~~~~\psi_{|D\otimes\{1\}}\sim id_D~~~~~ \psi_{|C\otimes[1]^\sharp}\sim i\cst_C
$$ 
A morphism $i:C\to D$ between marked $\io$-categories is a \wcnotion{left Gray deformation retract}{left or right Gray deformation retract} if it admits a left deformation retract structure. By abuse of language, such data will just be denoted by $(i,r,\psi)$.

We define dually the notion of \textit{right Gray deformation retract structure} and of \textit{right Gray deformation retract} in exchanging $0$ and $1$ in the previous definition.

We define similarly the notion of \notion{left or right deformation retract} by replacing $\otimes$ by $\times$.

\p
 A \textit{left Gray deformation retract structure for a morphism $i:f\to g$} in the $\iun$-category of arrows of $\ocatm$ is the data of a \textit{retract}
 $r:g\to f$, a \textit{deformation} $\psi:g\otimes [1]^\sharp\to g$ and equivalences
$$ri\sim id_f~~~~~\psi_{|g\otimes\{0\}}\sim ir~~~~~\psi_{|g\otimes\{1\}}\sim id_D~~~~~ \psi_{|f\otimes[1]^\sharp}\sim i\cst_C
$$ 
A morphism $i:C\to D$ between two arrows of $\ocatm$ is a \textit{left Gray deformation retract} if it admits a left deformation retract structure. By abuse of language, such data will just be denoted by $(i,r,\psi)$.

We define dually the notion of \textit{right Gray deformation retract structure} and of \textit{right Gray deformation retract} in exchanging $0$ and $1$ in the previous definition.

We define similarly the notion of \notion{left and right deformation retract} by replacing $\otimes$ by $\times$.

\begin{example}
\label{example:canonical example of left deformation retract}
Let $C$ be a marked $\io$-category. The morphism $C\otimes\{0\}\to C\otimes[1]^\sharp$ is a left Gray deformation retract.
Indeed, the retract is given by $C\otimes\Ib:C\otimes[1]^\sharp\to C\otimes\{0\}$, and the natural transformation is induced by
$$(C\otimes[1]^\sharp)\otimes[1]^\sharp\sim C\otimes([1]\times [1])^\sharp\xrightarrow{C\otimes\psi^\sharp} C\otimes[1]^\sharp$$
where the first equivalence is the one of proposition \ref{prop:associativity of Gray amput2}, and $\psi:[1]\times[1]\to [1]$ is the unique morphism sending $(\epsilon,\epsilon')$ to $\epsilon\wedge \epsilon'$.

Similarly, the morphism $C\otimes\{1\}\to C\otimes[1]^\sharp$ is a right deformation retract.
\end{example}

\p Left and right Gray retracts enjoy many stability properties: 
\begin{prop}
\label{prop:left Gray deformation retract stable under pushout}
Let $(i_a,r_a,\psi_a)$ be a natural family of left (resp. right) Gray deformation retract structures indexed by an $(\infty,1)$-category $A$.
The triple $(\colim_{A}i_a,\colim_{A}r_a,\colim_{A}\psi_a)$ is a left (resp. right) $k$-Gray deformation retract structure.
\end{prop}

\begin{prop}
\label{prop:stability under pullback}
Suppose given a diagram
\[\begin{tikzcd}
	X & Y & Z \\
	X & {Y'} & {Z'}
	\arrow[from=1-1, to=2-1]
	\arrow[from=1-2, to=2-2]
	\arrow[from=1-3, to=2-3]
	\arrow["p", from=1-1, to=1-2]
	\arrow["q"', from=1-3, to=1-2]
	\arrow["{p'}"', from=2-1, to=2-2]
	\arrow["{q'}", from=2-3, to=2-2]
\end{tikzcd}\]
such that $p\to p'$ and $q\to q'$ are left (resp. right) Gray deformation retract. The induced square $q^*p\to (q')^*p'$ is a left (resp. right) $k$-Gray deformation retract.
\end{prop}

\begin{prop}
\label{prop:stability by composition }
If $p\to p'$ and $p'\to p''$ are two left (resp. right) Gray deformation retracts, so is $p\to p''$.
\end{prop}

\begin{prop}
\label{prop:Gray deformation retract and passage to hom}
Let $(i:C\to D,r,\psi)$ be a left (resp. right) Gray deformation structure. For any $x: C$ and $y:D$ (resp. $x: D$ and $y:C$), the morphism
$$\begin{array}{cc}
&\hom_C(x,ry)\xrightarrow{i} \hom_D(ix,iry)\xrightarrow{{\psi_y}_!} \hom_D(ix,y)\\
(resp. &\hom_C(rx,y)\xrightarrow{i} \hom_D(irx,iy)\xrightarrow{{\psi_x}_!} \hom_D(x,iy))
\end{array}
$$
is a right (resp. left) Gray deformation retract, whose retract is given by 
$$\begin{array}{cc}
&\hom_D(ix,y)\xrightarrow{r}\hom_C(x,ry)\\
(resp. &\hom_D(x,iy)\xrightarrow{r}\hom_C(rx,y))
\end{array}$$

If $(i:C\to D,r,\psi)$ is a left (resp. right) deformation structure, for any $x: C$ and $y:D$ (resp. $x: D$ and $y:C$), the two morphisms above are inverses one of each other.
\end{prop}

\begin{prop}
\label{prop:Gray deformation retract and passage to hom v2}
For any left (resp. right) Gray deformation retracts between $p$ and $p'$:
\[\begin{tikzcd}
	C & D \\
	{C'} & {D'}
	\arrow["p"', from=1-1, to=2-1]
	\arrow["i", from=1-1, to=1-2]
	\arrow["{p'}", from=1-2, to=2-2]
	\arrow["{i'}"', from=2-1, to=2-2]
\end{tikzcd}\]
and for any pair of objects $x: C$ and $y:D$ (resp. $x: D$ and $y:C$), the outer square of the following diagram
\[\begin{tikzcd}
	{\hom_{C}(x,ry)} & {\hom_{D}(ix,iry)} & {\hom_{D}(ix,y)} \\
	{\hom_{C'}(px,pr'y)} & {\hom_{D'}(p'i'x,p'i'r'y)} & {\hom_{D'}(p'i'x,p'y)}
	\arrow["{i'}"', from=2-1, to=2-2]
	\arrow["{{\psi'_{p'y}}_!}"', from=2-2, to=2-3]
	\arrow[from=1-1, to=2-1]
	\arrow[from=1-3, to=2-3]
	\arrow["{{\psi_y}_!}", from=1-2, to=1-3]
	\arrow["i", from=1-1, to=1-2]
	\arrow[from=1-2, to=2-2]
\end{tikzcd}\]
(resp.
\[\begin{tikzcd}
	{\hom_{C}(rx,y)} & {\hom_{D}(irx,iy)} & {\hom_{D}(x,iy)} \\
	{\hom_{C'}(pr'x,py)} & {\hom_{D'}(p'i'r'x,p'i'y)} & {\hom_{D'}(p'x,p'i'y)\big)}
	\arrow["{i'}"', from=2-1, to=2-2]
	\arrow["{{\psi'_{p'x}}_!}"', from=2-2, to=2-3]
	\arrow[from=1-1, to=2-1]
	\arrow[from=1-3, to=2-3]
	\arrow["{{\psi_x}_!}", from=1-2, to=1-3]
	\arrow["i", from=1-1, to=1-2]
	\arrow[from=1-2, to=2-2]
\end{tikzcd}\]
is a left (resp. right) Gray deformation retract, whose retract is given by
\[\begin{tikzcd}
	{\hom_{D}(ix,y)} & {\hom_{C}(x,ry)} & {(resp.\hom_{D}(x,iy)} & {\hom_{C}(rx,y)} \\
	{\hom_{D'}(p'i'x,p'y)\big)} & {\hom_{C'}(px,pr'y)} & {\hom_{D'}(p'x,p'i'y)} & {\hom_{C'}(pr'x,py)\big)}
	\arrow[from=1-3, to=2-3]
	\arrow["r", from=1-3, to=1-4]
	\arrow["{r'}"', from=2-3, to=2-4]
	\arrow[from=1-4, to=2-4]
	\arrow["{r'}"', from=2-1, to=2-2]
	\arrow["r", from=1-1, to=1-2]
	\arrow[from=1-2, to=2-2]
	\arrow[from=1-1, to=2-1]
\end{tikzcd}\]
If $p\to p'$ is a left (resp. right) deformation structure, for any $x: C$ and $y:D$ (resp. $x: D$ and $y:C$), the two morphisms above are inverses one of each other.
\end{prop}

\begin{prop}
\label{prop:suspension of left Gray deformation retract}
If $i$ is a left Gray deformation retract, $[i,1]$ is a right Gray deformation retract. Conversely, if $i$ is a right Gray deformation retract, $[i,1]$ is a left Gray deformation retract morphism.
\end{prop}

\begin{prop}
\label{prop:when glob inclusion are left Gray deformation}
Let $a$ be a globular sum of dimension $(n+1)$. We denote by $s_n(a)$ and $t_n(a)$ the globular sum defined in \ref{para:definition of source et but}. If $n$ is even, $s_n(a)^\flat\to a^{\sharp_n}$ is a left Gray deformation retract, and $t_n(a)^\flat\to a^{\sharp_n}$ is a right Gray deformation retract. Dually, if $n$ is odd, $t_n(a)^\flat\to a^{\sharp_n}$ is a left Gray deformation retract, and $s_n(a)^\flat\to a^{\sharp_n}$ is a right Gray deformation retract.
\end{prop}

\begin{prop}
\label{prop:exemple of right deformation retract}
Let $i:C\to D$ be a left Gray deformation retract and $A$ a marked $\io$-category.
The morphism $A\times i$ is a left Gray deformation retract. 
\end{prop}
\begin{proof}
Let $r$ and $\psi$ be retracts and deformation of $i$.
We define $\psi_A$ as the composite
$$(A\times D)\otimes[1]^\sharp\to A\times (D\otimes[1]^\sharp)\xrightarrow{A\times \psi} A\times D$$
Remark that the triple $(A\times i,A\times r,\psi_A)$ is a left Gray deformation retract structure.
\end{proof}

\begin{prop}
\label{prop:retraction criter}
Let $(i:[C,1]\to D,r,\phi)$ be a left deformation retract structure. The following natural square is cartesian:
\[\begin{tikzcd}
	D & {\uHom([1]^\sharp,D)} \\
	{[C,1]} & D
	\arrow["i"', from=2-1, to=2-2]
	\arrow["r"', from=1-1, to=2-1]
	\arrow["{(i^-_0)_!}", from=1-2, to=2-2]
	\arrow["\phi", from=1-1, to=1-2]
\end{tikzcd}\]
\end{prop}
\begin{proof}
We set $P:=[C,1]\times_{D}\uHom([1]^\sharp,D)$ and $\psi:D\to P$ the induced morphism.
The proposition \ref{prop:example of a special colimit marked case} implies that $\hom_{ P}(\psi(x),\psi(y))$ is the limit of the diagram:
\[\begin{tikzcd}
	{\hom_{[C,1]}(rx,ry)} & {\hom_{D}(irx,iry)} & {\hom_{D}(irx,y)} & {\hom_{D}(x,y)}
	\arrow["i", from=1-1, to=1-2]
	\arrow["{{\phi_y}_!}", from=1-2, to=1-3]
	\arrow["{{\phi_x}_!}"', from=1-4, to=1-3]
\end{tikzcd}\]
The proposition \ref{prop:Gray deformation retract and passage to hom} then implies that the canonical morphism
$$\hom_D( x,y)\to \hom_{ P}(\psi(x),\psi(y))$$
is an equivalence.

The morphism $\psi$ is then fully faithful. According to proposition \ref{prop:fully faithful plus surjective on objet marked case}, it remains to show that it induces a surjection on objects. For this, let $v:x\to y$ be an element of $P$. As the only marked $1$-cells in $[C,1]$ are equivalences, $r(v)$ is an equivalence. The morphism 
$$[1]^\sharp\times[1]^\sharp\xrightarrow{v\times [1]^\sharp} D\times [1]^\sharp\xrightarrow{\phi} D$$
 induces a square in $D$ of shape
\[\begin{tikzcd}
	irx & x \\
	iry & y
	\arrow["{\phi(y)}"', from=2-1, to=2-2]
	\arrow["v", from=1-2, to=2-2]
	\arrow["\sim", from=1-1, to=1-2]
	\arrow["\sim", from=1-1, to=2-1]
	\arrow["{ir(v)}"', draw=none, from=1-1, to=2-1]
\end{tikzcd}\]
where all the arrows labeled by $\sim$ are equivalences. This implies that $v\sim \phi(y)$ and the morphism $\psi$ is then surjective on objects. This concludes the proof.
\end{proof}

\section{Cartesian fibrations}
\subsection{Left and right cartesian fibrations}
\label{subsection Left and right cartesian fibration}

\p We denote by \wcnotation{$\I$}{(i@$\I$} the set of morphisms of shape $X\otimes \{0\}\to X\otimes [1]^\sharp$ for $X$ being either $\Db_n^\flat$ or $(\Db_n)_t$. A morphism is \wcnotion{initial}{initial morphism} if it is in $\widehat{\I}$. Conversely, we denote by \wcnotation{$\F$}{(f@$\F$} the set of morphisms of shape $X\otimes \{1\}\to X\otimes [1]^\sharp$ for $X$ being either $\Db_n^\flat$ or $(\Db_n)_t$. A morphism is \wcnotion{final}{final morphism} if it is in $\widehat{\F}$.

Initial and final morphisms are stable under colimits, retract, composition and  left cancellation according to the result of section \ref{section:Factorization system}.  

The proposition \ref{prop:otimes et op marked version} implies that the full duality $(\uvar)^\circ$ sends final (resp. initial) morphisms to initial (resp. final) morphisms.

\begin{example}
\label{exe:the easiest example of initial and finla morphism}
By stability of initial and final morphisms by colimits, for any marked $\io$-category $C$, $C\otimes\{0\}\to C\otimes[1]^\sharp$ is initial, and $C\otimes\{1\}\to C\otimes[1]^\sharp$ is final.
\end{example}

\begin{prop}
\label{prop:left Gray deformation retract are initial}
Left Gray deformation retracts (resp. left deformation retract) are initial and right Gray deformation retracts (resp. right deformation retract) are final. 
\end{prop}
\begin{proof} 
Let $i:C\to D$ be a left Gray deformation retract. The diagram
\[\begin{tikzcd}
	C & {D\otimes\{0\}} & C \\
	{D\otimes\{1\}} & {D\otimes [1]^\sharp} & D
	\arrow["i"', from=1-1, to=2-1]
	\arrow[from=2-1, to=2-2]
	\arrow[from=1-2, to=2-2]
	\arrow["\psi"', from=2-2, to=2-3]
	\arrow["i", from=1-1, to=1-2]
	\arrow["r", from=1-2, to=1-3]
	\arrow["i", from=1-3, to=2-3]
\end{tikzcd}\]
expresses $i$ as a retract of $D\otimes \{0\}\to D\otimes [1]^\sharp$, which is an initial morphism according to example \ref{exe:the easiest example of initial and finla morphism}. The morphism  $i$ is then initial. 

As left deformation retracts are left Gray deformation retracts, they are initial.
The case of right (Gray) deformation retracts follows by duality.
\end{proof}

\begin{cor}
\label{cor:when glob inclusion are final and initial}
Let $a$ be a globular sum of dimension $(n+1)$. We denote by $s_n(a)$ and $t_n(a)$ the globular sum defined in \ref{para:definition of source et but}. If $n$ is even, $s_n(a)^\flat\to a^{\sharp_n}$ is initial, and $t_n(a)^\flat\to a^{\sharp_n}$ is final. Dually, if $n$ is odd, $t_n(a)^\flat\to a^{\sharp_n}$ is initial, and $s_n(a)^\flat\to a^{\sharp_n}$ is final
\end{cor}
\begin{proof}
This is a direct consequence of propositions \ref{prop:when glob inclusion are left Gray deformation} and \ref{prop:left Gray deformation retract are initial}.
\end{proof}

\begin{prop}
\label{prop:trivialization are initial}
For any $n$, the morphism $\Ib_n:(\Db_{n+1})_t\to \Db_n^\flat$ is both initial and final.
\end{prop}
\begin{proof}
According to lemma \ref{cor:when glob inclusion are final and initial} there exists $\alpha\in\{-,+\}$ such that 
 $i_{n}^\alpha:(\Db_n)^\flat\to (\Db_{n+1})_t$ is initial.
 As $\Ib_n$ is a retraction of this morphism, and as initial morphisms are closed under left cancellation according to proposition \ref{prop:closed under colimit imply saturated}, $\Ib_n$ is initial. The second case follows by duality.
\end{proof}
These morphisms will be called the \wcnotion{marked trivializations}{marked trivialization}.

\begin{prop}
\label{prop:cotimes 1 to c is a trivialization}
Let $C$ be a marked $\io$-category.
The morphism $C\otimes[1]^\sharp\to C$ is in the smallest cocomplete $\infty$-groupoid of morphism containing the marked trivialization. In particular, this morphism is both initial and final. 
\end{prop}
\begin{proof}
We denote $K$ the smallest cocomplete $\infty$-groupoid of morphisms containing the marked trivializations.
As the $\infty$-groupoid of objects $C$ fulfilling the wanted property is closed by colimits, it is sufficient to demonstrate the result for $C$ being either $\Db_n^\flat$ or $(\Db_{n+1})_t$ for $n$ an integer. We will then proceed by induction. Suppose first that $C$ is $\Db_0^\flat$ or $(\Db_1)_t$. The first case is trivial, for the second one, remark that $(\Db_1)_t\otimes[1]^\sharp\sim [1]^\sharp\times[1]^\sharp\to [1]^\sharp$ is the horizontal colimit of the diagram
\[\begin{tikzcd}
	{[2]^\sharp} & {[1]^\sharp} & {[2]^\sharp} \\
	{[1]^\sharp} & {[1]^\sharp} & {[1]^\sharp}
	\arrow["{s^0}"', from=1-1, to=2-1]
	\arrow[from=1-2, to=2-2]
	\arrow["{s^1}", from=1-3, to=2-3]
	\arrow[from=1-2, to=1-1]
	\arrow[from=1-2, to=1-3]
	\arrow[from=2-2, to=2-1]
	\arrow[from=2-2, to=2-3]
\end{tikzcd}\]
and is then in $K$. Suppose now the result is true at the stage $(n-1)$. Let $C$ be $\Db_n^\flat$ (resp.$(\Db_{n+1})_t$). We set $D:=\Db_{n-1}^\flat$ (resp. $D:=(\Db_{n})_t$). We then have $C\sim [D,1]$. The equation \eqref{eq:eq for cylinder marked version} implies that $C\otimes[1]^\sharp\to C$ is the horizontal colimit of the diagram:
\[\begin{tikzcd}
	{[1]^\sharp\vee[D,1]} & {[D\otimes\{0\},1]} & {[D\otimes[1]^\sharp,1]} & {[D\otimes\{1\},1]} & {[D,1]\vee[1]^\sharp} \\
	{[D,1]} & {[D,1]} & {[D,1]} & {[D,1]} & {[D,1]}
	\arrow[from=2-2, to=2-1]
	\arrow[from=2-2, to=2-3]
	\arrow[from=2-4, to=2-3]
	\arrow[from=2-4, to=2-5]
	\arrow[from=1-1, to=2-1]
	\arrow[from=1-3, to=2-3]
	\arrow[from=1-4, to=2-4]
	\arrow[from=1-5, to=2-5]
	\arrow[from=1-2, to=1-1]
	\arrow[from=1-2, to=1-3]
	\arrow[from=1-4, to=1-3]
	\arrow[from=1-4, to=1-5]
	\arrow[from=1-2, to=2-2]
\end{tikzcd}\]
The leftest and rightest morphisms obviously are in $K$.
As marked trivializations are stable by suspension, the induction hypothesis implies that the middle vertical morphisms of the previous diagram are in $K$, which concludes the proof.
\end{proof}
\begin{prop}
\label{prop:cotimes 1 to ctimes 1 is a trivialization}
Let $C$ be a marked $\io$-category.
The morphism $C\otimes[1]^\sharp\to C\times [1]^\sharp$ is in the smallest cocomplete $\infty$-groupoid of morphism containing the marked trivializations. In particular, this morphism is both initial and final. 
\end{prop}
\begin{proof}
We denote $K$ the smallest cocomplete $\infty$-groupoid of morphisms containing the marked trivializations.
As the $\infty$-groupoid of objects $C$ fulfilling the wanted property is closed by colimits, it is sufficient to demonstrate the result for $C$ being either $\Db_n^\flat$ or $(\Db_{n+1})_t$ for $n$ an integer. If $C$ is either $(\Db_0)^\flat$ or $(\Db_1)_t$ the considered morphism is the identity. We then suppose that $n>0$. Let $C$ be $\Db_n^\flat$ (resp.$(\Db_{n+1})_t$). We set $D:=\Db_{n-1}^\flat$ (resp. $D:=(\Db_{n})_t$). We then have $C\sim [D,1]$. The equation \eqref{eq:eq for cylinder marked version} and the equation given in \ref{prop:example of a special colimit marked case} imply that $C\otimes[1]^\sharp\to C\times[1]^\sharp$ is the horizontal colimit of the diagram:
\[\begin{tikzcd}
	{[1]^\sharp\vee[D,1]} & {[D\otimes\{0\},1]} & {[D\otimes[1]^\sharp,1]} & {[D\otimes\{1\},1]} & {[D,1]\vee[1]^\sharp} \\
	{[1]^\sharp\vee[D,1]} & {[D,1]} & {[D,1]} & {[D,1]} & {[D,1]\vee[1]^\sharp}
	\arrow[from=2-2, to=2-1]
	\arrow[from=2-2, to=2-3]
	\arrow[from=2-4, to=2-3]
	\arrow[from=2-4, to=2-5]
	\arrow[from=1-1, to=2-1]
	\arrow[from=1-3, to=2-3]
	\arrow[from=1-4, to=2-4]
	\arrow[from=1-5, to=2-5]
	\arrow[from=1-2, to=1-1]
	\arrow[from=1-2, to=1-3]
	\arrow[from=1-4, to=1-3]
	\arrow[from=1-4, to=1-5]
	\arrow[from=1-2, to=2-2]
\end{tikzcd}\]
The proposition \ref{prop:cotimes 1 to c is a trivialization} then states that the middle vertical morphisms of the previous diagram are in $K$, which concludes the proof.
\end{proof}

\begin{prop}
\label{prop:suspension of initial}
If $i$ is an initial morphism, $[i,1]$ is a final morphism. Conversely, if $i$ is a final morphism, $[i,1]$ is an initial morphism.
\end{prop}
\begin{proof}
As the suspension preserves colimits, we can restrict to the case where $i$ is of shape $C\otimes\{0\}\to C\otimes[1]^\sharp$, and this is then a consequence of propositions \ref{prop:suspension of left Gray deformation retract} and \ref{prop:left Gray deformation retract are initial}.
\end{proof}

\begin{prop}
\label{prop:initial stable under product}
For any marked $\io$-category $K$, the functor $K\times\uvar:\ocatm\to \ocatm$ preserves initial and final morphisms. 
\end{prop}
\begin{proof}
The functor $K\times\uvar$ preserves colimits and this is then enough to show that 
it preserves left and right Gray deformation retracts, which is a consequence of proposition \ref{prop:exemple of right deformation retract}.
\end{proof}

\p
A \wcnotion{left cartesian fibration}{left or right cartesian fibration} is a morphism $f:C\to D$ between marked $\io$-categories having the unique right lifting property against initial morphisms.
A \textit{right cartesian fibration} is a morphism $f:C\to D$ between marked $\io$-categories having the unique right lifting property against final morphisms.

Left and right cartesian fibrations are stable under limits, retract, composition and  right cancellation according to the result of section \ref{section:Factorization system}.

The proposition \ref{prop:otimes et op marked version} implies that the full duality $(\uvar)^\circ$ sends left (resp. right) cartesian fibrations to right (resp. left) cartesian fibrations.

The construction \ref{cons:small object argument} produces a unique factorization system between initial (resp final) morphisms and left (resp. right) cartesian fibrations. If $f:A\to B$ is any morphism, we will denote by $\Fb f: A'\to B$ the left cartesian fibration obtained via this factorization system. 

\begin{prop}
\label{prop:left fib over flat}
If $f:C\to D^\flat$ is a left cartesian fibration, then the canonical morphism $(C^\natural)^\flat\to C$ is an equivalence. Conversely, any morphism $C^\flat \to D^\flat$ is a left cartesian fibration.
\end{prop}
\begin{proof}
The first assertion is a consequence of the fact that marked trivializations are initial. The second assertion is a direct consequence of proposition \ref{prop:cotimes 1 to c is a trivialization}.
\end{proof}

\begin{prop}
\label{prop:cartesian fibration between arrow}
Let $p:X\to C$ be a morphism, and $x,y$ two objects of $X$. Then, if $p$ is a right (resp. left) cartesian fibration, the induced morphism $p:\hom_X(x,y)\to \hom_C(x,y)$ is a left (resp. right) cartesian fibration.
\end{prop}
\begin{proof}
This is a direct consequence of proposition \ref{prop:suspension of initial}.
\end{proof}

\begin{prop}
\label{prop:left Gray transfomration stable under pullback along cartesian fibration}
Consider a cocartesian square
\[\begin{tikzcd}
	{X''} & {X'} & X \\
	{Y''} & {Y'} & Y
	\arrow["p"', from=1-3, to=2-3]
	\arrow["{p''}"', from=1-1, to=2-1]
	\arrow["{p'}"', from=1-2, to=2-2]
	\arrow["j", from=1-1, to=1-2]
	\arrow[from=1-2, to=1-3]
	\arrow["i"', from=2-1, to=2-2]
	\arrow[from=2-2, to=2-3]
	\arrow["\lrcorner"{anchor=center, pos=0.125}, draw=none, from=1-1, to=2-2]
	\arrow["\lrcorner"{anchor=center, pos=0.125}, draw=none, from=1-2, to=2-3]
\end{tikzcd}\]
If $p$ is a left (resp. right) cartesian fibration and $i$ is a right (resp. left) Gray deformation retract, then $p''\to p'$ is a right (resp. left) Gray deformation retract. Moreover, this left (resp. right) Gray deformation retract structure is functorial in $p$.

Similarly, if $p$ is a left (resp. right) cartesian fibration and $i$ is a right (resp. left) deformation retract, then $p''\to p'$ is a right (resp. left) deformation retract.  This left (resp. right) deformation retract structure is functorial in $p$. 
\end{prop}
\begin{proof}
We suppose that $p$ is a right cartesian fibration. By stability under pullbacks, so is $p'$.
Let $(i:C\to D,r,\phi)$ be a left Gray deformation retract structure.
We define the morphism $\psi$ as the lift of the following commutative square:
\[\begin{tikzcd}
	{X''\otimes[1]^\sharp\cup X'\otimes\{0\}} && {X'} \\
	{X'\otimes[1]^\sharp} & {Y''\otimes[1]^\sharp} & {Y'}
	\arrow[from=1-1, to=2-1]
	\arrow["{p'}", from=1-3, to=2-3]
	\arrow["{(X''\otimes\Ib)\cup id}", from=1-1, to=1-3]
	\arrow[from=2-1, to=2-2]
	\arrow[from=2-2, to=2-3]
	\arrow["\psi"{description}, dotted, from=2-1, to=1-3]
\end{tikzcd}\]
Remark that the restriction of $\psi$ to $X'\otimes\{1\}$ factors through $X''$ and then defines a retract $s:Y\to X$ of $j$. This provides a right Gray deformation structure for $p\to p''$. We proceed similarly for the dual case.

The functoriality of the Gray deformation retract structure comes from the fact that only functorial operations were used. Indeed, pullbacks, pushouts and the Gray tensor product are functorial. The formation of the lift $\psi$ is also functorial according to proposition \ref{prop:fonctorialite des relevement}.

To verify the second claim, one may utilize the same proof, exchanging $\otimes$ with $\times$.
\end{proof}

\begin{cor}
\label{cor:morphism between is an equivalence when equivalence on fiber}
Let $p:X\to B^\sharp$ and $q:Y\to B^\sharp$ be two left cartesian fibrations and $\phi:p\to q$ a morphism over $ B^\sharp$. The morphism $\phi$ is an equivalence if and only if, for any object $b$ of $B$, the induced morphism $\{b\}^*\phi :\{b\}^*X\to \{b\}^*Y$ is an equivalence.
\end{cor}
\begin{proof}
As $\tiPsh{\Theta}$ is locally cartesian closed, pullback commutes with special colimits, and as every $\io$-category is the special colimit of its $k$-truncation for $k\in \Nb$ according to proposition \ref{prop:example of a special colimit 2 marked case} , one can suppose that $B$ is a marked $(\infty,k)$-category for $k<\omega$, and we then proceed by induction on $k$. 
Suppose then the result is true for $(\infty,k)$-categories and that $B$ is an $(\infty,k+1)$-category. Remark first that $\phi$ induces an equivalence between $\tau_0(X)$ and $\tau_0(Y)$. 

Let $x$ and $y$ be two objects of $X$ and
 $v:[1]^\sharp\to B^\sharp$ be a cell whose source is $px$ and target $py$. This induces cartesian squares
\[\begin{tikzcd}
	{X_1} & {X_v} & X \\
	{Y_1} & {Y_v} & Y \\
	{\{1\}} & {[1]^\sharp} & {B^\sharp}
	\arrow["{\phi_1}", from=1-1, to=2-1]
	\arrow["\phi", from=1-3, to=2-3]
	\arrow["{\phi_v}", from=1-2, to=2-2]
	\arrow[from=2-3, to=3-3]
	\arrow[from=2-2, to=3-2]
	\arrow[from=2-1, to=3-1]
	\arrow[from=3-1, to=3-2]
	\arrow["v"', from=3-2, to=3-3]
	\arrow[from=2-1, to=2-2]
	\arrow[from=1-1, to=1-2]
	\arrow[from=2-2, to=2-3]
	\arrow[from=1-2, to=1-3]
	\arrow["\lrcorner"{anchor=center, pos=0.125}, draw=none, from=1-1, to=2-2]
	\arrow["\lrcorner"{anchor=center, pos=0.125}, draw=none, from=1-2, to=2-3]
	\arrow["\lrcorner"{anchor=center, pos=0.125}, draw=none, from=2-2, to=3-3]
	\arrow["\lrcorner"{anchor=center, pos=0.125}, draw=none, from=2-1, to=3-2]
\end{tikzcd}\]
By hypothesis, $\phi_1$ is an equivalence.
According to proposition \ref{prop:left Gray transfomration stable under pullback along cartesian fibration}, $\phi_1\to \phi_v$ is a right deformation retract, and according to proposition \ref{prop:Gray deformation retract and passage to hom}, this induces a cartesian square
\[\begin{tikzcd}
	{\hom_{X_1}(x,ry)} & {\hom_{X_v}(x,y)} \\
	{\hom_{Y_1}(\psi x,r\psi y)} & {\hom_{Y_v}(\psi x,\psi y)}
	\arrow[from=1-2, to=2-2]
	\arrow[from=1-1, to=2-1]
	\arrow[from=1-1, to=1-2]
	\arrow[from=2-1, to=2-2]
	\arrow["\lrcorner"{anchor=center, pos=0.125}, draw=none, from=1-1, to=2-2]
\end{tikzcd}\]
where horizontal morphisms are equivalences. By hypothesis, the left vertical one is an equivalence, and then, by two out of three, so is the right vertical one. 

We then have, for any $1$-cell $v$, the following cartesian squares
\[\begin{tikzcd}
	{\hom_{X_v}(x,y)} & {\hom_{X}(x,y)} \\
	{\hom_{Y_v}(\psi x,\psi y)} & {\hom_{Y}(\psi x,\psi y)} \\
	{\{v\}} & {\hom_{B}(px,py)^\sharp}
	\arrow["\sim"', from=1-1, to=2-1]
	\arrow[from=3-1, to=3-2]
	\arrow[from=2-2, to=3-2]
	\arrow[from=1-2, to=2-2]
	\arrow[from=2-1, to=2-2]
	\arrow[from=2-1, to=3-1]
	\arrow[from=1-1, to=1-2]
\end{tikzcd}\]
where the arrow labeled  by $\sim$ is an equivalence. As $\hom_{B}(px,py)^\sharp$ is an $(\infty,k)$-category, the induction hypothesis implies that $\hom_{X}(x,y)\to \hom_{Y}(\psi x,\psi y)$  is an equivalence. The morphism $\phi$ is then fully faithful, and as we already know that it is essentially surjective, this concludes the proof.
\end{proof}

\p We have by construction a factorization system in initial morphism followed by left cartesian fibration, and another one in final morphism followed by right cartesian fibration. We are willing to find an explicit expression for such factorization in some easy cases. We then fix $i:C^\flat \to D$ with $D$ being any marked $\io$-category.

If $C^\flat\to D$ is a functor between marked $\io$-categories, we define $D_{/C^{\flat}}$ and $D_{C^{\flat}/}$ as the following pullbacks 
\[\begin{tikzcd}
	{D_{C^{\flat}/}} & {D^{[1]^\sharp}} && {D_{/C^{\flat}}} & {D^{[1]^\sharp}} \\
	{C^{\flat}} & D && {C^{\flat}} & D
	\arrow["{(i_0^-)_!}", from=1-2, to=2-2]
	\arrow[from=2-1, to=2-2]
	\arrow[from=1-1, to=2-1]
	\arrow[from=1-1, to=1-2]
	\arrow["\lrcorner"{anchor=center, pos=0.125}, draw=none, from=1-1, to=2-2]
	\arrow["{(i_1^-)_!}", from=1-5, to=2-5]
	\arrow[from=1-4, to=2-4]
	\arrow[from=2-4, to=2-5]
	\arrow[from=1-4, to=1-5]
	\arrow["\lrcorner"{anchor=center, pos=0.125}, draw=none, from=1-4, to=2-5]
\end{tikzcd}\]
If $C$ is the terminal $\io$-category, this notation is compatible with the one of the slice over and under introduced in paragraph \ref{para:slice and joint}.

\begin{lemma}
\label{lemma:explicit factoryzation 1}
The morphism $i:C^\flat\to D_{/C^{\flat}}$ appearing in the following diagram
\[\begin{tikzcd}
	& D \\
	{C^{\flat}} & {D_{C^{\flat}/}} & {D^{[1]^\sharp}} \\
	& {C^{\flat}} & D
	\arrow["{(i_0^-)_!}", from=2-3, to=3-3]
	\arrow[from=3-2, to=3-3]
	\arrow[from=2-2, to=3-2]
	\arrow[from=2-2, to=2-3]
	\arrow["\lrcorner"{anchor=center, pos=0.125}, draw=none, from=2-2, to=3-3]
	\arrow[curve={height=-12pt}, from=2-1, to=1-2]
	\arrow[curve={height=-12pt}, from=1-2, to=2-3]
	\arrow["id"', from=2-1, to=3-2]
	\arrow["i", dashed, from=2-1, to=2-2]
\end{tikzcd}\]
is initial.
\end{lemma}
\begin{proof}
Using proposition \ref{prop:associativity of Gray amput2}, we have a natural transformation
$$(\uvar\otimes[1]^\sharp)\otimes[1]^\sharp \sim \uvar\otimes([1]^\sharp\times[1]^\sharp)
\xrightarrow{\uvar\otimes\psi} \uvar\otimes[1]^\sharp$$
where $\psi$ sends $(\epsilon,\epsilon')$ on $\max(\epsilon,\epsilon')$. This induces a natural transformation $D^{[1]^\sharp}\to (D^{[1]^\sharp})^{[1]^\sharp}$, corresponding by adjunction to transformation $\phi:D^{[1]^\sharp}\otimes[1]^\sharp\to D^{[1]^\sharp}$. We set $r:D_{C^{\flat}/}\to C^\flat$ as the canonical projection. Eventually, remark that $(i,r,\phi)$ is a left Gray deformation retract. According to proposition \ref{prop:left Gray deformation retract are initial}, this concludes the proof.
\end{proof}

\begin{lemma}
\label{lemma:explicit factoryzation 2}
The composite $q:D_{C^{\flat}/}\to D^{[1]^\sharp}\xrightarrow{(i_0^+)_!} D$ is a left cartesian fibration.
\end{lemma}
\begin{proof}
Consider a commutative diagram
\begin{equation}
\label{eq:lemma:explicit factoryzation 2}
\begin{tikzcd}
	{K\otimes\{0\}} & {D_{C^{\flat}/}} \\
	{K\otimes[1]^\sharp} & D
	\arrow[from=1-1, to=2-1]
	\arrow[from=1-1, to=1-2]
	\arrow[from=2-1, to=2-2]
	\arrow[from=1-2, to=2-2]
\end{tikzcd}
\end{equation}
The $\infty$-groupoid of lifts of this previous diagram is equivalent to the $\infty$-groupoid of pairs consisting of a commutative triangle 
\[\begin{tikzcd}
	{K\otimes\{0\}\otimes\{0\}} \\
	{K\otimes[1]^\sharp\otimes\{0\}} & D
	\arrow[from=1-1, to=2-1]
	\arrow["f", from=1-1, to=2-2]
	\arrow[dashed, from=2-1, to=2-2]
\end{tikzcd}\]
where $f$ is induced by $K\otimes\{0\}\to D_{C^{\flat}/}$,
 and a lift in the induced diagram
\[\begin{tikzcd}
	{K\otimes\{0\}\otimes[1]^\sharp\cup K\otimes[1]^\sharp\otimes\{1\} \cup K\otimes[1]^\sharp\otimes\{0\}} & D \\
	{K\otimes[1]^\sharp\otimes[1]^\sharp} & 1
	\arrow[from=1-1, to=1-2]
	\arrow[from=1-1, to=2-1]
	\arrow[from=2-1, to=2-2]
	\arrow[from=1-2, to=2-2]
	\arrow[dashed, from=2-1, to=1-2]
\end{tikzcd}\]
According to proposition \ref{prop:cotimes 1 to c is a trivialization}, the morphism $K\otimes[1]^\sharp\otimes\{0\}\to C^\flat$ factors through a morphism $K\to C^\flat$, and is then uniquely determined by $f:K\otimes\{0\}\otimes\{0\}\to C^\flat$,  and proposition \ref{prop:associativity of Gray amput2} provides a natural equivalence between $(K\otimes[1]^\sharp)\otimes[1]^\sharp$ and $K\otimes([1]^\sharp\times [1]^\sharp)$. The $\infty$-groupoid of lifts of the diagram \eqref{eq:lemma:explicit factoryzation 2} is then equivalent  to the  $\infty$-groupoid of lifts of  the left square of the following diagram
\[\begin{tikzcd}[column sep =0.7cm]
	{K\otimes\{0\}\otimes[1]^\sharp\cup K\otimes[1]^\sharp\otimes\{1\} \cup K\otimes[1]^\sharp\otimes\{0\}} & {K\otimes[1]^\sharp\cup K\otimes[1]^\sharp} & D \\
	{K\otimes[1]^\sharp\otimes[1]^\sharp} & {K\otimes[2]^\sharp} & 1
	\arrow[from=1-1, to=2-1]
	\arrow[from=1-3, to=2-3]
	\arrow[from=2-1, to=2-2]
	\arrow[""{name=0, anchor=center, inner sep=0}, from=1-1, to=1-2]
	\arrow[from=1-2, to=2-2]
	\arrow[from=1-2, to=1-3]
	\arrow[from=2-2, to=2-3]
	\arrow[dashed, from=2-2, to=1-3]
	\arrow["\lrcorner"{anchor=center, pos=0.125, rotate=180}, draw=none, from=2-2, to=0]
\end{tikzcd}\]
As $K\otimes[1]^\sharp\coprod_{K\otimes[0]}K\otimes[1]^\sharp\to K\otimes[2]^\sharp$ is an equivalence, this   $\infty$-groupoid is contractible.
\end{proof}

\begin{prop}
\label{prop:explicit factoryzation}
The factorisation of $p:C^\flat \to D$
in an initial morphism followed by a left cartesian fibration is
$$C^\flat \xrightarrow{i} D_{C^\flat/}\xrightarrow{q} D,$$
and its factorization in a final morphism and a right cartesian fibration is 
$$C^\flat \xrightarrow{i} D_{/C^\flat}\xrightarrow{q} D.$$
\end{prop}
\begin{proof}
This is a direct application of lemma \ref{lemma:explicit factoryzation 1} and \ref{lemma:explicit factoryzation 1} and of their dual version.
\end{proof}
The more important example of the previous proposition is the case $C:=\{a\}$. In this case, the corresponding left cartesian fibration is the slice of $D$ under $a$
$$D_{a/}\to D$$
 and the corresponding right cartesian fibration is the slice of $D$ over $a$
$$D_{/a}\to D.$$
\p 
Let $p:X\to Y$ be a morphism between $\io$-categories. A marked $1$-cell $v:x\to x'$ is \wcnotion{left cancellable}{left or right cancellable $1$-cell} if for any $y$, the following natural square is cartesian:
\[\begin{tikzcd}
	{\hom_X(x',y)} & {\hom_X(x,y)} \\
	{\hom_Y(px',py)} & {\hom_Y(px,py)}
	\arrow["{v_!}", from=1-1, to=1-2]
	\arrow["{p(v)_!}"', from=2-1, to=2-2]
	\arrow[from=1-2, to=2-2]
	\arrow[from=1-1, to=2-1]
\end{tikzcd}\]

Conversely, a $1$-cell $v:y\to y'$ is \textit{right cancellable} if for any $x$, the following natural square is cartesian:
\[\begin{tikzcd}
	{\hom_X(x,y)} & {\hom_X(x,y')} \\
	{\hom_Y(px,py)} & {\hom_Y(px,py')}
	\arrow["{v_!}", from=1-1, to=1-2]
	\arrow["{p(v)_!}"', from=2-1, to=2-2]
	\arrow[from=1-2, to=2-2]
	\arrow[from=1-1, to=2-1]
\end{tikzcd}\]

\begin{lemma}
\label{lemma:technical lemma on cancellable cell 1}
Let $p$ be a morphism. 
The following conditions are equivalent:
\begin{enumerate}
\item $p$ has the unique right lifting property against $\{0\}\to [1]^\sharp$ and marked $1$-cells are left cancellable.
\item $p$ has the unique right lifting property against $[a,1]\xrightarrow{\triangledown} [1]^\sharp\vee[a,1]$ for any object $a$ of $t\Theta$.
\item $p$ has the unique right lifting property against $[a,1]\xrightarrow{\triangledown} [1]^\sharp\vee[a,1]$ and $[1]^\sharp\xrightarrow{\triangledown}[1]^\sharp\vee[1]^\sharp$ for any object $a$ of $t\Theta$.
\end{enumerate}
Conversely, the following are equivalent:
\begin{enumerate}
\item[(1)'] $p$ has the unique right lifting property against $\{1\}\to [1]^\sharp$ and marked $1$-cells are right cancellable.
\item[(2)'] $p$ has the unique right lifting property against $[a,1]\xrightarrow{\triangledown} [a,1]\vee[1]^\sharp$ for any object $a$ of $t\Theta$.
\item[(3)'] $p$ has the unique right lifting property against $[a,1]\xrightarrow{\triangledown}[a,1]\vee[1]^\sharp$ and $[1]^\sharp\xrightarrow{\triangledown}[1]^\sharp\vee[1]^\sharp$ for any object $a$ of $t\Theta$.
\end{enumerate}
\end{lemma}
\begin{proof}
 The fact that $1$-cells are left cancellable is equivalent to asking that $i$ has the unique right lifting property against 
$$[a,1]\amalg_{\{0\}} [1]^\sharp\to [1]^\sharp\vee[a,1]$$
 for any object $a$ of $t\Theta$.
 Suppose that $p$ fulfills $(1)$.
As the class of morphisms having the unique right lifting property against $p$ are closed under composition and by left cancellation according to \ref{prop:closed under colimit imply saturated}, this implies that $p$ has the unique right lifting property against 
$$[a,1]\xrightarrow{\triangledown} [1]^\sharp\vee[a,1]$$
and then that $(1)\Rightarrow (2)$. 

Suppose now that $p$ fulfills $(2)$. Remark that we have a retract
\[\begin{tikzcd}
	{\{0\}} & {[1]} & {\{0\}} \\
	{[1]^\sharp} & {[1]^\sharp\vee[1]} & {[1]^\sharp}
	\arrow[from=1-1, to=2-1]
	\arrow[hook, from=2-1, to=2-2]
	\arrow["id\vee\Ib"', from=2-2, to=2-3]
	\arrow[from=1-1, to=1-2]
	\arrow["\triangledown"', from=1-2, to=2-2]
	\arrow["\Ib", from=1-2, to=1-3]
	\arrow[from=1-3, to=2-3]
\end{tikzcd}\]
and as the class of morphisms having the unique right lifting property against $p$ is closed under retracts, this implies that $p$ has the unique right lifting property against $\{0\}\to [1]^\sharp$. By stability by left cancellation, $p$ has the unique right lifting property against 
$$[a,1]\amalg_{\{0\}} [1]^\sharp\to [1]^\sharp\vee[a,1].$$
As remarked above, this implies that $1$-cells are left cancellable. We then have $(1)\Leftrightarrow (2)$.

There is an obvious implication $(3)\Rightarrow (2)$. For the converse, remark that the
class of morphisms having the unique right lifting property against $p$ is closed under colimits and then contains $\{0\}\to [1]^\sharp\vee[1]^\sharp$, and so by left cancellation, it includes $ [1]^\sharp\xrightarrow{\triangledown}[1]^\sharp\vee[1]^\sharp$.
The proof of the equivalence of $(1)'$, $(2)'$ and $(3)'$ is symetrical.
\end{proof}

\begin{lemma}
\label{lemma:technical lemma on cancellable cell 3}
Let $p:X\to Y$ be a morphism having the unique right lifting property against marked trivializations, such that for any element $a$ of $t\Theta$, and any cartesian squares: 	
\[\begin{tikzcd}
	{X''} & {X'} & X \\
	{[a,1]} & {[1]^\sharp\vee[a,1]} & Y
	\arrow["p"', from=1-3, to=2-3]
	\arrow["{p''}"', from=1-1, to=2-1]
	\arrow["{p'}"', from=1-2, to=2-2]
	\arrow["k", from=1-1, to=1-2]
	\arrow[from=1-2, to=1-3]
	\arrow[hook, from=2-1, to=2-2]
	\arrow[from=2-2, to=2-3]
	\arrow["\lrcorner"{anchor=center, pos=0.125}, draw=none, from=1-1, to=2-2]
	\arrow["\lrcorner"{anchor=center, pos=0.125}, draw=none, from=1-2, to=2-3]
\end{tikzcd}\]
the square $p''\to p'$ is a right deformation retract.
Then, $p$ has the unique right lifting property against $[a,1]\xrightarrow{\triangledown} [1]^\sharp\vee[a,1]$ for any object $a$ of $t\Theta$. 
\end{lemma}
\begin{proof}
Suppose given a square
\[\begin{tikzcd}
	{[a,1]} & X \\
	{[1]^\sharp\vee[a,1]} & Y
	\arrow["p"', from=1-2, to=2-2]
	\arrow["\triangledown"', from=1-1, to=2-1]
	\arrow[from=1-1, to=1-2]
	\arrow["g"', from=2-1, to=2-2]
\end{tikzcd}\]
and let $p'$ and $p''$ be the morphisms appearing in the following cartesian squares:
\[\begin{tikzcd}
	{X''} & {X'} & X \\
	{[a,1]} & {[1]^\sharp\vee[a,1]} & Y
	\arrow["p"', from=1-3, to=2-3]
	\arrow["{p''}"', from=1-1, to=2-1]
	\arrow["{p'}"', from=1-2, to=2-2]
	\arrow["k", from=1-1, to=1-2]
	\arrow[from=1-2, to=1-3]
	\arrow[hook, from=2-1, to=2-2]
	\arrow["g"', from=2-2, to=2-3]
	\arrow["\lrcorner"{anchor=center, pos=0.125}, draw=none, from=1-1, to=2-2]
	\arrow["\lrcorner"{anchor=center, pos=0.125}, draw=none, from=1-2, to=2-3]
\end{tikzcd}\]
To show the proposition, one has to demonstrate that the induced diagram
\[\begin{tikzcd}
	{[a,1]} & {X'} \\
	{[1]^\sharp\vee[a,1]} & {[1]^\sharp\vee[a,1]}
	\arrow["\triangledown"', from=1-1, to=2-1]
	\arrow["{p'}"', from=1-2, to=2-2]
	\arrow["j", from=1-1, to=1-2]
	\arrow["id"', from=2-1, to=2-2]
	\arrow["\lrcorner"{anchor=center, pos=0.125}, draw=none, from=1-1, to=2-2]
\end{tikzcd}\]
admits a unique lifting. We denote by $x_0$ and $x_2$ the image of the object of $[a,1]$ via the morphism $j$, and $(k:X''\to X',r,\phi)$ the left deformation retract existing by hypothesis. According to the dual version of proposition \ref{prop:retraction criter}, the unique marked $1$-cell in $X'$ over $[1]^\sharp\hookrightarrow [1]^\sharp\vee[a,1]$ with $x_0$ for source is $\phi(x_0):x_0\to r(x_0)$.
The $\infty$-groupoid of lifts of this diagram is then equivalent to the $\infty$-groupoid of lifts of the following diagram
\[\begin{tikzcd}
	\emptyset & {\hom_{X'}(rx_0,x_2)} \\
	a & {\hom_{X'}(x_0,x_2)}
	\arrow["{{\phi_{x_0}}_!}", from=1-2, to=2-2]
	\arrow[from=2-1, to=2-2]
	\arrow[from=1-1, to=2-1]
	\arrow[from=1-1, to=1-2]
\end{tikzcd}\]
However, the right vertical morphism is an isomorphism according to proposition \ref{prop:Gray deformation retract and passage to hom} which concludes the proof.
\end{proof}
\p Keeping in mind the last lemma, we define \wcnotation{$\I_{g}$}{(ig@$\I_g$} and \wcnotation{$\F_{g}$}{(fg@$\F_g$} as the smallest sets of morphisms of $\zocatm$ fullfilling these conditions:
\begin{enumerate}
\item for any $a\in \Theta^t$, $[a,1]\hookrightarrow[1]^\sharp\vee[a,1]$ is in $\F_g$ and $[a,1]\hookrightarrow[a,1]\vee[1]^\sharp$ is in $\I_g$
\item for any $i$ in $\F_g$, $[i,1]$ is in $\I_g$, for any $j$ in $\I_g$, $[i,1]$ is in $\F_g$,
\end{enumerate}
Propositions \ref{prop:left Gray deformation retract stable under pushout} and \ref{prop:suspension of left Gray deformation retract} then imply that morphisms of $\I_g$ are left Gray deformation retracts and morphisms of $\F_g$ are right Gray deformation retracts.

\p 
We extend by induction the definition of right and left cancellable to cells of any dimension as follows: a $n$-cell $v$ is \wcnotion{left or right cancellable}{left cancellable $n$-cell} (resp. \textit{right cancellable}) if the corresponding $(n-1)$-cell of $\hom_X(x,y)$ is left cancellable (resp. right cancellable) for the morphism $\hom_X(x,y)\to \hom_Y(px,py)$, where $x$ and $y$ denote the $0$-sources and $0$-but of $v$.

\begin{lemma}
\label{lemma:technical lemma on cancellable cell 2}
Let $p':X'\to Y'$ be a morphism such that $p$ has the unique right lifting property against marked trivializations and suppose that we have a left Gray deformation retract $p'\to p$. We denote by $(r:Y'\to Y,i,\phi)$ the left deformation retract structure induced on the codomain, and suppose that the deformation $\phi:Y\otimes[1]^\sharp\to Y$ factors through $\psi:Y\times[1]^\sharp\to Y$. Then, the square $p'\to p$ is a left deformation retract.
\end{lemma}
\begin{proof}
Proposition \ref{prop:cotimes 1 to ctimes 1 is a trivialization} states that $Y\otimes[1]^\sharp\to Y\times [1]^\sharp$ is a colimit of marked trivializations. There is then a lift in the following diagram:
\[\begin{tikzcd}
	{X\otimes[1]^\sharp} && X \\
	{X\times[1]^\sharp} & {Y\times[1]^\sharp} & Y
	\arrow[from=1-1, to=2-1]
	\arrow[from=2-1, to=2-2]
	\arrow[from=1-3, to=2-3]
	\arrow["{\phi'}", from=1-1, to=1-3]
	\arrow["\psi"', from=2-2, to=2-3]
	\arrow["{\psi'}"{description}, dotted, from=2-1, to=1-3]
\end{tikzcd}\]
where $\phi'$ is the deformation induced on domains.
This endows $p'\to p$ with a structure of left deformation retract, where the retraction is the same, and the deformation is given by $(\psi',\psi)$.
\end{proof}

\begin{theorem}
\label{theo:other characterisation of left caresian fibration}
Consider the following shape of diagram
\begin{equation}
\label{eq:prop:other characterisation of left caresian fibration}
\begin{tikzcd}
	{X''} & {X'} & X \\
	{Y''} & {Y'} & Y
	\arrow["p"', from=1-3, to=2-3]
	\arrow["{p''}"', from=1-1, to=2-1]
	\arrow["{p'}"', from=1-2, to=2-2]
	\arrow[from=1-1, to=1-2]
	\arrow[from=1-2, to=1-3]
	\arrow["i"', from=2-1, to=2-2]
	\arrow[from=2-2, to=2-3]
	\arrow["\lrcorner"{anchor=center, pos=0.125}, draw=none, from=1-1, to=2-2]
	\arrow["\lrcorner"{anchor=center, pos=0.125}, draw=none, from=1-2, to=2-3]
\end{tikzcd}
\end{equation}
The following are equivalent:
\begin{enumerate}
\item The morphism $p$ is a left cartesian fibration.
\item $p$ has the unique right lifting property against marked trivialization, and for any diagram of shape \eqref{eq:prop:other characterisation of left caresian fibration},
if $i$ is a right Gray deformation retract, so is $p''\to p'$. 
\item $p$ has the unique right lifting property against marked trivialization and, for any diagram of shape \eqref{eq:prop:other characterisation of left caresian fibration},
if $i$ is in $\F_g$, the square $p''\to p'$ is a right Gray deformation retract.
\item For any even integer $n$, $p$ has the unique right lifting property against $i_n^+:\Db_{n}\to (\Db_{n+1})_t$ and marked $n$-cells are right cancellable; for any odd integer $p$ has the unique right lifting property against $i_n^-:\Db_{n}\to (\Db_{n+1})_t$ and marked $n$-cells are left cancellable.
\item $p$ as the unique right lifting property against $\{0\}\to [1]^\sharp$, marked $1$-cells are left cancellable, and
for any pair of objects $(x,y)$ of $X$, $\hom_X(x,y)\to \hom_Y(px,py)$ is a right cartesian fibration.
\end{enumerate} 
Conversely, the following are equivalent:
\begin{enumerate}
\item[(1)'] The morphism $p$ is a right cartesian fibration.
\item[(2)'] $p$ has the unique right lifting property against marked trivialization and for any diagram of shape \eqref{eq:prop:other characterisation of left caresian fibration},
if $i$ is a left Gray deformation retract, so is $p''\to p'$.
\item[(3)'] $p$ has the unique right lifting property against marked trivialization, and for any diagram of shape \eqref{eq:prop:other characterisation of left caresian fibration},
if $i$ is in $\I_g$, the square $p''\to p'$ is a left Gray deformation retract.
\item[(4)'] For any even integer $n$, $p$ has the unique right lifting property against $i_n^-:\Db_{n}\to (\Db_{n+1})_t$ and marked $n$-cells are left cancellable; for any odd integer $p$ has the unique right lifting property against $i_n^+:\Db_{n}\to (\Db_{n+1})_t$ and marked $n$-cells are right cancellable.
\item[(5)'] $p$ as the unique right lifting property against $\{1\}\to [1]^\sharp$, marked $1$-cells are right cancellable, and
for any pair of objects $(x,y)$ of $X$, $\hom_X(x,y)\to \hom_Y(px,py)$ is a left cartesian fibration.
\end{enumerate} 

\end{theorem}
\begin{proof}
The implication from $(1)$ to $(2)$ and $(1)'$ to $(2)'$ is the content of proposition \ref{prop:left Gray transfomration stable under pullback along cartesian fibration}.

The implication from $(2)$ to $(3)$ and $(2)'$ to $(3)'$ comes from the fact that $\I_g$ (resp. $\F_g$) consists of right (resp. left) Gray deformation retracts.

Suppose now that $p$ fulfills condition $(3)$. Lemma \ref{lemma:technical lemma on cancellable cell 2} implies that if $i$ is of shape $[a,1]\hookrightarrow [1]^\sharp\vee[a,1]$ for $a:t\Theta$, $p''\to p'$ is a right deformation retract. Lemma 
\ref{lemma:technical lemma on cancellable cell 3} and \ref{lemma:technical lemma on cancellable cell 1} then imply that $p$ has the unique right lifting property against $\{0\}\to [1]^\sharp$ and marked $1$-cells are left cancellable.

We are now willing to show that for any pair of objects $(x,y)$, $\hom_X(x,y)\to \hom_Y(px,py)$ fulfills condition $(3)'$, and an obvious induction will complete the proof of $(3)\Rightarrow (4)$.
We then consider $x,y$ two objects of $X$, $i:b\to a$ in $\I_g$ and any morphism $a\to \hom_Y(px,py)$. The previous data induces a pullback square
\[\begin{tikzcd}
	{X''} & {X'} & X \\
	{[b,1]} & {[a,1]} & Y
	\arrow["p"', from=1-3, to=2-3]
	\arrow["{p''}"', from=1-1, to=2-1]
	\arrow["{p'}"', from=1-2, to=2-2]
	\arrow[from=1-1, to=1-2]
	\arrow[from=1-2, to=1-3]
	\arrow["{[i,1]}"', from=2-1, to=2-2]
	\arrow[from=2-2, to=2-3]
	\arrow["\lrcorner"{anchor=center, pos=0.125}, draw=none, from=1-1, to=2-2]
	\arrow["\lrcorner"{anchor=center, pos=0.125}, draw=none, from=1-2, to=2-3]
\end{tikzcd}\]
where the bottom right morphism sends $\{0\}$ to $px$ and $\{1\}$ to $py$.
By construction, $[i,1]$ is in $\F_g$, and so by assumption, the morphism $p'\to p''$ is a right Gray deformation retract. Applying the functor $\hom_{\uvar}(\uvar,\uvar)$ we get the following pullback diagram:
\[\begin{tikzcd}
	{\hom_{X''}(x,y)} & {\hom_{X'}(x,y)} & {\hom_{X}(x,y)} \\
	b & a & {\hom_{Y}(px,py)}
	\arrow["{\tilde{p}}"', from=1-3, to=2-3]
	\arrow["{\tilde{p}''}"', from=1-1, to=2-1]
	\arrow["{\tilde{p}'}"', from=1-2, to=2-2]
	\arrow[from=1-1, to=1-2]
	\arrow[from=1-2, to=1-3]
	\arrow["i"', from=2-1, to=2-2]
	\arrow[from=2-2, to=2-3]
	\arrow["\lrcorner"{anchor=center, pos=0.125}, draw=none, from=1-1, to=2-2]
	\arrow["\lrcorner"{anchor=center, pos=0.125}, draw=none, from=1-2, to=2-3]
\end{tikzcd}\]
and the dual version of proposition \ref{prop:Gray deformation retract and passage to hom v2} implies that $\tilde{p}''\to \tilde{p}'$ is a left Gray deformation retract. As this is true for any $i:b\to a$ in $\I_g$, for any object of $X$, and any $a\to \hom_Y(px,py)$, this implies that $\hom_X(x,y)\to \hom_Y(px,py)$ fulfills condition $(3)'$. As mentioned above, an obvious induction induces $(3)\Rightarrow (4)$. We show similarly $(3)'\Rightarrow (4)'$.

Now let's show $(4)\Rightarrow (1)$ and $(4)'\Rightarrow (1)'$. We show by induction on $n$ that for any 
element $a$ of $t\Gb_n:=\{\Db_k\}_{0\leq k\leq n}\cup \{(\Db_k)_t\}_{1\leq k\leq n}$, if $p$ fulfills $(4)$ (resp. $(4)'$) $p$ has the unique right lifting property against
 $a\otimes\{0\}\to a\otimes[1]^\sharp$ (against $a\otimes\{1\}\to a\otimes[1]^\sharp$).
 
Suppose then that this is true at the stage $n$, and suppose that $p$ fulfills $(4)$. 
Let $a$ be an object of $t\Gb_n$. 	Remark that according to the equation \eqref{eq:eq for cylinder marked version}, $[a,1]\otimes\{0\}\to [a,1]\otimes[1]^\sharp$ fits in the sequence of pushouts
\[\begin{tikzcd}
	{[0]} & {[a,1]\otimes \{0\}} \\
	{[1]^\sharp} & {[a,1]\vee[1]^\sharp} & {[a\otimes\{1\},1]} \\
	{[a,1]} & {[a,1]\vee[1]^\sharp\cup[a\otimes[1]^\sharp,1]} & {[a\otimes[1]^\sharp,1]} \\
	{[1]^\sharp\vee[a,1]} & {[a,1]\otimes[1]^\sharp}
	\arrow["{i_0^-}"', from=1-1, to=2-1]
	\arrow[""{name=0, anchor=center, inner sep=0}, "{i_0^+}", from=1-1, to=1-2]
	\arrow[from=2-1, to=2-2]
	\arrow[hook, from=1-2, to=2-2]
	\arrow[""{name=1, anchor=center, inner sep=0}, from=2-3, to=2-2]
	\arrow[from=3-3, to=3-2]
	\arrow[from=2-2, to=3-2]
	\arrow[from=2-3, to=3-3]
	\arrow[""{name=2, anchor=center, inner sep=0}, from=3-1, to=3-2]
	\arrow[from=4-1, to=4-2]
	\arrow[from=3-2, to=4-2]
	\arrow["\triangledown"', from=3-1, to=4-1]
	\arrow["\lrcorner"{anchor=center, pos=0.125, rotate=180}, draw=none, from=2-2, to=0]
	\arrow["\lrcorner"{anchor=center, pos=0.125, rotate=90}, draw=none, from=3-2, to=1]
	\arrow["\lrcorner"{anchor=center, pos=0.125, rotate=180}, draw=none, from=4-2, to=2]
\end{tikzcd}\]
By induction hypothesis, for any pair of objects $(x,y)$ of $X$, $\hom_X(x,y)\to \hom_Y(px,py)$ has the unique right lifting property against $a\otimes\{1\}\to a\otimes[1]^\sharp$ for $a\in t\Gb_n$. Furthermore, lemma \ref{lemma:technical lemma on cancellable cell 1} implies that $p$ has the unique right lifting property against $\triangledown:[a,1]\to [1]^\sharp\vee[a,1]$. The morphism $p$ then has the unique right lifting property against $[a\otimes\{1\},1]\to [a\otimes[1]^\sharp,1]$ for $a\in t\Gb_n$. The class of morphisms having the unique right lifting property against $p$ being closed under colimits, this implies that it includes $[a,1]\otimes\{0\}\to [a,1]\otimes[1]^\sharp$. To conclude, one has to show that $p$ has the unique right lifting property against $[1]^\sharp\times\{0\}\to [1]^\sharp\times [1]^\sharp$. Remark that according to proposition \ref{prop:example of a special colimit marked case}, $[1]^\sharp\times\{0\}\to [1]^\sharp\times[1]^\sharp$ fits in the sequence of pushouts:
\[\begin{tikzcd}
	{[0]} & {[1]^\sharp\times\{0\}} \\
	{[1]^\sharp} & {[1]^\sharp\vee[1]^\sharp} & {[1]^\sharp} \\
	& {[1]^\sharp\times[1]^\sharp} & {[1]^\sharp\vee[1]^\sharp}
	\arrow["{i_0^-}"', from=1-1, to=2-1]
	\arrow["{i_0^+}", from=1-1, to=1-2]
	\arrow["\triangledown"', from=2-3, to=2-2]
	\arrow["\triangledown", from=2-3, to=3-3]
	\arrow[hook, from=2-1, to=2-2]
	\arrow[from=2-2, to=3-2]
	\arrow[from=3-3, to=3-2]
	\arrow[hook, from=1-2, to=2-2]
	\arrow["\lrcorner"{anchor=center, pos=0.125, rotate=180}, draw=none, from=2-2, to=1-1]
	\arrow["\lrcorner"{anchor=center, pos=0.125, rotate=90}, draw=none, from=3-2, to=2-3]
\end{tikzcd}\]
According to lemma \ref{lemma:technical lemma on cancellable cell 1}, $p$ has the unique right lifting property against $\triangledown:[1]^\sharp\to [1]^\sharp\vee[1]^\sharp$ and so also against $[1]^\sharp\times\{0\}\to [1]^\sharp\times [1]^\sharp$. This concludes the proof of the implication $(4)\Rightarrow (1)$. We show similarly $(4)'\Rightarrow (1)'$.

Eventually, the equivalences $(1)\Rightarrow (5)$ and $(1)'\Rightarrow (5)'$ are a consequence of proposition \ref{prop:cartesian fibration between arrow} and of the implications $(1)\Rightarrow (4)$ and $(1)'\Rightarrow (4)'$. The implications $(5)\Rightarrow (4)$ and $(5)'\Rightarrow (4)'$ are a consequence of the implications  $(1)'\Rightarrow (4)'$ and $(1)\Rightarrow (4)$ applied to the morphisms $\hom_X(x,y)\to \hom_Y(px,py)$ for all objects $x,y$.
\end{proof}

\begin{cor}
\label{cor:on the fact that fib are define against representable}
A morphism $p:X\to A^\sharp$ is a left cartesian fibration if and only if for any globular sum $b$ and morphism $j:b\to A$, $j^*p$ is a left cartesian fibration over $b^\sharp$.
\end{cor} 
\begin{proof}
This is a direct consequence of the equivalence between conditions $(1)$ and $(3)$ of theorem \ref{theo:other characterisation of left caresian fibration}, and the fact that the codomains of marked trivializations and the codomains of morphisms of $\F_g$ are marked globular sums. 
\end{proof}

\subsection{Cartesian fibration are exponentiable}
\label{subsection:A criterion to be a left cartesian fibration}

We recall that a {marked globular sum} is a marked $\io$-category whose underlying $\io$-category is a globular sum and such that for any pair of integers $k\leq n$, and any pair of $k$-composable $n$-cells $(x,y)$, $x\circ_k y$ is marked if and only if $x$ and $y$ are marked.

 A morphism $i:a\to b$ between marked globular sums is {globular} if the morphism $i^\natural$ is {globular}.

 A globular morphism $i$ between marked globular sums is then a discrete Conduché functor, which implies according to proposition \ref{prop:pullback by conduch marked preserves colimit} that the functor $i^*:\ocatm_{/b}\to \ocatm_{/a}$ preserves colimits.

\p Let $b$ be a globular sum and $f:X\to b^\sharp$ a morphism. We say that $f$ is \wcnotion{$b$-exponentiable}{expo@$b$-exponentiable} if 
the canonical morphism $$\colim_{i:\Sp_b^\sharp} {i}^*f\to f$$ is an equivalence. 

\begin{prop}
\label{prop:exponantiable stable under colim}
Let $F:I\to \ocatm_{/b^\sharp}$ be a functor which is pointwise $b$-exponentiable. The morphism $\colim_I F$ is $b$-exponentiable 
\end{prop}
\begin{proof}
Remark that all morphisms $\Db_n^\sharp\to b^\sharp$ in $\Sp^\sharp_b$ are globular, and so are discrete Conduché functors. We then have a sequence of equivalences
$$\colim_{i:\Sp_b^\sharp} {i}^*\colim_I F\sim \colim_I\colim_{i:\Sp_b^\sharp} {i}^*F\sim \colim_I F.$$
\end{proof}

\begin{prop}
\label{prop:how to create exponentiable}
Let $a$ be a globular sum, and $f:X\to a^\sharp$ be a morphism. The induced morphism $\colim_{i:\Sp_{a}^\sharp}i^*f$ is $a$-exponentiable.
\end{prop}
\begin{proof}
As marked globular morphisms are marked discrete Conduché functors, for any $j:\Db_n^\sharp\to a^\sharp\in \Sp_a$, $j^*\colim_{i:\Sp_{a}^\sharp}i^*f$ is equivalent to $j^*f$. We then have a sequence of equivalences
$$ \colim_{j:\Sp_{a}^\sharp}j^* \colim_{i:\Sp_{a}^\sharp}i^*f \sim \colim_{j:\Sp_{a}^\sharp}j^*f .$$
\end{proof}

\begin{prop}
\label{prop:exponantiable stable under pullback}
Let $f:X\to b^\sharp$ be exponentiable in $b$ and $j:a^\sharp\to b^\sharp$ a globular morphism. The morphism $j^*f:X\to a^\sharp$ is exponentiable in $a$.
\end{prop}
\begin{proof}
The morphism $j:a^\sharp\to b^\sharp$ is a marked discrete Conduché functor, so $j^*$ preserves colimits according to proposition \ref{prop:pullback by conduch marked preserves colimit}. 
We then have a sequence of equivalences 
$$j^*f\sim j^* \colim_{i:\Sp_b} i^*f \sim \colim_{i:\Sp_b} (ji)^*f\sim \colim_{k:\Sp_a} k^*f .$$
\end{proof}

\begin{lemma}
\label{lemma:technical lemma exponentiability}
Let $i:c\to d$ be in $\F_g$, $b$ a globular sum, and $f:d\to b^\sharp$ any morphism. 
Then, there exists a commutative square
\[\begin{tikzcd}
	{c'} & {d'} & {b^\sharp} \\
	c & d
	\arrow["h", from=2-1, to=1-1]
	\arrow["g", from=2-2, to=1-2]
	\arrow["{i'}", from=1-1, to=1-2]
	\arrow["i"', from=2-1, to=2-2]
	\arrow[from=1-2, to=1-3]
	\arrow["f"', from=2-2, to=1-3]
\end{tikzcd}\]
\begin{enumerate}
\item $d\to d'$ is a finite composition of pushouts of morphism of shape $i_n^\alpha:\Db_n\to (\Db_{n+1})_t$ with $n$ an integer and $\alpha:=+$ if $n$ is even, and $-$ if not.
\item $d'\to b^\sharp$ is globular.
\item $h\to g$ is a right Gray deformation retract.
\end{enumerate}
\end{lemma}
\begin{proof}
We obtain $(d')^\natural$ by factorizing $f^\natural$ into an algebraic morphism $g^\natural$ followed by a globular morphism. The marking $d'$ is the smaller one that makes $g$ a morphism of marked $\zo$-categories. By construction, $c\to d$ fits in a cocartesian square
\[\begin{tikzcd}
	{\Db_n^\flat} & c \\
	{ (\Db_{n+1})_t} & d
	\arrow[from=1-2, to=2-2]
	\arrow["{i^\alpha_n}"', from=1-1, to=2-1]
	\arrow[from=2-1, to=2-2]
	\arrow[from=1-1, to=1-2]
	\arrow["\lrcorner"{anchor=center, pos=0.125, rotate=180}, draw=none, from=2-2, to=1-1]
\end{tikzcd}\]
where all morphisms are globular, and where $\alpha$ is $+$ if $n$ is even, and $-$ if not. As the procedure is similar for any $n$, we will suppose that $n=0$, and $d$ is then equivalent to $[1]^\sharp\vee[a,1]$ for $a\in t\Theta$. The fact that $g$ is algebraic implies that there exists a marked globular sum $c'$ and an integer $k$, such that $d'$ is of shape $[k]^\sharp\vee c'$ and such that $gi$ factors through $c'$. These data verify the desired condition.
\end{proof}

\begin{prop}
\label{prop:criterion to be left cartesian firbation}
Let $p:X\to b^\sharp$ be a morphism exponentiable in $b$. Consider also the following shape of diagram
\begin{equation}
\label{eq:prop:criterion to be left cartesian firbation}
\begin{tikzcd}
	{X''} & {X'} & X \\
	C & {C'} & {b^\sharp}
	\arrow["p"', from=1-3, to=2-3]
	\arrow["i"', from=2-1, to=2-2]
	\arrow["j"', from=2-2, to=2-3]
	\arrow["{p'}"', from=1-2, to=2-2]
	\arrow["{p''}"', from=1-1, to=2-1]
	\arrow[from=1-2, to=1-3]
	\arrow[from=1-1, to=1-2]
	\arrow["\lrcorner"{anchor=center, pos=0.125}, draw=none, from=1-1, to=2-2]
	\arrow["\lrcorner"{anchor=center, pos=0.125}, draw=none, from=1-2, to=2-3]
\end{tikzcd}
\end{equation}
The following are equivalent.
\begin{enumerate}
\item For any globular morphism $i:[a,1]^\sharp\to b^\sharp$, $i^*p$ is a left cartesian fibration.
\item For any diagram of shape \eqref{eq:prop:criterion to be left cartesian firbation}, if $i$ is  $i_n^\alpha:\Db_n\to (\Db_{n+1})_t$ with $n$ an integer and $\alpha:=+$ if $n$ is even and $-$ if not, and $j$ is globular, then $p''\to p'$ is a right Gray deformation retract.
\item For any diagram of shape \eqref{eq:prop:criterion to be left cartesian firbation}, if $i$ is a finite composition of pushouts of morphism of shape $i_n^\alpha:\Db_n\to (\Db_{n+1})_t$ with $n$ an integer and $\alpha:=+$ if $n$ is even and $-$ if not, and $j$ is globular, then $p''\to p'$ is a right Gray deformation retract.
\item For any diagram of shape \eqref{eq:prop:criterion to be left cartesian firbation}, if $i$ is in $ \F_g$, then $p''\to p'$ is a right Gray deformation retract.
\item The morphism $p$ is a left cartesian fibration.
\end{enumerate}
\end{prop}
\begin{proof}
The implication $(1)\Rightarrow (2)$
comes from theorem \ref{theo:other characterisation of left caresian fibration} as morphisms of shape $ i_n^\alpha$ are right Gray deformation retracts according to proposition \ref{prop:when glob inclusion are left Gray deformation}, and as every globular morphism $\Db_{n+1}\to b$ factors through a globular morphism $[a,1]\to b$.

We suppose that the second condition is fulfilled. As left Gray deformation retracts are stable under composition according to proposition \ref{prop:stability by composition }, we can restrict to the case where $i':c\to d$ fits in a cocartesian square
\[\begin{tikzcd}
	{\Db_n^\flat} & c \\
	{ (\Db_{n+1})_t} & d
	\arrow[from=1-2, to=2-2]
	\arrow["{i^\alpha_n}"', from=1-1, to=2-1]
	\arrow[from=2-1, to=2-2]
	\arrow[from=1-1, to=1-2]
	\arrow["\lrcorner"{anchor=center, pos=0.125, rotate=180}, draw=none, from=2-2, to=1-1]
\end{tikzcd}\]
where all morphisms are globular, and where $\alpha$ is $+$ if $n$ is even, and $-$ if not.
Let $p_0$ and $p_1$ be the morphism fitting in cocartesian squares
\[\begin{tikzcd}
	{X_0} & {X_1} & X \\
	{\Db_{n}^\flat} & {(\Db_{n+1})_t} & { b^\sharp}
	\arrow["p"', from=1-3, to=2-3]
	\arrow[from=2-2, to=2-3]
	\arrow["{p_1}"', from=1-2, to=2-2]
	\arrow[from=1-2, to=1-3]
	\arrow["\lrcorner"{anchor=center, pos=0.125}, draw=none, from=1-2, to=2-3]
	\arrow["{ i_n^\alpha}"', from=2-1, to=2-2]
	\arrow["{p_0}"', from=1-1, to=2-1]
	\arrow[from=1-1, to=1-2]
\end{tikzcd}\]
This defines a diagram in the $(\infty,1)$-category of arrows of $\ocatm$:
\[\begin{tikzcd}
	{p_0} & {p_0} & {p''} \\
	{p_1} & {p_0} & {p''}
	\arrow[from=2-2, to=2-1]
	\arrow[from=1-1, to=2-1]
	\arrow[from=1-2, to=1-1]
	\arrow[from=1-2, to=1-3]
	\arrow[from=1-2, to=2-2]
	\arrow[from=2-2, to=2-3]
	\arrow[from=1-3, to=2-3]
\end{tikzcd}\]
As the proposition \ref{prop:exponantiable stable under pullback} implies that $p'$ is $d$-exponentiable, the morphism $p''\to p'$ is the horizontal colimit of the previous diagram.  According to proposition \ref{prop:left Gray transfomration stable under pullback along cartesian fibration}, $p_0\to p_1$ is a left Gray deformation retract, and proposition \ref{prop:left Gray deformation retract stable under pushout} implies that $p''\to p'$ also is a left Gray deformation retract. This proves $(2)\Rightarrow (3)$.

Suppose now that condition $(3)$ is fulfilled and let $i$ be in $\F_g$. Consider the diagram
\[\begin{tikzcd}
	{c'} & {d'} & {b^\sharp} \\
	c & d
	\arrow["h", from=2-1, to=1-1]
	\arrow["g", from=2-2, to=1-2]
	\arrow["{i'}", from=1-1, to=1-2]
	\arrow["i"', from=2-1, to=2-2]
	\arrow[from=1-2, to=1-3]
	\arrow["f"', from=2-2, to=1-3]
\end{tikzcd}\]
induced by lemma \ref{lemma:technical lemma exponentiability}. 
We denote by $\tilde{p}''$ and $\tilde{p}'$ the morphisms fitting in the following cartesian squares.
\[\begin{tikzcd}
	{\tilde{X}''} & {\tilde{X}'} & X \\
	{ c'} & { d'} & { b^\sharp}
	\arrow["{ i'}"', from=2-1, to=2-2]
	\arrow[from=2-2, to=2-3]
	\arrow["p"', from=1-3, to=2-3]
	\arrow["{\tilde{p}'}"', from=1-2, to=2-2]
	\arrow["{\tilde{p}''}"', from=1-1, to=2-1]
	\arrow[from=1-1, to=1-2]
	\arrow[from=1-2, to=1-3]
	\arrow["\lrcorner"{anchor=center, pos=0.125}, draw=none, from=1-2, to=2-3]
	\arrow["\lrcorner"{anchor=center, pos=0.125}, draw=none, from=1-1, to=2-2]
\end{tikzcd}\]
As $p$ fulfills $(3)$, $\tilde{p}''\to \tilde{p}'$ is a right Gray deformation retract. By construction, 
the square $h\to g$ also is a right Gray deformation retract. 
As $p''$ and $p'$ are respectively the pullback of $\tilde{p}''$ along $h$ and the pullback of $\tilde{p}'$ along $g$, the dual version of \ref{prop:stability under pullback} implies that $p''\to p'$ is a right Gray deformation retract.

The implication $(4)\Rightarrow (5)$ is induced by theorem \ref{theo:other characterisation of left caresian fibration}. Eventually, the implication $(5)\Rightarrow (1)$ is a consequence of the preservation of left cartesian fibration under pullback.
\end{proof}

\begin{cor}
\label{cor:fibration over representable are expenitalbe}
A fibration $p$ over $a^\sharp$ is $a$-exponentiable. 
\end{cor}
\begin{proof}
We define $q:=\colim_{i:\Sp_a^\sharp}i^*p$. This morphism comes with a canonical comparison $q\to p$. According to proposition \ref{prop:how to create exponentiable}, $q$ is $a$-exponentiable.  For any globular morphism $j:[b,1]^\sharp\to a$, we have $j^*q\sim j^*p$ as $j$ is a discrete Conduché functor. In particular, $j^*q$ is a left cartesian fibration and  $q$ then verifies the first condition of proposition \ref{prop:criterion to be left cartesian firbation}. This implies that $q$ is a left cartesian fibration.

As all morphisms $j:1\to a^\sharp$ are marked globular, and so are discrete Conduché functors, 
there are equivalences
$$j^*\colim_{i:\Sp_a^\sharp}i^*p\sim j^*p$$
and the morphism $q\to p$ induces an equivalence on fiber. This morphisms is then an equivalence according to corollary \ref{cor:morphism between is an equivalence when equivalence on fiber}.
\end{proof}

\begin{lemma}
\label{lemma:pulback of Wsat}
Let $f:A\to B^\sharp$ be a left cartesian fibration, $n$ an integer, and consider a diagram of $\tiPsh{\Theta}$ of shape
\[\begin{tikzcd}
	{A''} & {A'} & A \\
	{(\Sigma^nE^{eq})^\flat} & {\Db_n^\flat} & {B^\sharp}
	\arrow["f", from=1-3, to=2-3]
	\arrow["{f'}", from=1-2, to=2-2]
	\arrow["{f''}", from=1-1, to=2-1]
	\arrow["i"', from=2-1, to=2-2]
	\arrow["j", from=1-1, to=1-2]
	\arrow[from=1-2, to=1-3]
	\arrow[from=2-2, to=2-3]
	\arrow["\lrcorner"{anchor=center, pos=0.125}, draw=none, from=1-2, to=2-3]
	\arrow["\lrcorner"{anchor=center, pos=0.125}, draw=none, from=1-1, to=2-2]
\end{tikzcd}\]
Then $j$ is in $\widehat{\Wm}$.
\end{lemma}
\begin{proof}
As $f'$ and $f''$ are left cartesian fibrations, the only marked cell in $A'$ and $A''$ are the identities according to proposition \ref{prop:left fib over flat}. We can then suppose that the left square lies in $\ocat$, and then apply proposition \ref{prop:pulback of Wsat}.
\end{proof}

\begin{lemma}
\label{lemma:pullback along markkin}
Let $b$ be a globular sum, and $n$ an integer. For any cartesian squares in $\iPsh{\Theta}$,
\[\begin{tikzcd}
	{A''} & {A'} & {b^{\sharp_n}} \\
	{B''} & {B'} & {b^{\sharp}}
	\arrow[from=1-3, to=2-3]
	\arrow["{ }", from=1-2, to=2-2]
	\arrow[from=1-1, to=2-1]
	\arrow["i"', from=2-1, to=2-2]
	\arrow["j", from=1-1, to=1-2]
	\arrow[from=1-2, to=1-3]
	\arrow[from=2-2, to=2-3]
	\arrow["\lrcorner"{anchor=center, pos=0.125}, draw=none, from=1-2, to=2-3]
	\arrow["\lrcorner"{anchor=center, pos=0.125}, draw=none, from=1-1, to=2-2]
\end{tikzcd}\]
if $i$ is in $\widehat{\Wm}$, so is $j$.
\end{lemma}
\begin{proof}
As $\tiPsh{\Theta}$ is cartesian closed, one can suppose that $i$ is in $\W$. In this case the diagram can be seen as a diagram in $\Psh{\Theta}$. The proof is an easy verification of all the possible cases.
\end{proof}

\begin{prop}
\label{prop:W stable under pullback}
For any cartesian square of $\tiPsh{\Theta}$,
\[\begin{tikzcd}
	{A''} & {A'} & A \\
	{B''} & {B'} & {B^\sharp}
	\arrow["f", from=1-3, to=2-3]
	\arrow[from=1-2, to=2-2]
	\arrow[from=1-1, to=2-1]
	\arrow["i"', from=2-1, to=2-2]
	\arrow["j", from=1-1, to=1-2]
	\arrow[from=1-2, to=1-3]
	\arrow[from=2-2, to=2-3]
	\arrow["\lrcorner"{anchor=center, pos=0.125}, draw=none, from=1-2, to=2-3]
	\arrow["\lrcorner"{anchor=center, pos=0.125}, draw=none, from=1-1, to=2-2]
\end{tikzcd}\]
where $f$ is a left cartesian fibration, if $i$ is in $\widehat{\Wm}$, so is $j$.
\end{prop}
\begin{proof}
As $\tiPsh{\Theta}$ is cartesian closed, one can suppose that $i$ is in $\W$. Several cases have to be considered.
If $i$ is of shape $(\Sigma^nE^{eq})^\flat\to \Db_n^\flat$, this is lemma \ref{lemma:pulback of Wsat}. 
Suppose now that $i$ is of shape $\Sp_b^{\sharp_n}\to b^{\sharp_n}$. This induces a diagram
\[\begin{tikzcd}
	{A''} && {A'} && A \\
	& {A''''} && {A'''} \\
	{\Sp_b^{\sharp_n}} && {b^{\sharp_n}} && {B^\sharp} \\
	& {\Sp_b^{\sharp}} && {b^{\sharp}}
	\arrow[""{name=0, anchor=center, inner sep=0}, "f", from=1-5, to=3-5]
	\arrow[from=1-3, to=3-3]
	\arrow[from=1-1, to=3-1]
	\arrow["i"'{pos=0.6}, from=3-1, to=3-3]
	\arrow["j", from=1-1, to=1-3]
	\arrow[from=1-3, to=1-5]
	\arrow[from=3-3, to=3-5]
	\arrow["\lrcorner"{anchor=center, pos=0.125}, draw=none, from=1-1, to=3-3]
	\arrow[from=3-1, to=4-2]
	\arrow[from=3-3, to=4-4]
	\arrow["{i'}"', from=4-2, to=4-4]
	\arrow[from=4-4, to=3-5]
	\arrow[from=2-2, to=4-2]
	\arrow["{j'}"{pos=0.4}, from=2-2, to=2-4]
	\arrow[from=2-4, to=4-4]
	\arrow[from=2-4, to=1-5]
	\arrow[from=1-1, to=2-2]
	\arrow[from=1-3, to=2-4]
	\arrow["\lrcorner"{anchor=center, pos=0.125}, draw=none, from=2-4, to=3-5]
	\arrow["\lrcorner"{anchor=center, pos=0.125}, draw=none, from=1-3, to=0]
\end{tikzcd}\]
where all squares are cartesian. Corollary \ref{cor:fibration over representable are expenitalbe} implies that $j'$ is in $\widehat{\W}$, and according to lemma \ref{lemma:pullback along markkin}, so is $j$.
\end{proof}

 A left cartesian fibration $A\to B$ is \wcnotion{classified}{classified left cartesian fibration} if there exists a cocartesian square: 
\[\begin{tikzcd}
	A & {A'} \\
	B & {B^\sharp}
	\arrow[from=1-1, to=1-2]
	\arrow[from=1-2, to=2-2]
	\arrow[from=1-1, to=2-1]
	\arrow[from=2-1, to=2-2]
	\arrow["\lrcorner"{anchor=center, pos=0.125}, draw=none, from=1-1, to=2-2]
\end{tikzcd}\]

\begin{theorem}
\label{theo:pullback along un marked cartesian fibration}
Let $p:A\to B$ be a classified left cartesian fibration. The functor $p^*:\ocatm_{/B}\to \ocatm_{/A}$ preserves colimits. 
\end{theorem}
\begin{proof}
As $\tPsh{\Theta}$ is locally cartesian closed, it is enough to show that the functor $p^*:\tiPsh{\Theta}_{/B}\to \tiPsh{\Delta[\Theta]}_{/A}$ sends $\Wm$ onto $\widehat{\Wm}$.
As morphisms fulfilling this property are stable under pullback, one can suppose that $p$ is of shape $B\to A^\sharp$, then applies proposition \ref{prop:W stable under pullback}.
\end{proof}

\begin{cor}
\label{cor:fib over a colimit}
Let $B$ be the colimit of a diagram $F:I\to \ocat$, and
 $p:X\to \colim_i B_i$ a left cartesian fibration. The canonical morphism
 $$ \colim_{i:B_i\to B}i^*p\to p$$
 is an equivalence.
\end{cor}
\begin{proof}
This morphism corresponds to the square
\[\begin{tikzcd}
	{\colim_{i:I}p^*B_i} & X \\
	{\colim_{i:I}B_i} & {B^\sharp}
	\arrow[from=2-1, to=2-2]
	\arrow[from=1-1, to=2-1]
	\arrow[from=1-1, to=1-2]
	\arrow["p", from=1-2, to=2-2]
\end{tikzcd}\]
The lower horizontal morphism is an equivalence by hypothesis, and the upper one is an equivalence as $p^*$ preserves colimits.
\end{proof}

\subsection{Colimits of cartesian fibrations}
\label{section:Colimit of left cartesian fibrations}
Through this section, we will identify any marked $\io$-category $C$ with the canonical induced morphism $C\to1$. If $f:X\to Y$ is a morphism, $f\times C$ then corresponds to the canonical morphism $X\times C\to Y$.

\begin{lemma}
\label{lemma: colimit of fib over b}
Let $b$ be a globular sum and $F:I\to \ocatm_{/b^\sharp}$ be a diagram that is pointwise a left cartesian fibration. The induced morphism 
$\colim_IF$ is a left cartesian fibration over $b^\sharp$. 
\end{lemma}
\begin{proof}
We denote $G:I\to \ocatm$ the diagram induced by $F$ by taking the domain.
Remark first that proposition \ref{prop:exponantiable stable under colim} implies that $\colim_IF$ is $b$-exponentiable. 
Let $n$ be an integer. Suppose given cartesian squares
\[\begin{tikzcd}
	{Y'} & Y & {\colim_IX} \\
	{\Db_n^\flat} & {(\Db_{n+1})_t} & {b^\sharp}
	\arrow["{\colim_IF}", from=1-3, to=2-3]
	\arrow["{i_n^\alpha}"', from=2-1, to=2-2]
	\arrow[from=1-2, to=2-2]
	\arrow[from=1-2, to=1-3]
	\arrow["f", from=1-1, to=1-2]
	\arrow[from=1-1, to=2-1]
	\arrow["j"', from=2-2, to=2-3]
	\arrow["\lrcorner"{anchor=center, pos=0.125}, draw=none, from=1-1, to=2-2]
	\arrow["\lrcorner"{anchor=center, pos=0.125}, draw=none, from=1-2, to=2-3]
\end{tikzcd}\]
where $\alpha$ is $+$ is $n$ is even and $-$ if not and with $j$ globular. According to proposition \ref{prop:criterion to be left cartesian firbation}, we have to show that $f$ is a right Gray deformation retract to conclude. As $F$ is pointwise a left cartesian fibration, proposition \ref{prop:left Gray transfomration stable under pullback along cartesian fibration} implies that for any $i:I$, the morphism $f(i)$ appearing in the cartesian squares:
\[\begin{tikzcd}
	{Y'} & Y & {X(i)} \\
	{\Db_n^\flat} & {(\Db_{n+1})_t} & {b^\sharp}
	\arrow["{F(i)}", from=1-3, to=2-3]
	\arrow["{i_n^\alpha}"', from=2-1, to=2-2]
	\arrow[from=1-2, to=2-2]
	\arrow[from=1-2, to=1-3]
	\arrow["{f(i)}", from=1-1, to=1-2]
	\arrow[from=1-1, to=2-1]
	\arrow["j"', from=2-2, to=2-3]
	\arrow["\lrcorner"{anchor=center, pos=0.125}, draw=none, from=1-1, to=2-2]
	\arrow["\lrcorner"{anchor=center, pos=0.125}, draw=none, from=1-2, to=2-3]
\end{tikzcd}\]
is a right Gray deformation retract, and that the corresponding Gray deformation retract structure is functorial in $i:I$.
 As $j$ and $ji_n^\alpha$ are marked globular, they are discrete Conduché functors, and so exponentiable according to proposition \ref{prop:pullback by conduch marked preserves colimit}. The following canonical morphism
 $$\colim_I f(i)\to f$$
 is then an equivalence. As right Gray deformation retract structures are stable by colimits, this concludes the proof.
\end{proof}

\begin{lemma}
\label{lemma: colimit of fib over b2}
Let $A$ be an $\io$-category and $F:I\to \ocatm_{/A^\sharp}$ be a diagram that is pointwise a left cartesian fibration. Let $i:a^\sharp\to b^\sharp$ be a morphism between globular sums and $i:b^\sharp\to A^\sharp$ any morphism.
The canonical comparison $$\colim_I (ji)^*F\to i^*\colim_I j^*F$$
is an equivalence.
\end{lemma}
\begin{proof}
Lemma \ref{lemma: colimit of fib over b} implies that the two morphisms are left cartesian fibrations. As equivalences between these morphisms are detected on fibers, we can suppose that $a$ is $[0]$. In this case, the morphism $i$ is a discrete Conduché functor, and is then exponentiable according to proposition \ref{prop:pullback by conduch marked preserves colimit}. This directly concludes the proof.
\end{proof}

\begin{theorem}
\label{theo:left cart stable by colimit}
Let $A$ be an $\io$-category and $F:I\to \ocatm_{/A^\sharp}$ be a diagram that is pointwise a left cartesian fibration. The induced morphism 
$\colim_IF$ is a left cartesian fibration over $A^\sharp$.
\end{theorem}
\begin{proof}
Consider the functor $\psi:\Theta_{/A}\to \Arr(\ocatm)$ whose value on $j:b\to A$ is $\colim_I j^*F$.
As $F$ is pointwise a left cartesian fibration, the corollary \ref{cor:fib over a colimit} induces equivalences
$$\colim_{\Theta_{/A}}\psi:= \colim_{j:b\to A}\colim_I j^*F\sim \colim_I \colim_{j:b\to A}j^*F\sim \colim_I F$$

 The functor $\psi$  is cartesian according to lemma \ref{lemma: colimit of fib over b2}, and as $\codom \psi$ as a special colimit (given by $A^\sharp$), so has $\psi$ according to proposition \ref{prop:special colimit marked case}. In particular, this implies that for any $j:b\to A$, the following canonical morphism
$$\colim_I j^* F=: \psi(j)\to j^*\colim_{\Theta_{/A}}\psi\sim j^* \colim_I F$$
is an equivalence. As the left object is a left cartesian fibration according to lemma \ref{lemma: colimit of fib over b}, so is the right one.
As this is true for any $j:b\to A$, the corollary \ref{cor:on the fact that fib are define against representable} implies that $ \colim_I F$ is a left cartesian fibration.
\end{proof}

\begin{cor}
\label{cor:inclusion of lcatt into the slice preserves colimits}
Let $A$ be an $\io$-category. The inclusion $\LCart(A^\sharp)\to \ocatm_{/A^\sharp}$ preserves both colimits and limits.
\end{cor}
\begin{proof}
The preservation of limits is a consequence of the fact that that this inclusion is a right adjoint. The preservation of colimits is a direct consequence of the theorem \ref{theo:left cart stable by colimit}.
\end{proof}

\p We now use the last theorem to provide an alternative explicit expression of the left cartesian fibration $\Fb h^0_{[C,1]}$. We obtain this in the theorem \ref{theo:equivalence betwen slice and join}.

\begin{prop}
\label{prop:appendice version equivalence betwen slice and join strict word}
Let $C$ be an $\zo$-category with an atomic and loop free basis. The canonical projection $\gamma:1\costar C^\flat \to [C,1]^\sharp$ is a left cartesian fibration.
\end{prop}
\begin{proof}
Let $C$ be such $\zo$-category.
The corollary \ref{cor:otimes et op}, the theorem \ref{theo:join preserves stict VMG version} and the proposition \ref{prop:suspension preserves stricte} imply that both the domain and the codomain of $\gamma$ are strict. We can then show the result in $\zocatm$.
By construction, the basis of $1\costar \lambda C$ is given by the graduated set: 
$$(B_{1\costar \lambda C})_n:=
\left\{
\begin{array}{ll}
\{\emptyset \costar c,c\in (B_{C})_0\}\cup \{\emptyset \costar c,c \in (B_C)_0\}&\mbox{if $n=0$}\\
\{1 \costar c,c\in (B_{C})_{n-1}\}\cup \{\emptyset \costar c,c\in (B_C)_n\} &\mbox{if $n>0$}\\
\end{array}\right.
$$
where $B_C$ is the basis of $C$. The derivative is induced by: 
$$\partial (1\costar c):= 1\costar \partial c + (-1)^{|c|}\emptyset\otimes c~~~~~~~~~~\partial(\emptyset\star c):= \emptyset\costar \partial c$$
where we set the convention $\partial c:=0$ if $|c|=0$.
Let $n$ be an integer and $x$ an element of $(1\costar \lambda C)_n$. The induced morphism $\Db_n\to 1\costar C^\flat$ is marked if and only if there is no element of shape $\emptyset\star c$ in the support of $x$.

For an integer $n>0$, we define $s_n: (\Sigma \lambda C)_n\to (1\costar \lambda C)_n$ as the unique group morphism fulfilling $$s_n(\Sigma c):= 1\costar c$$ for $c$ any element of $\lambda C_{n-1}$. Remark that for any non negative integer $n$, and any element $d$ of $(1\costar \lambda C)_n$, $s_n(d)$ is contained in $d$. However, the family of morphism $\{s_n\}_{n\in \Nb}$ does not commute with the derivative. Let $n$ be an integer and $x$ an element of $(1\costar \lambda C)_n$. The induced morphism $\Db_n\to 1\costar C^\flat$ is therefore marked if and only if $x$ is equal to $s_n\gamma_n(x)$.

Eventually,
we recall that $(\Db_n)_t\otimes[1]^\sharp$ is the colimit of the diagram:
\[\begin{tikzcd}
	{(\Db_n)_t\otimes\{0\}\coprod (\Db_n)_t\otimes\{1\}} & {\Db_n^\flat\otimes\{0\}\coprod \Db_n^\flat\otimes\{1\}} & {\tau^ i_ n(\Db_n^\flat\otimes[1]^\sharp)}
	\arrow[from=1-2, to=1-1]
	\arrow[from=1-2, to=1-3]
\end{tikzcd}\]
We then have to show that for any integer $n$, any diagram of shape 
\[\begin{tikzcd}
	{\lambda\Db_n\otimes\{0\}\cup \lambda\partial\Db_n\otimes[1]} & {1\costar \lambda C} \\
	{\lambda\Db_n\otimes[1]} & {\Sigma \lambda C}
	\arrow[from=1-1, to=2-1]
	\arrow["f"', from=2-1, to=2-2]
	\arrow["g", from=1-1, to=1-2]
	\arrow[from=1-2, to=2-2]
\end{tikzcd}\]
with $f(e_n\otimes[1])$ and $f(e^\alpha_k\otimes[1])$ for $\alpha\in\{-,+\}$ and $k<n$ correponding to a marked cell, admits a unique lifting $l$ with the following extra condition: if $n>0$, if $f(e_n\otimes[1])$ is null and if $g(e_n\otimes\{0\})$ corresponds to a marked cell, then $l(e_n\otimes[1])$ is null and $l(e_n\otimes\{1\})$ corresponds to a marked cell.

Suppose first that $n=0$. We set $l_0:\lambda (\Db_0\otimes[1])_0\to (1\costar \lambda C)_0$ as the unique group morphism extending $g_0$ and such that 
$$l_0(e_0\otimes\{1\}):= \partial s_1(f_1(e_0\otimes[1])+ g_0(e_0\otimes\{1\}).$$
We also define $l_1:\lambda (\Db_0\otimes[1])_1\to (1\costar \lambda C)_1$ as the group morphism characterized by: 
$$l_1(e_0\otimes[1]):= s_1(f_1(e_0\otimes[1])).$$
For $k>1$, we set $l_k:\lambda (\Db_0\otimes[1])_k\to (1\costar \lambda C)_k$ as the constant morphism on $0$.
We directly deduce the equality $\partial l= l \partial$.
We then have defined the desired lifting, which is obviously the unique one possible.

Suppose now that $n>0$. We set $l_k:=g_k:\lambda (\Db_n\otimes[1])_k\to (1\costar \lambda C)_k$ for $k<n$ and $l_n:\lambda (\Db_n\otimes[1])_n\to (1\costar \lambda C)_n$ as the unique group morphism extending $g_n$ and such that 
$$l_n(e_n\otimes\{1\}) := (-1)^\alpha \partial s_{n+1}( f(e_n\otimes[1])) - (-1)^\alpha s_{n}( f((\partial e_n)\otimes[1])) + g_n(e_n\otimes\{0\})$$
where $\alpha$ is $+$ if $n$ is even and $-$ if not.
We define $l_{n+1}:\lambda (\Db_n\otimes[1])_{n+1}\to (1\costar \lambda C)_{n+1}$ as the group morphism characterized by: 
$$l_{n+1}(e_n\otimes[1]):= s_{n+1}(f_{n+1}(e_n\otimes[1])).$$
Eventually, for $k>n$, we set $l_k:\lambda (\Db_n\otimes[1])_k\to (1\costar \lambda C)_k$ as the constant morphism on $0$.

For an integer $k<n$ and $\alpha\in\{-,+\}$, as the $(k+1)$-cell corresponding to $g_{k+1}(e_k^\alpha\otimes [1])$ is marked, we have an equality
$$g_{k+1}(e_k^\alpha\otimes [1]) = s_{k+1}f_{k+1}(e_k^\alpha\otimes [1]).$$
This then implies the equalities
$$\begin{array}{rcl}
\partial(l_{n+1}(e_n\otimes[1])) &=& l_{n+1}(\partial (e_n\otimes[1]))\\
\partial(l_{n}(e_n\otimes \{1\}))&=& g_{n-1}(\partial e_n\otimes \{1\})
\end{array}$$
As it was the only non trivial case, we have $l\partial = \partial l.$
We then have defined the desired lifting, which is obviously the unique one possible. Moreover, if we suppose that $f(e_n\otimes[1])$ is null and $g(e_n\otimes\{0\})$ corresponds to a marked cell, this implies that 
$ s_{n+1}( f(e_n\otimes[1])) =0$ and that the $g_n(e_n\otimes\{0\})$ is in the image of $s_n$. The object $f(e_n\otimes[1])$ also is in the image of $s_n$ and so corresponds to a marked cell.
\end{proof}

\begin{lemma}
\label{lemma:equivalence betwen slice and join strict word}
There is a unique morphism $1\costar C^\flat\to [C,1]^\sharp_{0/}$ fitting in a square
\[\begin{tikzcd}
	1 & {[C,1]^\sharp_{0/}} \\
	{1\costar C^\flat} & {[C,1]^\sharp}
	\arrow[from=1-1, to=2-1]
	\arrow[from=1-1, to=1-2]
	\arrow[from=1-2, to=2-2]
	\arrow[from=2-1, to=2-2]
	\arrow[dotted, from=2-1, to=1-2]
\end{tikzcd}\]
This morphism is an equivalence whenever $C$ is a globular sum.
\end{lemma}
\begin{proof}
We have by construction a cocartesian square
\[\begin{tikzcd}
	{C^\flat\otimes\{0\}} & {C^\flat\otimes[1]^\sharp} \\
	1 & {1\costar C^\flat}
	\arrow[from=1-1, to=2-1]
	\arrow[from=1-1, to=1-2]
	\arrow[from=2-1, to=2-2]
	\arrow[from=1-2, to=2-2]
	\arrow["\lrcorner"{anchor=center, pos=0.125, rotate=180}, draw=none, from=2-2, to=1-1]
\end{tikzcd}\]
which implies that $1\to 1\costar C^\flat$ is initial. This directly implies the first assertion.
We now prove the second assertion. We suppose that $C$ is a globular sum $a$. The $\io$-categories $1\costar a^\flat$ is strict according to proposition \ref{prop:tensor of glboer are strics}. Proposition \ref{prop:appendice version equivalence betwen slice and join strict word} states that the canonical morphism $1\costar a^\flat \to [a,1]^\sharp$ is a left cartesian fibration. As the comparison map is initial by left cancellation, this concludes the proof.	 
\end{proof}

\begin{prop}
\label{prop:equivalence betwen slice and join strict word2}
Let $b$ be a globular form and $j:b\to C$ a morphism between $\io$-categories. The following diagram is cartesian
\[\begin{tikzcd}
	{1\costar b^\flat\coprod_{b^\flat}C^\flat} & {[C,1]^\sharp_{0/}} \\
	{[b,1]^\sharp} & {[C,1]^\sharp}
	\arrow["{[j,1]^\sharp}"', from=2-1, to=2-2]
	\arrow[from=1-1, to=1-2]
	\arrow[from=1-1, to=2-1]
	\arrow[from=1-2, to=2-2]
\end{tikzcd}\]
\end{prop}
\begin{proof}
The lemma \ref{lemma:equivalence betwen slice and join strict word} implies that the morphism $1\costar b^\flat\to [b,1]^\sharp$ is equivalent to $\Fb h_0^{[b,1]}$.
We then have to check that the canonical morphism 
\begin{equation}
\label{eq:in a tecnical lemma}
\Fb h_0^{[b,1]}\coprod_{b^\flat}C^\flat\to [j,1]^*\Fb h_0^{[C,1]}
\end{equation}
is an equivalence. According to theorem \ref{theo:left cart stable by colimit}, the two objects are left cartesian fibrations, and we then have to check that this morphism induce equivalences on fibers. Remark furthermore that the two morphisms $\{0\}\to [b,1]^\sharp$ and $\{1\}\to [b,1]^\sharp$ are discrete Conduché functors and then exponentiable according to proposition \ref{prop:pullback by conduch marked preserves colimit}. The fibers on $0$ and $1$ of the morphism \eqref{eq:in a tecnical lemma} then corresponds to the equivalences
$$1\coprod_{\emptyset}\emptyset\sim 1~~~~\mbox{ and }~~~~ b\coprod_bC\sim C.$$
\end{proof}

\begin{theorem}
\label{theo:equivalence betwen slice and join}
Let $C$ be a $\io$-category. The left cartesian fibration $\Fb h^0_{[C,1]}$ is equivalent to the projection 
$1\costar C^\flat\to [C,1]^\sharp$.
\end{theorem}
\begin{proof}
Let $i:[b,1]^\sharp\to [C,1]^\sharp$ be any morphism. The proposition
\ref{prop:equivalence betwen slice and join strict word2} states that the following square is cartesian:
\[\begin{tikzcd}
	{1\costar b^\flat\coprod_{b^\flat}C^\flat} & {[C,1]^\sharp_{0/}} \\
	{[b,1]^\sharp} & {[C,1]^\sharp}
	\arrow[from=2-1, to=2-2]
	\arrow[from=1-1, to=2-1]
	\arrow[from=1-1, to=1-2]
	\arrow[from=1-2, to=2-2]
\end{tikzcd}\]
Eventually, remark that we have an equivalence 
$$\colim_{b\to C}[b,1]\sim [C,1].$$
The theorem \ref{theo:pullback along un marked cartesian fibration} then induces equivalences
$$[C,1]^\sharp_{0/}\sim \colim_{i:b\to C}1\costar b^\flat\coprod_{b^\flat}C^\flat \sim 1\costar C^\flat\coprod_{C^\flat}C^\flat\sim 1\costar C^\flat$$
over $[C,1]^\sharp$. This concludes the proof.
\end{proof}

\begin{cor}
\label{cor:cor of the past10}
Let $b$ be a globular form and $j:b\to C$ any morphism. The following square is cartesian:
\[\begin{tikzcd}
	{1\costar b\coprod_b C} & {1\costar C} \\
	{[b,1]} & {[C,1]}
	\arrow[from=1-2, to=2-2]
	\arrow[from=1-1, to=2-1]
	\arrow[from=2-1, to=2-2]
	\arrow[from=1-1, to=1-2]
\end{tikzcd}\]
\end{cor}
\begin{proof}
We apply the functor $(\uvar)^\natural$ to the cartesian square given in proposition \ref{prop:equivalence betwen slice and join strict word2} and the equivalence given in theorem \ref{theo:equivalence betwen slice and join}.
\end{proof}

\begin{cor}
\label{cor:cor of the past3}
Let $C$ be an $\io$-category. We denote by $\gamma:C\star 1\to [C,1]$ and $\gamma':1\costar C\to [C,1]$ the two canonical projections. The functors $\gamma^*:\ocat_{/[C,1]}\to \ocat_{/C\star 1}$ and $\gamma^*:\ocat_{/[C,1]}\to \ocat_{/1\costar C}$ preserve colimits. 
\end{cor}
\begin{proof}
We have a cocartesian square
\[\begin{tikzcd}
	{(1\costar C)^\flat} & {1\costar C^\flat} \\
	{[C,1]^\flat} & {[C,1]^\sharp}
	\arrow["{\gamma^\flat}"', from=1-1, to=2-1]
	\arrow[from=1-2, to=2-2]
	\arrow[from=2-1, to=2-2]
	\arrow[from=1-1, to=1-2]
	\arrow["\lrcorner"{anchor=center, pos=0.125}, draw=none, from=1-1, to=2-2]
\end{tikzcd}\]
The theorem \ref{theo:equivalence betwen slice and join} implies that the right hand morphism is a left cartesian fibration, and $\gamma^\flat$ is then a classified left cartesian fibration. The result is then a direct consequence of theorem \ref{theo:pullback along un marked cartesian fibration}.
The other assertion follows by duality.
\end{proof}

\subsection{Smooth and proper morphisms}
\p 
For a marked $\io$-category $C$, we denote by \textit{$\LCart(C)$} \sym{(lcart@$\LCart(\uvar)$}\sym{(rcart@$\RCart(\uvar)$} (resp. $\RCart(C)$) the full sub $\iun$-category of $\ocatm_{/C}$ whose objects are left cartesian fibrations. We can equivalently define $\LCart(C)$ as the localization of $\ocatm_{/C}$ along $\widehat{\I_{/C}}$. For $E$, $F$ two objects of $\LCart(C)$ corresponding respectively to two left cartesian fibrations
$p:X\to C$ and $q:X\to C$, we denote by \wcnotation{$\Map(E,F)$}{(map@$\Map(\uvar,\uvar)$} the $\io$-category fitting in the cocartesian square:
\[\begin{tikzcd}
	{\Map(E,F)} & {\uHom(X,Y)} \\
	{\{p\}} & {\uHom(X,C)}
	\arrow["{q_!}", from=1-2, to=2-2]
	\arrow[from=2-1, to=2-2]
	\arrow[from=1-1, to=1-2]
	\arrow[from=1-1, to=2-1]
	\arrow["\lrcorner"{anchor=center, pos=0.125}, draw=none, from=1-1, to=2-2]
\end{tikzcd}\]

\p We recall that a left cartesian fibration $X\to C$ is \textit{classified} when there exists a cartesian square: 
\[\begin{tikzcd}
	X & {X'} \\
	C & {C^\sharp}
	\arrow[from=1-1, to=1-2]
	\arrow[from=1-2, to=2-2]
	\arrow[from=1-1, to=2-1]
	\arrow[from=2-1, to=2-2]
	\arrow["\lrcorner"{anchor=center, pos=0.125}, draw=none, from=1-1, to=2-2]
\end{tikzcd}\]
We denote by \wcnotation{$\LCartc(C)$}{(lcart@$\LCartc(\uvar)$} the full sub $\iun$-category of $\LCart(C)$ whose objects are classified left cartesian fibrations.

\p Remark that every morphism $f:C\to D$ induces an adjunction
\[\begin{tikzcd}
	{f_!:\ocat_{/C}} & {\ocat_{/D}:f^*}
	\arrow[shift left=2, from=1-1, to=1-2]
	\arrow[shift left=2, from=1-2, to=1-1]
\end{tikzcd}\]
where the left adjoint $f_!$ is the composition and the right one is the pullback.
This induces an adjunction at the level of localized $\iun$-category:
\[\begin{tikzcd}
	{\Lb f_!:\LCart(C)} & {\LCart(D):\Rb f^*=f^*}
	\arrow[shift left=2, from=1-1, to=1-2]
	\arrow[shift left=2, from=1-2, to=1-1]
\end{tikzcd}\]
 
\p A morphism $f:C\to D$ is \wcnotion{smooth}{smooth morphism} if $f^*:\ocatm_{/D}\to \ocatm_{/C}$ preserves colimits, and for every cartesian square of the form
\begin{equation}
\label{eq:smooth diagram}
\begin{tikzcd}
	{C''} & {C'} & C \\
	{D''} & {D'} & D
	\arrow["{v'}", from=1-1, to=1-2]
	\arrow[from=1-2, to=1-3]
	\arrow["f", from=1-3, to=2-3]
	\arrow[from=2-2, to=2-3]
	\arrow["v"', from=2-1, to=2-2]
	\arrow[from=1-2, to=2-2]
	\arrow[from=1-1, to=2-1]
	\arrow["\lrcorner"{anchor=center, pos=0.125}, draw=none, from=1-1, to=2-2]
	\arrow["\lrcorner"{anchor=center, pos=0.125}, draw=none, from=1-2, to=2-3]
\end{tikzcd}
\end{equation}
if $v$ is inital, so is $v'$.
When $f$ is smooth, the functor $f^*$ admits a left adjoint
\[\begin{tikzcd}
	{f^*:\ocatm_{/D}} & {\ocatm_{/C}:f_*}
	\arrow[""{name=0, anchor=center, inner sep=0}, shift left=2, from=1-2, to=1-1]
	\arrow[""{name=1, anchor=center, inner sep=0}, shift left=2, from=1-1, to=1-2]
	\arrow["\dashv"{anchor=center, rotate=-90}, draw=none, from=1, to=0]
\end{tikzcd}\]
and as $f^*$ preserves initial morphisms, this induces a derived adjunction:
\[\begin{tikzcd}
	{\Lb f^*:\LCart(D)} & {\LCart(C):\Rb f_*}
	\arrow[""{name=0, anchor=center, inner sep=0}, shift left=2, from=1-2, to=1-1]
	\arrow[""{name=1, anchor=center, inner sep=0}, shift left=2, from=1-1, to=1-2]
	\arrow["\dashv"{anchor=center, rotate=-90}, draw=none, from=1, to=0]
\end{tikzcd}\]
where $\Rb f_*$ is just the restriction of $f_*$.

\begin{prop}
\label{prop:projection are smooth}
Let $I, J$ be two marked $\io$-categories. The projection $I\times J\to I$ is smooth. 
\end{prop}
\begin{proof}
This is a direct consequence of the fact that cartesian product preserves colimits and initial morphisms.
\end{proof}
\begin{prop}
\label{prop:left cartesian fibration are smooth}
Classified right cartesian fibrations are smooth.
\end{prop}
\begin{proof}
The theorem \ref{theo:pullback along un marked cartesian fibration} states that $f^*$ preserves colimits. Suppose given a diagram of shape \eqref{eq:smooth diagram}. As initial morphisms are the smallest cocomplete class containing morphism $I$, and as $f^*$ preserves colimits, one can suppose that $v$ belongs to $I$, and then is a left Gray deformation retract. To conclude, one applies proposition
\ref{prop:left Gray transfomration stable under pullback along cartesian fibration}.
\end{proof}

\p A morphism $f:C\to D$ is \wcnotion{proper}{proper morphism} if $f^*:\ocatm_{/D}\to \ocatm_{/C}$ preserves colimits and for every cartesian square of the form
\begin{equation}
\label{eq:proper diagram}
\begin{tikzcd}
	{C''} & {C'} & C \\
	{D''} & {D'} & D
	\arrow["{v'}", from=1-1, to=1-2]
	\arrow[from=1-2, to=1-3]
	\arrow["f", from=1-3, to=2-3]
	\arrow[from=2-2, to=2-3]
	\arrow["v"', from=2-1, to=2-2]
	\arrow[from=1-2, to=2-2]
	\arrow[from=1-1, to=2-1]
	\arrow["\lrcorner"{anchor=center, pos=0.125}, draw=none, from=1-1, to=2-2]
	\arrow["\lrcorner"{anchor=center, pos=0.125}, draw=none, from=1-2, to=2-3]
\end{tikzcd}
\end{equation}
if $v$ is final, so is $v'$.
A morphism $f$ is then proper if and only if $f^{\circ}$ is smooth. Propositions \ref{prop:projection are smooth} and \ref{prop:left cartesian fibration are smooth} then imply that projections and classified right cartesian fibrations are proper.

\p We denote by $\bot:\ocatm\to \ocat$ the left Kan extension of the functor $t\Theta\to \ocat$ that sends $a^\flat$ on $a$ and $(\Db_{n+1})_t$ on $\Db_n$. Roughly speaking, $\bot$ sends a marked $\io$-category to it's localization by marked cells. By abuse of notation, we also denote\sym{((g3@$\bot$} $\bot: 
\Arr(\ocatm)\to \ocat$, the composite functor 
$$\Arr(\ocatm)\xrightarrow{\dom}\ocatm\xrightarrow{\bot} \ocat$$
This functor preserves colimits and sends initial and final morphisms to equivalences. For any object $E$ of $\LCart(A)$ and for any morphism $i:A\to B$, we then have a canonical equivalence 
\begin{equation}
\label{eq:bot kill pull}
\bot \Lb i_! E\sim \bot E.
\end{equation}

Let $A$ be an $\io$-category and $a:1\to A^\sharp$ an object of $A$. 
According to proposition \ref{prop:explicit factoryzation}, the factorisation of $a:1\to A^\sharp$ in a final morphism followed by a right cartesian fibration is given by the canonical inclusion $\{a\}\to A^\sharp_{a/}$ and the canonical projection $\pi_a:A^\sharp_{a/}\to A^\sharp$.
Let $E$ be an object of $\LCart(A^\sharp)$ corresponding to a left cartesian fibration $p:X\to A^\sharp$.
We then have a diagram
\[\begin{tikzcd}
	{X_a} & {X_{/a}} & X \\
	{\{a\}} & {A^\sharp_{a/}} & {A^\sharp}
	\arrow[from=1-1, to=2-1]
	\arrow[from=2-1, to=2-2]
	\arrow[from=1-2, to=1-3]
	\arrow["p", from=1-3, to=2-3]
	\arrow[from=1-2, to=2-2]
	\arrow["{\pi_a}"', from=2-2, to=2-3]
	\arrow["i", from=1-1, to=1-2]
	\arrow["\lrcorner"{anchor=center, pos=0.125}, draw=none, from=1-1, to=2-2]
	\arrow["\lrcorner"{anchor=center, pos=0.125}, draw=none, from=1-2, to=2-3]
\end{tikzcd}\]
and the morphism $i$ is final as $p$ is proper. As $\bot$ sends final morphisms to equivalences, we then have an invertible natural transformation: 
\begin{equation}
\label{eq:explicit derived fiber}
\Rb a^*E\sim \bot \Rb a^*E\sim \bot \Rb \pi_a^*E
\end{equation}

\begin{prop}
\label{prop:fiber preserves colimits}
The functor $\Rb a^*:\LCart(A^\sharp)\to \LCart(1)\sim \ocat$ preserves colimits. 
\end{prop}
\begin{proof}
As $\pi_a$ is a right cartesian fibration, it is smooth and $\Rb \pi_a^*$ then preserves colimits. The functor $\bot$ also preserves them. The result then follows from the equivalence \eqref{eq:explicit derived fiber}.
\end{proof}

\p Let $E$ be an object of $\ocatm_{/A^\sharp}$ corresponding to a morphism $X\to A^\sharp$. We denote $\tilde{X}\to A^\sharp$ the left fibrant replacement of $E$. We then have a diagram
\[\begin{tikzcd}
	{X_{a/}} & {\tilde{X}_{a/}} & {A^\sharp_{a/}} \\
	X & {\tilde{X}} & {A^\sharp}
	\arrow[from=1-1, to=1-2]
	\arrow[from=1-2, to=1-3]
	\arrow["{\Fb E}"', from=2-2, to=2-3]
	\arrow[from=2-1, to=2-2]
	\arrow["{\pi_a}", from=1-3, to=2-3]
	\arrow[from=1-2, to=2-2]
	\arrow[from=1-1, to=2-1]
	\arrow["\lrcorner"{anchor=center, pos=0.125}, draw=none, from=1-1, to=2-2]
	\arrow["\lrcorner"{anchor=center, pos=0.125}, draw=none, from=1-2, to=2-3]
\end{tikzcd}\] 
 As $\pi_a$ is smooth, the canonical morphism 
$X_{a/}\to \tilde{X}_{a/}$ is initial. Combined with \eqref{eq:explicit derived fiber}, this induces an equivalence:
\begin{equation}
\label{eq:explicit derived fiber2}
\Rb a^*(\Fb E)\sim \bot X_{/a}
\end{equation}
\begin{prop}
\label{prop:quillent theorem A}
For a morphism $X\to A^\sharp$, and an object $a$ of $A$, we denote by $X_{/a}$ the marked $\io$-category fitting in the following cartesian square: 
\[\begin{tikzcd}
	{X_{a/}} & X \\
	{A^\sharp_{a/}} & {A^\sharp}
	\arrow[from=2-1, to=2-2]
	\arrow[from=1-2, to=2-2]
	\arrow[from=1-1, to=1-2]
	\arrow[from=1-1, to=2-1]
	\arrow["\lrcorner"{anchor=center, pos=0.125}, draw=none, from=1-1, to=2-2]
\end{tikzcd}\]
We denote by $\bot:\ocatm\to \ocat$ the functor sending a marked $\io$-category to its localization by marked cells.
\begin{enumerate}
\item Let $E$, $F$ be two elements of $\ocatm_{/A^\sharp}$ corresponding to morphisms $X\to A^\sharp$, $Y\to A^\sharp$, and
 $\phi:E\to F$ a morphism between them. The induced morphism $\Fb\phi:\Fb E\to \Fb F$ is an equivalence if and only if for any object $a$ of $A$, the induced morphism 
$$\bot X_{/a}\to \bot Y_{/a}$$ 
is an equivalence of $\io$-categories.
\item A morphism $X\to A^\sharp$ is initial if and only if for any object $a$ of $A$, $\bot X_{/a}$ is the terminal $\io$-category.
\end{enumerate}
\end{prop}
\begin{proof}
The first assertion is a direct consequence of the equation \eqref{eq:explicit derived fiber2} and of the fact that equivalences between left cartesian fibrations are detected on fibers.

A morphism $p:X\to A$ is initial if and only if $\Fb p$ is equivalent to the identity of $A^\sharp$, and according to the first assertion, if and only if for any object $a$ of $A$, the canonical morphism $\bot X_{a/}\to \bot A^\sharp_{a/}$ is an equivalence. However, the canonical morphism $\{a\}\to A_{/a}^\sharp$ is final, and $\bot A^\sharp_{a/}$ is then the terminal $\io$-category. This concludes the proof of the second assertion.
\end{proof}

\p 
Suppose given a commutative square of marked $\io$-categories: 
\begin{equation}
\label{eq:BC data}
\begin{tikzcd}
	A & C \\
	{B^\sharp} & {D^\sharp}
	\arrow["j", from=1-1, to=1-2]
	\arrow["u", from=1-2, to=2-2]
	\arrow["v"', from=1-1, to=2-1]
	\arrow["i"', from=2-1, to=2-2]
\end{tikzcd}
\end{equation}
This induces a square 
\begin{equation}
\label{eq:BC lax commutative square}
\begin{tikzcd}
	{\LCartc(C)} & {\LCartc(A)} \\
	{\LCart(D^\sharp)} & {\LCart(B^\sharp)}
	\arrow["{\Rb j^*}", from=1-1, to=1-2]
	\arrow["{\Lb  u_!}"', from=1-1, to=2-1]
	\arrow["{\Lb v_!}", from=1-2, to=2-2]
	\arrow["{\Rb i^*}"', from=2-1, to=2-2]
	\arrow[shorten <=8pt, shorten >=8pt, Rightarrow, from=1-2, to=2-1]
\end{tikzcd}
\end{equation}
that commutes up to a natural transformation 
\begin{equation}
\label{eq:BC nat}
\begin{array}{rcl}
\Lb v_!\circ \Rb j^*&\to &\Lb v_!\circ \Rb j^* \circ \Rb u^* \circ \Lb u_!\\
&\sim & \Lb v_!\circ \Rb v^* \circ \Rb i^* \circ \Lb u_!\\
&\to& \Rb i^* \circ \Lb u_!
\end{array}
\end{equation}
A square \eqref{eq:BC data} verifies the \notion{Beck-Chevaley condition} if this natural transformation \eqref{eq:BC nat} is an equivalence. This square verifies the \notion{weak Beck-Chevaley condition} if the natural transformation once composed with $\bot$ becomes an equivalence.

\begin{prop}
\label{prop:base change}
If the square \eqref{eq:BC data} is cartesian and $i$ is smooth, then it verifies the Beck-Chevaley condition.
\end{prop}
\begin{proof}
By construction, $\Lb v_!\circ \Rb j^*$ sends an object $E$ of $\LCartc(C)$ onto the fibrant replacement of $ v_!j^* E$. 
As $i$ is smooth, $\Rb i^* \circ \Lb u_!$ sends an object $E$ of $\LCart(C)$ onto the fibrant replacement of $i^*u_! E$. As pullbacks are stable under composition, we have $i^*u_!\sim v_!j^*$.
\end{proof}

\begin{lemma}
\label{lemma:smoth technical 1}
A square \eqref{eq:BC data} where both $j$ and $i$ are final verifies the weak Beck-Chevaley condition.
\end{lemma}
\begin{proof}
As  $\bot$ sends initial and final morphisms to equivalences, for any $E: \LCartc(A)$ and any $F:\LCartc(C)$, we have equivalences
$$\bot \Lb v_! E \sim \bot E~~~\mbox{ and }~~~\bot \Lb v_! F \sim \bot F.$$
Moreover, as classified left cartesian fibrations are proper, for any $G:\LCartc(C)$ and $H:\LCart(D^\sharp)$,  we have equivalences
$$\bot \Lb j^*G \sim \bot G~~~\mbox{ and }~~~\bot \Lb i^* H \sim \bot H.$$
This implies  the result.
\end{proof}

\begin{lemma}
\label{lemma:smoth technical 2}
Suppose given a cartesian square 
\[\begin{tikzcd}
	A & C \\
	{B^\sharp} & {D^\sharp}
	\arrow["j", from=1-1, to=1-2]
	\arrow["u", from=1-2, to=2-2]
	\arrow["v"', from=1-1, to=2-1]
	\arrow["i"', from=2-1, to=2-2]
\end{tikzcd}\]
such that for any object $b$ of $B^\sharp$, the outer square of the induced diagram
\[\begin{tikzcd}
	{A_{b/}} & A & C \\
	{B^\sharp_{/b}} & {B^\sharp} & {D^\sharp}
	\arrow["j", from=1-2, to=1-3]
	\arrow["u", from=1-3, to=2-3]
	\arrow["v"', from=1-2, to=2-2]
	\arrow["i"', from=2-2, to=2-3]
	\arrow["{v'}"', from=1-1, to=2-1]
	\arrow["{\pi_b}"', from=2-1, to=2-2]
	\arrow["{\pi_b'}", from=1-1, to=1-2]
	\arrow["\lrcorner"{anchor=center, pos=0.125}, draw=none, from=1-1, to=2-2]
	\arrow["\lrcorner"{anchor=center, pos=0.125}, draw=none, from=1-2, to=2-3]
\end{tikzcd}\]
verifies the weak Beck Chevaley condition. Then the right hand square verifies the Beck Chevaley condition.
\end{lemma}
\begin{proof}
Let $E$ be an element of $\LCart(C)$. Using the hypothesis, the fact that $\pi_a$ is a right cartesian fibration, and so smooth,, we have a sequence of equivalences: 
$$\begin{array}{rcll}
\bot \Rb \pi_b^*  \Lb v_! \Rb j^*E&\sim &\bot \Lb v'_! \Rb {\pi'_b}^*  \Rb j^*E&(\ref{prop:base change})\\
&\sim & \bot \Rb \pi_b^*  \Rb i  \Lb u_! E&\mbox{(hypothesis)}
\end{array}$$
Using the equivalence \eqref{eq:explicit derived fiber}, this implies that for any element $b$ of $B$, we have an equivalence 
$$ \Rb b^*  \Lb v_! \Rb j^*E\to \Rb b^*  \Rb i  \Lb u_!E$$
which concludes the proof as equivalences between left cartesian fibrations are detected fiberwise.
\end{proof}

\begin{prop}
\label{prop:BC condition}
Let $i:I\to A^\sharp$ and $j:C^\sharp\to D^\sharp$ be two morphisms. The square 
\[\begin{tikzcd}
	{C^\sharp\times I} & {D^\sharp\times I} \\
	{C^\sharp\times A^\sharp} & {D^\sharp\times A^\sharp}
	\arrow[from=1-1, to=2-1]
	\arrow[from=2-1, to=2-2]
	\arrow[from=1-1, to=1-2]
	\arrow[from=1-2, to=2-2]
\end{tikzcd}\]
verifies the Beck-Chevaley condition.
\end{prop}
\begin{proof}
According to lemma \ref{lemma:smoth technical 2}, one has to show that for any pair $(a,c)$ where $a$ is an object of $A^\sharp$ and $c$ of $C^\sharp$, the induced cartesian square
\[\begin{tikzcd}
	{C^\sharp_{c/}\times I_{a/}} & {D^\sharp\times I} \\
	{C^\sharp_{c/}\times A_{a/}^\sharp} & {D^\sharp\times A^\sharp}
	\arrow[from=1-2, to=2-2]
	\arrow[from=1-1, to=2-1]
	\arrow[from=1-1, to=1-2]
	\arrow[from=2-1, to=2-2]
\end{tikzcd}\]
verifies the weak Beck-Chevaley condition. Remark that this square factors as two cartesian squares:
\[\begin{tikzcd}
	{C^\sharp_{c/}\times I_{a/}} & {D^\sharp_{j(c)/}\times I_{a/}} & {D^\sharp\times I} \\
	{C^\sharp_{c/}\times A_{a/}^\sharp} & {D^\sharp_{j(c)/}\times A_{a/}^\sharp} & {D^\sharp\times A^\sharp}
	\arrow[from=1-3, to=2-3]
	\arrow[from=1-1, to=2-1]
	\arrow[from=2-1, to=2-2]
	\arrow[from=2-2, to=2-3]
	\arrow[from=1-1, to=1-2]
	\arrow[from=1-2, to=1-3]
	\arrow[from=1-2, to=2-2]
\end{tikzcd}\]
The two morphisms $\{c\}\to C^\sharp_{c/}$ and $\{c\}\to D^\sharp_{j(c)/}$ are initial, and by stability by left cancellation, so is $C^\sharp_{c/}\to D^\sharp_{j(c)/}$. By stability by cartesian product, the two horizontal morphisms of the left square are initial. Lemma \ref{lemma:smoth technical 1} then implies that the left square verifies the weak Beck-Chevaley condition. According to proposition \ref{prop:base change}, the right square fulfills the Beck-Chevaley condition, and so \textit{a fortiori}, the weak one. The outer square then verified the weak Beck-Chevaley condition, which concludes the proof.
\end{proof}

\p 
Suppose given a commutative square of marked $\io$-categories:
\begin{equation}
\label{eq:BC data}
\begin{tikzcd}
	A & {C^\sharp} \\
	B & {D^\sharp}
	\arrow["j", from=1-1, to=1-2]
	\arrow["u", from=1-2, to=2-2]
	\arrow["v"', from=1-1, to=2-1]
	\arrow["i"', from=2-1, to=2-2]
\end{tikzcd}
\end{equation}
where  $j$ and $i$ are smooth. This induces a square
\begin{equation}
\label{eq:BC lax commutative square2}
\begin{tikzcd}
	{\LCartc(B)} & {\LCart(D^\sharp)} \\
	{\LCartc(A)} & {\LCart(C^\sharp)}
	\arrow["{\Rb j_*}"', from=2-1, to=2-2]
	\arrow["{\Lb  u^*}", from=1-2, to=2-2]
	\arrow["{\Lb v^*}"', from=1-1, to=2-1]
	\arrow["{\Rb i_*}", from=1-1, to=1-2]
	\arrow[shorten <=8pt, shorten >=8pt, Rightarrow, from=1-2, to=2-1]
\end{tikzcd}
\end{equation}
that commutes up to a natural transformation 
\begin{equation}
\label{eq:BC nat2}
\begin{array}{rcl}
\Lb u^*\circ \Rb i_*&\to & \Rb j_*\circ\Lb j^* \circ \Lb u^*\circ \Rb i_*\\
&\sim &\Rb j_*\circ\Lb v^*\circ\Lb i^*  \circ \Rb i_*\\
&\to &\Rb j_*\circ\Lb v^*
\end{array}
\end{equation}
A square \eqref{eq:BC data} verifies the \notion{opposed Beck-Chevaley condition} if $i$ and $j$ are smooth and  the natural transformation \eqref{eq:BC nat2} is an equivalence.

\begin{prop}
\label{prop:base change2}
If the square \eqref{eq:BC nat2} is cartesian,  and $i$ and $j$ are smooth, then it verifies the opposed Beck-Chevaley condition.
\end{prop}
\begin{proof}
By adjunction, it is sufficient to show that the induced natural transformation
$$\Lb v_!\circ \Rb j^*\to \Rb i^* \circ \Lb u_!:\LCart(C^\sharp)\to \LCart(B)$$
is an equivalence. By construction, $\Lb v_!\circ \Rb j^*$ sends an object $E$ of $\LCart(C^\sharp)$ onto the fibrant replacement of $ v_!j^* E$. 
As $i$ is smooth, $\Rb i^* \circ \Lb u_!$ sends an object $E$ of $\LCart(C^\sharp)$ onto the fibrant replacement of $i^*u_! E$. As pullbacks are stable under composition, we have $i^*u_!\sim v_!j^*$.
\end{proof}

\begin{prop}
\label{prop:BC condition 2}
Let $i:I\to A^\sharp$ be a smooth morphism and $j:C^\sharp\to D^\sharp$ any morphism. The square
\[\begin{tikzcd}
	{ C^\sharp\times I} & { C^\sharp\times A^\sharp} \\
	{D^\sharp\times I} & { D^\sharp\times A^\sharp}
	\arrow[from=1-1, to=2-1]
	\arrow[from=1-2, to=2-2]
	\arrow[from=1-1, to=1-2]
	\arrow[from=2-1, to=2-2]
\end{tikzcd}\]
verifies the opposed Beck-Chevaley condition.
\end{prop}
\begin{proof}
As $id_{C^\sharp}\times i$ and $id_{D^\sharp}\times i$ are pullbacks of $i$, they are smooth. The result  is then follows from proposition \ref{prop:base change2}.
\end{proof}

\subsection{The $\Wcard$-small $\io$-category of $\V$-small left cartesian fibrations}

\p Let $I$ be a marked $\io$-category, and $a$ a globular sum. We recall that the pullback along the canonical projection $\pi_a:I\times a^\flat\to I$ induces an adjunction
\[\begin{tikzcd}
	{{\pi_a}_!:\ocat_{/I\times a^\flat}} & {\ocatm_{/I}:{\pi_a}^*}
	\arrow[""{name=0, anchor=center, inner sep=0}, shift left=2, from=1-1, to=1-2]
	\arrow[""{name=1, anchor=center, inner sep=0}, shift left=2, from=1-2, to=1-1]
	\arrow["\dashv"{anchor=center, rotate=-90}, draw=none, from=0, to=1]
\end{tikzcd}\]
\begin{lemma}
\label{lemma:to show fully faithfullness1}
Let $E$ and $F$ be two objects of $\ocatm_{/I}$ and $\psi:\pi_{[a,1]}^*E\to \pi_{[a,1]}^*F$ an equivalence.
The exists a unique commutative diagram of shape
\[\begin{tikzcd}
	{(\pi_{[a,1]})_!\pi_{[a,1]}^*E} & {(\pi_{[a,1]})_!\pi_{[a,1]}^*F} \\
	E & F
	\arrow["{(\pi_{[a,1]})_!\psi}", from=1-1, to=1-2]
	\arrow["\epsilon", from=1-2, to=2-2]
	\arrow["\epsilon"', from=1-1, to=2-1]
	\arrow["\phi"', dashed, from=2-1, to=2-2]
\end{tikzcd}\]
Moreover, the arrow $\phi$ is an equivalence.
\end{lemma}
\begin{proof}
Unfolding the definition, we have to show the existence and unicity of commutative diagrams of shape
\begin{equation}
\label{eq:square to show fully faithfullness}
\begin{tikzcd}
	{X\times [a,1]^\flat} & {Y\times [a,1]^\flat} \\
	X & Y
	\arrow["{\dom \psi}", from=1-1, to=1-2]
	\arrow[from=1-2, to=2-2]
	\arrow[from=1-1, to=2-1]
	\arrow["\phi"', dashed, from=2-1, to=2-2]
\end{tikzcd}
\end{equation}
where the two vertical morphisms are the projection and 
where $X$ and $Y$ correspond respectively to the domain of $E$ and $F$. As $\dom\psi$ is a morphism over $I\times [a,1]^\flat$, we already have a commutative diagram of shape:
\[\begin{tikzcd}
	{X\times [a,1]^\flat} & {Y\times [a,1]^\flat} \\
	{[a,1]^\flat} & {[a,1]^\flat}
	\arrow["{\dom \psi}", from=1-1, to=1-2]
	\arrow[from=1-2, to=2-2]
	\arrow[from=1-1, to=2-1]
	\arrow["id"', from=2-1, to=2-2]
\end{tikzcd}\]
By the universal property of cartesian product, this directly implies that if a square of shape \eqref{eq:square to show fully faithfullness} exists, it has to be unique, and that the morphism $\phi$ will be an equivalence. It then remains to show the existence.

Let $\psi'$ be an inverse of $\psi$. We denote $\tilde{\psi}:X\times [a,1]^\flat\to Y$ and $\tilde{\psi}':Y\times [a,1]^\flat\to X$ the morphisms induce by the adjunction from $\psi$ and $\psi'$. For $\epsilon\in\{0,1\}$, we denote by $\psi_{\epsilon}:X\times \{\epsilon\} \to Y$ and $\psi'_{\epsilon}:Y\times \{\epsilon\} \to X$ the induced morphisms. In particular $\psi_{\epsilon}$ and $\psi'_{\epsilon}$ are inverse one of the other.

 By construction, we have a commutative diagram
\[\begin{tikzcd}
	{X\times [a,1]^\flat\times [a,1]^\flat} & {Y\times[a,1]^\flat} \\
	{X\times[a,1]^\flat} & X
	\arrow["{\tilde{\psi}\times[a,1]^\flat}", from=1-1, to=1-2]
	\arrow["{\tilde{\psi}'}", from=1-2, to=2-2]
	\arrow["{X\times \triangledown}", from=2-1, to=1-1]
	\arrow["\pi"', from=2-1, to=2-2]
\end{tikzcd}\]
where $\triangledown$ is the diagonal and $\psi$ the canonical projection. This corresponds to a commutative diagram in the $\iun$-category $[n]\mapsto \Hom(X\times [a,n]^\flat,X)$:
\[\begin{tikzcd}
	{id_X} & {id_X} \\
	{id_X} & {id_X}
	\arrow["{\tilde\psi'*\psi_0}"', from=1-1, to=2-1]
	\arrow["{\tilde\psi'*\psi_1}", from=1-2, to=2-2]
	\arrow["{\psi_1'*\tilde\psi}"', from=2-1, to=2-2]
	\arrow["{\psi'_0*\tilde\psi}", from=1-1, to=1-2]
	\arrow["{id_{id_X}}"{description}, from=1-1, to=2-2]
\end{tikzcd}\]
Remark that in the $\iun$-category $[n]\mapsto \Hom(X\times [a,n]^\flat,Y)$, we have equivalences
$$\tilde\psi\sim\psi_0' *\psi_0*\psi~~~\mbox{ and }\tilde\psi\sim\psi_1' *\psi_1*\psi$$
and the previous diagram then induces two commutative triangles
\[\begin{tikzcd}
	{\psi_1} &&& {\psi_0} & {\psi_1} \\
	{\psi_0} & {\psi_1} &&& {\psi_0}
	\arrow["{id_{\psi_0}}"', from=1-4, to=2-5]
	\arrow["\tilde\psi", from=1-4, to=1-5]
	\arrow["{\psi_0*\tilde\psi'*\psi_1}", from=1-5, to=2-5]
	\arrow["\tilde\psi"', from=2-1, to=2-2]
	\arrow["{\psi_1*\tilde\psi'*\psi_0}"', from=1-1, to=2-1]
	\arrow["{id_{\psi_1}}", from=1-1, to=2-2]
\end{tikzcd}\]
View as a $1$-cell of $[n]\mapsto \Hom(X\times [a,n]^\flat,Y)$, $\tilde{\psi}$ is then an equivalence. This implies the existence of a lifts in the following diagram
\[\begin{tikzcd}
	{[a,1]^\flat} & {\uHom(X,Y)} \\
	1
	\arrow["\tilde\psi", from=1-1, to=1-2]
	\arrow[from=1-1, to=2-1]
	\arrow["\phi"', dashed, from=2-1, to=1-2]
\end{tikzcd}\]
which induces the wanted square:
\[\begin{tikzcd}
	{X\times[a,1]^\flat} & {Y\times[a,1]^\flat} \\
	X & X
	\arrow[from=1-1, to=2-1]
	\arrow["\phi"', dashed, from=2-1, to=2-2]
	\arrow[from=1-2, to=2-2]
	\arrow["{\tilde{\psi}}", from=1-1, to=2-2]
	\arrow[from=1-1, to=1-2]
\end{tikzcd}\]
\end{proof}

\begin{lemma}
\label{lemma:to show fully faithfullness2}
Let $I$ be a marked $\io$-category and $a$ a globular form. 
The canonical morphisms of $\infty$-groupoids:
$$\pi_{[a,1]}^*:\tau_0\ocatm_{/I}\to \tau_0\ocatm_{/I\times [a,1]^\flat}$$
$$\pi_{[a,1]}^*:\tau_0 \Arr(\ocatm_{/I})\to \tau_0\Arr(\ocatm_{/I\times [a,1]^\flat})$$
are fully faithful.
\end{lemma}
\begin{proof}
Let $E$ and $F$ be two objects of $\ocatm_{/I}$. The morphism 
$$\Hom_{\tau_0\ocatm_{/I}}(E,F) \to \Hom_{\tau_0\ocatm_{/I\times[a,1]^\flat}}(\pi_{[a,1]}^*E,\pi_{[a,1]}^*F) $$ has an inverse that sends $\psi:\pi_{[a,1]}^*E\to \pi_{[a,1]}^*F$ onto the morphism $\phi:E\to F$ appearing in the commutative square provided by lemma \ref{lemma:to show fully faithfullness1}.

The second assertion is demonstrated similarly.
\end{proof}

\begin{prop}
\label{prp:to show fully faithfullness3}
Let $I$ be a marked $\io$-category and $a$ a globular form. We denote by $\pi_a:I\times a^\flat\to I$ the canonical projection.
The canonical morphisms of $\infty$-groupoids:
$$\Rb{\pi_a}^*:\tau_0\LCartc(I)\to \tau_0\LCartc(I\times a^\flat)$$
$$\Rb{\pi_a}^*:\tau_0 \Arr(\LCartc(I))\to \tau_0\Arr(\LCartc(I\times a^\flat))$$
are fully faithful.
\end{prop}
\begin{proof}
Let $[\textbf{b},n]:= a$. Considere first the adjunction:
\[\begin{tikzcd}
	{\LCartc(I\times [b_0,1]^\flat)\times_{\LCartc(I)}...\times_{\LCartc(I)}\LCartc(I\times [b_{n-1},1]^\flat)} \\
	{\LCartc(I^\flat\times [\textbf{b},n])}
	\arrow[""{name=0, anchor=center, inner sep=0}, shift left=2, from=2-1, to=1-1]
	\arrow[""{name=1, anchor=center, inner sep=0}, "{\colim_I}", shift left=2, from=1-1, to=2-1]
	\arrow["\dashv"{anchor=center, rotate=-180}, draw=none, from=1, to=0]
\end{tikzcd}\]
The corollary \ref{cor:fib over a colimit} implies that the counit of this adjunction is an equivalence.
This implies that the right adjoint
$$\LCartc(I^\flat\times [\textbf{b},n])\to \LCartc(I\times [b_0,1]^\flat)\times_{\LCartc(I)}...\times_{\LCartc(I)}\LCartc(I\times [b_{n-1},1]^\flat)$$ is fully faithful. 
By right cancellation and using the fact that fully faithful functors are stable by limits, it is sufficient to show that for any $k<n$, 
$$\Rb{\pi_{[b_i,1]}}^*:\tau_0\LCartc(I)\to \tau_0\LCartc(I\times [b_k,1]^\flat)$$
is fully faithful. 
Moreover, for any such $k$, we have a commutative square
\[\begin{tikzcd}
	{\tau_0\LCartc(I)} & {\tau_0\LCartc(I\times [b_k,1]^\flat)} \\
	{\tau_0\ocatm_{/I}} & {\tau_0\ocatm_{/I\times [b_k,1]^\flat}}
	\arrow["{\Rb{\pi_{[b_k,1]}}^*}", from=1-1, to=1-2]
	\arrow[from=1-1, to=2-1]
	\arrow[from=1-2, to=2-2]
	\arrow["{{\pi_{[b_k,1]}}^*}"', from=2-1, to=2-2]
\end{tikzcd}\]
whose vertical morphisms are fully faithful by construction. The results the follows from lemma \ref{lemma:to show fully faithfullness2} by right cancellation.

The second assertion is demonstrated similarly.

\end{proof}

\p For an $\io$-category $A$ and a globular sum $a$, we define $\LCart(A^\sharp;a)$ as the full sub $\iun$-category of $\LCartc(A^\sharp\times a^\flat)$ whose objects are of shape $E\times id_a^\flat$ for $E$ an object of $\LCart(A^\sharp)$. The proposition \ref{prp:to show fully faithfullness3} implies that the canonical morphism 
$$\tau_0\LCart(A^\sharp)\to \tau_0\LCart(A^\sharp;a)$$
is an equivalence of $\infty$-groupoid. We define \wcnotation{$\uLCart(A^\sharp)$}{(lcart@$\uLCart(\uvar)$} as the $\Wcard$-small $\io$-category whose value on $[a,n]$ is given by:
$$\uLCart(A^\sharp)([a,n]):=\Hom([n],\LCart(A^\sharp;a)).$$
For a marked $\io$-category $I$ and a globular sum $a$, we define similarly $\LCartc(I;a)$ as the full sub $\iun$-category of $\LCartc(I\times a^\flat)$ whose objects are of shape $E\times id_a^\flat$ for $E$ an object of $\LCartc(I)$. The proposition \ref{prp:to show fully faithfullness3} implies that the canonical morphism 
$$\tau_0\LCartc(I)\to \tau_0\LCartc(I;a)$$
is an equivalence of $\infty$-groupoid. We define \wcnotation{$\uLCartc(I)$}{(lcartc@$\uLCartc(\uvar)$} as the $\Wcard$-small $\io$-category whose value on $[a,n]$ is given by:
$$\uLCartc(I)([a,n]):=\Hom([n],\LCartc(I;a)).$$
These two definitions are compatible as we have an equivalence between $\uLCartc(A^\sharp)$ and $\uLCart(A^\sharp)$.

\p Let $E$ and $F$ be two objects of $\uLCartc(I)$, and $a$ a globular sum. Remark that a morphism $[a,1]\to \uLCartc(I)$ corresponds to a morphism $E\times id_a\to F\times id_a$, and so to a morphism $X\times a\to Y$ over $I$ where $X$ and $Y$ are respectively the domain of $E$ and $F$. We then have an equivalence: 
\begin{equation}
\hom_{\uLCart(I)}(E,F)\sim \Map_I(E,F). 
\end{equation}
This then implies that $\LCartc(I)$ is locally $\V$-small.

\p Let $i:I\to J$ be a morphism between marked $\io$-category, $a$ a globular sum, and $p$ a classified left cartesian fibration over 
$a^\flat\times J$. Remark that we have a canonical equivalence $$\Rb (i\times id_{a^\flat})^*(p\times id_{a^\flat})\sim (\Rb i^*p)\times id_{a^\flat}$$ natural in $a:\Theta^{op}$. The functor $\Rb (i\times id_{a^\flat})^*$ then restricts to a functor 
$$(i_a)^*:\LCartc(J;a)\to \LCartc(I;a)$$
natural in $a:\Theta^{op}$, and then to a morphism of $\io$-categories:
\begin{equation}
\label{eq:i pullback}
i^*:\uLCartc(J)\to \uLCartc(I)
\end{equation}
\index[notation]{(f5@$f^*:	\uLCartc(J)\to \uLCartc(I)$}

\p 
Let $i:I\to A^\sharp$ be a morphism between marked $\io$-categories. We are now willing to construct a morphism $i_!:\uLCartc(I)\to \uLCart(A^\sharp)$ which corresponds to $\Lb i_!:\LCartc(I)\to \LCart(A^\sharp)$ on the maximal sub $\iun$-category.

We denote by $E_0$ and $E_1$ the $\iun$-categories fitting in the cartesian square: 
\[\begin{tikzcd}
	{E_0} & \Theta & {E_1} & \Theta \\
	{\Arr^{fib}(\ocatm)} & \ocatm & {\Arr^{fib}(\ocatm)} & \ocatm
	\arrow["\codom"', from=2-1, to=2-2]
	\arrow["{\psi_0}", from=1-2, to=2-2]
	\arrow[from=1-1, to=2-1]
	\arrow[from=1-1, to=1-2]
	\arrow["\lrcorner"{anchor=center, pos=0.125}, draw=none, from=1-1, to=2-2]
	\arrow["{\psi_1}", from=1-4, to=2-4]
	\arrow[from=1-3, to=2-3]
	\arrow["\codom"', from=2-3, to=2-4]
	\arrow[from=1-3, to=1-4]
	\arrow["\lrcorner"{anchor=center, pos=0.125}, draw=none, from=1-3, to=2-4]
\end{tikzcd}\]
where $\Arr^{fib}(\ocatm)$ is the full sub $\iun$-category of $\Arr(\ocatm)$ whose objects are classified left cartesian fibrations, and where $\psi_0$ and $\psi_1$ send respectively $a$ on $I\times a^\flat$ and $A^\sharp\times a^\flat$. 
The morphism $i$ induces an adjunction
\begin{equation}
\label{eq:adj i pull}
\begin{tikzcd}
	{i_!:E_0} & {E_1:i^*}
	\arrow[""{name=0, anchor=center, inner sep=0}, shift left=2, from=1-1, to=1-2]
	\arrow[""{name=1, anchor=center, inner sep=0}, shift left=2, from=1-2, to=1-1]
	\arrow["\dashv"{anchor=center, rotate=-90}, draw=none, from=0, to=1]
\end{tikzcd}
\end{equation}
where the left adjoint sends a left cartesian fibration $p$ over $I\times a^\flat$ to $\Lb (i\times id_a)_!p$ and the right adjoint sends a left cartesian fibration $q$ over $A^\sharp\times a^\flat$ to $\Rb (i\times id_a)^* q$.
\begin{lemma}
\label{lemma:technical lemma i pull}
Let $p$ be a left cartesian fibration over $I^\sharp$. We have an equivalence $$\Lb (i\times id_{a^\flat})_!(p\times id_{a^\flat})\sim (\Lb i_! p)\times id_{a^\flat}.$$
Let $q$ be a left cartesian fibration over $A^\sharp$. We have an equivalence $$\Rb (i\times id_{a^\flat})^*(q\times id_{a^\flat})\sim (\Rb i^* q)\times id_{a^\flat}.$$
\end{lemma}
\begin{proof}
The first assertion is straightforward as the cartesian product with $a^\flat$ preserves initial morphisms and left cartesian fibrations. The second assertion is obvious.
\end{proof}
We define $\tilde{E_0}$ and $\tilde{E_1}$ as the full sub $\iun$-categories of $E_0$ and $E_1$ whose objects are respectively of shape $p\times id_a$ and $q\times id_a$ for $p$ and $q$ classified left cartesian fibrations over $I$ and $A^\sharp$.
The last lemma implies that \eqref{eq:adj i pull} restricts to an adjunction
\begin{equation}
\label{eq:adj i pull2}
\begin{tikzcd}
	{i_!:\tilde{E_0}} & {\tilde{E_1}:i^*}
	\arrow[""{name=0, anchor=center, inner sep=0}, shift left=2, from=1-1, to=1-2]
	\arrow[""{name=1, anchor=center, inner sep=0}, shift left=2, from=1-2, to=1-1]
	\arrow["\dashv"{anchor=center, rotate=-90}, draw=none, from=0, to=1]
\end{tikzcd}
\end{equation}

\begin{lemma}
\label{lemma:technical lemma i pull2} $~$
\begin{enumerate}
\item
Let $q\to q'$ be a morphism in $\tilde{E_0}$ corresponding to a cartesian square. The induced morphism $i_!(q)\to i_!(q')$ also corresponds to a cartesian square. 
\item
Let $q\to q'$ be a morphism in $\tilde{E_1}$ corresponding to a cartesian square. The induced morphism $i^*(q)\to i^*(q')$ also corresponds to a cartesian square. 
\end{enumerate}
\end{lemma}
\begin{proof}
Cartesian morphisms in $\tilde{E_0}$ corresponds to cartesian squares
\[\begin{tikzcd}
	{X\times a^{\flat}} & {X\times b^{\flat}} \\
	{I\times a^{\flat}} & {I\times b^{\flat}}
	\arrow["{p\times id_a}"', from=1-1, to=2-1]
	\arrow[from=2-1, to=2-2]
	\arrow["{p\times id_b}", from=1-2, to=2-2]
	\arrow[from=1-1, to=1-2]
\end{tikzcd}\]
and cartesian morphisms in $\tilde{E_1}$ corresponds to cartesian squares
\[\begin{tikzcd}
	{Y\times a^{\flat}} & {Y\times b^{\flat}} \\
	{A^\sharp\times a^{\flat}} & {A^\sharp\times b^{\flat}}
	\arrow["{q\times id_a}"', from=1-1, to=2-1]
	\arrow[from=2-1, to=2-2]
	\arrow["{q\times id_b}", from=1-2, to=2-2]
	\arrow[from=1-1, to=1-2]
\end{tikzcd}\]
The results directly follows from lemma \ref{lemma:technical lemma i pull}.
\end{proof}
The canonical projection $\tilde{E_0}\to \Theta$ and $\tilde{E_1}\to \Theta$ are Grothendieck fibrations in $\iun$-categories. The cartesian lifting is given by cartesian squares. Moreover, their Grothendieck deconstructions correspond respectively to 
$a\mapsto \LCartc(I;a)$ and $a\mapsto \LCart(A^\sharp;b)$. As both $i_!$ and $i^*$ preserve cartesian lifting according to lemma \ref{lemma:technical lemma i pull2}, they induce by Grothendieck deconstruction a family of adjunction
\begin{equation}
\label{eq:adj i pull3}
\begin{tikzcd}
	{(i_a)_!:\LCartc(I;a)} & {\LCart(A^\sharp;a):(i_a)^*}
	\arrow[""{name=0, anchor=center, inner sep=0}, shift left=2, from=1-1, to=1-2]
	\arrow[""{name=1, anchor=center, inner sep=0}, shift left=2, from=1-2, to=1-1]
	\arrow["\dashv"{anchor=center, rotate=-90}, draw=none, from=0, to=1]
\end{tikzcd}
\end{equation}
natural in $a:\Theta^{op}$. The family of functors $(i_a)_!$ then induces a morphism of $\io$-category\index[notation]{(f4@$f_{\mbox{$\exclam$}}:\uLCartc(I)\to \uLCart(A^\sharp)$}
\begin{equation}
\label{eq:i pull}
i_!:\uLCartc(I)\to \uLCart(A^\sharp)
\end{equation}
which corresponds to $\Lb i_!:\LCartc(I)\to \LCart(A^\sharp)$ on the maximal sub $\iun$-category.
The unit and counit of adjunction \eqref{eq:adj i pull3} induce morphisms
\begin{equation}
\label{eq:i pull unit an counit}
\mu:id\to i^*i_!~~~~ \epsilon:i_!i^*\to id
\end{equation}
and equivalences
$(\epsilon\circ_0 i_!)\circ_1(i_!\circ_0 \mu) \sim id_{i_!}$ and $(i^*\circ_0 \epsilon)\circ_1 (\mu \circ_0 i^* )\sim id_{i^*}$.

\p Let $j:C^\sharp\to D^\sharp$ be a morphism between $\io$-categories. We claim that the commutative square 
\[\begin{tikzcd}
	{\uLCart(D^\sharp\times A^\sharp)} & {\uLCartc(D^\sharp\times I)} \\
	{\uLCart(C^\sharp\times A^\sharp)} & {\uLCartc(C^\sharp\times I)}
	\arrow["{(j\times id_{I})^*}", from=1-2, to=2-2]
	\arrow["{( id_{D^\sharp}\times i)^*}", from=1-1, to=1-2]
	\arrow["{( id_{C^\sharp}\times i)^*}"', from=2-1, to=2-2]
	\arrow["{(j\times id_{A^\sharp})^*}"', from=1-1, to=2-1]
\end{tikzcd}\]
induces a commutative square
\begin{equation}
\label{eq:commutative pull push}
\begin{tikzcd}
	{\uLCartc(D^\sharp\times I)} & {\uLCartc(C^\sharp\times I)} \\
	{\uLCart(D^\sharp\times A^\sharp)} & {\uLCart(C^\sharp\times A^\sharp)}
	\arrow["{( id_{D^\sharp}\times i)_!}"', from=1-1, to=2-1]
	\arrow["{(j\times id_{I})^*}", from=1-1, to=1-2]
	\arrow["{( id_{C^\sharp}\times i)_!}", from=1-2, to=2-2]
	\arrow["{(j\times id_{A^\sharp})^*}"', from=2-1, to=2-2]
\end{tikzcd}
\end{equation}
\textit{A priori}, the natural transformations \eqref{eq:i pull unit an counit} implies that this square commutes up the natural transformation:
$$
\begin{array}{rcl}
( id_{C^\sharp}\times i)_!\circ (j\times id_{I})^*&\to &( id_{C^\sharp}\times i)_! \circ (j\times id_{I})^* \circ ( id_{D^\sharp}\times i)^*\circ ( id_{D^\sharp}\times i)_!\\
&\sim &( id_{C^\sharp}\times i)_!\circ ( id_{C^\sharp}\times i)^*\circ (j\times id_{A^\sharp})^*\circ ( id_{D^\sharp}\times i)_!\\
&\to&(j\times id_{A^\sharp})^*\circ ( id_{D^\sharp}\times i)_!
\end{array}
$$
Proposition \ref{prop:BC condition} implies that this natural transformation is pointwise an equivalence, and so is globally an equivalence.

\p We now suppose that the morphism $i:I\to A^\sharp$ is smooth, and we are willing to construct a morphism $i_*:\uLCart(A^\sharp)\to \uLCart(I)$ which corresponds to $\Rb i_*:\LCartc(I)\to \LCart(A^\sharp)$ on the sub maximal $\iun$-categories.. 

As smooth morphisms are stable by pullback, the maps $i\times id_b^\flat$ are smooth for any $b:\Theta$. The morphism $i^*:E_0\to E_1$ then preserves colimits and fits into an adjunction
\begin{equation}
\label{eq:adj i pullstar}
\begin{tikzcd}
	{i^*:E_1} & {E_0:i_*}
	\arrow[""{name=0, anchor=center, inner sep=0}, shift left=2, from=1-1, to=1-2]
	\arrow[""{name=1, anchor=center, inner sep=0}, shift left=2, from=1-2, to=1-1]
	\arrow["\dashv"{anchor=center, rotate=-90}, draw=none, from=0, to=1]
\end{tikzcd}
\end{equation}
where the left adjoint sends a left cartesian fibration $p$ over $A^\sharp\times a^\flat$ to $ (i\times id_a)^*p$ and the right adjoint sends a left cartesian fibration $q$ over $I\times a^\flat$ to $\Rb (i\times id_a)_* q$.
\begin{lemma}
\label{lemma:technical lemma i pullstar}
Let $p$ be a left cartesian fibration over $I$. We have an equivalence $$\Rb (i\times id_{a^\flat})_*(p\times id_{a^\flat})\sim (\Rb i_* p)\times id_{a^\flat}.$$
\end{lemma}
\begin{proof}
The morphism $p\times id_{a^\flat}$ is the limit of the cospan
$$p\to id_I\leftarrow id_I\times id_{a^\flat}$$
The result is then a direct consequence of the fact that $\Rb i_*$ preserves limits as it is a right adjoint.
\end{proof}

We recall that $\tilde{E_0}$ and $\tilde{E_1}$ are defined as the full sub $\iun$-categories of $E_0$ and $E_1$ whose objects are respectively of shape $p\times id_a$ and $q\times id_a$ for $p$ and $q$ classified left cartesian fibrations over $I$ and $A^\sharp$.
The lemma \ref{lemma:technical lemma i pullstar} and the second assertion of lemma \ref{lemma:technical lemma i pull} imply that \eqref{eq:adj i pullstar} restricts to an adjunction
\begin{equation}
\label{eq:adj i pull2star}
\begin{tikzcd}
	{i^*:\tilde{E_1}} & {\tilde{E_0}:i_*}
	\arrow[""{name=0, anchor=center, inner sep=0}, shift left=2, from=1-1, to=1-2]
	\arrow[""{name=1, anchor=center, inner sep=0}, shift left=2, from=1-2, to=1-1]
	\arrow["\dashv"{anchor=center, rotate=-90}, draw=none, from=0, to=1]
\end{tikzcd}
\end{equation}

\begin{lemma}
\label{lemma:technical lemma i pull2star} 
Let $q\to q'$ be a morphism in $\tilde{E_0}$ corresponding to a cartesian square. The induced morphism $i_*(q)\to i_*(q')$ also corresponds to a cartesian square. 
\end{lemma}
\begin{proof}
The proof is similar to that of the lemma \ref{lemma:technical lemma i pull2}, using lemma \ref{lemma:technical lemma i pullstar} instead of  lemma \ref{lemma:technical lemma i pull}.
\end{proof}

The lemmas \ref{lemma:technical lemma i pull2} and \ref{lemma:technical lemma i pull2star} imply that the two adjoints of \eqref{eq:adj i pull2star} preserve the cartesian cells of the Grothendieck fibrations $\tilde{E_0}\to \Theta$ and $\tilde{E_1}\to \Theta$. These two adjoints then induce by
 Grothendieck deconstruction a family of adjunction
\begin{equation}
\label{eq:adj i pull3star}
\begin{tikzcd}
	{(i_a)^*:\LCart(A^\sharp;a)} & {\LCartc(I;a):(i_a)_*}
	\arrow[""{name=0, anchor=center, inner sep=0}, shift left=2, from=1-1, to=1-2]
	\arrow[""{name=1, anchor=center, inner sep=0}, shift left=2, from=1-2, to=1-1]
	\arrow["\dashv"{anchor=center, rotate=-90}, draw=none, from=0, to=1]
\end{tikzcd}
\end{equation}
natural in $a:\Theta^{op}$. The family of functors $(i_a)_*$ then induces a morphism of $\io$-categories\index[notation]{(f6@$f_*:\uLCartc(I)\to \uLCart(A^\sharp)$}
\begin{equation}
\label{eq:i push op}
i_*:\uLCartc(I)\to \uLCart(A^\sharp)
\end{equation}
which is equivalent to $\Rb i_*:\LCartc(I)\to \LCart(A^\sharp)$ on the sub maximal $\iun$-categories.
The unit and counit of adjunction \eqref{eq:adj i pull3star} induce natural transformation
\begin{equation}
\label{eq:i pull unit an counit op}
\mu: id\to i_*i^*~~~~ \epsilon:i^*i_*\to id
\end{equation}
and equivalences
$(\epsilon\circ_0 i^*)\circ_1(i^*\circ_0 \mu) \sim id_{i^*}$ and $(i_*\circ_0 \epsilon)\circ_1 (\mu \circ_0 i_* )\sim id_{i_*}$.

\p Let $j:C^\sharp\to D^\sharp$ be a morphism between $\io$-categories. We claim that the commutative square 
\[\begin{tikzcd}
	{\uLCart(D^\sharp\times A^\sharp)} & {\uLCartc(D^\sharp\times I)} \\
	{\uLCart(C^\sharp\times A^\sharp)} & {\uLCartc(C^\sharp\times I)}
	\arrow["{(j\times id_{I})^*}", from=1-2, to=2-2]
	\arrow["{( id_{D^\sharp}\times i)^*}", from=1-1, to=1-2]
	\arrow["{( id_{C^\sharp}\times i)^*}"', from=2-1, to=2-2]
	\arrow["{(j\times id_{A^\sharp})^*}"', from=1-1, to=2-1]
\end{tikzcd}\]
induces a commutative square
\begin{equation}
\label{eq:commutative pull push op}
\begin{tikzcd}
	{\uLCartc(D^\sharp\times I)} & {\uLCart(D^\sharp\times A^\sharp)} \\
	{\uLCartc(C^\sharp\times I)} & {\uLCart(C^\sharp\times A^\sharp)}
	\arrow["{( id_{D^\sharp}\times i)_*}", from=1-1, to=1-2]
	\arrow["{(j\times id_{I})^*}"', from=1-1, to=2-1]
	\arrow["{( id_{C^\sharp}\times i)_*}"', from=2-1, to=2-2]
	\arrow["{(j\times id_{A^\sharp})^*}", from=1-2, to=2-2]
\end{tikzcd}
\end{equation}

\textit{A priori}, the natural transformations \eqref{eq:i pull unit an counit op} implies that this square commutes up the natural transformation:
$$
\begin{array}{rcl}
(j\times id_{A^\sharp})^*\circ (id_{D^\sharp}\times i)_*&\to & (id_{C^\sharp}\times i)_*\circ(id_{C^\sharp}\times i)^*\circ(j\times id_{A^\sharp})^*\circ (id_{D^\sharp}\times i)_*\\
&\sim &(id_{C^\sharp}\times i)_*\circ(j\times id_I)^*\circ(id_{D^\sharp}\times i)^*\circ (id_{D^\sharp}\times i)_*\\
&\to &(id_{C^\sharp}\times i)_*\circ(j\times id_I)^*
\end{array}
$$
Proposition \ref{prop:BC condition 2} implies that this natural transformation is pointwise an equivalence, and so is globally an equivalence.

\chapter{The $\io$-category of small $\io$-categories}
\label{chapter:The io-category of small io-categories}

\minitoc
\vspace{2cm}
This chapter aims to establish analogs of the fundamental categorical constructions to the $\io$ case. In the first section, we define the $\io$-category of small $\io$-categories $\uni$ (paragraph \ref{para:defi of uni}), and we prove a first incarnation of the Grothendieck construction:
\begin{icor}[\ref{cor: Grt equivalence}]
Let $\uni$ be the $\io$-category of small $\io$-categories, and $A$ an $\io$-category. There is an equivalence
$$\int_A:\Hom(A,\uni)\to \tau_0 \LCart(A^\sharp).$$
where $\tau_0 \LCart(A^\sharp)$ is the $\infty$-groupoid of left cartesian fibrations over $A^\sharp$ with small fibers.
\end{icor}
Given a functor $f:A\to \uni$, the left cartesian fibration $\int_Af$ is a colimit (computed in $\ocatm_{/A^\sharp}$) of
a simplicial object whose value on $n$ is of shape
$$\coprod_{x_0,...,x_n:A_0}X(x_0)^\flat\times\hom_A(x_0,...,x_n)^\flat\times A^\sharp_{x_n/}\to A^\sharp$$
This formula is similar to the one given in \cite{Gepner_Lax_colimits_and_free_fibration}
 for $\iun$-categories, and to the one given in \cite{Warren_the_strict_omega_groupoid_interpretation_of_type_theory} for strict $\omega$-categories.

We also prove a univalence result:

\begin{icor}[\ref{cor:univalence}]
Let $I$ be a marked $\io$-category. We denote by $I^\sharp$ the marked $\io$-category obtained from $I$	 by marking all cells and $\iota:I\to I^\sharp$ the induced morphism. There is a natural correspondence between \begin{enumerate}
\item functors
$f:I\otimes [1]^\sharp\to \uni^\sharp,$

\item pairs of small left cartesian fibration $X\to I^\sharp$, $Y\to I^\sharp$ together with diagrams 
\[\begin{tikzcd}
	& {\iota^*X} && X \\
	{\iota^*Y} && Y \\
	& I && {I^\sharp}
	\arrow[""{name=0, anchor=center, inner sep=0}, "\iota"', from=3-2, to=3-4]
	\arrow[from=2-1, to=3-2]
	\arrow[from=2-3, to=3-4]
	\arrow[from=2-1, to=2-3]
	\arrow[from=1-2, to=3-2]
	\arrow[from=1-4, to=3-4]
	\arrow[from=1-2, to=1-4]
	\arrow["\phi"{description}, from=1-2, to=2-1]
	\arrow["\lrcorner"{anchor=center, pos=0.125}, draw=none, from=1-2, to=0]
	\arrow["\lrcorner"{anchor=center, pos=0.125}, draw=none, from=2-1, to=0]
\end{tikzcd}\]
\end{enumerate}
\end{icor}

Recall that if $I$ is of shape $B^\sharp$, then the underlying $\io$-category of $B^\sharp\otimes[1]^\sharp$ is $B\times [1]$, and if $I$ is of shape $B^\flat$, the underlying $\io$-category of $B^\flat\otimes[1]^\sharp$ is $B\otimes[1]$. On the other hand, if $I$ is $B^\sharp$, $\iota$ is the identity, and $\phi$ then preserves all cartesian liftings, and if $I$ is $B^\flat$, $\phi$ doesn't need to preserve cartesian liftings.

By varying the marking, we can continuously move from the cartesian product with the interval to the Gray product with the interval on one side, and on the other side, we can continuously move from morphisms between left cartesian fibrations that preserve the marking to the ones that do not preserve it \textit{a priori}.

Eventually, we also get an $\io$-functorial Grothendieck construction, expressed by the following corollary:

\begin{icor}[\ref{cor:lcar et hom}]
Let $A$ be a $\U$-small $\io$-category.
Let $\uLCart(A^\sharp)$ be the $\io$-category of small left cartesian fibrations over $A^\sharp$. 
There is an equivalence
$$\uHom(A,\uni)\sim \uLCart(A^\sharp)$$
natural in $A$.
\end{icor}

In the second section of this chapter, for a locally small $\io$-category $C$, we construct the Yoneda embedding, which is a functor $y:C\to \widehat{C}$ where $\widehat{C}:=\uHom(C^t,\uni)$. We prove the Yoneda lemma:
\begin{itheorem}[\ref{theo:Yoneda ff}]
The Yoneda embedding is fully faithful.
\end{itheorem}
\begin{itheorem}[\ref{theo:Yoneda lemma}]
Let $C$ be an $\io$-category. There is an equivalence between the functor
$$\hom_{\w{C}}(y_{\uvar},\uvar):C^t\times \w{C}\to \uni$$ and
the functor 
$$ev:C^t\times \w{C}\to \uni .$$
\end{itheorem}
In the last three sections, we use these results to study and demonstrate the basic properties of adjunctions, lax (co)limits, and left Kan extensions.

We begin by studying adjunctions, and we establish the following expected result.
\begin{itheorem}[\ref{theo:two adjunction definition}]
Let $u:C\to D$ and $v:D\to C$ be two functors between locally $\U$-small $\io$-categories. 
The two following are equivalent. 
\begin{enumerate}
\item The pair $(u,v)$ admits an adjoint structure.
\item Their exists a pair of natural transformations $\mu: id_C \to vu$ and $\epsilon:uv\to id_D$ together with equivalences $(\epsilon\circ_0 u)\circ_1(u\circ_0 \mu) \sim id_{u}$ and $(v\circ_0 \epsilon)\circ_1 (\mu \circ_0 v )\sim id_{v}$.
\end{enumerate}
\end{itheorem}

In the next subsection, given a morphism $f:I\to C^\sharp$ between marked $\io$-categories, we define the notion of lax colimit and lax limit for the functor $f$. If $f$ admits such a lax colimit, for any $1$-cell $i:a\to b$ in $I$, we have a triangle
\[\begin{tikzcd}
	{} & {F(b)} \\
	{F(a)} & {\laxcolim_IF}
	\arrow["{F(i)}", curve={height=-30pt}, from=2-1, to=1-2]
	\arrow[from=2-1, to=2-2]
	\arrow[shorten <=8pt, shorten >=8pt, Rightarrow, from=1-2, to=2-1]
	\arrow[draw=none, from=1-1, to=2-1]
	\arrow[from=1-2, to=2-2]
\end{tikzcd}\]
If $i$ is marked, the preceding $2$-cell is an equivalence. 
For any $2$-cell $u:i\to j$, we have a diagram
\[\begin{tikzcd}
	& {F(b)} & {} & {F(b)} \\
	{F(a)} & {\laxcolim_IF} & {F(a)} & {\laxcolim_IF}
	\arrow[""{name=0, anchor=center, inner sep=0}, "{F(i)}"{description}, from=2-1, to=1-2]
	\arrow[""{name=1, anchor=center, inner sep=0}, from=2-1, to=2-2]
	\arrow[from=1-2, to=2-2]
	\arrow[""{name=2, anchor=center, inner sep=0}, from=1-2, to=2-2]
	\arrow[""{name=3, anchor=center, inner sep=0}, "{F(j)}", curve={height=-30pt}, from=2-1, to=1-2]
	\arrow["{F(j)}", curve={height=-30pt}, from=2-3, to=1-4]
	\arrow[from=2-3, to=2-4]
	\arrow[from=1-4, to=2-4]
	\arrow[shorten <=8pt, shorten >=8pt, Rightarrow, from=1-4, to=2-3]
	\arrow[""{name=4, anchor=center, inner sep=0}, draw=none, from=1-3, to=2-3]
	\arrow[shift right=2, shorten <=12pt, shorten >=12pt, Rightarrow, from=2, to=1]
	\arrow[shorten <=4pt, shorten >=4pt, Rightarrow, from=3, to=0]
	\arrow[shift left=0.7, shorten <=14pt, shorten >=16pt, no head, from=2, to=4]
	\arrow[shorten <=14pt, shorten >=14pt, from=2, to=4]
	\arrow[shift right=0.7, shorten <=14pt, shorten >=16pt, no head, from=2, to=4]
\end{tikzcd}\]
If $u$ is marked, the $3$-cell is an equivalence. We can continue these diagrams in higher dimensions and we have
similar assertions for lax limits.
The marking therefore allows us to play on the "lax character" of the universal property that the lax colimit must verify.

After providing several characterizations of lax colimits and limits, we prove the following result:
\begin{itheorem}[\ref{theo:presheaevs colimi of representable}]
Let $C$ be a $\U$-small $\io$-category. Let $f$ be an object of $\w{C}$. We define $C^\sharp_{/f}$ as the following pullback
\[\begin{tikzcd}
	{C^\sharp_{/f}} & {\w{C}^\sharp_{/f}} \\
	{C^\sharp} & {\w{C}^\sharp}
	\arrow[from=1-1, to=2-1]
	\arrow[from=1-1, to=1-2]
	\arrow[from=1-2, to=2-2]
	\arrow["{y^\sharp}"', from=2-1, to=2-2]
\end{tikzcd}\]
The colimit of the functor 
$\pi:C^\sharp_{/f}\to C^\sharp\xrightarrow{y^\sharp} \w{C}^\sharp$ is $f$.
\end{itheorem}

We conclude this chapter by studying Kan extensions.

\paragraph{Cardinality hypothesis.}
We fix during this chapter three Grothendieck universes $\U \in \V\in\Wcard$, such that $\omega\in \U$. 
All defined notions depend on a choice of cardinality. When nothing is specified, this corresponds to the implicit choice of the cardinality $\V$.
We denote by $\Set$ the $\Wcard$-small $1$-category of $\V$-small sets, $\igrd$ the $\Wcard$-small $\iun$-category of $\V$-small $\infty$-groupoids and $\icat$ the $\Wcard$-small $\iun$-category of $\V$-small $\iun$-categories. 

\section{Univalence}
\label{section:Univalence}
\subsection{Internal category}
\p For $X$ an object of $\iPsh{\Theta}$ and $K$ a simplicial $\infty$-groupoid, we define the simplicial object $\langle X, K\rangle$ of $\ocat$ whose value on $n$ is given by \index[notation]{((g20@$\langle a,n\rangle$}
$$\langle X,K\rangle_n := X\times K_n$$
If $K$ is the representable $[n]$, this object is simply denoted by $\langle X,n\rangle$.
We also define the following set of morphism of $\iPsh{\Delta\times \Theta}$:\sym{(t@$\T$}
$$\T:= \{\langle a,f\rangle,~ a\in \Theta, f\in \mbox{$\W_1$}\} \cup \{\langle g,n\rangle,~ g\in \W, [n]\in \Delta\}$$

\p A \wcnotion{$\ioun$-category}{category5@$\ioun$-category} is a $\T$-local $\infty$-presheaf $C\in \iPsh{\Theta\times \Delta}$. We then naturally define \sym{((a80@$\ouncat$}
$$\ouncat := \iPsh{\Theta\times \Delta}_{\T}.$$
Unfolding the definition, an $\ioun$-category is a simplicial object $C:\Delta^{op}\to \ocat$
such that the induced morphisms
$$C_0\to\lim_{[k]\to E^{eq}}C_k~~~~\mbox{ and }~~~C_n\to C_1\times_{C_0}\times...\times_{C_0}C_1~n\in \Nb$$
are equivalences. 
Remark that we have a cartesian square
\[\begin{tikzcd}
	\ouncat & {\Fun(\Theta^{op},\icat)} \\
	{\ocat\times \ocat} & {\Fun(\Theta^{op},\igrd)\times \Fun(\Theta^{op},\igrd)}
	\arrow[from=1-2, to=2-2]
	\arrow[""{name=0, anchor=center, inner sep=0}, from=2-1, to=2-2]
	\arrow[from=1-1, to=2-1]
	\arrow[from=1-1, to=1-2]
	\arrow["\lrcorner"{anchor=center, pos=0.125}, draw=none, from=1-1, to=0]
\end{tikzcd}\]
where the lower horizontal morphism is induced by the canonical inclusion of $\io$-category onto $\infty$-presheaves on $\Theta$, and the right vertical one is induced by the functor that maps an $\iun$-category to the pair consisting of the $\infty$-groupoid of objects and the $\infty$-groupoid of arrows.

\p
A morphism $p:X\to A$ between two $\infty$-presheaves on $\Theta\times \Delta$ is a \notion{left fibration} if it has the unique right lifting property against the set of morphism \sym{(j@$\J$}
$$\J:=\{\langle a,\{0\}\rangle \to \langle a,n\rangle~,a\in\Theta, [n]\in\Delta\}\cup \{\langle g,0\rangle,~ g\in\W\}$$
Unfolding the notation, this is equivalent to asking that $X_0\to A_0$ is $\W$-local, and that the natural square 
\[\begin{tikzcd}
	{X_n} & {X_{\{0\}}} \\
	{A_n} & {A_{\{0\}}}
	\arrow[from=1-1, to=1-2]
	\arrow[from=1-1, to=2-1]
	\arrow[from=2-1, to=2-2]
	\arrow[from=1-2, to=2-2]
\end{tikzcd}\]
is cartesian. 
\begin{prop}
\label{prop:if left fib the fib}
We have an inclusion $T\subset \widehat{J}$.
\end{prop}
\begin{proof}
Let $a$ be an object of $\Theta$.
The  $\infty$-groupoid of morphisms $i$ of $\iPsh{\Delta}$ such that $\langle a,i\rangle$ is in $\widehat{J}$ contains by definition $\{0\}\to [n]$, and is closed by colimits and left cancelation. This $\infty$-groupoid then contains all initial morphism between $\infty$-presheaves on $\Delta$. As morphisms of $\W_1$ are initial, $\widehat{J}$ includes morphisms of shape $\langle a, f\rangle$ for  $a\in \Theta$ and $f\in \W_1$.

Let $g:a\to b$ be a morphism of $\W$ and $n$ an integer. We have a commutative square
\[\begin{tikzcd}
	{\langle a,\{0\}\rangle} & {\langle a,n\rangle} \\
	{\langle b,\{0\}\rangle} & {\langle b,n\rangle}
	\arrow["{\langle g,\{0\}\rangle}"', from=1-1, to=2-1]
	\arrow["{\langle g,n\rangle}", from=1-2, to=2-2]
	\arrow[from=1-1, to=1-2]
	\arrow[from=2-1, to=2-2]
\end{tikzcd}\]
The two horizontal morphisms are in $\widehat{J}$. By left cancellation, this implies that  $\langle g,n\rangle$ is in $ \widehat{J}$ which concludes the proof.
\end{proof}
If $X\to A$ is a left fibration, with $A$ a $\ioun$-category, the last proposition implies that $X$ is also a $\ioun$-category. We denote by \wcnotation{$\Lfib(A)$}{(lfib@$\Lfib(\uvar)$} the full sub $\iun$-category of $\ouncat_{/A}$ whose objects are left fibrations.

\begin{prop}
\label{prop:lfib and W}
There is a canonical equivalence: 
$$\Lfib(\langle a,C \rangle)\sim \Fun(C,\ocat_{/a})$$
natural in $a:\Theta^{op}$ and $C:\icat^{op}$.
\end{prop}
\begin{proof}
Let $a$ be an object of $\Theta^{op}$ and $C$ an $\iun$-category. We have a canonical equivalence 
$$\iPsh{\Theta\times \Delta}_{/\langle a , C\rangle}\sim \iPsh{\Theta_{/a}\times \Delta_{/C}}\sim \Fun(\Theta_{/a}^{op},\iPsh{\Delta}_{/C})$$
The previous equivalence induces an equivalence
$$(\iPsh{\Theta\times \Delta}_{/\langle a,C\rangle})_{\{\langle b,\{0\}\rangle \to \langle b,[n]\rangle\}_{/\langle a , C\rangle}} \sim \Fun(\Theta_{/a}^{op}, (\iPsh{\Delta}_{/C})_{\I^0_{/C}})$$
where $\I^0_{/C}$ corresponds to the $\infty$-groupoid of morphisms of $\iPsh{\Delta}_{/C}$ of shape
\[\begin{tikzcd}
	& {[n]} \\
	{\{0\}} & C
	\arrow[from=2-1, to=2-2]
	\arrow[from=2-1, to=1-2]
	\arrow[from=1-2, to=2-2]
\end{tikzcd}\]
for $n$ any integer.
The $\iun$-category $(\iPsh{\Delta}_{/C})_{\I^0_{/C}}$ is equivalent to the $\iun$-category of Grothendieck $\V$-small opfibrations fibered in $\infty$-groupoid over $C$, which is itself equivalent to $\Fun(C,\igrd)$ according to the Grothendieck construction. 
We then have an equivalence 
\begin{equation}
\label{eq:lfib and W}
(\iPsh{\Theta\times \Delta}_{/\langle a,C\rangle})_{\{\langle b,\{0\}\rangle \to \langle b,[n]\rangle\}_{/\langle a , C\rangle}} \sim \Fun (\Theta_{/a}^{op}, \Fun(C,\igrd))\sim \Fun (C, \iPsh{\Theta}_{/a})
\end{equation}
By definition, $\Lfib(\langle a,C\rangle)$ is the fully faithful sub $\iun$-category of the left hand $\iun$-category corresponding to objects that are local with respect to the image of set of morphism
 $\{\langle g,0\rangle, g\in \W\}_{/\langle a,C\rangle}$ by the localization functor 
 $$(\iPsh{\Theta\times \Delta}_{/\langle a,C\rangle})\to (\iPsh{\Theta\times \Delta}_{/\langle a,C\rangle})_{\{\langle b,\{0\}\rangle \to \langle b,[n]\rangle\}_{/\langle a , C\rangle}}.$$
Such $\infty$-presheaves corresponds via the equivalence \eqref{eq:lfib and W} to functors $C\to \iPsh{\Theta}_{/a}$ that are pointwise $\W_{/a}$-local. As $\W_{/a}$-local $\infty$-presheaves on $\Theta_{/a}$ corresponds to elements of $\ocat_{/a}$, we have an equivalence
$$\Lfib(\langle a,C\rangle)\sim \Fun(C,\ocat_{/a}).$$
\end{proof}

\p
A morphism $f:A\to B$ between two $\infty$-presheaves on $\Theta\times \Delta$ induces an adjunction
\begin{equation}
\label{eq:adj between left fibration}
\begin{tikzcd}
	{f_!:\ouncat{/A}} & {\ouncat_{/B}:f^*}
	\arrow[shift left=2, from=1-1, to=1-2]
	\arrow[shift left=2, from=1-2, to=1-1]
\end{tikzcd}
\end{equation}
where $f_!$ is the composition and $f^*$ is the pullback.
As $\Lfib(A)$ is the localization of $\ouncat_{/A}$ along the class of morphisms $\widehat{\J_{/A}}$,
the previous adjunction induces a derived adjunction:
\begin{equation}
\label{eq:derived adj between left fibration}
\begin{tikzcd}
	{\Lb f_!:\Lfib(A)} & {\Lfib(B):\Rb f^*}
	\arrow[shift left=2, from=1-1, to=1-2]
	\arrow[shift left=2, from=1-2, to=1-1]
\end{tikzcd}
\end{equation}
where $\Lb f_!$ sends $E$ onto $\Fb f_!E$ and $\Rb f^*$ is just the restriction of $f^*$ to $\Lfib(B)$.

\p
We denote by $\pi_!:\Fun(\Delta^{op},\iPsh{\Theta})\to \iPsh{\Delta[\Theta]}$ the functor induced by extention by colimits by the canonical morphism $\pi:\Delta\times \Theta\to \Delta[\Theta]$. We also define $\Noiun:\iPsh{\Delta[\Theta]}\to \Fun(\Delta^{op},\iPsh{\Theta})$ as the right adjoint of $\pi_!$. As $\pi_!$ preserves representable, \wcnotation{$\Noiun$}{(noiun@$\Noiun$} preserves colimits. Remark that the image of $T$ by $\pi_!$ is contained in $\widehat{\M}$, and $\Noiun$ induces then by restriction a functor
$$\Noiun:\ocat\to \ouncat.$$
If $C$ is an $\io$-category, $\Noiun C$ corresponds to the simplicial object in $\ocat$:
\[\begin{tikzcd}
	\cdots & {\coprod_{x_0,x_1,x_2:\tau_0C}\hom_C(x_0,x_1,x_2)} & {\coprod_{x_0,x_1:\tau_0C}\hom_C(x_0,x_1)} & {\coprod_{x_0:\tau_0C}1}
	\arrow[from=1-2, to=1-3]
	\arrow[shift left=2, from=1-3, to=1-4]
	\arrow[shift left=4, from=1-2, to=1-3]
	\arrow[shift right=2, from=1-3, to=1-2]
	\arrow[shift right=4, from=1-2, to=1-3]
	\arrow[shift left=2, from=1-3, to=1-2]
	\arrow[from=1-4, to=1-3]
	\arrow[shift right=2, from=1-3, to=1-4]
\end{tikzcd}\]
If $p:X\to \Noiun C$ is a left fibration, and $x$ an object of $C$, we will denote by $X(x)$ the fiber of $p_0:X_0\to \Noiun C$ on $x$, and $E(x)$ the canonical morphism $X(x)\to 1$. Unfolding the definitions, and using corollary \ref{cor:if codomain a groupoid, then f is exponentiable}, we then have for any integer $n$ a canonical equivalence:
$$X_n \sim \coprod_{x_0,...,x_n}X(x_0)\times \hom_C(x_0,...,x_n)$$

\begin{prop}
\label{prop:equivalence beetwen left fibration}
Let $C$ be an $\io$-category, and $E$, $F$ two objects of $\Lfib(\Noiun C)$ corresponding to morphisms $X\to \Noiun C$, $Y\to \Noiun C$. Let $\phi:E\to F$ be a morphism.
The following are equivalent:
\begin{enumerate}
\item $\phi$ is an equivalence,
\item for any object $x$ of $C$, the induced morphism $\Rb x^*\phi:\Rb x^*E\to \Rb x^*E$ is an equivalence,
\item for any object $x$ of $C$, the induced morphism $\phi(x):X(x)\to Y(x)$ is an equivalence,
\end{enumerate}
\end{prop}
\begin{proof}
The implication $(1)\Rightarrow (2)$ is direct.
The implication $(2)\Rightarrow (3)$ comes from the fact that for any object $x$ of $C$, the value on $0$ of the simplicial object $\Rb x^*E$ (resp. $\Rb x^*F$) is $X(x)\to 1$ (resp. $Y(x)\to 1$). 

Suppose now that $\phi$ fulfills the last condition. As $\Noiun C$ is $C_0\sim \coprod_{C_0}1$, we have equivalences 
$$X_0\sim \coprod_{x:C_0} X(x)~~~~~Y_0\sim \coprod_{x:C_0} Y(x).$$
The morphism $\phi_0:X_0\to Y_0$ is then an equivalence. Eventually, as $E$ and $F$ are left fibrations, we have 
$$X_n\sim X_{\{0\}}\times_{ (\Noiun C)_{\{0\}} }(\Noiun C)_n\sim Y_{\{0\}}\times_{ (\Noiun C)_{\{0\}} }(\Noiun C)_n\sim Y_n.$$
This implies $(3)\Rightarrow (1)$, which concludes the proof. 
\end{proof}

\begin{prop}
\label{prop:lfib and W 2}
There is an equivalence natural in $C:\ocatm^{op}$ 
between $\Lfib(\Noiun [C,1])$ and the $\iun$-category whose objects are arrows  of shape
$$X(0)\times C\to X(1)$$
and morphisms are natural transformations such that the induced morphism
$X(0)\times C\to Y(0)\times C$
is of shape $f\times id_C$.

For a left fibration $E$ corresponding to a morphism $X\to [C,1]$, this arrow is the one appearing in the diagram:
\[\begin{tikzcd}[sep =0.3cm]
	& {X_1} && {X_0} \\
	{X(0)^\flat\times C^\flat} && {X(1)^\flat} \\
	& {\Noiun([C,1])_1} && {\Noiun([C,1])_{\{1\}}} \\
	{(C^\flat,0,1)} && {\{1\}}
	\arrow[from=4-1, to=3-2]
	\arrow[from=3-2, to=3-4]
	\arrow[from=2-1, to=1-2]
	\arrow[from=1-2, to=1-4]
	\arrow[from=2-1, to=4-1]
	\arrow[from=1-2, to=3-2]
	\arrow[from=1-4, to=3-4]
	\arrow[from=4-1, to=4-3]
	\arrow[from=2-3, to=4-3]
	\arrow[from=4-3, to=3-4]
	\arrow[from=2-3, to=1-4]
	\arrow[from=2-1, to=2-3]
\end{tikzcd}\]
where the left and the right squares are cartesian.
\end{prop}
\begin{proof}
Left fibrations are detected on pullback along representable. The functor $\Lfib(\uvar)$ then sends colimits of $\iPsh{\Theta\times \Delta}$ to limits.
Remark that we have a cocartesian square
\[\begin{tikzcd}
	{\coprod_{k\leq1}\langle C,\{k\}\rangle} & {\langle C,1\rangle} \\
	{\coprod_{k\leq1}\langle [0],\{k\}\rangle} & {\Noiun[C,1]}
	\arrow[from=1-1, to=1-2]
	\arrow[from=1-2, to=2-2]
	\arrow[from=1-1, to=2-1]
	\arrow[from=2-1, to=2-2]
	\arrow["\lrcorner"{anchor=center, pos=0.125, rotate=180}, draw=none, from=2-2, to=1-1]
\end{tikzcd}\]
According to proposition \ref{prop:lfib and W}, and as $\Lfib(\uvar)$ send colimits to limits, $\Lfib(\Noiun [C,1])$ fits in the cartesian square
\[\begin{tikzcd}
	{\Lfib(\Noiun [C,1])} & {\Fun([1],\ocat_{/C})} \\
	{\ocat\times \ocat} & {\ocat_{/C}\times \ocat_{/C}}
	\arrow[""{name=0, anchor=center, inner sep=0}, from=2-1, to=2-2]
	\arrow[from=1-1, to=2-1]
	\arrow[from=1-2, to=2-2]
	\arrow[from=1-1, to=1-2]
	\arrow["\lrcorner"{anchor=center, pos=0.125}, draw=none, from=1-1, to=0]
\end{tikzcd}\]
Using the adjunction 
\[\begin{tikzcd}
	{\dom:\ocat_{/C}} & {\ocat:\uvar\times C}
	\arrow[""{name=0, anchor=center, inner sep=0}, shift left=2, from=1-1, to=1-2]
	\arrow[""{name=1, anchor=center, inner sep=0}, shift left=2, from=1-2, to=1-1]
	\arrow["\dashv"{anchor=center, rotate=-90}, draw=none, from=0, to=1]
\end{tikzcd}\]
the $\iun$-category $\Lfib(\Noiun [C,1])$ fits in the cartesian square
\[\begin{tikzcd}
	{\Lfib(\Noiun [C,1])} & {\Fun([1],\ocat)} \\
	{\ocat\times \ocat} & {\ocat\times \ocat}
	\arrow["{(\uvar\times C,id)}"', from=2-1, to=2-2]
	\arrow[from=1-1, to=2-1]
	\arrow[from=1-2, to=2-2]
	\arrow[from=1-1, to=1-2]
	\arrow["\lrcorner"{anchor=center, pos=0.125}, draw=none, from=1-1, to=2-2]
\end{tikzcd}\]
The first assertion then follows from the last cartesian square and the proposition \ref{prp:to show fully faithfullness3} applied to $I:=1.$
 The second is obtained by walking through the equivalences used in the proof of proposition \ref{prop:lfib and W}.
\end{proof}

\begin{prop}
\label{prop:lfib and W 3}
There is an equivalence natural in $C:\ocatm^{op}$ between $\Lfib(([C,1]\otimes[1]^\sharp)^\natural)$ and the $\iun$-category whose objects are diagrams of shape
\[\begin{tikzcd}
	{X(0,0)\times C^\natural\otimes\{0\}} & {X(0,1)\times C^\natural} \\
	& {X(0,0)\times (C\otimes[1]^\sharp)^\natural} & {X(1,1)} \\
	{X(0,0)\times C^\natural\otimes\{1\}} & {X(1,0)}
	\arrow[from=1-1, to=1-2]
	\arrow[from=1-2, to=2-3]
	\arrow[from=3-2, to=2-3]
	\arrow[from=1-1, to=2-2]
	\arrow[from=2-2, to=2-3]
	\arrow[from=3-1, to=3-2]
	\arrow[from=3-1, to=2-2]
\end{tikzcd}\]
such that $X(0,0)\times C^\natural\otimes\{0\}\to X(0,1)\times C^\natural$ is of shape $f\times id_{C^\natural}$. Morphisms are natural transformations such that the induced morphisms 
$X(0,1)\times C^\natural\to Y(0,1)\times C^\natural$ and $X(0,0)\times (C\otimes[1]^\sharp)^\natural\to Y(0,0)\times (C\otimes[1]^\sharp)^\natural$
are of shape $g\times C^\natural$ and $h\times (C\otimes[1]^\sharp)^\natural$.
\end{prop}
\begin{proof}
The equation \eqref{eq:eq for cylinder marked version} implies that $([C,1]\otimes[1]^\sharp)^\natural$ is the colimit of the diagram
\[\begin{tikzcd}
	{[1]\vee[C,1]^\natural} & {[C\otimes^\natural\{0\},1]} & {[C\otimes[1]^\sharp,1]^\natural} & {[C^\natural\otimes\{1\},1]} & {[C,1]^\natural\vee[1]}
	\arrow[from=1-4, to=1-3]
	\arrow[from=1-4, to=1-5]
	\arrow[from=1-3, to=1-2]
	\arrow[from=1-1, to=1-2]
\end{tikzcd}\]
According to proposition \ref{prop:example of a special colimit3 marked case} and lemma \ref{lemma:a otimes 1 is strict}, this colimit is special, and the  $\iun$-category $\Noiun ([C,1]\otimes[1]^\sharp)^\natural$ is then colimit, computed in $\Psh{\Theta\times\Delta}$, of the diagram
\[\begin{tikzcd}[column sep =0.2cm]
	{\langle C^\natural,1\rangle\coprod\langle C^\natural,\{2\}\rangle} & {\langle C^\natural\otimes\{0\},1\rangle} && {\langle C^\natural\otimes\{1\},1\rangle} & {\langle C^\natural,\{0\}\rangle\coprod \langle C^\natural,1\rangle} \\
	{\langle [0],1\rangle\coprod \langle [0],1\rangle} & {\langle C^\natural,2\rangle} & {\langle (C\otimes[1]^\sharp)^\natural,1\rangle} & {\langle C^\natural,2\rangle} & {\langle [0],1\rangle\coprod \langle [0],1\rangle}
	\arrow[from=1-5, to=2-4]
	\arrow[from=1-1, to=2-2]
	\arrow[from=1-2, to=2-2]
	\arrow[from=1-2, to=2-3]
	\arrow[from=1-4, to=2-3]
	\arrow[from=1-4, to=2-4]
	\arrow[from=1-1, to=2-1]
	\arrow[from=1-5, to=2-5]
\end{tikzcd}\]
We then deduce the result from the proposition \ref{prop:lfib and W} in the same way as in the previous proof.
\end{proof}

\begin{prop}
\label{prop:Lfib commue with colimit}
Let $F:I\to \ocat$ be a $\Wcard$-small diagram. The canonical functor
$$\Lfib(\Noiun \colim_IF)\to \lim_I\Lfib(\Noiun F)$$
is an equivalence, where $\colim_IF$ denotes the colimit taken in $\ocat$.
\end{prop}
\begin{proof}
Let $C$ be an object of $\iPsh{\Theta}$.
As left fibrations are detected by unique right lifting property against morphisms whose codomains are of shape $\langle a,n\rangle$, a morphism $p:X\to \Noiun C$ is a left fibration if and only if for any $i:[a,n]\to C$, $(\Noiun i)^*p$ is a left fibration. 
The functor 
$$\begin{array}{ccl}
\Psh{\Delta[\Theta]}^{op}&\to &\icat_{\Wcard}\\
X&\mapsto & \Lfib(\Noiun X)
\end{array}$$
then sends colimits to limits, where $\icat_{\Wcard}$ denotes the (huge) $\iun$-category of $\Wcard$-small $\iun$-categories. To conclude the proof, we then have to show that it sends any morphism $f\in\M$ to an equivalence. If $f$ is of shape $[g,1]$ for $g\in\W$, this directly follows from proposition \ref{prop:lfib and W 2}. Suppose now that $f$ is $[a,\Sp_n]\to [a,n]$. Remark that we have a cocartesian square:
\[\begin{tikzcd}
	{\langle a, \Sp_n\rangle} & {\Noiun ([a,\Sp_n])} \\
	{\langle a,n\rangle} & {\Noiun ([a,n])}
	\arrow[from=1-1, to=2-1]
	\arrow[from=2-1, to=2-2]
	\arrow[from=1-1, to=1-2]
	\arrow[from=1-2, to=2-2]
	\arrow["\lrcorner"{anchor=center, pos=0.125, rotate=180}, draw=none, from=2-2, to=1-1]
\end{tikzcd}\]
The morphism $\Lfib(\Noiun [a,\Sp_n])\to \Lfib(\Noiun [a,n])$ then fits in the cartesian square: 
\[\begin{tikzcd}
	{\Lfib(\Noiun [a,n])} & {\Lfib(\langle a,n\rangle)} \\
	{\Lfib(\Noiun [a,\Sp_n])} & {\Lfib(\langle a, \Sp_n\rangle)}
	\arrow[from=1-2, to=2-2]
	\arrow[from=1-1, to=1-2]
	\arrow[""{name=0, anchor=center, inner sep=0}, from=2-1, to=2-2]
	\arrow[from=1-1, to=2-1]
	\arrow["\lrcorner"{anchor=center, pos=0.125}, draw=none, from=1-1, to=0]
\end{tikzcd}\]
According to proposition \ref{prop:lfib and W}, we have equivalences
$$\Lfib(\langle a ,\Sp_n\rangle)\sim \lim_{[k]\to\Sp_n}\Fun([k],\ocat_{/a})\sim \Fun([n],\ocat_{/a})\sim \Lfib(\langle a ,n\rangle)$$
It remains the case $f:=E^{eq}\to 1$. We have equivalences $\Noiun E^{eq}\sim \langle [0],E^{eq}\rangle$ and $\Noiun 1\sim 1$ .
The proposition \ref{prop:lfib and W} induces equivalences
$$ \Lfib( \langle [0],E^{eq}\rangle) \sim \lim_{[k]\to E^{eq}}\Fun([k],\ocat)\sim \Fun(1,\ocat)$$
which concludes the proof.
\end{proof}

\p 
\label{para:defi of uni}
Let $A$ be an $\ioun$-category. An object $E:\ouncat_{/A}$ is \wcnotion{$\U$-small}{small object@$\U$-small object of $\ouncat_{/A}$} if for any morphism $i:\langle b,n\rangle\to A$, the space of morphism between $i$ and $E$ is $\U$-small. Remark that an object $F$ of $\Lfib(\Noiun A)$ corresponding to a left fibration $X\to \Noiun A$ is $\U$-small if an only if for any object $a$ of $A$, $X(a)$ is $\U$-small .
Eventually, we define $\Lfib_{\U}( A)$ as the full sub $\iun$-category of $\Lfib( A)$ whose objects correspond to $\U$-small left fibrations. In particular, $\Lfib_{\U}( A)$ is a $\V$-small $\iun$-category unlike $\Lfib( A)$ which is a $\Wcard$-small $\iun$-category. Moreover, the proposition \ref{prop:Lfib commue with colimit} implies that the functor 
$$C:\ocat\mapsto \tau_0\mbox{$\Lfib_{\U}$}(\Noiun C)$$
sends colimits to limits. We then define $\uni$ as the $\io$-category that represents this object:
\begin{equation}
\label{eq:defi of uni}
\begin{array}{rcll}
\uni:&\Theta^{op} &\to &\igrd\\
& a&\mapsto & \tau_0\mbox{$\Lfib_{\U}$}(\Noiun a)
\end{array}
\end{equation}

 We then have by definition an equivalence 
\begin{equation}
\Hom(C,\uni)\sim \tau_0 \mbox{$\Lfib_{\U}$}(\Noiun C).
\end{equation}
As the functor $\Noiun$ preserves product, for any $\io$-category $D$, 
we also have a canonical equivalence
\begin{equation}
\Hom(C,\uHom(D,\uni))\sim \tau_0(\mbox{$\Lfib_{\U}$}(\Noiun C\times \Noiun D)).
\end{equation}
Eventually, by construction, the $\infty$-groupoid of objects of $\uni$ corresponds to the $\infty$-groupoid of $\U$-small $\io$-categories, and according to proposition \ref{prop:lfib and W 2}, we have an equivalence 
\begin{equation}
\label{eq:hom of uni}
\hom_{\uni}(C,D)\sim \uHom(C,D).
\end{equation}
The $\io$-category $\uni$ seems to be a decent candidate for the $\io$-category of $\U$-small $\io$-categories.

\p \label{par:dualities fo omega}
Let $S$ be a subset of $\Nb^*$. We define the subset $\Sigma S=\{i+1,i\in S\}$. 
Remark that for any $n$, we have \ssym{((b49@$(\uvar)^S$}{for $\uni$}
$$(\Noiun C)_n^S\sim (\Noiun C^{\Sigma S})_n$$
We then set the functor 
$$(\uvar)^S:\uni\to (\uni)^{\Sigma S}$$
sending a $\U$-small left fibration $X\to \Noiun C$ to the left fibration $n\mapsto (X_n^S\to (\Noiun C^{\Sigma S})_n^S)$. These functors are called \snotion{dualities}{for $\uni$}.
In particular, we have the \snotionsym{odd duality}{((b60@$(\uvar)^{op}$}{for $\uni$} $(\uvar)^{op}:\uni\to \uni^{co}$, corresponding to the set of odd integer, the \snotionsym{even duality}{((b50@$(\uvar)^{co}$}{for $\uni$} $(\uvar)^{co}:\uni\to (\uni^{t })^{op}$, corresponding to the subset of non negative even integer, the \snotionsym{full duality}{((b80@$(\uvar)^{\circ}$}{for $\uni$} $(\uvar)^{\circ}:\uni \to \uni^{t\circ}$, corresponding to $\Nb^*$ and the \snotionsym{transposition}{((b70@$(\uvar)^t$}{for $\uni$} $(\uvar)^t:\uni \to \uni^{\Sigma t}$, corresponding to the singleton $\{1\}$. Eventually, we have equivalences
$$((\uvar)^{co})^{op}\sim (\uvar)^{\circ} \sim ((\uvar)^{op})^{co}.$$

\subsection{Grothendieck construction}
\begin{notation*}
Through this section, we will identify any marked $\io$-categories $C$ with the canonical induced morphism $C\to1$. If $f:X\to Y$ is a morphism, $f\times C$ then corresponds to the canonical morphism $X\times C\to Y$.
\end{notation*}
\p Let $A$ be an $\io$-category and $a$ an object of $A$, we denote by \wcnotation{$h_a^A$}{(h@$h_{a}^{A}$} the morphism $1\to A^\sharp$ induces by $a$. At the end of section \ref{subsection Left and right cartesian fibration}, we have remarked that the left fibrant replacement of $h_a^A$, that we denoted by \wcnotation{$\Fb h^A_a$}{(fh@$\Fb h_{a}^{A}$}, is the fibration $A^\sharp_{a/}\to A^\sharp$. Equation \eqref{eq:fiber of marked splices} induces, for any object $b$ of $A^\sharp$, a cartesian square
\begin{equation}
\label{eq:fiber of slice}
\begin{tikzcd}
	{\hom_A(a,b)^\flat} & {A^{\sharp}_{a/}} \\
	{\{b\}} & {A^{\sharp}}
	\arrow[from=2-1, to=2-2]
	\arrow[from=1-1, to=2-1]
	\arrow[from=1-1, to=1-2]
	\arrow["{\Fb h_a^A}", from=1-2, to=2-2]
\end{tikzcd}
\end{equation}
which induces a canonical morphism $h^A_b\times \hom_A(a,b)^\flat\to \Fb h^A_a$, and consequently, a morphism $\Fb h^A_b\times \hom_A(a,b)^\flat\to \Fb h^A_a$.

The case of $A:=[C,1]$ will be of particular interest. The morphism $\Fb h^{[C,1]}_{1}$ is just $ h^{[C,1]}_{1}$ and theorem \ref{theo:equivalence betwen slice and join} implies that $\Fb h^{[C,1]}_{0}$ is the canonical morphism $1\costar C^\flat\to [C,1]^\sharp$. In this last case, the square \eqref{eq:fiber of slice} corresponds to the square
\[\begin{tikzcd}
	{C^\flat} & {1\costar C^\flat} \\
	{\{1\}} & {[C,1]^\sharp}
	\arrow[from=2-1, to=2-2]
	\arrow[from=1-1, to=2-1]
	\arrow[from=1-1, to=1-2]
	\arrow["{\Fb h_0^{[C,1]}}", from=1-2, to=2-2]
\end{tikzcd}\]
induces by the one of theorem \ref{theo:formula between pullback of slice and tensor marked case}.
When nothing is specified, the morphism $C^\flat \to \Fb h_0^{[C,1]}$ will always corresponds to this square.

\p\sym{(fh@$\Fb h^C_{\cdot}$}\sym{(cpoint@$C_{\cdot/}$}
Let $C$ be an $\io$-category. We define the simplicial marked $\io$-category $C_{\cdot/}$ and the simplicial arrow of marked $\io$-categories
 $\Fb h^C_{\cdot}$ whose value on an integer $n$ is given by the following pullback 
\[\begin{tikzcd}
	{(C_{\cdot/})_n} & {(C^{\sharp})^{[n+1]^\sharp}} \\
	{(\Noiun C)_n^\flat\times C^{\sharp}} & {(C^\sharp)^{[n]^\sharp}\times (C^\sharp)^{\{n+1\}}}
	\arrow[""{name=0, anchor=center, inner sep=0}, from=2-1, to=2-2]
	\arrow[from=1-2, to=2-2]
	\arrow["{(\Fb h_{\cdot})_n}"', from=1-1, to=2-1]
	\arrow[from=1-1, to=1-2]
	\arrow["\lrcorner"{anchor=center, pos=0.125}, draw=none, from=1-1, to=0]
\end{tikzcd}\]
 and where the functoriality in $n$ is induced by the universal property of pullback.
Unfolding the definition, on all integer $n$, the canonical morphism $(C_{\cdot/})_n\to C^\sharp$ corresponds to the morphism 
$$ \coprod\limits_{x_0,...,x_n:C_0} \hom_C^\flat(x_0,...,x_n)\times \Fb h_{x_n}^C$$
and is then a left cartesian fibration according to theorem \ref{theo:left cart stable by colimit}.

\p 
\label{para:definition of integral de grot}
Let $E$ be an object of $\ouncat_{/\Noiun C}$ corresponding to an arrow $X \to\Noiun C$. The \wcnotion{Grothendieck construction}{grothendieck construction@Grothendieck construction} of $E$, is the object of $\ocatm_{/C^\sharp}$ defined by the formula
$$\int_CE:=\colim_n (X^\flat \times_{(\Noiun C)^\flat } \Fb h_{\cdot})_n.$$
As the Grothendieck construction is by definition a colimit of left cartesian fibrations, the theorem \ref{theo:left cart stable by colimit} implies that it is also a left cartesian fibration. The Grothendieck construction then defines a functor
$$\int_{C}: \ouncat_{/\Noiun C}\to \LCart(C^\sharp).$$
Unfolding the definition, if $E$ is a left fibration, $\int_C E$ is the colimit of a simplicial diagram whose value on $n$ is:
$$
\coprod\limits_{x_0,...,x_n:C_0}X(x_0)\times \hom_C^\flat(x_0,...,x_n)\times \Fb h_{x_n}^C$$

\begin{example}
\label{exe:of int}
Let $E$ be an object of $\Lfib(\Noiun [a,1])$ corresponding to a morphism $X\to \Noiun ([a,1])$. According to proposition \ref{prop:lfib and W 2}, this object corresponds to a morphism $X(0)\times a\to X(1)$. The arrow $\int_{[a,1]}E$ corresponds to the colimit of the following diagram:
\[\begin{tikzcd}
	{E(0)^\flat\times\Fb h^{[a,1]}_{0}} & {E(0)^\flat\times a^\flat} & {E(1)^\flat}
	\arrow[from=1-2, to=1-1]
	\arrow[from=1-2, to=1-3]
\end{tikzcd}\]
The domain of this arrow is then the colimit of the following diagram:
\[\begin{tikzcd}
	{X(0)^\flat\times[a,1]^\sharp_{0/}} & {X(0)^\flat\times a^\flat} & {X(1)^\flat}
	\arrow[from=1-2, to=1-1]
	\arrow[from=1-2, to=1-3]
\end{tikzcd}\]
\end{example}

\begin{lemma}
\label{lemma:intpreserces initial}
The functor $\int_C:\ouncat_{/\Noiun C}\to \LCart(C^\sharp)$ preserves colimits. Moreover, it sends morphisms of $\J$ to equivalences. 
\end{lemma}
\begin{proof}
According to corollary \ref{cor:inclusion of lcatt into the slice preserves colimits}, it is sufficient to show that the composite 
$$\ouncat_{/\Noiun C}\xrightarrow{ \int_C} \LCart(C^\sharp)\xrightarrow{\dom}\ocatm$$
preserves colimits.

To this extend, we consider the functor
$$\alpha: \iPsh{\Theta\times \Delta}_{/\Noiun C}\to \iPsh{t\Theta\times \Delta}$$
sending an object $E$ of $\Lfib(\Noiun C)$ corresponding to a morphism $X\to (\Noiun C)$ to 
$X\times_{(\Noiun C)^\flat } C_{\cdot/}$, 
and the functor 
$$\beta:\iPsh{t\Theta\times \Delta}\to \ocatm$$
that is the left Kan extension of the functor $t\Theta\times \Delta\to t\Theta\to \mPsh{\Theta}$. As $\iPsh{\Theta\times \Delta}$ is locally cartesian closed, $\alpha$ preserves colimits.
The composite $\beta\circ\alpha$ then preserves colimits. Moreover, we have a commutative diagram
\[\begin{tikzcd}
	{\iPsh{\Theta\times \Delta}_{/\Noiun C}} & \ocatm \\
	{\ouncat_{/\Noiun C}} & {\LCart(C^\sharp)}
	\arrow["\beta\circ\alpha", from=1-1, to=1-2]
	\arrow["\Fb"', from=1-1, to=2-1]
	\arrow["{\int_C}"', from=2-1, to=2-2]
	\arrow["\dom"', from=2-2, to=1-2]
\end{tikzcd}\]
According to proposition \ref{prop:if left fib the fib}, one then has to show that $\beta\circ\alpha$ sends any morphism of $\J$ to an equivalence to conclude. Indeed, it will implies that $\beta\circ \alpha$ lifts to a colimit preserving functor $$\Db(\beta\circ \alpha):\ouncat_{/\Noiun C}\to \ocatm,$$ and the previous square implies that this morphism is equivalent to $\dom \int_C$.

Suppose given two cartesian squares
\[\begin{tikzcd}
	X & {X'} & { C_{\cdot/}} \\
	{\langle a, \{0\}\rangle} & {\langle a, [n]\rangle} & {(\Noiun C)^\flat}
	\arrow["f"', from=2-1, to=2-2]
	\arrow[from=1-3, to=2-3]
	\arrow[from=1-1, to=2-1]
	\arrow[from=1-2, to=2-2]
	\arrow[from=2-2, to=2-3]
	\arrow[from=1-2, to=1-3]
	\arrow["g", from=1-1, to=1-2]
	\arrow["\lrcorner"{anchor=center, pos=0.125}, draw=none, from=1-1, to=2-2]
	\arrow["\lrcorner"{anchor=center, pos=0.125}, draw=none, from=1-2, to=2-3]
\end{tikzcd}\]
By currying, we see these objects as functors $t\Theta^{op}\to \iPsh{\Delta}$. The right vertical morphism is then pointwise a right fibration of $\iun$-categories fibered in $\infty$-groupoids, as it corresponds, for a fixed $a:t\Theta$ and $n:\Delta$, to the morphism of $\infty$-groupoid:
$$\coprod_{x_0,...,x_n:C_0}\Hom(a,\hom_C(x_0,...,x_n)^\flat)\times\Hom(a,C^\sharp_{x_n/})\to \coprod_{x_0,...,x_n:C_0}\Hom(a,\hom_C(x_0,...,x_n)^\flat).$$

As the morphism $f$ is pointwise initial, so is $g$. 
As $\beta$ sends pointwise initial morphisms to equivalence, this implies that $\beta\alpha (f):= \beta(g)$ is an equivalence. 

Suppose now given two cartesian squares
\[\begin{tikzcd}
	X & {X'} & { C_{\cdot/}} \\
	{\langle a, 0\rangle} & {\langle b, 0\rangle} & {(\Noiun C)^\flat}
	\arrow["{\langle f,0\rangle}"', from=2-1, to=2-2]
	\arrow[from=1-3, to=2-3]
	\arrow[from=1-1, to=2-1]
	\arrow[from=1-2, to=2-2]
	\arrow[from=2-2, to=2-3]
	\arrow[from=1-2, to=1-3]
	\arrow["g", from=1-1, to=1-2]
	\arrow["\lrcorner"{anchor=center, pos=0.125}, draw=none, from=1-1, to=2-2]
	\arrow["\lrcorner"{anchor=center, pos=0.125}, draw=none, from=1-2, to=2-3]
\end{tikzcd}\]
with $f\in \W$. By currying, we see these objects as functors $\Delta\to\iPsh{t\Theta}$. The right vertical morphism is then pointwise a right cartesian fibration. As the morphism $\langle f,0\rangle$ is pointwise in $\widehat{\Wm}$, so is $g$. The morphism $\colim_n g_n$ is then in $\widehat{\Wm}$ and $\beta\alpha (f):= \beta(g)$ is an equivalence.
\end{proof}

\p We will denote also by 
$$\int_C:\Lfib(\Noiun C)\to \LCart(C^\sharp)$$ 
the restriction of the Grothendieck construction. 
This will not cause any confusion as from now on we will only consider the 
Grothendieck construction of left fibration.
 The lemma \ref{lemma:intpreserces initial} then implies that this functor is colimit preserving, and it is then part of an adjunction \index[notation]{(partial@$\partial_C$}
\begin{equation}
\label{eq:underived GR constuction}
\begin{tikzcd}
	{\int_C:\Lfib(\Noiun C)} & { \LCart(C^\sharp):\partial_C}
	\arrow[""{name=0, anchor=center, inner sep=0}, shift left=2, from=1-1, to=1-2]
	\arrow[""{name=1, anchor=center, inner sep=0}, shift left=2, from=1-2, to=1-1]
	\arrow["\dashv"{anchor=center, rotate=-90}, draw=none, from=0, to=1]
\end{tikzcd}
\end{equation}

\begin{lemma}
\label{lemma:partial fiber}
Let $i:C^\sharp\to D^\sharp$ be a morphism. The natural transformation $$\partial_{C}\circ\Rb i^*\to \Rb (\Noiun{i})^*\circ\partial_D$$ is an equivalence.
\end{lemma}
\begin{proof}
As equivalences between left fibrations are detected on fibers, one can suppose that $C$ is the terminal $\io$-category. Let $c$ denote the object of $D$ corresponding to $i$.
Let $E$ be an object of $\Lfib(\Noiun1)$, corresponding to a morphism $A\to 1$. According to lemma \ref{lemma:intpreserces initial}, we then have equivalences
$$\begin{array}{rclr}
\Lb i_! \int_1 E &\sim & \Lb i_! ( A^\flat\times h_1^1)\\
&\sim & A^\flat\times \Fb h_c^D\\
	&=: &\int_D {\Noiun{i}}_!E\\
	&\sim &\int_D \Lb(\Noiun{i})_!E& (\ref{lemma:intpreserces initial})\\
\end{array}$$
The canonical morphism $\Lb i_!\circ \int_1 \to \int_D \circ \Lb{(\Noiun{i})}_!$ is then an equivalence, which implies by adjunction that $\partial_{1}\circ\Rb_i^*\to \Rb (\Noiun{i})^*\circ\partial_D$ also is.
\end{proof}
\p Let $C$ be an $\io$-category and $c$ an object of $C^\sharp$.
We define $(\Noiun C)_{/c}$ as the simplicial object in $\ocat$ whose value on $(a,n)$ fits in the cocartesian square 
\[\begin{tikzcd}
	{((\Noiun C)_{/c})_{(a,n)}} & {(\Noiun C)_{(a,n+1)}} \\
	{\{c\}} & {(\Noiun C)_{(a,\{n+1\})}}
	\arrow[from=2-1, to=2-2]
	\arrow[from=1-1, to=2-1]
	\arrow[from=1-2, to=2-2]
	\arrow[from=1-1, to=1-2]
	\arrow["\lrcorner"{anchor=center, pos=0.125, rotate=45}, draw=none, from=1-1, to=2-2]
\end{tikzcd}\]
Unfolding the definition, $(\Noiun C)_{/c}$ is the simplicial diagram whose value on $n$ is
$$\coprod_{x_0,...,x_n}\hom_C(x_0,...,x_n,c)$$
\begin{lemma}
\label{lemma:fiber of F h .}
There is an equivalence
$$((\Noiun C)_{/c})^\flat \sim c^* \Fb h_{\cdot}.$$
\end{lemma}
\begin{proof}
A morphism $\langle a,n\rangle\to (c^* \Fb h_{\cdot})^\natural$ is the data of a commutative square
\[\begin{tikzcd}
	{\coprod_{k\leq n+1} a^\flat\otimes\{k\}} & {a^\flat\otimes[n+1]^\sharp} \\
	{\coprod_{k\leq n+1} \{k\}} & {C^\sharp}
	\arrow[from=1-2, to=2-2]
	\arrow[from=2-1, to=2-2]
	\arrow[from=1-1, to=2-1]
	\arrow[from=1-1, to=1-2]
\end{tikzcd}\]
which is, according to proposition \ref{prop:crushing of Gray tensor is identitye marked case}, equivalent to a morphism
$$[a,n+1]^\sharp\to C^\sharp$$
and so to a morphism $\langle a,n\rangle\to (\Noiun C)_{c/}$. As $c^*\Fb h_{\cdot}$ has a trivial marking, this shows the desired equivalence.
\end{proof}
\begin{lemma}
\label{lemma:int fiber 1}
Let $p:X\to \Noiun C$ be a left fibration, and $c$ an object of $C$. 
The canonical morphism 
$$X(c)\to \colim_n (X\times_{\Noiun C} (\Noiun C)_{/c})_n$$
is an equivalence.
\end{lemma}
\begin{proof}
We will show a slightly stronger statement, which is that the morphism
$$X(c)\to \colim_n (X\times_{(\Noiun C)} (\Noiun C)_{/c})_n$$
is an equivalence when the colimit is taken in $\infty$-presheaves on $\Theta$.
As the colimit in presheaves commutes with evaluation, one has to show that for any globular sum $a$, the canonical morphism of $\infty$-groupoids
$$\Hom(a,X(c))\to \colim_n (\Hom(a,X_n)\times_{\Hom(a,(\Noiun C)_n)}\Hom(a, (\Noiun C)_{/c})_n)$$
is an equivalence. Remark that the simplicial $\infty$-groupoid $ \Hom(a, ((\Noiun C)_{/c})_\bullet)$ is equivalent to the simplicial $\infty$-groupoid $(\Hom(a,\Noiun C)_\bullet)_{/c}$.
If we denote also by $\Hom(a,X(c))$ the constant simplicial $\infty$-groupoid $n\mapsto \Hom(a,X(c))$, we have a cartesian square
\[\begin{tikzcd}[column sep =0.3cm]
	{\Hom(a,X(c))} & { \Hom(a,X_\bullet)\times_{\Hom(a,(\Noiun C)_\bullet)}\Hom(a, (\Noiun C)_\bullet)_{/c}} & { \Hom(a,X_\bullet)} \\
	{\{c\}} & {\Hom(a, (\Noiun C)_\bullet)_{/c}} & {\Hom(a, (\Noiun C)_\bullet)}
	\arrow[from=1-3, to=2-3]
	\arrow[from=1-1, to=1-2]
	\arrow[from=1-2, to=2-2]
	\arrow[from=1-1, to=2-1]
	\arrow[""{name=0, anchor=center, inner sep=0}, from=2-1, to=2-2]
	\arrow[""{name=1, anchor=center, inner sep=0}, from=2-2, to=2-3]
	\arrow[from=1-2, to=1-3]
	\arrow["\lrcorner"{anchor=center, pos=0.125}, draw=none, from=1-2, to=1]
	\arrow["\lrcorner"{anchor=center, pos=0.125}, draw=none, from=1-1, to=0]
\end{tikzcd}\]
Moreover, the left vertical morphism is a left fibration of $\iun$-category fibered in $\infty$-groupoid.
As pullbacks along left fibrations preserve final morphisms,
the morphism
$$\Hom(a,X(c))\to \Hom(a,X_\bullet)\times_{\Hom(a,(\Noiun C)_\bullet)}\Hom(a, (\Noiun C)_\bullet)_{/c}$$
is final. Taking the colimit, this implies the result.
\end{proof}

\begin{lemma}
\label{lemma:int fiber 2}
Let $i:C^\sharp\to D^\sharp$ be a morphism. The natural transformation 
$$\int_D\circ \Rb(\Noiun i)^*\to \Rb i^* \circ\int_C$$
is an equivalence.
\end{lemma}
\begin{proof}
 As equivalences between left cartesian fibrations are detected on fibers, one can suppose that $C$ is the terminal $\io$-category. Let $c$ denote the object of $D$ corresponding to $i$ and let $E$ be an object of $\Lfib(\Noiun C)$, corresponding to a left fibration $X\to \Noiun C$. 
 By construction, $\int_CE$ is a colimit of left cartesian fibrations. However, as proposition \ref{prop:fiber preserves colimits} states that $\Rb i^*$ commutes with colimit, we have 
$$\begin{array}{rclc}
\Rb i^*\int_CE&\sim &\colim_n X_n^\flat\times_{(\Noiun C)_n^\flat}\Rb i^*\Fb h^C_{\cdot}\\
&\sim &\colim_{n}(X\times_{\Noiun C} (\Noiun C)_{/c})^\flat_n&(\ref{lemma:fiber of F h .})
\end{array}$$
Moreover, remark that $\int_1 \Rb (\Noiun i)^* E$ is equivalent to $X(c)$, and the canonical morphism
$\int_D \Rb(\Noiun i)^*E\to \Rb i^* \int_CE$ is then the image by $(\uvar)^\flat$ of the equivalence given by lemma \ref{lemma:int fiber 1}.
\end{proof}

\begin{prop}
\label{prop: derived int and partial are natural}
The functors $\int_C$ and $\partial_C$ are natural in $C:\ocat^{op}$.
\end{prop}
\begin{proof}
We denote by $\Arr^{fib}(\ocatm)$ (resp. $\Arr^{fib}(\ouncat)$) the full sub $\iun$-category of $\Arr(\ocatm)$ (resp. $\Arr(\ouncat)$) whose objects are $\U$-small left cartesian fibrations (resp. $\U$-small left fibrations). 
We also set $\ocat\times_{\ocatm}\Arr^{fib}(\ocatm)$ and $\ocat\times_{\ouncat}\Arr^{fib}(\ouncat)$ as the pullbacks:
\[\begin{tikzcd}
	{\ocat\times_{\ocatm}\Arr^{fib}(\ocatm)} & {\Arr^{fib}(\ocatm)} \\
	\ocat & \ocatm \\
	{\ocat\times_{\ouncat}\Arr^{fib}(\ouncat)} & {\Arr^{fib}(\ouncat)} \\
	\ocat & \ouncat
	\arrow[""{name=0, anchor=center, inner sep=0}, "{(\uvar)^{\sharp}}"', from=2-1, to=2-2]
	\arrow["\codom", from=1-2, to=2-2]
	\arrow[from=1-1, to=2-1]
	\arrow[from=1-1, to=1-2]
	\arrow[""{name=1, anchor=center, inner sep=0}, "\Noiun"', from=4-1, to=4-2]
	\arrow[from=3-1, to=4-1]
	\arrow["\codom", from=3-2, to=4-2]
	\arrow[from=3-1, to=3-2]
	\arrow["\lrcorner"{anchor=center, pos=0.125}, draw=none, from=1-1, to=0]
	\arrow["\lrcorner"{anchor=center, pos=0.125}, draw=none, from=3-1, to=1]
\end{tikzcd}\]
The two left vertical morphism inherit from the right vertical morphisms of a structure of Grothendieck fibrations fibered in $\iun$-categories, where cartesian liftings are given by morphisms between arrows corresponding to cartesian squares.

As the assignation $C\mapsto \Fb h_{\cdot}^C$ can be promoted in a functor $\ocat\to \Arr(\Fun(\Delta,\ocatm))$
the functors $\int_C$ and $\partial_C$ are the restrictions of two functors $\int$ and $\partial$ fitting in commutative triangles:
\[\begin{tikzcd}
	& {\ocat\times_{\ocatm}\Arr^{fib}(\ocatm)} \\
	{\ocat\times_{\ouncat}\Arr^{fib}(\ouncat)} & \ocat \\
	& {\ocat\times_{\ouncat}\Arr^{fib}(\ouncat)} \\
	{\ocat\times_{\ocatm}\Arr^{fib}(\ocatm)} & \ocat
	\arrow[from=1-2, to=2-2]
	\arrow[from=3-2, to=4-2]
	\arrow[from=2-1, to=2-2]
	\arrow["\int", from=2-1, to=1-2]
	\arrow["\partial", from=4-1, to=3-2]
	\arrow[from=4-1, to=4-2]
\end{tikzcd}\]
Lemmas \ref{lemma:partial fiber} and \ref{lemma:int fiber 2} imply that these two functors preserve cartesian arrows, and the Grothendieck deconstruction then implies the desired result.
\end{proof}

\begin{theorem}
\label{theo:gr construction}
For any $\io$-category $C$, the adjunction 
$$\begin{tikzcd}
	{\int_C:\Lfib(\Noiun C)} & { \LCart(C^\sharp):\partial_C}
	\arrow[""{name=0, anchor=center, inner sep=0}, shift left=2, from=1-1, to=1-2]
	\arrow[""{name=1, anchor=center, inner sep=0}, shift left=2, from=1-2, to=1-1]
	\arrow["\dashv"{anchor=center, rotate=-90}, draw=none, from=0, to=1]
\end{tikzcd}$$
defined in \eqref{eq:underived GR constuction}, is an adjoint equivalence.
\end{theorem}
\begin{proof}
As equivalences between left fibrations and between left cartesian fibrations are detected on fibers, and as the two functors are natural in $C$, it is sufficient to show the result for $C$ being the terminal $\io$-category. In this case remark that $\Lfib(\Noiun1)\sim \LCart(1)$ and that both $\int_1$ and $\partial_1$ are the identities. 
\end{proof}

\begin{cor}
\label{cor:fib over a colimit2}
Let $F:I\to \ocatm$ be a $\Wcard$-small diagram. The canonical functor
$$\LCartc(\colim_IF) \to \lim_I \LCartc(F)$$
is an equivalence.
\end{cor}
\begin{proof}
This functor fits in an adjunction:
\[\begin{tikzcd}
	{\colim_I:\lim_I\LCartc(F)} & {\LCartc(\colim_I F)}
	\arrow[""{name=0, anchor=center, inner sep=0}, shift left=2, from=1-2, to=1-1]
	\arrow[""{name=1, anchor=center, inner sep=0}, shift left=2, from=1-1, to=1-2]
	\arrow["\dashv"{anchor=center, rotate=-90}, draw=none, from=1, to=0]
\end{tikzcd}\]
The corollary \ref{cor:fib over a colimit} implies that the counit of this adjunction is an equivalence. To conclude, we have to show that the right adjoint is essentially surjective.
By definition, the morphism $\tau_0\LCart(I^\sharp)\to \tau_0\LCartc(I)$ is an equivalence.
According to theorem \ref{theo:gr construction}, on the $\infty$-groupoid of objects, the right adjoint corresponds to the equivalence
$$\tau_0\Lfib(\Noiun\colim_I F^\sharp) \to \lim_I \tau_0\Lfib(\Noiun F^\sharp)$$
given in proposition \ref{prop:Lfib commue with colimit}.
\end{proof}

\begin{cor}
\label{cor:antecedant of slice}
Let $C$ be an $\io$-category and $c$ be an object of $c$. The left fibration $\partial_C \Fb h_c$ is the morphism of simplicial objects:
\[\begin{tikzcd}[column sep =0.5cm]
	\cdots & {\coprod_{x_0,x_1,x_2:C_0}\hom_C(y,x_0,x_1,x_2)} & {\coprod_{x_0,x_1:C_0}\hom_C(y,x_0,x_1)} & {\coprod_{x_0:C_0}\hom_C(y,x_0)} \\
	\cdots & {\coprod_{x_0,x_1,x_2:C_0}\hom_C(x_0,x_1,x_2)} & {\coprod_{x_0,x_1:C_0}\hom_C(x_0,x_1)} & {\coprod_{x_0:C_0}1}
	\arrow[shift right=4, from=2-2, to=2-3]
	\arrow[shift left=4, from=2-2, to=2-3]
	\arrow[from=2-2, to=2-3]
	\arrow[shift left=2, from=2-3, to=2-2]
	\arrow[shift right=2, from=2-3, to=2-2]
	\arrow[shift left=2, from=2-3, to=2-4]
	\arrow[shift right=2, from=2-3, to=2-4]
	\arrow[from=2-4, to=2-3]
	\arrow[from=1-3, to=2-3]
	\arrow[from=1-4, to=2-4]
	\arrow[shift left=2, from=1-3, to=1-4]
	\arrow[from=1-4, to=1-3]
	\arrow[shift right=2, from=1-3, to=1-4]
	\arrow[shift right=4, from=1-2, to=1-3]
	\arrow[from=1-2, to=1-3]
	\arrow[shift left=4, from=1-2, to=1-3]
	\arrow[shift right=2, from=1-3, to=1-2]
	\arrow[shift left=2, from=1-3, to=1-2]
	\arrow[from=1-2, to=2-2]
\end{tikzcd}\]
\end{cor}
\begin{proof}
We denote by $E:=X\to \Noiun C$ this left fibration.
According to theorem \ref{theo:gr construction}, we can equivalently show that the Grothendieck integral of $E$ is the morphism $C^{\sharp}_{c/}\to C$. 
Remark that we have by construction a family of cartesian squares
\[\begin{tikzcd}
	{X_n\times_{(\Noiun C)_n} (C_{\cdot/})_n} & {(C^\sharp)^{[1+n+1]^\sharp}} & {(C^\sharp)^{[1]^\sharp}} \\
	{\{c\}\times (\Noiun C)_n\times C^\sharp} & {C^\sharp \times (C^\sharp)^{[n]^\sharp}\times C^\sharp} & {C^\sharp\times C^\sharp}
	\arrow[from=1-1, to=2-1]
	\arrow[from=1-1, to=1-2]
	\arrow[""{name=0, anchor=center, inner sep=0}, from=2-1, to=2-2]
	\arrow[from=1-2, to=2-2]
	\arrow["{(C^{\sharp})^{h_n}}", from=1-2, to=1-3]
	\arrow[from=2-2, to=2-3]
	\arrow[from=1-3, to=2-3]
	\arrow["\lrcorner"{anchor=center, pos=0.125}, draw=none, from=1-2, to=2-3]
	\arrow["\lrcorner"{anchor=center, pos=0.125}, draw=none, from=1-1, to=0]
\end{tikzcd}\]
natural in $n$, where $h_n$ is the simplicial morphism preserving the extremal points. The outer square factors in two cartesian squares:
\[\begin{tikzcd}
	{X_n\times_{(\Noiun C)_n} (C_{\cdot/})_n} & {C^{\sharp}_{c/}} & {(C^\sharp)^{[1]^\sharp}} \\
	{\{c\}\times (\Noiun C)_n\times C^\sharp} & {\{c\}\times C^\sharp} & {C^\sharp\times C^\sharp}
	\arrow[from=1-1, to=2-1]
	\arrow[from=1-1, to=1-2]
	\arrow[from=2-1, to=2-2]
	\arrow[from=1-2, to=2-2]
	\arrow["\lrcorner"{anchor=center, pos=0.125}, draw=none, from=1-1, to=2-2]
	\arrow[from=2-2, to=2-3]
	\arrow[from=1-3, to=2-3]
	\arrow[from=1-2, to=1-3]
	\arrow["\lrcorner"{anchor=center, pos=0.125}, draw=none, from=1-2, to=2-3]
	\arrow["\lrcorner"{anchor=center, pos=0.125}, draw=none, from=1-1, to=2-2]
\end{tikzcd}\]
This provides a canonical morphism 
$$\int_{C} E := \colim_n ( X_n\times_{(\Noiun C)_n} (\Fb h_{\cdot})_n)\to \Fb h_c^C$$
To conclude, one has to show that it is an equivalence, and for this, to check that this is the case on fibers, where it directly follows from the naturality of the integral given in proposition \ref{prop: derived int and partial are natural}.
\end{proof}

\begin{cor}
\label{cor:explicit partial}
Let $E$ be an object of $\ocat_{/[b,1]^\sharp}$ corresponding to a morphism $p:X\to [b,1]^\sharp$. 
Consider the induced cartesian squares:
\[\begin{tikzcd}
	{ X_{0}\times b^\flat} && { X_{/1}} \\
	& { X_{0}} && X \\
	{b^\flat} && {[b,1]^\sharp_{/1}} \\
	& {\{0\}} && {[b,1]^\sharp}
	\arrow[from=1-1, to=2-2]
	\arrow["g", from=1-1, to=1-3]
	\arrow["f", from=1-3, to=2-4]
	\arrow[from=3-3, to=4-4]
	\arrow[from=4-2, to=4-4]
	\arrow[from=3-1, to=3-3]
	\arrow[from=3-1, to=4-2]
	\arrow[from=1-3, to=3-3]
	\arrow[from=2-4, to=4-4]
	\arrow[from=1-1, to=3-1]
	\arrow[from=2-2, to=2-4]
	\arrow[from=2-2, to=4-2]
\end{tikzcd}\]
The span associated to $\partial_{[b,1]}\Fb E$ via the equivalence of proposition \ref{prop:lfib and W 2} is 
\begin{equation}
\label{eq:cor:explicit parital}
\bot X_{0}\leftarrow(\bot X_{0})\times b\xrightarrow{\bot g} \bot X_{/1}.
\end{equation}
\end{cor}
\begin{proof}
We denote $\tilde{X}\to [b,1]^\sharp$ the morphism associated to $\Fb E$. As, 
As $[b,1]^\sharp_{/1}\to [b,1]^\sharp$ and $\{0\}\to [b,1]^\sharp$ are right cartesian fibrations, they are smooth, and the canonical morphisms
$$X_{/1}\to \tilde{X}_{/1}~~~~~~~\mbox{ and }~~~~~~~X_{0}\to \tilde{X}_{0}$$
are initial.
As $\bot$ sends initial morphisms to equivalences, the induced morphisms
$$\bot X_{/1}\to \bot \tilde{X}_{/1}~~~~~~~\mbox{ and }~~~~~~~\bot X_{0}\to \bot \tilde{X}_{0}$$
are equivalences. We can then suppose that $E$ corresponds to a left cartesian fibration.

As $\{1\} \to [b,1]^{\sharp}$ is a right Gray deformation retract, so is the inclusion $X_1\to X_{/1}$ according to proposition \ref{prop:left Gray transfomration stable under pullback along cartesian fibration}. The right Gray deformation retract structure induces a diagram:
\[\begin{tikzcd}
	{X_{/1}\otimes\{0\}} \\
	& {X_{/1}\otimes[1]^\sharp} & {X_{/1}} \\
	{X_{/1}\otimes\{1\}} & {X_1\otimes\{1\}}
	\arrow["id", curve={height=-18pt}, from=1-1, to=2-3]
	\arrow[from=3-1, to=2-2]
	\arrow[from=3-2, to=2-3]
	\arrow["r"', from=3-1, to=3-2]
	\arrow[from=1-1, to=2-2]
	\arrow["\phi"{description}, from=2-2, to=2-3]
\end{tikzcd}\]
By post composing with $g:X_0\otimes b^\flat \to X_{/1}$ and post composing $f:X_{/1}\to X$, we get a diagram:
\[\begin{tikzcd}
	{( X_0\times b^\flat)\otimes\{0\}} & {X_0} \\
	& {( X_0\times b^\flat)\otimes[1]^\sharp} & X \\
	{( X_0\times b^\flat)\otimes\{1\}} & {X_1\otimes\{1\}}
	\arrow[from=1-1, to=2-2]
	\arrow[from=3-1, to=2-2]
	\arrow[from=2-2, to=2-3]
	\arrow[from=1-1, to=1-2]
	\arrow["rg"', from=3-1, to=3-2]
	\arrow[from=3-2, to=2-3]
	\arrow[from=1-2, to=2-3]
\end{tikzcd}\]
Remark furthermore that the following diagram:
\[\begin{tikzcd}
	{( X_0\times b^\flat)\otimes\{0\}} & {X_0\times\{0\}} & {X_0\times \Fb h^{[b,1]}_{0/}} \\
	{( X_0\times b^\flat)\otimes[1]^\sharp} && {[b,1]^\sharp}
	\arrow["l"{description}, dashed, from=2-1, to=1-3]
	\arrow[from=1-1, to=2-1]
	\arrow[from=1-3, to=2-3]
	\arrow[from=2-1, to=2-3]
	\arrow[from=1-1, to=1-2]
	\arrow[from=1-2, to=1-3]
\end{tikzcd}\]
admits a lift $l$. Indeed, the left vertical morphism is initial, and the right vertical one is a left cartesian fibration. 
All put together, we get a diagram 
\[\begin{tikzcd}
	{X_0\times b^\flat} & {X_0\times \Fb h^{[b,1]}_{0/}} \\
	{X_{1}} & E
	\arrow[from=1-2, to=2-2]
	\arrow[from=2-1, to=2-2]
	\arrow["rg"', from=1-1, to=2-1]
	\arrow[from=1-1, to=1-2]
\end{tikzcd}\]
where the upper horizontal morphism is induced by the restriction of $l$ to $(X_0\times b^\flat)\otimes\{1\}$.
As $X_1\to X_{/1}$ is initial, we have $\bot X_{/1}\sim \bot X_1$ and $\bot r$ is an equivalence. We denote by $F$ the left fibration associated to \eqref{eq:cor:explicit parital}. The previous square then corresponds to a morphism 
$$\int_{[b,1]}F\to E$$
Using the naturality of $\int_{[b,1]}$, one can see that this morphism induces an equivalence on fibers, and is then an equivalence. Applying $\partial_{[b,1]}$ and using theorem \ref{theo:gr construction}, this concludes the proof.
\end{proof}

\p A left cartesian fibration is \wcnotion{$\U$-small}{small left@$\U$-small left cartesian fibration} if its fibers are $\U$-small $\io$-categories. For an $\io$-category $A$, we denote by $\LCart_{\U}(A^\sharp)$ the full sub $\iun$-category of $\LCart(A^\sharp)$ whose objects correspond to $\U$-small left cartesian fibrations over $A^\sharp$.
\begin{cor}
\label{cor: Grt equivalence}
Let $\uni$ be the $\V$-small $\io$-category of $\U$-small $\io$-categories and $A$ a $\V$-small $\io$-category. There is an equivalence
$$\int_A:\Hom(A,\uni)\to \tau_0 \LCart_{\U}(A^\sharp)$$
natural in $A:\ocat^{op}$.
\end{cor}
\begin{proof}
This is a direct consequence of the theorem \ref{theo:gr construction} and the definition of $\uni$.
\end{proof}
\begin{cor}
\label{cor: universal fibration}
The left cartesian fibration $\int_{\uni}id$ is the universal left cartesian fibration with $\U$-small fibers, i.e for any left cartesian fibration $X\to A^\sharp$ with $\U$-small fibers, there exists a unique morphism $X\to \uni$ and a unique cartesian square:
\[\begin{tikzcd}
	X & {\dom\int_{\uni}id} \\
	{A^\sharp} & {\uni^\sharp}
	\arrow[from=1-1, to=2-1]
	\arrow[from=2-1, to=2-2]
	\arrow["{\int_{\uni}id}", from=1-2, to=2-2]
	\arrow[from=1-1, to=1-2]
	\arrow["\lrcorner"{anchor=center, pos=0.125}, draw=none, from=1-1, to=2-2]
\end{tikzcd}\]
\end{cor}
\begin{proof}
This is a direct consequence of the corollary \ref{cor: Grt equivalence} and the functoriality of the Grothendieck construction given in proposition \ref{prop: derived int and partial are natural}.
\end{proof}

\subsection{Univalence}
\begin{notation*}
Through this section, we will identify any marked $\io$-category $C$ with the canonical induced morphism $C\to1$. If $f:X\to Y$ is a morphism, $f\times C$ then corresponds to the canonical morphism $X\times C\to Y$.
\end{notation*}
\p For the remaining of this section, we fix a marked $\io$-category $I$.
Remark that $\Fb h^{[n]}_k$ corresponds to the inclusion $(d_{0}^\sharp)^k:[n-k]^\sharp\to [n]^\sharp$.
 We define the functor \index[notation]{(intt@${{\oint}_{n,I}}$}
$$\oint_{n,I}: \Fun([n],\ocatm_{/I})\to \ocatm_{/I\otimes[n]^\sharp}$$
whose value on a morphism $E:[n]\to \ocatm_{/I}$ corresponding to a sequence $E_0\to ....\to E_n$, is
$$\oint_{n,I}E:=\colim_{m} \coprod_{i_0\leq... \leq i_m\leq n} E_{i_0}\otimes \Fb h_{i_m}^{[n]}.$$
As this functor is colimit preserving, it induces an adjunction \index[notation]{(partiall@$\ringpartial_{n,I}$}
\begin{equation}
\label{eq:Gr adj lax 1}
\begin{tikzcd}
	{\oint_{n,I}:\Fun([n],\ocatm_{/I})} & {\ocatm_{/I\otimes[n]^\sharp}:\ringpartial_{n,I}}
	\arrow[""{name=0, anchor=center, inner sep=0}, shift left=2, from=1-1, to=1-2]
	\arrow[""{name=1, anchor=center, inner sep=0}, shift left=2, from=1-2, to=1-1]
	\arrow["\dashv"{anchor=center, rotate=-90}, draw=none, from=0, to=1]
\end{tikzcd}
\end{equation}

\begin{lemma}
\label{lemma:oint preserves init}
The functor $\oint_{n,I}$ sends a natural transformation that is pointwise initial to an initial morphism.
\end{lemma}
\begin{proof}
As initial morphisms are closed under colimits, we have to show that for any integer $k$, and any morphism $E\to F$ of $\ocatm_{/I}$ corresponding to a sequence $X\xrightarrow{i} Y\to I$, the induced morphism $X\otimes [n-k]^\sharp\to Y\otimes [n-k]^\sharp$ over $I\otimes[n]^\sharp$ is initial whenever $i$ is. For this, remark that there is a square
\[\begin{tikzcd}
	{X\otimes\{0\}} & {X\otimes[n-k]^\sharp} \\
	{Y\otimes\{0\}} & {Y\otimes[n-k]^\sharp}
	\arrow["i", from=1-1, to=2-1]
	\arrow[from=2-1, to=2-2]
	\arrow[from=1-1, to=1-2]
	\arrow[from=1-2, to=2-2]
\end{tikzcd}\]
where the two horizontal morphisms are initial.
By stability by composition and left cancellation of initial morphism, this implies the result.
\end{proof}

\p According to the last lemma, the adjunction \eqref{eq:Gr adj lax 1} induces a derived adjunction
\begin{equation}
\label{eq:Gr adj lax 2}
\begin{tikzcd}
	{\Lb \oint_{n,I}:\Fun([n],\LCart(I))} & {\LCart(I\otimes[n]^\sharp):\Rb \ringpartial_{n,I}}
	\arrow[""{name=0, anchor=center, inner sep=0}, shift left=2, from=1-1, to=1-2]
	\arrow[""{name=1, anchor=center, inner sep=0}, shift left=2, from=1-2, to=1-1]
	\arrow["\dashv"{anchor=center, rotate=-90}, draw=none, from=0, to=1]
\end{tikzcd}
\end{equation}
where $\Rb \ringpartial_{n,I}$ is just the restriction of $\ringpartial_{n,I}$ to $\LCart(I\otimes[n]^\sharp)$.
\begin{lemma}
\label{lemma:ringpartial fiber}
Let $i:[n]^\sharp \to [m]^\sharp$ and $j:I\to J$ be two morphisms. Let $E$ be an object of $\LCart(I\otimes [m]^\sharp)$. The natural transformation
$$\ringpartial_{n,I} (j\otimes i)^*E\to j^*\circ \ringpartial_{m,J} E\circ i^\natural $$
 is an equivalence.
\end{lemma}
\begin{proof}
As invertible natural transformations are detected pointwise, one can suppose that $n=0$, and let $k$ be the image of $[0]$ by $i$.
Let $E_0\to E_1\to.. \to E_m $ be the sequence of morphisms of $\LCart(J)$ corresponding to $\ringpartial_{m,J} E$.

The object $ j^*\circ \ringpartial_{m,J} E\circ i^\natural$ is then equivalent to $ j^* E_k$ by definition. 
As $\ringpartial_{0,I}$ is the identity, we have to show that the canonical morphism $(j\otimes \{k\})^*E\to j^*E_k$ is an equivalence. Remark that for any
 $F$ of $\ocatm_{/I}$, we have by adjunction a commutative square: 
\[\begin{tikzcd}
	{\Hom(F,(j\otimes \{k\})^*E)} & {\Hom(F,\Lb j^* E_k)} \\
	{\Hom((j\otimes \{k\})_!F,E)} & {\Hom((j_!F)\otimes h_k^{[n]},E)}
	\arrow["\sim", from=1-1, to=2-1]
	\arrow[from=2-1, to=2-2]
	\arrow[from=1-1, to=1-2]
	\arrow["\sim", from=1-2, to=2-2]
\end{tikzcd}\]
where the two vertical morphisms are equivalences. As $((j\otimes \{k\})_!F\sim (j_!F)\otimes h_k^{[n]}$, the lower morphism is an equivalence, and so is the top one. This implies the desired result.
\end{proof}

\p In the following lemmas and proposition, we focus on the case where $I$ is of the form $A^\sharp$, where everything happens more simply.
\begin{lemma}
\label{lemme:oint a sharp is natural1}
Let $j:A\to B$ be a morphism between $\io$-categories and $i:[n]\to [m]$ a morphism of $\Delta$. Let $E$ be an object of $\Fun([n],\LCart(A^\sharp))$.
The canonical morphism 
$$\Lb\oint_{n,A^\sharp}( \Rb j^*\circ E\circ i)\to \Rb(j\times i^\sharp)^* \Lb \oint_{m,B^\sharp} E$$
is an equivalence.
\end{lemma}
\begin{proof}
As equivalences in $\Fun([m],\LCart(B^\sharp))$ are detected on points, an equivalences on $\LCart(B^\sharp\times [m]^\sharp)$ are detected on fibers, we can suppose that $n=0$, $A=1$, and we denote by $k$ the image of $i$ and $a$ the image of $B$. As $\Lb\oint_{0,1}$ is the identity, one has to show that the canonical morphism 
\begin{equation}
\label{eq:equationoint a sharp is natural1}
\Rb a^* E_k\to \Rb(a\times \{k\})^* \Lb \oint_{m,B^\sharp} E
\end{equation}
 is an equivalence.

Moreover, for any $l\leq n$, the proposition \ref{prop:cotimes 1 to ctimes 1 is a trivialization} implies that the canonical morphism $\Fb(E_l\otimes \Fb h_l^{[n]})\to E_l\times \Fb h_l^{[n]}$ is an equivalence, as this two left cartesian fibrations are replacement of $E_l\otimes h_l^{[n]}\sim E_l\times h_l^{[n]}$. According to proposition \ref{prop:fiber preserves colimits}, $\Rb(a\times \{k\}^\sharp)^*$ preserves colimits, we then have 
$$ \Rb(a\times \{k\})^* \Lb \oint_{m,B^\sharp} E\sim \colim_m\coprod_{i_0\leq ...\leq i_m\leq k}\Rb a^*E_{i_0}\sim \colim_{i:[k]}\Rb a^*E_i\sim \Rb a^*E_k.$$
The morphism \eqref{eq:equationoint a sharp is natural1} is then an equivalence, which concludes the proof.
\end{proof}

\begin{prop}
\label{prop:ring partial is natural}
The functor $\Rb\ringpartial_{n,I}$ is natural in $n:\Delta^{op}$ and $I:\ocatm^{op}$. The functor $\oint_{n,A^\sharp}$ is natural in $n:\Delta^{op}$ and $A:\ocat^{op}$.
\end{prop}
\begin{proof}
The proof is similar to the one of proposition \ref{prop: derived int and partial are natural}, using lemma \ref{lemma:ringpartial fiber} and lemma \ref{lemme:oint a sharp is natural1} instead of lemma \ref{lemma:partial fiber} and lemma \ref{lemma:int fiber 2}. 
 \end{proof}

\begin{prop}
\label{prop:lax gr construction particular case}
For any $\io$-category $A$ and any integer $n$, the adjunction 
\[\begin{tikzcd}
	{\Lb \oint_{n,A^\sharp}:\Fun([n],\LCart(A^\sharp))} & {\LCart((A\times[n])^\sharp):\Rb \ringpartial_{n,A^\sharp}}
	\arrow[""{name=0, anchor=center, inner sep=0}, shift left=2, from=1-1, to=1-2]
	\arrow[""{name=1, anchor=center, inner sep=0}, shift left=2, from=1-2, to=1-1]
	\arrow["\dashv"{anchor=center, rotate=-90}, draw=none, from=0, to=1]
\end{tikzcd}\]
 is an adjoint equivalence.
\end{prop}
\begin{proof}
As in both case equivalences are detected on fibers, and as these functors are natural in $A$ and $n$, one can show the result for $A$ being the terminal $\io$-category and $n=0$. In this case remark that these two functors are the identities. 
\end{proof}

\p
We set \wcnotation{$\Fun^c([n],\LCart(I))$}{(func@$\Fun^c([\uvar],\uvar)$} as the pullback
\[\begin{tikzcd}
	{\Fun^c([n],\LCart(I))} & {\Fun([n],\LCart(I))} \\
	{\prod_{k\leq n}\LCart(I^\sharp)} & {\prod_{k\leq n}\Fun(\{k\},\LCart(I))}
	\arrow[from=1-1, to=2-1]
	\arrow[""{name=0, anchor=center, inner sep=0}, from=2-1, to=2-2]
	\arrow[from=1-2, to=2-2]
	\arrow[from=1-1, to=1-2]
	\arrow["\lrcorner"{anchor=center, pos=0.125}, draw=none, from=1-1, to=0]
\end{tikzcd}\]
where $I^\sharp$ stand for $(I^\natural)^\sharp$. An object of this $\iun$-category is then a sequence in $\LCart(I)$:
\[\begin{tikzcd}
	{F_0} & {...} & {F_ n}
	\arrow[from=1-1, to=1-2]
	\arrow[from=1-2, to=1-3]
\end{tikzcd}\]
such that for any integer $i\leq n$, $F_i$ is classified. 
A $1$-cell of this $\iun$-category is a sequence of square in $\LCart(I)$:
\[\begin{tikzcd}
	{F_0} & {...} & {F_ n} \\
	{G_0} & {...} & {G_n}
	\arrow[from=1-1, to=1-2]
	\arrow[from=1-2, to=1-3]
	\arrow[from=1-2, to=2-2]
	\arrow[from=1-1, to=2-1]
	\arrow[from=1-3, to=2-3]
	\arrow[from=2-1, to=2-2]
	\arrow[from=2-2, to=2-3]
\end{tikzcd}\]
such that for any $k\leq n$, the morphism $F_k\to G_k$ comes from a morphism beetwen the corresponding objects of $\LCart(I^\sharp)$.

\begin{prop}
\label{prop:Fun preserve colimies}
Let $F:I\to \ocatm$ be a $\Wcard$-small diagram. The canonical functor
$$\Fun^c([n],\LCart(\colim_I F))\to \lim_I\Fun^c([n],\LCart(F))$$
is an equivalence.
\end{prop}
\begin{proof}
This morphism fits in an adjunction:
\[\begin{tikzcd}
	{\colim_I:\lim_I\Fun^c([n],\LCart(F))} & {\Fun^c([n],\LCart(\colim_I F))}
	\arrow[""{name=0, anchor=center, inner sep=0}, shift left=2, from=1-2, to=1-1]
	\arrow[""{name=1, anchor=center, inner sep=0}, shift left=2, from=1-1, to=1-2]
	\arrow["\dashv"{anchor=center, rotate=-90}, draw=none, from=1, to=0]
\end{tikzcd}\]
The corollary \ref{cor:fib over a colimit} implies that the counit of this adjunction is an equivalence. To conclude, we have to show that the right adjoint is essentially surjective. On objects, this adjunction corresponds to the canonical equivalence 
$$\lim_I\Hom([n],\LCartc(F))\sim \Hom([n],\LCartc(\colim_IF))$$
induced by corollary \ref{cor:fib over a colimit2}
\end{proof}

\p As $\Rb\ringpartial_{0,I}$ is the identity, lemma \ref{lemma:ringpartial fiber} implies that the functor
$$\LCart((I\otimes[n]^\sharp)^\sharp)\to \LCart(I\otimes[n]^\sharp) \xrightarrow{\Rb\ringpartial_{n,I}} \Fun([n],\LCart(I))$$
 factors through a functor \index[notation]{(partialll@$\ringpartial_{n,I}^c$}
 \begin{equation}
 \label{eq:def of right partial classified}
\ringpartial_{n,I}^c:\LCart((I\otimes[n]^\sharp)^\sharp)\to \Fun^c([n],\LCart(I))
\end{equation}
We are now willing to show that this functor is an equivalence, and to this extent, we will construct an inverse.

\p We fix an object $a$ of $t\Theta$. We define $[a,1]^\sharp:=([a,1]^\natural)^\sharp$
 and $\iota$ the canonical inclusion $[a,1]\to [a,1]^\sharp$.
 
We directly have an equivalence 
$$\Lb\iota_!\Rb\iota^*\Fb h_{1}^{[a^\natural,1]}\sim \Fb h_{1}^{[a^\natural,1]}$$
The next lemma provides an explicit expression for $\Lb\iota_!\Rb\iota^*\Fb h_{0}^{[a^\natural,1]}$.

\begin{lemma}
\label{lemma:replacement of unmarked slice}
Let $a$ be an object of $t\Theta$. We have an equivalence
$$\Lb\iota_!\Rb\iota^*\Fb h_{0}^{[a^\natural,1]}\sim \Fb h_{0}^{[a^\natural,1]}\coprod_{a^\flat\otimes\{0\}}(a\otimes[1]^\sharp)^\flat.$$
Moreover the morphism 
$\Lb\iota_!(a^\flat \to \Fb h_{0}^{[a^\natural,1]})$ corresponds to the inclusion 
$$(a\otimes\{0\})^\flat \to (a\otimes[1]^\sharp)^\flat \to \Fb h_{0}^{[a^\natural,1]}\coprod_{a^\flat\otimes\{0\}}(a\otimes[1]^\sharp)^\flat.$$
\end{lemma}
\begin{proof}
The theorem \ref{theo:equivalence betwen slice and join} implies that 
$\iota_!\Rb\iota^*\Fb h_{0}^{[b,1]}$ and $\iota_!\Rb\iota^*\Fb h_{0}^{[(\Db_n)_t,1]}$ are respectively equivalent to 
$$(1\costar b)^\flat\to [b,1]^\sharp~~~~ \mbox{and}~~~~ (1\costar \Db_n)^{\sharp_{n+1}}\to [\Db_n,1]^\sharp$$
The theorem \ref{theo:formula between pullback of slice and tensor marked case} induces cartesian diagrams
\[\begin{tikzcd}[sep =0.1cm]
	{b^\flat\otimes\{0\}} && {(b\otimes[1])^\flat} && {\Db_n^\flat\otimes\{0\}} && {(\Db_n\otimes[1])^{\sharp_{n+1}}} \\
	& 1 && {(1\costar b)^\flat} && 1 && {(1\costar \Db_n)^{\sharp_{n+1}}} \\
	{b^\flat} && {b^\flat\star 1} && {\Db_n^\flat} && {\Db_n^\flat\star 1} \\
	& {\{0\}} && {[b,1]^\sharp} && {\{0\}} && {[\Db_n,1]^\sharp}
	\arrow[from=4-2, to=4-4]
	\arrow[from=2-4, to=4-4]
	\arrow[from=1-3, to=3-3]
	\arrow[from=1-1, to=3-1]
	\arrow[from=2-2, to=4-2]
	\arrow[from=3-1, to=4-2]
	\arrow[from=3-3, to=4-4]
	\arrow[from=1-3, to=2-4]
	\arrow[from=1-1, to=2-2]
	\arrow[from=1-1, to=1-3]
	\arrow[from=2-2, to=2-4]
	\arrow[from=3-1, to=3-3]
	\arrow[from=1-5, to=2-6]
	\arrow[from=1-5, to=1-7]
	\arrow[from=1-7, to=2-8]
	\arrow[from=3-7, to=4-8]
	\arrow[from=3-5, to=4-6]
	\arrow[from=1-5, to=3-5]
	\arrow[from=2-6, to=4-6]
	\arrow[from=1-7, to=3-7]
	\arrow[from=2-8, to=4-8]
	\arrow[from=3-5, to=3-7]
	\arrow[from=4-6, to=4-8]
	\arrow[from=2-6, to=2-8]
\end{tikzcd}\]
Remark furthermore that we have an equivalence
$$\bot(\Db_n\otimes[1])^{\sharp_{n+1}}\sim \tau^i_{n}(\Db_n\otimes[1])=: ((\Db_n)_t\otimes[1]^\sharp)^\natural.$$
Applying the full duality to theorem \ref{theo:equivalence betwen slice and join} and using the corollary \ref{cor:explicit partial}, this proves the first assertion.

The second assertion follows from the naturality in $E$ of the construction given in corollary \ref{cor:explicit partial} and from the squares
\[\begin{tikzcd}
	{b^\flat\otimes\{1\}} & {b^\flat} & {\Db_n^\flat\otimes\{1\}} & {\Db_n^\flat} \\
	{(b\otimes[1])^\flat} & {(1\costar b)^\flat} & {(\Db_n\otimes[1])^{\sharp_{n+1}}} & {(1\costar \Db_n)^{\sharp_{n+1}}} \\
	{b^\flat\star 1} & {[b,1]^\sharp} & {\Db_n^\flat\star 1} & {[\Db_n,1]^\sharp}
	\arrow[from=2-1, to=3-1]
	\arrow[from=2-1, to=2-2]
	\arrow[from=3-1, to=3-2]
	\arrow[from=2-2, to=3-2]
	\arrow[from=3-3, to=3-4]
	\arrow[from=2-4, to=3-4]
	\arrow[from=2-3, to=3-3]
	\arrow[from=2-3, to=2-4]
	\arrow[from=1-1, to=2-1]
	\arrow[from=1-1, to=1-2]
	\arrow[from=1-2, to=2-2]
	\arrow[from=1-4, to=2-4]
	\arrow[from=1-3, to=2-3]
	\arrow[from=1-3, to=1-4]
\end{tikzcd}\]
that are cartesian according to theorem \ref{theo:formula between pullback of slice and tensor marked case}.
\end{proof}

\p We fix an object $a$ of $t\Theta$. Let $E$ be an object of $\LCart([a,1]^\sharp)$. According to theorem \ref{theo:gr construction}, there exists a morphism $X(0)\times a^\natural \to X(1)$ such that $E$ corresponds to the colimit
$$X(0)^\flat\times \Fb h_0^{[a^\natural,1]}\coprod_{X(0)^\flat \times a^\flat}X(1)^\flat $$
We claim that $\Lb\iota_!\Rb \iota^*E$ is the left cartesian fibration
\begin{equation}
X(0)^\flat\times (\Fb h_0^{[a^\natural,1]}\coprod_{ a^\flat} (a\otimes[1]^\sharp)^\flat) \coprod_{X(0)^\flat\times (a\otimes\{1\})^\flat}X(1)^\flat 
\label{eq:explicit iota excalmation}
\end{equation}
Indeed, the lemma \ref{lemma:replacement of unmarked slice} provides an initial morphism from $\iota_!\Rb \iota^*E$ to this object, and the theorem \ref{theo:left cart stable by colimit} implies that this object is a left cartesian fibration.

\begin{lemma}
\label{lemma:characterisation of natural transoformation}
Let $\psi: \iota_!\Rb \iota^*\to \Lb\iota_!\Rb \iota^*$ be a natural transformation, endowed with a family of natural commutative squares:
\[\begin{tikzcd}
	{ \iota_!\Rb \iota^*(B^\flat\times E)} & {\Lb\iota_!\Rb \iota^*(B^\flat\times E)} \\
	{B^\flat \times\iota_!\Rb \iota^*E} & {B^\flat \times\iota_!\Rb \iota^*E}
	\arrow["{\psi_{B^\flat\times E}}", from=1-1, to=1-2]
	\arrow[from=1-1, to=2-1]
	\arrow["{B^\flat\times\psi_{E}}"', from=2-1, to=2-2]
	\arrow[from=1-2, to=2-2]
\end{tikzcd}\]
where we identify marked $\io$-categories with their canonical morphims to the terminal marked $\io$-category. The natural transformation $\psi$ is then the one obtained by the functorial factorization in initial morphisms followed by left cartesian fibrations. 
\end{lemma}
\begin{proof}
The natural transformation $\psi$ induces a natural transformation $\Db \psi:\Lb\iota_!\Rb \iota^*\to \Lb\iota_!\Rb \iota^*$ and we have to check that this last natural transformation is the identity. The explicit Grothendieck construction states that $E$ is a colimit of left cartesian fibration of shape $B^\flat \times \Fb h^{[a^\natural,1]}_{\epsilon}$ for $\epsilon\in \{0,1\}$. The hypothesis implies that we just have to show that $\Db \psi_{\Fb h^{[a^\natural,1]}_{0}}$ and $\Db \psi_{\Fb h^{[a^\natural,1]}_{1}}$ are equivalences, and we will check this on fibers.

Using the explicit expression of $\Lb\iota_!\Rb \iota$ given in \eqref{eq:explicit iota excalmation}, we have equivalences
$$\{0\}^*\Lb\iota_!\Rb \iota \Fb h_0^{[a^\natural,1]} \sim 1~~~~~~~~
\{0\}^*\Lb\iota_!\Rb \iota \Fb h_1^{[a^\natural,1]} \sim \emptyset~~~~~~~~
\{1\}^*\Lb\iota_!\Rb \iota \Fb h_0^{[a^\natural,1]} \sim 1$$
which directly implies that $\{0\}^*\Db \psi_{\Fb h_0^{[a^\natural,1]}}$, $\{0\}^*\Db \psi_{\Fb h_1^{[a^\natural,1]}}$ and $\{1\}^*\Db \psi_{\Fb h_1^{[a^\natural,1]}}$ are equivalences. The only case remaining is $\{1\}^*\Db \psi_{\Fb h_0^{[a^\natural,1]}}$. This morphism corresponds to an endomorphism of $(a\otimes[1]^\sharp)^\natural$, which is a strict object according to \ref{prop:tensor of glboer are strics}. By right cancellation, the morphism induced by the domain of $\Db\psi_{\Fb h_0^{[a^\natural,1]}}$ is a left cartesian fibration. There exists then a lift in the following diagram
\begin{equation}
\label{eq:square in proof of replement ofneofoeijfoepaj}
\begin{tikzcd}
	{\{0\}} & {[a,1]^{\sharp}_{0/}\coprod_{a^\flat\otimes\{0\}}(a\otimes[1]^\sharp)^\flat} \\
	{[a,1]^{\sharp}_{0/}} & {[a,1]^{\sharp}_{0/}\coprod_{a^\flat\otimes\{0\}}(a\otimes[1]^\sharp)^\flat}
	\arrow["{\dom\Db\psi_{\Fb h_0^{[a^\natural,1]}}}", from=1-2, to=2-2]
	\arrow[from=1-1, to=2-1]
	\arrow["\iota"', from=2-1, to=2-2]
	\arrow[from=1-1, to=1-2]
	\arrow["l"{description}, from=2-1, to=1-2]
\end{tikzcd}
\end{equation}
where $\iota$ is the canonical inclusion. As $l$ and $\iota$ are lifts in the following diagram:
\[\begin{tikzcd}
	{\{0\}} & {[a,1]^{\sharp}_{0/}\coprod_{a^\flat\otimes\{0\}}(a\otimes[1]^\sharp)^\flat} \\
	{[a,1]^{\sharp}_{0/}} & {[a,1]^\sharp}
	\arrow[from=1-2, to=2-2]
	\arrow[from=1-1, to=2-1]
	\arrow[from=2-1, to=2-2]
	\arrow[from=1-1, to=1-2]
	\arrow[from=2-1, to=1-2]
\end{tikzcd}\]
they are equivalent. Taking the fiber on $\{1\}$ of the cartesian square \eqref{eq:square in proof of replement ofneofoeijfoepaj}, this induces a commutative triangle:
\[\begin{tikzcd}
	{(a\otimes\{0\})^\natural} & {(a\otimes[1]^\sharp)^\natural} \\
	& {(a\otimes[1]^\sharp)^\natural}
	\arrow[from=1-1, to=1-2]
	\arrow["{\{1\}^*\Db \psi_{\Fb h_0^{[a^\natural,1]}}}", from=1-2, to=2-2]
	\arrow[from=1-1, to=2-2]
\end{tikzcd}\]
Eventually, the naturality induces a commutative squares. 
\[\begin{tikzcd}
	{(a\otimes[1]^\sharp)^\natural} & {(a\otimes\{1\})^\natural} \\
	{(a\otimes[1]^\sharp)^\natural} & {(a\otimes\{1\})^\natural}
	\arrow["{\{1\}^*\Db\psi_{\Fb h_0^{[a^\natural,1]}}}"', from=1-1, to=2-1]
	\arrow[from=1-1, to=1-2]
	\arrow[from=2-1, to=2-2]
	\arrow["{\{1\}^*\Db\psi_{a^\flat\times \Fb h_1^{[a^\natural,1]}}\sim id}", from=1-2, to=2-2]
\end{tikzcd}\]
The restriction of the morphism $\Db\psi_{\Fb h_0^{[a^\natural,1]}}$ to $a\otimes\{0\}$ and $a\otimes\{1\}$ is therefore the identity.
Using Steiner theory, we can easily show that it forces $\Db\psi_{\Fb h_0^{[a^\natural,1]}}$ to also be the identity.
\end{proof}

\p We fix an object $F$ of $\LCart([a,1]^\sharp)$, and a morphism $\phi:\Rb \iota^* E\to \Rb \iota^* F$. By adjunction, this corresponds to a morphism $\tilde{\phi}: \iota_! \Rb\iota^*E\to F$, and as $F$ corresponds to a left cartesian fibration, this induces a morphism $\Db\tilde{\phi}:\Lb\iota_! \Rb\iota^*E\to F$. Using once again theorem \ref{theo:gr construction}, this induces a morphism 
$\partial_{[a^\natural,1]} \Lb\iota_! \Rb\iota^*E\to \partial_{[a^\natural,1]} F$, that corresponds, according to the explicit expression of $\Lb\iota_!\Rb \iota$ given in \eqref{eq:explicit iota excalmation}, to a commutative square
\[\begin{tikzcd}
	{X(0)\times a^\natural} & {Y(0)\times a^\natural} \\
	{X(0)\times (a\otimes[1]^\sharp)^\natural \coprod_{X(0)\times a^\natural}X(1)} & {Y(1)}
	\arrow["{\Db \tilde{\phi}(1)}"', from=2-1, to=2-2]
	\arrow["{\Db \tilde{\phi}(0)\times a^\natural}", from=1-1, to=1-2]
	\arrow[from=1-2, to=2-2]
	\arrow[from=1-1, to=2-1]
\end{tikzcd}\]
where $Y(0)\times a^\flat \to Y(1)$ corresponds to $\partial_{[a^\natural,1]} F$.
This is equivalent to a diagram
\begin{equation}
\label{eq:lax technical big diagram}
\begin{tikzcd}
	{X(0)\times a^\natural} && {Y(0)\times a^\natural} \\
	& {X(0)\times (a\otimes [1]^\sharp)^\natural} && {Y(1)} \\
	{X(0)\times a^\natural} && {X(1)}
	\arrow[from=3-3, to=2-4]
	\arrow[from=1-3, to=2-4]
	\arrow[from=3-1, to=3-3]
	\arrow["{\Db \tilde{\phi}(0)\times a^\natural}", from=1-1, to=1-3]
	\arrow[from=1-1, to=2-2]
	\arrow[from=3-1, to=2-2]
	\arrow[from=2-2, to=2-4]
\end{tikzcd}
\end{equation}
According to proposition \ref{prop:lfib and W 3}, this corresponds to an object $\xi(\phi)$ of $\Lfib(\Noiun([a,1]\otimes[1]^\sharp)^\natural))$ endowed with two equivalences:
$$\partial_{[a^\natural,1]}E\sim \Noiun([a,1]\otimes\{0\})^*\xi(\phi)~~~~~~~
\partial_{[a^\natural,1]} F\sim \Noiun([a,1]\otimes\{1\})^*\xi(\phi)$$
Using the naturality of $\int_C$ demonstrated in proposition \ref{prop: derived int and partial are natural}, these equivalences induce equivalences:
\begin{equation}
\label{eq:fiber of xi}
E\sim ([a,1]\otimes\{0\})^*\int_{([a,1]\otimes[1]^\sharp)^\natural}\xi(\phi)~~~~~~~
 F\sim ([a,1]\otimes\{1\})^*\int_{([a,1]\otimes[1]^\sharp)^\natural}\xi(\phi)
\end{equation}
All the operations we performed were functorial and admitted inverses. 
We then have constructed an equivalence
\begin{equation}
\label{eq:inverse of ring partial}
\int_{([a,1]\otimes[1]^\sharp)^\natural}\xi:\Fun^c([1],\LCart([a,1]))\to \LCart(([a,1]\otimes[1]^\sharp)^\sharp)
\end{equation}
 
\begin{lemma}
\label{lemma:lax Gr construction technical}
There is a unique commutative square of shape
\begin{equation}
\label{eq:lemma:lax Gr construction technical}
\begin{tikzcd}
	{\iota_!\iota^*E\otimes \{0\}} & {E\otimes\{0\}} \\
	{\iota_!\iota^*E\otimes id_{[1]^\sharp}} & {\int_{([a,1]\otimes[1]^\sharp)^\natural}\xi(\phi)} \\
	{\iota_!\iota^*E\otimes\{1\}} & {F\otimes \{1\}}
	\arrow[from=1-2, to=2-2]
	\arrow[from=3-2, to=2-2]
	\arrow[from=1-1, to=2-1]
	\arrow[from=3-1, to=2-1]
	\arrow[from=3-1, to=3-2]
	\arrow[from=2-1, to=2-2]
	\arrow[from=1-1, to=1-2]
\end{tikzcd}
\end{equation}
where the upper horizontal morphism is induced by the unit of the adjunction $(\iota_!,\iota^*)$. Moreover, the bottom horizontal morphism is $\tilde\phi$.
\end{lemma}
\begin{proof}
The unicity and existence of the middle horizontal morphism come from the initiality of the morphism $\iota_!\iota^*E\otimes \{0\}\to \iota_!\iota^*E\otimes [1]^\sharp$. The unicity and existence of the lower horizontal morphism is a consequence of the equation \eqref{eq:fiber of xi}. 
As the diagram \eqref{eq:lax technical big diagram} factors as
\[\begin{tikzcd}
	& {Y(0)\times a^\natural} \\
	{X(0)\times a^\natural} & {X(0)\times a^\natural} & {Y(1)} \\
	& {X(0)\times (a\otimes [1]^\sharp)^\natural} & {X(0)\times (a\otimes [1]^\sharp)^\natural\coprod_{X(0)\times a^\natural}X(1)} \\
	{X(0)\times a^\natural} & {X(1)}
	\arrow[from=1-2, to=2-3]
	\arrow[from=4-1, to=4-2]
	\arrow[from=2-1, to=3-2]
	\arrow[from=4-1, to=3-2]
	\arrow["{\Db \tilde{\phi}(0)\times a^\natural}", from=2-1, to=2-2]
	\arrow[from=3-2, to=3-3]
	\arrow[from=4-2, to=3-3]
	\arrow[from=2-2, to=1-2]
	\arrow[from=2-2, to=3-3]
	\arrow[from=3-3, to=2-3]
\end{tikzcd}\]
the downer square of the diagram of \eqref{eq:lemma:lax Gr construction technical} factors as
\[\begin{tikzcd}
	{\iota_!\iota^*E\otimes [1]^\sharp} & {\int_{([a,1]\otimes[1]^\sharp)^\natural}\xi(\mu_E)} & {\int_{([a,1]\otimes[1]^\sharp)^\natural}\xi(\phi)} \\
	{\iota_!\iota^*E\otimes\{1\}} & {\Lb \iota_! \iota^*E\otimes\{1\}} & {F\otimes \{1\}}
	\arrow[from=2-3, to=1-3]
	\arrow[from=2-1, to=1-1]
	\arrow[from=2-1, to=2-2]
	\arrow[from=1-1, to=1-2]
	\arrow[from=2-2, to=1-2]
	\arrow[from=1-2, to=1-3]
	\arrow["{\Db \tilde{\phi}}"', from=2-2, to=2-3]
\end{tikzcd}\]
where $\mu_E$ denotes the canonical morphism $\iota_!\iota^*E\to \Lb \iota_! \iota^*E$. To conclude, one has to show that the lower left horizontal morphism is $\mu_E$. As these constructions are natural, and commute with the cartesian product with $B^\flat\to 1$ for $B$ an $\io$-category, the lemma \ref{lemma:characterisation of natural transoformation} implies the desired result.
\end{proof}

\begin{lemma}
\label{lemma:ring partial eq for a 1}
The functor $\ringpartial^c_{1,[a,1]}$ defined in \eqref{eq:def of right partial classified} in is an equivalence.
\end{lemma}
\begin{proof}
The lemma \ref{lemma:lax Gr construction technical} induces a diagram
\[\begin{tikzcd}
	{\iota^*E\otimes\Fb h^{[1]}_1} & {\iota^*E\otimes\Fb h^{[1]}_0} \\
	{\iota^*F\otimes\Fb h^{[1]}_1} & {(\iota\otimes id_{[1]})^*\int_{([a,1]\otimes[1]^\sharp)^\natural}\xi(\phi)}
	\arrow[from=2-1, to=2-2]
	\arrow[from=1-1, to=2-1]
	\arrow[from=1-1, to=1-2]
	\arrow[from=1-2, to=2-2]
\end{tikzcd}\]
which corresponds to a natural transformation 
$$\oint_{1,[a,1]} \phi \to (\iota\otimes id_{[1]})^* \int_{([a,1]\otimes[1]^\sharp)^\natural}\xi(\phi)~~~\leftrightsquigarrow~~~ \phi\to \ringpartial^c_{1,[a,1]}\int_{([a,1]\otimes[1]^\sharp)^\natural}\xi(\phi)$$
Eventually, remark that proposition \ref{prop:ring partial is natural} and the equivalences \eqref{eq:fiber of xi} imply that this natural transformation is pointwise an equivalence. 
The functor \eqref{eq:inverse of ring partial} is then a left inverse of $\ringpartial^c_{1,[a,1]}$. As it is an equivalence, so is $\ringpartial^c_{1,[a,1]}$.
\end{proof}

\begin{prop}
\label{prop:ring partial eq for I n}
For any marked $\io$-category $I$, and integer $n$, the morphism 
$$\ringpartial^c_{n,I}:\LCart((I\otimes[n]^\sharp)^\sharp) \to \Fun^c([n],\LCart(I))$$
defined in \eqref{eq:def of right partial classified} is an equivalence. 
\end{prop}
\begin{proof}
Corollary \ref{cor:fib over a colimit2}, and propositions \ref{prop:otimes marked preserves colimits} and \ref{prop:Fun preserve colimies} imply that the two functors on $\Delta^{op}\times \ocatm^{op}$:
$$\begin{array}{rcl}
(n,I)&\mapsto & \LCartc(I\otimes[n]^\sharp)\\
(n,I)&\mapsto &\Fun^c([n],\LCartc(I))
\end{array}$$
send colimits to limits. We can then reduce to the case where $I$ is an element of $t\Theta$ and $n=1$. 
If $I$ is $[1]^\sharp$, remark that $\ringpartial^c_{n,[1]^\sharp}$ is equivalent to $\ringpartial_{n,[1]^\sharp}$ which is an equivalence according to proposition \ref{prop:lax gr construction particular case}.
If $I$ is of shape $[a,1]$ for $a$ in $t\Theta$, this is the content of lemma \ref{lemma:ring partial eq for a 1}.
\end{proof}

\p We recall that a left cartesian fibration is $\U$-small if its fibers are $\U$-small $\io$-categories. For an $\io$-category $A$, we denote by $\LCart_{\U}(A^\sharp)$ the full sub $\iun$-category of $\LCart_{\U}(A^\sharp)$ whose objects correspond to $\U$-small left cartesian fibrations over $A^\sharp$. For a marked $\io$-category $I$, we define similarly $\LCartc_{\U}(I)$ as the full sub $\iun$-category of $\LCartc_{\U}(I)$ whose objects correspond to $\U$-small classified left cartesian fibrations over $I$.

\begin{cor}
\label{cor:univalence}
Let $\uni$ be the $\V$-small $\io$-category of $\U$-small $\io$-categories.
Let $n$ be an integer and $I$ be a $\V$-small marked $\io$-category. We denote by $I^\sharp$ the marked $\io$-category obtained from $I$ by marking all cells, and $\iota:I\to I^\sharp$ the induced morphism. There is an equivalence, natural in $[n]:\Delta^{op}$ and $I:\ocatm^{op}$, between functors
$$f:I\otimes[n]^\sharp\to \uni^\sharp$$
and sequences
$$\iota^*\int_{I^\natural}f_0\to ... \to \iota^*\int_{I^\natural}f_n$$
where for any $k\leq n$, $f_k$ is the functor $I^\natural\to \uni$ induced by $I\otimes\{k\}\to I\otimes[n]^\sharp\to \uni^\sharp$.
\end{cor}
\begin{proof}
This is a direct application of the equivalence 
$$\tau_0\LCart((I\otimes[n]^\sharp)^\sharp) \to \Hom([n],\LCartc(I))$$
induced by proposition \ref{prop:ring partial eq for I n}.
\end{proof}

\begin{cor}
\label{cor:univalence tranche}
Let $I$ be a $\V$-small marked $\io$-category and $c$ an object of $\uni$. We denote by $I^\sharp$ the marked $\io$-category obtained from $I$ by marking all cells, and $\iota:I\to I^\sharp$ the induced morphism. There is an equivalence, natural in $I:\ocatm^{op}$, between functors
$$f:I\to \uni^\sharp_{c/}$$
and arrows:
$$I\times \int_1c\to \iota^* \int_{I^\natural}\tilde{f}$$
where $\tilde{f}$ is the induced functor $I^\natural\to \uni_{c/}\to\uni $.
\end{cor}
\begin{proof}
By construction, we have a cocartesian square.
\[\begin{tikzcd}
	{I\otimes\{0\}} & {I\otimes[1]^\sharp} \\
	1 & {1\costar I}
	\arrow[from=1-1, to=2-1]
	\arrow[from=2-1, to=2-2]
	\arrow[from=1-1, to=1-2]
	\arrow[from=1-2, to=2-2]
	\arrow["\lrcorner"{anchor=center, pos=0.125, rotate=180}, draw=none, from=2-2, to=1-1]
\end{tikzcd}\]
As $\tau_0\LCart(\uvar)$ sends colimits to limits, this is a consequence of the last corollary.
\end{proof}

\begin{cor}
\label{cor:parametric univalence}
Let $n$ be an integer, $I$ a $\V$-small marked $\io$-category, and $A$ an $\io$-category. We denote by $I^\sharp$ the marked $\io$-category obtained from $I$ by marking all cells, and $\iota:I\to I^\sharp$ the induced morphism. There is an equivalence, natural in $[n]:\Delta^{op}$ and $I:\ocatm^{op}$, between functors
$$f:I\otimes[n]^\sharp\to \uHom(A,\uni)$$
and sequences
$$(\iota\times A^\sharp)^*\int_{I^\natural\times A}f_0\to ... \to (\iota\times A^\sharp)^*\int_{I^\natural\times A}f_n$$
where for any $k\leq n$, $f_k$ is the functor $I^\natural\times A\to \uni$ induced by $(I\otimes\{k\})\times A^\sharp\to (I\otimes[n]^\sharp)\times A^\sharp\to \uni^\sharp$.
\end{cor}
\begin{proof}
This is a direct application of the last corollary and the equivalence $(I\otimes[n]^\sharp)\times A^\sharp\sim (I\times A^\sharp)\otimes[n]^\sharp$ given in proposition \ref{prop:associativity of Gray2}.
\end{proof}

\begin{cor}
\label{cor:parametric univalence tranche}
Let $I$ be a $\V$-small marked $\io$-category, $A$ an $\io$-category, and $g$ an object of $\uHom(A,\uni)$. We denote by $I^\sharp$ the marked $\io$-category obtained from $I$ by marking all cells, and $\iota:I\to I^\sharp$ the induced morphism. There is an equivalence, natural in $I:\ocatm^{op}$, between functors
$$f:I\to \uHom(A,\uni)^\sharp_{g/}$$
and arrows:
$$I\times \int_Ag\to (\iota\times A^\sharp)^* \int_{I^\natural\times A}\tilde{f}$$
where $\tilde{f}:I^\natural\times A\to \uni$ is the functor corresponding to $I^\natural\to\uHom(A,\uni)_{g/}\to \uHom(A,\uni)$.
\end{cor}
\begin{proof}
We once again have a cocartesian square
\[\begin{tikzcd}
	{I\otimes\{0\}} & {I\otimes[1]^\sharp} \\
	1 & {1\costar I}
	\arrow[from=1-1, to=2-1]
	\arrow[from=2-1, to=2-2]
	\arrow[from=1-1, to=1-2]
	\arrow[from=1-2, to=2-2]
	\arrow["\lrcorner"{anchor=center, pos=0.125, rotate=180}, draw=none, from=2-2, to=1-1]
\end{tikzcd}\]
As $\tau_0\LCart(\uvar)$ sends colimits to limits, this is a consequence of the last corollary and the equivalence $(I\otimes[1]^\sharp)\times A^\sharp\sim (I\times A^\sharp)\otimes[1]^\sharp$ given in proposition \ref{prop:associativity of Gray2}.
\end{proof}

\subsection{$\io$-Functorial Grothendieck construction}

\p For $I$ a marked $\io$-category and $A$ an $\io$-category, we define the $\io$-category \wcnotation{$\gHom(I,A)$}{(hom@$\gHom(\uvar,\uvar)$}, whose value on a globular sum $a$, is given by 
$$\Hom(a,\gHom(I,A)):=\Hom(I\ominus a^\sharp,A^\sharp)$$

The section is devoted to the proof of the following theorem:
\begin{theorem}
\label{theo:lcartc et ghom}
Let $I$ be a $\U$-small marked $\io$-category.
Let $\uni$ be the $\V$-small $\io$-category of $\U$-small $\io$-categories, and $\uLCartc_{\U}(I)$ the $\V$-small $\io$-category of $\U$-small left cartesian fibrations. 
There is an equivalence
$$\gHom(I,\uni)\sim \uLCartc_{\U}(I)$$
natural in $I$.
On the maximal sub $\infty$-groupoid, this equivalence corresponds to the Grothendieck construction of theorem \ref{theo:gr construction}.
\end{theorem}

\begin{cor}
\label{cor:lcar et hom}
Let $A$ be a $\U$-small $\io$-category.
Let $\uLCartc_{\U}(A^\sharp)$ be the $\V$-small $\io$-category of $\U$-small left cartesian fibrations. 
There is an equivalence
$$\uHom(A,\uni)\sim \uLCart_{\U}(A^\sharp)$$
natural in $A$.
On the maximal sub $\infty$-groupoid, this equivalence corresponds to the Grothendieck construction of theorem \ref{theo:gr construction}.
\end{cor}
\begin{proof}
This is a consequence of the equivalences $\uLCart(A^\sharp)\sim \uLCartc(A^\sharp)$, of the previous theorem and of the equivalence 
$\uHom(A,\uni)\sim \gHom(A^\sharp,\uni)$ induced by the second assertion of proposition \ref{prop:associativity of ominus}.
\end{proof}

\p 
\label{par: i pull and push beetwe io category of morphism}
The previous results provide equivalences \index[notation]{(f8@$f^*:\gHom(I,\uni)\to \gHom(J,\uni)$}
$$ \gHom(I,\uni)\sim \uLCartc(I) ~~~~\mbox{and}~~~~\uHom(A,\omega)\sim \uLCart(A^\sharp)$$
By construction, for any morphism $f:I\to J$ between marked $\omega$-categories, we have a morphism 
$$f^*:\gHom(J,\uni)\to \uHom(I,\uni)$$
Suppose now that the codomain of $f$ is of shape $A^\sharp$.
 The morphism \eqref{eq:i pull} induces a morphism \index[notation]{(f7@$f_{\mbox{$\exclam$}}:\gHom(I,\uni)\to \uHom(A,\uni)$}
$$f_!:\gHom(I,\uni)\to \uHom(A,\uni)$$ and \eqref{eq:i pull unit an counit} induces natural transformations:
$$
\mu:id\to f^*f_!~~~~ \epsilon:f_!f^*\to id
$$
coming along with equivalences:
$(\epsilon\circ_0 f_!)\circ_1(f_!\circ_0 \mu) \sim id_{f_!}$ and $(f^*\circ_0 \epsilon)\circ_1 (\mu \circ_0 f^* )\sim id_{f^*}$.
When $f$ is proper, the morphism \eqref{eq:i push op} induces a morphism \index[notation]{(f9@$f_*:\gHom(I,\uni)\to \uHom(A,\uni)$}
$$f_*:\gHom(I,\uni)\to \uHom(A,\uni)$$
 and \eqref{eq:i pull unit an counit op} induces natural transformations:
 $$
\mu: id\to f_*f^*~~~~ \epsilon:f^*f_*\to id
$$
coming along with equivalences:
$(\epsilon\circ_0 f^*)\circ_1(f^*\circ_0 \mu) \sim id_{f^*}$ and $(f_*\circ_0 \epsilon)\circ_1 (\mu \circ_0 f_* )\sim id_{f_*}$.
Moreover, for every morphism $j:C\to D^\sharp$, \eqref{eq:commutative pull push} 
induces a canonical commutative square
\[\begin{tikzcd}
	{\gHom(D^\sharp\times I,\uni)} & {\uHom(D\times A ,\uni)} \\
	{\gHom(C^\sharp\times I,\uni)} & {\uHom(C\times A,\uni)}
	\arrow["{( id_{D^\sharp}\times f)_!}", from=1-1, to=1-2]
	\arrow["{(j\times id_{I})^*}"', from=1-1, to=2-1]
	\arrow["{( id_{C^\sharp}\times f)_!}"', from=2-1, to=2-2]
	\arrow["{(j\times id_{A^\sharp})^*}", from=1-2, to=2-2]
\end{tikzcd}\]
and when $f$ is proper, \eqref{eq:commutative pull push op} induces a canonical commutative square
\[\begin{tikzcd}
	{\gHom(D^\sharp\times I,\uni)} & {\uHom(D\times A ,\uni)} \\
	{\gHom(C^\sharp\times I,\uni)} & {\uHom(C\times A,\uni)}
	\arrow["{( id_{D^\sharp}\times f)_*}", from=1-1, to=1-2]
	\arrow["{(j\times id_{I})^*}"', from=1-1, to=2-1]
	\arrow["{( id_{C^\sharp}\times f)_*}"', from=2-1, to=2-2]
	\arrow["{(j\times id_{A^\sharp})^*}", from=1-2, to=2-2]
\end{tikzcd}\]

\p 	
We now turn our attention back to the proof of the theorem \ref{theo:lcartc et ghom}.
\begin{lemma}
\label{lemma:lax univalence 0}
Let $I$ be a marked $\io$-category and $b^\flat$ a globular sum. We denote by $\pi_b:I\times b^\flat\to I$ the canonical projection.
There is an equivalence of $\iun$-categories:
$$\LCart(I\times b^\flat)\sim \LCart(I)_{/\pi_b}$$
\end{lemma}
\begin{proof}
Remark first that we have an equivalence 
$$(\ocatm_{/I})_{/\pi_b}\sim \ocatm_{/I\times b}$$
Now suppose given a triangle 
\[\begin{tikzcd}
	& {I\times b^\flat} \\
	X & I
	\arrow[from=2-1, to=1-2]
	\arrow["{\pi_b}", from=1-2, to=2-2]
	\arrow[from=2-1, to=2-2]
\end{tikzcd}\]
As left cartesian fibrations are stable by composition and right cancellation, and as $\pi_b$ is a left cartesian fibration,  the diagonal morphism is a left cartesian fibration if and only if the horizontal morphism is. 

The $\iun$-categories $\LCart(I)_{/\pi_b}$ and  $\LCart(I\times b^\flat)$ then identity with the same full sub $\iun$-category of $(\ocatm_{/I})_{/\pi_b}\sim \ocatm_{/I\times b}$.
\end{proof}

\begin{lemma}
\label{lemma:lax univalence 1}
There is a family of cartesian squares
\[\begin{tikzcd}
	{\tau_0\LCart([a\times b,n]^\sharp)} & {\tau_0\LCart([a,n]^\sharp\times b^\flat)} \\
	{\prod_{k\leq n}\tau_0\LCart(\{k\})} & {\prod_{k\leq n}\tau_0\LCart(\{k\}\times b^\flat)}
	\arrow[from=2-1, to=2-2]
	\arrow[from=1-2, to=2-2]
	\arrow[from=1-1, to=2-1]
	\arrow[from=1-1, to=1-2]
\end{tikzcd}\]
natural in $a,b$ and $n$.
\end{lemma}
\begin{proof}
Remark first that the proposition
\ref{prop:crushing of Gray tensor is identitye marked case} provides cocartesian squares:
\[\begin{tikzcd}
	{\coprod_{k\leq n}(a^\flat\times b^\flat)\otimes\{k\}} & {(a^\flat\times b^\flat)\otimes[n]^\sharp} & {\coprod_{k\leq n}a^\flat\otimes\{k\}} & {a^\flat\otimes[n]^\sharp} \\
	{\coprod_{k\leq n}\{k\}} & {[a\times b,n]^\sharp} & {\coprod_{k\leq n}\{k\}} & {[a,n]^\sharp}
	\arrow[from=1-1, to=2-1]
	\arrow[from=2-1, to=2-2]
	\arrow[from=1-1, to=1-2]
	\arrow[from=1-2, to=2-2]
	\arrow["\lrcorner"{anchor=center, pos=0.125, rotate=180}, draw=none, from=2-2, to=1-1]
	\arrow[from=2-3, to=2-4]
	\arrow[from=1-4, to=2-4]
	\arrow[from=1-3, to=2-3]
	\arrow[from=1-3, to=1-4]
	\arrow["\lrcorner"{anchor=center, pos=0.125, rotate=180}, draw=none, from=2-4, to=1-3]
\end{tikzcd}\]
According to the corollary \ref{cor:fib over a colimit2}, and proposition \ref{prop:ring partial eq for I n}, and as $\Rb (\pi_{\uvar})_!:\LCart(1)\to \LCart(\uvar^\flat)$ factors through $\LCartc(\uvar^\flat)$, 
this induces cartesian squares:
\begin{equation}
\label{eq:lemma:lax univalence 2}
\begin{tikzcd}[column sep = 0.3cm]
	{\LCart([a\times b,n]^\sharp)} & {\Fun([n],\LCart(a^\flat\times b^{\flat}))} & {\LCart([a,n]^\sharp)} & {\Fun([n],\LCart( a^{\flat}))} \\
	{\prod_{k\leq n}\LCart(1)} & {\prod_{k\leq n}\LCart(a^\flat\times b^{\flat})} & {\prod_{k\leq n}\LCart(1)} & {\prod_{k\leq n}\LCart(a^\flat)}
	\arrow[from=1-3, to=2-3]
	\arrow[""{name=0, anchor=center, inner sep=0}, from=2-3, to=2-4]
	\arrow[from=1-3, to=1-4]
	\arrow[from=1-4, to=2-4]
	\arrow[from=1-1, to=1-2]
	\arrow[from=1-1, to=2-1]
	\arrow[""{name=1, anchor=center, inner sep=0}, from=2-1, to=2-2]
	\arrow[from=1-2, to=2-2]
	\arrow["\lrcorner"{anchor=center, pos=0.125}, draw=none, from=1-1, to=1]
	\arrow["\lrcorner"{anchor=center, pos=0.125}, draw=none, from=1-3, to=0]
\end{tikzcd}
\end{equation}
For a marked $\io$-category $I$, we denote $\pi_b:I\times b\to I $ the canonical projection. As the $\iun$-categorical slice and the maximal full sub $\infty$-groupoid preserve cartesian squares,
the second cartesian square induces a cartesian square
\[\begin{tikzcd}
	{\tau_0(\LCart( [a,n]^\sharp)_{/\pi_b})} & {\Hom([n],\LCart( a^{\flat})_{/\pi_b})} \\
	{\prod_{k\leq n}\tau_0(\LCart(1)_{/b^\flat})} & {\prod_{k\leq n}\tau_0(\LCart( a^{\flat})_{/\pi_b})}
	\arrow[from=1-1, to=2-1]
	\arrow[""{name=0, anchor=center, inner sep=0}, from=2-1, to=2-2]
	\arrow[from=1-1, to=1-2]
	\arrow[from=1-2, to=2-2]
	\arrow["\lrcorner"{anchor=center, pos=0.125}, draw=none, from=1-1, to=0]
\end{tikzcd}\]
and according to  lemma \ref{lemma:lax univalence 0}, this corresponds to  a cartesian square
\[\begin{tikzcd}
	{\tau_0\LCart([a,n]^\sharp\times b^\flat)} & {\Hom([n],\LCart( a^{\flat}\times b^\flat))} \\
	{\prod_{k\leq n}\tau_0\LCart(  b^\flat)} & {\prod_{k\leq n}\tau_0\LCart( a^{\flat}\times b^\flat)}
	\arrow[from=1-1, to=2-1]
	\arrow[""{name=0, anchor=center, inner sep=0}, from=2-1, to=2-2]
	\arrow[from=1-1, to=1-2]
	\arrow[from=1-2, to=2-2]
	\arrow["\lrcorner"{anchor=center, pos=0.125}, draw=none, from=1-1, to=0]
\end{tikzcd}\]
Combined with the first cartesian square of \eqref{eq:lemma:lax univalence 2}, this induces a commutative diagram
\[\begin{tikzcd}
	{\tau_0\LCart([a\times b, n]^\sharp)} & {\tau_0\LCart([a,n]^\sharp\times b^\flat)} & {\Hom([n],\LCart( a^{\flat}\times b^\flat))} \\
	{\prod_{k\leq n}\tau_0\LCart(  1)} & {\prod_{k\leq n}\tau_0\LCart(  b^\flat)} & {\prod_{k\leq n}\tau_0\LCart( a^{\flat}\times b^\flat)}
	\arrow[from=1-2, to=2-2]
	\arrow[""{name=0, anchor=center, inner sep=0}, from=2-2, to=2-3]
	\arrow[from=1-2, to=1-3]
	\arrow[from=1-3, to=2-3]
	\arrow[from=2-1, to=2-2]
	\arrow[from=1-1, to=2-1]
	\arrow[from=1-1, to=1-2]
	\arrow["\lrcorner"{anchor=center, pos=0.125}, draw=none, from=1-2, to=0]
\end{tikzcd}\]
where the right and the outer square are cartesian. By right cancellation, the left square is cartesian which concludes the proof.

\end{proof}

\begin{lemma}
\label{lemma:lax univalence 2.5}
Let $b$ be a globular sum and let  $F:I\to \ocat$ be a $\Wcard$-small diagram. 
The canonical morphism
$$\LCart(\colim_IF^\sharp\times b^\flat) \to \lim_I \LCart(F^\sharp\times b^\flat)$$
is an equivalence.
\end{lemma}
\begin{proof}
The corollary \ref{cor:fib over a colimit2} implies that the canonical morphism
$$\LCart(\colim_IF^\sharp) \to \lim_I \LCart(F^\sharp)$$
is an equivalence.
We recall that for any $\io$-category $A$, we denote by $\pi_b:A^\sharp\times b^\flat\to A^\sharp$ the canonical projection. As the $\iun$-categorical slice preserves limits, the previous equivalence induces an equivalence
$$\LCart(\colim_IF^\sharp)_{/\pi_b} \to \lim_I \LCart(F^\sharp)_{/\pi_b}.$$
The results then follows from lemma \ref{lemma:lax univalence 0}.
\end{proof}

\begin{lemma}
\label{lemma:lax univalence 2}
There is a family of cartesian squares
\[\begin{tikzcd}
	{\tau_0\LCart((I\ominus[b,n]^\sharp)^\sharp)} & {\tau_0\LCart((I\otimes[n]^\sharp)^\sharp\times b^\flat)} \\
	{\prod_{k\leq n}\tau_0\LCart(I^\sharp\otimes\{k\})} & {\prod_{k\leq n}\tau_0\LCart((I^\sharp\otimes\{k\})\times b^\flat)}
	\arrow[from=2-1, to=2-2]
	\arrow[from=1-2, to=2-2]
	\arrow[from=1-1, to=2-1]
	\arrow[from=1-1, to=1-2]
\end{tikzcd}\]
natural in $I,b$ and $n$.
\end{lemma}
\begin{proof}
By definition, $(I\ominus [b,n]^\sharp)^\sharp$ fits in the following cartesian square:
\[\begin{tikzcd}
	{\colim_{[a,m]\to \amalg_kI^\natural\otimes \{k\}}[a\times b,m]^\sharp} & {\colim_{[a,m]\to \amalg_kI^\natural\otimes \{k\}}[a,m]^\sharp} \\
	{\colim_{[a,m]\to (I\otimes[n]^\sharp)^\natural}[a\times b,m]^\sharp} & {(I\ominus [b,n]^\sharp)^\sharp}
	\arrow[from=1-1, to=2-1]
	\arrow[from=2-1, to=2-2]
	\arrow[from=1-2, to=2-2]
	\arrow[""{name=0, anchor=center, inner sep=0}, from=1-1, to=1-2]
	\arrow["\lrcorner"{anchor=center, pos=0.125, rotate=180}, draw=none, from=2-2, to=0]
\end{tikzcd}\]
Combined with corollary \ref{cor:fib over a colimit2}, this implies that the $\infty$-groupoid $\tau_0\LCart((I\ominus [b,n]^\sharp)^\sharp)$ fits in the cartesian square:
\[\begin{tikzcd}
	{\tau_0\LCartc((I\ominus [b,n]^\sharp)^\sharp)} & {\lim_{[a,m]\to \amalg_kI^\natural\otimes \{k\}}\tau_0\LCart([a,m]^\sharp)} \\
	{\lim_{[a,m]\to (I\otimes[n]^\sharp)^\natural}\tau_0\LCart([a\times b,m]^\sharp)} & {\lim_{[a,m]\to \amalg_kI^\natural\otimes \{k\}}\tau_0\LCart([a\times b,m]^\sharp)}
	\arrow[""{name=0, anchor=center, inner sep=0}, from=2-1, to=2-2]
	\arrow[from=1-1, to=2-1]
	\arrow[from=1-1, to=1-2]
	\arrow[from=1-2, to=2-2]
	\arrow["\lrcorner"{anchor=center, pos=0.125}, draw=none, from=1-1, to=0]
\end{tikzcd}\]
Applying lemma \ref{lemma:lax univalence 1}, and the fact that any morphism $\{l\}\to [a,m]\to (I\otimes[n]^\sharp)^\natural$ uniquely factors through $\coprod_{k}I^\natural \otimes\{k\}$,
we get a cartesian square
\[\begin{tikzcd}
	{\tau_0\LCartc((I\ominus [b,n]^\sharp)^\sharp)} & {\lim_{[a,m]\to \amalg_kI^\natural\otimes \{k\}}\tau_0\LCart([a,m]^\sharp)} \\
	{\lim_{[a,m]\to (I\otimes[n]^\sharp)^\natural}\tau_0\LCart([a,m]^\sharp\times b^\flat)} & {\lim_{[a,m]\to \amalg_kI^\natural\otimes \{k\}}\tau_0\LCart([a,m]^\sharp\times b^\flat)}
	\arrow[""{name=0, anchor=center, inner sep=0}, from=2-1, to=2-2]
	\arrow[from=1-1, to=2-1]
	\arrow[from=1-1, to=1-2]
	\arrow[from=1-2, to=2-2]
	\arrow["\lrcorner"{anchor=center, pos=0.125}, draw=none, from=1-1, to=0]
\end{tikzcd}\]
Eventually, the lemma \ref{lemma:lax univalence 2.5} induces equivalences
$$\begin{array}{rll}
\lim_{[a,m]\to (I\otimes[n]^\sharp)^\natural}\tau_0\LCart([a,m]^\sharp\times b^\flat)&\sim&\tau_0 \LCart((I\otimes[n]^\sharp)^\sharp\times b^\flat)\\
\lim_{[a,m]\to I^\natural}\tau_0\LCart([a,m]^\sharp\times b^\flat)&\sim&\tau_0 \LCart(I^\sharp\times b^\flat)\\
\lim_{[a,m]\to I^\natural}\tau_0\LCart([a,m]^\sharp)&\sim &\tau_0\LCart(I^\sharp)
\end{array}
$$
This concludes the proof.
\end{proof}

\begin{lemma}
\label{lemma:lax univalence 3}
There is a family of cartesian squares
\[\begin{tikzcd}
	{\Hom([n], \LCartc(I;b))} & {\tau_0\LCart((I\otimes[n]^\sharp)^\sharp\times b^\flat)} \\
	{\prod_{k\leq n}\tau_0\LCart(I^\sharp\otimes\{k\})} & {\prod_{k\leq n}\tau_0\LCart((I^\sharp\otimes\{k\})\times b^\flat)}
	\arrow[from=2-1, to=2-2]
	\arrow[from=1-2, to=2-2]
	\arrow[from=1-1, to=2-1]
	\arrow[from=1-1, to=1-2]
\end{tikzcd}\]
natural in $I,b$ and $n$.
\end{lemma}
\begin{proof}
By the construction of $\LCartc(I;b)$, we have a cartesian square 
\[\begin{tikzcd}
	{\Hom([n], \LCartc(I;b))} & {\Hom([n], \LCart(I\times b^\flat))} \\
	{\prod_{k\leq n} \tau_0\LCartc(I)} & {\prod_{k\leq n} \tau_0\LCart(I\times b^\flat)}
	\arrow[""{name=0, anchor=center, inner sep=0}, from=2-1, to=2-2]
	\arrow[from=1-2, to=2-2]
	\arrow[from=1-1, to=2-1]
	\arrow[from=1-1, to=1-2]
	\arrow["\lrcorner"{anchor=center, pos=0.125}, draw=none, from=1-1, to=0]
\end{tikzcd}\]
According to lemma \ref{lemma:lax univalence 0}, this induces a cartesian square 
\[\begin{tikzcd}
	{\Hom([n], \LCartc(I;b))} & {\Hom([n], \LCart(I)_{/\pi_b})} \\
	{\prod_{k\leq n} \tau_0\LCartc(I)} & {\prod_{k\leq n}\tau_0 (\LCart(I)_{/\pi_b})}
	\arrow[from=1-1, to=2-1]
	\arrow[from=1-1, to=1-2]
	\arrow[""{name=0, anchor=center, inner sep=0}, from=2-1, to=2-2]
	\arrow[from=1-2, to=2-2]
	\arrow["\lrcorner"{anchor=center, pos=0.125}, draw=none, from=1-1, to=0]
\end{tikzcd}\]
As the functor $\LCartc(I)\to \LCart(I)_{/\pi_b}$ factors through $\LCartc(I)_{/\pi_b}$, the proposition 
\ref{prop:ring partial eq for I n} induces a cartesian square
\[\begin{tikzcd}
	{\Hom([n], \LCartc(I;b))} & {\tau_0(\LCart((I\otimes[n]^\sharp)^\sharp)_{/\pi_b})} \\
	{\prod_{k\leq n}\tau_0\LCart(I^\sharp\otimes\{k\})} & {\prod_{k\leq n}\tau_0(\LCart(I^\sharp\otimes\{k\})_{/\pi_b})}
	\arrow[""{name=0, anchor=center, inner sep=0}, from=2-1, to=2-2]
	\arrow[from=1-2, to=2-2]
	\arrow[from=1-1, to=2-1]
	\arrow[from=1-1, to=1-2]
	\arrow["\lrcorner"{anchor=center, pos=0.125}, draw=none, from=1-1, to=0]
\end{tikzcd}\]
Eventually, a last application of lemma \ref{lemma:lax univalence 0} concludes the proof.
\end{proof}

\begin{lemma}
\label{lemma:lax univalence 4}
There is an equivalence 
$$\tau_0(\LCart((I\ominus [b,n]^\sharp)^\sharp) \sim \Hom([n],\LCartc(I;b))$$
natural in $I:\ocatm^{op}$, $b:\Theta^{op}$ and $[n]:\Delta^{op}$.
\end{lemma}
\begin{proof}
This is a direct consequence of lemmas \ref{lemma:lax univalence 2} and \ref{lemma:lax univalence 3}.
\end{proof}

\begin{proof}[Proof of theorem \ref{theo:lcartc et ghom}]
Lemma \ref{lemma:lax univalence 4} provides an natural equivalence 
$$\tau_0(\LCart((I\ominus [b,n]^\sharp)^\sharp) \sim \Hom([n],\LCartc(I;b))$$
that preserves smallness.
\end{proof}

\section{Yoneda lemma and applications}
\subsection{Yoneda lemma}
\p An $\io$-category $C$ is \wcnotion{locally $\U$-small}{locally $\U$-small $\io$-category} if for any pair of objects $x$ and $y$, $\hom_C(x,y)$ is $\U$-small. 

\begin{example}
\label{exe:Hom uni is lsm}
For all $\U$-small $\io$-category $A$, the corollary \ref{cor:lcar et hom} provides an equivalence
$$\hom_{\uHom(A,\uni)}(f,g)\sim \Map(\int_Af,\int_Ag)$$
As $\int_Af$ and $\int_Ag$ are $\U$-small left cartesian fibrations over a $\U$-small basis, their codomains are $\U$-small and $\Map(\int_Af,\int_Ag)$ is then $\U$-small. The $\io$-category $\uHom(A,\uni)$ is then locally $\U$-small.
\end{example}
We can generalize this example as follow:
\begin{prop}
\label{prop:when Hom A B is locally small}
Let $A$ be a $\U$-small $\io$-category, and $C$ is a locally $\U$-small $\io$-category. The $\io$-category $\uHom(A,C)$ is locally $\U$-small. 
\end{prop}
\begin{proof}
We have to check that for any globular sum $b$, the morphism 
$$\Hom(A\times [b,1],C)\to \Hom(A\times (\{0\}\amalg\{1\}),C)$$
has $\U$-small fibers. As $A$, seen as an $\infty$-presheaves on $\Theta$, is a $\U$-small colimit of representables, we can reduce to the case where $A\in \Theta$. As $C$ is local with respect to Segal extensions, and as the cartesian product conserves them, we can reduce to the case where $A$ is of shape $[a,1]$ for $a$ a globular sum. We now fix a morphism $f:[a,1]\times (\{0\}\amalg\{1\})\to C$. Eventually, using the canonical equivalence between $[a,1]\times [b,1]$ and the colimit of the span
$$[a,1]\vee[b,1]\leftarrow [a\times b,1]\to [b,1]\vee[a,1],$$
the $\infty$-groupoid $\Hom([a,1]\times [b,1],C)_f$ fits in a cartesian square:
\[\begin{tikzcd}
	{\Hom([a,1]\times [b,1],C)_f} & {\Hom(b,\hom(f(0,0),f(0,1)))} \\
	{\Hom(b,\hom(f(1,0),f(1,1)))} & {\Hom(a\times b,\hom(f(0,0),f(1,1)))}
	\arrow[from=2-1, to=2-2]
	\arrow[from=1-1, to=2-1]
	\arrow[from=1-1, to=1-2]
	\arrow[from=1-2, to=2-2]
\end{tikzcd}\]
As all these objects are $\U$-small by assumption, this concludes the proof.
\end{proof}

\p Let $C$ be an $\io$-category $C$. We define the simplicial object $S(\Noiun C)$ by the formula
$$S(\Noiun C)_n:= \coprod_{x_0,...,x_n:A_0} \coprod_{y_0,...,y_n:A_0}\hom_C(x_n,...,x_0,y_0,...,y_n)$$
This object comes along with a canonical projection 
\begin{equation}
\label{eq:definition of the hom0}
S(\Noiun C)\to \Noiun{C^t}\times \Noiun C.
\end{equation}
which obviously is a left fibration. As this construction if functorial, it induces a functor:
$$\begin{array}{rcl}
\ocat &\to &\Arr(\ouncat)\\
C&\mapsto & (S(\Noiun C)\to \Noiun{C^t}\times \Noiun C)
\end{array}$$

\p Through this section, we fix a locally $\U$-small $\io$-category $C$. 
The left fibration \eqref{eq:definition of the hom0} is then $\U$-small, and by definition of $\uni$, this induces a morphism
\begin{equation}
\label{eq:definition of the hom}
\hom_C(\uvar,\uvar):C^t\times C\to \uni
\end{equation}
Using the canonical equivalence
$$\Fb h^{C^t\times C}_{(x,y)}\sim \Fb h^{C^t}_{x}\times \Fb h^{C}_{y}$$
the corresponding left cartesian fibration is then the colimit of a simplicial object whose value on $n$ is given by:
$$\coprod_{x_0,...,x_n}\coprod_{y_0,...,y_n} \Fb h^{C^t}_{x_n}\times \hom_{C}(x_n,...,x_0,y_0,...,y_n)^\flat\times \Fb h^{C}_{y_n}$$
\p We define the \wcnotion{$\io$-category of $\io$-presheaves on $C$}{presheaves@$\io$-presheaves} \sym{(c@$\w{C}$}:$$\w{C}:=\uHom(C^t,\uni ).$$ This $\io$-category is locally $\U$-small according to proposition \ref{prop:when Hom A B is locally small}. The \notion{Yoneda embedding}\sym{(y@$y_{\uvar}$} $y: C\to \w{C}$ is the functor induced by the hom functor \eqref{eq:definition of the hom} by currying.

An $\io$-presheaves is \wcnotion{representable}{representable $\io$-presheaves} if it is in the image of $y$.

\p We recall that for a subset $S$ of $\Nb^*$, and an object $X$ of $\ouncat$, we denote by $X^S$ the simplicial object $n\mapsto X_n^S$. We also set $\Sigma S:=\{i+1,i\in S\}$. We then have equivalences
$$(\Noiun C)^S\sim \Noiun (C^{\Sigma C}) ~~~\mbox{and}~~~ S(\Noiun C))^S\sim S(\Noiun (C^{\Sigma C}))$$
For an object $X$ of $\ouncat$, we denote by $X_{op}$ the simplicial object $n\mapsto X_{n^{op}}$. We then have equivalences 
$$(\Noiun C)_{op}\sim \Noiun (C^{t}) ~~~\mbox{and}~~~ S(\Noiun C))_{op}\sim S(\Noiun (C^t))$$
Using the dualities defined in paragraph \ref{par:dualities fo omega}, we then have commutative diagrams
\[\begin{tikzcd}
	{(C^{t\Sigma S}\times C^{\Sigma S})^{\Sigma S}} && {\uni^{\Sigma S}} && {C\times C^t} \\
	{C^t\times C} && \uni && {C^t\times C} & \uni
	\arrow["{\hom_{C}}"', from=2-1, to=2-3]
	\arrow["{\hom^{\Sigma S}_{C^{\Sigma S}}}", from=1-1, to=1-3]
	\arrow["\sim"', from=1-1, to=2-1]
	\arrow["{(\uvar)^S}", from=1-3, to=2-3]
	\arrow["{\hom_C}"', from=2-5, to=2-6]
	\arrow["\tw"', from=1-5, to=2-5]
	\arrow["{\hom_{C^t}}", from=1-5, to=2-6]
\end{tikzcd}\]
where $\tw$ is the functor exchanging the argument. This two diagram corresponds to the natural transformations
$$\hom_{C^{\Sigma S}}(x,y)\sim \hom_{C}(x,y)^S~~~\mbox{and}~~~\hom_{C^t}(x,y)\sim \hom_{C}(y,x).$$

In combining the two previous diagrams, we get a commutative square:
\[\begin{tikzcd}
	{(C^{\circ t}\times C^{\circ})^{\circ t}} && {\uni^{{\circ t}}} \\
	{C^t\times C} && \uni
	\arrow["{\hom_{C}}"', from=2-1, to=2-3]
	\arrow["{\hom^{\circ t}_{C^\circ}}", from=1-1, to=1-3]
	\arrow["\tw"', from=1-1, to=2-1]
	\arrow["{(\uvar)^\circ}", from=1-3, to=2-3]
\end{tikzcd}\]
corresponding to the natural transformation
$$\hom_{C^\circ}(x,y)\sim \hom_{C}(y,x)^\circ.$$
\begin{prop}
\label{prop:Yoneda is Fb} 
Let $A$ be an locally $\U$-small $\io$-category.
Let $a$ be an object of $A$.
There is an equivalence
$$\int_{A}\hom_A(a,\uvar)\to \Fb h^{A}_a$$
Taking the fibers on $a$, the induced morphism $\hom_A(a,a)\to \hom_A(a,a)$ preserves the identity.
 In particular, for any object $c$ of $C$, this induces an equivalence
$$\int_{C^t}y_c\to \Fb h^{C^t}_c$$
\end{prop}
\begin{proof}
By construction, $\int_{A}\hom_A(a,\uvar)$ is the Grothendieck construction of the left fibration:
\[\begin{tikzcd}[column sep =0.4cm]
	\cdots & {\coprod_{x_0,x_1,x_2:A_0}\hom_{A}(a,x_0,x_1,x_2)} & {\coprod_{x_0,x_1:A_0}\hom_{A}(a,x_0,x_1)} & {\coprod_{x_0:A_0}\hom_{A}(a,x_0)} \\
	\cdots & {\coprod_{x_0,x_1,x_2:A_0}\hom_{A}(x_0,x_1,x_2)} & {\coprod_{x_0,x_1:A_0}\hom_{A}(x_0,x_1)} & {\coprod_{x_0:A_0}1}
	\arrow[shift right=4, from=2-2, to=2-3]
	\arrow[shift left=4, from=2-2, to=2-3]
	\arrow[from=2-2, to=2-3]
	\arrow[shift left=2, from=2-3, to=2-2]
	\arrow[shift right=2, from=2-3, to=2-2]
	\arrow[shift left=2, from=2-3, to=2-4]
	\arrow[shift right=2, from=2-3, to=2-4]
	\arrow[from=2-4, to=2-3]
	\arrow[from=1-3, to=2-3]
	\arrow[from=1-4, to=2-4]
	\arrow[shift left=2, from=1-3, to=1-4]
	\arrow[from=1-4, to=1-3]
	\arrow[shift right=2, from=1-3, to=1-4]
	\arrow[shift right=4, from=1-2, to=1-3]
	\arrow[from=1-2, to=1-3]
	\arrow[shift left=4, from=1-2, to=1-3]
	\arrow[shift right=2, from=1-3, to=1-2]
	\arrow[shift left=2, from=1-3, to=1-2]
	\arrow[from=1-2, to=2-2]
\end{tikzcd}\]
The results then follow from the corollary \ref{cor:antecedant of slice}.
\end{proof}

\p The identity $\w{C}\to \w{C}$ induces by currying a canonical morphism \sym{(ev@$\ev$}
$$\ev: C^t\times \w{C}\to \uni$$
called the \textit{evaluation functor}. Given an object $c$ of $C$ and $f$ of $\widehat{C}$, we then have $\ev(c,f)\sim f(c)$ and so 
$$(c,\{f\})^*\int_{C\times \w{C}}\ev\sim c^*\int_{C^t}f$$

Let $E$ be an object of $\ocatm_{/\w{C}^\sharp}$ corresponding to a morphism $g:X\to \w{C}^\sharp$.
We denote $\iota:X\to X^\sharp$ the canonical inclusion.
 A morphism 
 $$E\to \int_{\w{C}}\ev(c,\uvar)$$
  corresponds by adjunction to a morphism 
\begin{equation}
\label{eq:evaluation 1}
id_X\to g^*\int_{\w{C}}\ev(c,\uvar)
\end{equation}
However, we have a canonical commutative square
\[\begin{tikzcd}
	{X^\natural} & {\w{C}} \\
	{X^\natural\times C^t} & \uni
	\arrow["{\ev(c,\uvar)}", from=1-2, to=2-2]
	\arrow["{g^\natural}", from=1-1, to=1-2]
	\arrow["{\tilde{g}}"', from=2-1, to=2-2]
	\arrow["{X^\natural\times \{c\}}"', from=1-1, to=2-1]
\end{tikzcd}\]
where $\tilde{g}$ is the morphism defined by currying from $g^\natural:X^\natural\to \w{C}$. Using the naturality of the Grothendieck construction, the previous commutative square implies that the data of \eqref{eq:evaluation 1} corresponds to a morphism
$$id_X\to (\iota\times \{c\})^* \int_{X^\natural\times C^t}\tilde{g}$$
an by adjunction, to a morphism 
$$
X\times \Fb h_c^{C^t}\to (\iota\times (C^t)^\sharp)^*\int_{X^\natural\times C^t}\tilde{g}
$$
We then have constructed an equivalence 
\begin{equation}
\label{eq:evaluation 3}
\Hom(E, \int_{\w{C}}\ev(c,\uvar))\sim \Hom(X\times \Fb h_c^{C^t}, (\iota\times (C^t)^\sharp)^*\int_{X^\natural\times C^t}\tilde{g})
\end{equation}
natural in $E$.

Remark furthermore that if $E$ is $h^{\w{C}}_f$ for $f$ an object of $\w{C}$, the equivalence corresponds to the canonical equivalences
$$\begin{array}{rcl}
\Hom_{\ocatm_{/\w{C}^\sharp}}(h^{\w{C}}_f, \int_{\w{C}}\ev(c,\uvar))&\sim &\Hom_{\ocatm}(1,\{f\}^*\int_{\w{C}}\ev(c,\uvar))\\
&\sim& \Hom_{\ocatm}(1,c^*\int_{C^t}f)\\
&\sim& \Hom_{\ocatm_{/(C^t)^\sharp}}( \Fb h_c^{C^t},\int_{C^t}f)
\end{array}$$

\begin{prop}
\label{prop:un fonctorial Yoneda}
For any object $c$ of $C$, there exists a unique pair consisting of a morphism
$$\int_{\w{C}} \hom_{\w{C}}(y_c,\uvar)\to \int_{\w{C}}\ev(c,\uvar)$$
and a commutative square of shape
\begin{equation}
\label{eq:prop:un fonctorial Yoneda}
\begin{tikzcd}
	{\{id_{y_c}\}} & {\hom_{\widehat{C}}(y_c,y_c)\sim \{y_c\}^* \int_{\w{C}}\hom_{\w{C}}(y_c,\uvar)} \\
	{\{id_c\}} & {\hom_C(c,c)\sim \{y_c\}^* \int_{\w{C}}\ev(c,\uvar)}
	\arrow[Rightarrow, no head, from=1-1, to=2-1]
	\arrow[from=1-1, to=1-2]
	\arrow[from=2-1, to=2-2]
	\arrow[from=1-2, to=2-2]
\end{tikzcd}
\end{equation}
Moreover, this comparison morphism is an equivalence.
\end{prop}
\begin{proof}
The proposition \ref{prop:Yoneda is Fb} implies that $\int_{\w{C}}\hom_{\w{C}}(y_c,\uvar)$ is equivalent to $\Fb h^{\w{C}}_{y_c}$. A natural transformation $\int_{\w{C}}\hom_{\widehat{C}}(y_c,\uvar)\to g$ then corresponds to a morphism 
$\Fb h^{\w{C}}_{y_c}\to \int_{\w{C}}g$ and is then uniquely characterized by the value on $\{id_{y_c}\}$, which proves the uniqueness.

It remains to show the existence.
Let $E$ be an object of $\ocatm_{/\w{C}^\sharp}$ corresponding to a morphism $g:X\to \w{C}^\sharp$ . We denote $\iota:X\to X^\sharp$ the canonical inclusion. According to proposition \ref{prop:Yoneda is Fb}, a morphism $E\to \int_{\w{C}}\hom_{\widehat{C}}(y_c,\uvar)$ corresponds to a morphism $E\to \Fb h^{\w{C}}_{y_c}$, and so to a triangle
\[\begin{tikzcd}
	& {\w{C}^\sharp_{y_c/}} \\
	X & {\w{C}^\sharp}
	\arrow[from=2-1, to=1-2]
	\arrow[from=1-2, to=2-2]
	\arrow[from=2-1, to=2-2]
\end{tikzcd}\]
According to corollary \ref{cor:parametric univalence tranche}, this data is equivalent to the one of 
$$X\times \int_{C^t}y_c\to (\iota\times (C^t)^\sharp)^*\int_{X^\natural\times C^t}\tilde{g}$$
where $\tilde{g}$ is the morphism defined by currying from $g^\natural:X^\natural\to \w{C}$. The proposition \ref{prop:Yoneda is Fb}, and 
 the equivalence \eqref{eq:evaluation 3} then induce an equivalence:
 $$\Hom_{\ocatm_{/\w{C}^\sharp}}(E, \int_{\w{C}}\hom_{\w{C}}(y_c,\uvar))\sim \Hom_{\ocatm_{/\w{C}^\sharp}}(E,\int_{\w{C}}\ev(c,\uvar))$$
Walking through all the equivalences, we can easily see that when $E$ is $h^{\w{C}}_{y_c}$, this equivalence sends the upper horizontal morphism of \eqref{eq:prop:un fonctorial Yoneda} to the lower horizontal one.
We then have an equivalence
$$\int_{\w{C}} \hom_{\w{C}}(y_c,\uvar)\sim \int_{\w{C}}\ev(c,\uvar).$$
that comes along with the desired commutative square.
\end{proof}

\begin{theorem}
\label{theo:Yoneda ff}
The Yoneda embedding is fully faithful. As a consequence, every morphism $A\to \w{C}$ that is pointwise representable uniquely factors through the Yoneda embedding.
\end{theorem}
\begin{proof}
We fix an object $c$ of $C$.
By construction of the Yoneda embedding and the evaluation, we have an equivalence $\ev(c,y_d)\sim \hom_C(c,d)$ natural in $d:C$.  Applying the Grothendieck deconstruction to the equivalence given in proposition \ref{prop:un fonctorial Yoneda}, we then get an equivalence 
$$\eta_d:\hom_{\widehat{C}}(y_c,y_d)\sim \hom_C(c,d)$$
natural in $d:C$ and that preserves the identity. 

We also have a transformation 
$$\hom_{y}(c,d):\hom_{C}(c,d)\to \hom_{\w{C}}(y_c,y_d)$$
natural in $d:C$, that also preserves the identity. 
We then have constructed a natural transformation
$$\psi_{c,d}:\hom_C(c,d)\xrightarrow{\hom_{y}(c,d)} \hom_{\w{C}}(y_{c},y_{d})\xrightarrow{\eta_d}\hom_C(c,d)$$
natural in $d:C$, and which preserves the identity. As the Grothendieck construction of $\hom_{C}(c,\uvar)$ is $\Fb h^C_c$ according to proposition \ref{prop:Yoneda is Fb}, the morphism
$$\int_C\psi_{c}:\Fb h^C_c\to \Fb h^C_c$$ 
is characterized by its value on $\{id_c\}$ and is then the identity. This implies that $\psi_c$ is the identity.
By two out of three, this implies that $\hom_{y}(c,\uvar)$ also is an equivalence, which concludes the proof.
\end{proof}

\begin{lemma}
\label{lemma:a particular Kan extension}
Let $i:C\to D$ be a morphism between locally $\U$-small $\io$-categories.
The canonical morphism of $\LCart((C^t)^\sharp\times D^\sharp)$:
$$\Lb(id\times i)_!\int_{C^t\times C}\hom_{C} \to \int_{C^t\times D}\hom_D(i(\uvar),\uvar)$$
is an equivalence.
\end{lemma}
\begin{proof}
Let $c,d$ be any objects of respectively $C$ and $D$.  We then have equivalences
$$\begin{array}{rcll}
\Rb (c,d)^* \Lb (id\times i)_!\int_{ C^t\times C^t}\hom_{ C}&\sim&\Rb \{d\}^*  \Lb i_! \Rb (id\times \{c\})^*\int_{ C^t\times C}\hom_{C}& (\ref{prop:BC condition})\\
&\sim& \Rb \{d\}^* \Lb i_!\Fb h^{C}_{c} &(\ref{prop:Yoneda is Fb})\\
&\sim & \Rb \{d\}^* \Fb h^{D}_{i(c)}\\
&\sim & \hom_D(i(c),d)^\flat
\end{array}$$
Remark that we also have an equivalence 
$$\Rb (c,d)^*\int_{C^t\times D}\hom_D(i(\uvar),\uvar)\sim \hom_D(i(c),d)^\flat$$
and that the induced endomorphism of $ \hom_D(i(c),d)^\flat$ is the identity. As equivalences are detected pointwise, this concludes the proof.
\end{proof}

\begin{theorem}
\label{theo:Yoneda lemma}
Let $C$ be a locally $\U$-small $\io$-category. There is an equivalence between the functor
$$\hom_{\w{C}}(y_{\uvar},\uvar):C^t\times \w{C}\to \uni$$ and
the functor 
$$\ev:C^t\times \w{C}\to \uni.$$
Restricted to $\w{C}\times \{c\}$ for $c$ an object of $C$, this equivalence is the one of proposition \ref{prop:un fonctorial Yoneda}.
\end{theorem}
\begin{proof}
The triangle
\[\begin{tikzcd}
	C \\
	{\w{C}} & {\w{C}}
	\arrow["y", from=1-1, to=2-2]
	\arrow["y"', from=1-1, to=2-1]
	\arrow["id"', from=2-1, to=2-2]
\end{tikzcd}\] 
induces by adjunction a triangle
\[\begin{tikzcd}
	{C^t\times C} \\
	{C^t\times\w{C}} & {\w{C}}
	\arrow["\hom", from=1-1, to=2-2]
	\arrow["{id\times y}"', from=1-1, to=2-1]
	\arrow["\ev"', from=2-1, to=2-2]
\end{tikzcd}\]
This corresponds to an equivalence
$$\int_{C^t\times C}\hom_{C}(\uvar,\uvar)\to (id\times y)^*\int_{C^t\times \w{C}}\ev.$$
By naturality, for any object $c$ of $C$, the pullback of the previous equivalence along $C^t\times\{c\}$ is the identity. In particular, the induced morphism $\hom(c,c)\to \hom(c,c)$ between the fibers over $(c,c)$ preserves the object $\{id_c\}$. According to lemma \ref{lemma:a particular Kan extension}, the previous equivalence induces a morphism
\begin{equation}
\label{eq:proof of yoneda}
\int_{C^t\times \w{C}}\hom_{\w{C}}(y_{\uvar},\uvar)\to \int_{C^t\times \w{C}}\ev.
\end{equation}
that comes along, by construction, with a commutative square
\[\begin{tikzcd}
	{\{id_{y_c}\}} & {\hom_{\widehat{C}}(y_c,y_c)\sim \{y_c\}^* \int_{\w{C}}\hom_{\w{C}}(y_c,\uvar)} \\
	{\{id_c\}} & {\hom_C(c,c)\sim \{y_c\}^* \int_{\w{C}}\ev(c,\uvar)}
	\arrow[Rightarrow, no head, from=1-1, to=2-1]
	\arrow[from=1-1, to=1-2]
	\arrow[from=2-1, to=2-2]
	\arrow[from=1-2, to=2-2]
\end{tikzcd}\]
for any object $c$ of $C$. The restriction of the morphism \eqref{eq:proof of yoneda} to $\w{C}\times \{c\}$ is then equivalent to the natural transformation given in proposition \ref{prop:un fonctorial Yoneda}, and is an equivalence. As equivalences between left cartesian fibrations are detected on fibers, this concludes the proof.
\end{proof}

\begin{cor}
\label{cor: universal fibration 2}
The universal left cartesian fibration with $\U$-small fibers is the canonical projection 
$\uni^\sharp_{1/}\to \uni^\sharp$.
\end{cor}
\begin{proof}
The corollary \ref{cor: universal fibration 2} implies that universal left cartesian fibration with $\U$-small fibers is $\int_{\uni}id$. The Yoneda lemma implies that this left cartesian fibration is equivalent to $\int_{\uni}\hom_{\uni}(1,\uvar)$. Eventually, the proposition \ref{prop:Yoneda is Fb} states that this left cartesian fibration is equivalent to $\uni^\sharp_{1/}\to \uni^\sharp$.
\end{proof}

\subsection{Adjoint functors}

\begin{definition}
Let $C$ and $D$ be two locally $\U$-small $\io$-categories and $u:C\to D,$ $v:D\to C$ two functors. An \notion{adjoint structure} for the pair $(u,v)$ is the data of a invertible natural transformation
$$\phi: \hom_D(u(\uvar),\uvar)\sim \hom_C(\uvar,v(\uvar))$$
In this case, $u$ is a \wcnotion{left adjoint}{left or right adjoint} of $v$ and $v$ is a \textit{right adjoint} of $u$.
\end{definition}

\begin{prop}
\label{prop:adj if slice as terminal}
Let $u:C\to D$ be a functor between locally $\U$-small $\io$-categories. 
For $b$ an object of $D$, we define $(C^t)^\sharp_{b/}$ and $C^\sharp_{b/}$ as the marked $\io$-categories fitting in the cartesian squares:
\[\begin{tikzcd}
	{(C^t)^\sharp_{/b}} & {(D^t)^\sharp_{b/}} & {C^\sharp_{b/}} & {D^\sharp_{b/}} \\
	{(C^t)^\sharp} & {(D^t)^\sharp} & {C^\sharp} & {D^\sharp}
	\arrow[from=1-3, to=2-3]
	\arrow["u"', from=2-3, to=2-4]
	\arrow[from=1-4, to=2-4]
	\arrow["\lrcorner"{anchor=center, pos=0.125}, draw=none, from=1-3, to=2-4]
	\arrow[from=1-3, to=1-4]
	\arrow[from=1-1, to=2-1]
	\arrow[from=1-1, to=1-2]
	\arrow[from=1-2, to=2-2]
	\arrow["{u^t}"', from=2-1, to=2-2]
	\arrow["\lrcorner"{anchor=center, pos=0.125}, draw=none, from=1-1, to=2-2]
\end{tikzcd}\]
The following are equivalent.
\begin{enumerate}
\item The functor $u$ admits a right adjoint.
\item For any element $b$ of $D$, the marked $\io$-category $(C^t)^\sharp_{b/}$ 
admits an initial element.
\end{enumerate}
Similarly, the following are equivalent.
\begin{enumerate}
\item[(1)'] The functor $u$ admits a left adjoint.
\item[(2)'] For any element $b$ of $D$, $C^\sharp_{b/}$ admits an initial element.
\end{enumerate}
\end{prop}
\begin{proof}
Suppose first that $(1)$ is fulfilled, and let $v:D\to C$ be a functor and $\phi:\hom(u(a),b)\sim\hom(a,v(b))$ be an invertible natural transformation. In particular, this implies that we have an equivalence
$$\int_{C^t\times D}\hom_D(u(a),b)\sim \int_{C^t\times D}\hom_C(a,v(b))$$
Pulling back along $C^t\times \{b\}$ where $b$ is any object of $D$, we get an equivalence between 
$(C^t)^\sharp_{b/}$ and $(C^t)^\sharp_{v(b)/}$. As this last marked $\io$-category admits an initial element, given by the image $id_{v(b)}$, this shows the implication $(1)\Rightarrow (2)$.

For the converse, suppose that $u$ fulfills condition $(2)$. The functor $\hom_D(u(\uvar),\uvar)):C^t\times D\to \uni$ corresponds by adjonction to a functor $v':D\to \w{C}$. By assumption, for any $b\in B$, $v'(b)$ is a representable $\io$-presheaf. The Yoneda lemma then implies that $v$ factors through a functor $v:D\to C$. Using once again Yoneda lemma, we have a sequence of equivalences
$$\hom_D(u(a),b)\sim v'(b)(a)\sim \hom_C(b,v(a)).$$

The equivalence between $(1)'$ and $(2)'$ is proved similarly.
\end{proof}

\p
 Let $(u,v,\phi)$ be an adjoint structure. There is a transformation 
$$\hom_C(a,a')\to \hom_D(u(a),u(a'))\to \hom_C(a,vu(a'))$$
natural in $a:C^t$, $a':C$. According to the Yoneda lemma, this corresponds to a natural transformation $\mu: id_C \to vu$, called the \wcnotion{unit of the adjunction}{unit and counit of an adjunction}. Similarly, the natural transformation:
$$\hom_D(b,b')\to \hom_C(v(b),v(b'))\to \hom_C(uv(b),b')$$
induces a natural transformation $\epsilon:uv\to id_D$, called \textit{the counit of the adjunction.}

\begin{lemma}
\label{lemma:naturality of hom apply to natural transformation}
Suppose we have two morphisms $f:C\to D$ and $g:C\to D$ between locally $\U$-small $\io$-categories as well as a natural transformation $\nu:f\to g$. This induces a commutative diagram 
\[\begin{tikzcd}
	{\hom_C(a,b)} & {\hom_D(g(a),g(b))} \\
	{\hom_D(f(a),f(b))} & {\hom_D(f(a),g(b))}
	\arrow[from=1-1, to=2-1]
	\arrow[from=1-1, to=1-2]
	\arrow["{(\nu_{a})_!}", from=1-2, to=2-2]
	\arrow["{(\nu_{b})_!}"', from=2-1, to=2-2]
\end{tikzcd}\]
natural in $a:C^t, b:C$.
\end{lemma}
\begin{proof}
Remark that $\hom_{[1]}(0,1)\sim \hom_{[1]}(1,1)\sim \hom_{[1]}(0,0)=1$.
Using the naturality of the hom functor, we have a commutative diagram 
\[\begin{tikzcd}
	{\hom_C(a,b)\times \hom_{[1]}(0,0)} & {\hom_D(f(a),f(b))} \\
	{\hom_C(a,b)\times \hom_{[1]}(0,1)} & {\hom_D(f(a),g(b))} \\
	{\hom_C(a,b)\times \hom_{[1]}(1,1)} & {\hom_D(g(a),g(b))}
	\arrow["{(\nu_{a})_!}"', from=3-2, to=2-2]
	\arrow["{(\nu_{b})_!}", from=1-2, to=2-2]
	\arrow["\sim"', from=1-1, to=2-1]
	\arrow["\sim", from=3-1, to=2-1]
	\arrow[from=2-1, to=2-2]
	\arrow[from=3-1, to=3-2]
	\arrow[from=1-1, to=1-2]
\end{tikzcd}\]
where the left-hand vertical morphisms are equivalences.
\end{proof}

\begin{prop}
\label{prop:If unit and counit so adjunction}
Let $u:C\to D$ and $v:D\to C$ be two functors between locally $\U$-small $\io$-categories, $\mu:id_C\to vu$, $\epsilon:uv\to id_D$ be two natural transformations coming along with equivalences 
$$(\epsilon\circ_0 u)~\circ_1~(u\circ_0 \mu) \sim id_{u}~~~~ (v\circ_0 \epsilon)~\circ_1(\mu \circ_0 v )\sim id_{v}.$$
If we set $\phi$ as the composite 
$$\hom_D(u(a),b)\to \hom_C(vu(a),v(b))\xrightarrow{(\mu_a)_!} \hom_C(a,v(b)),$$
the triple $(u,v,\phi)$ is an adjoint structure.
Moreover, the unit of the adjunction is $\mu$ and its counit is $\epsilon$.
\end{prop}
\begin{proof}
Suppose we have such data. We define $\psi$ as the composite
$$\hom_C(a,v(b))\to \hom_D(u(a),uv(b))\xrightarrow{(\epsilon_a)_!} \hom_D(u(a),b)
$$
natural in $a:C^t$ and $b:D$. We then have to show that these two morphisms are inverse of each other. For this consider the diagram
\[\begin{tikzcd}
	{\hom_D(u(a),b)} & {\hom_C(vu(a),v(b))} & {\hom_C(a,v(b))} \\
	& {\hom_D(uvu(a),uv(b))} & {\hom_D(u(a),uv(b))} \\
	{\hom_D(u(a),b)} & {\hom_D(uvu(a),b)} & {\hom_D(u(a),b)}
	\arrow["{(\mu_a)_!}", from=1-2, to=1-3]
	\arrow["{(u(\mu_{a}))_!}"', from=3-2, to=3-3]
	\arrow[from=1-1, to=1-2]
	\arrow["{(u(\mu_{a}))_!}", from=2-2, to=2-3]
	\arrow[from=1-2, to=2-2]
	\arrow[from=1-3, to=2-3]
	\arrow["{(\epsilon_b)_!}", from=2-2, to=3-2]
	\arrow["{(\epsilon_b)_!}", from=2-3, to=3-3]
	\arrow["{(\epsilon_{u(a)})_!}"', from=3-1, to=3-2]
	\arrow[Rightarrow, no head, from=1-1, to=3-1]
\end{tikzcd}\]
which is commutative thanks to lemma \ref{lemma:naturality of hom apply to natural transformation} and the naturality of the $\hom$.
By hypothesis, the left lower horizontal morphism is equivalent to the identity.
The outer square then defines an equivalence between $\psi\circ \phi$ and the identity. We show similarly $\phi\circ \psi\sim id$.

For the second assertion, remark that the composition 
$$\hom_C(a,a')\to \hom_D(u(a),u(a'))\xrightarrow{\phi(a,u(a')} \hom_C(a,vu(a'))$$
is by definition equivalent to 
$$\hom_C(a,a')\to \hom_D(vu(a),vu(a'))\xrightarrow{(\mu_a)_!} \hom_C(a,vu(a'))$$
and according to the lemma \ref{lemma:naturality of hom apply to natural transformation}, to
$$\hom_C(a,a')\xrightarrow{(\mu_{a'})_!} \hom_C(a,vu(a'))$$
The Yoneda lemma then implies that the unit of the adjunction is $\mu$. We proceed similarly for the counit.
\end{proof}

\p In paragraph \ref{par: i pull and push beetwe io category of morphism}, for a morphism $i:I\to A^\sharp$ between marked $\io$-categories, we define the morphism $i_!:\gHom(I,\uni)\to \uHom(A,\uni)$ and when $i$ is proper, a morphism $i_*:\gHom(I,\uni)\to \uHom(A,\uni)$.

\begin{cor}
\label{cor:naive kan extension}
Let $i:I\to A^\sharp$ be a morphism between $\U$-small $\io$-category. The functor $i^*:\uHom(A,\uni)\to \gHom(I,\uni)$ has a left adjoint given by the functor $i_!:\gHom(I,\uni)\to \uHom(A,\uni)$. If $i$ is proper, the functor $i^*$ has a right adjoint $i_*:\gHom(I,\uni)\to \uHom(A,\uni)$.
\end{cor} 
\begin{proof}
With the characterization of adjunction given in proposition \ref{prop:If unit and counit so adjunction}, this is a direct consequence of natural transformations given in paragraph \ref{par: i pull and push beetwe io category of morphism}.
\end{proof}

\p
We conclude this section with the proof of the following theorem.
\begin{theorem}
\label{theo:two adjunction definition}
Let $u:C\to D$ and $v:D\to C$ be two functors between locally $\U$-small $\io$-categories. 
The two following are equivalent. 
\begin{enumerate}
\item The pair $(u,v)$ admits an adjoint structure.
\item Their exists a pair of natural transformations $\mu: id_C \to vu$ and $\epsilon:uv\to id_D$ together with equivalences $(\epsilon\circ_0 u)\circ_1(u\circ_0 \mu) \sim id_{u}$ and $(v\circ_0 \epsilon)\circ_1 (\mu \circ_0 v )\sim id_{v}$.
\end{enumerate}
\end{theorem}
We directly give a corollary:

\begin{cor}
\label{cor:adjonction induced adjunction by post composition}
Let $(u:B\to C,v:C\to B)$ be an adjoint pair between locally $\U$-small $\io$-categories and $D$ a locally $\U$-small $\io$-category.
If $C$ and $B$ are $\U$-small, this induces an adjunction
\[\begin{tikzcd}
	{\uvar\circ u:\uHom(C,D)} & {\uHom(B,D):\uvar\circ v}
	\arrow[""{name=0, anchor=center, inner sep=0}, shift left=2, from=1-1, to=1-2]
	\arrow[""{name=1, anchor=center, inner sep=0}, shift left=2, from=1-2, to=1-1]
	\arrow["\dashv"{anchor=center, rotate=-90}, draw=none, from=0, to=1]
\end{tikzcd}\]
and if $D$ is $\U$-small an adjunction
\[\begin{tikzcd}
	{u\circ \uvar:\uHom(D,C)} & {\uHom(D,B):v\circ\uvar}
	\arrow[""{name=0, anchor=center, inner sep=0}, shift left=2, from=1-1, to=1-2]
	\arrow[""{name=1, anchor=center, inner sep=0}, shift left=2, from=1-2, to=1-1]
	\arrow["\dashv"{anchor=center, rotate=-90}, draw=none, from=0, to=1]
\end{tikzcd}\]
\end{cor}
\begin{proof}
Let $\mu$ and $\epsilon$ be the unit and the counit of the adjunction. We define $\mu': \uHom(C,D)\times [1]\to \uHom(C,D)$, induced by currying the morphism 
$$ \uHom(C,D)\times [1]\times C\xrightarrow{id\times \mu} \uHom(C,D)\times C\xrightarrow{\ev} D$$
and $\epsilon':\uHom(B,D)\times [1]\to \uHom(B,D)$ by currying the morphism 
$$ \uHom(B,D)\times [1]\times B\xrightarrow{id\times \epsilon} \uHom(B,D)\times B\xrightarrow{\ev} B$$
We can easily check that $\mu'$ and $\epsilon'$ fulfill the triangle identities, and theorem \ref{theo:two adjunction definition} then implies that the pair $(\uvar\circ u,\uvar\circ v)$ admits an adjunction structure. We proceed similarly for the second assertion. 
\end{proof}

\p For the remaining, we fix two functors $u:C\to D$ and $v:D\to C$ between $\io$-categories as well as an equivalence
$$\phi:\hom_{D}(u(a),b)\sim \hom_C(a,v(b))$$
natural in $a:C^t$ and $b:D$.

\begin{lemma}
\label{lemma: if ajdunction then unit 1}
The natural transformation
$$\hom_D(u(a),b)\to \hom_C(vu(a),v(b))\xrightarrow{(\mu_a)_!}\hom_C(a,v(b))$$
is equivalent to $\phi:\hom_D(u(a),b)\to \hom_D(a,v(b))$.
Similarly, the natural transformation

$$\hom_C(a,v(b))\to \hom_D(u(a),uv(b))\xrightarrow{(\epsilon_b)_!}\hom_D(u(a),b)$$
is equivalent to $\phi^{-1}:\hom_D(a,v(b))\to \hom_D(u(a),b)$.
\end{lemma}
\begin{proof}
Remark that we have a commutative diagram
\[\begin{tikzcd}
	{\hom_C(a,b)} & {\hom_D(u(a),u(b))} & {\hom_C(vu(a),vu(b))} \\
	{\hom_D(u(a),u(b))} && {\hom_D(a,vu(b))}
	\arrow["{(\mu_a)_!}", from=1-3, to=2-3]
	\arrow[from=1-2, to=1-3]
	\arrow[from=1-1, to=1-2]
	\arrow["{(\mu_b)_!}"{description}, from=1-1, to=2-3]
	\arrow["\phi"{description}, from=2-1, to=2-3]
	\arrow[from=1-1, to=2-1]
\end{tikzcd}\]
The commutativity of the left triangle comes from the definition of $\mu$, and the second one, from the lemma \ref{lemma:naturality of hom apply to natural transformation}, applied to $\mu$.
This then induces a commutative square
\[\begin{tikzcd}
	{\int_{C^t\times C}\hom_C} && {\Rb(id\times u)^*\int_{C^t\times D}\hom_D(u(\uvar),\uvar)} \\
	{\Rb(id\times u)^*\int_{C^t\times D}\hom_D(u(\uvar),\uvar)} && {\Rb(id\times u)^*\int_{C^t\times D}\hom_C(\uvar,v(\uvar))}
	\arrow[from=1-1, to=1-3]
	\arrow["{\Rb (id\times u)^*\int_{C^t\times D}(\mu_a)_!\circ \hom_v}", from=1-3, to=2-3]
	\arrow[from=1-1, to=2-1]
	\arrow["{\Rb(id\times u)^*\int_{C^t\times D}\phi}"', from=2-1, to=2-3]
\end{tikzcd}\]
By adjunction, this corresponds to a commutative square 
\[\begin{tikzcd}
	{\Lb(id\times u)_!\int_{C^t\times C}\hom_C} && {\int_{C^t\times D}\hom_D(u(\uvar),\uvar)} \\
	{\int_{C^t\times D}\hom_D(u(\uvar),\uvar)} && {\int_{C^t\times D}\hom_C(\uvar,v(\uvar))}
	\arrow[from=1-1, to=1-3]
	\arrow["{\int_{C^t\times D}(\mu_a)_!\circ \hom_v}", from=1-3, to=2-3]
	\arrow["{\int_{C^t\times D}\phi}"', from=2-1, to=2-3]
	\arrow[from=1-1, to=2-1]
\end{tikzcd}\]
However, the top horizontal and left vertical morphisms are equivalences according to lemma \ref{lemma:a particular Kan extension}.
We then have an equivalence $$ \int_{C^t\times D}(\mu_a)_!\circ \hom_v\sim \int_{C^t\times D}\phi$$
which implies the result.
The other assertion is shown similarly.
\end{proof}

\begin{lemma}
\label{lemma:if ajdunction then unit 2}
There are equivalences
$(\epsilon\circ_0 u)\circ_1(u\circ_0 \mu) \sim id_{u}$ and $(v\circ_0 \epsilon)\circ_1 (\mu \circ_0 v )\sim id_{v}$.
\end{lemma}
\begin{proof}
As the proof of the two assertions are similar, we will only show the second one.
To demonstrate this, it is enough to show that the induced natural transformation 
\begin{equation}
\label{eq:sequence ajdunction}
\hom_C(a,v(b))\xrightarrow{(\mu_{v(b)})_!} \hom_C(a,vuv(b)) \xrightarrow{(v(\epsilon_{(b)}))_!} \hom_C(a,v(b))\xrightarrow{\phi^{-1}}\hom_D(u(a),b)
\end{equation}
is equivalent to $\phi^{-1}$.
By definition, the first morphism is equivalent to the composition
$$\hom_C(a,v(b))\to \hom_D(u(a),uv(b))\xrightarrow{\phi} \hom_C(a,vuv(b))$$
and as $\phi^{-1}$ is a natural transformation, we have a commutative square
\[\begin{tikzcd}
	{\hom_C(a,vuv(b))} & {\hom_C(a,v(b))} \\
	{\hom_C(u(a),uv(b))} & {\hom_D(u(a),b)}
	\arrow["{(v(\epsilon_{b}))_!}", from=1-1, to=1-2]
	\arrow["{\phi^{-1}}", from=1-2, to=2-2]
	\arrow["{\phi^{-1}}"', from=1-1, to=2-1]
	\arrow["{(\epsilon_{b})_{!}}"', from=2-1, to=2-2]
\end{tikzcd}\]
The composite of the sequence \eqref{eq:sequence ajdunction} is then equivalent to 
$$\hom_C(a,v(b))\to \hom_D(u(a),uv(b))\xrightarrow{(\epsilon_{b})_{!}} \hom_D(u(a),b)$$
which is itself equivalent to $\phi^{-1}$ according to lemma \ref{lemma: if ajdunction then unit 1}.
\end{proof}

\begin{proof}[Proof of theorem \ref{theo:two adjunction definition}]
The implication $(1)\Rightarrow (2)$ is given by proposition \ref{prop:If unit and counit so adjunction} and the contraposed by the lemma \ref{lemma:if ajdunction then unit 2}.
\end{proof}

\subsection{Lax colimits}
\p According to corollary \ref{cor:naive kan extension}, a morphism $f:A\to B$ between $\U$-small $\io$-categories induces an adjoint pair:
\begin{equation}
\label{eq:adjoint presheaves}
\begin{tikzcd}
	{f_!:\w{A}} & {\w{B}:f^*}
	\arrow[""{name=0, anchor=center, inner sep=0}, shift left=2, from=1-1, to=1-2]
	\arrow[""{name=1, anchor=center, inner sep=0}, shift left=2, from=1-2, to=1-1]
	\arrow["\dashv"{anchor=center, rotate=-90}, draw=none, from=0, to=1]
\end{tikzcd}
\end{equation}

\begin{prop}
\label{prop:left extension commutes with Yoneda}
Let $f:A\to B$ be a morphism between $\U$-small $\io$-categories.
There is an equivalence 
$$f_!(y_a)\sim y_{f(a)}$$
natural in $a:A$.
\end{prop}
\begin{proof}
Consider the sequence of equivalences
$$\begin{array}{rcll}
\hom_{\w{B}}(f_!(y_a), g)&\sim &\hom_{\w{A}}(y_a, f^*(g))& \eqref{eq:adjoint presheaves}\\
&\sim &\ev(a,f^*(g))&(\mbox{Yoneda lemma})\\
&\sim &\ev(f(a),g)&(\mbox{naturality of $\ev$})\\
&\sim & \hom_{\w{B}}(y_{f(a)}, g)&(\mbox{Yoneda lemma})\\
\end{array}$$
Eventually, the Yoneda lemma applied to $(\w{B})^t$ concludes the proof.
\end{proof}
\p For $I$ a marked $\io$-category and $A$ an $\io$-category,
we recall that $\gHom(I,A)$ is the $\io$-category whose value on a globular sum $a$ is given by:
$$\Hom(a,\gHom(I,A)):=\Hom(I\ominus a^\sharp, A^\sharp)$$
\begin{remark}
Let $B$ be an $\io$-category.
We want to give an intuition of the object $\gHom(B^\flat,\omega)$. The objects of this $\io$-category are the functors $I\to \omega$. The $1$-cells are the lax transformations $F\Rightarrow G$. For $n>1$, the $n$-cells are the lax transformations $F^{\times \Db_{n-1}}\Rightarrow G$ where $F^{\times \Db_{n-1}}:I\to \omega$ is the functor that sends $i$ onto $F(i)\times \Db_{n-1}$.
This last assertion is a consequence of the equivalence 
$$\tau_0(\LCart((I\ominus [b,n]^\sharp)^\sharp) \sim \Hom([n],\LCartc(I;b))$$
provided by the lemma \ref{lemma:lax univalence 4}.
\end{remark}
\begin{prop}
If $I$ is $\U$-small and $A$ is locally $\U$-small, the $\io$-category $\gHom(I,A)$ is locally $\U$-small.
\end{prop}
\begin{proof}
We have to check that for any globular sum $b$, the morphism 
$$\Hom(I\ominus [b,1]^\sharp,A^\sharp)\to \Hom(I\ominus (\{0\}\amalg\{1\}),A^\sharp)$$
has $\U$-small fibers. As $I$, seen as an $\infty$-presheaves on $t\Theta$, is a $\U$-small colimit of representables, we can reduce to the case where $I\in t\Theta$. As $A$ is local with respect to Segal extensions, and as $\ominus$ conserves them, we can reduce to the case where $I$ is of shape $[1]^\sharp$ or $[a,1]$ for $a$ a in $t\Theta$. If $I$ is $[1]^\sharp$, according to the second assertion of proposition \ref{prop:associativity of ominus}, $[1]^\sharp\ominus [b,1]^\sharp$ is equivalent to $([1]\times [b,1])^\sharp$ and the result follows from proposition \ref{prop:when Hom A B is locally small}.

For the second case, we fix a morphism $f:[a,1]\times (\{0\}\amalg\{1\})\to A$. 
Using the canonical equivalence between $[a,1]\ominus [b,1]^\sharp$ and the colimit of the diagram \eqref{eq:formula for the ominus marked case},
the $\infty$-groupoid $\Hom(I\ominus [b,1]^\sharp,A^\sharp)_f$ is the limit of the diagram:
\[\begin{tikzcd}[column sep=0.1cm]
	{\Hom(a^\natural,\hom(f(1,0),f(1,1)))} & {\Hom((a\otimes\{0\}^\sharp)^\natural\times b,\hom(f(0,0),f(1,1))} \\
	{\Hom((a\otimes[1]^\sharp)^\natural\times b,\hom(f(0,0),f(1,1)))} \\
	{\Hom(a^\natural,\hom(f(0,0),f(1,0)))} & {\Hom((a\otimes\{0\}^\sharp)^\natural\times b,\hom(f(0,0),f(1,1))))}
	\arrow[from=2-1, to=3-2]
	\arrow[from=2-1, to=1-2]
	\arrow[from=3-1, to=3-2]
	\arrow[from=1-1, to=1-2]
\end{tikzcd}\]
As all these objects are $\U$-small by assumption, this concludes the proof.
\end{proof}

\p
Let $I$ be a $\U$-small marked $\io$-category, $A$ a locally $\U$-small $\io$-category $A$ and $F:I\to A^\sharp$ a functor. 
A \notion{lax colimit} of $F$ is an object \wcnotation{$\laxcolim_IF$}{(laxcolim@$\laxcolim$} of $A$ together with an equivalence
$$\hom_{A}(\laxcolim_IF, b)\sim \hom_{\gHom(I,A)}(F,\cst b)$$
natural in $b:A$. 
Conversely, a \notion{lax limit} of $F$ is an object \wcnotation{$\laxlim_IF$}{(laxlim@$\laxlim$} of $A$ together with an equivalence
$$\hom_{A}(b,\laxlim_IF)\sim \hom_{\gHom(I,A)}(\cst b,F)$$
natural in $b:A$. 
We say that a locally $\U$-small $\io$-category $C$ is \notion{lax $\U$-complete} (resp. \notion{lax $\U$-cocomplete}), if for any $\U$-small marked $\io$-category $I$ and any functor $F:I\to C$, $F$ admits limits (resp. colimits).

Using proposition \ref{prop:adj if slice as terminal}, $C$ is lax $\U$-complete (resp. lax $\U$-cocomplete) if and only if for any $\U$-small marked $\io$-category $I$, the functor $\cst:C\to \gHom(I,C)$ admits a right adjoint (resp. a left adjoint).

The proposition \ref{prop:ominus and opmarked}  induces an equivalence
$$\gHom(I,A)^{\circ}\sim\gHom(I^{\circ},A^{\circ})$$
As a consequence, a functor $F:I\to A^\sharp$ admits a lax colimit if and only if $F^\circ:I^\circ\to (A^\circ)^\sharp$ admits a lax limit. If $F$ admits such lax colimit, the lax limit of $F^\circ$ is the image by the canonical equivalence $A_0\sim A^\circ_0$ of the lax colimit of $F$.

\begin{remark}
We want to give an intuition of the lax colimits.
Let $I$ be a $\U$-small marked $\io$-category, $A$ a locally $\U$-small $\io$-category $A$ and $F:I\to A^\sharp$ a functor admitting a lax colimit $\laxcolim_IF$. For any $1$-cell $i:a\to b$ in $I$, we have a triangle
\[\begin{tikzcd}
	{} & {F(b)} \\
	{F(a)} & {\laxcolim_IF}
	\arrow["{F(i)}", curve={height=-30pt}, from=2-1, to=1-2]
	\arrow[from=2-1, to=2-2]
	\arrow[shorten <=8pt, shorten >=8pt, Rightarrow, from=1-2, to=2-1]
	\arrow[draw=none, from=1-1, to=2-1]
	\arrow[from=1-2, to=2-2]
\end{tikzcd}\]
If $i$ is marked, the preceding $2$-cell is an equivalence. 
For any $2$-cell $u:i\to j$, we have a diagram
\[\begin{tikzcd}
	& {F(b)} & {} & {F(b)} \\
	{F(a)} & {\laxcolim_IF} & {F(a)} & {\laxcolim_IF}
	\arrow[""{name=0, anchor=center, inner sep=0}, "{F(i)}"{description}, from=2-1, to=1-2]
	\arrow[""{name=1, anchor=center, inner sep=0}, from=2-1, to=2-2]
	\arrow[from=1-2, to=2-2]
	\arrow[""{name=2, anchor=center, inner sep=0}, from=1-2, to=2-2]
	\arrow[""{name=3, anchor=center, inner sep=0}, "{F(j)}", curve={height=-30pt}, from=2-1, to=1-2]
	\arrow["{F(j)}", curve={height=-30pt}, from=2-3, to=1-4]
	\arrow[from=2-3, to=2-4]
	\arrow[from=1-4, to=2-4]
	\arrow[shorten <=8pt, shorten >=8pt, Rightarrow, from=1-4, to=2-3]
	\arrow[""{name=4, anchor=center, inner sep=0}, draw=none, from=1-3, to=2-3]
	\arrow[shift right=2, shorten <=12pt, shorten >=12pt, Rightarrow, from=2, to=1]
	\arrow[shorten <=4pt, shorten >=4pt, Rightarrow, from=3, to=0]
	\arrow[shift left=0.7, shorten <=14pt, shorten >=16pt, no head, from=2, to=4]
	\arrow[shorten <=14pt, shorten >=14pt, from=2, to=4]
	\arrow[shift right=0.7, shorten <=14pt, shorten >=16pt, no head, from=2, to=4]
\end{tikzcd}\]
If $u$ is marked, the $3$-cell is an equivalence. We can continue these diagrams in higher dimensions and we have
similar assertions for lax limits.

The marking therefore allows us to play on the "lax character" of the universal property that the lax colimit must verify.
\end{remark}

\p Let $A$ be a $\U$-small $\io$-category and $I$ a $\U$-small marked $\io$-category.
Recall that $\gHom(I,\w{A})$ is equivalent to $\gHom(I\times (A^t)^\sharp,\uni)$. Let $t$ be the canonical morphism $I\to 1$. 
As $t$ is smooth, corollary \ref{cor:naive kan extension} induces adjunctions
\begin{equation}
\label{eq:expliciti colimit for presheaves}
\begin{tikzcd}
	{\gHom(I,\w{A})} && {\w{A}}
	\arrow[""{name=0, anchor=center, inner sep=0}, "{(t\times id_A)^*}"{description}, from=1-3, to=1-1]
	\arrow[""{name=1, anchor=center, inner sep=0}, "{(t\times id_A)_!}", shift left=5, from=1-1, to=1-3]
	\arrow[""{name=2, anchor=center, inner sep=0}, "{(t\times id_A)_*}"', shift right=5, from=1-1, to=1-3]
	\arrow["\dashv"{anchor=center, rotate=-90}, draw=none, from=0, to=2]
	\arrow["\dashv"{anchor=center, rotate=-90}, draw=none, from=1, to=0]
\end{tikzcd}
\end{equation}
and $\w{A}$ is then lax $\U$-complete and lax $\U$-cocomplete. For a morphism $g:I\to \w{A}^\sharp$ corresponding to an object $E$ of $\LCartc(I\times (A^t)^\sharp)$, we then have 
\begin{equation}
\label{eq:expliciti colimit for presheaves2}
\int_{A^t}\laxcolim_I g \sim \Lb (t\times id_{(A^t)^\sharp})_!E~~~\int_{A^t}\laxlim_I g \sim \Rb (t\times id_{(A^t)^\sharp})_*E
\end{equation}
Let $i:B^\sharp\to A^\sharp$ be any morphism. The squares given in paragraph \ref{par: i pull and push beetwe io category of morphism} induce the commutative squares
\[\begin{tikzcd}
	{\gHom(I,\w{A})} & {\w{A}} & {\gHom(I,\w{A})} \\
	{\gHom(I,\w{B})} & {\w{B}} & {\gHom(I,\w{B})}
	\arrow["{\laxcolim_I}", from=1-1, to=1-2]
	\arrow["{\laxcolim_I}"', from=2-1, to=2-2]
	\arrow["{(id_I\times i^t)^*}"', from=1-1, to=2-1]
	\arrow["{i^*}", from=1-2, to=2-2]
	\arrow["{(id_I\times i^t)^*}", from=1-3, to=2-3]
	\arrow["{\laxlim_I}"', from=1-3, to=1-2]
	\arrow["{\laxlim_I}", from=2-3, to=2-2]
\end{tikzcd}\]
In particular, choosing $B:=1$, this implies that the lax colimits and limits in $\io$-presheaves commute with evaluation.

The next proposition implies that limits and colimits in $\io$-presheaves can be detected as the level of the sub maximal $\iun$-categories of $\gHom(I,\w{A})$ and $\w{A}$. We recall that the sub maximal $\iun$-categories of $\gHom(I,\w{A})$, denoted by $\tau_1\gHom(I,\w{A})$, is the adjoint of the functor $[n]\mapsto I\otimes[n]^\sharp$.
\begin{prop}
Let $I$ be a $\U$-small marked $\io$-category, and $g:I\to A^\sharp$ a functor. An object $f$ of $\w{A}$ has a structure of colimit of the functor $g$ if and only if there exists an equivalence
$$\Hom_{\tau_1\w{A}}(f,h)\sim \Hom_{\tau_1\gHom(I,\w{A})}(F,\cst h)$$
natural in $h:(\tau^1 \w{A})^{op}$.
Similarly, the object $f$ has a structure of limit of the functor $F$ if and only if there exists an equivalence
$$\Hom_{\tau_1\w{A}}(h,f)\sim \Hom_{\tau_1\gHom(I,\w{A})}(\cst h,F)$$
natural in $h:(\tau^1 \w{A})^{op}$.
\end{prop}
\begin{proof}
We recall that theorem \ref{theo:lcartc et ghom} and corollary \ref{cor:lcar et hom} induces equivalences
$$\tau_1\w{A}\sim \LCart_{\U}((A^t)^\sharp)~~~ \tau_1\gHom(I,A)\sim \LCartc_{\U}(I\otimes (A^t)^\sharp)$$
and that we have a triplet of adjoints
\[\begin{tikzcd}
	{\LCartc_{\U}(I\otimes (A^t)^\sharp)} && {\LCart_{\U}((A^t)^\sharp)}
	\arrow[""{name=0, anchor=center, inner sep=0}, "{(t\times id_{A^t})^*}"{description}, from=1-3, to=1-1]
	\arrow[""{name=1, anchor=center, inner sep=0}, "{\Lb (t\times id_{A^t})_!}", shift left=5, from=1-1, to=1-3]
	\arrow[""{name=2, anchor=center, inner sep=0}, "{\Rb(t\times id_{A^t})_*}"', shift right=5, from=1-1, to=1-3]
	\arrow["\dashv"{anchor=center, rotate=-90}, draw=none, from=1, to=0]
	\arrow["\dashv"{anchor=center, rotate=-90}, draw=none, from=0, to=2]
\end{tikzcd}\]
which is the image by $\tau_1$ of the triplet of adjoints \eqref{eq:expliciti colimit for presheaves}.
The first hypothesis induces an equivalence $$\int_{A^t}f \sim \Lb (t\times id_{(A^t)^\sharp})_!E$$ and the second one an equivalence 
$$\int_{A^t}f \sim \Rb (t\times id_{(A^t)^\sharp})_*E$$ where $E$ denote the object of $\LCartc(I\times (A^t)^\sharp)$ corresponding to $g$. The assertions then follow from the equivalences \eqref{eq:expliciti colimit for presheaves2}.
\end{proof}

\begin{example}
We recall that we denote by $\bot:\Arr(\ocatm)\to \ocat$ the functor sending a left fibration $Y\to A$ to the localization of $Y$ by marked cells. This functors sends initial and final morphisms to equivalences. If $E$ is a left cartesian fibration over a marked $\io$-category $I$, we then have $\bot E\sim \Lb t_! E$ where $t$ denotes the morphism $I\to 1$.

 Let $g:I\to \uni$ be a diagram. We denote $\iota:I\to I^\sharp$ the canonical inclusion.
By the explicit expression of lax colimit given above, we then have an equivalence 
$$\laxcolim_I g \sim \bot \iota^*\int_{I^{\natural}}g^\natural.$$
If $I$ is equivalent to $I^\flat$, we then have
$$\laxcolim_I g \sim \dom(\int_{I^{\natural}}g^\natural)^\natural.$$
 \begin{enumerate}
 \item[$-$]
Let $c:1\to \uni$ be a morphism corresponding to an $\io$-category $C$. For any $\io$-category $A$, we then have 
$$\laxcolim_{A^\sharp} \cst_c\sim (\tau_0 A)\times C~~~~~\laxcolim_{A^\flat} \cst_c\sim A\times C$$

 \item[$-$] Let $f:[b,1]\to \uni$ be a morphism corresponding to a morphism $A\times b\to B$. We then have 
 $$\laxcolim_{[b,1]^\flat} f\sim A\times (1\costar b)\coprod_{A\times b}B$$
 \end{enumerate}
\end{example}

\begin{example}
 
Using the explicit expression of lax limit given above, we have an equivalence
$$\laxlim_I g \sim \Map(id_I,\iota^*\int_{I^{\natural}}g^\natural)$$
 \begin{enumerate}
 \item[$-$]
Let $c:1\to \uni$ be a morphism corresponding to an $\io$-category $C$. For any $\io$-category $A$, we then have 
$$\laxlim_{A^\sharp} \cst_c\sim \uHom(\tau_0 A, C)~~~~~\laxlim_{A^\flat} \cst_c\sim \uHom(A,C)$$

 \item[$-$] Let $f:[b,1]\to \uni$ be a morphism corresponding to a morphism $A\times b\to B$. Let $c$ be a globular sum. According to corollary \ref{cor:univalence tranche}, a morphism $id_{[b,1]^\flat}\times c^\flat\to \iota^*\int_{[b,1]^{\flat}}g^\natural$ corresponds to a diagram 
\[\begin{tikzcd}
	1 \\
	& {1\costar [b,1]} & \uni \\
	{[b,1]}
	\arrow["f"', curve={height=12pt}, from=3-1, to=2-3]
	\arrow["{\{b\}}", curve={height=-12pt}, from=1-1, to=2-3]
	\arrow[from=1-1, to=2-2]
	\arrow[from=3-1, to=2-2]
	\arrow[from=2-2, to=2-3]
\end{tikzcd}\]
and according to proposition \ref{prop:lfib and W 3}, to a diagram	
\[\begin{tikzcd}
	{c\times b\otimes\{0\}} && {A\times b} \\
	& {c\times( b\otimes[1])} && B \\
	{c\times b\otimes\{1\}} && c
	\arrow[from=1-3, to=2-4]
	\arrow[from=2-2, to=2-4]
	\arrow[from=1-1, to=2-2]
	\arrow[from=3-1, to=3-3]
	\arrow[from=3-1, to=2-2]
	\arrow[from=3-3, to=2-4]
	\arrow[from=1-1, to=1-3]
\end{tikzcd}\]
where the upper horizontal morphism is of shape $g\times b$. We then have 
 $$\laxlim_{[b,1]^\flat} f\sim A\prod_{\Hom(b,B)} \Hom(b\star 1,B).
 $$
 \end{enumerate}
 \end{example}

\begin{prop}
\label{prop:colimit restricted to final}
Let $i:I\to J$ be a morphism between $\U$-small marked $\io$-categories, $A$ a $\U$-small $\io$-category and $f:J\to \w{A}^\sharp$ a morphism.  If $i$ is final, then the canonical morphism
$$\laxcolim_{I}f\circ i\to \laxcolim_{J}f$$
is an equivalence.

If $i$ is initial, then the canonical morphism
$$\laxlim_{J}f\to \laxlim_{I}f\circ i$$
is an equivalence. 
\end{prop}
\begin{proof}
We only show the first assertion as the second follows by duality.
As equivalences are detected pointwise and as the lax colimit commutes with evaluation, one can suppose that $A:=1$, and so $\w{A}:=\uni$. 
We denote by $E$ (resp. $H$) the object of $\LCart(J)$ (resp.$\LCart(I)$) corresponding to $f$ (resp.$f\circ i$) and $X\to I$ (resp. $Y\to J$) the corresponding left cartesian fibration. We then have a cartesian square
\[\begin{tikzcd}
	Y & X \\
	J & I
	\arrow["i"', from=2-1, to=2-2]
	\arrow["{i'}", from=1-1, to=1-2]
	\arrow["H", from=1-2, to=2-2]
	\arrow["E"', from=1-1, to=2-1]
\end{tikzcd}\]
As classified left cartesian fibrations are proper, $i'$ is final. We recall that we denote by $\bot:\ocatm\to \ocat$ the functor sending a marked $\io$-category to its localization by marked cells, and that $\bot$ sends final morphism to equivalences. If we denote by $t$ the two morphisms $I\to 1$ and $J\to 1$, we then have a sequence of equivalences:
$$\laxcolim_{I}f\circ i\sim \Lb t_! H \sim \bot Y\sim \bot X\sim \Lb t_! E \sim \laxcolim_{J}f$$
\end{proof}

\begin{lemma}
\label{lemma:tehcnical colimit}
Let $F:I\to A^\sharp$ be a morphism between $\U$-small marked $\io$-categories. 
There is an equivalence 
$$ \hom_{\gHom(I,A)}(\cst_a,F)\sim \laxlim_I\hom_A(a,F)$$
natural in $F: \gHom(I,A)$ and $a:A^t$.
\end{lemma}
\begin{proof}
Remark that there is a commutative square:
\[\begin{tikzcd}
	A & {\gHom(I,A)} \\
	{\w{A}} & {\gHom(I,\w{A})}
	\arrow["\cst", from=1-1, to=1-2]
	\arrow["\cst"', from=2-1, to=2-2]
	\arrow["y"', from=1-1, to=2-1]
	\arrow["{\gHom(I,y)}", from=1-2, to=2-2]
\end{tikzcd}\]
and that the right vertical morphism is fully faithful as $y$ is.
We then have a sequence of equivalences
$$
\begin{array}{rcll}
\hom_{\gHom(I,A)}(\cst_a,F)&\sim& \hom_{\gHom(I,\w{A})}(\cst_{y_a},\gHom(I,y)(F))\\
&\sim& \hom_{\w{A}}(y_a,\laxlim_{I}\gHom(I,y)(F)))\\
&\sim& (\laxlim_{I}\gHom(I,y)(F))(a)&\mbox{(Yoneda lemma)}\\
&\sim& \laxlim_I\hom_A(a,F(i))
\end{array}$$
where the last one comes from the fact that evaluations commute with lax limits.
\end{proof}

\begin{prop}
\label{prop:other characthereisation of limits}
Consider a functor $F:I\to A^\sharp$ between $\U$-small marked $\io$-categories. Then $F$ admits a lax limit if and only if there exists an object $l$ and an equivalence
$$\hom_A(a,l)\sim \laxlim_I\hom_A(a,F(i))$$
natural in $a:A^t$. If such an object exists, then $l$ is the lax limit of $F$.
Dually, $F$ admits a lax colimit if and only if there exists an object $c$ and an equivalence
$$\hom_A(c,a)\sim \laxlim_I\hom_A(F(i),a)$$
natural in $a:A$. If such an object exists, then $c$ is the lax colimit of $F$.
\end{prop}
\begin{proof}
The first assertion  is a direct application of lemma \ref{lemma:tehcnical colimit}. The second one follows by duality, using the fact that the functor 
$$(\uvar)^\circ:\uni\to \uni^{t\circ}$$
preserves limits as it is an equivalence.
\end{proof}

\begin{cor}
\label{cor:left adjoint preserves limits}
Left adjoints between $\U$-small $\io$-categories preserve colimits and right adjoints preserve limits.
\end{cor}
\begin{proof}
Let $u:C\to D$ and $v:D\to C$ be two adjoint functors. Let $F:I\to C^\sharp$ be a functor admitting a colimit.
We then have a sequence of equivalences
$$
\begin{array}{rclc}
\hom_C(u(\laxcolim_IF),b)&\sim &\hom_D(\laxcolim_IF,v(b))\\
&\sim & \laxlim_I\hom_D(F,v(b))&(\ref{prop:other characthereisation of limits})\\
&\sim &\laxlim_I\hom_C(u(F),b)\\
&\sim &\hom_C(\laxlim_Iu(F),b)&(\ref{prop:other characthereisation of limits})
\end{array}
$$
natural in $b:D$. The result then follows from the Yoneda lemma applied to $C^t$. The other assertion is proved similarly.
\end{proof}

\begin{cor}
Consider a functor $F:I\to A^\sharp$ between $\U$-small marked $\io$-categories. Then $F$ admits a limit if and only if there exists an object $l$ and an equivalence
$$\hom_A(a,l)\sim \hom_{\gHom(I,\uni)}(\cst 1,\hom_A(a,F(\uvar))$$
natural in $a:A^t$. If such an object exists, then $l$ is a limit of $F$.
Dually, $F$ admits a colimit if and only if there exists an object $c$ and an equivalence
$$\hom_A(c,a)\sim \hom_{\gHom(I,\uni)}( \cst 1,\hom_A(F(\uvar),a))$$
natural in $a:A$. If such an object exists, then $c$ is the colimit of $F$.
\end{cor}
\begin{proof}
Remark that we have an equivalence
$$\hom_{\gHom(I,\uni)}(\cst 1,\hom_A(a,F(\uvar)))\sim \hom_{\uni}(1,\laxlim_I\hom_A(a,F(\uvar))$$
Eventually, the Yoneda lemma implies that 
$$\hom_{\uni}(1,\laxlim_I\hom_A(a,F(\uvar))\sim\laxlim_I\hom_A(a,F(\uvar))$$
The result then follows from proposition \ref{prop:other characthereisation of limits}.
\end{proof}

\begin{remark}
The characterization of the lax colimit and limit given in previous corollary is the generalization to the case $\io$ of the characterization of lax colimit and limit for $(\infty,2)$-categories given in \cite[corollary 5.1.7]{Gagna_fibrations_and_lax_limit_infini_2_categories}.
\end{remark}

\begin{prop}
\label{prop:limit and final}
Let $i:I\to J$ and $F:J\to A^\sharp$ be two morphisms between $\U$-small marked $\io$-categories. If $i$ is initial, and $F$ admits a lax limit, the functor $F\circ i$ also admits a lax limit, and the canonical morphism:
$$\laxlim_{I}F\to \laxlim_{J}F\circ i$$
is an equivalence.
Dually, if $i$ is final, and $F$ admits a lax colimit, the functor $F\circ i$ also admits a lax colimit, and the canonical morphism:
$$ \laxcolim_{J}F\circ i \to \laxlim_{I}F$$
is an equivalence.
\end{prop}
\begin{proof}
The first assertion is a direct application of the characterization of limits given in proposition \ref{prop:other characthereisation of limits} and of proposition \ref{prop:colimit restricted to final}.
The second assertion follows by duality.
\end{proof}

The proof of the following lemma is a direct adaptation of the one of proposition 5.1 of \cite{Gepner_Lax_colimits_and_free_fibration}.
\begin{prop}
\label{prop:explicit hom between morphism}
Let $f:A\to B$ be any morphism between $\U$-small $\io$-categories..
There is an equivalence
$$\hom_{\uHom(A,B)}(f,g)\sim \laxlim_{a\to b: S(A)}\hom_{B}(f(a),g(a))$$
natural in $f$ and $g$.
\end{prop}
\begin{proof}
Remark first that the left term is in fact equivalent to 
$$\laxlim_{a\to b: S(A)}h^*\hom_{B}(\uvar,\uvar)$$
where $h$ is the left cartesian fibration $S(A)\to A^t\times A$ corresponding to $\hom_A: A^t\times A\to \uni$. We then have 
$$
\begin{array}{rcll}
\laxlim_{a\to b: S(A)}\hom_{B}(f(a),g(a)) &\sim &\hom_{\uni}(1,\laxlim_{a\to b: S(A)}h^*\hom_{B}(\uvar,\uvar)) &(\ref{theo:Yoneda lemma})\\
&\sim &\hom_{\gHom(S(A),\uni)}(\cst 1,h^*\hom_{B}(\uvar,\uvar))\\
&\sim &\hom_{\uHom(A^t\times A,\uni)}(h_! \cst 1,\hom_{B}(\uvar,\uvar))&(\ref{cor:naive kan extension})\\
\end{array}$$
By construction, $h_! \cst 1$ is the Grothendieck deconstruction of the left cartesian fibration $\Lb h_!id \sim h$, and so is equivalent to $\hom_A$. 
We then have 
$$\laxlim_{a\to b: S(A)}\hom_{B}(f(a),g(a))\sim \hom_{\uHom(A^t\times A,\uni)}(\hom_A(\uvar,\uvar),\hom_B(f(\uvar),g(\uvar)))$$
We have a canonical equivalence $\uHom(A^t\times A,\uni)\sim \uHom(A,\w{A})$ sending the functor $\hom_A$ to the Yoneda embedding $y^A$, and $\hom_B(f(\uvar),g(\uvar))$ is $f^*(y^B\circ g)$. 
We then have 
$$
\begin{array}{rcll}
\hom_{}(\hom_A(\uvar,\uvar),\hom_B(f(\uvar),g(\uvar)))&\sim &\hom_{\uHom(A,\w{A})}(y^A,f^*(y^B\circ g))\\
&\sim &\hom_{\uHom(A,\w{B})}(f_!\circ y^A,y^B\circ g)&(\ref{cor:naive kan extension})\\
&\sim &\hom_{\uHom(A,\w{B})}(y^B\circ f,y^B\circ g)&(\ref{prop:left extension commutes with Yoneda})\\
&\sim &\hom_{\uHom(A,B)}(f, g)&(\mbox{Yoneda lemma})\\
\end{array}$$
\end{proof}

\p We suppose the existence of a Grothendieck universe $\Z$ containing $\Wcard$. As a consequence, we can use all the results of the last three subsections to respectively $\V$-small and locally $\V$-small objects.

Let $A$ be a $\U$-small $\io$-category. Let $f$ be an object of $\w{A}$. We define $A^\sharp_{/f}$ as the following pullback
\[\begin{tikzcd}
	{A^\sharp_{/f}} & {\w{A}^\sharp_{/f}} \\
	{A^\sharp} & {\w{A}^\sharp}
	\arrow[from=1-1, to=2-1]
	\arrow[from=1-1, to=1-2]
	\arrow[from=1-2, to=2-2]
	\arrow[from=2-1, to=2-2]
\end{tikzcd}\]

\begin{theorem}
\label{theo:presheaevs colimi of representable}
The colimit of the functor 
$\pi:A^\sharp_{/f}\to A^\sharp\to \w{A}^\sharp$ is $f$.
\end{theorem}
\begin{proof}
We denote by $\pi'$ the canonical projection $\w{A}^\sharp_{/f}\to \w{A}^\sharp$, and 
$t_{A^\sharp_{/f}}:A^\sharp_{/f}\to 1$, $t_{ \w{A}^\sharp_{/f}}:\w{A}^\sharp_{/f}\to 1$ the canonical morphisms.
By the explicit construction of colimits in $\io$-presheaves, we have equivalences 
$$\int_{A^t}\colim_{A^\sharp_{/f}}\pi \sim (id_{(A^t)^\sharp}\times t_{A^\sharp_{/f}})_!E
~~~~~~~~
\int_{A^t}\colim_{\w{A}^\sharp_{/f}}\pi' \sim (id_{(A^t)^\sharp}\times t_{ \w{A}^\sharp_{/f}})_!F$$
where $E$ is the object of $\LCart(A^\sharp\times A^\sharp_{/f})$ induced by currying $\pi$, 
and $F$ is the object of $\LCart(A^\sharp\times \w{A}^\sharp_{/f})$ induced by currying $\pi'$.
We denote by $X\to A^\sharp\times A^\sharp_{/f}$ the left cartesian fibration corresponding to $E$,
and by $Y\to (A^t)^\sharp\times \w{A}^\sharp_{/f}$ the left fibration corresponding to $F$. All this data fits in the diagram
\[\begin{tikzcd}
	X && {S(A)} \\
	& Y && {\dom(\int_{A^t\times \w{A}}\ev)} \\
	{(A^t)^\sharp\times A^{\sharp}_{/f}} && {(A^t)^\sharp\times A^{\sharp}} \\
	& {(A^t)^\sharp\times (\w{A})^{\sharp}_{/f}} && {(A^t)^\sharp\times \w{A}^\sharp}
	\arrow["E", from=1-1, to=3-1]
	\arrow[from=3-1, to=3-3]
	\arrow[from=1-1, to=1-3]
	\arrow[from=1-3, to=3-3]
	\arrow[from=3-3, to=4-4]
	\arrow[from=3-1, to=4-2]
	\arrow["i", from=1-3, to=2-4]
	\arrow[from=2-4, to=4-4]
	\arrow[from=4-2, to=4-4]
	\arrow["j", from=1-1, to=2-2]
	\arrow[from=2-2, to=2-4]
	\arrow["F"{pos=0.4}, from=2-2, to=4-2]
\end{tikzcd}\]
where all squares are cartesian. 
Furthermore, according to the Yoneda lemma, $\dom(\int_{A^t\times \w{A}}\ev))$ is equivalent to $\dom(\int_{A^t\times \w{A}}\hom_{\w{A}}(y_{\uvar},\uvar))$, and 
lemma \ref{lemma:a particular Kan extension} implies that $i$ is initial. As the lower horizontal morphism is a right cartesian fibration, and the dual version of proposition \ref{prop:left cartesian fibration are smooth} induces that $j$ is initial. 
This implies that the canonical morphism
$$(id_{(A^t)^\sharp}\times \bot_{A^\sharp_{/f}})_!E\to (id_{(A^t)^\sharp}\times \bot_{ \w{A}^\sharp_{/f}})_!F$$
is an equivalence, and we then have
$$\colim_{A^\sharp_{/f}}\pi\sim \colim_{\w{A}^\sharp_{/f}}\pi'$$
However, 
$A^\sharp_{/f}$ admits a terminal element, given by $id_f$, and according to proposition \ref{prop:other characthereisation of limits}, we have 
$$\colim_{A^\sharp_{/f}}\pi\sim f.$$
\end{proof}

\begin{cor}
\label{cor:if cocomplete then Yoneda right adjoint}
A $\U$-small $\io$-category $A$ is lax $\U$-cocomplete if and only if the Yoneda embedding has a left adjoint, which we will also note by \wcnotation{$\laxcolim$}{(laxcolim@$\laxcolim:\widehat{C}\to C$}.
\end{cor}
\begin{proof}
If such a left adjoint exists, as $\w{A}$ is lax $\U$-cocomplete, corollary 
\ref{cor:left adjoint preserves limits} implies that $A$ is lax $\U$-cocomplete. Suppose now that $A$ is lax $\U$-cocomplete and let $f:A^t\to \uni$ be a functor. Let $c$ be the colimit of the functor $A^\sharp_{/f}\to A^\sharp$.
According to theorem \ref{theo:presheaevs colimi of representable}, we have a sequence of equivalences
$$\begin{array}{rcl}
 \hom_{\w{A}}(f,y(a))&\sim &\hom_{\w{A}}(\laxcolim_{A^\sharp_{/f}}y(\uvar),y(a))\\
 &\sim & \laxlim_{A^\sharp_{/f}}\hom_{\w{A}}(y(\uvar),y(a))\\
 &\sim & \laxlim_{A^\sharp_{/f}} \hom_A(\uvar,a)\\
 &\sim &\hom_{A} \hom(c,a)\\
\end{array}$$
natural in $a:A^t$. The functor 
$$a:A\mapsto \hom_{\w{A}}(f,y(a))$$
is then representable, which concludes the proof according to proposition \ref{prop:adj if slice as terminal}.
\end{proof}

\p Let $i:A\to B$ be a functor between two $\U$-small $\io$-categories. We define $N_i:B\to \w{A}$ as
$$a:A^t, b:B\mapsto \hom_B(i(a),b)$$
\begin{cor}
\label{cor:adjonction with prehseaves}
Let $i:A\to B$ be a functor between two $\U$-small $\io$-categories with $B$ lax $\U$-cocomplete. 
The morphism $N_i:B\to \w{A}$
admits a left adjoint that sends an $\io$-presheaf $f$ to $\laxcolim_{A^\sharp_{/f}} i(\uvar)$
\end{cor}
\begin{proof}
The proof is similar to the one of corollary \ref{cor:if cocomplete then Yoneda right adjoint}.
\end{proof}

\subsection{Kan extentions}

We suppose the existence of a Grothendieck universe $\Z$ containing $\Wcard$. As a consequence, we can use all the results of the last three subsections to respectively $\V$-small and locally $\V$-small objects.

\p 
Let $f:A\to B^\sharp$ be a morphism between marked $\U$-small $\io$-categories. This induces for any $\io$-category $C$ a morphism 
$$\uvar\circ f:\gHom(B,C)\to \uHom(A,C).$$
Let $g:A\to C$ be a morphism.
A \notion{left Kan extension} of $g$ along $f$ is a functor \sym{(lanf@$\Lan_fg$}$\Lan_fg:B\to C$ and an equivalence
$$\hom_{\uHom(B,C)}(\Lan_fg,h)\sim \hom_{\gHom(A,C)}(g, h\circ f).$$
Remark that if the left Kan extension along $f$ exists for any $g$, the proposition \ref{prop:adj if slice as terminal} implies that
the assignation $g\mapsto \Lan_fg$ can be promoted to a left adjoint, which is called the \notion{global left Kan extension} of $f$.

\begin{prop}
\label{prop:Kan extension an naive kan extension}
Let $C$ be a $\U$-small $\io$-category, $f:I\to B^\sharp$ a functor between $\U$-small $\io$-categories and $g:I\to \uHom(C,\uni)$ a functor. The functor $g$ then corresponds to a morphism $\tilde{g}:\gHom( C^\sharp\times I,\uni)$.
The left Kan extension of $f$ along $g$ corresponds to the morphism $(id_{C^\sharp}\times f)_!\tilde{g}$.
\end{prop}
\begin{proof}
This is a direct consequence of corollary \ref{cor:naive kan extension}.
\end{proof}
	
\begin{cor}
\label{cor:Kan extension of Yonedal along i}
Let $i:A\to B$ be a morphism between $\U$-small $\io$-categories. The left Kan extension of the Yoneda embedding $y:A\to \w{A}$ along $i$ is $N_i:B\to \widehat{A}$.
\end{cor}
\begin{proof}
According to proposition \ref{prop:Kan extension an naive kan extension}, the desired left Kan extension is given by 
$$(B^t\times i)_!\hom_B$$
which is $N_i$ according to lemma \ref{lemma:a particular Kan extension}.
\end{proof}

\begin{prop}
Let $i:A\to B$ a functor between $\U$-small $\io$-categories. The left Kan extension of $y^B\circ i$ along $y^A$ is given by $i_!$.
\end{prop}
\begin{proof}
Let $i:A\to B$ be any functor. Remark first that the Yoneda lemma and the corollary \ref{cor:Kan extension of Yonedal along i} imply that the left Kan extension of $y:A\to \w{A}$ along $y:A\to \w{A}$ is the identity of $\w{A}$.
We then have a sequence of equivalences
$$
\begin{array}{rcll}
\hom_{\uHom(\w{A},\w{A})}(i_!,f)&\sim &\hom_{\uHom(\w{A},\w{A})}(id,i^*\circ f)&(\ref{cor:naive kan extension}) \\
&\sim & \hom_{\uHom(A,\w{A})}(y_A,i^*\circ f\circ y^A)&(\mbox{Yoneda lemma}) \\
&\sim & \hom_{\uHom(A,\w{B})}(i_! \circ y^A, f\circ y^A)&(\ref{cor:naive kan extension}) \\
&\sim & \hom_{\uHom(A,\w{B})}( y_B\circ i, f\circ y^A)&(\ref{prop:left extension commutes with Yoneda})\\
\end{array}
$$
natural in $f:\uHom(\w{A},\w{B})$.
\end{proof}

\begin{cor}
For any morphism $A\to B$ between $\U$-small $\io$-categories with $B$ lax $\U$-cocomplete, there exists a unique colimit preserving functor $\w{A}\to B$ extending $i$.
\end{cor}
\begin{proof}
Let $|\uvar|_i: \w{A}\to B$ be the functor defined in corollary \ref{cor:adjonction with prehseaves}. 
As this functor is an extension of $A$, it fulfills the desired condition, that shows the existence.
The $\io$-category of functors verifying the desired property is given by the pullback 
\[\begin{tikzcd}
	{\uHom_!(\w{A},B)_{i}} & {\uHom_!(\w{A},B)} \\
	{\{i\}} & {\uHom(A,B)}
	\arrow[from=2-1, to=2-2]
	\arrow[from=1-2, to=2-2]
	\arrow[from=1-1, to=2-1]
	\arrow[from=1-1, to=1-2]
\end{tikzcd}\]
where $\uHom_!(\w{A},B)$ is the full sub $\io$-category of $\uHom(\w{A},B)$ whose objects are colimit preserving functors.
As $|\uvar|_i$ is the left Kan extension of $i$ along the Yoneda embedding, there is a transformation 
$$|\uvar|_i\to h$$ natural in $h:\uHom(\w{A},B))_{i}$. To conclude, one has to show that for any object $h$ of $\uHom(\w{A},B))_{i}$ , $|\uvar|_i\to h$ is an equivalence, and so that for any object $f$ of $\w{A}$, $|f|_i\to h(f)$ is an equivalence. As $f$ is a lax colimit of representables as shown in theorem \ref{theo:presheaevs colimi of representable} and as both $|\uvar|_i$ and $h$ preserve lax colimits, this is immediate.
\end{proof}

\begin{cor}
Let $A,B$ and $C$ be three $\U$-small $\io$-categories with $B$ lax $\U$-cocomplete, and
$i:A\to C$ and $f:A\to B$ two functors. The left Kan extension of $i$ along $f$ is given by the composite functor.
$$B\xrightarrow{N_f}\w{A}\xrightarrow{i_!}\w{C}\xrightarrow{\laxcolim_{}} C$$
\end{cor}
\begin{proof}
We have a sequence of equivalences
$$\begin{array}{rcll}
\hom_{\uHom(C,B)}(\laxcolim_{}\circ i_!\circ N_f,h)&\sim & \hom_{\uHom(C,\w{A})}( N_f,i^*\circ y^B\circ h)\\
&\sim & \hom_{\uHom(A,\w{A})}( y^A,i^*\circ y^B\circ h\circ f)&(\ref{cor:Kan extension of Yonedal along i})\\
&\sim & \hom_{\uHom(A,\w{B})}(i_!\circ y^A, y^B\circ h\circ f)&(\ref{cor:naive kan extension})\\
&\sim & \hom_{\uHom(A,\w{B})}(y^B\circ i, y^B\circ h\circ f)&(\ref{prop:left extension commutes with Yoneda})\\
&\sim & \hom_{\uHom(A,B)}( i, h\circ f)&(\ref{theo:Yoneda ff})
\end{array}$$
natural in $h:\uHom(C,B)$.
\end{proof}


\cleardoublepage
\phantomsection
\addcontentsline{toc}{part}{Index of symbols} 
\printindex[notation]
\clearpage
\phantomsection
\addcontentsline{toc}{part}{Index of notions} 
\printindex[notion]

\cleardoublepage
\phantomsection
\addcontentsline{toc}{part}{Bibliography} 
\bibliography{biblio}{}
\bibliographystyle{alpha}

\newpage
$~$
\end{document}